\documentclass[11pt]{article}

\usepackage{fullpage, amsmath}
\usepackage{amsthm,amsfonts,url,amssymb}
\usepackage{mathrsfs}
\usepackage{amssymb,amscd}
\usepackage{color}
\usepackage{chemarrow}
\usepackage{pictexwd,dcpic}
\usepackage[textwidth=18cm,centering]{geometry}
\usepackage{extarrows}
\usepackage[colorlinks,linkcolor=red,anchorcolor=green,citecolor=blue]{hyperref}
\usepackage{enumitem}
\usepackage{lipsum}
\usepackage[all,cmtip]{xy}
\usepackage{leftidx}
\usepackage[version=3]{mhchem}
\usepackage{makeidx}

\makeindex
\newtheorem{theorem}{Theorem}[section]
\newtheorem{corollary}[theorem]{Corollary}
\newtheorem{lemma}[theorem]{Lemma}
\newtheorem{proposition}[theorem]{Proposition}
\newtheorem{definition}[theorem]{Definition}
\newtheorem{hypothesis}[theorem]{Hypothesis}
\newtheorem{remark}[theorem]{Remark}
\newtheorem{notation}[theorem]{Notation}
\newtheorem{example}[theorem]{Example}
\newtheorem{conjecture}[theorem]{Conjecture}
\newtheorem{condition}[theorem]{Condition}

\newcommand{\hooklongrightarrow}{\lhook\joinrel\longrightarrow}
\newcommand{\twoheadlongrightarrow}{\relbar\joinrel\twoheadrightarrow}

\newcommand{\ra}{\rightarrow}
\newcommand{\lra}{\longrightarrow}

\newcommand{\ul}{\underline}

\newcommand{\ttr}{\texttt r}

\newcommand{\bA}{\mathbb A}
\newcommand{\bC}{\mathbb C}

\newcommand{\bG}{\mathbb G}

\newcommand{\Q}{\mathbb Q}
\newcommand{\bR}{\mathbb R}

\newcommand{\bT}{\mathbb T}
\newcommand{\Z}{\mathbb Z}

\newcommand{\bU}{\mathbb U}

\newcommand{\cN}{\mathcal N}
\newcommand{\cL}{\mathcal L}
\newcommand{\co}{\mathcal O}

\newcommand{\cR}{\mathcal R}
\newcommand{\cH}{\mathcal H}
\newcommand{\cC}{\mathcal C}
\newcommand{\cS}{\mathcal S}
\newcommand{\cD}{\mathcal D}
\newcommand{\cI}{\mathcal I}
\newcommand{\cW}{\mathcal W}

\newcommand{\cM}{\mathcal M}
\newcommand{\cF}{\mathcal F}
\newcommand{\cV}{\mathcal V}
\newcommand{\cE}{\mathcal E}

\newcommand{\cU}{\mathcal U}
\newcommand{\cZ}{\mathcal Z}

\newcommand{\ur}{\mathfrak r}
\newcommand{\fn}{\mathfrak n}
\newcommand{\fh}{\mathfrak h}
\newcommand{\fm}{\mathfrak{m}}
\newcommand{\ub}{\mathfrak b}
\newcommand{\fl}{\mathfrak l}
\newcommand{\fp}{\mathfrak p}

\newcommand{\ug}{\mathfrak g}

\newcommand{\fX}{\mathfrak X}
\newcommand{\fN}{\mathfrak N}
\newcommand{\fM}{\mathfrak M}

\newcommand{\fC}{\mathfrak C}

\newcommand{\ft}{\mathfrak t}
\newcommand{\fa}{\mathfrak a}
\newcommand{\fz}{\mathfrak z}

\newcommand{\fd}{\mathfrak d}

\newcommand{\fe}{\mathfrak e}
\newcommand{\fL}{\mathfrak L}
\newcommand{\fZ}{\mathfrak Z}

\newcommand{\sS}{\mathscr S}

\newcommand{\sF}{\mathscr F}

\newcommand{\sZ}{\mathscr Z}
\newcommand{\sU}{\mathscr U}

\newcommand{\sW}{\mathscr W}

\newcommand{\sV}{\mathscr V}

\newcommand{\sI}{\mathscr I}
\newcommand{\sG}{\mathscr G}
\newcommand{\sT}{\mathscr T}

\newcommand{\Addresses}{{
\bigskip
\footnotesize	
		
C. ~Breuil, \textsc{CNRS, B\^atiment 307, 
	Facult\'e d'Orsay, Universit\'e Paris-Saclay, 
	91405 Orsay Cedex, France}\par\nopagebreak
\textit{E-mail address}, C.~Breuil: \texttt{christophe.breuil@universite-paris-saclay.fr}

\medskip

Y.~Ding, \textsc{B.I.C.M.R., Peking University,
	No.5 Yiheyuan Road Haidian District,
	Beijing, P.R. China 100871}\par\nopagebreak
\textit{E-mail address}, Y.~Ding: \texttt{yiwen.ding@bicmr.pku.edu.cn}
}}

\DeclareMathOperator{\diag}{\mathrm diag}

\DeclareMathOperator{\la}{\mathrm la}

\DeclareMathOperator{\gl}{\mathfrak gl}

\DeclareMathOperator{\GL}{\mathrm GL}
\DeclareMathOperator{\gr}{\mathrm gr}

\DeclareMathOperator{\Fil}{\mathrm Fil}
\DeclareMathOperator{\Res}{\mathrm Res}

\DeclareMathOperator{\Gal}{\mathrm Gal}
\DeclareMathOperator{\Hom}{\mathrm Hom}

\DeclareMathOperator{\End}{\mathrm End}
\DeclareMathOperator{\cris}{\mathrm cris}
\DeclareMathOperator{\rig}{\mathrm rig}
\DeclareMathOperator{\an}{\mathrm an}
\DeclareMathOperator{\Spec}{\mathrm Spec}
\DeclareMathOperator{\Symm}{\mathrm Symm}
\DeclareMathOperator{\dR}{\mathrm dR}
\DeclareMathOperator{\Frob}{\mathrm Frob}
\DeclareMathOperator{\Ind}{\mathrm Ind}
\DeclareMathOperator{\unr}{\mathrm unr}

\DeclareMathOperator{\Ker}{\mathrm Ker}
\DeclareMathOperator{\pr}{\mathrm pr}

\DeclareMathOperator{\Ext}{\mathrm Ext}

\DeclareMathOperator{\Spm}{\mathrm Spm}
\DeclareMathOperator{\Spf}{\mathrm Spf}
\DeclareMathOperator{\Ima}{\mathrm Im}

\DeclareMathOperator{\lalg}{\mathrm lalg}

\DeclareMathOperator{\id}{\mathrm id}

\DeclareMathOperator{\dett}{\mathrm det}
\DeclareMathOperator{\alg}{\mathrm alg}
\DeclareMathOperator{\cyc}{\mathrm cyc}

\DeclareMathOperator{\Supp}{\mathrm Supp}

\DeclareMathOperator{\Ad}{\mathrm Ad}

\DeclareMathOperator{\sss}{\mathrm ss}

\DeclareMathOperator{\red}{\mathrm red}

\DeclareMathOperator{\Art}{\mathrm Art}

\DeclareMathOperator{\rk}{\mathrm rk}
\DeclareMathOperator{\Tor}{\mathrm Tor}
\DeclareMathOperator{\wt}{\mathrm wt}

\DeclareMathOperator{\fss}{\mathrm fs}
\DeclareMathOperator{\val}{\mathrm val}
\DeclareMathOperator{\univ}{\mathrm univ}
\DeclareMathOperator{\Coker}{\mathrm Coker}
\DeclareMathOperator{\rec}{\mathrm rec}
\DeclareMathOperator{\pcr}{\mathrm pcr}
\DeclareMathOperator{\tw}{\mathrm tw}
\DeclareMathOperator{\adm}{\mathrm adm}
\DeclareMathOperator{\tri}{\mathrm tri}
\DeclareMathOperator{\reg}{\mathrm reg}

\DeclareMathOperator{\cind}{c-\mathrm{ind}}
\DeclareMathOperator{\pst}{pst}
\DeclareMathOperator{\DF}{DF}
\DeclareMathOperator{\loc}{loc}
\DeclareMathOperator{\pdR}{\mathrm pdR}

\DeclareMathOperator{\ver}{\mathrm ver}
\DeclareMathOperator{\gen}{\mathrm gen}
\DeclareMathOperator{\Aut}{\mathrm Aut}
\DeclareMathOperator{\Fro}{\mathrm Fr}
\DeclareMathOperator{\rh}{\mathrm rh}
\DeclareMathOperator{\Ch}{\mathrm Ch}
\DeclareMathOperator{\Mod}{\mathrm Mod}
\DeclareMathOperator{\Loc}{\mathrm Loc}
\DeclareMathOperator{\BB}{\mathrm BB}
\DeclareMathOperator{\sd}{\mathrm sd}
\newcommand{\lWPmin}{\ce{^$P$_{min}$\sW$}}
\newcommand{\lWPmax}{\ce{^$P$_{max}$\sW$}}

\DeclareMathOperator{\sm}{\mathrm sm}

\DeclareMathOperator{\qc}{\mathrm qc}
\begin{document}	
	
\title{Bernstein eigenvarieties}

\author{Christophe Breuil and Yiwen Ding}
\date{}
\maketitle

\begin{abstract}
We construct parabolic analogues of (global) eigenvarieties, of patched eigenvarieties and of (local) trianguline varieties, that we call respectively Bernstein eigenvarieties, patched Bernstein eigenvarieties, and Bernstein paraboline varieties. We study the geometry of these rigid analytic spaces, in particular (generalizing results of Breuil-Hellmann-Schraen) we show that their local geometry can be described by certain algebraic schemes related to the generalized Grothendieck-Springer resolution. We deduce several local-global compatibility results, including a classicality result (with no trianguline assumption at $p$), and new cases towards the locally analytic socle conjecture of Breuil in the non-trianguline case.
\end{abstract}
	
{\hypersetup{linkcolor=black}
\tableofcontents}

\numberwithin{equation}{section}
	
\numberwithin{theorem}{section}	

\newpage

\section{Introduction}\label{intro}

Let $p$ be a prime number. The study of $p$-adic eigenvarieties is an important and fruitful theme in arithmetic geometry. This paper is motivated by the role that eigenvarieties play in the study of local-global compatibility problems in the ($p$-adic) Langlands program (e.g.\ see \cite{BHS3} or \cite{Ding7}). The ``classical'' theory of eigenvarieties has the restriction that one can only see finite slope $p$-adic automorphic forms or representations, or, in terms of Galois representations, trianguline representations. In order to extend the theory to the \textit{non-trianguline} case, we construct in this paper a parabolic version of eigenvarieties, that we call Bernstein eigenvarieties. Note that there was already some work in that direction, cf.\ \cite{HL}, \cite{Lo20p}, \cite{BW20} (see after Theorem \ref{centerintro} below for a brief comparison with \cite{HL}). These spaces parametrize certain $p$-adic automorphic (resp.\ Galois) representations which are not of finite slope (resp.\ not trianguline). Following the strategy and methods in the series of articles \cite{BHS1}, \cite{BHS2}, \cite{BHS3} (which themselves are based on many previous results by other people), we study and use the geometry of these Bernstein eigenvarieties and of their patched and (local) Galois avatars to obtain various local-global compatibility results in the non-trianguline case. 

Before stating our main results, we briefly give the global setup of the paper (with some simplifications for convenience). Let $n\geq 2$ an integer, $F^+$ a totally real number field and $F$ a totally imaginary quadratic extension of $F^+$ such that all places of $F^+$ dividing $p$ split in $F$. We fix a unitary algebraic group $G$ over $F^+$ which becomes $\GL_n$ over $F$ and such that $G(F^+\otimes_{\Q}\mathbb R)$ is compact and $G$ is split at all places above $p$. Let $U^p=\prod_{v\nmid \infty, p} U_v$ be a compact open subgroup of $G(\bA_{F^+}^{\infty,p})$, and $E$ be a sufficiently large finite extension of $\Q_p$. Put 
\begin{equation*}
	\widehat{S}(U^p,E):=\{f: G(F^+)\backslash G(\bA_{F^+}^{\infty})/U^p \ra E, \ \text{$f$ continuous}\},
\end{equation*}
which is a Banach space for the supremum norm and is equipped with a continuous (unitary) action of $G(F^+_p):=G(F^+ \otimes_{\Q} \Q_p)$ by right translation on functions. For simplicity, we assume in this introduction that $p$ is inert in $F^+$, and we fix a place $\wp$ of $F$ dividing $p$. We have then $G(F^+_p)\cong \GL_n(F^+_p)\cong \GL_n(F_{\wp})$ and $\widehat{S}(U^p,E)$ is a unitary Banach space representation of $\GL_n(F_{\wp})$. The space $\widehat{S}(U^p,E)$ is also equipped with a faithful action of a certain commutative global Hecke algebra $\bT(U^p)$ over $\co_E$ (the ring of integers of $E$) which is generated by sufficiently many prime-to-$p$ Hecke operators. We recall that the locally algebraic vectors $\widehat{S}(U^p,E)^{\lalg}$ (for the action of $\GL_n(F_{\wp})$) admit a description in terms of classical automorphic representations of $G$ (see for example Proposition \ref{pAF0} (1)). With some more assumptions (on $G$, $F$, etc., see \S~\ref{galois} and \S~\ref{secPBern}), one can show that
\begin{itemize} 
	\item $\bT(U^p)$ is isomorphic to a finite product of complete noetherian $\co_E$-algebras $\bT(U^p)_{\overline{\rho}}$, indexed by some $n$-dimensional continuous representations $\overline{\rho}$ of $\Gal_F$ over the residue field $k_E$ of $\co_E$;
	\item for each $\overline{\rho}$ such that $\bT(U^p)_{\overline{\rho}}\ne 0$, there is a surjective morphism $R_{\overline{\rho},\cS} \twoheadrightarrow \bT(U^p)_{\overline{\rho}}$, where $R_{\overline{\rho},\cS}$ is the universal deformation ring of certain deformations of $\overline{\rho}$ (see \S~\ref{secPBern}). 
\end{itemize}
These properties are not necessary for our construction of Bernstein eigenvarieties, but for convenience we assume they hold in the introduction. With respect to $\bT(U^p)\cong \prod_{\overline{\rho}} \bT(U^p)_{\overline{\rho}}$, we have a decomposition $\widehat{S}(U^p,E)\cong \bigoplus_{\overline{\rho}} \widehat{S}(U^p,E)_{\overline{\rho}}$ of $\GL_n(F_{\wp})$-representations. We fix henceforth $\overline{\rho}$ such that $\bT(U^p)_{\overline{\rho}}\neq 0$, and thus $\widehat{S}(U^p,E)_{\overline{\rho}}\neq 0$. Let $\fm_{\rho}\subset \bT(U^p)_{\overline{\rho}}[1/p]$ be a maximal ideal and $\rho: \Gal_F \ra \GL_n(E)$ the associated representation, where $\Gal_L:=\Gal(\overline L/L)$ for a field $L$. The subspace $\widehat{S}(U^p,E)_{\overline{\rho}}[\fm_{\rho}]$ of $\widehat{S}(U^p,E)_{\overline{\rho}}$ annihilated by $\fm_{\rho}$ is an admissible unitary Banach space representation of $\GL_n(F_{\wp})$. The study of its relation with the local Galois representation $\rho_{\wp}:=\rho|_{\Gal_{F_{\wp}}}$, referred to as local-global compatibility, is one of the main themes in the ($p$-adic) Langlands program. We first state our local-global compatibility results, which are obtained using the aforementioned Bernstein eigenvarieties, and we describe these latter afterwards. For these results, as in \cite{BHS3} we also need the following so-called Taylor-Wiles hypothesis:

\begin{hypothesis}\label{TayWil0}
	(1) $p>2$;
	
	(2) the field $F$ is unramified over $F^+$, $F$ does not contain a non trivial root $\sqrt[p]{1}$ of $1$ and $G$ is quasi-split at all finite places of $F^+$;
	
	(3) $U_{v}$ is hyperspecial when the finite place $v$ of $F^+$ is inert in $F$;
	
	(4) $\overline{\rho}$ is absolutely irreducible and $\overline{\rho}(\Gal_{F(\sqrt[p]{1})})$ is adequate.
\end{hypothesis}

Let $B\subset \GL_n$ be the Borel subgroup of upper triangular matrices and $T\subset B$ the subgroup of diagonal matrices. We first have the following classicality result.

\begin{theorem}[cf.\ Theorem \ref{class00}]\label{intcl}
	Assume Hypothesis \ref{TayWil0} and the following hypothesis (at $p$)\footnote{We also need some mild assumption on $\rho_{v}$ for finitely many finite places $v\nmid p$ of $F$ that we omit in the introduction, see Theorem \ref{class00}.}:
	\begin{enumerate}[label=(\arabic*)]
		\item $\rho_{\wp}$ is potentially crystalline with distinct Hodge-Tate weights and is generic (in the sense of \S~\ref{introPcr});
		\item there exists a parabolic subgroup $P\supseteq B$ of $\GL_n$ such that the corresponding Jacquet-Emerton module $J_P(\widehat{S}(U^p,E)_{\overline{\rho}}^{\an}[\fm_{\rho}])$ has non-zero locally algebraic vectors for $L_P^D(F_{\wp})$, where $L_P^D$ is the derived subgroup of the Levi subgroup $L_P\supseteq T$ of $P$.
	\end{enumerate}
	Then $\widehat{S}(U^p,E)[\fm_{\rho}]^{\lalg}\neq 0$, i.e.\ $\rho$ is associated to a classical automorphic representation of $G(\bA_{F^+})$. 
\end{theorem}

The theorem in the case $P=B$ was proved in \cite{BHS3}. Indeed, the assumption (2) in Theorem \ref{intcl} when $P=B$ is equivalent to $J_B(\widehat{S}(U^p,E)_{\overline{\rho}}^{\an}[\fm_{\rho}])\neq 0$, which means that $\rho$ appears on the (classical) eigenvariety associated to $G$ with tame level $U^p$. The main novelty of Theorem \ref{intcl} is that $\rho_{\wp}$ is not necessarily trianguline.

 Assume that $\rho_{\wp}$ is generic potentially crystalline and let $\ttr(\rho_{\wp})$ be the Weil-Deligne representation associated to $\rho_{\wp}$. It admits a decomposition $\ttr(\rho_{\wp})\cong \oplus_{i=1}^r \ttr_i$ by absolutely irreducible Weil-Deligne representations $\ttr_i$ (which are distinct as $\rho_{\wp}$ is generic). Each ordering of the $\ttr_i$ defines a partial flag $\sF$ on $\ttr(\rho_{\wp})$, and we let $P\supseteq B$ be the associated parabolic subgroup. We say $\sF$ is a $P$-filtration on $\ttr(\rho_{\wp})$, or a refinement of $\rho_{\wp}$. By Fontaine's theory, the filtration $\sF$ induces a $P$-filtration $\sF_{\tau}$ on $D_{\dR}(\rho_{\wp})_{\tau}:=D_{\dR}(\rho_{\wp}) \otimes_{(F_{\wp} \otimes_{\Q_p} E)} (F_{\wp} \otimes_{F_{\wp}, \tau} E)$ for each $\tau: F_{\wp} \hookrightarrow E$. Let $\Fil_{\tau}$ be the full flag of the Hodge filtration on $D_{\dR}(\rho_{\wp})_{\tau}$. The relative position of the partial flag $\sF_{\tau}$ and the full flag $\Fil_{\tau}$ is measured by an element $w\in \sW_P^{\max}$, where $\sW_P^{\max}$ is the set of maximal length representatives in the Weyl group $\sW\cong S_n$ of $\GL_n$ of the {\it right} cosets in $\sW_{L_P}\backslash \sW$ (here $\sW_{L_P}$ is the Weyl group of $L_P$). More precisely, fixing a basis of $D_{\dR}(\rho_{\wp})_{\tau}$ over $E$, then $\sF_{\tau}$ (resp.\ $\Fil_{\tau}$) corresponds to an element in $\GL_n/P$ (resp.\ $\GL_n/B$), still denoted by $\sF_{\tau}$ (resp.\ $\Fil_{\tau}$). There exists a unique $w_{\tau}\in \sW_P^{\max}$ such that $(\sF_{\tau}, \Fil_{\tau})$ lies in the $\GL_n$-orbit of $(1,w_{\tau})\in \GL_n/P \times \GL_n/B$ for the diagonal $\GL_n$-action (where $w_{\tau}$ here denotes a lifting of $w_{\tau}\in \sW$ in $N_{\GL_n}(T)$). We let $w_{\sF}:=(w_{\tau})\in \sW_{P,F_{\wp}}^{\max}:=\prod_{\tau: F_{\wp} \hookrightarrow E} \sW_P^{\max}$.

On the automorphic side, to any $w\in \sW_{P,F_{\wp}}^{\min}$ (defined as $\sW_{P,F_{\wp}}^{\max}$ with ``maximal length" replaced by ``minimal length"), one can associate as in \cite[\S~6]{Br13I} a topologically irreducible locally $\Q_p$-analytic representation $C(w, \sF)$ of $\GL_n(F_{\wp})$ over $E$. We refer to \textit{loc.\ cit.} and \S~\ref{seccompCP} for its precise definition. We recall that, when $w=1$, $C(1,\sF)$ is isomorphic to the locally algebraic representation of $\GL_n(F_{\wp})$ associated to $\rho_{\wp}$ by the classical (suitably normalized) local Langlands correspondence (and is actually independent of $\sF$) and that if $\widehat{S}(U^p,E)[\fm_{\rho}]^{\lalg}\neq 0$, then $\widehat{S}(U^p,E)[\fm_{\rho}]^{\lalg}\cong C(1,\sF)^{\oplus m}$ for some $m\in \Z_{\geq 1}$. Let $w_0\in \sW_{F_{\wp}}:=\prod_{\tau: F_{\wp} \ra E} \sW$ be the element of maximal length, the following theorem establishes several new cases of \cite[Conj.\ 5.3]{Br13II}.

\begin{theorem}[cf.\ Theorem \ref{casparticulier}]\label{intcs}
	Assume Hypothesis \ref{TayWil0}, $\widehat{S}(U^p,E)[\fm_{\rho}]^{\lalg}\neq 0$, and that $\rho_{\wp}$ is generic potentially crystalline with distinct Hodge-Tate weights. Assume moreover 
	\begin{itemize}
		\item[($*$)] any two factors $\GL_{n_i}$ in $L_P=\begin{pmatrix} \GL_{n_1} & \cdots &0 \\ \vdots & \ddots & \vdots \\ 0 & \cdots & \GL_{n_r}\end{pmatrix}$ with $n_i>1$ (if they exist) are not adjacent. 
	\end{itemize}
	Then for $w\in \sW_{P,F_{\wp}}^{\min}$ we have a $\GL_n(F_{\wp})$-equivariant injection $C(w,\sF) \hookrightarrow \widehat{S}(U^p,E)^{\an}_{\overline{\rho}}[\fm_{\rho}]$ if and only if $ww_0\geq w_{\sF}$ for the Bruhat order\footnote{Note that for $w\in \sW_{F_{\wp}}$, $w\in \sW_{P,F_{\wp}}^{\min}$ is equivalent to $ww_0\in \sW_{P,F_{\wp}}^{\max}$.}.
\end{theorem}

The case $P=B$ was proved in \cite{BHS3} (see also \cite{Br13II}, \cite{Ding6} for related work). When $P\neq B$, almost nothing was known, except very partial results in \cite{Ding9}. Note that, if there is at most one $\ttr_i$ with $\dim_E \ttr_i>1$, then the assumption ($*$) is empty and \cite[Conj.\ 5.3]{Br13II} is proved for any refinement $\sF$ of $\rho_{\wp}$. The technical assumption ($*$) comes from some properties of the geometry of Bernstein eigenvarieties (see the discussion at the end of this introduction). Without ($*$), we still have some partial results towards \cite[Conj.\ 5.3]{Br13II}, for instance the ``only if" part in the conclusion of Theorem \ref{intcs} holds without any assumption on $P$. 

When $P=B$, both Theorem \ref{intcl} and Theorem \ref{intcs} followed from an extensive study of eigenvarieties (and of the corresponding patched eigenvarieties and trianguline varieties) in \cite{BHS1}, \cite{BHS2}, \cite{BHS3}. We follow the same strategy in this work.

Let us start by defining the global Bernstein eigenvarieties. We need more notation. We fix a parabolic subgroup $P\supseteq B$ of $\GL_n$ and write its Levi subgroup $L_P$ as $\begin{pmatrix} \GL_{n_1} & \cdots & 0 \\ \vdots & \ddots & \vdots \\ 0 & \cdots &\GL_{n_r}\end{pmatrix}$. We fix $\Omega=\prod_{i=1}^r \Omega_i$ a cuspidal Bernstein component of $L_P(F_{\wp})\cong \prod_{i=1}^r \GL_{n_i}(F_{\wp})$, and denote by $\cZ_{\Omega}=\otimes_{i=1}^r \cZ_{\Omega_i}$ the corresponding Bernstein centre (over $E$). Thus a point of $(\Spec \cZ_{\Omega})^{\rig}$ corresponds to an irreducible smooth cuspidal representation of $L_P(F_{\wp})$, and we frequently use the associated representation to denote the point. Let $\cZ_0:=Z_{L_P}(\co_{F_{\wp}})$ (where $Z_{L_P}$ is the center of $L_P$) and $\widehat{\cZ_0}$ be the rigid space over $E$ parametrizing continuous characters of $\cZ_0$. We fix a uniformizer $\varpi$ of $F_{\wp}$. For a continuous character $\chi$ of $(\co_{F_{\wp}}^{\times})^{\oplus s}$ with $s\in \Z_{\geq 1}$, we denote by $\chi_{\varpi}$ the character of $(F_{\wp}^{\times})^{\oplus s}$ such that $\chi_{\varpi}|_{(\co_{F_{\wp}}^{\times})^{\oplus r}}=\chi$ and $\chi_{\varpi}\big((\varpi^{k_1},\dots, \varpi^{k_s})\big)=1$ for any $k_1, \dots, k_s \in \Z$. Finally we also fix $\lambda=(\lambda_i)_{1 \leq i \leq n}=(\lambda_{i,\tau})_{\substack{1\leq i \leq n\\ \tau: F_{\wp} \hookrightarrow E}}$ an integral $P$-dominant weight of $\GL_n(F_{\wp})$. The following theorem summarizes some key features of Bernstein eigenvarieties:

\begin{theorem}[cf.\ \S~\ref{secBern}]\label{centerintro}
	For each $(\Omega,\lambda)$ as above there is a rigid analytic space $\cE_{\Omega, \lambda}(U^p, \overline{\rho})$ over $E$ and an injection of rigid spaces over $E$
	\[\iota_{\Omega, \lambda}: \cE_{\Omega, \lambda}(U^p, \overline{\rho})\hookrightarrow (\Spf \bT(U^p)_{\overline{\rho}})^{\rig} \times (\Spec \cZ_{\Omega})^{\rig} \times \widehat{\cZ_0}\ \big(\hookrightarrow (\Spf R_{\overline{\rho},\cS})^{\rig} \times (\Spec \cZ_{\Omega})^{\rig} \times \widehat{\cZ_0}\big)\]
	satisfying the following properties:
	
	(1) the induced morphism $\cE_{\Omega, \lambda}(U^p, \overline{\rho}) \ra \widehat{\cZ_0}$ is locally finite;
	
	(2) a point $(\eta, \pi_{L_P},\chi)\in (\Spf \bT(U^p)_{\overline{\rho}})^{\rig} \times (\Spec \cZ_{\Omega})^{\rig} \times \widehat{\cZ_0}$ lies in $\cE_{\Omega, \lambda}(U^p, \overline{\rho})$ if and only if there exists an injection of locally $\Q_p$-analytic representations of $L_P(F_{\wp})$ (where $\fm_{\eta}$ denotes the maximal ideal of $\bT(U^p)_{\overline{\rho}}[1/p]$ associated to $\eta$):
	\[\pi_{L_P}\otimes_E (\chi_{\varpi} \circ \dett_{L_P}) \otimes_E L(\lambda)_P \hooklongrightarrow J_P(\widehat{S}(U^p,E)^{\an}_{\overline{\rho}})[\fm_\eta]\]
	where $J_P(\cdot)$ is Emerton's locally analytic Jacquet functor for $P$ (\cite{Em11}), $L(\lambda)_P$ is the algebraic representation of $L_P(F_{\wp})$ of highest weight $\lambda$ (with respect to $B\cap L_P$), and $\dett_{L_P}: L_P(F_{\wp}) \ra Z_{L_P}(F_{\wp})$ is the determinant map;
	
	(3) $\cE_{\Omega, \lambda}(U^p,\overline{\rho})$ is equidimensional of dimension $[F_{\wp}:\Q_p]r$;
	
	(4) the set of classical points, i.e.\ the points $(\eta, \pi_{L_P},\chi)\in \cE_{\Omega, \lambda}(U^p,\overline{\rho})$ such that
	\[\pi_{L_P}\otimes_E (\chi_{\varpi} \circ \dett_{L_P}) \otimes_E L(\lambda)_P \hooklongrightarrow J_P(\widehat{S}(U^p,E)_{\overline{\rho}}^{\lalg})[\fm_\eta],\]
	is Zariski-dense in $\cE_{\Omega, \lambda}(U^p,\overline{\rho})$.
\end{theorem}

Similar spaces have been constructed in \cite{HL} (see also \cite{BW20} for a construction via overconvergent cohomology, and also \cite{Lo20p}), but the new feature in Theorem \ref{centerintro} is that we take into account the action of the full Bernstein centre (rather than just the action of $Z_{L_P}(F_{\wp})$), obtained by applying Bushnell-Kutzko's theory of types. This allows to parametrize the full $L_P(F_{\wp})$-action, which is particularly important for our applications. When $P=B$, we have an isomorphism $\widehat{\cZ_0}\times (\Spec \cZ_{\Omega})^{\rig}\cong \widehat{T(F_{\wp})}$ (= continuous characters of $T(F_{\wp})$) and we can show that all (reduced) varieties $\cE_{\Omega, \lambda}(U^p,\overline{\rho})^{\red}$ are isomorphic to the (finite slope) reduced eigenvariety $\cE(U^p,\overline{\rho})^{\red}$, see Remark \ref{remP=B2}.

We next discuss $p$-adic families of Galois representations on $\cE_{\Omega, \lambda}(U^p,\overline{\rho})$. For a point $x_i$ of $(\Spec \cZ_{\Omega_i})^{\rig}$ ($i\in \{1,\dots,r\}$), we denote by $\pi_{x_i}$ the associated irreducible supercuspidal representation of $\GL_{n_i}(F_{\wp})$ over $E$ and we let $\ttr_{x_i}:=\rec(\pi_{x_i})$ be the associated (absolutely irreducible) Weil representation via the classical local Langlands correspondence (normalized as in \cite{HT}). The Weil representation $\ttr_{x_i}$ corresponds to a Deligne-Fontaine module $\DF_{x_i}$ (see for instance \cite[Prop.\ 4.1]{BS07}), which by Berger's theory \cite[Thm.\ A]{Ber08a} corresponds in turn to a $p$-adic differential equation $\Delta_{x_i}$, i.e.\ a $(\varphi, \Gamma)$-module of rank $n_i$ over $\cR_{k(x_i),F_{\wp}}$ which is de Rham of constant Hodge-Tate weight $0$ (here $k(x_i)$ is the residue field at $x_i$ and $\cR_{k(x_i),F_{\wp}}$ is the Robba ring associated to $F_{\wp}$ with $k(x_i)$-coefficients). Let $x:=(\eta, (x_i), \chi=(\chi_i)) \in \cE_{\Omega, \lambda}(U^p,\overline{\rho})$, the image of $\eta$ via the injection $(\Spf \bT(U^p)_{\overline{\rho}})^{\rig} \hookrightarrow (\Spf R_{\overline{\rho},\cS})^{\rig}$ corresponds to a continuous representation $\rho_x: \Gal(\overline F/F) \ra \GL_n(k(x))$ over the residue field $k(x)$ of $\cE_{\Omega, \lambda}(U^p,\overline{\rho})$ at $x$. We put $\rho_{x,\wp}:=\rho_x |_{\Gal_{F_{\wp}}}$. 

\begin{theorem}[cf.\ Theorem \ref{FOBE}, Proposition \ref{galBE1}]\label{iThm2}
Let $x$ as above.
	
	(1) The $(\varphi, \Gamma)$-module $D_{\rig}(\rho_{x,\wp})$ associated to $\rho_{x,\wp}$ admits a $P$-filtration $\Fil_{\bullet}D_{\rig}(\rho_{x,\wp})$ by saturated $(\varphi, \Gamma)$-submodules of $D_{\rig}(\rho_{x,\wp})$ such that for $i=1, \dots, r$:
	\begin{equation}\label{fil1}
		\big(\gr_i D_{\rig}(\rho_{x,\wp})\big)[1/t] \cong \big(\Delta_{x_i}'
		\otimes_{\cR_{k(x),F_{\wp}}} \cR_{k(x),F_{\wp}}(\chi_{i,\varpi})\big)[1/t],
	\end{equation} 
	where $\cR_{k(x),F_{\wp}}(\delta)$ denotes the rank one $(\varphi, \Gamma)$-module associated to a continuous character $\delta$ and $\Delta_{x_i}':=\Delta_{x_i} \otimes_{\cR_{k(x),F_{\wp}}} \cR_{k(x),F_{\wp}}(\unr(q^{-s_{i-1}+(1-n_i)/2}))$, $q$ being the cardinality of the residue field of $F_{\wp}$, $s_{i-1}:=r_1+\cdots+ r_{i-1}$ and $\unr(a)$ denoting the unramified character of $F_{\wp}^{\times}$ sending any uniformizer to $a$. 
	
	(2) For $\tau: F_{\wp} \hookrightarrow E$, the Sen $\tau$-weights of $\rho_{x,\wp}$ are given by $\{h_{j_i,\tau}+\wt(\chi_{i})_{\tau}\}_{\substack{1\leq i \leq r \\ s_{i-1}+1 \leq j_i \leq s_i}}$, where $\wt(\chi')_{\tau}$ denotes the $\tau$-weight of $\chi'$ for a continuous character $\chi'$ of $\co_{F_{\wp}}^{\times}$ (or of $F_{\wp}^{\times}$).
\end{theorem}

We call a filtration on a $(\varphi, \Gamma)$-module satisfying the property (\ref{fil1}) an \textit{$\Omega$-filtration}, and call $(\{\Delta_{x_i}\},\{\chi_i\})$ a parameter of the $\Omega$-filtration. Theorem \ref{iThm2} (1) gives an analogue of the fact that $p$-adic Galois representations over eigenvarieties are trianguline. The proof of Theorem \ref{iThm2} (1) is based on an interpolation result for $\Omega$-filtrations in families given in \S~\ref{globaltriangulation}, which is an analogue of the theory of global triangulation of \cite{KPX}, \cite{Liu}, \cite{Bergd14}. 

We now define rigid analytic spaces which are (local) Galois avatars of Bernstein eigenvarieties and analogues of the trianguline variety of \cite{HeSc}, \cite{BHS1} when $P=B$. They parametrize Galois representations admitting an $\Omega$-filtration. Let $\overline{\rho}_{\wp}:=\overline{\rho}|_{\Gal_{F_{\wp}}}$ and $\textbf{h}:=(\textbf{h}_i)_{i=1,\dots,n}:=(h_{i,\tau})_{\substack{i=1,\dots, n\\ \tau: F_{\wp} \hookrightarrow E}}$ with $h_{i,\tau}=\lambda_{i,\tau}-i+1$. Note that the weight $\textbf{h}$ is strictly $P$-dominant. Let $R_{\overline{\rho}_{\wp}}$ be the universal framed deformation ring of $\overline{\rho}_{\wp}$. Define $U_{\Omega, \textbf{h}}(\overline{\rho}_{\wp})$ as the subset of $(\Spf R_{\overline{\rho}_{\wp}})^{\rig} \times (\Spec \cZ_{\Omega})^{\rig} \times \widehat{\cZ_0} $ which consists of the points $(\varrho, (x_i), (\chi_i))$ such that:
\begin{itemize}
	\item the parameter $\big((x_i), (\chi_i)\big)$ is generic in the sense of \S~\ref{s: DO};
	\item $D_{\rig}(\varrho)$ admits an $\Omega$-filtration $\{\Fil_i D_{\rig}(\varrho)\}$ such that one has embeddings
	\begin{equation*}
		\gr_i D_{\rig}(\varrho) \otimes_{\cR_{k(x),F_{\wp}}} \cR_{k(x),F_{\wp}}(\chi_{i,\varpi}^{-1}) \hooklongrightarrow \Delta_{x_i} \otimes_{\cR_{k(x),F_{\wp}}} \cR_{k(x),F_{\wp}}(z^{\textbf{h}_{s_i}})
	\end{equation*}
	where the image has Hodge-Tate weights $(\textbf{h}_{s_{i-1}+1}, \dots, \textbf{h}_{s_i})$ (here $z^{\textbf{h}_{s_i}}\!:=\!\prod_{\tau: F_{\wp} \hookrightarrow E} \tau(z)^{h_{s_i,\tau}}$).
\end{itemize}
We define $X_{\Omega, \textbf{h}}(\overline{\rho}_{\wp})$ to be the (reduced) Zariski-closure of $U_{\Omega, \textbf{h}}(\overline{\rho}_{\wp})$ in $(\Spf R_{\overline{\rho}_{\wp}})^{\rig} \times (\Spec \cZ_{\Omega})^{\rig} \times \widehat{\cZ_0}$ and denote by $\iota_{\Omega, \textbf{h}}: X_{\Omega, \textbf{h}}(\overline{\rho}_{\wp})\hookrightarrow (\Spf R_{\overline{\rho}_{\wp}})^{\rig} \times (\Spec \cZ_{\Omega})^{\rig} \times \widehat{\cZ_0}$ the closed embedding. We call $X_{\Omega, \textbf{h}}(\overline{\rho}_{\wp})$ a \textit{Bernstein paraboline variety}. When $P=B$, one can check that $X_{\Omega, \textbf{h}}(\overline{\rho}_{\wp})$ is isomorphic to the trianguline variety $X_{\tri}^{\square}(\overline{\rho})$ of \cite{BHS1} (see Remark \ref{remBPE} (1)). It also has similar geometric properties as $X_{\tri}^{\square}(\overline{\rho})$ (compare the following theorem with \cite[Thm.\ 2.6]{BHS1}):

\begin{theorem}[cf.\ Theorem \ref{DFOL}]\label{triangulineintro}
	(1) The rigid analytic space $X_{\Omega, \textbf{h}}(\overline{\rho}_{\wp})$ is equidimensional of dimension
	\[n^2 + [F_{\wp}:\Q_p] \Big(\frac{n(n-1)}{2}+r\Big).\]
	
	(2) The set $U_{\Omega, \textbf{h}}(\overline{\rho}_{\wp})$ is Zariski-open and Zariski-dense in $X_{\Omega, \textbf{h}}(\overline{\rho}_{\wp})$.
	
	(3) The rigid analytic space $U_{\Omega, \textbf{h}}(\overline{\rho}_{\wp})$ is smooth over $E$, and the natural morphism $U_{\Omega, \textbf{h}}(\overline{\rho}_{\wp}) \ra (\Spec \cZ_{\Omega})^{\rig} \times \widehat{\cZ_0}$ is smooth.
\end{theorem}

Note that, differently from the trianguline case, the proof of Theorem \ref{triangulineintro} crucially uses Kisin's results on potentially crystalline deformation rings (\cite{Kis08}).

Similarly as in \cite{BHS1}, the Bernstein paraboline varieties are closely related to a {\it patched} version of the Bernstein eigenvarieties. Recall that under Hypothesis \ref{TayWil0}, it was constructed in \cite{CEGGPS} a $p$-adic unitary Banach space representation $\Pi_{\infty}$ of $\GL_n(F_{\wp})$ equipped with an action of a certain patched Galois deformation ring $R_{\infty} \cong R_{\infty}^p \widehat{\otimes}_{\co_E} R_{\overline{\rho}_{\wp}}$ (commuting with $\GL_n(F_{\wp})$). There is an ideal $\fa$ of $R_{\infty}$ such that one has a surjection $R_{\infty}/\fa \twoheadrightarrow R_{\overline{\rho},\cS}$ and $\Pi_{\infty}[\fa]\cong \widehat{S}(U^p,E)_{\overline{\rho}}$. Applying the construction of Bernstein eigenvarieties with $\widehat{S}(U^p,E)$ replaced by $\Pi_{\infty}$, we can construct an embedding of rigid analytic spaces 
\begin{multline}\label{paBE0}
\iota_{\Omega, \lambda}: \cE_{\Omega, \lambda}^{\infty}(\overline{\rho})\hooklongrightarrow (\Spf R_{\infty})^{\rig}\times (\Spec \cZ_{\Omega})^{\rig} \times \widehat{\cZ_0} \\
	\cong (\Spf R_{\infty}^p)^{\rig} \times (\Spf R_{\overline{\rho}_{\wp}})^{\rig} \times (\Spec \cZ_{\Omega})^{\rig}\times \widehat{\cZ_0}
	\end{multline}
which satisfies the properties in the following theorem:

\begin{theorem}[cf.\ \S~\ref{secPBern}]\label{intpBE}
(1) A point $(\fm_y, \pi_{L_P},\chi)\in (\Spf R_{\infty})^{\rig} \times(\Spec \cZ_{\Omega})^{\rig} \times \widehat{\cZ_0}$ lies in $\cE_{\Omega, \lambda}^{\infty}(\overline{\rho})$ if and only if one has an $L_P(F_{\wp})$-equivariant embedding
	\[\pi_{L_P}\otimes_E (\chi_{\varpi} \circ \dett_{L_P}) \otimes_E L(\lambda)_P \hooklongrightarrow J_P(\Pi_{\infty}^{R_{\infty}-\an})[\fm_y]\]	
	where $\Pi_{\infty}^{R_{\infty}-\an}$ denotes the $R_{\infty}$-analytic vectors in the sense of \cite[\S~3.1]{BHS1}.
	
(2) The rigid analytic space $\cE_{\Omega, \lambda}^{\infty}(\overline{\rho})$ is reduced and equidimensional of dimension
	\[\dim (\Spf R_{\infty}^p)^{\rig}+ n^2+[F_{\wp}:\Q_p]\Big(\frac{n(n-1)}{2}+r\Big).\]
	
(3) There is a natural morphism of rigid spaces
\[\cE_{\Omega, \lambda}(U^p) \lra \cE_{\Omega, \lambda}^{\infty}(\overline{\rho}) \times_{(\Spf R_{\infty})^{\rig}} (\Spf R_{\overline{\rho},\cS})^{\rig}\]
which is bijective on points (where the morphism $ (\Spf R_{\overline{\rho},\cS})^{\rig} \ra (\Spf R_{\infty})^{\rig}$ is induced by $R_{\infty}/\fa \twoheadrightarrow R_{\overline{\rho},\cS}$). 
	
(4) The embedding (\ref{paBE0}) factors through
	\begin{equation}\label{R->T}
		\cE_{\Omega, \lambda}^{\infty}(\overline{\rho})\hooklongrightarrow(\Spf R_{\infty}^p)^{\rig} \times \jmath(X_{\Omega, \textbf{h}}(\overline{\rho}_{\wp}))
	\end{equation}
	(where $\jmath$ is a certain shift of the natural embedding, see (\ref{RTnorm})) and induces an isomorphism between $\cE_{\Omega, \lambda}^{\infty}(\overline{\rho})$ and a union of irreducible components of $(\Spf R_{\infty}^p)^{\rig} \times \jmath(X_{\Omega, \textbf{h}}(\overline{\rho}_{\wp}))$ equipped with the reduced closed rigid subspace structure. 
\end{theorem}

To gain a better understanding of the embedding (\ref{R->T}), we are led to study the local geometry of Bernstein paraboline varieties. As in the trianguline case (when $P=B$), it is closely related to certain schemes appearing in parabolic generalizations of Grothendieck's and Springer's resolution of singularities. We now define these schemes. For a closed algebraic subgroup $H$ of $\GL_n$ (seen over $F_{\wp}$), put $H_{\wp}:=(\Res_{\Q_p}^{F_{\wp}} H) \times_{\Spec \Q_p} \Spec E$. Likewise, for a Lie subalgebra $\fh$ of the Lie algebra $\gl_n$ over $F_{\wp}$, put $\fh_{\wp}:=\fh \otimes_{\Q_p} E$. Let $\tilde{\ug}_{\wp}$ (resp.\ $\tilde{\ug}_{P,\wp}$) be the closed $E$-subscheme of $(\GL_{n,\wp}/B_{\wp}) \times \gl_{n,\wp}$ (resp.\ of $(\GL_{n,\wp}/P_{\wp}) \times \gl_{n,\wp}$) defined by:
\begin{equation*}
	\{(gB_{\wp}, \psi) \in (\GL_{n,\wp}/ B_{\wp}) \times \gl_{n,\wp} \ |\ \Ad(g^{-1}) \psi \in \ub_{\wp}\}
\end{equation*}
\begin{equation*}
	\big(\text{resp.\ }\{(gP_{\wp}, \psi) \in (\GL_{n,\wp}/ P_{\wp}) \times \gl_{n,\wp} \ |\ \Ad(g^{-1}) \psi \in \ur_{P,\wp}\}\big),
\end{equation*}
where $\ub\subset \gl_n$ (resp.\ $\ur_P\cong \fn_P \rtimes \fz_{L_P}$) is the Lie algebra of $B$ (resp.\ of the full radical of $P\cong N_P\rtimes L_P$, where $\fz_{L_P}$ is the Lie algebra of $Z_{L_P}$). Put 
\begin{equation*}
	X_{P,\wp}:=\tilde{\ug}_{P,\wp} \times_{\gl_{n,\wp}}\tilde{\ug}_{\wp}.
\end{equation*}
We define a morphism $\pi$ as the composition:
\begin{multline*}
\pi: X_{P,\wp} \hooklongrightarrow (\GL_{n,\wp}/P_{\wp}) \times (\GL_{n,\wp}/B_{\wp}) \times \gl_{n,\wp} \twoheadlongrightarrow (\GL_{n,\wp}/P_{\wp}) \times (\GL_{n,\wp}/B_{\wp}) \\
=\amalg_{w\in \sW_{P,F_{\wp}}^{\max}} \GL_{n,\wp} (1, w) (P_{\wp} \times B_{\wp})
\end{multline*}
where the second map is the canonical projection. For $w\in \sW_{P,F_{\wp}}^{\max}$, let $U_w:=\GL_{n,\wp} (1, w) (P_{\wp} \times B_{\wp})$ and $X_w$ the reduced closed subscheme of $X_{P,\wp}$ defined as the Zariski-closure of $\pi^{-1}(U_w)$. One can show that $\{X_w\}_{w\in \sW_{P,F_{\wp}}^{\max}}$ are the irreducible components of $X_{P,\wp}$ and that $X_{P,\wp}$ is equidimensional of dimension 
\begin{equation*}
	[F_{\wp}:\Q_p]\frac{n(n-1)}{2}+\dim \ur_{P,\wp}=[F_{\wp}:\Q_p]\Big(\frac{n(n-1)}{2}+r\Big)+\dim \fn_{P,\wp}.
\end{equation*}
The following geometric property of $X_w$ is particularly important for our applications:

\begin{theorem}[cf.\ Theorem \ref{unibranch}]\label{intuni}
	Let $w\in \sW_{P,F_{\wp}}^{\max}$ and $x=(g_1P_{\wp}, g_2B_{\wp}, 0)\in X_w\!\hookrightarrow\! (\GL_{n,\wp}/P_{\wp}) \times (\GL_{n,\wp}/B_{\wp}) \times \gl_{n,\wp}$, then the scheme $X_w$ is unibranch at $x$.
\end{theorem}

When $P=B$, it was showed in \cite[Thm.\ 2.3.6]{BHS3} that the whole scheme $X_w$ is normal. When $P\neq B$, the geometry of $X_w$ appears seriously more difficult and we don't know if $X_w$ is normal. The proof of Theorem \ref{intuni} is inspired by a result of Le, Le Hung, Morra and Levin in the setting of deformation rings (\cite[Lemma 3.4.8]{LLHLM}) and {\it a priori} only works for those points $x\in X_w$ with $\gl_{n,\wp}$-factor zero (as in Theorem \ref{intuni}).

Let $x=(\varrho,(x_i), (\chi_i))\in X_{\Omega, \textbf{h}}(\overline{\rho}_{\wp})$ such that $\varrho$ is almost de Rham (in the sense of \cite{Fo04}) with distinct Sen weights. We fix an $F_{\wp} \otimes_{\Q_p} k(x)$-linear isomorphism $\alpha: F_{\wp}\otimes_{\Q_p} k(x) \xrightarrow{\sim} D_{\pdR}(\varrho)=(B_{\pdR} \otimes_{\Q_p} \varrho)^{\Gal_{F_{\wp}}}$ where $B_{\pdR}$ is the ring of \cite[\S~4.3]{Fo04} (see also \S~\ref{sec61}). By an analogue of Theorem \ref{iThm2} (1) for the variety $X_{\Omega, \textbf{h}}(\overline{\rho}_{\wp})$ instead of $\cE_{\Omega, \lambda}(U^p,\overline{\rho})$, the $(\varphi,\Gamma)$-module $D_{\rig}(\varrho)[\frac{1}{t}]$ over $\cR_{k(x),F_{\wp}}[\frac{1}{t}]$ admits a $P$-filtration $\sF_{\varrho}$ which induces a $P$-filtration (still denoted) $\sF_{\varrho}$ on $D_{\pdR}(\varrho)\cong \big(B_{\pdR} \otimes_{B_{\dR}} W_{\dR}(D_{\rig}(\varrho)[\frac{1}{t}])\big)^{\Gal_{F_{\wp}}}$ (see for instance \cite[Lemma 3.3.5 (ii)]{BHS3} for $W_{\dR}(-)$, see also \S~\ref{sec62}). We use $\Fil_\varrho$ to denote the complete flag given by the Hodge filtration on $D_{\pdR}(\varrho)$. Using the framing $\alpha$, the filtration $\sF_{\varrho}$ (resp.\ $\Fil_{\varrho}$) corresponds to an element still denoted by $\sF_{\varrho}$ (resp.\ $\Fil_{\varrho}$) in $\GL_{n,\wp}/P_{\wp}$ (resp. in $\GL_{n,\wp}/B_{\wp}$). Finally, let $\nu_{\pdR}$ be Fontaine's nilpotent operator on $D_{\pdR}(\varrho)$. As both $\sF_{\varrho}$ and $\Fil_{\varrho}$ are stabilized by $\nu_{\pdR}$, we have
\begin{equation}\label{intypdR}
	y_{\pdR}:=(\sF_{\varrho}, \Fil_{\varrho}, \nu_{\pdR})\in X_{P,\wp}\hookrightarrow (\GL_{n,\wp}/P_{\wp}) \times (\GL_{n,\wp}/B_{\wp}) \times \gl_{n,\wp}.
\end{equation}
There exists thus $w_{\sF_{\varrho}}\in \sW_{P,F_\wp}^{\max}$ such that $y_{\pdR}\in \pi^{-1}(U_{w_{\sF_{\varrho}}})$. In fact, $w_{\sF_{\varrho}}$ is independent of the choice of $\alpha$ and is equal to the element $w_{\sF}$ defined above Theorem \ref{intcs} if $\varrho=\rho_{\wp}$ and $\sF_{\varrho}=\sF$ (note that we have $\nu_{\pdR}=0$ in this case).

By an analogue of Theorem \ref{iThm2} (2), one can show that for each $\tau: F_{\wp}\hookrightarrow E$, the set of $\tau$-Sen weights of $\varrho$ is given by $\{k_{x,j,\tau}:=h_{j,\tau}+\wt(\chi_{i(j)})_{\tau}\}_{j=1,\dots, n}$ where $i(j)\in \{1,\dots, r\}$ is such that $s_{i(j)-1}<j\leq s_{i(j)}$. There exist then a unique element $w_x=(w_{x,\tau})\in \sW_{P,F_{\wp}}^{\min}$ and a strictly dominant weight $\textbf{h}_x$ (which is in fact given by the (decreasing) Sen weights of $\varrho$) such that $\textbf{k}_x=w_x(\textbf{h}_x)$ where $\textbf{k}_x:=\{k_{x,j,\tau}\}_{\substack{j=1,\dots, n\\ \tau: F_{\wp} \hookrightarrow E}}$. 

By generalizing \cite[\S~3]{BHS3}, we have:

\begin{theorem}[cf.\ \S~\ref{secGrDV}]\label{intLocM}
	Keep the above notation and assume $\varrho$ is de Rham (equivalently $\nu_{\pdR}=0$) and $\sF_{\varrho}$ is generic (in the sense of (\ref{geneMf})). 
	
	(1) We have $w_{\sF}\leq w_xw_0$.
	
	(2) There exists a formal scheme $X_{\varrho, \cM_{\bullet}}^{\square, w_x}$ over $E$ such that the associated reduced formal scheme $(X_{\varrho, \cM_{\bullet}}^{\square, w_x})^{\red}$ is formally smooth of dimension $n^2+\dim \fp_{\wp}$ over the completion $\widehat{X}_{w_x,y_{\pdR}}$ of $X_{w_x}$ at $y_{\pdR}$, and formally smooth of dimension $n^2[F_{\wp}:\Q_p]$ over the completion $\widehat{X_{\Omega, \textbf{h}}(\overline{\rho}_{\wp})}_x$ of $X_{\Omega, \textbf{h}}(\overline{\rho}_{\wp})$ at $x$:	
	\begin{equation*}
		\widehat{X_{\Omega, \textbf{h}}(\overline{\rho}_{\wp})}_x \longleftarrow (X_{\varrho, \cM_{\bullet}}^{\square, w_x})^{\red} \longrightarrow \widehat{X}_{w_x,y_{\pdR}}.
	\end{equation*}
	In particular, $X_{\Omega, \textbf{h}}(\overline{\rho}_{\wp})$ is unibranch at the point $x$.
\end{theorem}

By Theorem \ref{intLocM} (2) and Theorem \ref{intpBE} (4), we deduce the following ``$R=T$"-type result:

\begin{corollary}
	Let $x=(\fm^p, \varrho, \pi_{L_P}, \chi)\in \cE_{\Omega, \lambda}^{\infty}(\overline{\rho}) \hookrightarrow (\Spf R_{\infty}^p)^{\rig} \times (\Spf R_{\overline{\rho}_{\wp}})^{\rig} \times (\Spec \cZ_{\Omega})^{\rig} \times \widehat{\cZ_0}$. Assume that $\fm^p$ is a smooth point of $(\Spf R_{\infty}^p)^{\rig}$ and that $\varrho$ is generic potentially crystalline with distinct Hodge-Tate weights. Then the embedding (\ref{R->T}) induces an isomorphism after taking completions at $x$. 
\end{corollary}

We now discuss the problem of companion points on Bernstein eigenvarieties (resp.\ on patched Bernstein eigenvarieties, resp.\ on Bernstein paraboline varieties), which will be crucial to attack the socle conjecture. Let $y$ be a point of $(\Spf R_{\overline{\rho},\cS})^{\rig}$ (resp.\ of $(\Spf R_{\infty})^{\rig}$, resp.\ of $(\Spf R_{\overline{\rho}_{\wp}})^{\rig}$), and $\varrho$ be the $\Gal_{F_{\wp}}$-representation associated to $y$. We assume $\varrho$ is generic potentially crystalline with distinct Hodge-Tate weights. We let $\textbf{h}$ be the (decreasing) Hodge-Tate weights of $\varrho$ and $\lambda=(\lambda_{i,\tau})_{\substack{i=1,\dots,n\\ \tau: F_{\wp} \hookrightarrow E}}$ with $\lambda_{i,\tau}=h_{i,\tau}+i-1$ (so $\lambda$ is dominant with respect to $B$). Assume $\ttr(\varrho)\cong \oplus_{i=1}^r \ttr_{x_i}$ with $\ul{x}=(x_i)\in (\Spec \cZ_{\Omega})^{\rig}$. Note that, as $\varrho$ is generic, the $P$-filtration $\ttr_{x_1}\subset \ttr_{x_1}\oplus \ttr_{x_2}\subset \cdots$ on $\ttr(\varrho)$ corresponds to a unique $\Omega$-filtration on $D_{\rig}(\varrho)$, see \S~\ref{introPcr}. Consider the point $x:=(y, \ul{x},1)$ in $(\Spf R_{\overline{\rho},\cS})^{\rig} \times (\Spec \cZ_{\Omega})^{\rig} \times \widehat{\cZ_0}$ \big(resp.\ in $(\Spf R_{\infty}^p)^{\rig} \times (\Spec \cZ_{\Omega})^{\rig} \times \widehat{\cZ_0}$, resp.\ in $(\Spf R_{\overline{\rho}_{\wp}})^{\rig}\times (\Spec \cZ_{\Omega})^{\rig} \times \widehat{\cZ_0}$\big). As usual, we denote by $w\cdot \mu$ the dot action on a weight $\mu$.

\begin{conjecture}\label{intccpt}
	Let $w\in \sW_{\min,F_{\wp}}^P$, then $x\in \cE_{\Omega, w\cdot \lambda}(U^p,\overline{\rho})$ \big(resp.\ $x\in \cE_{\Omega, w\cdot \lambda}^{\infty}(\overline{\rho})$, resp.\ $x\in X_{\Omega, w(\textbf{h})}(\overline{\rho}_{\wp})$\big) if and only if $w w_0\geq w_{\sF_{\varrho}}$.
\end{conjecture}
As $w\in \sW^P_{\min,F_{\wp}}$, $w(\textbf{h})$ is strictly $P$-dominant and $w\cdot \lambda$ is $P$-dominant, so the corresponding rigid spaces in Conjecture \ref{intccpt} are well defined. The reader who is familiar with companion points in the trianguline case may find the statement of Conjecture \ref{intccpt} a little strange. Indeed, for the classical eigenvariety $\cE(U^p,\overline{\rho})$ (the case of the patched eigenvariety or of the trianguline variety being similar), there is a canonical embedding $\cE(U^p,\overline{\rho}) \hookrightarrow (\Spf R_{\overline{\rho},\cS})^{\rig} \times \widehat{T(F_{\wp})}$ and the companion points are the distinct points that lie above a same point $y\in (\Spf R_{\overline{\rho},\cS})^{\rig}$. In our case however, as there are different rigid spaces depending on $(\Omega, \lambda)$, it seems more convenient to {\it fix} the point $x \in (\Spf R_{\overline{\rho},\cS})^{\rig} \times (\Spec \cZ_{\Omega})^{\rig} \times \widehat{\cZ_0}$ and let the Bernstein eigenvarieties (together with the embedding $\iota_{\Omega, \lambda}$) {\it vary}. See Remark \ref{cptPBE} and Remark \ref{remNPara} for more details.

By Theorem \ref{intLocM} (1) and a study of the relation between potentially crystalline deformation spaces and Bernstein paraboline varieties (cf.\ \S~\ref{secPCD}), we can prove Conjecture \ref{intccpt} for Bernstein paraboline varieties:

\begin{theorem}[cf.\ Corollary \ref{coLoccomp}]\label{intcpt2}
	Let $w\in \sW_{\min,F_{\widetilde{v}}}^P$, then $x\in X_{\Omega, w(\textbf{h})}(\overline{\rho}_{\wp})$ if and only if $w w_0\geq w_{\sF_{\varrho}}$.
\end{theorem}

Using Theorem \ref{intpBE} (3) and Theorem \ref{intpBE} (4), Theorem \ref{intcpt2} then implies the ``only if" part of Conjecture \ref{intccpt}. 

We now move back to the global applications in Theorem \ref{intcl} and Theorem \ref{intcs}, both of which are about irreducible constituents in the socle of $\widehat{S}(U^p,E)^{\an}[\fm_{\rho}]$. Let $y$ be the image of $\fm_{\rho}$ via $(\Spf R_{\overline{\rho},\cS})^{\rig} \hookrightarrow (\Spf R_{\infty})^{\rig}$. Let $\fm_y\subset R_{\infty}[1/p]$ be the associated maximal ideal, then $\widehat{S}(U^p,E)^{\an}[\fm_{\rho}]\cong \Pi_{\infty}^{R_{\infty}-\an}[\fm_y]$. In particular, to prove Theorem \ref{intcl} and Theorem \ref{intcs}, it suffices to show the same statement with $\Pi_{\infty}^{R_{\infty}-\an}[\fm_y]$ instead of $\widehat{S}(U^p,E)^{\an}[\fm_{\rho}]$. Let $\sF$ be a $P$-filtration on $\ttr(\rho_{\wp})$ as in the discussion above Theorem \ref{intcs}. As before let $\textbf{h}$ be the (decreasing) Hodge-Tate weights of $\rho_{\wp}$ and $\lambda=(\lambda_{i,\tau})$ with $\lambda_{i,\tau}=h_{i,\tau}+i-1$. Let $\Omega$ be a Bernstein component of $L_P(F_{\wp})$ and $\ul{x}=(x_i)\in (\Spec \cZ_{\Omega})^{\rig}$ such that $\ttr_{x_i}\cong \ttr_i$ for $i=1,\cdots, r$ (so $\Omega$ and $\ul{x}$ are determined by $\rho$ and $\sF$). We then associate to $\rho$ and the $P$-filtration $\sF$ on $\ttr(\rho_{\wp})$ a point $x_{\sF}=(y,\ul{x},1)\in (\Spf R_{\infty})^{\rig} \times (\Spec \cZ_{\Omega})^{\rig} \times \widehat{\cZ_0}$. From the locally analytic representation theory, for each $w\in \sW_{P,F_{\wp}}^{\min}$, we can construct a cycle $[\cN_{w\cdot \lambda,y,\sF}]\in Z^{[F^+:\Q]\frac{n(n+1)}{2}}(\Spec \widehat{\co}_{\fX_{\infty},y})$, where $Z^{d}(-)$ is the free abelian group generated by the irreducible closed subschemes of codimension $d$ and $\widehat{\co}_{\fX_{\infty},y}$ is the completion of $\fX_{\infty}:=(\Spf R_{\infty})^{\rig}$ at the point $y$, such that:
\begin{itemize}
	\item $[\cN_{w\cdot \lambda,y,\sF}]\neq 0$ if and only if $C(w,\sF)$ embeds into $\Pi_{\infty}^{R_{\infty}-\an}[\fm_y]$.
\end{itemize}
It is also not difficult to prove the implication $[\cN_{w\cdot \lambda,y,\sF}]\neq 0\Rightarrow x_{\sF}\in \cE_{\Omega, w\cdot \lambda}^{\infty}(\overline{\rho})$. It turns out that this implication is in fact an equivalence (which then implies that, in this case, the existence of companion points is equivalent to the existence of companion constituents):

\begin{proposition}[cf.\ Proposition \ref{nonVancyc}]\label{inteqpt}
	For $w\in \sW_{P,F_{\wp}}^{\min}$, $[\cN_{w\cdot \lambda,y,\sF}]\neq 0$ if and only if $x_{\sF}\in \cE_{\Omega, w\cdot \lambda}^{\infty}(\overline{\rho})$.
\end{proposition}

We first discuss the proof of Theorem \ref{intcl}. The assumption (2) in {\it loc.\ cit.}\ guarantees that there exist a parabolic subgroup $P\supseteq B$ of $\GL_n$ and a $P$-filtration $\sF$ on $\ttr(\rho_{\wp})$ such that the above associated point $x_{\sF}=(y,\ul{x},1)\in (\Spf R_{\infty})^{\rig} \times (\Spec \cZ_{\Omega})^{\rig} \times \widehat{\cZ_0}$ lies in a certain patched Bernstein eigenvariety $\cE_{\Omega, w\cdot \lambda}^{\infty}(\overline{\rho})$ for some $w\in \sW_{P,F_{\wp}}^{\min}$. Using Proposition \ref{inteqpt} (and a bit of representation theory), one can deduce $x_{\sF}\in \cE_{\Omega, \lambda}^{\infty}(\overline{\rho})$. Then the classicality follows by applying Proposition \ref{inteqpt} to $w=1$ and using that $C(1,\sF)$ is locally algebraic.

We now discuss the proof of Theorem \ref{intcs}. We henceforth fix a filtration $\sF$ and write $x:=x_{\sF}$, $[\cN_{w\cdot \lambda, y}]:=[\cN_{w \cdot \lambda, y,\sF}]$. Note that, as $\Pi_{\infty}^{R_{\infty}-\an}[\fm_y]^{\lalg}\cong \widehat{S}(U^p,E)[\fm_{\rho}]^{\lalg}$ is non-zero and isomorphic to a direct sum of copies of $C(1,\sF)$, we have $[\cN_{\lambda, y}]\neq 0$ and $x\in \cE_{\Omega, \lambda}^{\infty}(\overline{\rho})$. The ``only if" part (with no assumption on $P$) already follows from Proposition \ref{inteqpt} and Theorem \ref{intcpt2}. So we need to show $x\in \cE_{\Omega, w\cdot \lambda}^{\infty}(\overline{\rho})$ for $w\in \sW_{P,F_{\wp}}^{\min}$ such that $ww_0\geq w_{\sF}$. Let $\lg(-)$ denotes the length function on Weyl group. The case where $\lg(w_{\sF})\geq \lg(w_0)-1$ is not very difficult (and holds for any $P$). We assume in the sequel $\lg(w_{\sF})\leq \lg(w_0)-2$. By induction and some similar arguments as in the proof of the ``if" part of Theorem \ref{intcpt2}, one is reduced to showing the following statement:
\begin{itemize}
	\item if $x\in \cE_{\Omega, w'w_0\cdot \lambda }^{\infty}(\overline{\rho})$ for all $w'\in \sW_{P,F_\wp}^{\max}$ such that $w'>w_{\sF}$, then $x\in \cE_{\Omega, w_{\sF}w_0\cdot \lambda}^{\infty}(\overline{\rho})$.
\end{itemize}
Thus, assuming $x\in \cE_{\Omega, w'w_0\cdot \lambda }^{\infty}(\overline{\rho})$ for all $w'\in \sW_{P,F_\wp}^{\max}$, $w'>w_{\sF}$, we need to show $[\cN_{w_{\sF}w_0\cdot \lambda,y}]\neq 0$. One important fact is that the cycles $[\cN_{w\cdot \lambda,y}]$ can be related to cycles coming from irreducible components of certain generalized Steinberg varieties. Let $Z_{P,\wp}$ be the fibre of $X_{P,\fp}$ at $0\in \fz_{L_P, \wp}$ via
\begin{equation*}
	X_{P,\wp} \lra \fz_{L_P, \wp}, \ (g_1 P_{\wp}, g_2 B_{\wp}, \psi)\longmapsto \overline{\Ad(g_1^{-1})\psi}
\end{equation*}
where $\overline{(-)}$ means the natural projection $\ur_{P,\wp} \twoheadrightarrow \fz_{L_P, \wp}$. One can show that $Z_{P,\wp}$ is equidimensional with (reduced) irreducible components given by $\{Z_w:=(Z_{P,\wp} \cap X_w)^{\red}\}_{w\in \sW_{P,F_{\wp}}^{\max}}$. Let $y_{\pdR}$ be the point of $X_{P,\wp}$ associated to $\rho_{\wp}$ and $\sF$ as in (\ref{intypdR}). As $\nu_{\pdR}=0$, $y_{\pdR}\in Z_{P,\wp}$. Similarly as in Theorem \ref{intuni}, one can prove that for $w\in \sW_{P,F_{\wp}}^{\max}$, if $y_{\pdR}\in Z_w$, then $Z_w$ is unibranch at $y_{\pdR}$ (cf.\ Theorem \ref{unibranch2}). Using Theorem \ref{intLocM} (2), for $w\in \sW_{P,F_{\wp}}^{\max}$, one can then associate to the completion $\widehat{\co}_{Z_w, y_{\pdR}}$ of $Z_w$ at $y_{\pdR}$ a unique irreducible cycle $[\fZ_{w,y}]\in Z^{[F^+:\Q]\frac{n(n+1)}{2}}(\Spec \widehat{\co}_{\fX_{\infty},y})$. We have 
\begin{itemize}
	\item $\fZ_{w,y}\neq 0$ if and only if $w\geq w_{\sF}$. 
\end{itemize}
By results on the characteristic cycles associated to {\it generalized} Verma modules (that we couldn't really find in the literature and that we prove in \S~\ref{appenB}), we have the following statements:
\begin{itemize}
	\item $[\cN_{w_{\sF}w_0\cdot \lambda,y}]\in \Z_{\geq 0} \fZ_{w_{\sF},y}$;
	\item if $x$ is a {\it smooth} point of $\cE_{\Omega, ww_0\cdot \lambda}^{\infty}(\overline{\rho})$ for $w\in \sW_{P,F_{\wp}}^{\min}$, there is an integer $m_y\in \Z_{\geq 1}$ such that in $Z^{[F^+:\Q]\frac{n(n+1)}{2}}(\Spec \widehat{\co}_{\fX_{\infty},y})$:
	\begin{equation}\label{inteqc}
		\sum_{\substack{w'\in \sW_{P,F_{\wp}}^{\min}\\ w\leq w'\leq w_{\sF} w_0}} b_{w,w'} [\cN_{w' \cdot \lambda, y}] = 	m_y\Bigg(\sum_{\substack{w'\in \sW_{P,F_{\wp}}^{\min}\\ w\leq w'\leq w_{\sF} w_0}} b_{w,w'} \fZ_{w'w_0,y}\Bigg)
	\end{equation}
	where $b_{w,w'}$ is the multiplicity of the simple $\text{U}(\ug_{\wp})$-module $L(w'\cdot 0)$ of highest weight $w' \cdot 0$ in the parabolic Verma module $M_P(w \cdot 0)$ of highest weight $w\cdot 0$.
\end{itemize}
If $x$ is not smooth, we also have an equation similar to (\ref{inteqc}) but with the coefficients $m_yb_{w,w'}$ on the right hand side replaced by certain non-negative integers that we do not know how to control (the left hand side staying unchanged). 

Let us assume $m_y=1$ in the rest of the argument for simplicity. To use these equations to deduce $[\cN_{w_{\sF}w_0 \cdot \lambda,y}]\neq 0$, we are led to three cases (recall we have assumed $\lg(w_{\sF})\leq \lg(w_0)-2$), with the third case still resisting without a further assumption on $P$:

Case (1) is exactly the same as in \cite{BHS3} (in particular we are always in this case when $P=B$): assume that there exist $w_1$, $w_2$, $w\in \sW_{P,F_{\wp}}^{\max}$ such that $w\geq w_{\sF} $, $\lg(w)=\lg(w_{\sF})+2$, $\{w_1, w_2\}=[w_{\sF}, w]:=\{w'\ |\ w_{\sF}<w'<w\}$ and $\dim \fz_{L_P, \wp}^{w_{\sF}w^{-1}}=\dim \fz_{L_P, \wp}-2$. Under these conditions, by a tangent space argument, one can prove that $\cE_{\Omega, w'w_0\cdot \lambda }^{\infty}(\overline{\rho})$ is smooth at $x$ for $w'\in \{w_1,w_2, w\}$. We can deduce from (\ref{inteqc}) equalities:
\begin{equation}\label{inteqc0}
	\begin{cases}
		[\cN_{w_1w_0\cdot \lambda,y}]+[\cN_{w_{\sF}w_0 \cdot \lambda,y}]=\fZ_{w_1, y}+\fZ_{w_{\sF},y} \\
		[\cN_{w_2w_0\cdot \lambda,y}]+[\cN_{w_{\sF}w_0 \cdot \lambda,y}]=\fZ_{w_2,y}+\fZ_{w_{\sF},y} \\
		[\cN_{ww_0\cdot \lambda,y}]+[\cN_{w_1w_0 \cdot \lambda,y}]+[\cN_{w_2w_0 \cdot \lambda, y}]+[\cN_{w_{\sF} w_0\cdot \lambda,y}]=\fZ_{w,y}+\fZ_{w_1,y}+\fZ_{w_2,y}+\fZ_{w_{\sF},y}.
	\end{cases}
\end{equation}
Using that $\fZ_{w'',y}$ can only have non-negative coefficients in $[\cN_{w'w_0\cdot \lambda,y}]$ for $w', w''\in \sW_{P,F_{\wp}}^{\max}$ and $[\cN_{w_{\sF}w_0\cdot \lambda,y}]\in \Z_{\geq 0} \fZ_{w_{\sF},y}$, it is not difficult to deduce from the equalities in (\ref{inteqc0}) that $[\cN_{w_{\sF} w_0 \cdot \lambda,y}]\neq 0$\footnote{We recommend the interested reader to work it out as an exercise.}.

Case (2): assume that there exists $w>w_{\sF}$ with $\lg(w)=\lg(w_{\sF})+2$ such that there is a unique $w_1\in \sW_{P,F_{\wp}}^{\max}$ such that $w_{\sF}<w_1<w$ (in other words, the Bruhat interval $[w_{\sF},w]$ is not full in the quotient $\sW_{L_P,F_{\wp}}\backslash \sW_{F_{\wp}}$). In this case, we deduce from (\ref{inteqc}) equalities:
\begin{equation*}
	\begin{cases}
		[\cN_{w_1w_0\cdot \lambda,y}]+[\cN_{w_{\sF}w_0 \cdot \lambda,y}]=\fZ_{w_1, y}+\fZ_{w_{\sF},y} \\
		[\cN_{ww_0\cdot \lambda,y}]+[\cN_{w_1w_0 \cdot \lambda,y}]=a_0\fZ_{w,y}+a_1\fZ_{w_1,y}
	\end{cases}
\end{equation*}
for some $a_i\in \Z_{\geq 0}$. Though we don't have more control on the $a_i$, these equalities are (again) sufficient to imply $[\cN_{w_{\sF} w_0 \cdot \lambda,y}]\neq 0$. 

Case (3): assume that, for any $w\leq w_{\sF} w_0$ with $\lg(w)=\lg(w_{\sF}w_0)-2$, we have $\dim \fz_{L_P, \wp}^{w_{\sF}w_0 w^{-1}}>\dim \fz_{L_P,\wp}-2$ and there exist $w_1,w_2\in \sW_{P,F_{\wp}}^{\min}$ such that $\{w'\ |\ w<w'<w_{\sF}w_0\}=\{w_1,w_2\}$. The main difference with Case (1) is that we do not know if $\cE_{\Omega, ww_0\cdot \lambda}^{\infty}(\overline{\rho})$ is smooth at the point $x$ (the tangent space argument collapses because of $\dim \fz_{L_P, \wp}^{w_{\sF}w_0 w^{-1}}>\dim \fz_{L_P,\wp}-2$). 
Consequently, the third equation in (\ref{inteqc0}) has to be replaced by an equation of the form (the two others staying unchanged)
\begin{equation*}
	[\cN_{ww_0\cdot \lambda,y}]+[\cN_{w_1w_0 \cdot \lambda,y}]+[\cN_{w_2w_0 \cdot \lambda, y}]+[\cN_{w_{\sF}w_0 \cdot \lambda,y}]=a_0\fZ_{w,y}+a_1\fZ_{w_1,y}+a_2\fZ_{w_2,y}+a_3\fZ_{w_{\sF},y}
\end{equation*}
for some $a_i\in \Z_{\geq 0}$. Without more control on these coefficients $a_i$, this equation together with the first two in (\ref{inteqc0}) seem not enough to imply $[\cN_{w_{\sF}w_0\cdot \lambda,y}]\neq 0$. 

The assumption ($*$) in Theorem \ref{intcs} is there precisely to avoid Case (3) (cf.\ Proposition \ref{bruhInt}, see also Remark \ref{bruhInv} for an example of Case (3) for $\GL_4$). 

One may expect other arithmetic applications of Bernstein eigenvarieties. In fact, the results in this work provide a framework to which many arguments for classical eigenvarieties may be adapted (as what we already do in this paper). In a forthcoming work \cite{He21} of Yiqin He, Bernstein eigenvarieties are used to establish some local-global compatibility results on simple $\cL$-invariants for certain $\GL_n$-representations attached to Zelevinsky's linked segments (which was previously only known in the trianguline case). 

Finally, we remark that in his PhD.\ thesis \cite{Huang}, Shanxiao Huang proves results that parallel Theorem \ref{centerintro}, Theorem \ref{iThm2}, Theorem \ref{triangulineintro} and the global analogue of part (4) of Theorem \ref{intpBE} (i.e.\ a version without patched objects).

\subsection*{Acknowledgement}

This work would not exist without the work of Eugen Hellmann and Benjamin Schraen in \cite{BHS1}, \cite{BHS2}, \cite{BHS3}, and it is a pleasure to acknowledge this.
The authors thank Tsao-Hsien Chen, Lucas Fresse, Victor Ginzburg, Eugen Hellmann, Wen-Wei Li, Ruochuan Liu, Anne Moreau, Alexandre Pyvovarov, Simon Riche, Peng Shan, Benjamin Schraen, Changjian Su, Julianna Tymoczko, Zhixiang Wu and Liang Xiao for discussions or answers to questions.
Special thanks to Liang Xiao for many helpful discussions on \S~\ref{globaltriangulation}, and to Anne Moreau for many discussions and for her help on the singularities of Springer fibers.

C.\ B.\ is supported by the C.N.R.S\ and is a member of the A.N.R.\ project CLap-CLap ANR-18-CE40-0026. Y.\ D.\ is supported by the NSFC Grant No.\ 8200905010 and No.\ 8200800065. 

Finally, the first named author thanks the second for his work and for his patience.

\numberwithin{theorem}{subsection}	
\section{Preliminaries}

\subsection{General notation}\label{Nota2.1}

Let $L$ be a finite extension of $\Q_p$ and $E$ be a finite extension of $\Q_p$ sufficiently large such that $\Sigma_L:=\{\tau: L \hookrightarrow \overline{\Q_p}\}=\{\tau: L \hooklongrightarrow E\}$. For $\mathbf{k}=(k_{\tau})_{\tau \in \Sigma_L} \in \Z^{\oplus |\Sigma_L|}$, denote by $z^{\mathbf{k}}:=\prod_{\tau\in \Sigma_L} \tau(z)^{k_{\tau}}: L^{\times} \ra E^{\times}$. Let $\co_L$, resp.\ $\co_E$ be the ring of integers of $L$, resp.\ $E$, $k_E$ the residue field of $E$, $\varpi_L$ be a uniformizer of $\co_L$, $q_L:=|\co_L/\varpi_L|$ and $\val_L(x)$ the valuation on $L^\times$ such that $\val_L(\varpi_L)=1$. For a character $\chi$ of $\co_L^{\times}$, denote by $\chi_{\varpi_L}$ the character of $L^{\times}$ such that $\chi_{\varpi_L}|_{\co_L^{\times}}=\chi$ and $\chi_{\varpi_L}(\varpi_L)=1$; for a character $\delta$ of $L^{\times}$, denote by $\delta_0:=\delta|_{\co_L^{\times}}$.\index{$\chi_{\varpi_L}$} \index{$\delta^0$} We use the convention that the Hodge-Tate weight of the $p$-adic cyclotomic character if $1$. For a group $A$ and $a\in A$, we denote by $\unr(a): L^{\times} \ra A$ the unramified character sending any uniformizer to $a$. 

Let $A$ (resp.\ $X$) be an affinoid algebra (resp.\ a rigid analytic space), we write $\cR_{A,L}$ (resp.\ $\cR_{X,L}$) for the Robba ring associated to $L$ with $A$-coefficients (resp.\ with $\co_X$-coefficients) (see \cite[Def.\ 6.2.1]{KPX}), and $\cR_{A,L}(\delta)$ for the $(\varphi,\Gamma)$-module of character type over $\cR_{A,L}$ associated to a continuous character $\delta: L^{\times} \ra A^{\times}$ in \cite[Const.\ 6.2.4]{KPX}.

Let $m\in \Z_{\geq 1}$, $\pi$ be an irreducible smooth admissible representation of $\GL_m(L)$, denote by $\rec(\pi)$ the $F$-semi-simple Weil-Deligne representation associated to $\pi$ via the local Langlands correspondence normalized as in \cite{HT}. We normalize local class field theory by sending a uniformizer to a (lift of the) geometric Frobenius. In this way, we identify characters of the Weil group $W_L\subset \Gal_L:=\Gal(\overline L/L)$ and characters of $L^{\times}$ without further mention. Let $\chi_{\cyc}$ denote the cyclotomic character of $\Gal_L$ (and of $L^{\times}$). \index{$\chi_{\cyc}$}

Let $\Omega$ be a cuspidal Bernstein component of $\GL_m(L)$ (\cite{Bern84}) and $\pi\in \Omega$. We put\index{$\mu_{\Omega}$}\index{$\mu_{\Omega}^{\unr}$}
\begin{eqnarray}
	\mu_{\Omega}&:=&\{\eta: L^{\times} \ra E^{\times} \ |\ \pi \otimes_E \eta \circ \dett \cong \pi\}, \label{muOmega} \\
	\mu_{\Omega}^{\unr}&:=&\{\eta: L^{\times} \ra E^{\times} \text{ unramified} \ |\ \pi\otimes_E \eta \circ \dett \cong \pi\} \nonumber.
\end{eqnarray}
We have $\mu_{\Omega}^{\unr} \subseteq \mu_{\Omega}$ and it is easy to see that both are finite groups (look at the central characters) and independent of the choice of $\pi$ in $\Omega$. Denote by $\cZ_{\Omega}$ the corresponding Bernstein centre (see \S~\ref{sec_pDf} below for more details). For a closed point $x\in \Spec \cZ_{\Omega}$, denote by $\pi_x$ the associated irreducible cuspidal smooth representation of $\GL_m(L)$ over $k(x)$, $\ttr_x:=\rec(\pi_x)$ and $\Delta_x$ the $p$-adic differential equation associated to $\ttr_x$. Recall from \cite{Ber08a} that $\Delta_x$ is the $(\varphi, \Gamma)$-module of rank $m$ over $\cR_{k(x),L}$ which is de Rham of constant Hodge-Tate weight $0$ such that $D_{\pst}(\Delta_x)$ (forgetting the Hodge filtration) is isomorphic to the Deligne-Fontaine module associated by Fontaine to $\ttr_x$ (\cite[Prop.\ 4.1]{BS07}). We may use the associated $\GL_m(L)$-representation or the associated Weil-Deligne representation or the associated $p$-adic differential equation to denote a closed point of $\Spec \cZ_{\Omega}$ depending on the situation.

Throughout the paper, we denote by $B$ the upper triangular matrices in $\GL_n$ and we will consider parabolic subgroups $P$ of $\GL_n$ containing $B$, i.e.\ of the form
\begin{equation}\label{paraP}\begin{pmatrix}
		\GL_{n_1} & * & \cdots & * \\
		0 & \GL_{n_2} & \cdots & * \\
		\vdots & \vdots & \ddots & * \\
		0 & 0 &\cdots & \GL_{n_r}
	\end{pmatrix},\end{equation}
where $n_i\in \Z_{\geq 1}$ such that $\sum_{i=1}^r n_i=n$. For $i\in \{1,\dots,r\}$ we define $s_i:=\sum_{j=1}^i n_j$ and $s_{0}:=0$. We denote by $L_P$ the Levi subgroup of $P$ containing the group $T$ of diagonal matrices. An integral weight $\lambda=(\lambda_1, \dots, \lambda_n)$ of $\GL_n$ is called $P$-dominant (resp.\ strictly $P$-dominant) if for $j=1, \dots, r$ with $n_j>1$, and $s_{j-1}\leq i \leq s_{j-1}+n_j-1$, we have $\lambda_i\geq \lambda_{i+1}$ (resp.\ $\lambda_i> \lambda_{i+1}$).

We use $\lg(-)$ to denote the length function on elements in Weyl groups. Let $\sW\cong S_n$ be the Weyl group of $\GL_n$, and $\sW_{L_P}\subset \sW$ be the Weyl group of $L_P$. Denote by $\sW^P_{\min}\subset \sW$ (resp.\ $\sW^P_{\max}$) the set of minimal (resp.\ maximal) length representatives in $\sW$ of the {\it right} cosets in $\sW_{L_P}\backslash \sW$. Let $w_0\in \sW$ be the element of maximal length. Then $w\in \sW^P_{\min}$ if and only if $ww_0\in \sW^P_{\max}$. We denote by $\sW_{L}:=\sW^{|\Sigma_L|}$ (resp.\ $\sW_{L_P,L}:=\sW_{L_P}^{|\Sigma_L|}$) which is the Weyl group of $\Res^L_{\Q_p} \GL_n$ (resp.\ $\Res^L_{\Q_p} L_P$). Then $\sW_{\min,L}^{P}=(\sW^P_{\min})^{|\Sigma_L|}$ (resp.\ $\sW_{\max,L}^{P}=(\sW^P_{\max})^{|\Sigma_L|}$) is the set of minimal (resp.\ maximal) length representatives in $\sW_L$ of $\sW_{L_P,L} \backslash \sW_L$. Put $w_{0,L}:=(w_0,\dots,w_0)\in \sW_L$ for the element of maximal length in $\sW_L$. We use ``$\cdot$" to denote the \textit{dot} action of a Weyl group on the corresponding weight space (cf.\ \cite[Def. 1.8]{Hum08}).
\index{$w_0$}\index{$w_{0,L}$}\index{$\sW$}\index{$\sW_L$}\index{$\sW^P_{\min}$}\index{$\sW^P_{\max}$} 

If $X$ is a scheme locally of finite type over $E$, or a locally noetherian formal scheme over $\co_E$ whose reduction (modulo an ideal of definition) is locally of finite type over $k_E$, we denote by $X^{\rig}$ the associated rigid analytic space over $E$. If $X$ is a scheme locally of finite type over $E$ or a rigid analytic space over $E$, we denote by $X^{\red}$ the associated reduced Zariski-closed subspace. If $x$ is a point of $X$, we denote by $k(x)$ the residue field at $x$, $\co_{X,x}$ the local ring at $x$, $\widehat{\co}_{X,x}$ its $\fm_{\co_{X,x}}$-adic completion and $\widehat{X}_x$ the affine formal scheme $\Spf \widehat{\co}_{X,x}$. If $x$ is a closed point of $X$, then $\widehat{\co}_{X,x}$ is a noetherian complete local $k(x)$-algebra of residue field $k(x)$. \index{$\co(X)$} \index{$k(x)$}\index{$\co_{X,x}$} \index{$\widehat{\co}_{X,x}$}\index{$\widehat{X}_x$}

\subsection{$p$-adic differential equations over Bernstein components}\label{sec_pDf}

Let $m\geq 1$, $\Omega$ be cuspidal type of $\GL_m(L)$ and $\cZ_{\Omega}$ be the associated Bernstein centre over $E$ (that we recall below). In this section, we construct a ``universal" $p$-adic differential equation on $(\Spec \cZ_{\Omega})^{\rig}$ that interpolates $\{\Delta_x\}_{x\in (\Spec \cZ_{\Omega})^{\rig}}$.

Let $\pi$ be an irreducible smooth representation of $\GL_m(L)$ over $E$ of type $\Omega$. We assume that $E$ contains the $m$-th roots of unity. By comparing the central characters, we see that there exists $m_0|m$ such that $\mu_{m_0}=\{a \in E^{\times}\ |\ \pi \otimes_E \unr(a)\circ \dett \cong \pi \}$, where $\mu_{m_0}$ denotes the group of $m_0$-th roots of unity in $E^{\times}$. We equip $E[z,z^{-1}]$ with an action of $\mu_{m_0}$ by $a(z):=az$ for $a\in \mu_{m_0}$. We then have a natural isomorphism
\[\cZ_{\Omega} \cong E[z,z^{-1}]^{\mu_{m_0}}\cong E[z^{m_0}, z^{-m_0}]\]
such that the induced map $\bG_m:=\Spec E[z,z^{-1}] \ra \Spec \cZ_{\Omega}$ sends $\alpha$ to $\pi \otimes_E \unr(\alpha) \circ \dett$.

We first construct a $(\varphi, \Gamma)$-module on $\bG_m^{\rig}$. Let $\Delta$ be the $p$-adic differential equation over $\cR_{E,L}$ associated to $\pi$ (or equivalently to $\rec(\pi)$). Let
\[A_j:=E\langle z_j,t_j\rangle / (z_jt_j-p^{2j}),\]
then the maximal spectrum $\Spm A_j$ is $\{z\in \bG_m^{\rig}\ |\ p^{-j}\leq |z|_p\leq p^{j}\}$ (mapping $z$ to $p^{-j}z_j$ and $z^{-1}$ to $p^{-j}t_j$) and $\{\Spm A_j\}_{j\in \Z_{\geq 0}}$ form an admissible covering of $\bG_m^{\rig}$. We define $\Delta_{A_j}:=\Delta {\otimes}_{\cR_{E,L}} \cR_{A_j,L}(\unr(z))$, which is a $(\varphi, \Gamma)$-module free of rank $m$ over $\cR_{A_j,L}$. These $\{\Delta_{A_j}\}_{j\in \Z_{\geq 0}}$ glue to a $(\varphi,\Gamma)$-module over $\cR_{\bG_m^{\rig},L}$, where $\cR_{\bG_m^{\rig},L}$ is defined as in \cite[Def.\ 6.2.1]{KPX}. 

Let $\varsigma_{m_0}$ be a primitive $m_0$-th root of unity. Since $\pi \cong \pi \otimes_E \unr(\varsigma_{m_0})\circ \dett$, we have $\Delta\cong \Delta\otimes_{\cR_{E,L}} \cR_{E,L}(\unr(\varsigma_{m_0}))$. Let $\iota_1: \Delta \ra \Delta \otimes_{\cR_{E,L}} \cR_{E,L}(\unr(\varsigma_{m_0}))$ be an isomorphism of $(\varphi,\Gamma)$-modules. For a continuous character $\delta$ of $L^{\times}$, we also use $\iota_{1}$ to denote the induced morphism $\Delta \otimes_{\cR_{E,L}} \cR_{E,L}(\delta) \ra \Delta \otimes_{\cR_{E,L}} \cR_{E,L}(\unr(\varsigma_{m_0})) \otimes_{\cR_{E,L}} \cR_{E,L}(\delta)$. We put for $i\in \Z_{\geq 1}$:
\begin{equation*}
	\iota_{i}:=\iota_{1} \circ \iota_{1} \circ \cdots \circ \iota_{1}: \Delta \xrightarrow{\iota_{1}} \Delta \otimes_{\cR_{E,L}} \cR_{E,L}(\unr(\varsigma_{m_0})) \xrightarrow{\iota_{1}} \cdots \xrightarrow{\iota_{1}} \Delta \otimes_{\cR_{E,L}} \cR_{E,L}(\unr(\varsigma_{m_0}^i)).
\end{equation*}
Since $\Hom_{(\varphi, \Gamma)}(\Delta,\Delta)\cong E$ and $E$ contains all $m_0$-th roots of unity, we can multiply $\iota_{1}$ by a scalar in $E^{\times}$ so that $\iota_{m_0}=\id_\Delta$. Let $\Delta_{A_j}^{i}:=\Delta{\otimes}_{\cR_{E,L}} \cR_{A_j,L}(\unr(\varsigma_{m_0}^i z))$. The isomorphism $\iota_{i}$ induces an isomorphism (still denoted) $\iota_{i}: \Delta_{A_j} \xrightarrow{\sim} \Delta_{A_j}^{i}$ satisfying $\iota_{m_0}=\id_{\Delta_{A_j}}$. We fix a basis $\ul{e}$ of $\Delta$ over $\cR_{E,L}$, and still denote by $\ul{e}$ the corresponding basis $\ul{e}\otimes 1$ of $\Delta_{A_j}^{i}$ over $\cR_{A_j,L}$. We do not ask that $\iota_1$ respects $\ul{e}$ (i.e.\ sends $\ul{e}$ to $\ul{e}\otimes 1$), hence the isomorphisms $\iota_{i}$ in general do not stabilize $\ul{e}$. 

It is clear that $\Spm A_j$ is stable by the induced action of $\mu_{m_0}$ on $\bG_m^{\rig}$. The action of $\mu_{m_0}$ on $A_j$ induces an action on $\cR_{A_j,L}$. We equip $\Delta_{A_j}$ with an $\cR_{A_j,L}$-semi-linear action of $\mu_{m_0}$ such that $\varsigma_{m_0}^i$ acts via
\begin{equation*}
	\Delta_{A_j} \lra \Delta_{A_j}^{i} \xrightarrow{\iota_{i}^{-1}} \Delta_{A_j}
\end{equation*}
where the first map sends $v \otimes a \in \Delta{\otimes}_{\cR_{E,L}} \cR_{A_j,L}(\unr(z))$ to $v \otimes \varsigma_{m_0}^i(a)\in \Delta{\otimes}_{\cR_{E,L}} \cR_{A_j,L}(\unr(\varsigma_{m_0}^iz))$. Indeed, one can check that this defines a group action of $\mu_{m_0}$, that commutes with the $(\varphi, \Gamma)$-action. 

Define $\Delta_{B_j}:=\Delta_{A_j}^{\mu_{m_0}}$. By \cite[Prop.\ 2.2.1]{BeCol08} (applied first to $B=A_j$, $S=\cR_{E,L}^{[r,s]}$ and $G=\mu_{m_0}$, then letting $r,s$ vary), we can deduce that $\Delta_{B_j}$ is a $(\varphi, \Gamma)$-module free of rank $m$ over $\cR_{B_j,L}$ where $B_j:=A_j^{\mu_{m_0}}\cong E\langle z_j^{m_0},t_j^{m_0}\rangle / (z_j^{m_0}t_j^{m_0}-p^{2jm_0})$. The affinoids $\{\Spm B_j\}_{j\in \Z_{\geq 0}}$ then form an admissible covering of $(\Spec \cZ_{\Omega})^{\rig}$. Moreover it is easy to see that $\{\Delta_{B_j}\}_{j\in \Z_{\geq 0}}$ glue to a $(\varphi, \Gamma)$-module over $\cR_{(\Spec \cZ_{\Omega})^{\rig},L}$ that we denote by $\Delta_{\Omega}$. One checks that $\Delta_{\Omega}$ is independent of the choice of $\pi$ of type $\Omega$. It is also clear that for a point $x\in (\Spec \cZ_{\Omega})^{\rig}$ with $\pi_x$ the associated smooth representation of $\GL_m(L)$ over $k(x)$, the fibre $\Delta_x:=x^* \Delta_{\Omega}$ is isomorphic to the $p$-adic differential equation associated to $\pi_x$.

\subsection{Potentially crystalline representations}\label{introPcr}

We recall the structure of potentially crystalline Galois representations.

Let $\rho$ be an $n$-dimensional potentially crystalline representation of $\Gal_L$ over $E$. Let $L'$ be a finite Galois extension of $L$ such that $\rho|_{\Gal_{L'}}$ is crystalline. Consider the Deligne-Fontaine module associated to $\rho$: 
\begin{equation*}
	\DF(\rho):=\big(D_{L'}:=(B_{\cris} \otimes_{\Q_p} \rho)^{\Gal_{L'}}, \varphi, \Gal(L'/L)\big),
\end{equation*}
where $D_{L'}:=(B_{\cris}\otimes_{\Q_p} \rho)^{\Gal_{L'}}$ is a finite free $L'_0\otimes_{\Q_p} E$-module of rank $n$, $L_0'$ being the maximal unramified subextension of $L'$ (over $\Q_p$), where the $\varphi$-action on $D_{L'}$ is induced from the $\varphi$-action on $B_{\cris}$, and where the $\Gal(L'/L)$-action on $D_{L'}$ is the residual action of $\Gal_L$. By Fontaine's equivalence of categories as in \cite[Prop.\ 4.1]{BS07}, we can associate to $\DF(\rho)$ an $n$-dimensional Weil-Deligne representation $\ttr(\rho)$ of $W_L$ over $E$ (and we can recover $\DF(\rho)$ from $\ttr(\rho)$ as in \textit{loc.\ cit.}).\index{$\ttr(\rho)$}

Let $P$ be a parabolic subgroup of $\GL_n$ as in (\ref{paraP}). Assume that $\ttr(\rho)$ admits a filtration 
\begin{equation*}
	\sF: \Fil_{\bullet}^{\sF}\!\ttr(\rho)=\big(0=\Fil_0^{\sF}\!\ttr(\rho)\subsetneq \Fil_1^{\sF}\!\ttr(\rho) \subsetneq \cdots \subsetneq \Fil_r^{\sF}\!\ttr(\rho)=\ttr(\rho)\big)
\end{equation*}
by Weil-Deligne subrepresentations such that $\dim_E \Fil_i^{\sF}\!\ttr(\rho)=\sum_{j=1}^i n_j$. We call such a filtration a \textit{$P$-filtration}. We call the filtration $\sF$ a \textit{minimal filtration} if $\ttr(\rho)_i:=\gr_i^{\sF}\!\ttr(\rho)$ is an irreducible Weil-Deligne representation for all $i$. We assume that $\sF$ is minimal in the sequel. In this case, the Galois representation $\rho$ is called \textit{generic} if $\Hom(\ttr(\rho)_i, \ttr(\rho)_j)=0$ and $\Hom(\ttr(\rho)_i, \ttr(\rho)_j(1))=0$ for all $i \neq j$ (where $\Hom$ is taken in the category of Weil-Deligne representations and $(1)$ means the twist by $x\in L^\times \mapsto \frac{1}{q_L^{\val_L(x)}}$). It is easy to see that being generic does not depend on the choice of minimal filtrations on $\ttr(\rho)$, and that if $\rho$ is generic then $\ttr(\rho)\cong \oplus_{i=1}^r \ttr(\rho)_i$. Let $\Omega_i$ be the Bernstein component of $\GL_{n_i}(L)$ such that the smooth irreducible representation corresponding to $\ttr(\rho)_i$ via the classical local Langlands correspondence (normalized as in \cite{HT}) lies in $\Omega_i$. Let $\Omega:=\prod_i \Omega_i$, which is a Bernstein component of $L_P(L)$. The minimal $P$-filtration $\sF$ will also be called an $\Omega$-filtration.

The $P$-filtration $\sF$ corresponds to a filtration (still denoted) $\sF=\Fil_{\bullet}^{\sF}\!\DF(\rho)$ (and still called a $P$-filtration) on $\DF(\rho)$ by Deligne-Fontaine submodules, such that $\Fil_i^{\sF}\!\DF(\rho)$ is associated to $\Fil_i^{\sF}\!\ttr(\rho)$ via \cite[Prop.\ 4.1]{BS07}. If $\rho$ is generic, we have then $\DF(\rho)\cong \oplus_{i=1}^r \gr_i^{\sF}\!\DF(\rho)$.

As $\rho$ is potentially crystalline, it is de Rham, and thus we have $D_{\dR}(\rho)\cong (D_{L'}\otimes_{L_0'} L')^{\Gal(L'/L)}$, which is a free $L\otimes_{\Q_p} E$-module of rank $n$. The $P$-filtration $\sF$ on $\DF(\rho)$ induces a $P$-filtration $\sF$ on $D_{\dR}(\rho)$ by free $L \otimes_{\Q_p} E$-submodules $\Fil_i^{\sF}\!D_{\dR}(\rho):=(\Fil_i^{\sF}\!D_{L'} \otimes_{L_0'} L')^{\Gal(L'/L)}$. Recall also that $D_{\dR}(\rho)$ is equipped with a natural decreasing Hodge filtration $\Fil_{\bullet}^{H} D_{\dR}(\rho)$ (induced by the one on $B_{\cris}$) given by (not necessarily free) $L \otimes_{\Q_p} E$ submodules. We assume that $\rho$ has distinct Hodge-Tate weights. Hence, for each $\tau\in \Sigma_L$, we have a complete flag (with an obvious notation for $D_{\dR}(\rho)_{\tau}$)
\begin{equation*}\label{hflag}
	0 \subsetneq \Fil_{-h_{n,\tau}}^{H}\! D_{\dR}(\rho)_{\tau} \subsetneq \Fil_{-h_{n-1,\tau}}^H\! D_{\dR}(\rho)_{\tau} \subsetneq \cdots \subsetneq \Fil_{-h_{1,\tau}}^H\! D_{\dR}(\rho)_{\tau}=D_{\dR}(\rho)_{\tau}
\end{equation*}
where $h_{i,\tau}$ are the integers such that $\dim_E \gr_{-h_{i,\tau}}^H\! D_{\dR}(\rho)_{\tau}=1$. Thus the Hodge-Tate weights of $\rho$ are $\textbf{h}=(\textbf{h}_i)_{i=1,\dots, n}=(h_{1,\tau}>\cdots >h_{n,\tau})_{\tau \in \Sigma_L}$. 

We fix a basis of $D_{\dR}(\rho)_{\tau}$ over $E$ for each $\tau\in \Sigma_L$. The filtration $\sF$ (resp.\ $\Fil^H_{\bullet}$) on $D_{\dR}(\rho)$ thus corresponds to an $E$-point of the flag variety $\Res_{\Q_p}^L \GL_n/\Res_{\Q_p}^L P$ (resp.\ $\Res_{\Q_p}^L \GL_n/ \Res_{\Q_p}^L B)$):
\begin{equation*}
	g_1 (\Res_{\Q_p}^L P)(E)=(g_{1,\tau}P(E))_{\tau\in \Sigma_L}\ \ \ \text{(resp.\ }	g_2 (\Res_{\Q_p}^L B)(E)=(g_{2,\tau}B(E))_{\tau\in \tau_L})\text{).}
\end{equation*}
For each $\tau \in \Sigma_L$, there exists thus a unique $w_{\sF, \tau}\in \sW^P_{\max}$ such that 
\begin{equation*}
	(g_{1,\tau} P(E), g_{2,\tau} B(E)) \in \GL_n(E) (1,w_{\sF, \tau})(P \times B)(E) \subseteq (\GL_n/P \times \GL_n/B)(E)
\end{equation*}
where we still use $w_{\sF, \tau}\in N_{\GL_n}(T)$ to denote a lifting of the corresponding element in $\sW$ (which is traditionally denoted by $\dot{w}_{\sF,\tau}$). We write $w_{\sF}:=(w_{\sF, \tau})_{\tau \in \Sigma_L}\in \sW^{|\Sigma_L|}$. We call $(\rho, \cF)$ \textit{non-critical} if $w_{\sF}=w_{0,L}$, or equivalently $w_{\sF, \tau}=w_0$ for all $\tau\in \Sigma_L$. 

Consider the $(\varphi,\Gamma)$-module $D_{\rig}(\rho)$ over $\cR_{E,L}$ associated to $\rho$ (see \cite{Ber08a} and the references therein). Let $\Delta$ (resp.\ $\Fil_i^{\sF}\!\Delta$) be the $p$-adic differential equation over $\cR_{E,L}$ associated to $\DF(\rho)$ (resp.\ to $\Fil_i^{\sF}\!\DF(\rho)$), or equivalently to $\ttr(\rho)$ (resp.\ to $\Fil_i^{\sF}\!\ttr(\rho)$). Then $\sF:=\Fil_{\bullet}^{\sF}\!\Delta$ gives an increasing filtration on $\Delta$ by saturated $(\varphi, \Gamma)$-submodules. Consider
\[\cM(\rho):=D_{\rig}(\rho)[1/t] \cong \Delta[1/t].\]
By inverting $t$, the filtration $\sF$ on $\Delta$ induces an increasing filtration (still denoted) $\sF:=\big\{\Fil_i^{\sF}\!\cM(\rho):=(\Fil_i^{\sF}\!\Delta)[1/t]\big\}$ on $\cM(\rho)$ by $(\varphi, \Gamma)$-submodules over $\cR_{E,L}[1/t]$. Finally, the filtration $\sF$ on $\cM(\rho)$ induces a filtration on $D_{\rig}(\rho)$:
\begin{equation}\label{OmeFil00}
	\sF:= \{\Fil_i^{\sF}\!D_{\rig}(\rho):= \Fil_i^{\sF}\!\cM(\rho) \cap D_{\rig}(\rho)\}
\end{equation}
by saturated $(\varphi, \Gamma)$-submodules of $D_{\rig}(\rho)$. By Berger's equivalence of categories (\cite[Thm.~A]{Ber08a}), $\Fil_i^{\sF}\!D_{\rig}(\rho)$ corresponds to the filtered Deligne-Fontaine module $\Fil_i^{\sF}\!\DF(\rho)$ equipped with the induced filtration from the Hodge filtration on $D_{L'}=(B_{\cris} \otimes_{\Q_p} \rho)^{\Gal_{L'}}$ (coming from the filtration on $B_{\cris}$). Such a filtration will be called an $\Omega$-filtration on $D_{\rig}(\rho)$ (see \S~\ref{secDefOD} for a definition in a more general setting).

One sees that the Hodge-Tate weights of $\Fil_i^{\sF}\!D_{\rig}(\rho)$ are given by (recall $s_i=\sum_{j=1}^i n_j$) $\{h_{(w_{\sF, \tau}w_0)^{-1}(1),\tau}, \dots, h_{(w_{\sF, \tau}w_0)^{-1}(s_i),\tau}\}_{\tau\in \Sigma_L},$ hence the Hodge-Tate weights of $\gr_i^{\sF}\!D_{\rig}(\rho)$ are
\[(w_{\sF}(\textbf{h})_{s_{i-1}+1}, \dots, w_{\sF}(\textbf{h})_{s_i}) = \big(h_{(w_{\sF, \tau}w_0)^{-1}(s_{i-1}+1),\tau}, \dots, h_{(w_{\sF, \tau}w_0)^{-1}(s_i),\tau}\big)_{\tau\in \Sigma_L}\]
(which are decreasing as $w_{\sF,\tau}w_0\in \sW^P_{\min}$).
In particular, $(\rho,\sF)$ is non-critical if and only if the Hodge-Tate weights of $\gr_i^{\sF}\!D_{\rig}(\rho)$ are $(\textbf{h}_{s_{i-1}+1}, \dots, \textbf{h}_{s_i})$ for $i=1, \dots, r$. Since $\gr_i^{\sF}\!D_{\rig}(\rho)\subseteq t^{-N} \gr_i^{\sF}\!\Delta$ for $N$ sufficiently large, using \cite[Thm.\ A]{Ber08a} and comparing the weights, we have an injection of $(\varphi, \Gamma)$-modules over $\cR_{E,L}$ for $i=1, \dots, r$:
\begin{equation}\label{inj000}
	\gr_i^{\sF}\!D_{\rig}(\rho) \otimes_{\cR_{E,L}} \cR_{E,L}(z^{-w_{\sF}(\textbf{h})_{s_{i}}}) \hooklongrightarrow \gr_i^{\sF}\!\Delta.
\end{equation}

Finally let $\overline{\rho}: \Gal_L \ra \GL_n(k_E)$ be a continuous representation, $\xi: I_L \ra \GL_n(E)$ be an inertial type (where $I_L\subset W_L$ denotes the inertial subgroup), and $\textbf{h}\in \Z^{n|\Sigma_L|}$ be a strictly dominant weight as above. We denote by $R_{\overline{\rho}}^{\pcr}(\xi, \textbf{h})$ the universal potentially crystalline \textit{framed} deformation ring of $\overline{\rho}$ of inertial type $\xi$ and of Hodge-Tate weights $\textbf{h}$ (cf.\ \cite{Kis08}).

\section{Bernstein eigenvarieties}\label{secBE}

In this section, we construct Bernstein eigenvarieties from $p$-adic automorphic representations. In \S~\ref{abCon}, we give the general formalism of the construction, which can be applied to any admissible locally analytic representation of (a product of copies of) $\GL_n$. We then apply in \S~\ref{secBern} this formalism to $p$-adic automorphic representations on compact unitary group, to get what we call Bernstein eigenvarieties. We prove basic properties of the latter, like the density of classical points, etc. We also show that the Galois representations associated to points on Bernstein eigenvarieties admit a certain filtration, and address the problem of companion points. Finally, in \S~\ref{secPBern}, we apply the general formalism to the ``patched'' $p$-adic automorphic representation of \cite{CEGGPS} to obtain a patched version of Bernstein eigenvarieties (that has a more local flavor).

\subsection{Abstract construction}\label{abCon}

This section gives a general formal construction of certain rigid analytic spaces from Emerton's Jacquet modules of locally analytic representations, using Bushnell-Kutzko's theory of types.

\subsubsection{Notation and setup}\label{sec3.1.1}

We will assume that the reader has some familiarity with $p$-adic functional analysis, and we use - most of the time without further mention - the various foundational results in \cite{Sch02}, \cite{ST02}, \cite{ST03}, \cite{ST05} and \cite{Em04}.

For a locally $\Q_p$-analytic group $H$, denote by $\cC^{\Q_p-\la}(H,E)$ the space of locally $\Q_p$-analytic functions on $H$ with values in $E$ and by $D(H,E):=\cC^{\Q_p-\la}(H,E)^\vee$ its strong dual (the distribution algebra), which is a Fr\'echet-Stein algebra when $H$ is compact. Denote by $\cC^{\infty}(H,E)\hookrightarrow \cC^{\Q_p-\la}(H,E)$ the closed subspace of locally constant functions on $H$ with values in $E$, and set $D^{\infty}(H,E):=\cC^{\infty}(H,E)^{\vee}$, which is a Hausdorff quotient of $D(H,E)$.

For a topologically finitely generated locally $\Q_p$-analytic abelian group $Z$, denote by $\widehat{Z}$ the rigid space over $E$ parameterizing locally $\Q_p$-analytic characters of $Z$ (cf.\ \cite[Prop.\ 6.4.5]{Em04}). By \cite[Prop.\ 6.4.6]{Em04}, there is a natural injection $D(Z,E)\hookrightarrow \Gamma(\widehat{Z},\co_{\widehat{Z}})$. For $\chi$ a locally $\Q_p$-analytic character of $Z$, we denote by $\fm_{\chi}$ the associated maximal ideal of $E[Z]$. If $Z\cong (L^{\times})^n$ or is a compact open subgroup of $(L^{\times})^n$ (for $L$ as in \S~\ref{Nota2.1}), and $\chi$ is a locally $\Q_p$-analytic character of $Z$, we denote by $\wt(\chi)$ the weight of $\chi$ (see \cite[Notation]{BHS3}). For instance if $\chi$ is $E$-valued, then $\wt(\chi)=(\wt(\chi)_{\tau})_{\tau\in \Sigma_L}=(\wt(\chi)_{i,\tau})_{\substack{i=1,\dots, n \\ \tau\in \Sigma_L}}\in (E^n)^{|\Sigma_L|}$. In this case, for $\lambda=(\lambda_{i,\tau})\in (\Z^n)^{|\Sigma_L|}$, we denote by $\delta_{\lambda}$ the algebraic character of $Z$ of weight $\lambda$. \index{$\delta_{\lambda}$}

For a continuous Banach representation $\Pi$ of a $p$-adic Lie group $G$, we denote by $\Pi^{\an}$ the locally $\Q_p$-analytic subrepresentation of $\Pi$ and by $\Pi^{\lalg}\subseteq \Pi^{\an}$ the locally $\Q_p$-algebraic subrepresentation of $\Pi$.

Let $\sI$ be a finite index set. For any $i\in \sI$, assume we have a finite extension $F_i$ of $\Q_p$. For each $F_i$, we fix a uniformizer $\varpi_i$, and denote by $\kappa_{\varpi_i}: F_i^{\times} \twoheadrightarrow \co_{F_i}^{\times}$ the map sending $\varpi_i$ to $1$ and being the identity on $\co_{F_i}^{\times}$.

Let $G:=\prod_{i\in \sI} \Res_{\Q_p}^{F_i} \GL_n$ (an algebraic group), and $G_p:=G(\Q_p)$. For each $i\in \sI$, we fix a parabolic subgroup $P_i$ of $\GL_n$ containing the Borel subgroup $B$ of upper triangular matrices, and let $L_{P_i}$ the Levi subgroup of $P_i$ containing the group $T$ of diagonal matrices. Let $P:=\prod_{i\in \sI} \Res_{\Q_p}^{F_i} P_i\supseteq B_{\sI}:=\prod_{i\in \sI} \Res_{\Q_p}^{F_i} B$, and $L_P:=\prod_{i\in \sI} \Res_{\Q_p}^{F_{i}} L_{P_i}$ which is the Levi subgroup of $P$ containing $T_{\sI}:=\prod_{i\in \sI} \Res_{\Q_p}^{F_i} T$. Let $N_P$ (resp.\ $N_{P_i}$) be the unipotent radical of $P$ (resp.\ of $P_i$), $P^-$ (resp.\ $P^-_i$) the parabolic subgroup opposite to $P$ (resp.\ to $P_i$), $N_{P^-}$ (resp.\ $N_{P^-_i}$) the unipotent radical of $P^-$ (resp.\ of $P^-_i$), $Z_{L_P}$ (resp.\ $Z_{L_{P_i}}$) the centre of $L_P$ (resp.\ of $L_{P_i}$), and $L_P^D$ (resp.\ $L_{P_i}^D$) the derived subgroup of $L_P$ (resp.\ $L_{P_i}$). We have therefore
\[\begin{array}{ccc}
	N_P=\prod_{i\in \sI} \Res_{\Q_p}^{F_i} N_{P_i},&P^-=\prod_{i\in \sI} \Res_{\Q_p}^{F_i} P_i^-,&N_{P^-}=\prod_{i\in \sI} \Res_{\Q_p}^{F_i} N_{P^-_i},\\
	Z_{L_P}=\prod_{i\in \sI} \Res_{\Q_p}^{F_i} Z_{L_{P_i}},&L_P^D=\prod_{i\in \sI} \Res_{\Q_p}^{F_i} L_{P_i}^D.&
\end{array}\]
We denote by $\ug$, $\ub_{\sI}$, $\fp$, $\fn_P$, $\fl_P$, $\fn_{P^-}$, $\fz_{L_P}$, $\fl_P^D$ the Lie algebra over $E$ of $G$, $B_{\sI}$, $P$, $N_P$, $L_P$, $N_{P^-}$, $Z_{L_P}$, $L_P^D$ respectively. For a Lie algebra $\fh$ over $E$, denote by $\text{U}(\fh)$ the universal enveloping algebra over $E$. 
We define
\[\begin{array}{cl}
	L_P^0:=\prod_{i\in \sI} L_{P_i}(\co_{F_i})\subset L_P(\Q_p),&Z_{L_P}^0:=\prod_{i\in \sI} Z_{L_{P_i}}(\co_{F_i})=L_P^0\cap Z_{L_P}(\Q_p)\subset Z_{L_P}(\Q_p)\\
	N_P^0:=\prod_{i\in \sI} N_P(\co_{F_i})\subset N_P(\Q_p), & N_{P^-}^0:=\prod_{i\in \sI} N_{P^-}(\co_{F_i})\subset N_{P^-}(\Q_p)
\end{array}\]
and denote by $\dett_{L_P}$ the determinant map $L_P(\Q_p) \ra Z_{L_P}(\Q_p)$. For each $i\in \sI$, $L_{P_i}(F_i)$ has the form 
$\begin{pmatrix}
	\GL_{n_{i,1}} & 0 & \cdots & 0 \\
	0 & \GL_{n_{i,2}} & \cdots & 0 \\
	\vdots & \vdots & \ddots & 0 \\
	0 & 0 &\cdots & \GL_{i, n_{r_i}}
\end{pmatrix}$
for some $r_i\in \Z_{\geq 1}$ and integers $n_{i,j}\in \Z_{\geq 1}$, $1\leq j\leq r_i$ with $\sum_{j=1}^{r_i} n_{i,j}=n$. Hence $L_P(\Q_p)\cong \prod_{i\in \sI} \prod_{j=1}^{r_i} \GL_{n_{i,j}}(F_i)$. For each $(i,j)\in \sI\times \{1,\dots,r_i\}$ we fix a cuspidal Bernstein component $\Omega_{i,j}$ for $\GL_{n_{i,j}}(F_i)$, and we let $\Omega:=\prod_{i\in \sI} \prod_{j=1}^{r_i} \Omega_{i,j}$. Let $\cZ_{\Omega_{i,j}}$ (resp.\ $\cZ_{\Omega}$) be the Bernstein centre of $\Omega_{i,j}$ (resp.\ of $\Omega$) over $E$ (see \S~\ref{sec_pDf}), we have an isomorphism of commutative $E$-algebras\index{$\cZ_{\Omega}$}
\begin{equation*}
	\cZ_{\Omega} \cong \otimes_{i\in \sI} \otimes_{j=1}^{r_i} \cZ_{\Omega_{i,j}}.
\end{equation*}
Let $(J_{i,j},\sigma_{i,j}^0)$ be a maximal simple type of $\Omega_{i,j}$ (cf.\ \cite[\S~6]{BK91}) such that the compact open subgroup $J_{i,j}$ is contained in $\GL_{n_{i,j}}(\co_{F_i})$. Recall that $\sigma_{i,j}^0$ is an absolutely irreducible smooth representation of $J_{i,j}$ over $E$.
Put $\sigma_{i,j}:=\Ind_{J_{i,j}}^{\GL_{n_{i,j}}(\co_{F_i})} \sigma_{i,j}^0$, which is an absolutely irreducible smooth representation of $\GL_{n_{i,j}}(\co_{F_i})$ over $E$ (e.g.\ see the proof of \cite[Cor.\ 6.1]{Pyv20}).
Let
\[\sigma^0:=\boxtimes_{i\in \sI}\boxtimes_{j=1}^{r_i} \sigma_{i,j}^0,\ \ J:=\prod_{i,j} J_{i,j}\ \ {\rm and}\ \ \sigma:=\boxtimes_{i\in \sI}\boxtimes_{j=1}^{r_i} \sigma_{i,j}\cong \Ind_{J}^{L_P^0} \sigma^0\]
which is an absolutely irreducible smooth representation of $L_P^0$ over $E$. Recall we have natural isomorphisms of commutative $E$-algebras (where ``$\cind$" denotes the compact induction)
\[\begin{array}{rclcc}
	\End_{\GL_{n_{i,j}}(F_i)}\Big(\cind_{\GL_{n_{i,j}}(\co_{F_i})}^{\GL_{n_{i,j}}(F_i)} \sigma_{i,j}\Big)&\cong &\End_{\GL_{n_{i,j}}(F_i)}\Big(\cind_{J_{i,j}}^{\GL_{n_{i,j}}(F_i)} \sigma_{i,j}^0\Big)& \cong &\cZ_{\Omega_{i,j}}\\
	\End_{L_P(\Q_p)}(\cind_{L_P^0}^{L_P(\Q_p)} \sigma)&\cong &\End_{L_P(\Q_p)}(\cind_{J}^{L_P(\Q_p)} \sigma) &\cong &\cZ_{\Omega}.
\end{array}\]

Let $\lambda$ be an integral $P$-dominant weight of $G$, i.e.\ $\lambda=(\lambda_i)_{i\in \sI}=(\lambda_{i,\tau})_{\substack{i\in \sI \\ \tau\in \Sigma_{F_i}}}$ with each weight $\lambda_{i,\tau}$ of $\GL_n$ being $P_i$-dominant (cf.\ \S~\ref{Nota2.1}), and $L(\lambda)_P$ the algebraic representation of $L_P(\Q_p)$ over $E$ of highest weight $\lambda$. If $\lambda$ is moreover dominant, we denote by $L(\lambda)$ the algebraic representation of $G_p$ (over $E$) of highest weight $\lambda$. \index{$L(\lambda)_P$}\index{$L(\lambda)$}

\subsubsection{$(\Omega, \lambda)$-part of Jacquet-Emerton modules}\label{secAbsOL}

Let $V$ be an admissible locally $\Q_p$-analytic representation of $G(\Q_p)$ over $E$. Using Emerton's locally analytic Jacquet functor and the type theory (\`a la Bushnell-Kutzko), we associate to $V$ a certain $\cZ_{\Omega}\times Z_{L_P}^0$-module $B_{\Omega, \lambda}(V)$. 

First applying Emerton's Jacquet functor $J_P(-)$ (\cite{Em11}) to $V$, we obtain an essentially admissible locally $\Q_p$-analytic representation $J_P(V)$ of $L_P(\Q_p)$ over $E$. Let $\lambda$ be an integral $P$-dominant weight of $G$. We then define
\begin{equation}\label{jplambda}
	J_P(V)_{\lambda}:=\Hom_{\fl_P^D}\big(L(\lambda)_P, J_P(V)\big) \cong \big(J_P(V) \otimes_E L(\lambda)^{\vee}_P\big)^{\fl_P^D} 
	\cong \varinjlim_{H} \big(J_P(V) \otimes_E L(\lambda)_P^{\vee}\big)^{H }
\end{equation} 
where $H$ runs through compact open subgroups of $L_P^D(\Q_P)$. This is a closed $L_P(\Q_p)$-subrepresen\-tation of $J_P(V) \otimes_E L(\lambda)_P^{\vee}$ (with the induced topology). 

We equip the space of locally analytic functions $\cC^{\Q_p-\la}(Z_{L_P}^0, E)$ with an $L_P(\Q_p)$-action given by the regular action of $Z_{L_P}^0$ on $\cC^{\Q_p-\la}(Z_{L_P}^0, E)$ precomposed with:
\begin{equation}\label{dett1}
	L_P(\Q_P) \xrightarrow{(\dett_{L_P})^{-1}} Z_{L_P}(\Q_p)\cong \prod_{i\in \sI} Z_{L_{P_i}}(F_i) \xrightarrow{\prod_i\kappa_{\varpi_i}}\prod_{i\in \sI} Z_{L_{P_i}}(\co_{F_i}) = Z_{L_P}^0.
\end{equation}
Let $\sigma$ be an irreducible smooth representation of $L_P^0$ as in \S~\ref{sec3.1.1}. 
We put\index{$B_{\sigma, \lambda}(V)$}:
\begin{multline}\label{BslV}
	B_{\sigma, \lambda}(V):=\Hom_{L_P^0}\big(\sigma, J_P(V)_{\lambda} \widehat{\otimes}_E \cC^{\Q_p-\la}(Z_{L_P}^0, E)\big)
	\\ \cong \Big(\sigma^{\vee} \otimes_E \big(J_P(V)_{\lambda} \widehat{\otimes}_E \cC^{\Q_p-\la}(Z_{L_P}^0, E)\big)\Big)^{L_P^0}
\end{multline}
where $J_P(V)_{\lambda} \widehat{\otimes}_E \cC^{\Q_p-\la}(Z_{L_P}^0, E)$ is the completion of $J_P(V)_{\lambda} {\otimes}_E \cC^{\Q_p-\la}(Z_{L_P}^0, E)$ equipped with the projective - or equivalently injective - tensor product topology (note that both factors are vector spaces of compact type) and with the diagonal action of $L_P(\Q_p)$. We view $B_{\sigma, \lambda}(V)$ as a closed subspace of $\sigma^{\vee} \otimes_E J_P(V)_{\lambda} \widehat{\otimes}_E \cC^{\Q_p-\la}(Z_{L_P}^0, E)$ which is an $E$-vector space of compact type (recall the finite dimensional $\sigma^{\vee}$ is equipped with the finest locally convex topology). Hence $B_{\sigma, \lambda}(V)$ is also an $E$-vector space of compact type. 

\begin{remark}
The definition of $B_{\sigma, \lambda}(V)$ might appear somewhat artificial. The motivation is to construct an object parametrizing $L_P(\Q_p)$-subrepresentations of $J_P(V)_{\lambda}$ that lie in the Bernstein component $\Omega$ up to twist by continuous characters of $Z_{L_P}(\Q_p)$ (see Proposition \ref{pts1}). One may consider removing the factor $\cC^{\Q_p-\la}(Z_{L_P}^0,E)$ and using types for $L_P^D(\Q_p)$ instead of the type $(\sigma,L_P^0)$ for $L_P(\Q_p)$. The resulting object is actually more natural. However, it is not clear to the authors how to use such an object to parametrize the $L_P(\Q_p)$-subrepresentations of $J_P(V)_{\lambda}$ discussed above. 
\end{remark}

Next, we discuss various group actions on $B_{\sigma, \lambda}(V)$.

There \ \ is \ \ a \ \ natural \ \ locally \ \ $\Q_p$-analytic \ \ action \ \ of \ \ $Z_{L_P}(\Q_p)\times Z_{L_P}^0$ \ \ on \ $\sigma^{\vee} \otimes_E J_P(V)_{\lambda} \widehat{\otimes}_E \cC^{\Q_p-\la}(Z_{L_P}^0, E)$ where $Z_{L_P}(\Q_p)$ acts on $J_P(V)_{\lambda}$, $Z_{L_P}^0$ acts on $\cC^{\Q_p-\la}(Z_{L_P}^0,E)$, and there is no action on $\sigma^\vee$. It is easy to check that this action of $Z_{L_P}(\Q_p)\times Z_{L_P}^0$ commutes with the diagonal $L_P^0$-action ($L_P^0$ acting on all $3$ factors). So we see that $B_{\sigma, \lambda}(V)$ inherits a locally $\Q_p$-analytic action of $Z_{L_P}(\Q_p)\times Z_{L_P}^0$. In order to avoid confusion, we write $\cZ_0$ for $Z_{L_P}^0$ when it acts on $\cC^{\Q_p-\la}(Z_{L_P}^0, E)$ alone and we use the notation $\Upsilon_0$ for this action. Likewise we write $\cZ_1$ for $Z_{L_P}(\Q_p)$ when it acts on $J_P(V)_{\lambda}$ alone and use the notation $\Upsilon_1$ for this action.

We write $\Delta_0$ for the action of $Z_{L_P}(\Q_p)$ on $B_{\sigma, \lambda}(V)$ induced by the diagonal action of $Z_{L_P}(\Q_p)$ on $J_P(V)_{\lambda} \widehat{\otimes}_E \cC^{\Q_p-\la}(Z_{L_P}^0, E)$ (and the trivial action on $\sigma$), i.e.\ $\Delta_0$ is given by the action of $\cZ_1 \times \cZ_0$ composed with the morphism
\begin{equation}\label{twaa}
	Z_{L_P}(\Q_p) \xlongrightarrow{(\id, (\ref{dett1}))} Z_{L_P}(\Q_p) \times Z_{L_P}^0=\cZ_1 \times \cZ_0.
\end{equation}
Denoting by $\psi_{\sigma}$ the central character of $\sigma$ (a character of $Z_{L_P}(\Q_p)$), we see that the restriction $\Delta_0|_{Z_{L_P}^0}$ on $B_{\sigma, \lambda}(V)$ is given by $\psi_{\sigma}$. We deduce hence for any $z^0\in Z_{L_P}^0$ (see the $(\dett_{L_P})^{-1}$ in (\ref{dett1})):
\begin{equation}\label{gamma10}
\Upsilon_1(z^0)=\psi_{\sigma}(z^0) \Upsilon_0 \big(\dett_{L_P}(z^0)\big).
\end{equation}
With our choice of the uniformizers $\varpi_i$, we have a map $\Z \hookrightarrow F_i^{\times}$, $1 \mapsto \varpi_i$, that induces a map $\oplus_{i\in\sI} \Z^{r_i} \hookrightarrow Z_{L_P}(\Q_p)$. We denote by $Z_{\ul{\varpi}}$ its image. Then the action of $\Delta_0$ is determined by $\Delta_0|_{Z_{\ul{\varpi}}}$ (since $\Delta_0|_{Z_{L_P}^0}$ acts via $\psi_{\sigma}$ and $Z_{L_P}(\Q_p)\cong Z_{L_P}^0 \times Z_{\ul{\varpi}}$). On the other hand, the action of $\cZ_1 \times \cZ_0$ restricts to an action of $Z_{\ul{\varpi}} \times \cZ_0$ on $B_{\sigma, \lambda}(V)$. Since (\ref{dett1}) is trivial on $Z_{\ul{\varpi}}$, using (\ref{twaa}) we see that $\Upsilon_1|_{Z_{\ul{\varpi}}}=\Delta_0|_{Z_{\ul{\varpi}}}$. 

From the natural bijection
\begin{equation*}
	B_{\sigma, \lambda}(V) \cong \Hom_{L_P(\Q_p)}\Big(\cind_{L_P^0}^{L_P(\Q_p)} \sigma, J_P(V)_{\lambda} \widehat{\otimes}_E \cC^{\Q_p-\la}(Z_{L_P}^0,E)\Big)
\end{equation*}
we deduce that $B_{\sigma, \lambda}(V)$ is also equipped with a natural action of $\cZ_{\Omega}\cong \End_{L_P(\Q_p)}(\cind_{L_P^0}^{L_P(\Q_p)} \sigma)$ which commutes with the action of $\cZ_1 \times \cZ_0$. Moreover, we have a natural morphism
\[Z_{L_P}(\Q_p) \lra \End_{L_P(\Q_p)}\big(\cind_{L_P^0}^{L_P(\Q_p)} \sigma\big) \cong \cZ_{\Omega}\]
and it is easy to see that the $Z_{L_P}(\Q_p)$-action on $B_{\sigma, \lambda}(V)$ induced by this map coincides with the $\Delta_0$-action. In particular (from the last assertion in the previous paragraph) the action $\Upsilon_1|_{Z_{\ul{\varpi}}}$ can be recovered from the $\cZ_{\Omega}$-action. With (\ref{gamma10}), we finally see that the action of the full $\cZ_1$ can be read out from the action of $\cZ_0 \times \cZ_{\Omega}$. 

The following lemma is straightforward (using tensor-Hom adjunction).

\begin{lemma}\label{BCmod}
Let $M$ be a finite length $\cZ_{\Omega}$-module over $E$, then we have
	\begin{equation*}
		\Hom_{\cZ_{\Omega}}(M, B_{\sigma, \lambda}(V))\xlongrightarrow{\sim} \Hom_{L_P(\Q_p)}\Big(\big(\cind_{L_P^0}^{L_P(\Q_p)} \sigma\big) \otimes_{\cZ_{\Omega}} M, J_P(V)_{\lambda} \widehat{\otimes}_E \cC^{\Q_p-\la}(Z_{L_P}^0,E)\Big).
	\end{equation*}
\end{lemma}

\begin{lemma}\label{Besadm}
The action of $Z_{\ul{\varpi}} \times \cZ_0$ (seen inside $\cZ_1 \times \cZ_0$) on $B_{\sigma, \lambda}(V)$ makes it an essentially admissible representa\-tion of $Z_{\ul{\varpi}} \times \cZ_0$.
\end{lemma}
\begin{proof}
We first consider $J_P(V)_{\lambda}$. Since $J_P(V)_{\lambda}$ is an essentially admissible representation of $L_P(\Q_p)$ (by \cite[Thm.\ 4.2.32]{Em11} and \cite[Lemma 2.8]{Ding5}), the topological dual (with the strong topology) $J_P(V)_{\lambda}^{\vee}$ is a coadmissible module over $D(H,E) \widehat{\otimes}_E \Gamma\big(\widehat{Z_{L_P}(\Q_p)},\co_{\widehat{Z_{L_P}(\Q_p)}}\big)$ for an arbitrary compact open subgroup $H$ of $L_P(\Q_p)$. Shrinking $H$, we can and do assume $H$ has the form $H\cong H^D \times Z_H$ where $H^D$ (resp.\ $Z_H$) is a compact open subgroup of $L_P^D(\Q_p)$ (resp.\ of $Z_{L_P}(\Q_p)$). Hence $D(H,E) \cong D(H^D,E) \widehat{\otimes}_E D(Z_H,E)$. Since the $H^D$-action on $J_P(V)_{\lambda}$ is smooth, the action of $D(H,E) \widehat{\otimes}_E \Gamma(\widehat{Z_{L_P}(\Q_p)},\co_{\widehat{Z_{L_P}(\Q_p)}})$ on $J_P(V)_{\lambda}^\vee$ factors through
\[D^{\infty}(H^D,E) \widehat{\otimes}_E D(Z_H,E) \widehat{\otimes}_E \Gamma(\widehat{Z_{L_P}(\Q_p)},\co_{\widehat{Z_{L_P}(\Q_p)}}),\]
and further through its quotient
\[D^{\infty}(H^D,E) \widehat{\otimes}_E \Gamma(\widehat{Z_{L_P}(\Q_p)},\co_{\widehat{Z_{L_P}(\Q_p)}})\]
via the embedding $D(Z_H,E)\hookrightarrow \Gamma(\widehat{Z_{L_P}(\Q_p)},\co_{\widehat{Z_{L_P}(\Q_p)}})$.

We now consider $J_P(V)_{\lambda} \widehat{\otimes}_E \cC^{\Q_p-\la}(Z_{L_P}^0,E)$. We denote $Z_{L_P}^0$ by $\cZ_0$ (resp.\ $\cZ_0'$, resp.\ $\cZ_0''$) when it acts on $J_P(V)_{\lambda} \widehat{\otimes}_E \cC^{\Q_p-\la}(Z_{L_P}^0,E)$ by only acting on the term $\cC^{\Q_p-\la}(Z_{L_P}^0,E)$ (resp.\ by only acting on $J_P(V)_{\lambda}$, resp.\ by acting via $Z_{L_P}^0 \xrightarrow{(\id, \det_{L_P}^{-1})} \cZ_0' \times \cZ_0$). In particular the $\cZ_0'$-action coincides with $\Upsilon_1|_{Z_{L_P}^0}$. Note that (\ref{dett1}) defines a trivial action of $H^D \times Z_{\ul{\varpi}}$ on $\cC^{\Q_p-\la}(Z_{L_P}^0,E)$, so the diagonal action of $H^D \times Z_{\ul{\varpi}}$ is the same as the one given by only acting on $J_P(V)_{\lambda}$. It follows from the previous discussion that the strong dual of $J_P(V)_{\lambda} \widehat{\otimes}_E \cC^{\Q_p-\la}(Z_{L_P}^0,E)$ is a coadmissible module over (recall $\cZ_1\cong Z_{L_P}(\Q_p)$ acts via $\Upsilon_1$)
\begin{multline*}D^{\infty}(H^D,E) \widehat{\otimes}_E \Gamma(\widehat{\cZ_1},\co_{\widehat{\cZ_1}})\widehat{\otimes}_E D(\cZ_0,E)\\ 
\cong D^{\infty}(H^D,E) \widehat{\otimes}_E \Gamma(\widehat{Z_{\ul{\varpi}}},\co_{\widehat{Z_{\ul{\varpi}}}}) \widehat{\otimes}_E D(\cZ_0',E) \widehat{\otimes}_E D(\cZ_0,E).
\end{multline*}
Using the group isomorphism
	\begin{equation*}
		\cZ_0'' \times \cZ_0 \xlongrightarrow{\sim} \cZ_0'\times \cZ_0, (a,b) \mapsto (a, b \dett_{L_P}(a)^{-1}), 
	\end{equation*}
	we see that the strong dual of $J_P(V)_{\lambda} \widehat{\otimes}_E \cC^{\Q_p-\la}(Z_{L_P}^0,E)$ is also a coadmissible module over
	\begin{equation*}
		D_0:=D^{\infty}(H^D,E) \widehat{\otimes}_E \Gamma(\widehat{Z_{\ul{\varpi}}},\co_{\widehat{Z_{\ul{\varpi}}}}) \widehat{\otimes}_E D(\cZ_0'',E) \widehat{\otimes}_E D(\cZ_0,E).
	\end{equation*}
	
We now finally consider $B_{\sigma, \lambda}(V)$. Let $M:=\big(\sigma^{\vee} \otimes_E J_P(V)_{\lambda} \widehat{\otimes}_E \cC^{\Q_p-\la}(Z_{L_P}^0,E)\big)^{\vee}$ (a Fr\'echet space). We use $\cZ_0$, $\cZ_0'$, $\cZ_0''$ to denote the corresponding induced action of $Z_{L_P}^0$ on $M$ that acts trivially on $\sigma^\vee$. Shrinking $H$ if necessary, we assume the $H$-action on $\sigma$ is trivial. Using the previous paragraph, we see that $M$ is a coadmissible module over $D_0$. Noting that the action of $\cZ_0''$ coincides with $\Delta_0|_{Z_{L_P}^0}$, $B_{\sigma, \lambda}(V)$ is by definition a direct summand of:
	\begin{equation*}
W:=\big(\sigma^{\vee} \otimes_E J_P(V)_{\lambda} \widehat{\otimes}_E \cC^{\Q_p-\la}(Z_{L_P}^0,E)\big)^{H^D}[\cZ_0''=\psi_{\sigma}].
	\end{equation*}
Endow \ $M{\otimes}_{D_0, \kappa_{\psi_{\sigma}}}\big( \Gamma(\widehat{Z_{\ul{\varpi}}},\co_{\widehat{Z_{\ul{\varpi}}}}) \widehat{\otimes}_E D(\cZ_0,E)\big)$ \ with \ the locally \ convex \ quotient \ topology \ from $M {\otimes}_{E}\big( \Gamma(\widehat{Z_{\ul{\varpi}}},\co_{\widehat{Z_{\ul{\varpi}}}}) \widehat{\otimes}_E D(\cZ_0,E)\big)$ where $\kappa_{\psi_{\sigma}}$ denotes the projection $D_0 \twoheadrightarrow \Gamma(\widehat{Z_{\ul{\varpi}}},\co_{\widehat{Z_{\ul{\varpi}}}}) \widehat{\otimes}_E D(\cZ_0,E)$ which sends $D^{\infty}(H^D,E)$ to $E$ by moding out by the augmentation ideal, which sends $D(\cZ_0'',E)$ to $E$ via the character $\psi_{\sigma}$, and which is the identity for the other factors of $D_0$. Then we have an isomorphism
\begin{equation*}
W^{\vee} \cong M \widehat{\otimes}_{D_0, \kappa_{\psi_{\sigma}}}\big( \Gamma(\widehat{Z_{\ul{\varpi}}},\co_{\widehat{Z_{\ul{\varpi}}}}) \widehat{\otimes}_E D(\cZ_0,E)\big)
\end{equation*}
where the right hand side is the Hausdorff completion of $M{\otimes}_{D_0, \kappa_{\psi_{\sigma}}}\big( \Gamma(\widehat{Z_{\ul{\varpi}}},\co_{\widehat{Z_{\ul{\varpi}}}}) \widehat{\otimes}_E D(\cZ_0,E)\big)$. As $M$ is coadmissible over $D_0$, we deduce that $W^{\vee}$ is coadmissible over $\Gamma(\widehat{Z_{\ul{\varpi}}},\co_{\widehat{Z_{\ul{\varpi}}}}) \widehat{\otimes}_E D(\cZ_0,E)$. Hence $W$ is an essentially admissible locally $\Q_p$-analytic representation of $Z_{\ul{\varpi}} \times \cZ_0$, and so is $B_{\sigma, \lambda}(V)$.
\end{proof}

We use the notation $\fz_0:=\fz_{L_P}$ to emphasize the action on $B_{\sigma, \lambda}(V)$ derived from $\cZ_0$.\index{$\fz_0$} For an $E$-algebra $A$, $\fm \subset A$ an ideal, and an $A$-module $M$, we denote by $M[\fm^{\infty}]:=\varinjlim_n M[\fm^n]$ the $A$-submodule of $M$ consisting of elements annihilated by $\fm^n$ for some $n\geq 0$. \index{$M[\fm^{\infty}]$}

\begin{lemma}\label{compfini}
Let $\fd$ be a weight of $\fz_0$, then we have
\[B_{\sigma,\lambda}(V)[\fz_0=\fd]=\bigoplus_{\delta, \chi}B_{\sigma, \lambda}(V)[\fz_0=\fd][\fm_{\chi}][\fm_{\delta}^{\infty}]=\bigoplus_{\fm \in \Spm \cZ_{\Omega},\chi}B_{\sigma, \lambda}(V)[\fz_0=\fd][\fm_{\chi}][\fm^{\infty}]\] 
where $\delta$ (resp.\ $\chi$) runs through the smooth characters of $\Delta_0\cong Z_{L_P}(\Q_p)$ (resp.\ through the locally algebraic characters of $\cZ_0 \cong Z_{L_P}^0$ of weight $\fd$) and $\fm_{\delta}\subset E[\Delta_0]$ (resp.\ $\fm_{\chi}\subset E[\cZ_0]$) is the maximal ideal associated to $\delta$ (resp.\ $\chi$). Moreover, each summand in the direct sums is finite dimensional over $E$. 
\end{lemma}
\begin{proof}
We have by definition using (\ref{dett1}):
	\begin{multline}\label{disfib0}
		B_{\sigma, \lambda}(V)[\fz_0=\fd]=(J_P(V)_{\lambda} \widehat{\otimes}_E \cC^{\Q_p-\la}(Z_{L_P}^0,E)[\fz_{L_P}=\fd] \otimes_E \sigma^{\vee})^{L_P^0}\\
		=(J_P(V)_{\lambda}[\fz_{L_P}=\fd\circ \dett_{L_P}] \widehat{\otimes}_E \cC^{\Q_p-\la}(Z_{L_P}^0,E)[\fz_{L_P}=\fd] \otimes_E \sigma^{\vee})^{L_P^0}.
	\end{multline}
	As $\cC^{\Q_p-\la}(Z_{L_P}^0,E)[\fz_{L_P}=\fd]\cong \cC^{\infty}(Z_{L_P}^0,E)$ is topologically isomorphism to a direct limit of finite dimensional $E$-vector spaces. We deduce by \cite[Prop.\ 1.2]{Koh2011} that (\ref{disfib0}) remains unchanged if $\widehat{\otimes}_{E}$ is replaced by $\widehat{\otimes}_{E}$. As in Lemma \ref{Besadm}, let $H$ be a compact open subgroup of $L_P^0$ such that $H$ acts trivially on $\sigma$ and $H\cong H^D \times Z_H$ where $H^D:=H \cap L_P^D$ and $Z_H:=H \cap Z_{L_P}^0$. By (\ref{disfib0}), we see that $B_{\sigma, \lambda}(V)[\fz_0=\fd]$ is a $\Delta_0$-equivariant direct summand of
\[\big(J_P(V)_{\lambda}[\fz_{L_P}=\fd\circ \dett_{L_P}] \otimes_E \cC^{\Q_p-\la}(Z_{L_P}^0,E)[\fz_{L_P}=\fd] \otimes_E \sigma^{\vee}\big)^{H}.\]
It thus suffices to prove the statement for the latter. We have 
	\begin{multline*}
		\big(J_P(V)_{\lambda} [\fz_{L_P}=\fd\circ \dett_{L_P}] \otimes_E \cC^{\Q_p-\la}(Z_{L_P}^0,E)[\fz_{L_P}=\fd] \otimes_E \sigma^{\vee}\big)^{H}\\
		\cong \Big(\big(J_P(V)_{\lambda}\otimes_E \sigma^{\vee}\big)^{H^D}[\fz_{L_P}=\fd\circ \dett_{L_P}]\otimes_E \cC^{\Q_p-\la}(Z_{L_P}^0,E)[\fz_{L_P}=\fd]\Big)^{Z_H}.
	\end{multline*}
By \cite[Thm.\ 4.10]{HL}, we have
\[(J_P(V)_{\lambda}\otimes_E \sigma^{\vee}\big)^{H^D}[\fz_{L_P}=\fd\circ \dett_{L_P}]=\bigoplus_{\delta'} (J_P(V)_{\lambda}\otimes_E \sigma^{\vee}\big)^{H^D}[\fz_{L_P}=\fd\circ \dett_{L_P}][\fm_{\delta'}^{\infty}]\]
where $\delta'$ runs though the locally algebraic characters of $\cZ_1\cong Z_{L_P}(\Q_p)$ of weight $\fd$ (and $[\fm_{\delta'}^{\infty}]$ is for the corresponding $\cZ_1$-action). Moreover, by the proof of \textit{loc.\ cit.}\ and the same argument as in the proof of \cite[Prop.\ 4.1]{BH2}, each $\big(J_P(V)_{\lambda}\otimes_E \sigma^{\vee}\big)^{H^D}[\fz_{L_P}=\fd\circ \dett_{L_P}][\fm_{\delta'}^{\infty}]$ is finite dimensio\-nal over $E$. For such $\delta'$, we easily get:
	\begin{eqnarray}\footnotesize
		\label{disfib2}
		&&\Big(\big(J_P(V)_{\lambda}\otimes_E \sigma^{\vee}\big)^{H^D}[\fz_{L_P}=\fd\circ \dett_{L_P}][\fm_{\delta'}^{\infty}]\otimes_E \cC^{\Q_p-\la}(Z_{L_P}^0,E)[\fz_{L_P}=\fd]\Big)^{Z_H} \\
		&\cong &
		\bigoplus_{\delta,\chi}\Big(\big(J_P(V)_{\lambda}\otimes_E \sigma^{\vee}\big)^{H^D}[\fz_{L_P}=\fd\circ \dett_{L_P}][\fm_{\delta'}^{\infty}]\otimes_E \cC^{\Q_p-\la}(Z_{L_P}^0,E)[\fz_{L_P}=\fd]\Big)^{Z_H}[\fm_{\chi}][\fm_{\delta}^{\infty}] \nonumber \\ 
		&\cong& 	\bigoplus_{\fm,\chi}\Big(\big(J_P(V)_{\lambda}\otimes_E \sigma^{\vee}\big)^{H^D}[\fz_{L_P}=\fd\circ \dett_{L_P}][\fm_{\delta'}^{\infty}]\otimes_E \cC^{\Q_p-\la}(Z_{L_P}^0,E)[\fz_{L_P}=\fd]\Big)^{Z_H}[\fm_{\chi}][\fm^{\infty}]\nonumber
	\end{eqnarray}
where $[\fm_{\delta}^{\infty}]$ is for the smooth action of $\Delta_0\cong Z_{L_P}(\Q_p)$, $[\fm_{\chi}]$ for the locally algebraic action of $\cZ_0\cong Z_{L_P}^0$ (with $\chi$ as in the statement) and $[\fm^{\infty}]$ for the smooth action of $\cZ_{\Omega}$ ($\fm$ as in the statement), noting that smooth representations of $\cZ_0$ over $E$ are semi-simple.
By unwinding the actions of $\cZ_1$ and $\Delta_0$, if a summand in the right hand side of (\ref{disfib2}) is non-zero, then $\delta' \delta^{-1}$ is a locally algebraic character of weight $\fd \circ \dett_{L_P}$ which is trivial on $Z_{\ul{\varpi}}$ and $\delta'|_{Z_H}=(\chi \circ \dett_{L_P})|_{Z_H}$. We deduce that for each smooth character $\delta$ of $\Delta_0$ and each locally algebraic character $\chi$ of $\cZ_0$ of weight $\fd$, there exist at most finitely many $\delta'$ such that (\ref{disfib2}) is non-zero. Likewise replacing characters $\delta$ of $\Delta_0$ by maximal ideals $\fm$ of $\cZ_{\Omega}$. The lemma follows. 
\end{proof}

We discuss the problem of the choice of $\sigma$ and $\lambda$. We first introduce some notation.

\begin{notation}\label{charanot}\index{$\delta_{\ul{\varpi}}^0$} \index{$\delta_{\ul{\varpi}}^{\unr}$} \index{$\delta_{\varpi_i}^0$} \index{$\chi_{\ul{\varpi}}$} \index{$\chi_{\varpi_i}$} \index{$\delta_{\varpi_i}^{\unr}$}
(1) For a continuous character $\chi$ of $Z_{L_P}^0$ (resp.\ of $Z_{L_{P_i}}(\co_{F_i})$), denote by $\chi_{\ul{\varpi}}$ (resp.\ $\chi_{\varpi_i}$) the character of $Z_{L_P}(\Q_p)$ (resp.\ of $Z_{L_{P_i}}(F_i)$) that is trivial on $Z_{\ul{\varpi}}$ (resp.\ on $Z_{\varpi_i}=Z_{\ul{\varpi}}\cap Z_{L_{P_i}}(F_i)$) and is equal to $\chi$ on $Z_{L_P}^0$ (resp.\ on $Z_{L_{P_i}}(\co_{F_i})$).

(2) For a continuous character $\delta$ of $Z_{L_P}(\Q_p)$ (resp.\ of $Z_{L_{P_i}}(F_i)$), put $\delta^0:=\delta|_{Z_{L_P}^0}$ (resp.\ $\delta^0:=\delta|_{Z_{L_{P_i}}(\co_{F_i})}$), $\delta^0_{\ul{\varpi}}:=(\delta^0)_{\ul{\varpi}}$ (resp.\ $\delta^0_{\varpi_i}:=(\delta^0)_{\varpi_i}$) and $\delta_{\ul{\varpi}}^{\unr}:=\delta (\delta_{\ul{\varpi}}^0)^{-1}$ (resp.\ $\delta_{\varpi_i}^{\unr}:=\delta (\delta_{\varpi_i}^0)^{-1}$). Hence $\delta_{\ul{\varpi}}^{\unr}$ (resp.\ $\delta_{\varpi_i}^{\unr}$) is an unramified character of $Z_{L_P}(\Q_p)$ (resp.\ $Z_{L_{P_i}}(F_i)$) and $\delta=\delta_{\ul{\varpi}}^{\unr}\delta^0_{\ul{\varpi}}$ (resp.\ $\delta=\delta_{\varpi_i}^{\unr}\delta_{\varpi_i}^0$). 	
\end{notation}

Let $\psi$ be a smooth character of $Z_{L_P}(\Q_p)$ over $E$, $\sigma':=\sigma \otimes_E (\psi^0 \circ \dett_{L_P})$ and $\Omega'$ the Bernstein component associated to $\sigma'$. We have an $L_P(\Q_p)$-equivariant isomorphism
\begin{equation}\label{twistun}
\cind_{L_P^0}^{L_P(\Q_p)} \sigma' \xrightarrow{\sim} \Big(\cind_{L_P^0}^{L_P(\Q_p)} \sigma\Big)\otimes_E (\psi \circ \dett_{L_P}), \ f\mapsto [g \mapsto f(g)\psi^{-1}(\dett_{L_P}(g)) ]\otimes 1.
\end{equation}
Denote by $\iota_{\psi}$ the following composition
\begin{multline*}
\iota_{\psi}:\End_{L_P(\Q_p)} \big(\cind_{L_P^0}^{L_P(\Q_p)} \sigma\big) \xlongrightarrow{\sim} \End_{L_P(\Q_p)} \big((\cind_{L_P^0}^{L_P(\Q_p)} \sigma) \otimes_E (\psi\circ \dett_{L_P})\big)\\ 
	\xlongrightarrow{\sim} \End_{L_P(\Q_p)}\big(\cind_{L_P^0}^{L_P(\Q_p)} \sigma'\big)
\end{multline*} 
where the first map is given by twisting by $\psi\circ \dett_{L_P}$, and the second map is induced by (\ref{twistun}). We easily check that $\iota_{\psi}: \cZ_{\Omega} \ra \cZ_{\Omega'}$ sends $\pi \in \Spec \cZ_{\Omega'}$ to $\pi \otimes_E (\psi^{-1} \circ \dett_{L_P})\in \Spec \cZ_{\Omega}$. Let $\lambda'$ be another integral $P$-dominant weight for $G$ such that there exists an integral weight $\fd$ of $Z_{L_P}(\Q_p)$ with $\lambda'-\lambda=\fd \circ \dett_{L_P}$ (in particular $L(\lambda)_P |_{L_P^D(\Q_p)} \cong L(\lambda')_P|_{L_P^D(\Q_p)}$). Let $\delta_{\fd}$ be the algebraic character of $Z_{L_P}(\Q_p)$ (over $E$) of weight $\fd$. Recall $\delta_{\fd}=\delta_{\fd, \ul{\varpi}}^{\unr}\delta_{\fd,\ul{\varpi}}^0$ from Notation \ref{charanot}.

\begin{lemma}\label{twBE}
There is a natural isomorphism of vector spaces of compact type
	\begin{equation*}	
		\tw_{\psi,\fd}: B_{\sigma',\lambda'}(V) \xlongrightarrow{\sim} B_{\sigma, \lambda}(V)
	\end{equation*}
satisfying the following compatibility for the $\cZ_{\Omega'}\times Z_{\ul{\varpi}}\times \cZ_0$- and $\cZ_{\Omega}\times Z_{\ul{\varpi}}\times \cZ_0$-actions, where $v\in B_{\sigma',\lambda'}(V)$, $\alpha \in \cZ_{\Omega'}$, $\beta\in Z_{\ul{\varpi}}$ and $\gamma\in \cZ_0$:
\begin{equation*}
\tw_{\psi,\fd}\big((\alpha, \beta, \gamma)\cdot v\big)=(\delta_{\fd}\circ \dett_{L_P})(\beta) (\delta_{\fd}\psi)^{-1}(\gamma)\Big(\big(\iota_{\delta_{\fd,\ul{\varpi}}^{\unr}}\big(\iota_{\psi_{\ul{\varpi}}^0}^{-1} (\alpha)\big), \beta, \gamma\big) \cdot \tw_{\psi,\fd}(v)\Big).
\end{equation*}
\end{lemma}
\begin{proof}
We have $\cZ_0$-equivariant isomorphisms (using (\ref{twistun}) for the second):
\begin{multline*}
\tw_0: B_{\sigma',\lambda'}(V) 
		\xlongrightarrow{\sim} \Hom_{L_P(\Q_p)}\Big(\cind_{L_P^0}^{L_P(\Q_p)} \sigma', J_P(V)_{\lambda} \otimes_E (\delta_{\fd}^{-1} \circ \dett_{L_P}) \widehat{\otimes}_E \cC^{\Q_p-\la}(Z_{L_P}^0, E)\Big)\\ 
		\xlongrightarrow{\sim} \Hom_{L_P(\Q_p)}\Big(\cind_{L_P^0}^{L_P(\Q_p)} \sigma, J_P(V)_{\lambda} \otimes_E (\delta_{\fd}^{-1}\circ \dett_{L_P}) \widehat{\otimes}_E \cC^{\Q_p-\la}(Z_{L_P}^0, E)\otimes_E (\psi^{-1}\circ \dett_{L_P})\Big)	
\end{multline*}
such that $\tw_0((\alpha,\beta) \cdot v)=(\psi\circ \dett_{L_P})(\beta)\big((\iota_{\psi}^{-1}(\alpha),\beta) \cdot \tw_0(v)\big)$. We have an isomorphism
\begin{eqnarray}\label{twO2}
		((\delta_{\fd}^{-1}\psi^{-1})_{\ul{\varpi}}^0 \circ \dett_{L_P}) \widehat{\otimes}_E \cC^{\Q_p-\la}(Z_{L_P}^0, E) &\xlongrightarrow{\sim} &\cC^{\Q_p-\la}(Z_{L_P}^0, E)\\
\nonumber		1\otimes f &\longmapsto &[z \mapsto f(z)(\delta_{\fd}\psi)^0(z)]
\end{eqnarray}
which is $L_P(\Q_p)$-equivariant (but not $\cZ_0$-equivariant) where $L_P(\Q_p)$ acts on the left hand side by the diagonal action (with $L_P(\Q_p)$ acting on $\cC^{\Q_p-\la}(Z_{L_P}^0,E)$ via (\ref{dett1})). This isomorphism, together with $\tw_0$, induce an isomorphism 
\begin{equation*}
\tw_1: B_{\sigma',\lambda'} \xrightarrow{\sim}\Hom_{L_P(\Q_p)}\Big(\cind_{L_P^0}^{L_P(\Q_p)} \sigma, J_P(V)_{\lambda} \otimes_E ((\delta_{\fd}^{-1}\psi^{-1})_{\ul{\varpi}}^{\unr} \circ \dett_{L_P}) \widehat{\otimes}_E \cC^{\Q_p-\la}(Z_{L_P}^0, E)\Big)
\end{equation*}
satisfying $\tw_1((\alpha, \beta, \gamma) \cdot v)=(\psi\circ \dett_{L_P})(\beta)(\delta_{\fd}^{-1}\psi^{-1})(\gamma) \big((\iota_{\psi}^{-1}(\alpha),\beta,\gamma) \cdot \tw_1(v)\big)$ for $(\alpha,\beta,\gamma)\in \cZ_{\Omega'} \times Z_{\ul{\varpi}}\times \cZ_0$. We have isomorphisms:
\begin{multline*}
\tw_2: \Hom_{L_P(\Q_p)}\Big(\cind_{L_P^0}^{L_P(\Q_p)} \sigma, J_P(V)_{\lambda} \otimes_E ((\delta_{\fd}^{-1}\psi^{-1})_{\ul{\varpi}}^{\unr} \circ \dett_{L_P}) \widehat{\otimes}_E \cC^{\Q_p-\la}(Z_{L_P}^0,E)\Big)\\
\xlongrightarrow{\sim} \Hom_{L_P(\Q_p)}\Big((\cind_{L_P^0}^{L_P(\Q_p)} \sigma) \otimes_E ((\delta_{\fd}\psi)_{\ul{\varpi}}^{\unr} \circ \dett_{L_P}), J_P(V)_{\lambda} \widehat{\otimes}_E \cC^{\Q_p-\la}(Z_{L_P}^0, E)\Big)\\
\xlongrightarrow{\sim} \Hom_{L_P(\Q_p)}\Big(\cind_{L_P^0}^{L_P(\Q_p)} \sigma, J_P(V)_{\lambda} \widehat{\otimes}_E \cC^{\Q_p-\la}(Z_{L_P}^0, E)\Big)
	\end{multline*}
where the second isomorphism is induced by (\ref{twistun}) (for the unramified character $(\delta_{\fd}\psi)_{\ul{\varpi}}^{\unr}$) and $\tw_2$ means the composition. It is straightforward to check that $\tw_2$ is $\cZ_0$-equivariant and satisfies for $(\alpha, \beta)\in \cZ_{\Omega}\times Z_{\ul{\varpi}}$:
\[\tw_2((\alpha, \beta) \cdot v)=((\delta_{\fd}^{-1} \psi^{-1})^{\unr}_{\ul{\varpi}}\circ \dett_{L_P})(\beta) \big((\iota_{(\delta_{\fd}\psi)^{\unr}_{\ul{\varpi}}}(\alpha), \beta) \cdot \tw_2(v)\big).\]
We put $\tw_{\psi,\fd}:=\tw_2\circ \tw_1$. {\it Forgetting the actions of $\cZ_{\Omega}$ and $\cZ_{\Omega'}$}, we see that $\tw_{\psi,\fd}$ is given by the following composition
\begin{multline*}
		B_{\sigma',\lambda'}(V)\cong \Big((\sigma')^{\vee}\otimes_E J_P(V)_{\lambda} \otimes_E (\delta_{\fd}^{-1} \circ \dett_{L_P}) \widehat{\otimes}_E \cC^{\Q_p-\la}(Z_{L_P}^0, E)\Big)^{L_P^0}\\
		\xlongrightarrow{\sim} \Big(\sigma^{\vee}\otimes_E J_P(V)_{\lambda} \otimes_E ((\delta_{\fd,\ul{\varpi}}^{\unr})^{-1}\circ \dett_{L_P}) \otimes_E ((\delta_{\fd,\ul{\varpi}}^0 \psi_{\ul{\varpi}}^0)^{-1}\circ \dett_{L_P}) \widehat{\otimes}_E \cC^{\Q_p-\la}(Z_{L_P}^0, E)\Big)^{L_P^0}\\
		\xlongrightarrow{\sim} \Big(\sigma^{\vee}\otimes_E J_P(V)_{\lambda} \otimes_E ((\delta_{\fd,\ul{\varpi}}^{\unr})^{-1}\circ \dett_{L_P}) \widehat{\otimes}_E \cC^{\Q_p-\la}(Z_{L_P}^0, E)\Big)^{L_P^0}\\ 
		\cong \Big(\sigma^{\vee}\otimes_E J_P(V)_{\lambda} \widehat{\otimes}_E \cC^{\Q_p-\la}(Z_{L_P}^0, E)\Big)^{L_P^0}\cong B_{\sigma,\lambda}(V)
	\end{multline*}
	where the third isomorphism is induced by (\ref{twO2}). In particular, $\tw_{\psi,\fd}$ is a topological isomorphism.
\end{proof}	

Let $(J_{i,j}', (\sigma_{i,j}^0)')$ be another maximal simple type of $\Omega_{i,j}$ and $K'_{i,j}$ be a maximal compact open subgroup of $\GL_{n_{i,j}}(F_i)$ containing $J_{i,j}'$. Define $(\sigma')^0:=\boxtimes_{i,j} (\sigma_{i,j}^0)'$, $\sigma_{i,j}':=\cind_{J_{i,j}'}^{K'_{i,j}} (\sigma_{i,j}^0)'$, $\sigma':=\boxtimes_{i,j} \sigma_{i,j}'$ and $K':=\prod_{i,j} K'_{i,j}$ (so that $\sigma'$ is an absolutely irreducible representation of $K'$ over $E$). By \cite[Cor.\ 7.6]{BK91}, $(J', (\sigma')^0)$ is conjugate to $(J, \sigma^0)$, i.e.\ there exist $h\in L_P(\Q_p)$ such that $J'=hJh^{-1}$ and a bijection $\iota_h: \sigma^0 \xrightarrow{\sim} (\sigma^0)'$ satisfying $\iota_h(\alpha v)=(h\alpha h^{-1}) \iota_h(v)$ for $\alpha\in J$, $v\in \sigma^0$. The morphism $\iota_h$ induces a bijection:
\begin{equation*}
	\tilde{\iota}_h: \cind_J^{L_P(\Q_p)} \sigma^0 \xrightarrow{\sim} \cind_{J'}^{L_P(\Q_p)} (\sigma')^0, f \mapsto [g \mapsto \iota_h(f(h^{-1}gh))]
\end{equation*}
satisfying $\tilde{\iota}_h(g f)=(hgh^{-1}) (\tilde{\iota}_h(f))$ for $g\in L_P(\Q_p)$. We deduce isomorphisms
\[\cZ_{\Omega}\cong \End_{L_P(\Q_p)}(\cind_J^{L_P(\Q_p)} \sigma^0)\cong \End_{L_P(\Q_p)}(\cind_{J'}^{L_P(\Q_p)} (\sigma')^0).\] We fix such isomorphisms in the rest of this paragraph. We define:
\[B_{\sigma',\lambda}(V):=\big((\sigma')^{\vee} \otimes_E J_P(V)_{\lambda} \widehat{\otimes}_E \cC^{\Q_p-\la}(Z_{L_P}^0, E)\big)^{K'}\]
where $K'$ acts diagonally on $J_P(V)_{\lambda} \widehat{\otimes}_E \cC^{\Q_p-\la}(Z_{L_P}^0, E)$ via the embedding $K'\hookrightarrow L_P(\Q_p)$ and (\ref{dett1}). We endow $B_{\sigma',\lambda}(V)$ with an action of $\cZ_0\cong Z_{L_P}^0$, $Z_{\ul{\varpi}}$ and $\cZ_{\Omega}$ as in \S~\ref{secAbsOL}.

\begin{lemma}\label{diftyp}
There exists a natural isomorphism of vector spaces of compact type which is equivariant under the action of $\cZ_{\Omega}\times Z_{\ul{\varpi}} \times \cZ_0$:
	\begin{equation*}
		B_{\sigma, \lambda}(V) \cong B_{\sigma',\lambda}(V).
	\end{equation*}
\end{lemma}
\begin{proof}
We have a $\cZ_{\Omega}\times Z_{\ul{\varpi}} \times \cZ_0$-equivariant isomorphism
	\begin{equation*}
		B_{\sigma, \lambda}(V)\cong \Hom_{J}(\sigma^0, J_P(V)_{\lambda} \widehat{\otimes}_E \cC^{\Q_p-\la}(Z_{L_P}^0,E))
	\end{equation*}
and a similar isomorphism for $B_{\sigma', \lambda}(V)$. Let $W:= J_P(V)_{\lambda} \widehat{\otimes}_E \cC^{\Q_p-\la}(Z_{L_P}^0,E)$, one can check that the following diagram commutes
	\begin{equation*}
		\begin{CD}
			\Hom_{J}(\sigma^0, W) @>>> \Hom_{J'}((\sigma')^0, W) \\
			@V \wr VV @V \wr VV \\
			\Hom_{L_P(\Q_p)}\big(\cind_J^{L_P(\Q_p)}\sigma^0, W\big) @>>> \Hom_{L_P(\Q_p)}\big(\cind_{J'}^{L_P(\Q_p)}(\sigma')^0,W\big), 
		\end{CD}
	\end{equation*}
where the top map is given by $f \mapsto [v\mapsto h (f(\iota_h^{-1}(v)))]$, and the bottom map is given by $F \mapsto [w \mapsto h(F(\tilde{\iota}_h^{-1}(w)))]$. One sees moreover that the top map is a topological isomorphism, that the bottom map is bijective and $\cZ_{\Omega}\times Z_{\ul{\varpi}} \times \cZ_0$-equivariant. The lemma follows.
\end{proof}

By Lemma \ref{diftyp}, the $\cZ_{\Omega}\times Z_{\ul{\varpi}} \times \cZ_0$-module $B_{\sigma, \lambda}(V)$ does not depend on the choice $\sigma$. We denote hence $B_{\Omega,\lambda}(V):=B_{\sigma, \lambda}(V)$ in the sequel.\index{$B_{\Omega, \lambda}(V)$} 

\subsubsection{Local Bernstein eigenvarieties}\label{localbernstein}

We keep the setting of \S~\ref{secAbsOL}. We construct certain rigid analytic spaces parametrizing irreducible $\cZ_{\Omega}\times \cZ_0$-submodules of $B_{\Omega, \lambda}(V)$. 

By Lemma \ref{Besadm} and \cite[\S~6.4]{Em04}, there exists a coherent sheaf $\cM_{\Omega, \lambda}(V)$ over the rigid analytic space $\widehat{Z_{\ul{\varpi}}} \times \widehat{\cZ_0}$ such that $B_{\Omega, \lambda}(V)$ is isomorphic to the global sections of $\cM_{\Omega, \lambda}(V)$. Moreover, $\cM_{\Omega, \lambda}(V)$ is equipped with an $\co_{\widehat{Z_{\ul{\varpi}}} \times \widehat{\cZ_0}}$-linear action of $\cZ_{\Omega}$. Using the fact that the action of $Z_{\ul{\varpi}}$ factors through $\cZ_{\Omega}$ (see before Lemma \ref{BCmod}), we conclude that $\cM_{\Omega, \lambda}(V)$ gives rise to a coherent sheaf, still denoted by $\cM_{\Omega, \lambda}(V)$, over $(\Spec \cZ_{\Omega})^{\rig} \times \widehat{\cZ_0}$ such that 
\begin{equation}
	\Gamma\big((\Spec \cZ_{\Omega})^{\rig} \times \widehat{\cZ_0}, \cM_{\Omega, \lambda}(V)\big)\cong B_{\Omega, \lambda}(V)^{\vee}.
\end{equation}
We let $\Supp \cM_{\Omega, \lambda}(V)$ be the Zariski-closed support of $\cM_{\Omega, \lambda}(V)$ (defined by the annihilator of $\cM_{\Omega, \lambda}(V)$).

\begin{proposition}\label{pts1}
Let $x=(\pi_{L_P}, \chi):=((\pi_{i,j}), (\chi_{i,j}))\in (\Spec \cZ_{\Omega})^{\rig} \times \widehat{\cZ_0}$. Then we have a bijection of $k(x)$-vector spaces ($\chi_{\ul{\varpi}}$ as in Notation \ref{charanot}):
	\begin{equation*}
		\big(x^* \cM_{\Omega, \lambda}(V)\big)^{\vee} \cong \Hom_{L_P(\Q_p)}\big(\pi_{L_P}\otimes_{k(x)} (\chi_{\ul{\varpi}} \circ \dett_{L_P}) \otimes_E L(\lambda)_P, J_P(V)\big).
	\end{equation*}
	In particular, $(\pi_{L_P}, \chi)\in \Supp \cM_{\Omega, \lambda}(V)$ if and only if there is an $L_P(\Q_p)$-equivariant embedding
	\begin{equation*}
		\pi_{L_P}\otimes_{k(x)} (\chi_{\ul{\varpi}} \circ \dett_{L_P}) \otimes_E L(\lambda)_P \hooklongrightarrow J_P(V).
	\end{equation*}
\end{proposition}
\begin{proof}
By definition, we have a bijection 
	\begin{equation*}
		\big(x^* \cM_{\Omega, \lambda}(V)\big)^{\vee} \cong B_{\Omega, \lambda}(V)[\cZ_{\Omega}=\pi_{L_P}, \cZ_0=\chi].
	\end{equation*}
By unwinding $B_{\Omega, \lambda}(V)$ ($\cong B_{\sigma, \lambda}(V)$), we see that the right hand side is isomorphic to (denoting by $\fm_{\pi_{L_P}}$ the maximal ideal of $\cZ_{\Omega}$ corresponding to $\pi_{L_P}$)
\begin{multline*}
\Hom_{\cZ_{\Omega}}\Big( \cZ_{\Omega}/\fm_{\pi_{L_P}}, \Hom_{L_P(\Q_p)}\big(\cind_{L_P^0}^{L_P(\Q_p)} \sigma, J_P(V)_{\lambda} \otimes_{k(x)} (\chi_{\ul{\varpi}}^{-1}\circ \dett_{L_P})\big)\Big) \\ 
\ \ \ \ \ \ \ \ \ \ \ \ \ \ \ \ \ \ \ \ \ \ \ \ \ \ \ \ \ \ \ \ \ \ \ \ \ \ \ \ \xlongrightarrow{\sim} \Hom_{L_P(\Q_p)}\big(\pi_{L_P}, J_P(V)_{\lambda} \otimes_{k(x)} (\chi_{\ul{\varpi}}^{-1} \circ \dett_{L_P})\big)\\
\ \ \ \ \ \ \ \ \ \ \ \ \ \ \ \ \ \ \ \ \ \ \ \ \ \ \ \ \ \ \ \ \ \ \ \ \ \ \cong \Hom_{L_P(\Q_p)}\big(\pi_{L_P}\otimes_{k(x)} (\chi_{\ul{\varpi}}\circ \dett_{L_P}), J_P(V)_{\lambda}\big)\\ 
\ \ \ \ \ \ \ \ \ \ \ \ \ \ \ \ \ \ \ \ \ \ \ \ \ \ \ \ \ \ \ \ \ \ \ \ \ \ \ \ \ \ \ \ \ \ \ \ \ \ \cong \Hom_{L_P(\Q_p)}\big(\pi_{L_P}\otimes_{k(x)} (\chi_{\ul{\varpi}} \circ \dett_{L_P}), J_P(V)\otimes_{E} L(\lambda)_P^{\vee}\big)\\
\ \ \ \ \ \ \ \ \ \ \ \ \ \ \ \ \ \ \ \ \ \ \ \ \ \ \ \ \ \ \ \ \ \ \ \ \ \ \ \ \ \ \ \ \ \ \ \ \ \ \ \ \ \ \ \ \ \ \ \ \ \cong \Hom_{L_P(\Q_p)}\big(\pi_{L_P}\otimes_{k(x)} (\chi_{\ul{\varpi}} \circ \dett_{L_P}) \otimes_E L(\lambda)_P, J_P(V)\big)
\end{multline*}
where the first isomorphism follows from Lemma \ref{BCmod}, the second is obvious, the third follows from (\ref{jplambda}) and the fact that $\pi_{L_P}\otimes_E (\chi_{\ul{\varpi}}\circ \dett_{L_P})$ is smooth for the $L_P^D(\Q_p)$-action, and the fourth is easily induced by the natural map $L(\lambda)_P \otimes_E L(\lambda)_P^{\vee} \ra E$. The proposition follows.
\end{proof}

The following proposition easily follows from Lemma \ref{twBE}:

\begin{proposition}\label{twBE2}
With the setting of Lemma \ref{twBE}, we have an isomorphism of rigid analytic spaces 
\begin{equation*}
\Supp \cM_{\Omega, \lambda}(V) \xrightarrow{\sim} \Supp \cM_{\Omega', \lambda'}(V), \ \ (\pi_{L_P},\chi) \mapsto \big(\pi_{L_P} \otimes_E \big((\psi^{0}_{\ul{\varpi}}(\delta_{\fd, \ul{\varpi}}^{\unr})^{-1})\circ \dett_{L_P}\big), \chi (\psi^0)^{-1}(\delta_{\fd}^0)^{-1}\big).
\end{equation*}
\end{proposition}
\begin{remark}\label{remP=B}
	Assume $P=B_{\sI}$, we have $\cZ_{\Omega}\cong Z_{\ul{\varpi}}$, $L_P(\Q_p)\cong T_{\sI}(\Q_p)$ and $L(\lambda)_P=\delta_{\lambda}$. Using the isomorphism 
	\begin{equation*}
	\iota_{\Omega, \lambda}: (\Spec \cZ_{\Omega})^{\rig} \times \widehat{\cZ_0} \xlongrightarrow{\sim}\widehat{Z_{L_P}(\Q_p)}\cong \widehat{T_{\sI}(\Q_p)},\ \ (\pi_{L_P}, \chi) \mapsto \pi_{L_P}\chi_{\ul{\varpi}}\delta_{\lambda}
	\end{equation*}
	we can also view $\cM_{\Omega, \lambda}(V)$ as a coherent sheaf over $\widehat{T_{\sI}(\Q_p)}$ and $\Supp \cM_{\Omega, \lambda}(V)$ as a closed rigid subspace of $\widehat{T_{\sI}(\Q_p)}$. By Proposition \ref{twBE2}, the resulting rigid subspaces $\Supp \cM_{\Omega, \lambda}(V)$ of $\widehat{T_{\sI}(\Q_p)}$ for different $(\Omega,\lambda)$ are all the same. Moreover, in this case, one can directly associate to the (essentially admissible) locally analytic representation $J_P(V)$ of $T_{\sI}(\Q_p)$ a coherent sheaf $\cM(V)$ over $\widehat{T_{\sI}(\Q_p)}$ without using $L(\lambda)_P$ and without tensoring by the factor $\cC^{\Q_p-\la}(Z_{L_P}^0, E)$ such that $\Gamma(\widehat{T_{\sI}(\Q_p)}, \cM(V))\cong J_P(V)^{\vee}$. For each point $x=\delta$ of $\widehat{T_{\sI}(\Q_p)}$, the fibre $(x^* \cM(V))^{\vee}$ is naturally isomorphic to $J_P(V)[Z_{L_P}(\Q_p)=\delta]$. By Proposition \ref{pts1} (and using $\iota_{\Omega, \lambda}$), we see that $\cM(V)$ and $\cM_{\Omega, \lambda}(V)$ have isomorphic fibres (as $E$-vector spaces) at each point. 
\end{remark}

Set (see (\ref{muOmega}) for $\mu_{\Omega_{i,j}}$):
\[\mu_{\Omega}:=\big\{(\psi_{i,j})_{\!\!\substack{i\in \sI\\ j=1, \dots, r_i}}: Z_{L_P}(\Q_p) \ra E^{\times}\ |\ \psi_{i,j}\in \mu_{\Omega_{i,j}} \big\}\index{$\mu_{\Omega}$},\]
we define an action of $\mu_{\Omega}$ on $(\Spec \cZ_{\Omega})^{\rig} \times \widehat{\cZ_0}$ such that $\psi=(\psi_{i,j})\in \mu_{\Omega}$ sends $((\pi_{i,j}), (\chi_{i,j}))$ to
\[\Big(\big(\pi_{i,j} \otimes_E \unr(\psi_{i,j}(\varpi_i))\big), \big(\chi_{i,j} \psi_{i,j}|_{\co_{F_i}^{\times}}\big)\Big).\]
By Proposition \ref{twBE2} (or by Proposition \ref{pts1}), we have

\begin{corollary}\label{cor:muom0}
For $x \in(\Spec \cZ_{\Omega})^{\rig} \times \widehat{\cZ_0}$, we have $x\in \Supp \cM_{\Omega, \lambda}(V)$ if and only if $\psi(x)\in \Supp \cM_{\Omega, \lambda}(V)$ for one (or any) $\psi\in \mu_{\Omega}$.
\end{corollary}

Finally we show that, under certain assumptions, $\Supp \cM_{\Omega, \lambda}(V)$ is closely related to Fredholm hypersurfaces. We first unwind a bit the definition of the Jacquet-Emerton modules.

Let $H$ be a compact open subgroup of $G_p=G(\Q_p)$, and $N_H:=H\cap N_P(\Q_p)$. Let $L_P(\Q_p)^+=\{z \in L_P(\Q_p)\ |\ z N_H z^{-1} \subseteq N_H\}$. Recall that $V^{N_H}$ is equipped with a natural Hecke action of $ L_P(\Q_p)^+$ given by 
\begin{equation}\label{Upheck}
	\pi_z(v)=\frac{1}{|N_H/(zN_Hz^{-1})|}\sum_{g\in N_H/(zN_Hz^{-1})} gz(v). 
\end{equation}
Then $J_P(V)$ is defined to be $(V^{N_H})_{\fss}$, where $(-)_{\fss}$ denotes the finite slope part functor for the action of $Z_{L_P}(\Q_p)^+:= L_P(\Q_p)^+ \cap Z_{L_P}(\Q_p)$ (cf.\ \cite[\S~3.2]{Em11}).

\begin{lemma}\label{lemSLnalg}
We have $J_P(V)_{\lambda} \cong \big((V^{N_H} \otimes_E L(\lambda)_P^{\vee})^{\fl_P^D}\big)_{\fss}\cong \big((V^{N_H} \otimes_E L(\lambda)_P^{\vee})_{\fss}\big)^{\fl_P^D}$.
\end{lemma}
\begin{proof}
Recall $J_P(V)_{\lambda}\cong \big(J_P(V)\otimes_E L(\lambda)_P^{\vee}\big)^{\fl_P^D}$ by (\ref{jplambda}). By \cite[Prop.\ 3.2.9]{Em11}, we have $J_P(V)\otimes_E L(\lambda)_P^{\vee} \cong \big(V^{N_H} \otimes_E L(\lambda)_P^{\vee}\big)_{\fss}$, where the Hecke action of $z \in Z_{L_P}(\Q_p)^+$ on $V^{N_H} \otimes_E L(\lambda)_P^{\vee}$ is given by $\pi_z\otimes z$. It is clear that the action of $\fl_P^D$ commutes with $Z_{L_P}(\Q_p)^+$. By \cite[Prop.\ 3.2.11]{Em11}, the lemma follows.
\end{proof}

We fix $\sigma$ as in (\ref{BslV}) and recall that we have an isomorphism $B_{\Omega, \lambda}(V)\cong B_{\sigma, \lambda}(V)$.

\begin{lemma}
We have an isomorphism of locally analytic representations of $Z_{\ul{\varpi}} \times \cZ_0$:
	\begin{equation*}
		B_{\Omega, \lambda}(V) \cong \Big(\big((V^{N_H} \otimes_E L(\lambda)_P^{\vee})\widehat{\otimes}_E \cC^{\Q_p-\la}(Z_{L_P}^0, E) \otimes_E\sigma^{\vee}\big)^{L_P^0}\Big)_{\fss},
	\end{equation*}
where $Z_{\ul{\varpi}}$ (resp.\ $\cZ_0$) acts on the RHS by its action on $V^{N_H} \otimes_E L(\lambda)_P^\vee$ (resp.\ on $\cC^{\Q_p-\la}(Z_{L_P}^0, E)$).
\end{lemma}
\begin{proof}
We have isomorphisms
\begin{eqnarray*}
B_{\Omega, \lambda}(V) &\cong &\Big(\big((V^{N_H} \otimes_E L(\lambda)_P^{\vee})^{\fl_P^D}\big)_{\fss}\widehat{\otimes}_E \cC^{\Q_p-\la}(Z_{L_P}^0, E) \otimes_E \sigma^{\vee}\Big)^{L_P^0} \\
&\cong &\Big(\big((V^{N_H} \otimes_E L(\lambda)_P^{\vee})^{\fl_P^D}\widehat{\otimes}_E \cC^{\Q_p-\la}(Z_{L_P}^0, E) \otimes_E \sigma^{\vee}\big)_{\fss}\Big)^{L_P^0}\\ 
&\cong &\Big(\big((V^{N_H} \otimes_E L(\lambda)_P^{\vee})^{\fl_P^D}\widehat{\otimes}_E \cC^{\Q_p-\la}(Z_{L_P}^0, E) \otimes_E \sigma^{\vee}\big)^{L_P^0}\Big)_{\fss} \\
&\cong &\Big(\big((V^{N_H} \otimes_E L(\lambda)_P^{\vee})\widehat{\otimes}_E \cC^{\Q_p-\la}(Z_{L_P}^0, E) \otimes_E \sigma^{\vee}\big)^{L_P^0}\Big)_{\fss},
\end{eqnarray*}
where the first isomorphism follows from Lemma \ref{lemSLnalg}, the second from \cite[Prop.\ 3.2.9]{Em11}, the third from \cite[Prop.\ 3.2.11]{Em11}, and the last from the fact that $\fl_P^D$ acts trivially on the factors $\cC^{\Q_p-\la}(Z_{L_P}^0,E)$ and $\sigma^{\vee}$.
\end{proof}

\begin{remark}
Let $z\in Z_{L_P}(\Q_p)^+$ and let $Y_z$ be the subgroup of $Z_{L_P}(\Q_p)$ generated by $z$. Assume $Y_zZ_{L_P}(\Q_p)^+=Z_{L_P}(\Q_p)$. By \cite[Prop.\ 3.2.27]{Em11}, the lemmas also hold with ``$(-)_{\fss}$" (for the whole group $Z_{L_P}(\Q_p)$) replaced by the finite slope part ``$(-)_{Y_z-\fss}$" for $Y_z$. 
\end{remark}

We now take $H$ uniform pro-$p$ in $G_p$ and $z\in Z_{\ul{\varpi}}$ such that 
\begin{itemize}
	\item[(1)] $H \cong (N_{P^-}(\Q_p) \cap H) \times (L_P(\Q_p) \cap H) \times (N_P(\Q_p) \cap H)=: N_H^- \times L_H \times N_H$, and $H$ is normalized by $L_P^0$;
	\item[(2)] $L_H \cong (L_H \cap L_P^D(\Q_p)) \times (L_H \cap Z_{L_P}(\Q_p))=: L_H^D \times Z_{L_H}$;
	\item[(3)] $L_H$ acts trivially on $\sigma$; 
	\item[(4)] $z\in Z_{L_P}(\Q_p)^+$ and satisfies $\cap_n (z^n N_H z^{-n})=0$, $N_H^- \subset z N_H^- z^{-1}$ and $Y_zZ_{L_P}(\Q_p)^+=Z_{L_P}(\Q_p)$ where $Y_z$ is the subgroup of $Z_{L_P}(\Q_p)$ generated by $z$;
	\item[(5)] $H$ is normalized by $z^{-1} N_H z$.
\end{itemize}
The existence of $H$ satisfying (1), (2) and (3) is clear. It is also clear that there exists $z\in Z_{L_P}(\Q_p)^+$ such that (4) holds. By multiplying $z$ by an element in $Z_{L_P}^0$, we can take $z \in Z_{\ul{\varpi}}$. Finally, replacing $H$ by $H^{p^m}$ for some $m \geq 1$, (5) also holds (with the other properties unchanged). As an example, one can take $H$ to be $\prod_{i \in \sI} (1+\varpi_i^k M_n(\co_{F_i}))$ for $k$ sufficiently large, and $z=\prod_{i\in \sI} z_i$ with \[z_i:=\diag\big(\underbrace{\varpi_i^{r_i-1}, \dots, \varpi_i^{r_i-1}}_{n_{i,1}}, \underbrace{\varpi_i^{r_i-2}, \dots,\varpi_i^{r_i-2}}_{n_{i,2}}, \dots,\underbrace{1, \dots,1}_{n_{i,r_i}}\big)\in Z_{L_{P_i}}(L_i).\]
By (1) and (5), one can deduce that 
\begin{equation*}
	H':=(zHz^{-1}) N_H \cong (zN_H^-z^{-1}) \times L_H \times N_H
\end{equation*}
contains $H$ (note that $H'$ is also an open uniform pro-$p$ subgroup of $G_p$). The following proposition is analogous to \cite[Prop.\ 5.3]{BHS1} (see also \cite[Prop.\ 4.2.36]{Em11}).

\begin{proposition}\label{Pr}
	Assume that $V|_{H'} \cong \cC^{\Q_p-\la}(H',E)^{\oplus k}$ for some $k\in \Z_{\geq 1}$. There exist an admissible covering of $\widehat{\cZ_0}$ by affinoid opens $\cU_1 \subset \cU_2 \subset \cdots \subset \cU_h \subset \cdots$ and the following data for any $h\geq 1$ where $A_h:=\co_{\widehat{\cZ_0}}(\cU_h)$:
	\begin{itemize}
		\item a Banach $A_h$-module $M_h$ satisfying the condition (Pr) of \cite{Bu};
		\item an $A_h$-linear compact operator, denoted by $z_h$, on the $A_h$-module $M_h$;
		\item $A_h$-linear continuous maps $\begin{cases}
				\alpha_h: M_h \lra M_{h+1} \widehat{\otimes}_{A_{h+1}} A_h \\
				\beta_h: M_{h+1} \widehat{\otimes}_{A_{h+1}} A_h \lra M_h
			\end{cases}$
		such that $\beta_h \circ \alpha_h=z_h$ and $\alpha_h \circ \beta_h=z_{h+1}$ with $\beta_h$ $A_h$-compact;
		\item a topological $\co(\widehat{\cZ_0})$-linear isomorphism 
		\begin{equation*}
			M:= \Big(\big((V^{N_H} \otimes_E L(\lambda)_P^{\vee})\widehat{\otimes}_E \cC^{\Q_p-\la}(Z_{L_P}^0, E) \otimes_E\sigma^{\vee}\big)^{L_P^0} \Big)^{\vee} \cong \varprojlim_h M_h
		\end{equation*}
commuting with the action induced by $(\pi_z\otimes z)\otimes 1\otimes 1$ on the LHS and the action of $(z_h)_{h\geq 1}$ on the RHS.
	\end{itemize}
One can visualize all the above conditions in the following commutative diagram
	\begin{equation*}
		\begindc{\commdiag}[40]
		\obj(0,10)[a]{$M$}
		\obj(12,10)[b]{$\cdots$}
		\obj(25,10)[c]{$M_{h+1}$}
		\obj(48,10)[d]{$M_{h+1}\otimes_{A_{h+1}} A_h$}
		\obj(68,10)[e]{$M_h$}
		\obj(80,10)[f]{$\cdots$}
		\obj(0,0)[a']{$	M$}
		\obj(12,0)[b']{$\cdots$}
		\obj(25,0)[c']{$M_{h+1}$}
		\obj(48,0)[d']{$M_{h+1}\otimes_{A_{h+1}} A_h$}
		\obj(68,0)[e']{$M_h$}
		\obj(80,0)[f']{$\cdots$}
		\mor{a}{b}{}
		\mor{b}{c}{}
		\mor{c}{d}{}
		\mor{d}{e}{$\beta_h$}
		\mor{a'}{b'}{}
		\mor{b'}{c'}{}
		\mor{c'}{d'}{}
		\mor{d'}{e'}{$\beta_h$}
		\mor{a}{a'}{$\pi_z$}
		\mor{c}{c'}{$z_{h+1}$}
		\mor{d}{d'}{$z_{h+1}\otimes 1_{A_h}$}[\atright,\solidarrow]
		\mor{e}{e'}{$z_h$}
		\mor{e}{d'}{$\alpha_h$}
		\mor{e}{f}{}
		\mor{e'}{f'}{}
		\enddc.
	\end{equation*}
\end{proposition}
\begin{proof}

We put
	\begin{equation}\label{Wfun}
		W:=\big((V^{N_H} \otimes_E L(\lambda)_P^{\vee})\widehat{\otimes}_E \cC^{\Q_p-\la}(Z_{L_P}^0, E) \otimes_E\sigma^{\vee}\big)^{L_H}
	\end{equation}
and $N:=W^{\vee}$. By definition, $M$ is a direct summand of $N$ equivariant under the action of $\cZ_0$ and of $(\pi_z\otimes z)\otimes 1\otimes 1$. Let $s:=|H'/H|$, thus $V|_H \cong \cC^{\Q_p-\la}(H,E)^{\oplus ks}$. We have then an $L_H$-equivariant isomorphism $V^{N_H}\cong \big(\cC^{\Q_p-\la}(N_H^-, E) \widehat{\otimes}_E \cC^{\Q_p-\la}(L_H, E)\big)^{\oplus ks}$. Let $r:=\dim_E L(\lambda)_P^{\vee}$, we have $L_H$-equivariant isomorphisms (see for example \cite[Lemma 2.19]{Ding5} for the second isomorphism):
	\begin{eqnarray*}
		V^{N_H} \otimes_E L(\lambda)_P^{\vee} &\cong& \big(\cC^{\Q_p-\la}(N_H^-, E) \widehat{\otimes}_E \cC^{\Q_p-\la}(L_H, E)\big)^{\oplus ks} \otimes_E L(\lambda)_P^{\vee} \\
		&\cong& \big(\cC^{\Q_p-\la}(N_H^-, E) \widehat{\otimes}_E \cC^{\Q_p-\la}(L_H, E)\big)^{\oplus rks} \\
		&\cong& \big(\cC^{\Q_p-\la}(N_H^-,E) \widehat{\otimes}_E \cC^{\Q_p-\la}(L_H^D,E) \widehat{\otimes}_E \cC^{\Q_p-\la}(Z_{L_H},E)\big)^{\oplus rks}.
	\end{eqnarray*}
	There exists $s'$ such that $(\cC^{\Q_p-\la}(Z_{L_P}^0, E) \otimes_E\sigma^{\vee})|_{L_H} \cong \cC^{\Q_p-\la}(Z_{L_H}, E)^{\oplus s'}$, where the $L_H$-action on the right hand side is induced from the regular $Z_{L_H}$-action via $L_H \xrightarrow{\dett_{L_P}^{-1}} Z_{L_H}$. Let $m:=krss'$, we then have 
	\begin{eqnarray*}
		W &\cong& \Big(\big(\cC^{\Q_p-\la}(N_H^-,E) \widehat{\otimes}_E \cC^{\Q_p-\la}(L_H^D,E) \widehat{\otimes}_E \cC^{\Q_p-\la}(Z_{L_H},E) \widehat{\otimes}_E \cC^{\Q_p-\la}(Z_{L_H}, E) \big)^{L_H}\Big)^{\oplus m} \\
		&\cong& \Big(\cC^{\Q_p-\la}(N_H^-,E) \widehat{\otimes}_E \big(\cC^{\Q_p-\la}(Z_{L_H},E) \widehat{\otimes}_E \cC^{\Q_p-\la}(Z_{L_H}, E) \big)^{Z_{L_H}}\Big)^{\oplus m}\\ 
		&\cong& \Big(\cC^{\Q_p-\la}(N_H^-,E) \widehat{\otimes}_E \big(\cC^{\Q_p-\la}(Z_{L_H}\times Z_{L_H}, E) \big)^{Z_{L_H}}\Big)^{\oplus m},
	\end{eqnarray*}
where the $Z_{L_H}$-fixed vectors in the last term are for the $Z_{L_H}$-action induced from the regular $Z_{L_H}\times Z_{L_H}$-action via the morphism $Z_{L_H} \hookrightarrow Z_{L_H} \times Z_{L_H}$, $a \mapsto (a, \dett_{L_P}^{-1}(a))$. Using the isomorphism $Z_{L_H} \times Z_{L_H} \buildrel\sim\over \lra Z_{L_H}\times Z_{L_H}$, $(a,b) \longmapsto (a, \dett_{L_P}^{-1}(a) b)$, we finally deduce a $Z_{L_H}$-equivariant isomorphism:
	\begin{equation*}
		W \cong \big(\cC^{\Q_p-\la}(N_H^-,E) \widehat{\otimes}_E \cC^{\Q_p-\la}(Z_{L_H}, E) \big)^{\oplus m}
	\end{equation*}
where $Z_{L_H}$ acts on $W$ via $Z_{L_H}\hookrightarrow \cZ_0$, i.e.\ by only acting on the factor $\cC^{\Q_p-\la}(Z_{L_P}^0, E)$ in (\ref{Wfun}) by the right regular action, and similarly with $\cC^{\Q_p-\la}(Z_{L_H}, E)$ on the right hand side. The proposition follows then from (an easy variation of) the argument in the proof of \cite[Prop.\ 5.3]{BHS1} (with $\Pi^{N_0}$ of \textit{loc.\ cit.}\ replaced by $W$).
\end{proof}

\subsection{Bernstein eigenvarieties}\label{secBern}

In this section, we first give our global setup. Then we apply the formalism of \S~\ref{localbernstein} to construct rigid analytic spaces, called \textit{Bernstein eigenvarieties}, parametrizing certain $p$-adic automorphic representations. We also show some basic properties of Bernstein eigenvarieties.

\subsubsection{$p$-adic automorphic representations}\label{s:pAF}

We briefly recall the global setting of \cite[\S~5]{Br13II} (which will be the same as ours) and introduce some notation.

We fix embeddings $\iota_{\infty}: \overline{\Q} \hookrightarrow \bC$ and $\iota_p: \overline{\Q}\hookrightarrow \overline{\Q_p}$. We let $F$ be a CM field that is a totally imaginary quadratic extension of a totally real field $F^+$ such that all the places of $F^+$ dividing $p$ split in $F$. We let $G/F^+$ be a unitary group of rank $n\geq 2$ associated to $F/F^+$, i.e.\ such that $G\times_{F^+} F \cong \GL_n/F$. We assume that $G(F_v^+)$ is compact at all archimedean places $v$ of $F^+$. For a finite place $v$ of $F^+$ such that $v$ splits in $F$, we choose a place $\widetilde{v}$ of $F$ dividing $v$. For such places, we have natural isomorphisms $F_v^+ \cong F_{\widetilde{v}}$ and $i_{\widetilde{v}}: G(F_v^+)\xrightarrow{\sim} G(F_{\widetilde{v}}) \xrightarrow{\sim} \GL_n(F_{\widetilde{v}})$. 

We let $U^p$ be a compact open subgroup of $G(\bA_{F^+}^{\infty, p})$ of the form $U^p=\prod_{v \nmid p, \infty} U_v$ where $U_v$ is a compact open subgroup of $G(F_v^+)$ which is hyperspecial when $v$ is inert in $F$. We choose a finite set $S$ of finite places of $F^+$ containing the set $S_p$ of places dividing $p$ and the set of places $v$ such that either $v$ is ramified in $F$ or $U_v$ is {\it not} maximal at $v$. We assume moreover that for all places $v\notin S$ that are split in $F$, $U_v=i_{\widetilde{v}}^{-1}(\GL_n(\co_{F_{\widetilde{v}}}))$. We let $\bT^S:=\varinjlim_I \otimes_{v\in I} \bT_v$ where $\bT_v:=\co_E[U_v \backslash G(F_v^+)/U_v]$ and $I$ runs through the finite sets of places $v$ of $F^+$ which are not in $S$ and split in $F$ (recall that $E$ is a sufficiently large finite extension of $\Q_p$).\index{$\bT^S$} Note that $\bT_v$ is polynomially generated over $\co_E$ by the operators\index{$\bT_v$} $T_{\widetilde{v},i}=\bigg[U_v i_{\widetilde{v}}^{-1}\begin{pmatrix}
	\varpi_{\widetilde{v}} 1_i & 0 \\ 0 & 1_{n-i}
\end{pmatrix}
U_v\bigg]$ for $1\leq i \leq n$.

We consider the usual spaces of $p$-adic automorphic forms of level $U^p$ in that context:
\begin{eqnarray*}
	\widehat{S}(U^p,E)&:=&\{f: G(F^+)\backslash G(\bA_{F^+}^{\infty})/U^p \ra E, \ \text{$f$ continuous}\},\\
	\widehat{S}(U^p,\co_E)&:=&\{f: G(F^+)\backslash G(\bA_{F^+}^{\infty})/U^p \ra \co_E, \ \text{$f$ continuous}\}.
\end{eqnarray*}
We equip $\widehat{S}(U^p,E)$ with the norm defined using the unit ball $\widehat{S}(U^p,\co_E)$, in particular $\widehat{S}(U^p,E)$ is a $p$-adic Banach space. This Banach space is also equipped with a natural continuous unitary action of $G(F^+ \otimes_{\Q} \Q_p)\cong \prod_{v\in S_p} \GL_n(F_{\widetilde{v}})$, and an action of $\bT^S$ (with each element acting via a continuous operator) that commutes with $G(F^+\otimes_{\Q} \Q_p)$. Note that all these actions preserve $\widehat{S}(U^p, \co_E)$. We also have
\begin{equation*}
	\widehat{S}(U^p,\co_E)\cong \varprojlim_s S(U^p,\co_E/\varpi_E^s) \cong \varprojlim_s \varinjlim_{U_p} S(U^pU_p, \co_E/\varpi_E^s)
\end{equation*}
where $S(U^pU_p, \co_E/\varpi_E^s)$ denotes the space of functions $G(F^+)\backslash G(\bA_{F^+}^{\infty})/(U^pU_p) \ra \co_E/\varpi_E^s$. 

For an automorphic representation $\pi\cong \pi_{\infty} \otimes_{\bC} \pi^{\infty} \cong \pi_{\infty} \otimes_{\bC} \pi^{\infty,p} \otimes_{\bC} \pi_p$ of
\[G(\bA_{F^+})\cong G(F^+ \otimes_{\Q} \bR)\times G(\bA_{F^+}^{\infty}) \cong G(F^+ \otimes_{\Q} \bR) \times G(\bA_{F^+}^{\infty,p}) \times G(F^+ \otimes_{\Q} \Q_p),\]
we associate an algebraic representation $W_{\pi,p}$ of $G(F^+ \otimes_{\Q} \Q_p)$ over $\overline{\Q_p}$ as in \cite[Prop.\ 5.1]{Br13II} (and as in the references therein). Recall that we have (for example see \cite[Prop.\ 5.1]{Br13II} for (1) and part (a) in the proof of \cite[Lemma 6.1]{BD1} for (2)):

\begin{proposition}\label{pAF0}
	(1) We have a $G(F^+ \otimes_{\Q} \Q_p) \times \bT^S$-equivariant isomorphism
	\begin{equation*}
		\widehat{S}(U^p,E)^{\lalg} \otimes_E \overline{\Q_p} \cong \oplus_{\pi} \big((\pi^{\infty,p})^{U^p} \otimes_{\overline{\Q}} (\pi_p \otimes_{\overline{\Q}} W_{\pi,p})\big)^{\oplus m(\pi)},
	\end{equation*}
	where $\pi$ runs through the automorphic representations of $G(\bA_{F^+})$.
	
	(2) Assume $U^p$ is sufficiently small, then for any compact open subgroup $H$ of $G(F^+ \otimes_{\Q} \Q_p)$, there exists $m$ such that $\widehat{S}(U^p,E) |_H \cong \cC(H,E)^{\oplus m}$.
\end{proposition}

\subsubsection{Bernstein eigenvarieties: construction and first properties}\label{s: B0}

We apply the construction in \S~\ref{abCon} to the locally $\Q_p$-analytic representation $\widehat{S}(U^p,E)^{\an}$ of $G(F^+ \otimes_{\Q} \Q_p)\cong \prod_{v\in S_p} \GL_n(F_{\widetilde{v}})$. 

We first modify the notation in \S~\ref{abCon} to be consistent with \S~\ref{s:pAF} in an obvious way. The index set $\sI$ will be $S_p$, and $G_p$ will be $G(F^+ \otimes_{\Q} \Q_p)\cong \prod_{v\in S_p} \GL_n(F_{\widetilde{v}})$. The element $i\in \sI$ will be replaced by $\widetilde{v}$ everywhere (for example, $F_i$ in \S~\ref{abCon} will be $F_{\widetilde{v}}$ etc.) and we fix a uniformizer $\varpi_{\widetilde{v}}$ for each $F_{\widetilde{v}}$. As in \ref{abCon}, we fix a parabolic subgroup $P_{\widetilde{v}}\supset B$ of $\GL_n$ for each $\widetilde{v}$ with a fixed Levi subgroup $L_{P_{\widetilde{v}}} \supset T$ and denote by $N_{P_{\widetilde{v}}}$ its nilpotent radical. We fix a cuspidal Bernstein component $\Omega$ for $L_P(\Q_p)\cong \prod_{v\in S_p} L_{P_{\widetilde{v}}}(F_{\widetilde{v}}) \cong \prod_{v\in S_p} \prod_{j=1}^{r_{\widetilde{v}}} \GL_{n_{\widetilde{v},j}}(F_{\widetilde{v}})$, and let $\sigma\cong\boxtimes_{v\in S_p} \boxtimes_{j=1}^{r_{\widetilde{v}}} \sigma_{\widetilde{v},j}$ be a smooth absolutely irreducible representation of $L_P^0:=\prod_{v\in S_p}\prod_{j=1}^{r_{\widetilde{v}}} \GL_{n_{\widetilde{v},j}}(\co_{F_{\widetilde{v}}})$ over $E$ associated to $\Omega$ as in \S~\ref{abCon}. We finally fix an integral $P$-dominant weight $\lambda=(\lambda_{\widetilde{v},i,\tau})_{\substack{v\in S_p \\ i=1, \dots, n\\ \tau\in \Sigma_{\widetilde{v}}}}$ for $G_p$ where $\Sigma_{\widetilde{v}}:=\Sigma_{F_{\widetilde{v}}}$. For $i\in \{1,\dots,r_{\widetilde v}\}$ we set $s_{\widetilde{v},i}:=\sum_{j=1}^i n_{\widetilde{v},j}$ and $s_{\widetilde{v},0}=0$.

By the discussion above Proposition \ref{pts1}, $B_{\Omega, \lambda}(\widehat{S}(U^p,E)^{\an})^{\vee}$ gives rise to a coherent sheaf $\cM_{\Omega, \lambda}(U^p)$ over $(\Spec \cZ_{\Omega})^{\rig} \times \widehat{\cZ_0}$ where $\cZ_0\cong Z_{L_P}^0$.\index{$\cM_{\Omega, \lambda}(U^p)$} By functoriality, $B_{\Omega, \lambda}(\widehat{S}(U^p,E)^{\an})$ is naturally equipped with an action of $\bT^S$ that commutes with the action of $\cZ_{\Omega}\times \cZ_0$. We deduce that $\cM_{\Omega, \lambda}(U^p)$ is equipped with a natural $\co_{(\Spec \cZ_{\Omega})^{\rig} \times \widehat{\cZ_0}}$-linear action of $\bT^S$. For each affinoid open $U=\Spm R$ of $(\Spec \cZ_{\Omega})^{\rig} \times \widehat{\cZ_0}$, the (commutative) $R$-subalgebra $A_R$ of $\End_R(\cM_{\Omega, \lambda}(U^p)\vert_U)$ generated by $\bT^S$ is a finite type $R$-module. These $\{ \Spm A_R\}$ then glue to a rigid analytic space, denoted by $\cE_{\Omega, \lambda}(U^p)$, which is finite over $(\Spec \cZ_{\Omega})^{\rig} \times \widehat{\cZ_0}$.\index{$\cE_{\Omega, \ul{\lambda}}(U^p)$} From the definition of $\cE_{\Omega, \lambda}(U^p)$, we see that $\cM_{\Omega, \lambda}(U^p)$ is also a coherent sheaf over $\cE_{\Omega, \lambda}(U^p)$.

The following properties follow easily from the construction, Proposition \ref{pts1} and Corollary \ref{cor:muom0}.

\begin{proposition}\label{basicBE}
(1) For a finite extension $E'$ of $E$, an $E'$-point of $\cE_{\Omega, \lambda}(U^p)$ can be identified to a triple $(\eta, \pi_{L_P}, \chi)$ where $\eta: \bT^S \ra E'$ is a system of Hecke eigenvalues and $(\pi_{L_P},\chi)$ is an $E'$-point of $(\Spec \cZ_{\Omega})^{\rig} \times \widehat{\cZ_0}$.
	
(2) We have an isomorphism equivariant under the action of $\cZ_{\Omega}\times \cZ_0 \times \bT^S$: \begin{equation*}
		\Gamma(\cE_{\Omega, \lambda}(U^p), \cM_{\Omega, \lambda}(U^p)) \cong B_{\Omega, \lambda}(\widehat{S}(U^p,E)^{\an})^{\vee}.
	\end{equation*}
	Moreover, for $x=(\eta, \pi_{L_P}, \chi) \in \cE_{\Omega, \lambda}(U^p)$, the above isomorphism induces an isomorphism of $k(x)$-vector spaces:
	\begin{equation}\label{fiberE}
		(x^* \cM_{\Omega, \lambda}(U^p))^{\vee} \cong \Hom_{L_P(\Q_p)}\Big(\pi_{L_P}\otimes_E \big(\chi_{\ul{\varpi}} \circ \dett_{L_P}\big) \otimes_E L(\lambda)_P, J_P(\widehat{S}(U^p,E)^{\an})[\bT^S=\eta]\Big).
	\end{equation}
	
(3) For $\eta: \bT^S \ra E'$ a system of Hecke eigenvalues and $(\pi_{L_P},\chi)$ an $E'$-point of $(\Spec \cZ_{\Omega})^{\rig} \times \widehat{\cZ_0}$, the following are equivalent:
	\begin{itemize}
		\item there exists a point $x \in \cE_{\Omega, \lambda}(U^p)$ of parameter $(\eta,\pi_{L_P},\chi)$;
		\item there exists a point $x \in \cE_{\Omega, \lambda}(U^p)$ of parameter $(\eta,\psi(\pi_{L_P},\chi))$ for any $\psi \in \mu_{\Omega}$ (see the discussion above Corollary \ref{cor:muom0} for the action of $\mu_{\Omega}$ on $(\Spec \cZ_{\Omega})^{\rig} \times \widehat{\cZ_0}$);
		\item the vector space on the right hand side of (\ref{fiberE}) is non-zero. 
	\end{itemize} 
\end{proposition}

By Proposition \ref{twBE2}, we have

\begin{proposition}\label{twBEi}
	With the notation of Proposition \ref{twBE2}, we have an isomorphism of rigid spaces
	\begin{equation*}
		\cE_{\Omega, \lambda}(U^p) \xlongrightarrow{\sim} \cE_{\Omega', \lambda'}(U^p), (\eta, \pi_{L_P}, \chi) \mapsto \Big(\eta, \pi_{L_P} \otimes_E \big((\psi^0_{\ul{\varpi}}(\delta_{\fd,\ul{\varpi}}^{\unr})^{-1})\circ \dett_{L_P}\big), \chi (\psi_{\ul{\varpi}}^0)^{-1}(\delta_{\fd,\ul{\varpi}}^0)^{-1}\Big).
	\end{equation*}
\end{proposition}

Let $z\in Z_{\ul{\varpi}}\subset Z_{L_P}(\Q_p)$ be as in the discussion above Proposition \ref{Pr}. We define $\kappa_z$ as the composition
\begin{equation}\label{kappaz}
\kappa_z: \cE_{\Omega, \lambda}(U^p) \lra (\Spec \cZ_{\Omega})^{\rig} \times \widehat{\cZ_0} \lra (\Spec E[Z_{\ul{\varpi}}])^{\rig} \times \widehat{\cZ_0} \lra \bG_m^{\rig} \times \widehat{\cZ_0}
\end{equation}
where the last two morphisms are induced by $E[Y_z] \hookrightarrow E[Z_{\ul{\varpi}}] \ra \cZ_{\Omega}$ (recall $Y_z$ is the subgroup of $Z_{\ul{\varpi}}$ generated by $z$). It follows from \cite[Prop.\ 3.2.23]{Em11} and Proposition \ref{Pr} (see also the proof of \cite[Lemma 3.10]{BHS1}) that $(\kappa_z)_* \cM_{\Omega, \lambda}(U^p)$ is a coherent sheaf over $\bG_m^{\rig} \times \widehat{\cZ_0}$ and $\kappa_z$ is finite. We denote by $Z_z(U^p)$ its scheme-theoretic support in $\bG_m^{\rig} \times \widehat{\cZ_0}$. Note that the first morphism in (\ref{kappaz}) factors through the scheme-theoretic support of $\cM_{\Omega, \lambda}(U^p)$ in $(\Spec \cZ_{\Omega})^{\rig} \times \widehat{\cZ_0}$, and $\kappa_z$ factors through $Z_z(U^p)$. We define $\kappa$ as the composition (the second map being the canonical projection)
\[\kappa: \cE_{\Omega, \lambda}(U^p) \ra (\Spec \cZ_{\Omega})^{\rig} \times \widehat{\cZ_0} \twoheadrightarrow \widehat{\cZ_0}\] which obviously factors through $\kappa_z$. By exactly the same argument as in the proofs of \cite[Lemma 3.10]{BHS1} and \cite[Prop.\ 3.11]{BHS1} (with \cite[Prop.\ 5.3]{BHS1} replaced by Proposition \ref{Pr}), we have:

\begin{proposition}\label{fredholm1}
	(1) The rigid space $Z_z(U^p)\hookrightarrow \bG_m^{\rig} \times \widehat{\cZ_0}$ is a Fredholm hypersurface of $\bG_m^{\rig} \times \widehat{\cZ_0}$ (cf.\ \cite[\S~3.3]{BHS1}). Moreover, there exists an admissible covering $\{U_i'\}$ of $Z_z(U^p)$ by affinoids $U_i'$ such that the composition
	\begin{equation*}
		g: Z_z(U^p) \hookrightarrow \bG_m^{\rig} \times \widehat{\cZ_0} \twoheadrightarrow \widehat{\cZ_0}
	\end{equation*} 
	induces a finite surjective morphism from $U_i'$ to an affinoid open $W_i$ of $\widehat{\cZ_0}$, and such that $U_i'$ is a connected component of $g^{-1}(W_i)$. For each $i$, $\Gamma\big(U_i', (\kappa_z)_* \cM_{\Omega, \lambda}(U^p)\big)$ is a finitely generated projective $\co_{\widehat{\cZ_0}}(W_i)$-module. 
	
	(2) There exists an admissible covering $\{U_i\}$ of $\cE_{\Omega, \lambda}(U^p)$ by affinoids $U_i$ such that
	\begin{itemize}
		\item there exists an affinoid open $W_i$ of $\widehat{\cZ_0}$ satisfying that $\kappa$ is a finite surjective morphism from each irreducible component of $U_i$ to $W_i$;
		\item $\co_{\cE_{\Omega, \lambda}(U^p)}(U_i)$ is isomorphic to an $\co_{\widehat{\cZ_0}}(W_i)$-algebra of endomorphisms of a finitely generated projective $\co_{\widehat{\cZ_0}}(W_i)$-module.
	\end{itemize}
\end{proposition}

We also have as in \cite[Cor.\ 3.12]{BHS1}, \cite[Cor.\ 3.13]{BHS1} and \cite[Lemma\ 3.8]{BHS2} by the same arguments:

\begin{corollary}\label{CMC1}
(1) The rigid space $\cE_{\Omega, \lambda}(U^p)$ is nested (\cite[Def.\ 7.2.10]{BCh}), equidimensional of dimension $\sum_{v\in S_p} ([F_{\widetilde{v}}:\Q_p]r_{\widetilde{v}})$, and has no embedded component. 
	
(2) The morphism $\kappa_z$ is finite and the image of an irreducible component of $\cE_{\Omega, \lambda}(U^p)$ is an irreducible component of $Z_z(U^p)$. The image of an irreducible component of $\cE_{\Omega, \lambda}(U^p)$ by $\kappa$ is a Zariski-open of $\widehat{\cZ_0}$. 
	
(3) The coherent sheaf $\cM_{\Omega, \lambda}(U^p)$ is Cohen-Macaulay over $\cE_{\Omega, \lambda}(U^p)$. 
\end{corollary}

\begin{remark}\label{remP=B2}
Assume $P=B_p:=\prod_{v\in S_p} B(F_{\widetilde{v}})$, and let $T_p:=\prod_{v\in S_p} T(F_{\widetilde{v}})$. Consider the composition (cf.\ Remark \ref{remP=B}) 
\begin{equation*}
\cE_{\Omega, \lambda}(U^p) \lra (\Spec \cZ_{\Omega})^{\rig} \times \widehat{\cZ_0} \xlongrightarrow{\iota_{\Omega, \lambda}} \widehat{T_p}.
\end{equation*}
Using Proposition \ref{twBEi}, the Bernstein eigenvarieties $\cE_{\Omega, \lambda}(U^p)$ equipped with the above morphism over $\widehat{T_p}$ are all isomorphic for different $(\Omega, \lambda)$. Moreover by Remark \ref{remP=B} and (the same argument as in the proof of) \cite[Prop.\ 7.2.8]{BCh}, one can show that $\cE_{\Omega, \lambda}(U^p)^{\red}$ is isomorphic to the standard reduced eigenvariety $\cE(U^p)^{\red}$ constructed directly from $J_{B_p}(\widehat{S}(U^p,E)^{\an})$ (see for example \cite[\S~7]{Br13II}).
\end{remark}

\subsubsection{Density of classical points}

We prove Theorem \ref{dens1} below.

\begin{definition}\label{defclass}
Let $x=(\eta_x, \pi_{x,L_P}, \chi_x)$ be a point in $\cE_{\Omega, \lambda}(U^p)$.

(1) We call $x$ classical if
	\begin{equation}\label{defClas}
		\Hom_{L_P(\Q_p)}\Big(\pi_{x,L_P}\otimes_E \big((\chi_x)_{\ul{\varpi}} \circ \dett_{L_P}\big) \otimes_E L(\lambda)_P, J_P(\widehat{S}(U^p,E)^{\lalg})[\bT^S=\eta_x]\Big)\neq 0.
	\end{equation}
	
(2) We call $x$ very classical if 
	\begin{itemize}
		\item $\chi_x$ is locally algebraic and the weight $\lambda^x:=(\wt(\chi_x)\circ \dett_{L_P})+\lambda$ is dominant; 
		\item any irreducible constituent of the locally analytic parabolic induction
		\begin{equation}\label{parabol}
			\Big(\Ind_{P^-(\Q_p)}^{G(\Q_p)} \big(\pi_{x,L_P}\otimes_E ((\chi_x)_{\ul{\varpi}}\circ \dett_{L_P}) \otimes_E L(\lambda)_P\big)\Big)^{\an}
		\end{equation}	
		which is not locally algebraic, does not admit a $G(\Q_p)$-invariant $\co_E$-lattice.
	\end{itemize}
\end{definition}

\begin{lemma}
A very classical point $x=(\eta_x, \pi_{x,L_P},\chi_x)$ is classical.
\end{lemma}
\begin{proof}
We write $(\chi_x)_{\ul{\varpi}}$ in the form $(\chi_x)_{\ul{\varpi}}^{\infty} \delta_{\wt(\chi_x)}$ where $(\chi_x)_{\ul{\varpi}}^{\infty}$ is a smooth character of $Z_{L_P}(\Q_p)$ (recall from \S~\ref{sec3.1.1} that $\delta_{\wt(\chi_x)}$ is the algebraic character of $Z_{L_P}(\Q_p)$ of weight $\wt(\chi_x)$). By Proposition \ref{basicBE} (3) we have
	\begin{equation*}
		\Hom_{L_P(\Q_p)}\Big(\pi_{x,L_P}\otimes_E (\chi_x)_{\ul{\varpi}} \circ \dett_{L_P} \otimes_E L(\lambda)_P, J_P(\widehat{S}(U^p,E)^{\an})[\bT^S=\eta_x]\Big)\neq 0.
	\end{equation*}
	By \cite[Thm.\ 4.3]{Br13II} (the notation of which we freely use), any non-zero element in the above vector space induces a non-zero morphism of $G(\Q_p)$-representations (recall $\ug$, $\fp^-$ denote the Lie algebra of $G(\Q_p)$, $P^-(\Q_p)$ respectively, and $\delta_P$ denotes the modulus character of $P(\Q_p)$)
	\begin{multline*}
		\cF_{P^-}^{G}\Big(\big(\text{U}(\ug) \otimes_{\text{U}(\fp^-)} L^-(-\lambda^x)_P\big)^{\vee}, \pi_{x,L_P}\otimes_E \delta_P^{-1} \otimes_E \big((\chi_x)_{\ul{\varpi}}^{\infty}\circ \dett_{L_P}\big)\Big)\\
		\lra \widehat{S}(U^p,E)^{\an}[\bT^S=\eta_x]\hooklongrightarrow \widehat{S}(U^p,E)^{\an}.
	\end{multline*}
	By the results of \cite{OS}, the representation on the left hand side has the following properties:
	\begin{itemize}
		\item it has the same irreducible constituents as the representation (\ref{parabol});
		\item there is a $G(\Q_p)$-equivariant surjection (where $(\Ind_{P^-(\Q_p)}^{G(\Q_p)}(-))^\infty$ is the smooth parabolic induction)
		\begin{multline*}
			\cF_{P^-}^{G}\Big(\big(\text{U}(\ug) \otimes_{\text{U}(\fp^-)} L^-(-\lambda^x)_P\big)^{\vee}, \pi_{x,L_P}\otimes_E \delta_P^{-1} \otimes_E \big((\chi_x)_{\ul{\varpi}}^{\infty}\circ \dett_{L_P}\big)\Big) \\
			\twoheadlongrightarrow \cF_{P^-}^{G_p}\Big( L^-(-\lambda^x), \pi_{x,L_P}\otimes_E \delta_P^{-1} \otimes_E \big((\chi_x)_{\ul{\varpi}}^{\infty} \circ \dett_{L_P}\big)\Big)\\
			\cong \Big(\Ind_{P^-(\Q_p)}^{G(\Q_p)}\big(\pi_{x,L_P}\otimes_E \delta_P^{-1} \otimes_E ((\chi_x)_{\ul{\varpi}}^{\infty}\circ \dett_{L_P})\big)\Big)^{\infty}\otimes_E L(\lambda^x) 
		\end{multline*}
		and any irreducible constituent of the kernel is {\it not} locally algebraic. 
	\end{itemize} 
The lemma then follows easily by definition.
\end{proof}

We have the following numerical classicality criterion.

\begin{proposition}[Numerical classicality]\label{class}
Let
\[x=(\eta_x, \pi_{x,L_P},\chi_x)=\big(\eta_x, \otimes_{v\in S_p} \pi_{x,\widetilde{v}}, \otimes_{v\in S_p} \chi_{x,\widetilde{v}}\big)\in \cE_{\Omega, \lambda}(U^p)\]
such that $\chi_x$ is locally algebraic and $\lambda^x=(\wt(\chi_x)\circ \dett_{L_P})+\lambda$ is dominant. Assume that for all $v\in S_p$ and $k\in \{1,\dots,r_{\widetilde{v}}\}$, we have 
	\begin{multline}\label{numCri}
		\sum_{j=1}^k \val_{\widetilde{v}}(\omega_{\pi_{x,\widetilde{v},j}}(\varpi_{\widetilde{v}})) < \sum_{j=1}^k(s_{\widetilde{v},j}+s_{\widetilde{v},j-1}-s_{\widetilde{v},k}) - \sum_{\tau \in \Sigma_{\widetilde{v}}}\sum_{j=1}^{s_{\widetilde{v},k}} \lambda_{\widetilde{v},j,\tau} \\
		+\inf_{\tau \in \Sigma_{\widetilde{v}}}\big\{\wt(\chi_x)_{\widetilde{v},k,\tau}-\wt(\chi_x)_{\widetilde{v},k+1,\tau}+\lambda_{\widetilde{v},s_{\widetilde{v},k},\tau}-\lambda_{\widetilde{v},s_{\widetilde{v},k+1},\tau}+1\big\},
	\end{multline}
where \ $\val_{\widetilde{v}}$ \ denotes \ the \ $p$-adic \ valuation \ normalized \ by \ sending \ $\varpi_{\widetilde{v}}$ \ to \ $1$, \ $\wt(\chi_x)_{\widetilde{v},k',\tau'}:=\wt(\chi_{x,\widetilde{v}})_{k',\tau'}$, and where $\omega_{\pi_{x, \widetilde{v},j}}$ denotes the central character of $\pi_{x, \widetilde{v},j}$. Then the point $x$ is very classical.
\end{proposition}
\begin{proof}
	By \cite[Thm.\ (i),(ii)]{OS}, the representation $\big(\Ind_{P^-(\Q_p)}^{G(\Q_p)} (\pi_{x,L_P}\otimes_E ((\chi_{x})_{\ul{\varpi}}\circ \dett_{L_P}) \otimes_E L(\lambda)_P)\big)^{\an}$ admits a Jordan-Holder filtration with graded pieces of the form:
	\begin{equation*}
		\cF_{P^-}^{G}(L^-(-w\cdot \lambda^x), \pi_{x,L_P} \otimes_E ((\chi_x)_{\ul{\varpi}}^{\infty} \circ \dett_{L_P})\otimes_E \delta_P^{-1})
	\end{equation*}
where $L^-(-w\cdot \lambda^x)$ runs through irreducible constituents of $\text{U}(\ug) \otimes_{\text{U}(\fp^-)} L^-(-\lambda^x)_P$. In particular, each $w\cdot \lambda^x$ is $P$-dominant. Assume that there exists an irreducible constituent $V$ of $\big(\Ind_{P^-(\Q_p)}^{G(\Q_p)} \pi_{x,L_P}\otimes_E ((\chi_x)_{\ul{\varpi}}\circ \dett_{L_P}) \otimes_E L(\lambda)_P\big)^{\an}$ such that $V$ admits a $G(\Q_p)$-invariant lattice and $V$ is not locally algebraic. We deduce that there exists
\[1\neq w=(w_{\widetilde{v}})_{v\in S_p}=(w_{\widetilde{v},\tau})_{\substack{v\in S_p \\ \tau \in \Sigma_{\widetilde{v}}}}\in \sW^{|\Sigma_p|}\cong \prod_{v\in S_p} \sW_{F_{\widetilde{v}}}=:\sW_F\index{$\sW_F$}\]
such that $w\cdot \lambda^x$ is $P$-dominant and $V$ is a constituent of $\cF_{P^-}^{G_p}(L^-(-w\cdot \lambda^x), \pi_{x,L_P} \otimes_E ((\chi_x)_{\ul{\varpi}}^{\infty}\circ \dett_{L_P}) \otimes_E \delta_P^{-1})$. We let $Q_{\widetilde{v}}\supset P_{\widetilde{v}}$ be the maximal parabolic of $\GL_n$ such that $w\cdot \lambda$ is $\prod_{v\in S_p} \Res^{F_{\widetilde{v}}}_{\Q_p} Q_{\widetilde{v}}$-dominant, i.e.\ $(w\cdot \lambda^x)_{\widetilde{v},\tau}$ is $Q_{\widetilde{v}}$-dominant for all $v\in S_p$ and $\tau \in \Sigma_{\widetilde{v}}$. We set $Q:=\prod_{v\in S_p} \Res^{F_{\widetilde{v}}}_{\Q_p} Q_{\widetilde{v}}$. We have by \cite[Thm.\ (iii)]{OS}
\begin{multline*}
\cF_{P^-}^{G}\Big(L^-(-w\cdot \lambda^x), \pi_{x,L_P} \otimes_E \big((\chi_x)_{\ul{\varpi}}^{\infty}\circ \dett_{L_P}\big) \otimes_E \delta_P^{-1}\Big)\\
\cong \cF_{Q^-}^{G}\Big(L^-(-w\cdot \lambda^x), \big(\Ind_{L_Q(\Q_p) \cap P^-(\Q_p)}^{L_Q(\Q_p)}(\pi_{x,L_P} \otimes_E ((\chi_x)_{\ul{\varpi}}^{\infty}\circ \dett_{L_P}) \otimes_E \delta_P^{-1})\big)^{\infty}\Big).
\end{multline*}	
By \cite[Thm.\ (i),(ii),(iv)]{OS}, there exists an irreducible constituent $\pi_{L_Q}$ of the smooth representation $\big(\Ind_{L_Q(\Q_p) \cap P^-(\Q_p)}^{L_Q(\Q_p)}(\pi_{x,L_P} \otimes_E ((\chi_x)_{\ul{\varpi}}^{\infty}\circ \dett_{L_P}) \otimes_E \delta_P^{-1})\big)^{\infty}$ such that
	\begin{multline*}
		V \cong \cF_{Q^-}^{G}\big(L^-(-w\cdot \lambda^x), \pi_{L_Q}\big)\\
		\hooklongrightarrow \cF_{Q^-}^{G}\Big(L^-(-w\cdot \lambda^x), \big(\Ind_{L_Q(\Q_p) \cap P^-(\Q_p)}^{L_Q(\Q_p)}(\pi_{x,L_P} \otimes_E ((\chi_x)_{\ul{\varpi}}^{\infty}\circ \dett_{L_P}) \otimes_E \delta_P^{-1})\big)^{\infty}\Big).
	\end{multline*}
Since $V$ admits a $G(\Q_p)$-invariant $\co_E$-lattice, we see by \cite[Cor.\ 3.5]{Br13I} that for any $z\in Z_{L_Q}(\Q_p)^+\subset Z_{L_P}(\Q_p)^+$ we have (where $\omega_{\pi_{L_Q}}$ denotes the central character of $\pi_{L_Q}$)
	\begin{equation}\label{inte0}
		\val_p\big(\delta_{w\cdot \lambda^x}(z) \delta_P^{-1}(z) \omega_{\pi_{L_Q}}(z)\big) \geq 0.
	\end{equation}
	By \cite[Lemma 3.18]{Ding5}, we easily deduce (noting that $\beta_a$ in {\it loc.\ cit.}\ is independent of the choice of $\lambda$) that for each $v\in S_p$ such that $w_{\widetilde v}\ne 1$ there exists $k\in \{1,\dots,r_{\widetilde{v}}\}$ such that $z:=(\underbrace{\varpi_{\widetilde{v}}, \dots, \varpi_{\widetilde{v}}}_{s_{\widetilde{v},k}}, 1, \dots, 1)\in Z_{L_{Q_{\widetilde{v}}}}(F_{\widetilde{v}})^+ \subset Z_{L_Q}(\Q_p)^+$ and 
\begin{eqnarray*}
\val_{\widetilde{v}}\big(\delta_{w\cdot \lambda^x-\lambda^x}(z)\big) &\leq& \sum_{\substack{\tau \in \Sigma_{\widetilde{v}}\\ w_{\widetilde{v},\tau}\neq 1}} \big(\lambda_{\widetilde{v}, s_{\widetilde{v},k+1},\tau}^x-\lambda_{\widetilde{v},s_{\widetilde{v},k},\tau}^x-1\big)\\
&=&\sum_{\substack{\tau \in \Sigma_{\widetilde{v}}\\ w_{\widetilde{v},\tau}\neq 1}} \big(\wt(\chi_x)_{\widetilde{v},k+1,\tau}-\wt(\chi_x)_{\widetilde{v},k,\tau}+\lambda_{\widetilde{v}, s_{\widetilde{v},k+1},\tau}-\lambda_{\widetilde{v},s_{\widetilde{v},k},\tau}-1\big).
\end{eqnarray*}
Together with (\ref{inte0}), we deduce
	\begin{equation*}
\val_{\widetilde{v}}\big(\delta_{\lambda^x}(z) \delta_P^{-1}(z) \omega_{\pi_{L_Q}}(z)\big) \geq \sum_{\substack{\tau \in \Sigma_{\widetilde{v}}\\ w_{\widetilde{v},\tau}\neq 1}} \big(\wt(\chi_x)_{\widetilde{v},k,\tau}-\wt(\chi_x)_{\widetilde{v},k+1,\tau}+\lambda_{\widetilde{v},s_{\widetilde{v},k},\tau}-\lambda_{\widetilde{v}, s_{\widetilde{v},k+1},\tau}+1\big).
	\end{equation*}
We compute:
\begin{eqnarray*}
\val_{\widetilde{v}}\big(\delta_P^{-1}(z)\big)&=&\sum_{j=1}^k \big(s_{\widetilde{v}, k}-s_{\widetilde{v},j}-s_{\widetilde{v},j-1}\big)\\
\val_{\widetilde{v}}\big(\delta_{\lambda^x}(z)\big)&=&\sum_{\tau \in \Sigma_{\widetilde{v}}}\sum_{j=1}^{s_{\widetilde{v},k}}\lambda_{\widetilde{v},j,\tau}^x =\sum_{\tau \in \Sigma_{\widetilde{v}}}\sum_{j=1}^{s_{\widetilde{v},k}} \lambda_{\widetilde{v},j,\tau}+\sum_{\tau \in \Sigma_{\widetilde{v}}}\sum_{j=1}^{k}\big(n_{\widetilde{v},j}\wt(\chi_x)_{\widetilde{v},j,\tau}\big)\\
\val_{\widetilde{v}}\big(\omega_{\pi_{L_Q}}(z)\big)&=&\sum_{j=1}^k \val_{\widetilde{v}}\big(\omega_{\pi_{\widetilde{v},j}}(\varpi_{\widetilde{v}})\big)-\sum_{\tau \in \Sigma_{\widetilde{v}}}\sum_{j=1}^{k}\big(n_{\widetilde{v},j}\wt(\chi_x)_{\widetilde{v},j,\tau}\big).
\end{eqnarray*}
Hence we deduce
	\begin{multline*}
		\sum_{j=1}^k \val_{\widetilde{v}}\big(\omega_{\pi_{\widetilde{v},j}}(\varpi_{\widetilde{v}})\big) \geq \sum_{j=1}^k\big(s_{\widetilde{v},j}+s_{\widetilde{v},j-1}-s_{\widetilde{v},k}\big) - \sum_{\tau \in \Sigma_{\widetilde{v}}}\sum_{j=1}^{s_{\widetilde{v},k}} \lambda_{\widetilde{v},j,\tau} \\
		+\sum_{\substack{\tau \in \Sigma_{\widetilde{v}}\\ w_{\widetilde{v},\tau}\neq 1}} \big(\wt(\chi_x)_{\widetilde{v},k,\tau}-\wt(\chi_x)_{\widetilde{v},k+1,\tau}+\lambda_{\widetilde{v},s_{\widetilde{v},k},\tau}-\lambda_{\widetilde{v}, s_{\widetilde{v},k+1},\tau}+1\big),
	\end{multline*}
which contradicts (\ref{numCri}) noting that $\wt(\chi_x)_{\widetilde{v},k,\tau}-\wt(\chi_x)_{\widetilde{v},k+1,\tau}+\lambda_{\widetilde{v},s_{\widetilde{v},k},\tau}-\lambda_{\widetilde{v}, s_{\widetilde{v},k+1},\tau}+1\geq 1$ for all $\tau$ as $\lambda^x$ is dominant and that the set $\{\tau \in \Sigma_{\widetilde{v}}\ |\ w_{\widetilde{v},\tau}\neq 1\}$ is non-empty since $w_{\widetilde{v}}\ne 1$. The proposition follows.
\end{proof}

\begin{remark}
	(1) Similar results (but in the setting of overconvergent cohomology) were obtained in \cite{BW20}.

(2) Note that all the terms on the right hand side of (\ref{numCri}) except $\wt(\chi_x)_{\widetilde{v},k,\tau}-\wt(\chi_x)_{\widetilde{v},k+1,\tau}$ are constants for points in $\cE_{\Omega, \lambda}(U^p)$. 

(3) Recall there is a finite morphism of $E$-algebras $E[Z_{\ul{\varpi}}] \ra \cZ_{\Omega}$ such that the associated morphism $\Spec \cZ_{\Omega} \ra \Spec E[Z_{\ul{\varpi}}]$ sends a point $\pi_{x,L_P}\in \Spec \cZ_{\Omega}$ to the character $z \mapsto \omega_{\pi_{x,L_P}}(z)$. We have hence finite morphisms (see also (\ref{kappaz})):
\begin{equation*}
\cE_{\Omega, \lambda}(U^p) \lra (\Spec \cZ_{\Omega})^{\rig} \times \widehat{\cZ_0} \lra (\Spec E[Z_{\ul{\varpi}}])^{\rig} \times \widehat{\cZ_0}. 
\end{equation*}
Note that the criterion in (\ref{numCri}) only uses the information of the image of $x$ in $(\Spec E[Z_{\ul{\varpi}}])^{\rig} \times \widehat{\cZ_0}$.
\end{remark}

The following theorem follows from the classicality criterion (Proposition \ref{class}) by the same argument as in the proof of \cite[Thm.\ 3.19]{BHS1} (see also \cite[\S~6.4.5]{Che}).

\begin{theorem}\label{dens1}
The set of very classical points is Zariski-dense in $\cE_{\Omega, \lambda}(U^p)$. Moreover, for any point $x=(\eta, \pi_{L_P},\chi)\in \cE_{\Omega, \lambda}(U^p)$ with $\chi$ locally algebraic, and for any admissible open $U\subseteq \cE_{\Omega, \lambda}(U^p)$ containing $x$, there exists an admissible open $V\subseteq U$ containing $x$ such that the set of very classical points in $V$ is Zariski-dense in $V$. 
\end{theorem}

\subsubsection{Galois representations}\label{galois}

We study families of Galois representations on $\cE_{\Omega, \lambda}(U^p)$. In particular, we show that the associated $(\varphi,\Gamma)$-modules admit a special kind of filtration.

We now assume that that $G$ is quasi-split at all finite places and that $F/F^+$ is unramified at all finite places. By \cite[Thm.\ 2.3]{Guer}, to an automorphic representation $\pi$ of $G$, one can associate an $n$-dimensional continuous essentially self-dual representation $\rho_{\pi}$ of $\Gal_F$ over $E$ (enlarging $E$ if necessary). In fact, if $\pi^{U^p}\neq 0$, the $\bT^S$-action on $\pi^{U^p}$ is given by a system of eigenvalues $\eta_{\pi}: \bT^S \ra E$. The representation $\rho_{\pi}$ is unramified outside $S$. And for any $v\notin S$ that splits in $F$, the characteristic polynomial of $\rho_{\pi}(\Frob_{\widetilde{v}})$ (where $\Frob_{\widetilde{v}}$ is a geometric Frobenius at $\widetilde{v}$), is given by 
\begin{equation}\label{ES}
	X^n + \cdots + (-1)^j (N \widetilde{v})^{\frac{j(j-1)}{2}} \eta_{\pi}(T_{\widetilde{v}}^{(j)}) X^{n-j}+ \cdots + (-1)^n (N \widetilde{v})^{\frac{n(n-1)}{2}} \eta_{\pi}(T_{\widetilde{v}}^{(n)}),
\end{equation}
where $N \widetilde{v}$ is the cardinality of the residue field at $\widetilde{v}$. Using Proposition \ref{pAF0}, to all classical points $x=(\eta_x, \pi_{x,L_P}, \chi_x)\in \cE_{\Omega, \lambda}(U^p)$ (Definition \ref{defclass} (1)), we can associate a continuous representation $\rho_x$ of $\Gal_F$ that is unramified outside $S$ and satisfies (\ref{ES}) (with $\eta_{\pi}$ replaced by $\eta_x$). Denote by $\Gal_F^S$ the Galois group of the maximal extension of $F$ that is unramified outside $S$.
Put 
\begin{equation*}
	\co(\cE_{\Omega, \lambda}(U^p))^0:=\{f\in \co(\cE_{\Omega, \lambda})(U^p)\ |\ |f(x)|_p \leq 1 \text{ for all $x\in \cE_{\Omega, \lambda}(U^p)$}\}.
\end{equation*}
Using that $\bT^S$ preserves $\widehat{S}(U^p,\co_E)$, it is easy to see that the natural morphism $\bT^S \ra \co(\cE_{\Omega, \lambda}(U^p))$ has image in $\co(\cE_{\Omega, \lambda}(U^p))^0$. Using the density of classical points (Theorem \ref{dens1}) and \cite[Prop.\ 7.1.1]{Che}, we deduce

\begin{proposition}
	There exists a unique $n$-dimensional continuous pseudocharacter
	\begin{equation*}
		\sT: \Gal_{F}^S \lra \co(\cE_{\Omega, \lambda}(U^p))^0
	\end{equation*}
	such that the evaluation of $\sT$ at any classical point $x$ of $\cE_{\Omega, \lambda}(U^p)$ coincides with ${\textrm Trace}(\rho_x)$.
\end{proposition}

By \cite[Thm.\ 1(2)]{Tay91}, we deduce
 
\begin{corollary}
For each point $x=(\eta_x, \pi_{x,L_P}, \chi_x)\in \cE_{\Omega, \lambda}(U^p)$, there exists a (unique) semi-simple continuous representation $\rho_x$ of $\Gal_F$ over $k(x)$ which is unramified outside $S$ and such that ${\textrm Trace}(\rho_x(\Frob_{\widetilde{v}}))=\eta_x(T_{\widetilde{v},1})$ for any $v\notin S$ split in $F$.
\end{corollary}

Next, we study the behaviour of the restriction of the Galois representations $\{\rho_x\}_{x\in \cE_{\Omega, \lambda}(U^p)}$ at $p$-adic places.

Let $x=(\eta_x, \pi_{x,L_P}, \chi_x)\in \cE_{\Omega, \lambda}(U^p)$. Recall $\pi_{x,L_P}= \boxtimes_{v\in S_p} \pi_{x, L_{P_{\widetilde{v}}}}=\boxtimes_{v\in S_p}\big(\boxtimes_{i=1}^{r_{\widetilde{v}}} \pi_{x, \widetilde{v}, i}\big)$. Let
$\ttr_{x,\widetilde{v},i}:= \rec(\pi_{x, \widetilde{v}, i})(\frac{1-n_{\widetilde{v},i}}{2}-s_{\widetilde{v},i-1})$, an irreducible Weil-Deligne representation of $W_{F_{\widetilde v}}$ of dimension $n_{\widetilde{v},i}$.

Assume first that the point $x$ is classical. By Definition \ref{defClas} and \cite[Thm.\ 4.3]{BHS1}, there exists a non-zero $G_p$-equivariant morphism (recalling $\lambda^x=(\wt(\chi_x) \circ \dett_{L_P})+\lambda$)
\begin{equation}\label{injlalg}
\Big(\Ind_{P^-(\Q_p)}^{G(\Q_p)} \pi_{x,L_P} \otimes_{k(x)} \big((\chi_x)_{\ul{\varpi}}^{\infty} \circ \dett_{L_P}\big) \otimes_E \delta_P^{-1}\Big)^{\infty} \otimes_E L(\lambda^x) \lra \widehat{S}(U^p,E)^{\lalg}[\bT^S=\eta_x]
\end{equation}
where $(\chi_x)_{\ul{\varpi}}^{\infty}=\boxtimes_{v\in S_p} (\chi_{x, \widetilde{v}})_{\varpi_{\widetilde{v}}}^{\infty} =\boxtimes_{v\in S_p}\boxtimes_{i=1}^{r_{\widetilde{v}}} (\chi_{x, \widetilde{v}, i} )_{\varpi_{\widetilde{v}}}^{\infty}$ denotes the smooth character of $Z_{L_P}(\Q_p)$ such that $(\chi_x)_{\ul{\varpi}}^{\infty} \delta_{\wt(\chi_x)}=(\chi_x)_{\ul{\varpi}}$. By Proposition \ref{pAF0} (1), there exists a classical automorphic representation $\pi=\pi_{\infty} \boxtimes \pi^{\infty,p} \boxtimes (\boxtimes_{v\in S_p} \pi_{\widetilde{v}})$ such that $\eta_{\pi}=\eta_x$, and a $\GL_n(F_{\widetilde{v}})$-equivariant surjection
\begin{equation*}
	\Big(\Ind_{P_{\widetilde{v}}^-(F_{\widetilde{v}})}^{\GL_n(F_{\widetilde{v}})} \pi_{x, L_{P_{\widetilde{v}}}} \otimes_{k(x)} \big((\chi_{x,\widetilde{v}})_{\varpi_{\widetilde{v}}}^{\infty} \circ \dett_{L_{P_{\widetilde{v}}}}\big) \otimes_E \delta_{P_{\widetilde{v}}}^{-1}\Big)^{\infty} \twoheadlongrightarrow \pi_{\widetilde{v}}.
\end{equation*}
We have then (see for example \cite[Thm.\ 1.2(b)]{Scho13})
\begin{equation}\label{recInd}
\rec(\pi_{\widetilde{v}})|_{W_{F_{\widetilde{v}}}}=\bigoplus_{i=1}^{r_{\widetilde{v}}} \Big(\rec(\pi_{x,\widetilde{v},i})\big(\frac{n-n_{\widetilde{v},i}}{2}-s_{\widetilde{v},i-1}\big) \otimes_{k(x)} \rec((\chi_{x,\widetilde{v},i})_{\varpi_{\widetilde{v}}}^{\infty})\Big).
\end{equation} 
Denote by $\ttr_{x,\widetilde{v}}$ the Weil-Deligne representation associated to (the Deligne-Fontaine module of) $\rho_{x,\widetilde{v}}$, and $\ttr_{x,\widetilde{v}}^{\sss}$ its $F$-semi-simplification. By the local-global compatibility in classical local Langlands correspondence for $\ell=p$ (e.g.\ \cite{Car12}), we have $\ttr_{x,\widetilde{v}}^{\sss} \cong \rec(\pi_{\widetilde{v}})(\frac{1-n}{2})$. We deduce hence 
\begin{equation*}
	\ttr_{x,\widetilde{v}}^{\sss}|_{W_{F_{\widetilde{v}}}} \cong \bigoplus_{i=1}^{r_{\widetilde{v}}} \big(\ttr_{x,\widetilde{v},i} \otimes_{k(x)} \rec((\chi_{x,\widetilde{v},i})_{\varpi_{\widetilde{v}}}^{\infty})\big).
\end{equation*}
We call the point $x$ \textit{generic} if, for $\epsilon=0,1$ and for all $i\neq i'$ and $v\in S_p$ we have
\[\ttr_{x,\widetilde{v},i} \otimes_{k(x)} \rec((\chi_{x,\widetilde{v},i})_{\varpi_{\widetilde{v}}}^{\infty})\ncong \ttr_{x,\widetilde{v},i'}(\epsilon) \otimes_{k(x)} \rec((\chi_{x,\widetilde{v},i'})_{\varpi_{\widetilde{v}}}^{\infty})\]
or equivalently 
\begin{multline}\label{defGene}
	\pi_{x,\widetilde{v},i} \otimes_{k(x)} \big(\big(\unr(q_{\widetilde{v}}^{s_{\widetilde{v},i-1}-(1-n_{\widetilde{v},i})/2})(\chi_{x,\widetilde{v},i})_{\varpi_{\widetilde{v}}}^{\infty}\big) \circ \dett_{L_{P_{\widetilde{v}}}}\big)\\
	\ncong \pi_{x,\widetilde{v},i'} \otimes_{k(x)} \big(\big(\unr(q_{\widetilde{v}}^{s_{\widetilde{v},i'-1}-(1-n_{\widetilde{v},i'})/2+\epsilon})(\chi_{x,\widetilde{v},i'})_{\varpi_{\widetilde{v}}}^{\infty}\big) \circ \dett_{L_{P_{\widetilde{v}}}}\big).
\end{multline}
Assume that $x$ is generic, it is easy to see that $\rho_{x,\widetilde{v}}$ is potentially crystalline and generic for all $v\in S_p$ (see \S~\ref{introPcr}), hence we have (here there is no need to take semi-simplification and restriction to $W_{F_{\widetilde{v}}}$)
\begin{equation}\label{noN}
	\ttr_{x,\widetilde{v}} \cong \bigoplus_{i=1}^{r_{\widetilde{v}}} \big(\ttr_{x,\widetilde{v},i} \otimes_{k(x)} \rec((\chi_{x,\widetilde{v},i})_{\varpi_{\widetilde{v}}}^{\infty})\big).
\end{equation}
We obtain thus a $P_{\widetilde{v}}$-filtration $\sF_{\widetilde{v}}$ on $\ttr_{x,\widetilde{v}}$ (\S~\ref{introPcr}) given by
\begin{equation}\label{Pvfil1}\Fil^{\sF_{\widetilde{v}}}_i \ttr_{x,\widetilde{v}}:=\bigoplus_{j=1}^i \big(\ttr_{x,\widetilde{v},j} \otimes_{k(x)} \rec((\chi_{x,\widetilde{v},j})_{\varpi_{\widetilde{v}}}^{\infty})\big),\ \ i\in \{0,\dots,r_{\widetilde v}\}.\end{equation}
Recall also that $\rho_{x,\widetilde{v}}$ is de Rham of Hodge-Tate weights
\[\textbf{h}_{x,\widetilde{v}}:=(\textbf{h}_{x,\widetilde{v},i})_{1\leq i \leq n}:=\{h_{x,\widetilde{v}, i, \tau}:=\lambda^x_{\widetilde{v},i,\tau}-i+1\}_{\substack{\tau \in \Sigma_{\widetilde{v}}\\ 1\leq i \leq n}}.\]
As in \S~\ref{introPcr}, we associate to $(\rho_{x,\widetilde{v}},\sF_{\widetilde{v}})$ an element $w_{x,\sF_{\widetilde{v}}}\in \sW_{F_{\widetilde{v}}}\cong \sW^{\oplus |\Sigma_{\widetilde{v}}|}$ (denoted by $w_{\sF}$ in \textit{loc.\ cit.}). Let $\Delta_{x,\widetilde{v}}$ be the $p$-adic differential equation associated to $\ttr_{x,\widetilde{v}}$ and $\Delta_{x,\widetilde{v},i}$ the $p$-adic differential equation associated to $\ttr_{x,\widetilde{v},i}$ (\S~\ref{sec_pDf}). We have $D_{\rig}(\rho_{x,\widetilde{v}})[1/t]\cong \Delta_{x,\widetilde{v}}[1/t]$ and by (\ref{noN}):
\begin{equation*}
\Delta_{x,\widetilde{v}} \cong \bigoplus_{i=1}^{r_{\widetilde{v}}} \big(\Delta_{x,\widetilde{v},i} \otimes_{\cR_{k(x),F_{\widetilde{v}}}} \cR_{k(x),F_{\widetilde{v}}}((\chi_{x,\widetilde{v},i})_{\varpi_{\widetilde{v}}}^{\infty})\big). 
\end{equation*}
As discussed in \S~\ref{introPcr}, the $P_{\widetilde{v}}$-filtration $\sF_{\widetilde{v}}$ on $\ttr_{x,\widetilde{v}}$ induces a $P_{\widetilde{v}}$-filtration on $\Delta_{x,\widetilde{v}}$ which further induces a $P_{\widetilde{v}}$-filtration (still denoted) $\sF_{\widetilde{v}}$ on $D_{\rig}(\rho_{x,\widetilde{v}})$. Recall also that, if $\sF_{\widetilde{v}}$ is non-critical, then $\gr_i^{\sF_{\widetilde{v}}} D_{\rig}(\rho_{x,\widetilde{v}})$ is de Rham of Hodge-Tate weights $\{h_{x,\widetilde{v},j,\tau}\}_{\!\!\!\substack{\tau \in \Sigma_{\widetilde{v}}\\ s_{\widetilde{v},i-1}< j \leq s_{\widetilde{v},i}}}$. Let 
\begin{equation}\label{textbfh0}
\textbf{h}_{\widetilde{v}}:=(\textbf{h}_{\widetilde{v},i})_{1\leq i \leq n}=(h_{\widetilde{v},i,\tau}:=\lambda_{\widetilde{v},i,\tau}-i+1)_{\substack{\tau \in \Sigma_{\widetilde{v}}\\ 1\leq i \leq n}},
\end{equation} 
and note that $\textbf{h}_{\widetilde{v}}$ is strictly $P_{\widetilde{v}}$-dominant. Suppose that $\sF_{\widetilde{v}}$ is non-critical, we have then injections for $i\in\{1, \dots, r_{\widetilde{v}}\}$ (cf.\ (\ref{inj000})):
\begin{equation*}
\gr_i^{\sF_{\widetilde{v}}} D_{\rig}(\rho_{x,\widetilde{v}}) \otimes_{\cR_{k(x),F_{\widetilde{v}}}} \cR_{k(x),F_{\widetilde{v}}}\big((\chi_{x,\widetilde{v},i})_{\varpi_{\widetilde{v}}}^{-1}z^{-\textbf{h}_{\widetilde{v}, s_{i}}}\big) \hooklongrightarrow \Delta_{x,\widetilde{v},i}.
\end{equation*}

\begin{lemma}\label{lemGenec}
Let $x=(\eta_x, \pi_{x,L_P}, \chi_x)\in \cE_{\Omega, \lambda}(U^p)$ be a classical point. Assume
\begin{multline*}
\val_{\widetilde{v}}(\omega_{\pi_{x,\widetilde{v},i}}(\varpi_{\widetilde{v}}))-\val_{\widetilde{v}}(\omega_{\pi_{x,\widetilde{v},i'}}(\varpi_{\widetilde{v}}))\\
\neq \sum_{\tau \in \Sigma_{\widetilde{v}}} \big(\wt(\chi_x)_{\widetilde{v},i,\tau}-\wt(\chi_x)_{\widetilde{v},i',\tau}\big) +[F_{\widetilde{v}}:\Q_p](s_{\widetilde{v},i}-s_{\widetilde{v},i'}+\epsilon)
\end{multline*}
for $\epsilon=0,1$ and for all $i\neq i'$, then $\rho_{x,\widetilde{v}}$ is generic. If moreover (\ref{numCri}) holds then $(\rho_{x,\widetilde{v}}, \sF_{\widetilde{v}})$ is non-critical.
\end{lemma}
\begin{proof}
The first part of the lemma is straightforward to check. By the proof of Proposition \ref{class}, if (\ref{numCri}) holds, then for any $1\neq w_{\widetilde{v}}\in (\sW^{P_{\widetilde{v}}}_{\min})^{\oplus |\Sigma_{\widetilde{v}}|}$, there exists $z\in Z_{L_{P_{\widetilde{v}}}}(F_{\widetilde{v}})^+$ such that 
	\begin{equation*}
		\delta_{w_{\widetilde{v}}\cdot \lambda_{\widetilde{v}}^x}(z) \delta_{P_{\widetilde{v}}}^{-1}(z) \omega_{\pi_{x,L_{P_{\widetilde{v}}}}}(z) \notin \co_{k(x)}.
	\end{equation*}
By \cite[Thm.\ 7.6]{Br13I}, we deduce $w_{x,\sF_{\widetilde{v}}}=w_{0,F_{\widetilde{v}}}$ (noting that the $w^{\alg}$ of \textit{loc.\ cit.}\ is $w_{x,\sF_{\widetilde{v}}} w_{0,F_{\widetilde{v}}}$ in our case). Hence $(\rho_{x,\widetilde{v}}, \sF_{\widetilde{v}})$ is non-critical by definition (\S~\ref{introPcr}).
\end{proof}

We call a classical point $x\in \cE_{\Omega, \lambda}(U^p)$ {\it non-critical} if $(\rho_{x,\widetilde{v}}, \sF_{\widetilde{v}})$ is non-critical for all $v\in S_p$. By the same argument as in the proof of \cite[Thm.\ 3.9]{BHS1} and using Proposition \ref{class}, Lemma \ref{lemGenec}, we have the following strengthening of Theorem \ref{dens1}:

\begin{theorem}\label{denNCG}
	The \ set \ of \ very \ classical \ non-critical \ generic \ points \ is \ Zariski-dense \ in $\cE_{\Omega, \lambda}(U^p)$. Moreover, for any point $x=(\eta, \pi_{L_P},\chi)\in \cE_{\Omega, \lambda}(U^p)$ with $\chi$ locally algebraic, and for any admissible open $U\subseteq \cE_{\Omega, \lambda}(U^p)$ containing $x$, there exists an admissible open $V\subseteq U$ containing $x$ such that the set of very classical non-critical generic points in $V$ is Zariski-dense in $V$. 
\end{theorem}

The following theorem is an analogue of the statement ``Galois representations on eigenvarieties are trianguline''.

\begin{theorem}\label{FOBE}
Let $x=(\eta_x, \pi_{x,L_P},\chi_x)$ be a point of $\cE_{\Omega, \lambda}(U^p)$. Then for $v\in S_p$, $D_{\rig}(\rho_{x,\widetilde{v}})$ admits a $P_{\widetilde{v}}$-filtration $\sF_{\widetilde{v}}=\Fil_{\bullet}^{\sF_{\widetilde{v}}}D_{\rig}(\rho_{x,\widetilde{v}})$ of saturated $(\varphi, \Gamma)$-submodules of $D_{\rig}(\rho_{x,\widetilde{v}})$ such that
	\begin{equation*}
		\big(\gr_i^{\sF_{\widetilde{v}}} D_{\rig}(\rho_{x,\widetilde{v}})\big)[1/t] \cong \big(\Delta_{x,\widetilde{v},i} 
		\otimes_{\cR_{k(x),F_{\widetilde{v}}}} \cR_{k(x),F_{\widetilde{v}}}((\chi_{x,\widetilde{v},i})_{\varpi_{\widetilde{v}}})\big)[1/t].
	\end{equation*} 
\end{theorem}
\begin{proof}
Since $\cE_{\Omega, \lambda}(U^p)^{\red}$ is nested, by \cite[Lemma 7.2.9]{BCh} and Theorem \ref{denNCG}, there exists an irreducible affinoid neighbourhood $U$ of $x$ in $\cE_{\Omega, \lambda}(U^p)^{\red}$ such that the set $C(U)$ of very classical generic non-critical points in $U$ is Zariski-dense in $U$. By pulling-back the universal character of $\cZ_0$ over $\widehat{\cZ_0}$ via the composition
\[U\hookrightarrow \cE_{\Omega, \lambda}(U^p)^{\red}\hookrightarrow \cE_{\Omega, \lambda}(U^p)\ra (\Spec \cZ_{\Omega})^{\rig} \times \widehat{\cZ_0} \twoheadrightarrow \widehat{\cZ_0},\]
we obtain a continuous character $\chi_U=\boxtimes_{v\in S_p}\boxtimes_{i=1}^{r_{\widetilde{v}}} \chi_{U,\widetilde{v},i}: Z_{L_P}^0 \ra \co(U)^{\times}$. For $v\in S_p$ and $1\leq i \leq r_{\widetilde{v}}$, by pulling-back the ``universal" $p$-adic differential equation over $(\Spec \cZ_{\Omega_{\widetilde{v},i}})^{\rig}$ constructed in \S~\ref{sec_pDf} via the composition
\[U\hookrightarrow \cE_{\Omega, \lambda}(U^p)\ra (\Spec \cZ_{\Omega})^{\rig} \times \widehat{\cZ_0} \twoheadrightarrow (\Spec \cZ_{\Omega})^{\rig}\simeq \prod_{v\in S_p}\prod_{i=1}^{r_{\widetilde v}}(\Spec \cZ_{\Omega_{\widetilde{v},i}})^{\rig}\twoheadrightarrow (\Spec \cZ_{\Omega_{\widetilde{v},i}})^{\rig},\]
we obtain a $(\varphi, \Gamma)$-module $\Delta_{U,\widetilde{v},i}'$ over $\cR_{U,F_{\widetilde{v}}}$. We let
\[\Delta_{U, \widetilde{v},i}:=\Delta_{U,\widetilde{v},i}' \otimes_{\cR_{E,F_{\widetilde{v}}}} \cR_{E,F_{\widetilde{v}}}\big(\unr(q_{\widetilde{v}}^{s_{\widetilde{v},i-1}-(1-n_{\widetilde{v},i})/{2}})\big).\]
Thus for each point $x$ in $U$, the evaluation $x^* \Delta_{U,\widetilde{v},i}$ is isomorphic to $\Delta_{x, \widetilde{v},i}$. Applying \cite[Lemma 7.8.11]{BCh}, we obtain a rigid analytic space $\widetilde{U}$ with a finite dominant (surjective) morphism $g: \widetilde{U} \ra U$ and a locally free $\co(\widetilde{U})$-module $\rho_{\widetilde{U}}$ of rank $n$ equipped with a continuous $\Gal_F^S$-action such that $g^{-1}(C(U))$ is Zariski-dense in $\widetilde{U}$ and $\rho_{\widetilde{U}}|_x=\rho_{g(x)}$ for all $x \in \widetilde{U}$. Applying Corollary \ref{para} and Corollary \ref{rgloOF} (1) to (noting that, for $x\in \widetilde{U}^{\red}$, $\Delta_i|_x=\Delta_{g(x),\widetilde{v},i}$ and $\delta_i|_x=\chi_{g(x),\widetilde{v},i}$)
	\begin{equation*}
		\big\{X=\widetilde{U}^{\red}, M=D_{\rig}(\rho_{\widetilde{U}^{\red},\widetilde{v}}), \Delta_i:=g^* \Delta_{U, \widetilde{v},i}|_{\widetilde{U}^{\red}}, \delta_i:= (g^* \chi_{U,\widetilde{v},i}) z^{\textbf{h}_{\widetilde{v}, s_{i}}}|_{\widetilde{U}^{\red}}\big\},
	\end{equation*}
	the theorem follows.
\end{proof}

\begin{proposition}\label{galBE1}
	Let $x=(\eta_x, \pi_{x,L_P},\chi_x)$ be a point of $\cE_{\Omega, \lambda}(U^p)$. Then for $v\in S_p$, and $\tau \in \Sigma_{\widetilde{v}}$, the Sen $\tau$-weights of $\rho_{x,\widetilde{v}}$ are given by $\{h_{j_i,\tau}+\wt(\chi_{x,\widetilde{v},i})_{\tau}\}_{\substack{1\leq i \leq r \\ s_{i-1}+1 \leq j_i \leq s_i}}$.
\end{proposition}
\begin{proof}
We use the notation in the proof of Theorem \ref{FOBE}. The statement holds for very classical non-critical generic points, hence for points in $g^{-1}(C(U))\subset \widetilde{U}$. Since the Sen $\tau$-weights are analytic functions on $\widetilde{U}$ (see for example \cite[Def. 6.2.11]{KPX}), we deduce the statement holds for any point in $\widetilde{U}$. The proposition follows.
\end{proof}

Let $\overline{\rho}$ be a continuous representation of $\Gal_F$ of dimension $n$ over $k_E$ such that $\overline{\rho}(c\cdot c) \cong \overline{\rho}^\vee \otimes \overline{\chi_{\cyc}}^{1-n}$ (where $\Gal(F/F^+)=\{1,c\}$) and $\overline{\rho}$ is unramified outside $S$. To $\overline{\rho}$, we associate a maximal ideal $\fm_{\overline{\rho}}$ of $\bT^S$ of residue field $k_E$ such that, denoting $\eta_{\overline{\rho}}: \bT^S /\fm_{\overline{\rho}} \twoheadrightarrow k_E$ the corresponding morphism, the characteristic polynomial of $\overline{\rho}(\Frob_{\widetilde{v}})$ is given by (\ref{ES}) with $\eta_{\pi}$ replaced by $\eta_{\overline{\rho}}$ for all $v\notin S$ that splits in $F$. The representation $\overline{\rho}$ is called \textit{$U^p$-modular} if the localization $S(U^p,\co_E/\varpi_E^s)_{\overline{\rho}}:=S(U^p, \co_E/\varpi_E^s)_{\fm_{\overline{\rho}}}$ is non-zero (by e.g.\ \cite[Lemma 6.5]{BD1}) there exist only finitely many $\overline{\rho}$ such that $\overline{\rho}$ is $U^p$-modular). We define 
\[\widehat{S}(U^p, \co_E)_{\overline{\rho}}:=\varprojlim_s S(U^p, \co_E/\varpi_E^s)_{\overline{\rho}}\textrm{\ \ \ and\ \ \ }\widehat{S}(U^p,E)_{\overline{\rho}}:=\widehat{S}(U^p,\co_E)_{\overline{\rho}}\otimes_{\co_E} E.\]
We construct $\cE_{\Omega, \lambda}(U^p,\overline{\rho})$ from $\widehat{S}(U^p,E)_{\overline{\rho}}$ exactly as we construct $\cE_{\Omega,\lambda}(U^p)$. Since $\widehat{S}(U^p,E)_{\overline{\rho}}$ is a direct summand of $\widehat{S}(U^p,E)$ (equivariant under the action of $G(F \otimes_{\Q} \Q_p) \times \bT^S$), one easily sees that all the previous results hold with $\cE_{\Omega, \lambda}(U^p)$ replaced by $\cE_{\Omega, \lambda}(U^p, \overline{\rho})$.

\subsubsection{Locally analytic socle and companion points}\label{seccompCP}

We recall the locally analytic socle conjecture of \cite{Br13I}, \cite{Br13II} for generic potentially crystalline representations and discuss its relation with companion points on Bernstein eigenvarieties. 

To any $n$-dimensional continuous representation $\rho$ of $\Gal_F$ over $E$ we associate a maximal ideal $\fm_{{\rho}}$ of $\bT^S\otimes_{\co_E}E$ of residue field $E$ such that, denoting $\eta_{{\rho}}: \bT^S \twoheadrightarrow E$ the corresponding morphism, the characteristic polynomial of ${\rho}(\Frob_{\widetilde{v}})$ is given by (\ref{ES}) with $\eta_{\pi}$ replaced by $\eta_{{\rho}}$ for all $v\notin S$ that splits in $F$. Let $\rho$ be such a representation and assume that $\widehat{S}(U^p,E)[\fm_{\rho}]=\widehat{S}(U^p,E)[\bT^S=\eta_{\rho}]$ is non-zero. For $v\in S_p$, assume also that $\rho_{\widetilde{v}}$ is generic potentially crystalline (see \S~\ref{introPcr}) of Hodge-Tate weights $(h_{\widetilde{v}, 1,\tau}>h_{\widetilde{v},2,\tau}> \cdots > h_{\widetilde{v},n,\tau})_{\tau \in \Sigma_{\widetilde{v}}}$. Let $\sF_{\widetilde{v}}$ be a minimal parabolic filtration of $\ttr(\rho_{\widetilde{v}})$, and $P_{\widetilde{v}}$ be the associated parabolic subgroup of $\GL_n$ (cf.\ \S~\ref{introPcr}). We use the notation $n_{\widetilde{v},i}$, $r_{\widetilde{v}}$ and $s_{\widetilde{v},i}$ of \S~\ref{s: B0}. We let $\pi_{L_{P_{\widetilde{v}}}}:=\boxtimes_{i=1}^{r_{\widetilde{v}}} \pi_{\widetilde{v},i}$ be the smooth irreducible representation of $L_{P_{\widetilde{v}}}(F_{\widetilde{v}})$ over $E$ such that $\rec(\pi_{\widetilde{v},i})((1-n_{\widetilde{v},i})/2-s_{\widetilde{v},i})\cong \gr_i^{\sF_{\widetilde{v}}} \ttr(\rho)$. For $\tau\in \Sigma_{\widetilde{v}}$ and $j=1,\dots, n$, let $\lambda_{\widetilde{v},\tau,j}=h_{\widetilde{v},, \tau,j}+j-1$, then $\lambda_{\widetilde{v}}:=(\lambda_{\widetilde{v},1,\tau}, \dots, \lambda_{\widetilde{v},n,\tau})_{\tau\in \Sigma_{\widetilde{v}}}$ is a dominant weight of $\Res^{F_{\widetilde{v}}}_{\Q_p} \GL_n$ (with respect to $\Res^{F_{\widetilde{v}}}_{\Q_p} B$). For $w_{\widetilde{v}}\in \sW^{P_{\widetilde{v}}}_{\min,F_{\widetilde{v}}}$, consider the following locally $\Q_p$-analytic representation of $\GL_n(F_{\widetilde{v}})$ over $E$:
\begin{equation*}
C(w_{\widetilde{v}},\sF_{\widetilde{v}}):=\cF_{P_{\widetilde{v}}^-(F_{\widetilde{v}})}^{\GL_n(F_{\widetilde{v}})}\big(L^-(-w_{\widetilde{v}}\cdot \lambda_{\widetilde{v}}), \pi_{L_{P_{\widetilde{v}}}} \otimes_E \delta_{P_{\widetilde{v}}}^{-1}\big).
\end{equation*} 
It is topologically irreducible by \cite[Thm.\ (iv)]{OS}) (indeed, since $\rho_{\widetilde{v}}$ is generic, the smooth induction $(\Ind_{P_{\widetilde{v}}(F_{\widetilde{v}})\cap L_{Q}(F_{\widetilde{v}})}^{L_Q(F_{\widetilde{v}})} \pi_{L_{P_{\widetilde{v}}}} \otimes_E \delta_{P_{\widetilde{v}}}^{-1})^{\infty}$ is irreducible for any parabolic $Q\supset P_{\widetilde{v}}$). Note that we have $C(1,\sF_{\widetilde{v}}) \cong \pi_{\widetilde{v}} \otimes_E L(\lambda_{\widetilde{v}})$ where $\pi_{\widetilde{v}}:= (\Ind_{P_{\widetilde{v}}^-}^{\GL_n}\pi_{L_{P_{\widetilde{v}}}} \otimes_E \delta_{P_{\widetilde{v}}}^{-1})^{\infty}$ is the smooth irreducible representation of $\GL_n(F_{\widetilde{v}})$ such that $\rec(\pi_{\widetilde{v}})(\frac{1-n}{2})\cong \ttr(\rho_{\widetilde{v}})$. As discussed in \S~\ref{introPcr}, to the filtration $\sF_{\widetilde{v}}$, we associate an element $w_{\sF_{\widetilde{v}}}\in \sW^{P_{\widetilde{v}}}_{\max, F_{\widetilde{v}}}$. The following conjecture is a special case of \cite[Conj.\ 5.3]{Br13II}:

\begin{conjecture}\label{conjSoc}
For $v\in S_p$, let $w_{\widetilde{v}}\in \sW^{P_{\widetilde{v}}}_{\min, F_{\widetilde{v}}}$. Then $\widehat{\otimes}_{v\in S_p}C(w_{\widetilde{v}},\sF_{\widetilde{v}})$ is a subrepresentation of $\widehat{S}(U^p,E)[\fm_{\rho}]$ if and only if $w_{\widetilde{v}} \leq w_{\sF_{\widetilde{v}}}w_{0,F_{\widetilde{v}}}$ for all $v\in S_p$. 
\end{conjecture}

Conjecture \ref{conjSoc} in the crystalline case has been proved in \cite{BHS3} under Taylor-Wiles hypothesis. However, almost nothing (except for some very partial results in \cite{Ding9}) was known when $\rho_{\widetilde{v}}$ is {\it not} trianguline. As already mentioned in \S~\ref{intro}, one main motivation for this paper is to prove Conjecture \ref{conjSoc} for certain parabolic $P_{\widetilde{v}}$ (under the Taylor-Wiles hypothesis \ref{TayWil}) by using the Bernstein eigenvarieties of \S~\ref{s: B0}.

We now state a weaker version of Conjecture \ref{conjSoc} which is formulated in terms of companion points on Bernstein eigenvarieties.

For $w=(w_{\widetilde{v}})_{v\in S_p}\in \sW^{P}_{\min}\cong \prod_{v\in S_p} \sW^{P_{\widetilde{v}}}_{\min, F_{\widetilde{v}}}$, the weight $w \cdot \lambda$ is $P$-dominant where $\lambda:=(\lambda_{\widetilde v})_{v\in S_p}$. Let $\Omega_{\sF}$ be the Bernstein component of $\pi_{L_P}:=\boxtimes_{v\in S_p} \pi_{L_{P_{\widetilde{v}}}}$.

\begin{conjecture} \label{conjCP}
	We have
	\begin{equation}\label{eqCP}
		\big(\eta_{\rho},\boxtimes_{v\in S_p} \pi_{L_{P_{\widetilde{v}}}}, 1\big) \in \cE_{\Omega_{\sF}, w\cdot \lambda}(U^p)
	\end{equation}
	if and only if $w_{\widetilde{v}} \leq w_{\sF_{\widetilde{v}}}w_{0,F_{\widetilde{v}}}$ for all $v\in S_p$.
\end{conjecture}

\begin{lemma}\label{lemcompCP}
(1) Conjecture \ref{conjSoc} implies Conjecture \ref{conjCP}.
	
(2) The ``only if" part of Conjecture \ref{conjCP} implies the ``only if" part of Conjecture \ref{conjSoc}.
\end{lemma}
\begin{proof}
(1) Suppose Conjecture \ref{conjSoc} holds. By (\ref{fiberE}), (\ref{eqCP}) is equivalent to 
	\begin{equation}\label{CptiJac}
		\Hom_{L_P(\Q_p)}\big(\pi_{L_P} \otimes_E L(w\cdot \lambda)_P, J_P(\widehat{S}(U^p,E)^{\an})[\fm_{\rho}]\big) \neq 0.
	\end{equation} 
	If $w_{\widetilde{v}} \leq w_{\sF_{\widetilde{v}}}w_{0,F_{\widetilde{v}}}$ for all $v\in S_p$, then by assumption
\[\cF_{P^-}^{G_p}(L^-(-w\cdot \lambda), \pi_{L_P} \otimes_E \delta_P^{-1})\cong \widehat{\otimes}_{v\in S_p}C(w_{\widetilde{v}},\sF_{\widetilde{v}})\hookrightarrow \widehat{S}(U^p,E)[\fm_{\rho}].\]
By \cite[Thm.\ 4.3]{Br13II}, this implies (\ref{CptiJac}). Conversely, suppose (\ref{CptiJac}) holds. By \cite[Thm.\ 4.3]{Br13II} and \cite[Thm.]{OS} (and the fact that the $\rho_{\widetilde{v}}$ are all generic), one deduces there exists $w'=(w_{\widetilde{v}}')$ with $w_{\widetilde{v}}'\geq w_{\widetilde{v}}$ for all $v \in S_p$ such that one has an embedding
	\begin{equation*}
		\cF_{P^-}^{G_p}(L^-(-w'\cdot \lambda), \pi_{L_P} \otimes_E \delta_P^{-1})\cong \widehat{\otimes}_{v\in S_p}C(w_{\widetilde{v}}',\sF_{\widetilde{v}})\hookrightarrow \widehat{S}(U^p,E)[\fm_{\rho}].
	\end{equation*} By Conjecture \ref{conjSoc}, this implies $w_{\widetilde{v}}'\leq w_{\sF_{\widetilde{v}}}w_{0,F_{\widetilde{v}}}$, and hence $w_{\widetilde{v}} \leq w_{\sF_{\widetilde{v}}}w_{0,F_{\widetilde{v}}}$.
	
(2) Suppose the ``only if" part of Conjecture \ref{conjCP} holds. For $w=(w_{\widetilde{v}})$, if one has an embedding $\widehat{\otimes}_{v\in S_p}C(w_{\widetilde{v}},\sF_{\widetilde{v}})\hookrightarrow \widehat{S}(U^p,E)[\fm_{\rho}]$, then as in the proof of (1), we deduce using \cite[Thm.\ 4.3]{Br13II} that (\ref{CptiJac}) holds and hence (\ref{eqCP}) holds. This implies $w_{\widetilde{v}} \leq w_{\sF_{\widetilde{v}}}w_{0,F_{\widetilde{v}}}$. 
\end{proof}

\begin{remark}
Suppose $\widehat{S}(U^p,E)^{\lalg}[\fm_{\rho}]\neq 0$ (for $\rho_{\widetilde{v}}$ generic potentially crystalline with distinct Hodge-Tate weights for all $v\in S_p$). Then using \cite{Car12}, we have an embedding
\[\boxtimes_{v\in S_p}\pi_{\widetilde{v}} \hookrightarrow \widehat{S}(U^p,E)^{\lalg}[\fm_{\rho}].\]
To each minimal parabolic filtration $\sF$ (that can be viewed as an analogue of a refinement in the crystalline case), we can associate (using \cite[Thm.\ 4.3]{Br13II}) a classical point 
	\begin{equation*}
		\big(\eta_{\rho}, \boxtimes_{v\in S_p} \pi_{L_{P_{\widetilde{v}}}}, 1\big) \in \cE_{\Omega_{\sF}, \lambda}(U^p).
	\end{equation*}
	When $w_{\sF}\neq w_{0,F}:=(w_{0,F_{\widetilde{v}}})_{v\in S_p}$, Conjecture \ref{conjCP} then predicts the existence of some non-classical points on the Bernstein eigenvarieties such that the associated Galois representation is still isomorphic to $\rho$. Such points are referred to as companion points of the classical point. \index{$w_{0,F}$}
\end{remark}

\subsection{Patched Bernstein eigenvarieties}\label{secPBern}

As in \cite[\S~3]{BHS1} for the ``usual'' eigenvarieties, we construct patched Bernstein eigenvarieties by applying the formalism in \S~\ref{abCon} to the patched $p$-adic automorphic representations of \cite{CEGGPS}. We give some basic properties of patched Bernstein eigenvarieties, and show that they have a close relation with certain (purely local) ``paraboline'' varieties (on the Galois side) that will be constructed in \S~\ref{secDef1}.

Let $\overline{\rho}$ be a $U^p$-modular continuous representation of $\Gal_F$ over $k_E$ (see the end of \S~\ref{galois}). We assume henceforth the following so-called Taylor-Wiles hypothesis (see \cite[\S~1]{BHS3}):

\begin{hypothesis}\label{TayWil}
(1) $p>2$;
	
(2) the field $F$ is unramified over $F^+$, $F$ does not contain a non trivial root $\sqrt[p]{1}$ of $1$ and $G$ is quasi-split at all finite places of $F^+$;
	
(3) $U_{v}$ is hyperspecial when the finite place $v$ of $F^+$ is inert in $F$;
	
(4) $\overline{\rho}$ is absolutely irreducible and $\overline{\rho}(\Gal_{F(\sqrt[p]{1})})$ is adequate.
\end{hypothesis}

The following functor
\begin{equation*}
	\Big\{\substack{\text{Local artinian $\co_E$-algebras} \\ \text{of residue field $k_E$}}\Big\} \ra \{\text{Sets}\}, \ \ A \mapsto \Big\{\substack{\text{Deformations $\rho_A$ of $\overline{\rho}$ over $A$}\\ \text{such that $\rho_A$ unramified outside $S$} \\ \text{and } \rho_A(c\cdot c)\cong \rho_A^\vee \otimes \chi_{\cyc}^{1-n}}\Big\}/\sim
\end{equation*} 
is pro-representable by a complete local Noetherian $\co_E$-algebra of residue field $k_E$, denoted by $R_{\overline{\rho}, S}$. For $s\in \Z_{\geq 1}$, and a compact open subgroup $U_p\subset \prod_{v\in S_p} \GL_n(\co_{F_{\widetilde{v}}})$, let $\bT(U^pU_p,\co_E/\varpi_E^s)_{\overline{\rho}}$ be the image of $\bT^S$ in the $\co_E/\varpi_E^s$-algebra of endomorphisms of $S(U^pU_p, \co_E/\varpi_E^s)_{\overline{\rho}}$. Put
\[\bT(U^p)_{\overline{\rho}}:=\varprojlim_s \varprojlim_{U_p} \bT(U^pU_p, \co_E/\varpi_E^s)_{\overline{\rho}},\]
which is a complete local $\co_E$-algebra. By \cite[Prop.\ 6.7]{Thor}, there is a natural surjective morphism of $\co_E$-algebras (hence $\bT(U^p)_{\overline{\rho}}$ is also Noetherian)
\begin{equation*}
	R_{\overline{\rho},S} \twoheadlongrightarrow \bT(U^p)_{\overline{\rho}}. 
\end{equation*}
The rigid space $\cE_{\Omega, \lambda}(U^p,\overline{\rho})$ is then a closed subspace of the rigid spaces
\[(\Spf \bT(U^p)_{\overline{\rho}})^{\rig} \times (\Spec \cZ_{\Omega})^{\rig} \times \widehat{\cZ_0} \hookrightarrow (\Spf R_{\overline{\rho},S})^{\rig} \times (\Spec \cZ_{\Omega})^{\rig} \times \widehat{\cZ_0}.\]
Indeed $B_{\Omega, \lambda}(\widehat{S}(U^p,E)_{\overline{\rho}}^{\Q_p-\an})^{\vee}$ gives rise to a coherent sheaf over $(\Spf R_{\overline{\rho},S})^{\rig} \times (\Spec \cZ_{\Omega})^{\rig} \times \widehat{\cZ_0}$, whose Zariski-support is exactly $\cE_{\Omega, \lambda}(U^p,\overline{\rho})$.

We let $R_{\overline{\rho}_{\widetilde{v}}}$ be the maximal reduced and $\Z_p$-flat quotient of the \textit{framed} local deformation ring of $\overline{\rho}_{\widetilde{v}}$ (which was denoted by $R_{\overline{\rho}_{\widetilde{v}}}^{\overline{\square}}$ in \cite{BHS1}) and we put 
\begin{eqnarray*}
	&&R^{\loc}:=\widehat{\otimes}_{v\in S} R_{\overline{\rho}_{\widetilde{v}}}, \ R_{\overline{\rho}^p}:=\widehat{\otimes}_{v\in S\setminus S_p} R_{\overline{\rho}_{\widetilde{v}}}, \\
	&&R_{\overline{\rho}_p}:=\widehat{\otimes}_{v\in S_p} R_{\overline{\rho}_{\widetilde{v}}}, \ R_{\infty}:=R^{\loc}[[x_1, \dots, x_g]], \ R_{\infty}^p:=R_{\overline{\rho}^p}[[x_1, \dots, x_g]],
\end{eqnarray*}
where $g\geq 1$ is an integer which will be fixed below. We let $S_{\infty}:=\co_E[[y_1, \dots, y_t]]$ where $t=g+[F^+:\Q]\frac{n(n-1)}{2}+|S| n^2$ and $\fa:=(y_1, \dots, y_t)\subset S_{\infty}$. Shrinking $U^p$ (and $S$), we can and do assume 
\begin{equation*}
	G(F) \cap (hU^p K_p h^{-1})=\{1\} \text{ for all $h\in G(\bA_{F^+}^{\infty})$}
\end{equation*}
where $K_p:=\prod_{v\in S_p} K_v:=\prod_{v\in S_p}i_{\widetilde{v}}^{-1}(\GL_n(\co_{F_{\widetilde{v}}}))$. Then the action of $R_{\overline{\rho},S}$ on $\widehat{S}(U^p,E)_{\overline{\rho}}$ factors through a quotient $R_{\overline{\rho},S} \twoheadrightarrow R_{\overline{\rho},\cS}$, where $R_{\overline{\rho},\cS}$ denotes the universal deformation ring of the deformation problem:
\begin{equation*}
	\cS=(F/F^+, S, \widetilde{S}, \co_E, \overline{\rho}, \chi_{\cyc}^{1-n} \delta_{F/F^+}^n, \{R_{\overline{\rho}_{\widetilde{v}}}\}_{v\in S}).
\end{equation*}
By \cite[Thm.\ 3.5]{BHS1} (generalizing \cite{CEGGPS}), there exist an integer $g\geq 1$ and 
\begin{enumerate}
	\item[(1)] a continuous $R_{\infty}$-admissible unitary representation $\Pi_{\infty}$ of $G_p$ over $E$ together with a $G_p$-stable and $R_{\infty}$-stable unit ball $\Pi_{\infty}^0 \subset \Pi_{\infty}$;
	\item[(2)] a morphism of local $\co_E$-algebras $S_{\infty} \ra R_{\infty}$ such that $M_{\infty}:=\Hom_{\co_E}(\Pi_{\infty}^0, \co_E)$ is a finite projective $S_{\infty}[[K_p]]$-module;
	\item[(3)] a surjection $R_{\infty}/\fa R_{\infty} \twoheadrightarrow R_{\overline{\rho},\cS}$ and a compatible $G_p$-equivariant isomorphism $\Pi_{\infty}[\fa]\cong \widehat{S}(U^p, E)_{\overline{\rho}}$.
\end{enumerate}
Denote by $\Pi_{\infty}^{R_{\infty}-\an}$ the subrepresentation of $G_p$ of locally $R_{\infty}$-analytic vectors of $\Pi_{\infty}$ (cf.\ \cite[\S~3.1]{BHS1}. Let $\Omega$, $\sigma$, $\lambda$ be as in \S~\ref{s: B0} and consider (with $J_P(\Pi_{\infty}^{R_{\infty}-\an})_{\lambda}$ defined as in (\ref{jplambda})):
\[B_{\Omega, \lambda}(\Pi_{\infty}^{R_{\infty}-\an}):=B_{\sigma, \lambda}(\Pi_{\infty}^{R_{\infty}-\an})=\Hom_{K_p}\big(\sigma, J_P(\Pi_{\infty}^{R_{\infty}-\an})_{\lambda} \widehat{\otimes}_E \cC^{\Q_p-\la}(Z_{L_P}^0,E)\big).\]
By \cite[Prop.\ 3.4]{BHS1} and an easy variation of the proof of Lemma \ref{Besadm}, we see that $B_{\Omega, \lambda}(\Pi_{\infty}^{R_{\infty}-\an})^{\vee}$ is a coadmissible module over ${\mathcal O}((\Spf R_{\infty})^{\rig} \times \widehat{Z_{\ul{\varpi}}}\times \widehat{\cZ_0})$ which corresponds to a coherent sheaf $\cM_{\Omega, \lambda}^{\infty,0}$ over $(\Spf R_{\infty})^{\rig} \times \widehat{Z_{\ul{\varpi}}}\times \widehat{\cZ_0}$. Taking into account the $\cZ_{\Omega}$-action, $\cM_{\Omega, \lambda}^{\infty, 0}$ gives rise to a coherent sheaf $\cM_{\Omega, \lambda}^{\infty}$ over $(\Spf R_{\infty})^{\rig} \times (\Spec \cZ_{\Omega})^{\rig} \times \widehat{\cZ_0}$, such that 
\begin{equation}\label{glosecBE}
\Gamma\big((\Spf R_{\infty})^{\rig} \times (\Spec \cZ_{\Omega})^{\rig} \times \widehat{\cZ_0}, \cM_{\Omega, \lambda}^{\infty}\big)\cong B_{\Omega, \lambda}(\Pi_{\infty}^{R_{\infty}-\an})^{\vee}.
\end{equation}
We let
\begin{equation}\label{patchedbernstein}
\cE_{\Omega, \lambda}^{\infty}(\overline{\rho})\hookrightarrow(\Spf R_{\infty})^{\rig} \times (\Spec \cZ_{\Omega})^{\rig} \times \widehat{\cZ_0}
\end{equation}
be the Zariski-closed support of $\cM_{\Omega, \lambda}^{\infty}$, and call $\cE_{\Omega, \lambda}^{\infty}(\overline{\rho})$ a \textit{patched Bernstein eigenvariety}. \index{$\cE_{\Omega, \lambda}^{\infty}(\overline{\rho})$} \index{$\cM_{\Omega, \lambda}^{\infty}$}

Let $\cW_{\infty}:=(\Spf S_{\infty})^{\rig} \times \widehat{\cZ_0}$, and $z\in Z_{\ul{\varpi}}$ be as in the discussion above Proposition \ref{Pr}. As in (\ref{kappaz}) denote by $\kappa_z$ the following composition
\begin{multline*}
	\kappa_z: \cE_{\Omega, \lambda}^{\infty}(\overline{\rho})\hookrightarrow(\Spf R_{\infty})^{\rig} \times (\Spec \cZ_{\Omega})^{\rig} \times \widehat{\cZ_0} \\
	\ra (\Spf S_{\infty})^{\rig} \times (\Spec E[Z_{\ul{\varpi}}])^{\rig}\times \widehat{\cZ_0} \ra (\Spf S_{\infty})^{\rig} \times \bG_m^{\rig} \times \widehat{\cZ_0} \cong \cW_{\infty} \times \bG_m^{\rig}
\end{multline*}
where the second and third maps come from $E[Y_z] \hookrightarrow E[Z_{\ul{\varpi}}] \ra \cZ_{\Omega}$. Denote by $\omega$ the composition:
\[\omega: \cE_{\Omega, \lambda}^{\infty}(\overline{\rho}) \xlongrightarrow{\kappa_z} \cW_{\infty} \times \bG_m^{\rig} \lra \cW_{\infty}.\]
By an easy generalization of Proposition \ref{Pr} to the case $\widetilde{G}_p:=G_p \times \Z_p^q$, $V=\Pi_{\infty}^{R_{\infty}-\an}=\Pi_{\infty}^{S_{\infty}-\an}$ (adding the extra factor $\Z_p^q$ everywhere in the proof), one can show that $B_{\Omega, \lambda}(\Pi_{\infty}^{R_{\infty}-\an})^{\vee}$ is a coherent sheaf over $\cW_{\infty} \times \bG_m^{\rig}$. We denote by $Z_z(\overline{\rho})\hookrightarrow \cW_{\infty} \times \bG_m^{\rig}$ its Zariski-closed support. We have the following analogue of Proposition \ref{fredholm1}, which is proven by the same argument as in the proofs of \cite[Lemma 3.10]{BHS1} and \cite[Prop.\ 3.11]{BHS1}.

\begin{proposition}\label{PBEFre}
(1) The rigid variety $Z_z(\overline{\rho})$ is a Fredholm hypersurface in $\cW_{\infty} \times \bG_m^{\rig}$. Moreover, there exists an admissible covering $\{U_i'\}$ of $Z_z(\overline{\rho})$ by affinoids $U_i'$ such that the morphism
	\begin{equation*}
		g: Z_z(\overline{\rho}) \hookrightarrow \cW_{\infty} \times \bG_m^{\rig} \twoheadrightarrow \cW_{\infty}
	\end{equation*} 
	induces a surjective finite morphism from $U_i'$ to an affinoid open $W_i$ of $\cW_{\infty}$, and such that $U_i'$ is a connected component of $g^{-1}(W_i)$. For each $i$, $\Gamma\big(U_i', (\kappa_z)_* \cM_{\Omega, \lambda}^{\infty}\big)$ is a finitely generated projective $\co_{\cW_{\infty}}(W_i)$-module. 
	
(2) There exists an admissible covering $\{U_i\}$ of $\cE_{\Omega, \lambda}^{\infty}(\overline{\rho})$ by affinoids $U_i$ such that
	\begin{itemize}
		\item there exists an affinoid open $W_i$ of $\cW_{\infty}$ satisfying that $\omega$ is a finite surjective morphism from each irreducible component of $U_i$ to $W_i$;
		\item $\co_{\cE_{\Omega, \lambda}^{\infty}(\overline{\rho})}(U_i)$ is isomorphic to an $\co_{\cW_{\infty}}(W_i)$-algebra of endomorphisms of a finitely generated projective $\co_{\cW_{\infty}}(W_i)$-module.
	\end{itemize}
\end{proposition}

We deduce the following analogue of Corollary \ref{CMC1} by the same arguments as in the proofs of \cite[Cor.\ 3.12]{BHS1}, \cite[Cor.\ 3.13]{BHS1} and \cite[Lemma 3.8]{BHS2}:

\begin{corollary}\label{PBEdim}
(1) The rigid analytic space $\cE_{\Omega, \lambda}^{\infty}(\overline{\rho})$ is equidimensional of dimension
	\begin{equation*}
		g+|S|n^2+\sum_{v\in S_p} \Big([F_{\widetilde{v}}:\Q_p] \big(\frac{n(n-1)}{2}+r_{\widetilde{v}}\big)\Big)
	\end{equation*}
	and has no embedded component.
	
(2) The morphism $\kappa_z$ is finite and the image of an irreducible component of $\cE_{\Omega, \lambda}^{\infty}(\overline{\rho})$ is an irreducible component of $Z_z(\overline{\rho})$. The image of an irreducible component of $\cE_{\Omega, \lambda}^{\infty}(\overline{\rho})$ under $\omega$ is a Zariski-open of $\cW_{\infty}$. 
	
(3) The coherent sheaf $\cM_{\Omega, \lambda}^{\infty}$ is Cohen-Macaulay over $\cE_{\Omega, \lambda}^{\infty}(\overline{\rho})$. 
\end{corollary}

We say that a point $x=(y, \pi_{x,L_P}, \chi_x)\in \cE_{\Omega, \lambda}^{\infty}(\overline{\rho})\hookrightarrow (\Spf R_{\infty})^{\rig} \times (\Spec \cZ_{\Omega})^{\rig} \times \widehat{\cZ_0}$ is \textit{classical} if (\ref{defClas}) is satisfied with $\widehat{S}(U^p,E)^{\lalg}$ replaced by $(\Pi_{\infty}^{R_{\infty}-\an})^{\lalg}$ and $[\bT^S=\eta_x]$ replaced by $[\fm_y]$, where $\fm_y$ is the maximal ideal of $R_{\infty}[\frac{1}{p}]$ associated to $y$. We say $x$ is \textit{very classical} if the conditions in Definition \ref{defclass} (2) hold (these conditions only concern $\pi_{x,L_P}$ and $\chi_x$). For $v\in S$, denote by $\rho_{x,\widetilde{v}}$ the $\Gal_{F_{\widetilde{v}}}$-representation associated to $x$ via the image of $y \in (\Spf R_{\infty})^{\rig}\cong \prod_{v\in S_p} (\Spf R_{\overline{\rho}_{\widetilde{v}}})^{\rig} \times \bU^g$ in $(\Spf R_{\overline{\rho}_{\widetilde{v}}})^{\rig}$ where $\bU$ denotes the open unit ball in $\bA^1$. Let $\ttr_{x,\widetilde{v}}$ be the Weil-Deligne representation associated to $\rho_{x,\widetilde{v}}$, $\ttr_{x,\widetilde{v},i}:=\rec(\pi_{x,\widetilde{v},i})(\frac{1-n_{\widetilde{v},i}}{2}-s_{\widetilde{v},i-1})$, and $(\chi_{x,\widetilde{v},i})_{\varpi_{\widetilde{v}}}^{\infty}$ be the smooth character of $F_{\widetilde{v}}^{\times}$ associated to $\chi_{x,\widetilde{v}}$ as in the discussion below (\ref{injlalg}). A classical point $x$ is called \textit{generic} if (\ref{defGene}) holds for $\epsilon=0,1$ and $i\neq i'$. 

\begin{proposition}\label{galPBE}
Let $(y, \pi_{x,L_P}, \chi_x)\in \cE_{\Omega, \lambda}^{\infty}(\overline{\rho})$ be a generic classical point. For all $v\in S_p$, $\rho_{x,\widetilde{v}}$ is potentially crystalline of Hodge-Tate weights $\{\lambda^x_{\widetilde{v},i,\tau}-i+1\}_{\substack{\tau \in \Sigma_{\widetilde{v}}\\ 1\leq i \leq n}}$. Moreover, we have
	\begin{equation}\label{noN2}
		\ttr_{x,\widetilde{v}} \cong \bigoplus_{i=1}^{r_{\widetilde{v}}} \big(\ttr_{x,\widetilde{v},i} \otimes_E \rec((\chi_{x,\widetilde{v},i})_{\varpi_{\widetilde{v}}}^{\infty})\big).
	\end{equation}
\end{proposition}
\begin{proof}
	The proposition follows from the results in \cite[\S~4]{CEGGPS} by similar arguments as in the proof of \cite[Prop.\ 3.16]{BHS1}. Since $x$ is classical, we have a non-zero morphism (recall $\lambda^x=\lambda+(\wt(\chi_x) \circ \dett_{L_P})$)
	\begin{equation}\label{injlalg1}
		\Big(\Ind_{P^-(\Q_p)}^{G_p} \big(\pi_{x,L_P} \otimes_E ((\chi_x)_{\ul{\varpi}}^{\infty} \circ \dett_{L_P}) \otimes_E \delta_P^{-1}\big)\Big)^{\infty} \otimes_E L(\lambda^x) \lra (\Pi_{\infty}^{R_{\infty}-\an})^{\lalg}[\fm_y].
	\end{equation}
	Since $x$ is generic, the representation on the left hand side is absolutely irreducible and hence (\ref{injlalg1}) is injective. For $v\in S_p$, let $G_{\widetilde{v}}:=\GL_n(F_{\widetilde{v}})$, and $\Omega_{G_{\widetilde{v}}}$ be the Bernstein component of
\[\pi_{x, \widetilde{v}}:=\Big(\Ind_{P^-_{\widetilde{v}}(F_{\widetilde{v}})}^{G_{\widetilde{v}}} \big(\pi_{x,L_{P_{\widetilde{v}}}} \otimes_E ((\chi_{x,\widetilde{v}})_{\varpi_{\widetilde{v}}}^{\infty} \circ \dett_{L_{P_{\widetilde{v}}}}) \otimes_E \delta_{P_{\widetilde{v}}}^{-1}\big)\Big)^{\infty}\]
	(do not confuse $\Omega_{G_{\widetilde{v}}}$ with the Bernstein component $\Omega_{\widetilde{v}}$ of $L_{P_{\widetilde{v}}}(F_{\widetilde{v}})$). By \cite[Thm.\ 3.7]{CEGGPS} and \cite[Cor.\ 3.12]{CEGGPS}, there exists a smooth irreducible representation $\sigma_{\widetilde{v}}$ of $K_{\widetilde{v}}:=\GL_n(\co_{F_{\widetilde{v}}})$ such that $\End_{G_{\widetilde{v}}}(\cind_{K_{\widetilde{v}}}^{G_{\widetilde{v}}} \sigma_{\widetilde{v}})\cong \cZ_{\Omega_{G_{\widetilde{v}}}}$ and $\pi_{x, \widetilde{v}} \cong \big(\cind_{K_{\widetilde{v}}}^{G_{\widetilde{v}}} \sigma_{\widetilde{v}}\big)\otimes_{\cZ_{\Omega_{G_{\widetilde{v}}}}, \theta_{x,\widetilde{v}}} E$, where $\theta_{x,\widetilde{v}}:\cZ_{\Omega_{G_{\widetilde{v}}}}\ra E$ denotes the character corresponding to $\pi_{x,\widetilde{v}}$. Let $K_p:=\prod_{v\in S_p} K_{\widetilde{v}}$, $\sigma_{K_p}:=\boxtimes_{v\in S_p} \sigma_{K_{\widetilde{v}}}$, $\Omega_{G_p}:=\prod_{v\in S_p} \Omega_{G_{\widetilde{v}}}$, and
\[\pi_x:=\boxtimes_{v\in S_p} \pi_{x,\widetilde{v}} \cong \big(\cind_{K_p}^{G_p} \sigma_{K_p}\big) \otimes_{\cZ_{\Omega_{G_p}}, \theta_x} E\cong \Big(\Ind_{P^-(\Q_p)}^{G_p} \big(\pi_{x,L_P} \otimes_E ((\chi_x)_{\ul{\varpi}}^{\infty} \circ \dett_{L_P}) \otimes_E \delta_P^{-1}\big)\Big)^{\infty}\]
(where $\theta_x:\cZ_{\Omega_{G_p}}\ra E$ denotes the character corresponding to $\pi_x$). By Frobenius reciprocity, we have
	\begin{equation*}
		\Pi_{\infty}(\sigma_{K_p}, \lambda^x):=\Hom_{K_p}\big(\sigma_{K_p}, \Pi_{\infty} \otimes_E L(\lambda^x)^{\vee}\big) \cong \Hom_{G_p}\big(\cind_{K_p}^{G_p} \sigma_{K_p}, \Pi_{\infty} \otimes_E L(\lambda^x)^{\vee}\big).
	\end{equation*}
By the injection (\ref{injlalg1}), we deduce that $\Pi_{\infty}(\sigma_{K_p}, \lambda^x)[\fm_y]\neq 0$ and that there exists a non-zero subspace of $\Pi_{\infty}(\sigma_{K_p}, \lambda^x)[\fm_y]$ on which $\cZ_{\Omega_{G_p}}$ acts via $\theta_x$. Let $\xi_{\widetilde{v}}$ be the inertial type associated to $\sigma_{\widetilde{v}}$, $\textbf{h}_{\widetilde{v}}^x:=(h_{\widetilde{v},i,\tau}^x)_{\substack{i=1, \dots, n\\ \tau \in \Sigma_{\widetilde{v}}}}$ with $h_{\widetilde{v},i,\tau}^x=\lambda^x_{\widetilde{v},i,\tau}-i+1$. By \cite[Lemma 4.17 (1)]{CEGGPS}, the action of $\widehat{\otimes}_{v\in S_p} R_{\overline{\rho}_{\widetilde{v}}}$ on $\Pi_{\infty}(\sigma_{K_p}, \lambda^x)$ factors through $\widehat{\otimes}_{v\in S_p} R_{\overline{\rho}_{\widetilde{v}}}^{\pcr}(\xi_{\widetilde{v}},\textbf{h}^x_{\widetilde{v}})$. The first part of the proposition follows. By \cite[Thm.\ 4.1]{CEGGPS}, there is a morphism $\cZ_{\Omega_{G_{\widetilde{v}}}}\ra R_{\overline{\rho}_{\widetilde{v}}}^{\pcr}(\xi_{\widetilde{v}},\lambda^x_{\widetilde{v}})[1/p]$ that interpolates the local Langlands correspondence (with the same normalization as in \S~\ref{Nota2.1}). By \cite[Lemma 4.17 (2)]{CEGGPS}, the action of $\cZ_{\Omega_{G_p}}$ on $\Pi_{\infty}(\sigma_{K_p}, \lambda^x)$ factors through $\cZ_{\Omega_{G_p}}\ra (\widehat{\otimes}_{v\in S_p} R_{\overline{\rho}_{\widetilde{v}}}^{\pcr}(\xi_{\widetilde{v}},\lambda^x_{\widetilde{v}}))[1/p]$. We deduce $\ttr_{x,\widetilde{v}}\cong \rec(\pi_{x,\widetilde{v}})((1-n)/2)$, and the second part follows then from (\ref{recInd}).
\end{proof}
For a generic classical point $x\in \cE_{\Omega, \lambda}^{\infty}(\overline{\rho})$, we use (\ref{noN2}) to define a $P_{\widetilde{v}}$-filtration $\sF_{\widetilde{v}}$ on $\ttr_{x,\widetilde{v}}$ as in (\ref{Pvfil1}). We call $x$ {\it non-critical} if $(\rho_{x,\widetilde{v}}, \sF_{\widetilde{v}})$ is non-critical for all $v\in S_p$. 
Using Proposition \ref{class}, Lemma \ref{lemGenec}, and by the same argument as in the proof of \cite[Thm.\ 3.9]{BHS1}, we have the following analogue of Theorem \ref{denNCG}:

\begin{theorem}\label{classPatc}
The set of very classical non-critical generic points is Zariski-dense in $\cE_{\Omega, \lambda}^{\infty}(\overline{\rho})$. Moreover, for any point $x=(y, \pi_{L_P},\chi)\in \cE_{\Omega, \lambda}^{\infty}(\overline{\rho})$ with $\chi$ locally algebraic, and for any admissible open $U\subseteq \cE_{\Omega, \lambda}^{\infty}(\overline{\rho})$ containing $x$, there exists an admissible open $V\subseteq U$ containing $x$ such that the set of very classical non-critical generic points in $V$ is Zariski-dense in $V$. 
\end{theorem}

\begin{proposition}\label{PBEred}
The rigid space $\cE_{\Omega, \lambda}^{\infty}(\overline{\rho})$ is reduced.
\end{proposition}
\begin{proof}
The proposition follows by an easy variation of the proof of \cite[Cor.\ 3.20]{BHS1}. We briefly indicate below the changes. We define verbatim the $R_{\infty}$-module $\Sigma$ of \textit{loc.\ cit.} In \cite{BHS1}, it was a finite length $T_p$-module on which $T_p^0:=\prod_{v\in S_p} T(\co_{F_{\widetilde{v}}})$ acts by the character $\delta_{\lambda}$ of \textit{loc.\ cit.} In our case, it becomes a finite length $\cZ_{\Omega}$-module on which $\cZ_0\cong Z_{L_P}^0$ acts by an algebraic character. Similarly as in \textit{loc.\ cit.}, we are reduced to show that the $R_{\infty}$-action and the $\cZ_{\Omega}$-action on $\Sigma$ are both semi-simple. To obtain an analogue of \cite[(3.13)]{BHS1}, we use Lemma \ref{BCmod} and \cite[Thm.\ 4.3]{Br13II}. The ``$\cH$" in the proof of \cite[Cor.\ 3.20]{BHS1} has to be replaced by a Bernstein component $\Omega_{G_p}=\boxtimes_{v|p} \Omega_{G_{\widetilde{v}}}$ of $G_p$ similarly as the one appearing in the proof of Proposition \ref{galPBE}, and ``$\rm{ind}_{K_p}^{G_p} 1$" is replaced by $\cind_{K_p}^{G_p} \sigma_{K_p}$ where $\sigma_{K_p}=\otimes_v \sigma_{\widetilde{v}}$ is the $K_p$-representation associated to $\Omega_{G_p}$ as in the proof of Proposition \ref{galPBE}. Finally, ``$R_{\overline{\rho}_p}^{\square, \textbf{k}-\rm{cr}}$" has to be replaced by $\widehat{\otimes}_{v\in S_p}R_{\overline{\rho}_{\widetilde{v}}}^{\pcr}(\xi_{\widetilde{v}},\textbf{h}^x_{\widetilde{v}})$ (as in the proof of Proposition \ref{galPBE}).
\end{proof}

\begin{remark}
Assume $P=B_p$ where $B_p$ is as in Remark \ref{remP=B2}. Using the isomorphism $\iota_{\Omega, \lambda}$ in Remark \ref{remP=B}, we can view $\cE_{\Omega, \lambda}^{\infty}(\overline{\rho})$ as a (reduced) closed rigid analytic subspace of $(\Spf R_{\infty})^{\rig} \times \widehat{T_p}$ ($T_p$ as in Remark \ref{remP=B2}), which is independent of the choice of $(\Omega, \lambda)$ by Proposition \ref{twBEi} (trivially generalized to the patched case). By construction, the rigid analytic space $\cE_{\Omega, \lambda}^{\infty}(\overline{\rho})$ has the same points as the patched eigenvariety $X_p(\overline{\rho})$ of \cite{BHS1}. Using Proposition \ref{PBEred} and \cite[Cor.\ 3.20]{BHS1}, we actually obtain $\cE_{\Omega, \lambda}^{\infty}(\overline{\rho}) \cong X_p(\overline{\rho})$.
\end{remark}

\begin{proposition}\label{bePbe}
Let $(\Spf S_{\infty})^{\rig} \ra \Spec E$ be the morphism corresponding to the quotient by the ideal $\fa$. We have a natural morphism of rigid analytic spaces
\[\cE_{\Omega, \lambda}(U^p, \overline{\rho}) \lra \cE_{\Omega, \lambda}^{\infty}(\overline{\rho}) \times_{(\Spf S_{\infty})^{\rig}} \Spec E\]
which is bijective on points. 
\end{proposition}
\begin{proof}
By the same argument as in the proof of \cite[Thm.\ 4.2]{BHS1}, we have 
	\begin{equation*}
		J_P(\Pi_{\infty}^{R_{\infty}-\an})[\fa] \cong J_P(\Pi_{\infty}^{R_{\infty}-\an}[\fa])\cong J_P(\widehat{S}(U^p,E)_{\overline{\rho}}^{\an}).
	\end{equation*}
	Hence $B_{\Omega, \lambda}(\Pi_{\infty}^{R_{\infty}-\an})[\fa]\cong B_{\Omega, \lambda}(\widehat{S}(U^p,E)_{\overline{\rho}}^{\an})$. The proposition follows.
\end{proof}

For $v\in S_p$, let $\jmath_{\widetilde{v}}$ denote the isomorphism
\begin{equation*}
	\jmath_{\widetilde{v}}: (\Spec \cZ_{\Omega_{\widetilde{v}}})^{\rig} \xlongrightarrow{\sim} (\Spec \cZ_{\Omega_{\widetilde{v}}})^{\rig}
\end{equation*}
such that $\pi_{\jmath_{\widetilde{v}}(\ul{x})_i}\cong \pi_{x_i} \otimes_E \unr\big(q_{\widetilde{v}}^{s_{\widetilde{v},i-1}-\frac{1-n_{\widetilde{v},i}}{2}}\big) \circ \dett$ for $\ul{x}=(x_i) \in (\Spec \cZ_{\Omega_{\widetilde{v}}})^{\rig}$. Let $\jmath:=(\jmath_{\widetilde{v}})_{v\in S_p}: (\Spec \cZ_{\Omega})^{\rig} \ra (\Spec \cZ_{\Omega})^{\rig}$. Denote by $\cZ_{0,\widetilde{v}}:=Z_{L_{P_{\widetilde{v}}}}(\co_{F_{\widetilde{v}}})$, thus $\widehat{\cZ_0}\cong \prod_{v\in S_p} \widehat{\cZ_{0,\widetilde{v}}}$. Consider the composition:
\begin{multline}\label{RTnorm}
	\cE_{\Omega, \lambda}^{\infty}(\overline{\rho}) \hooklongrightarrow (\Spf R_{\infty}^p)^{\rig} \times \prod_{v\in S_p} \big((\Spf R_{\overline{\rho}_{\widetilde{v}}})^{\rig} \times (\Spec \cZ_{\Omega_{\widetilde{v}}})^{\rig} \times \widehat{\cZ_{0,\widetilde{v}}}\big) \\
	\xlongrightarrow{\jmath} (\Spf R_{\infty}^p)^{\rig} \times \prod_{v\in S_p}\big( (\Spf R_{\overline{\rho}_{\widetilde{v}}})^{\rig} \times (\Spec \cZ_{\Omega_{\widetilde{v}}})^{\rig} \times\widehat{\cZ_{0,\widetilde{v}}}\big)
\end{multline}
where the second map (still denoted by $\jmath$) is the identity on the factors other than $\prod_v(\Spec \cZ_{\Omega_{\widetilde{v}}})^{\rig}$ and is $\jmath$ on $\prod_v(\Spec \cZ_{\Omega_{\widetilde{v}}})^{\rig}\cong (\Spec \cZ_{\Omega})^{\rig}$. Let $\textbf{h}=(\textbf{h}_{\widetilde{v}})_{v\in S_p}$ be associated to $\lambda$ as in (\ref{textbfh0}). In the next section (see \S~\ref{s: DO} below), we construct a reduced closed subspace $X_{\Omega_{\widetilde{v}}, \textbf{h}_{\widetilde{v}}}(\overline{\rho}_{\widetilde{v}})$ of $(\Spf R_{\overline{\rho}_{\widetilde{v}}})^{\rig} \times (\Spec \cZ_{\Omega_{\widetilde{v}}})^{\rig} \times \widehat{\cZ_{0,\widetilde{v}}}$. We let $X_{\Omega, \textbf{h}}(\overline{\rho}_p):=\prod_{v\in S_p} X_{\Omega_{\widetilde{v}}, \textbf{h}_{\widetilde{v}}}(\overline{\rho}_{\widetilde{v}})$. \index{$X_{\Omega, \textbf{h}}(\overline{\rho}_p)$}

\begin{theorem}\label{R=T0}
The composition in (\ref{RTnorm}) factors through $(\Spf R_{\infty}^p)^{\rig} \times X_{\Omega, \textbf{h}}(\overline{\rho}_p)$, and induces an isomorphism between $\cE_{\Omega, \lambda}^{\infty}(\overline{\rho})$ and a union of irreducible components of $(\Spf R_{\infty}^p)^{\rig} \times X_{\Omega, \textbf{h}}(\overline{\rho}_p)$ equipped with the reduced closed rigid subspace structure. 
\end{theorem}
\begin{proof}
By Proposition \ref{galPBE}, and the construction of $\{X_{\Omega_{\widetilde{v}}, \textbf{h}_{\widetilde{v}}}(\overline{\rho}_{\widetilde{v}})\}$ in \S~\ref{s: DO}, one sees that all the generic classical points of $\cE_{\Omega, \lambda}^{\infty}(\overline{\rho})$ are sent to
\[(\Spf R_{\infty}^p)^{\rig} \times \prod_{v\in S_p} U_{\Omega, \textbf{h}}(\overline{\rho}_{\widetilde{v}}) \hookrightarrow (\Spf R_{\infty}^p)^{\rig} \times X_{\Omega, \textbf{h}}(\overline{\rho}_p)\]
where we refer to the discussion above Proposition \ref{twGal} for the rigid analytic space $U_{\Omega, \textbf{h}}(\overline{\rho}_{\widetilde{v}})$. The first part then follows from the density of generic classical points (Theorem \ref{classPatc}). The second part follows from Proposition \ref{PBEred}, from the fact $\cE_{\Omega, \lambda}^{\infty}(\overline{\rho})$ is closed in the right hand side of (\ref{RTnorm}), hence in $(\Spf R_{\infty}^p)^{\rig} \times X_{\Omega, \textbf{h}}(\overline{\rho}_p)$, and from the fact that both rigid spaces $\cE_{\Omega, \lambda}^{\infty}(\overline{\rho})$ and $(\Spf R_{\infty}^p)^{\rig} \times X_{\Omega, \textbf{h}}(\overline{\rho}_p)$ have the same dimension (by Corollary \ref{PBEdim} for the first and Proposition \ref{DFOL} (1) for $X_{\Omega_{\widetilde{v}}, \textbf{h}_{\widetilde{v}}}(\overline{\rho}_{\widetilde{v}})$ together with $\dim(\Spf R_{\infty}^p)^{\rig}=g+\sum_{v\in S\setminus S_p} \dim(\Spf R_{\overline{\rho}_{\widetilde{v}}})^{\rig}=g+n^2\vert S\setminus S_p\vert$). 
\end{proof}

\begin{remark}
Let $\fX^p$ be an irreducible component of $(\Spf R_{\infty}^p)^{\rig}$. Call an irreducible component $X_p$ of $X_{\Omega, \textbf{h}}(\overline{\rho}_p)$ $\fX^p$-automorphic if $\fX^p \times X_p$ is contained in the image of (\ref{RTnorm}). One may expect that $X_p$ is $\fX^p$-automorphic if and only if $X_p$ contains a generic potentially crystalline point with distinct Hodge-Tate weights. We refer to \cite[\S~3.6]{BHS1} for related discussions in the trianguline case.
\end{remark}

We finally discuss the problem of companion constituents and companion points in the patched setting. Let $\fm$ be a maximal ideal of $R_{\infty}[1/p]$ such that $\Pi_{\infty}[\fm]\neq 0$. For $v\in S_p$, we assume that the $\Gal_{F_{\widetilde{v}}}$-representation $\rho_{\widetilde{v}}$ associated to $\fm$ (as above Proposition \ref{galPBE}) is generic potentially crystalline with distinct Hodge-Tate weights. We use the notation in \S~\ref{seccompCP}.

\begin{conjecture}\label{conjCPP}
For $v\in S_p$, let $w=(w_{\widetilde{v}})\in \sW^{P}_{\min}=\prod_{v\in S_p} \sW^{P_{\widetilde{v}}}_{\min,F_{\widetilde{v}}}$. 
	
(1) The representation 
	$\widehat{\otimes}_{v\in S_p}C(w_{\widetilde{v}},\sF_{\widetilde{v}})$ is a subrepresentation of $\Pi_{\infty}^{R_{\infty}-\an}[\fm]$ if and only if $w_{\widetilde{v}} \leq w_{\sF_{\widetilde{v}}}w_0$ for all $v\in S_p$.
	
(2) The point $(\fm, \boxtimes_{v\in S_p} \pi_{L_{P_{\widetilde{v}}}}, 1)\in (\Spf R_{\infty})^{\rig} \times (\Spec \cZ_{\Omega})^{\rig} \times \widehat{\cZ_0}$ lies in $\cE_{\Omega, w\cdot \lambda}^{\infty}(\overline{\rho})$ if and only if $w_{\widetilde{v}} \leq w_{\sF_{\widetilde{v}}}w_0$ for all $v\in S_p$.
\end{conjecture}

\begin{remark}\label{cptPBE}
	The point $(\fm, \boxtimes_{v\in S_p} \pi_{L_{P_{\widetilde{v}}}}, 1)\in \cE_{\Omega, w\cdot \lambda}^{\infty}(\overline{\rho})$ of Conjecture \ref{conjCPP} (2) may be referred to as ``a companion point of $x=(\fm, \boxtimes_{v\in S_p} \pi_{L_{P_{\widetilde{v}}}}, 1)$ seen in $\cE_{\Omega, \lambda}^{\infty}(\overline{\rho})$''. In the case of the patched eigenvariety $X_p(\overline{\rho})$ of \cite{BHS3}, there is a canonical embedding $X_p(\overline{\rho}) \hookrightarrow (\Spf R_{\infty})^{\rig} \times \widehat{T_p}$ and the companion points are the distinct points that lie above a same point $y\in (\Spf R_{\infty})^{\rig}$. In our case however, as there are different rigid spaces depending on $(\Omega, \lambda)$, it seems more convenient to fix the point $x \in (\Spf R_{\infty})^{\rig} \times (\Spec \cZ_{\Omega})^{\rig} \times \widehat{\cZ_0}$ and let the Bernstein patched eigenvarieties (together with the embedding into $(\Spf R_{\infty})^{\rig} \times (\Spec \cZ_{\Omega})^{\rig} \times \widehat{\cZ_0}$) vary. See also Remark \ref{remNPara}.
\end{remark}

By the same argument, we have as in Lemma \ref{lemcompCP}:

\begin{lemma}\label{lemCPCC}
(1) Conjecture \ref{conjCPP} (1) implies Conjecture \ref{conjCPP} (2).
	
(2) The ``only if" part of Conjecture \ref{conjCPP} (2) implies the ``only if" part of Conjecture \ref{conjCPP} (1).
\end{lemma}

Finally, using the isomorphism $\Pi_{\infty}[\fa]\cong \widehat{S}(U^p, E)_{\overline{\rho}}$ (resp.\ using Proposition \ref{bePbe}), one easily deduces:

\begin{lemma}
Conjecture \ref{conjCPP} (1) (resp.\ Conjecture \ref{conjCPP} (2)) implies Conjecture \ref{conjSoc} (resp.\ Conjecture \ref{conjCP}).
\end{lemma}

\section{Bernstein paraboline varieties}\label{secDefva}

We now move to the Galois side. In \S~\ref{secDef1}, we study certain paraboline deformations of $(\varphi, \Gamma)$-modules which admit an $\Omega$-filtration (where $\Omega$ is a cuspidal Bernstein component as in \S~\ref{sec3.1.1}). In \S~\ref{s: DO}, we construct and study what we call Bernstein paraboline varieties (analogous to the trianguline variety of \cite[\S~2.2]{BHS1} when $P=B$). Finally in \S~\ref{secPCD}, we study the relation between Bernstein paraboline varieties and potentially crystalline deformation spaces, and show the existence of local companion points. We frequently denote a point in a Bernstein component by its associated Weil-Deligne representation.

\subsection{Deformations of $(\varphi, \Gamma)$-modules}\label{secDef1}

We prove various results on deformations of $(\varphi, \Gamma)$-modules $D$ that admit an $\Omega$-filtration (where $\Omega$ is a cuspidal Bernstein component as in \S~\ref{sec3.1.1}). We first study in \S~\ref{secD0} deformations of irreducible constituents of $D$. By combining the results in \S~\ref{secD0} with results of Chenevier on paraboline deformations, we study deformations of type $\Omega$ in \S~\ref{secDefOD} (which are special cases of paraboline deformations).

\subsubsection{Deformations of certain irreducible $(\varphi,\Gamma)$-modules}\label{secD0}

We study deformations of certain irreducible $(\varphi, \Gamma)$-modules which are de Rham up to twist. The results in this section will be used in our study of deformations of type $\Omega$ in \S~\ref{secDefOD}.

We let $L$ be a finite extension of $\Q_p$ and we use the notation in \S~\ref{Nota2.1}. We begin with some useful facts on extensions of $p$-adic differential equations. Let $\Delta$ be an irreducible $(\varphi, \Gamma)$-module of rank $k$ over $\cR_{E,L}$, de Rham of constant Hodge-Tate weight $0$.

\begin{lemma} \label{lemPdif}
(1) Let $M$ be a $(\varphi, \Gamma)$-module over $\cR_{E,L}$. Assume that $M$ admits an increasing filtration $\Fil_{\bullet} M$ by $(\varphi, \Gamma)$-submodules such that the graded pieces are all isomorphic to $\Delta$. Let $N$ be a saturated $(\varphi, \Gamma)$-submodule of $M$. Then both $N$ and $M/N$ admit a filtration by $(\varphi, \Gamma)$-submodules such that the graded pieces are all isomorphic to $\Delta$.
	
(2) Let $M_1$, $M_2$ be $(\varphi, \Gamma)$-modules over $\cR_{E,L}$ which both admit an increasing filtration by $(\varphi, \Gamma)$-submodules such that the graded pieces are all isomorphic to $\Delta$. Let $f: M_1 \ra M_2$ be a morphism of $(\varphi, \Gamma)$-modules, then $\Ima(f)$ is saturated in $M_2$.
	
(3) Let $M$ be as in (1), and $N_1$, $N_2$ be two saturated $(\varphi, \Gamma)$-submodules of $M$. Then $N_1+N_2$ is also saturated in $M$.
\end{lemma}
\begin{proof}
(1) We endow $N$ with the induced filtration from $M$ and $M/N$ with the quotient filtration, and want to prove that all graded pieces are $\Delta$. By induction on the rank of $M$, we only need to show the statement in the case $N$ is irreducible. By d\'evissage on $\Fil_{\bullet} M$, we have $\Hom_{(\varphi, \Gamma)}(N,\Delta)\neq 0$. Hence we have an injection $N \hookrightarrow \Delta$, and so $N$ is de Rham. By considering the Sen weights and using the fact $N$ is saturated, we see that $N$ is of constant Hodge-Tate weights $0$. We deduce that $N$ is actually isomorphic to $\Delta$. Let $i\in \Z$ such that the injection $j: N\hookrightarrow M$ has image contained in $\Fil_i M$ but not in $\Fil_{i-1} M$. The non-zero composition $N \xrightarrow{j} \Fil_i M \twoheadrightarrow \gr_i M\cong \Delta$ has to be an isomorphism, and gives a splitting $\Fil_i M\cong N \oplus \Fil_{i-1} M$. The filtration $\Fil_{\bullet}$ then induces a filtration on $M/N$ such that all the graded pieces are isomorphic to $\Delta$.
	
(2) Using (1), we are reduced to the case where $f$ is injective. Then by induction on the rank of $M_2$, we are reduced to the case where $M_1\cong \Delta$. But in this case, by the argument at the end of (1), there exists $i$ such that $\Fil_i M_2\cong \Ima(f) \oplus \Fil_{i-1} M_2$. In particular $\Ima(f)$ is saturated.
	
(3) We have that $N_1\cap N_2$ is saturated in $M$. By (1) and (2), we see that the image of the composition 
	\begin{equation*}
		(N_1+N_2)/(N_1\cap N_2) \cong N_1/(N_1 \cap N_2) \oplus N_2/(N_1 \cap N_2) \hookrightarrow N/(N_1 \cap N_2)
	\end{equation*}
	is saturated. (3) follows.
\end{proof}

We let $\Art(E)$ denote the category of local artinian $E$-algebras of residue field $E$.\index{$\Art(E)$} Let $D$ be a $(\varphi, \Gamma)$-module of rank $k$ over $\cR_{E,L}$. Denote by $F_{D}$ the functor of deformations of $D$ on $A\in \Art(E)$.\index{$\Art(E)$} Suppose that there exist a continuous character $\delta: L^{\times} \ra E^{\times}$ such that one has an embedding of $(\varphi, \Gamma)$-modules
\begin{equation*}
	D \otimes_{\cR_{E,L}} \cR_{E,L}(\delta^{-1})\hooklongrightarrow \Delta
\end{equation*}
(which implies that $D$ is irreducible).\index{$F_D^0$} We consider the following functor
\[F_D^0: \Art(E) \lra \{\text{sets}\},\ \ A \longmapsto \{(D_A, \pi_A, \delta_A)\}/\sim\]
where $D_A$ is a $(\varphi, \Gamma)$-module over $\cR_{A,L}$, $\pi_{A,1}: D_A \otimes_{A} E \xrightarrow{\sim} D$, $\delta_A: \co_L^{\times} \ra A^{\times}$ such that $\delta_A \equiv \delta \pmod{\fm_A}$ ($\fm_A$ is the maximal ideal of $A$), and there is an injection of $(\varphi, \Gamma)$-module over $\cR_{A,L}$:
\begin{equation}\label{injpDEA}
	D_A \hooklongrightarrow \Delta \otimes_{\cR_{E,L}} \cR_{A,L}(\delta_A).
\end{equation}
For $\tau\in \Sigma_L$, let $h_{1,\tau}$ be the maximal $\tau$-Hodge-Tate weight of $D \otimes_{\cR_{E,L}} \cR_{E,L}(\delta^{-1})$ (thus $h_{1,\tau}\geq 0$), and put $\textbf{h}_1:=(h_{1,\tau})_{\tau\in \Sigma_L}$. By comparing the Hodge-Tate weights of $\Delta \otimes_{\cR_{E,L}} \cR_{A,L}$ and $D_A \otimes_{\cR_{A,L}} \cR_{A,L}(\delta_A^{-1})$, and using \cite[Thm.~A]{Ber08a}, we see that the existence of the injection (\ref{injpDEA}) is equivalent to the existence of an injection 	
\begin{equation}\label{injfrDF}
	\Delta \otimes_{\cR_{E,L}} \cR_{A,L}(z^{\textbf{h}_1}\delta_A) \hooklongrightarrow D_A.
\end{equation}
Indeed, both are equivalent to the existence of an isomorphism of $(\varphi, \Gamma)$-modules over $\cR_{A,L}[1/t]$: $D_A[1/t]\cong \Delta \otimes_{\cR_{E,L}} \cR_{A,L}(\delta_A)[1/t]$. 
\begin{lemma}\label{subf}
	$F_D^0$ is a subfunctor of $F_D$.
\end{lemma}
\begin{proof}
	Let $A\in \Art(E)$, and $(D_A, \pi_A,\delta_A)\in F_D^0(A)$. It is enough to show that $\delta_A$ is uniquely determined by $D_A$. Suppose we have another $\delta_A'$ such that $\delta_A'\equiv \delta_A \pmod{\fm_A}$ and $ D_A \hookrightarrow \Delta \otimes_{\cR_{E,L}} \cR_{A,L}(\delta_A')$. This map, together with (\ref{injfrDF}), induce
	\begin{equation*}
		\Delta \otimes_{\cR_{E,L}} \cR_{A,L}(z^{\textbf{h}_1}\delta_A) \hooklongrightarrow D_A \hooklongrightarrow \Delta \otimes_{\cR_{E,L}} \cR_{A,L}(\delta_A').
	\end{equation*}
	Hence we have $\Delta \otimes_{\cR_{E,L}} \cR_{A,L} \hookrightarrow \Delta \otimes_{\cR_{E,L}} \cR_{A,L}(\delta_A'\delta_A^{-1} z^{-\textbf{h}_1})$. Since $\Delta \otimes_{\cR_{E,L}} \cR_{A,L} $ is de Rham, so is $\Delta \otimes_{\cR_{E,L}} \cR_{A,L}(\delta_A'\delta_A^{-1} z^{-\textbf{h}_1})$ (using that both have the same rank over $\cR_{E,L}$). By looking at the Sen weights and using $\delta_A'\equiv \delta_A \pmod{\fm_A}$, we deduce that $\delta_A'\delta_A^{-1}$ is smooth. Then by comparing the Hodge-Tate weights (and using \cite[Thm.~A]{Ber08a}), we obtain an isomorphism
	\begin{equation}\label{isoAA'}
		\Delta \otimes_{\cR_{E,L}} \cR_{A,L} \cong \Delta \otimes_{\cR_{E,L}} \cR_{A,L}(\delta_A'\delta_A^{-1}).
	\end{equation}
	Let $(\Delta\otimes_{\cR_E,L} \Delta^{\vee})^0:=(\Delta \otimes_{\cR_E,L} \Delta^{\vee})/\cR_{E,L}$, then $\Delta \otimes_{\cR_E,L} \Delta^{\vee}\cong (\Delta\otimes_{\cR_E,L} \Delta^{\vee})^0 \oplus \cR_{E,L}$ and $H^0_{(\varphi, \Gamma)}((\Delta\otimes_{\cR_E,L} \Delta^{\vee})^0)=0$. We have isomorphisms
\begin{multline}\label{isoCpG0}
H^0_{(\varphi,\Gamma)}\big(\Delta^{\vee} \otimes_{\cR_{E,L}} \Delta \otimes_{\cR_{E,L}} \cR_{A,L}(\delta_A'\delta_A^{-1})\big)\\
\cong H^0_{(\varphi,\Gamma)}(\cR_{A,L}(\delta_A'\delta_A^{-1})) \oplus H^0_{(\varphi,\Gamma)}\big((\Delta^{\vee} \otimes_{\cR_{E,L}} \Delta)^0 \otimes_{\cR_{E,L}} \cR_{A,L}(\delta_A'\delta_A^{-1})\big)\\
\cong H^0_{(\varphi,\Gamma)}(\cR_{A,L}(\delta_A'\delta_A^{-1}))
\end{multline}
	where the second isomorphism follows from $H^0_{(\varphi,\Gamma)}((\Delta^{\vee} \otimes_{\cR_{E,L}} \Delta)^0)=0$ and an easy d\'evissage using $\delta_A'\delta_A^{-1} \equiv 1\pmod{\fm_A}$. From (\ref{isoAA'}) and (\ref{isoCpG0}), we deduce an embedding $A \hookrightarrow H^0_{(\varphi,\Gamma)}(\cR_{A,L}(\delta_A'\delta_A^{-1}))$, hence an injection $\cR_{A,L} \hookrightarrow \cR_{A,L}(\delta_A'\delta_A^{-1})$ that has to be an isomorphism by comparing the Hodge-Tate weights (recall $\delta_A'\delta_A^{-1}$ is smooth). By \cite[Prop.\ 2.3.1]{BCh}, we obtain $\delta_A'\delta_{A}^{-1}=1$, which concludes the proof.
\end{proof}

\begin{proposition}\label{repab0}
	The functor $F_D^0$ is relatively representable over $F_D$.
\end{proposition}
\begin{proof}
As in \cite[Prop.\ 2.3.9]{BCh}, the proposition follows from the following three properties that we will prove. 
	
(1) If $A \ra A'$ is a morphism in $\Art(E)$ and $(D_A, \pi_A,\delta_A)\in F_D^0(A)$, then $(D_A \otimes_A A', \pi_A \otimes_A A') \in F_D^0(A')$. 
	
(2) Let $A\hookrightarrow A'$ be an injection in $\Art(E)$, $(D_A, \pi_A)\in F_D(A)$, and assume $(D_{A} \otimes_A A',\pi_A\otimes_A A')\in F_D^0(A')\hookrightarrow F_D(A')$ (Lemma \ref{subf}), then $(D_A, \pi_A) \in F_D^0(A)$.
	
(3) For $A$ and $A'$ in $\Art(E)$, if $(D_A, \pi_A, \delta_A)\in F_{D}^0(A)$ and $(D_{A'}, \pi_{A'}, \delta_{A'}) \in F_{D}^0(A')$, then for $B=A\times_E A'$ we have $(D_B:=D_A \times_D D_{A'}, \pi_B:=\pi_A\circ \pr = \pi_{A'} \circ \pr')\in F_D^0(B)$ where $\pr:B\twoheadrightarrow A$, $\pr':B\twoheadrightarrow A'$.
	
The properties (1) and (3) are straightforward to verify. We prove (2). By (the proof of) Lemma \ref{subf}, there is a unique continuous character $\delta_{A'}: L^{\times} \ra (A')^{\times}$ such that $D_{A} \otimes_A A' \hookrightarrow \Delta \otimes_{\cR_{E,L}} \cR_{A',L}(\delta_{A'})$ and $\delta_{A'}\equiv \delta \pmod{\fm_A}$. Let $M$ be the saturated closure of $D_A$ in $\Delta \otimes_{\cR_{E,L}} \cR_{A',L}(\delta_{A'})$ (see \cite[\S~2.2.3]{BCh}). Since $\Delta \otimes_{\cR_{E,L}} \cR_{A',L}(\delta_{A'})$ admits a filtration with graded pieces all isomorphic to $\Delta\otimes_{\cR_{E,L}} \cR_{E,L}(\delta)$, so does $M$ by Lemma \ref{lemPdif} (1) (twisting by $\cR_{E,L}(\delta)$). For $x\in \fm_A$, consider the morphism $x: M \ra M$ given by multiplying by $x$. By Lemma \ref{lemPdif} (2), we know $xM$ is saturated in $M$ and hence by induction $\fm_A M$ is saturated in $M$. We deduce $M\otimes_A E\cong M/\fm_A M$ is a $(\varphi, \Gamma)$-module over $\cR_{E,L}$ (in particular is free of finite type over $\cR_{E,L}$). Using the isomorphism $D_A[1/t] \cong M[1/t]$, and $D_A[1/t]\otimes_A E\cong D[1/t]$, we see $M \otimes_A E$ is of rank $k$ over $\cR_{E,L}$. From Lemma \ref{lemPdif} (1), we deduce $M \otimes_A E\cong \Delta \otimes_{\cR_{E,L}} \cR_{E,L}(\delta)$.
	
Consider the following $(\varphi, \Gamma)$-module over $\cR_{E,L}$:
\[\begin{array}{rll}
Q&:=&\Delta \otimes_{\cR_{E,L}} \Delta^{\vee}\cong \cR_{E,L} \oplus (\Delta\otimes_{\cR_{E,L}} \Delta^{\vee})^0\\
Q_A&:=& \big(\Delta^{\vee} \otimes_{\cR_{E,L}} \cR_{E,L}(\delta^{-1})\big) \otimes_{\cR_{E,L}} M\\
Q_{A'}&:=&\big(\Delta^{\vee}\otimes_{\cR_{E,L}} \cR_{E,L}(\delta^{-1})\big) \otimes_{\cR_{E,L}} \big(\Delta \otimes_{\cR_{E,L}} \cR_{A',L}(\delta_{A'})\big)\\
&\cong& \cR_{A',L}(\delta_{A'}\delta^{-1}) \oplus \big((\Delta \otimes_{\cR_{E,L}} \Delta^{\vee})^0 \otimes_{\cR_{E,L}} \cR_{A',L}(\delta_{A'}\delta^{-1})\big).
\end{array}\]
We have $Q_A\hookrightarrow Q_{A'}$ and both $Q_{A'}$, $Q_A$ are isomorphic to a successive extension of $Q$. We apply the functor $F$ defined right above \cite[Lemma 2.3.8]{BCh}. By \cite[Lemma 2.3.8]{BCh} and $\Hom_{(\varphi, \Gamma)}(\cR_{E,L}, (\Delta \otimes_{\cR_{E,L}} \Delta^{\vee})^0)=0$, we see $F(Q)\cong E$ and 
	\begin{equation}\label{FQA'}
		F(Q_{A'} ) \cong F(\cR_{A',L}(\delta_{A'}\delta^{-1}))\cong A'.
	\end{equation}
By the left exactness of the functor $F$ and an obvious d\'evissage, we deduce
	\begin{equation}\label{dimQA}
	\dim_E F(Q_A)\leq \dim_E A.
	\end{equation}
	Consider the exact sequence
	\begin{equation*}
		0 \ra F(Q_A) \ra F(Q_{A'}) \ra F(Q_{A'}/Q_A).
	\end{equation*}
We have
\[Q_{A'}/Q_A\cong \big(\Delta^{\vee} \otimes_{\cR_{E,L}} \cR_{E,L}(\delta^{-1})\big)\otimes_{\cR_{E,L}} \Big(\big(\Delta \otimes_{\cR_{E,L}} \cR_{A',L}(\delta_{A'})\big)/M\Big)\]
which, by Lemma \ref{lemPdif} (1) applied to $\big(\Delta \otimes_{\cR_{E,L}} \cR_{A',L}(\delta_{A'})\big)/M$, is also isomorphic to a successive extension of $Q$. By d\'evissage, we deduce $\dim_E F(Q_{A'}/Q_A)\leq \dim_E (A'/A)$ over $E$. This, together with (\ref{FQA'}), (\ref{dimQA}) and an easy dimension counting, imply (\ref{dimQA}) is in fact an equality. Consider now
	\begin{equation}\label{A-mod}
		0 \ra F( \fm_A Q_A) \ra F(Q_A) \ra F(Q).
	\end{equation}
By d\'evissage, we have again $\dim_E F(\fm_A Q_A)\leq \dim_E\fm_A$. Using $\dim_E F(Q_A)=\dim_E A$, we deduce that the right morphism is surjective (and $\dim_E F(\fm_A Q_A)= \dim_E\fm_A$). Noting that (\ref{A-mod}) is a sequence of $A$-modules (with $\fm_A$ acting by $0$ on $F(Q)$) and considering the $A$-submodule of $F(Q_A)$ generated by a lifting of a generator of $F(Q)\cong E$, we easily deduce $A\buildrel\sim\over\ra F(Q_A)$. 
	
	Consider the $(\varphi,\Gamma)$-submodule $Q_A^0$ of $Q_A$ generated by $F(Q_A)$. We claim it is a rank one $(\varphi, \Gamma)$-module over $\cR_{A,L}$. Let $Q_{A'}^0$ be the $(\varphi,\Gamma)$-submodule of $Q_{A'}$ generated by $F(Q_{A'})$. Since we have $F(\cR_{A',L}(\delta_{A'}\delta^{-1}))\xrightarrow{\sim} F(Q_{A'})$, we see $Q_{A'}^0$ is also the $(\varphi,\Gamma)$-submodule of $\cR_{A',L}(\delta_{A'}\delta^{-1})$ generated by $F(Q_{A'})$. Since $\cR_{A',L}(\delta_{A'}\delta^{-1})$ has a filtration with all graded pieces isomorphic to $\cR_{E,L}$, by d\'evissage and \cite[Lemma 2.3.8 (ii)]{BCh} any strict (saturated) $(\varphi, \Gamma)$-submodule $C$ of $\cR_{A',L}(\delta_{A'}\delta^{-1})$ is such that $\dim_E F(C)<\dim_E A'$. As $\dim_E F(Q_{A'}^0)=\dim_E F(Q_{A'})=\dim_E A'$ by (\ref{FQA'}), we deduce $Q_{A'}^0\xrightarrow{\sim} \cR_{A',L}(\delta_{A'}\delta^{-1})$. We also see that the natural morphism $\cR_{A',L}\otimes_{A'} F(Q_{A'}) \ra Q_{A'}^0$ is an isomorphism. Now consider
	\begin{equation*}
		\cR_{A,L} \otimes_A F(Q_A) \hookrightarrow \cR_{A,L} \otimes_A F(Q_{A'}) \cong \cR_{A',L} \otimes_{A'} F(Q_{A'}) \xrightarrow{\sim} Q_{A'}^0 \hookrightarrow Q_{A'}.
	\end{equation*}
	The composition is injective and factors through $Q_A$. We deduce then $\cR_{A,L} \otimes_A F(Q_A) \xrightarrow{\sim} Q_A^0$ and hence the latter is a $(\varphi, \Gamma)$-module of rank $1$ over $\cR_{A,L}$ as $A\cong F(Q_A)$.
	
Let $\varepsilon_A: L^{\times} \ra A^{\times}$ be the continuous character such that $Q_A^0 \cong \cR_{A,L}(\varepsilon_A)$. By \cite[Lemma 2.3.8 (1)]{BCh} and the fact $Q_A^0$ is generated by $F(Q_A^0)$, it is not difficult to see $\varepsilon_A\equiv 1 \pmod{\fm_A}$. Twisting by $\Delta$ on both sides, the injection $\cR_{A,L}(\varepsilon_A)\hookrightarrow Q_A$ induces a morphism
	\begin{equation*}
		\iota: \Delta\otimes_{\cR_{E,L}} \cR_{A,L}(\varepsilon_A) \lra M \otimes_{\cR_{E,L}} \cR_{E,L}(\delta^{-1}).
	\end{equation*}
	We prove that $\iota$ is an isomorphism. It is sufficient to show it is surjective since both source and target have the same rank over $\cR_{E,L}$. One easily checks that the morphism $\cR_{A,L}(\varepsilon_A)\hookrightarrow Q_A$ factors as 
	\begin{equation*}
	\cR_{A,L}(\varepsilon_A) \hookrightarrow \Delta^{\vee} \otimes_{\cR_{E,L}}\Delta\otimes_{\cR_{E,L}} \cR_{A,L}(\varepsilon_A) \ra \Delta^{\vee} \otimes_{\cR_{E,L}} M \otimes_{\cR_{E,L}} \cR_{E,L}(\delta^{-1})\cong Q_A
	\end{equation*}
	where the second map is induced by $\iota$ tensored with $\Delta^{\vee}$. In particular, $\cR_{A,L}(\varepsilon_A)\hookrightarrow Q_A$ factors through $\cR_{A,L}(\varepsilon_A) \hookrightarrow \Delta^{\vee}\otimes_{\cR_{E,L}}\Ima(\iota)$. Since $\Ima(\iota)\cong (\Delta\otimes_{\cR_{E,L}} \cR_{A,L}(\varepsilon_A))/\Ker(\iota)$, by Lemma \ref{lemPdif} (1) it admits a filtration with graded pieces all isomorphic to $\Delta$. If $\iota$ is not surjective, the multiplicity of $\Delta$ in this filtration on $\Ima(\iota)$ is strictly smaller than the multiplicity of $\Delta$ in the filtration of $M \otimes_{\cR_{E,L}} \cR_{E,L}(\delta^{-1})$, which is $\dim_E A$ using the above equality of ranks over $\cR_{E,L}$. Applying the functor $F$ and using again a d\'evissage, we have in that case $\dim_E F( \Ima(\iota) \otimes_{\cR_{E,L}} \Delta^{\vee})<\dim_E A=\dim_E F(\cR_{A,L}(\varepsilon_A))$, which contradicts $\cR_{A,L}(\varepsilon_A) \hookrightarrow \Ima(\iota) \otimes_{\cR_{E,L}} \Delta^{\vee}$. We deduce thus $M\cong \Delta\otimes_{\cR_{E,L}} \cR_{A,L}(\delta\varepsilon_A)$. Since $D_A \hookrightarrow M\cong \Delta\otimes_{\cR_{E,L}} \cR_{A,L}(\delta\varepsilon_A)$, we have $D_A\in F_D^0(A)$. This finishes the proof.
\end{proof}

Now suppose moreover $D$ has distinct Sen weights, hence $D \otimes_{\cR_{E,L}} \cR_{E,L}(\delta^{-1})$ has distinct Hodge-Tate weights. Twisting $\delta$ by some algebraic character of $L^{\times}$, we can and do assume that the Hodge-Tate weights of $D \otimes_{\cR_{E,L}} \cR_{E,L}(\delta^{-1})$ are given by $\textbf{h}=(h_{1,\tau} > h_{2,\tau} \cdots> h_{k,\tau}=0)_{\tau \in \Sigma_L}$.

\begin{proposition}\label{apxfm0}
The functor $F_D^0$ is formally smooth of dimension $1+[L:\Q_p] (1+\frac{k(k-1)}{2})$.
\end{proposition}
\begin{proof}
Let $A \twoheadrightarrow A/I$ be as surjection in $\Art(E)$ with $I^2=0$. We show the natural map $F_D^0(A) \ra F_D^0(A/I)$ is surjective. Let $(D_{A/I}, \pi_{A/I}, \delta_{A/I})\in F_D^0(A/I)$. Let $\delta_A: L^{\times} \ra A^{\times}$ be a continuous character such that $\delta_{A} \equiv \delta_{A/I} \pmod{I}$. We have by definition an embedding
\[D_{A/I} \otimes_{\cR_{A/I,L}} \cR_{A/I,L}(\delta_{A/I}^{-1}) \hookrightarrow \Delta \otimes_{\cR_{E,L}} \cR_{A/I,L}.\]
We choose a basis $\ul{e}$ of $D_{\dR}(\Delta\otimes_{\cR_{E,L}} \cR_{A,L})$ over $L \otimes_{\Q_p} A$ (note that the latter is a free $L \otimes_{\Q_p} A$-module), and denote by $\ul{e}_{A/I}$ the image of $\ul{e}$ in $D_{\dR}(\Delta\otimes_{\cR_{E,L}} \cR_{A/I,L})$. In the basis $\ul{e}_{A/I}$, the Hodge filtration on $D_{\dR}\big(D_{A/I} \otimes_{\cR_{A/I,L}} \cR_{A/I,L}(\delta_{A/I}^{-1})\big)\cong D_{\dR}(\Delta\otimes_{\cR_{E,L}} \cR_{A/I,L})$ induces an increasing filtration by free $L \otimes_{\Q_p} A/I$-submodules:
	\begin{multline}\label{filHAI}
		0\neq	\Fil_{-\textbf{h}_k} D_{\dR}\big(\Delta \otimes_{\cR_{E,L}} \cR_{A/I,L}\big) \subsetneq \\
		\cdots \subsetneq	\Fil_{-\textbf{h}_1} D_{\dR}\big(\Delta \otimes_{\cR_{E,L}} \cR_{A/I,L}\big)= D_{\dR}\big(\Delta \otimes_{\cR_{E,L}} \cR_{A/I,L}\big)
	\end{multline}
which then corresponds to an element $\nu_{A/I}\in (\Res^L_{\Q_p} (\GL_k/B))(A/I)$. Since the flag variety is smooth (hence formally smooth), we can choose a lifting $\nu_A\in (\Res^L_{\Q_p} (\GL_k/B))(A)$ of $\nu_{A/I}$. Then $\nu_A$ gives an increasing filtration by free $L \otimes_{\Q_p} A$-submodules in $D_{\dR}(\Delta\otimes_{\cR_{E,L}} \cR_{A,L})$, to which we associate the Hodge filtration (still denoted by $\nu_A$) on $D_{\dR}(\Delta\otimes_{\cR_{E,L}} \cR_{A,L})$ defined by (\ref{filHAI}) with $A/I$ replaced by $A$. By \cite[Thm.~A]{Ber08a}, the filtered Deligne-Fontaine module $(D_{\pst}(\Delta \otimes_{\cR_{E,L}} \cR_{A,L}), \nu_A)$ corresponds to a $(\varphi,\Gamma)$-submodule $M_A$ of $\Delta\otimes_{\cR_{E,L}} \cR_{A,L}$. Then we see that $D_A:= M_A \otimes_{\cR_{A,L}}\cR_{A,L}(\delta_A)$ satisfies $D_A \equiv D_{A/I} \pmod{I}$ and $D_A \hookrightarrow \Delta \otimes_{\cR_{E,L}} \cR_{A,L}(\delta_A)$. Hence $F_D^0$ is formally smooth.
	
We next compute the dimension of the $E$-vector space $F_D^0(E[\varepsilon/\varepsilon^2])$. Recall that $F_D^0(E[\varepsilon/\varepsilon^2])\!\hookrightarrow F_D(E[\varepsilon/\varepsilon^2])$ and that $F_D(E[\varepsilon/\varepsilon^2])$ is identified with $\Ext^1_{(\varphi,\Gamma)}(D,D) \cong \Ext^1_{(\varphi, \Gamma)}(D_0, D_0)$, where we put $D_0:=D \otimes_{\cR_{E,L}} \cR_{E,L}(\delta^{-1})\hookrightarrow \Delta$. Consider the following morphisms
	\begin{equation}\label{Ext1}
		\Ext^1_{(\varphi, \Gamma)}(D_0, D_0) \lra \Ext^1_{(\varphi, \Gamma)}(D_0, \Delta),
	\end{equation} 
	\begin{equation}\label{Exts0}
		\Ext^1_{(\varphi,\Gamma)}(\Delta,\Delta) \lra \Ext^1_{(\varphi, \Gamma)}(D_0, \Delta). 
	\end{equation}
For a $(\varphi, \Gamma)$-module $D'$ over $\cR_{E,L}$, denote by $W_{\dR^+}(D')$ the $B_{\dR}^+\otimes_{\Q_p} E$-representation of $\Gal_L$ associated to $D'$ (see for example \cite[Prop.\ 2.2.6(2)]{Ber08II}). For $\tau \in \Sigma_L$, let $B_{\dR,\tau,E}^+ :=B_{\dR}^+ \otimes_{L, \tau} E$.	We have 
	\begin{equation*}
W_{\dR}^+(\Delta \otimes_{\cR_{E,L}} D_0^{\vee})/W_{\dR}^+(\Delta \otimes_{\cR_{E,L}} \Delta ^{\vee})\cong \bigoplus_{\tau\in \Sigma_L} \bigoplus_{i=1}^k (t^{-h_{i,\tau}}B_{\dR,\tau,E}^+/B_{\dR, \tau,E}^+)^{\oplus k}.
\end{equation*}
	Using \cite[Lemma 5.1.1]{BD2} (which easily generalizes to finite extensions of $\Q_p$), we get
	\begin{equation*}
		H^0_{(\varphi,\Gamma)}\Big(\big(\Delta \otimes_{\cR_{E,L}} D_0^{\vee}\big)/\big(\Delta \otimes_{\cR_{E,L}} \Delta^{\vee}\big)\Big) =0
	\end{equation*}
and we deduce that the morphism (\ref{Exts0}) is injective. Consider then 
	\begin{equation*}
		\Hom(L^{\times},E) \hooklongrightarrow \Ext^1_{(\varphi,\Gamma)}(\Delta,\Delta) \hooklongrightarrow \Ext^1_{(\varphi, \Gamma)}(D_0, \Delta)
	\end{equation*}
where the first map sends $\psi\in \Hom(L^{\times},E)$ to $\Delta \otimes_{\cR_{E,L}}\cR_{E[\varepsilon]/\varepsilon^2}(1+\psi \varepsilon)$. We denote by $V$ the image of the composition and by $[D_{\psi}]$ the element in $V$ associated to $\psi$. Then it is not difficult to see that $[D\otimes_{\cR_{E,L}} \cR_{E[\varepsilon]/\varepsilon^2}(\delta^{-1} (1+\psi \varepsilon)) ]$ is sent to $[D_{\psi}]$ via (\ref{Ext1}) (up to non-zero scalars). This implies $V$ is contained in the image of (\ref{Ext1}). By definition, $[D_{E[\varepsilon]/\varepsilon^2}] \in F_D^0(E[\varepsilon]/\varepsilon^2)$ if and only if $[D_{E[\varepsilon]/\varepsilon^2} \otimes_{\cR_{E,L}} \cR_{E,L}(\delta^{-1})]$ lies in the preimage of $V$ via (\ref{Ext1}). We compute the dimension of the kernel of (\ref{Ext1}). We have 
	\begin{equation*}
		W_{\dR}^+(\Delta \otimes_{\cR_{E,L}} D_0^{\vee})/W_{\dR}^+(D_0 \otimes_{\cR_{E,L}} D_0^{\vee})\cong \bigoplus_{\tau\in \Sigma_L} \bigoplus_{i=1}^k \bigoplus_{j=1}^k t^{-h_{i,\tau}}B_{\dR,\tau,E}^+/t^{h_{j,\tau}-h_{i,\tau}}.
	\end{equation*}
 By \cite[Lemma 5.1.1]{BD2}, we deduce then 
	\begin{multline*}
		\dim_E H^0_{(\varphi, \Gamma)}\big((\Delta \otimes_{\cR_{E,L}} D_0^{\vee})/(D_0 \otimes_{\cR_{E,L}} D_0^{\vee})\big)\\
		=\dim_E H^0\big(\Gal_L, W_{\dR}^+(\Delta \otimes_{\cR_{E,L}} D_0^{\vee})/W_{\dR}^+(D_0 \otimes_{\cR_{E,L}} D_0^{\vee})\big)\\
		=\sum_{\tau\in \Sigma_L}\sum_{h_{j,\tau}>h_{i,\tau}}1=\sum_{\tau\in \Sigma_L}\frac{k(k-1)}{2}=\frac{k(k-1)}{2}[L:\Q_p].
	\end{multline*}
Using $\Ext^0_{(\varphi, \Gamma)}(D_0, D_0) \cong \Ext^0_{(\varphi, \Gamma)}(D_0, \Delta)$, we see that the kernel of (\ref{Ext1}) is isomorphic to the above vector space and hence has dimension $\frac{k(k-1)}{2}[L:\Q_p]$ over $E$. As $\dim_E V=[L:\Q_p]+1$, the proposition follows.
\end{proof}

At last, we discuss some relations between $F_{D}^0$ and de Rham deformations. Denote by $F_D^{\tw,\dR}$ the functor $\Art(E) \ra \{\text{sets}\}$ sending $A$ to the isomorphism class of $(D_A, \pi_A, \chi_A)$ where $(D_A, \pi_A)\in F_D(A)$, $\chi_A: \co_L^{\times} \ra A^{\times}$ such that $\chi_A \equiv 1 \pmod{\fm_A}$ and $D_A\otimes_{\cR_{E,L}} \cR_{A,L}(\delta^{-1} \chi_{A,\varpi_L})$ is de Rham. As $D$ is irreducible and de Rham, $D\cong D_{\rig}(V)\otimes_{\cR_{E,L}} \cR_{E,L}(\psi)$ for a certain de Rham $\Gal_L$-representation $V$ and a smooth character $\psi$ of $L^{\times}$. Let $R_{D}^{\dR}$ be the universal deformation ring of de Rham deformations of $D$ on $\Art(E)$, and $R_V^{\dR}$ the universal deformation ring of de Rham deformations of $V$ on $\Art(E)$. Recall that $R_V^{\dR}$ is isomorphic to the completion at $V$ of the generic fibre of the universal potentially semi-stable deformation ring of the modulo $p$ reduction of (a lattice in) $V$, cf.\ \cite[\S~2.3]{Kis09}. The functor $D_{\rig}(-) \otimes_{\cR_{E,L}} \cR_{E,L}(\psi)$ induces then an isomorphism $R_V^{\dR}\xrightarrow{\sim} R_D^{\dR}$. Let $(\widehat{\co_L^{\times}})_1$ be the completion of $\widehat{\co_{L}^{\times}}$ at the trivial character. One directly checks that the $E$-formal scheme $R_D^{\dR} \widehat{\otimes}_E (\widehat{\co_L^{\times}})_1$ (pro-)represents the functor $F_D^{\tw, \dR}$. In particular, using \cite[Thm.\ 3.3.8]{Kis08}, we see that $F_D^{\tw, \dR}$ is formally smooth of dimension $1+[L:\Q_p](1+\frac{k(k-1)}{2})$.

When $(D_A,\pi_A,\delta_A)\in F_D^0(A)$, we have that $D_A \otimes_{\cR_{E,L}} \cR_{A,L}(\delta_A^{-1})$ is de Rham, and thus $(D_A, \pi_A,\chi_A)\in F_D^{\tw,\dR}(A)$ where $\chi_A:=(\delta_A^{-1}\delta)|_{\co_L^{\times}}$.

\begin{lemma}\label{twDR}
The morphism $F_{D}^0 \ra F_D^{\tw,\dR}$, $(D_A, \pi_A, \delta_A) \mapsto (D_A, \pi_A, \chi_A)$ is an isomorphism.
\end{lemma}
\begin{proof}
As both functors are formally smooth of dimension $1+[L:\Q_p](1+\frac{k(k-1)}{2})$, we only need to show $F_D^0(E[\varepsilon]/\varepsilon^2) \ra F_D^{\tw, \dR}(E[\varepsilon]/\varepsilon^2)$ is injective. But this is clear. 
\end{proof}

\subsubsection{Deformations of type $\Omega$}\label{secDefOD}

In this section, we study the universal deformation functor for certain paraboline deformations of $(\varphi, \Gamma)$-modules which admit an $\Omega$-filtration (where $\Omega$ is a cuspidal Bernstein component as in \S~\ref{sec3.1.1}). In particular, we show that, under a genericness assumption, this functor is pro-representable and formally smooth. 

Let $r\in \Z_{\geq 1}$. For $1\leq i \leq r$, let $n_i\in \Z_{\geq 1}$ with $\sum_{i=1}^r n_i = n$. For $1 \leq i \leq r$, let $\Omega_i$ be a cuspidal type for $\GL_{n_i}(L)$ and $\cZ_{\Omega_i}$ the associated Bernstein centre over $E$. Recall that for each $E$-point $x_i$ of $\Spec \cZ_{\Omega_i}$, we have a smooth irreducible cuspidal representation $\pi_{x_i}$ of $\GL_{n_i}(L)$ over $E$, an $F$-semi-simple Weil-Deligne representation $\ttr_{x_i}:=\rec(\pi_{x_i})$, and a $(\varphi, \Gamma)$-module $\Delta_{x_i}$ of rank $n_i$ over $\cR_{E,L}$, de Rham of constant Hodge-Tate weight $0$ (see \S~\ref{Nota2.1}). Let $P\subseteq \GL_n$ be the parabolic subgroup as in (\ref{paraP}). We let $\Omega:=(\Omega_i)_{i=1, \dots, r}$ and $\cZ_{\Omega}:=\otimes_{i=1}^r \cZ_{\Omega_i}$. We let $\cZ_{0,L}:=Z_{L_P}(\co_L)$ (to be consistent with the notation in \S~\ref{secBE}).

In this paragraph we fix a $(\varphi,\Gamma)$-module $D$ of rank $n$ over $\cR_{E,L}$.

\begin{definition}\label{defOF}
(1) We say that $D$ admits an $\Omega$-filtration $\sF$ if $D$ admits an increasing filtration by $(\varphi, \Gamma)$-submodules $0=\Fil_0 D \subsetneq \Fil_1 D \subsetneq \cdots \subsetneq \Fil_r D=D$ such that, for $i=1,\dots, r$:
\begin{itemize}
	\item $\gr_i D$ is a $(\varphi, \Gamma)$-module of rank $n_i$;
	\item there exist an $E$-point $x_i\in \Spec \cZ_{\Omega_i}$ and a continuous character $\delta_i: L^{\times} \ra E^{\times}$ such that one has an embedding $\gr_i D \otimes_{\cR_{E,L}} \cR_{E,L}(\delta_i^{-1}) \hookrightarrow \Delta_{x_i}$.
\end{itemize}

(2) Let $\sF$, $\ul{x}=(x_i)$, $\delta=\boxtimes_{i=1}^r \delta_i$ be as in (1), we call the corresponding point $(\ul{x}, \delta)$ in $(\Spec \cZ_{\Omega})^{\rig} \times \widehat{Z_{L_P}(L)}$ a parameter of the $\Omega$-filtration $\sF$ if, for each $\tau \in \Sigma_L$, $0$ is a $\tau$-Hodge-Tate weight (hence is the minimal $\tau$-Hodge-Tate weight) of $\gr_i D \otimes_{\cR_{E,L}} \cR_{E,L}(\delta_i^{-1})$.

(3) Let $\sF$ be as in (1). We call $(\ul{x}, \chi=\boxtimes_{i=1}^r\chi_i)\in (\Spec \cZ_{\Omega})^{\rig} \times \widehat{\cZ_{0,L}}$ a parameter of the $\Omega$-filtration $\sF$ in $(\Spec \cZ_{\Omega})^{\rig} \times \widehat{\cZ_{0,L}}$ if $(\ul{x},\chi_{\varpi_L}=\boxtimes_{i=1}^r \chi_{i,\varpi_L})$ is a parameter of $\sF$ in $(\Spec \cZ_{\Omega})^{\rig} \times \widehat{Z_{L_P}(L)}$.
\end{definition}

\begin{remark}
	(1) Let $\sF$, $\ul{x}$, $\delta$ be as in Definition \ref{defOF} (1). We can twist each $\delta_i$ by a certain algebraic character of $L^{\times}$ so that $(\ul{x}, \delta)$ is parameter of $\sF$. 
	
	(2) For convenience, we may use these two kinds of parameters depending on the situation. Note that the parameters of $\sF$ \big(either in $(\Spec \cZ_{\Omega})^{\rig} \times \widehat{Z_{L_P}(L)}$ or in $(\Spec \cZ_{\Omega})^{\rig} \times \widehat{\cZ_{0,L}}$\big) are in general not unique (see Lemma \ref{RmtOP} below).
\end{remark}

\begin{example}\label{expcrpara}
(1) By Theorem \ref{FOBE} (using the notation there), for any point $x\in \cE_{\Omega, \lambda}(U^p)$, $D_{\rig}(\rho_{x,\widetilde{v}})$ admits an $\Omega_{\widetilde{v}}$-filtration. This is our main motivation to study $(\varphi, \Gamma)$-modules with $\Omega$-filtrations. 
	
(2) Let $\rho$ be as in \S~\ref{introPcr} and use the notation of {\it loc.\ cit.} Let $\ul{x}\in (\Spec \cZ_{\Omega})^{\rig}$ be the point such that $\Delta_{x_i}\cong \gr_i^{\sF}\!\Delta$ for $i=1, \dots, r$. Then by (\ref{inj000}) (and comparing the Hodge-Tate weights), we see that $\sF$ in (\ref{OmeFil00}) is an $\Omega$-filtration of parameter $(\ul{x}, \delta=\boxtimes_{i=1}^r \delta_i:=\boxtimes_{i=1}^r z^{w_{\sF}(\textbf{h})_{s_i}}) \in (\Spec \cZ_{\Omega})^{\rig} \times \widehat{Z_{L_P}(L)}$. Let $\ul{x}'=(x_i')$ be such that $\Delta_{x_i'}\cong \Delta_{x_i} \otimes_{\cR_{E,L}} \cR_{E,L}\big(\unr(\varpi_L^{w_{\sF}(\textbf{h})_{s_i}})\big)$, then $(\ul{x}', \delta^0=\delta|_{\cZ_{0,L}}) \in (\Spec \cZ_{\Omega})^{\rig} \times \widehat{\cZ_{0,L}}$ is a parameter of $\sF$. 
\end{example}

\begin{lemma}\label{RmtOP}
Let $\sF$ be an $\Omega$-filtration of $D$.

(1) Let $(\ul{x},\delta)\in (\Spec \cZ_{\Omega})^{\rig} \times \widehat{Z_{L_P}(L)}$ be a parameter of $\sF$, then all parameters of $\sF$ in $(\Spec \cZ_{\Omega})^{\rig} \times \widehat{Z_{L_P}(L)}$ are of the form $(\ul{x'}, \delta')$ such that, for $i=1, \cdots, r$, $\ttr_{x_i'}\cong \ttr_{x_i} \otimes_E \unr(\alpha_i)$ and $\delta_i'=\delta_i \unr(\alpha_i^{-1}) \eta_i$ for some $\alpha_i \in \overline{E}^{\times}$ and $\eta_i \in \mu_{\Omega_i}$. 

(2) Let $(\ul{x}, \chi)\in (\Spec \cZ_{\Omega})^{\rig} \times \widehat{\cZ_{0,L}}$ be a parameter of $\sF$, then all parameters of $\sF$ in $(\Spec \cZ_{\Omega})^{\rig} \times \widehat{\cZ_{0,L}}$ are of the form $(\ul{x'},\chi')$ such that, for $i=1, \cdots, r$, $\ttr_{x_i'}\cong \ttr_{x_i} \otimes_E \unr(\eta_i(\varpi_L))$ and $\chi'_i =\chi_i\eta_i|_{\co_L^{\times}}$ for some $\eta_i\in \mu_{\Omega_i}$. 
\end{lemma}
\begin{proof}
(2) is an easy consequence of (1). We prove (1). Let $(\ul{x'}, \delta')\in (\Spec \cZ_{\Omega})^{\rig} \times \widehat{\cZ_{0,L}}$ be another parameter of $\sF$. By definition, we have injections 
	\begin{eqnarray} \label{para1}
		\gr_i D \otimes_{\cR_{E,L}} \cR_{E,L}(\delta_i^{-1}) &\hooklongrightarrow &\Delta_{x_i}\\
	\label{para2}
		\gr_i D \otimes_{\cR_{E,L}} \cR_{E,L}((\delta_i')^{-1}) &\hooklongrightarrow &\Delta_{x_i'} \cong \Delta_{x_i} \otimes_{\cR_{E,L}} \cR_{E,L}(\unr(\alpha_i)).
	\end{eqnarray}
From (\ref{para2}), we deduce 
	\[\gr_i D \otimes_{\cR_{E,L}} \cR_{E,L}((\delta_i')^{-1} \unr(\alpha_i))\cong\gr_i D \otimes_{\cR_{E,L}} \cR_{E,L}(\delta_i^{-1}) \otimes_{\cR_{E,L}} \cR_{E,L}(\delta_i (\delta_i')^{-1} \unr(\alpha_i^{-1})) \hooklongrightarrow \Delta_{x_i}.\]
Since both $\gr_i D \otimes_{\cR_{E,L}} \cR_{E,L}(\delta_i^{-1})$ and $\gr_i D \otimes_{\cR_{E,L}} \cR_{E,L}((\delta_i')^{-1} \unr(\alpha_i))$ are de Rham and have $0$ as the minimal $\tau$-Hodge-Tate weight for all $\tau\in \Sigma_L$, we deduce $\delta_i (\delta_i')^{-1}$ is smooth. Using (\ref{injfrDF}) (applied to (\ref{para1}) and (\ref{para2})), we see that there exists $N\in \Z_{\geq 0}$ sufficiently large such that 
	\begin{equation*}
		\Delta_{x_i} \hooklongrightarrow t^{-N} \gr_i D \otimes_{\cR_{E,L}}\cR_{E,L}(\delta_i^{-1}) \hooklongrightarrow t^{-N} \Delta_{x_i} \otimes_{\cR_{E,L}} \cR_{E,L}(\delta_i^{-1} \delta_i' \unr(\alpha_i)).
	\end{equation*}
Using \ \cite[Thm.~A]{Ber08a} \ and \ comparing \ the \ Hodge-Tate \ weights, \ we \ deduce \ $\Delta_{x_i}\cong \Delta_{x_i} \otimes_{\cR_{E,L}} \cR_{E,L}(\delta_i^{-1} \delta_i' \unr(\alpha_i))$, hence $\ttr_{x_i} \cong \ttr_{x_i} \otimes_E (\delta_i^{-1} \delta_i' \unr(\alpha_i))$, implying $\delta_i^{-1} \delta_i' \unr(\alpha_i) \in \mu_{\Omega_i}$. This concludes the proof.
\end{proof}

\begin{remark}
In particular, $\sF$ only has finitely many parameters in $\Spec \cZ_{\Omega} \times \widehat{\cZ_{0,L}}$.
\end{remark}

Let $(\ul{x}, \delta) \in \Spec \cZ_{\Omega} \times \widehat{Z_{L_P}(L)}$ be a parameter of the $\Omega$-filtration $\sF$. We call $\sF$ \textit{generic} if the following condition is satisfied:
\begin{equation}\label{conGene}
\begin{gathered}
\begin{array}{lll}
	&&\text{for $i \neq j$, if $\ttr_{x_j}\cong \ttr_{x_i} \otimes_E \eta $ for some smooth character $\eta$ of $L^{\times}$, } \\
	&&\text{then $\delta_i \delta_j^{-1} \eta \neq z^{-\textbf{k}}$ and $\delta_i \delta_j^{-1} \eta \neq \unr(q_L^{-1}) z^{\textbf{k}}$ for any $\textbf{k} \in \Z_{\geq 0}^{|\Sigma_L|}$.}
\end{array}
\end{gathered}
\end{equation}
By Lemma \ref{RmtOP}, this definition is independent of the choice of the parameter of $\sF$.

\begin{example}
Let $D$ be as in Example \ref{expcrpara} (2), one directly checks that the $\Omega$-filtration $\sF$ is generic if $\rho$ is generic in the sense of \S~\ref{introPcr}.
\end{example}

\begin{lemma}\label{lem:generic}
Assume that the $\Omega$-filtration $\sF$ on $D$ is generic, then we have 
	\begin{equation}\label{generic1}
		\Hom_{(\varphi,\Gamma)}(\gr^i D, \gr^j D)= \Ext^2_{(\varphi,\Gamma)}(\gr^i D, \gr^j D)=0
	\end{equation}
	for $i\neq j$, $i,j \in \{1, \dots, r\}$.
\end{lemma}
\begin{proof}
We \ prove \ the \ statement \ for \ $\Hom$, \ the \ proof \ for \ $\Ext^2$ \ being \ similar \ using $\Ext^2_{(\varphi,\Gamma)}(\gr^i D, \gr^j D)=H^2_{(\varphi,\Gamma)}((\gr^i D)^\vee\otimes_{\cR_{E,L}}\gr^j D)$ and Tate duality (\cite[Thm.\ 1.2 (2)]{Liu07}). We let $(\ul{x}, \delta) \in (\Spec \cZ_{\Omega})^{\rig} \times \widehat{Z_{L_P}(L)}$ be a parameter of $\sF$. Suppose we have a non-zero morphism $f: \gr^i D \ra \gr^j D$ for $i\neq j$. Let $N\in \Z_{\geq 0}$ be sufficiently large, then $f$ induces a non-zero hence injective morphism (see (\ref{injfrDF}))
	\begin{equation}\label{comp}t^N \Delta_{x_i} \ra \gr^i D \otimes_{\cR_{E,L}} \cR_{E,L} (\delta_i^{-1}) \ra \gr^j D \otimes_{\cR_{E,L}} \cR_{E,L} (\delta_i^{-1}) \ra \Delta_{x_j} \otimes_{\cR_{E,L}} \cR_{E,L} (\delta_{i}^{-1}\delta_{j}).
	\end{equation}
Consider the induced injective morphism $\Delta_{x_i} \ra \Delta_{x_j} \otimes_{\cR_{E,L}} \cR_{E,L}(\delta_{i}^{-1} \delta_{j} z^{-\textbf{N}})$ where $\textbf{N}:=(N)_{\tau \in \Sigma_L}$. Since both the source and target are irreducible, the morphism (and hence (\ref{comp})) becomes an isomorphism inverting $t$. Since the left hand side of (\ref{comp}) is de Rham, so is $\Delta_{x_j} \otimes_{\cR_{E,L}} \cR_{E,L} (\delta_{i}^{-1}\delta_{j})$. This implies (e.g.\ by considering the Sen weights) that $\delta_{i}^{-1}\delta_{j}$ is locally algebraic, say of the form $\eta z^{\textbf{k}}$ for a smooth character $\eta$ of $L^{\times}$ and some $\textbf{k}\in \Z^{|\Sigma_L|}$. By using \cite[Thm.~A]{Ber08a} and comparing the Hodge-Tate weights, we see that (\ref{comp}) induces $\Delta_{x_i} \xrightarrow{\sim} \Delta_{x_j} \otimes_{\cR_{E,L}} \cR_{E,L} (\eta)$ (so $\ttr_{x_i}\cong \ttr_{x_j} \otimes_E \eta$). 

We next show $\textbf{k}\in \Z^{|\Sigma_L|}_{\leq 0}$. By definition, $0$ is the minimal $\tau$-Hodge-Tate weight of $\gr^i D \otimes_{\cR_{E,L}} \cR_{E,L} (\delta_{i}^{-1})$ for all $\tau$, while $\wt(\delta_i^{-1} \delta_j )_{\tau}$ is the minimal $\tau$-Hodge-Tate weight of
	\begin{equation*}
		\gr^j D \otimes_{\cR_{E,L}} \cR_{E,L} (\delta_{i}^{-1}) \cong \gr^j D \otimes_{\cR_{E,L}} \cR_{E,L} (\delta_{j}^{-1})\otimes_{\cR_{E,L}} \cR_{E,L} (\delta_{j}\delta_{i}^{-1}).
	\end{equation*}
The second (injective) morphism in (\ref{comp}) then implies $\wt(\delta_i^{-1} \delta_j )_{\tau}\leq 0$ for all $\tau$, hence $\textbf{k}\in \Z^{|\Sigma_L|}_{\leq 0}$. But $\delta_i^{-1}\delta_j \eta^{-1}=z^{\textbf{k}}$ with $\textbf{k}\in \Z^{|\Sigma_L|}_{\leq 0}$ contradicts the genericity (\ref{conGene}) of $\sF$.
\end{proof}

\begin{corollary}\label{Filuniq}
Assume that $D$ admits a generic $\Omega$-filtration $\sF$, and let $(\ul{x}, \delta)\in \Spec \cZ_{\Omega} \times \widehat{Z_{L_P}(L)}$ be a parameter of $\sF$. Then $D$ has a unique $\Omega$-filtration of parameter $(\ul{x}, \delta)$. 
\end{corollary}
\begin{proof}
Suppose $D$ has two $\Omega$-filtrations $\Fil_{\bullet}$, $\Fil_{\bullet}'$ of parameter $(\ul{x},\delta)$. By Lemma \ref{lem:generic} and a standard d\'evissage, we easily deduce $\Hom_{(\varphi, \Gamma)}(\Fil_{r-1} D, \gr'_r D)=0$ and hence the injection $\Fil_{r-1} D \hookrightarrow D$ induces an isomorphism $\Fil_{r-1} D \xrightarrow{\sim} \Fil_{r-1}' D$. We go on replacing $D$ by $\Fil_{r-1} D=\Fil_{r-1}'D$, and we see that $\Fil_i D \hookrightarrow D$ induces an isomorphism $\Fil_{i} D \xrightarrow{\sim} \Fil_{i}' D$ for $i=1, \dots, r$.
\end{proof}

Assume that $D$ admits a (uniquely determined) generic $\Omega$-filtration $\sF$ and let $F_{D,\sF}$ denote the functor $\Art(E) \ra \{\text{Sets}\}$ which sends $A\in \Art(E)$ to the set of isomorphism classes\index{$F_{D,\sF}$}
\begin{equation*}
	F_{D,\sF}(A)=\{(D_A, \pi_A, \sF_A)\}/\sim
\end{equation*}
where (the isomorphisms being defined in an obvious way)
\begin{enumerate}
\item[(1)] $D_A$ is a $(\varphi,\Gamma)$-module of rank $n$ over $\cR_{A,L}$ with $\pi_A: D_A \otimes_A E \xrightarrow{\sim} D$;
\item[(2)] $\sF_A=\Fil_{\bullet} D_A$ is an increasing filtration by $(\varphi,\Gamma)$-submodules over $\cR_{A,L}$ on $D_A$ such that $\Fil_i D_A$, $i\in \{0,\dots,r\}$ is a direct summand of $D_A$ as $\cR_{A,L}$-modules and $\pi_A(\Fil_i D_A)=\Fil_i D$.
\end{enumerate}

\begin{lemma}\label{subf1}
Assume $\Hom_{(\varphi, \Gamma)}(\gr_i D, \gr_j D)=0$ for $i\neq j$. Then $F_{D,\sF}$ is a subfunctor of $F_D$. 	
\end{lemma}
\begin{proof}
By the assumption and a d\'evissage similar to the one for Lemma \ref{Filuniq}, we easily deduce that, if $\Fil_{\bullet} D_A$ and $\Fil'_{\bullet} D_A$ are two filtrations on $(D_A, \pi_{A})\in F_D(A)$, then they have to be equal. The lemma follows.
\end{proof}

Denote by $\End_{\sF}(D):=\{f\in \End_{\cR_{E,L}}(D)\ |\ f(\Fil_i D)\subset \Fil_i D, \ \forall i\}$, which is equipped with a natural $(\varphi,\Gamma)$-action as in the discussion below \cite[Rem.\ 3.5]{Che11}. Recall the following result:

\begin{proposition}[$\text{\cite[Prop.\ 3.6 (2), (3)]{Che11}}$] \label{paradef}
(1) There is a natural isomorphism of $E$-vector spaces $F_{D,\sF}(E[\varepsilon/\varepsilon^2]) \xrightarrow{\sim} H^1_{(\varphi,\Gamma)}(\End_{\sF}(D))$ and 
	\begin{equation*}
		\dim_E F_{D,\sF}(E[\varepsilon/\varepsilon^2])=	\dim_E H^0_{(\varphi,\Gamma)}(\End_{\sF}(D))+\dim_E H^2_{(\varphi,\Gamma)}(\End_{\sF}(D)) +[L:\Q_p]\sum_{i\leq j} n_in_j.
	\end{equation*}
(2) Assume $H^2_{(\varphi,\Gamma)}(\End_{\sF}(D))=0$, then the functor $F_{D,\sF}$ is formally smooth of dimension $\dim_E H^0_{(\varphi,\Gamma)}(\End_{\sF}(D))+[L:\Q_p]\sum_{i\leq j} n_in_j$.
\end{proposition}

We have a natural morphism
\[F_{D,\sF} \lra \prod_i F_{\gr_i D}\]
sending $(D_A, \pi_A, \sF_A)$ to $(\gr_i D_A, \pi_A|_{\gr_i D_A})_{i=1, \dots,r}$.

\begin{proposition}[$\text{\cite[Prop.\ 3.7]{Che11}}$]\label{apxfm1}
Assume $H^2_{(\varphi,\Gamma)}(\Hom_{\cR_{E,L}}(D/\Fil_i D, \gr_i D))=0$ for all $i$, then the morphism $F_{D,\sF} \ra \prod_i F_{\gr_i D}$ is formally smooth. 
\end{proposition}

Let $(\ul{x}, \delta) \in \Spec \cZ_{\Omega} \times \widehat{Z_{L_P}(L)}$ be a parameter of $\sF$, and let $F_{D,\sF}^0$ be the functor $ \Art(E) \ra \{\text{Sets}\}$ sending $A\in \Art(E)$ to the set of isomorphism classes\index{$F_{D,\sF}^0$}
\begin{equation*}
	F^0_{D,\sF}(A)=\{(D_A, \pi_A, \sF_A, \delta_A)\}/\sim
\end{equation*}
where (the isomorphisms being again defined in an obvious way)
\begin{enumerate}
	\item[(1)] $(D_A,\pi_A,\sF_A)\in F_{D,\sF}(A)$;
	\item[(2)] $\delta_{A}=(\delta_{A,i})_{i=1,\dots,r}$ where $\delta_{A,i}: L^{\times} \ra A^{\times}$ is a continuous character such that $\delta_{A,i} \equiv \delta_i \pmod{\fm_A}$ and there exists an injection of $(\varphi,\Gamma)$-modules over $\cR_{A,L}$:
	\begin{equation*}
		\gr_i D_A \hooklongrightarrow \Delta_i \otimes_{\cR_{E,L}} \cR_{A,L}(\delta_{A,i}).
	\end{equation*}
\end{enumerate}
By definition, we have $F_{D,\sF}^0 \cong F_{D,\sF} \times_{\prod_i F_{\gr_i D}} \prod_i F_{\gr_i D}^0$.

\begin{proposition}\label{apxfm2}
	(1) Assume $\Hom_{(\varphi, \Gamma)}(\gr_i D, \gr_j D)=0$ for $i\neq j$, then the functor $F_{D,\sF}^0$ is a subfunctor of $F_D$.
	
	(2) Assume $H^2_{(\varphi,\Gamma)}(\Hom_{\cR_{E,L}}(D/\Fil_i D, \gr_i D))=0$ and $\gr_i D \otimes_{\cR_{E,L}} \cR_{E,L}(\delta_i^{-1})$ has distinct Hodge-Tate weights for $i=1, \dots, r$. Then $F_{D,\sF}^0$ is formally smooth of dimension 
	\begin{equation*}
		\dim_E H^2_{(\varphi, \Gamma)}(\End_{\sF}(D))+\dim_E H^0_{(\varphi,\Gamma)}(\End_{\sF}(D))+[L:\Q_p]\big(\frac{n(n-1)}{2}+r\big).
	\end{equation*}
\end{proposition}
\begin{proof}
	(1) By Lemma \ref{subf} $F_{D,\sF}^0$ is a subfunctor of $F_{D,\sF}$, and by Lemma \ref{subf1} $F_{D,\sF}$ is a subfunctor of $F_D$.
	
	(2) From Proposition \ref{apxfm1}, we deduce that $F_{D,\sF}^0$ is formally smooth over $\prod_i F_{\gr_i D}^0$ and hence is formally smooth by Proposition \ref{apxfm0}. We have then
	\begin{multline*}
		\dim_E F_{D,\sF}^0(E[\varepsilon/\varepsilon^2])\\
		= \!\dim_E F_{D,\sF}(E[\varepsilon/\varepsilon^2]) - \sum_i \dim_E F_{\gr_i D}(E[\varepsilon/\varepsilon^2])+\sum_i \dim_E F_{\gr_i D}^0 (E[\varepsilon/\varepsilon^2]) \\
		=\!\dim_E H^0_{(\varphi, \Gamma)}(\End_{\sF}(D))+\dim_E H^2_{(\varphi,\Gamma)}(\End_{\sF}(D))+[L:\Q_p]\Big(\sum_{i< j} n_in_j\\
		\ \ \ \ \ \ \ \ \ \ \ \ \ \ \ \ \ \ \ \ \ \ \ \ \ \ \ \ \ \ \ \ \ \ \ \ \ \ \ \ \ \ \ \ \ \ \ \ \ \ \ \ \ \ \ \ \ \ \ \ \ \ \ \ \ \ \ \ \ \ \ \ \ \ \ \ \ \ \ \ \ \ \ \ \ \ \ \ \ \ + \sum_{i} \big(\frac{n_i(n_i-1)}{2}+1\big)\Big)\\
		=\!\dim_E H^0_{(\varphi, \Gamma)}(\End_{\sF}(D))+\dim_E H^2_{(\varphi,\Gamma)}(\End_{\sF}(D))+[L:\Q_p]\big(\frac{n(n-1)}{2}+r\big),
	\end{multline*}
	where the second equality follows from Proposition \ref{paradef} (1), Proposition \ref{apxfm0} and the standard fact that $\dim_E F_{\gr_i D}(E[\varepsilon/\varepsilon^2])=1+[L:\Q_p]n_i^2$ (noting that $\gr_i D$ is irreducible). 
\end{proof}

Let $\rho: \Gal_L \ra \GL_n(E)$ be a continuous group morphism and $V$ the associated representation of $\Gal_L$ over $E$. We let $F_{\rho}$ (resp.\ $F_V$) denote the deformation functor of $\rho$ (resp.\ $V$) over $\Art(E)$. So $F_{\rho}$ can be viewed as the framed deformation functor of $V$ over $\Art(E)$. Assume $D\cong D_{\rig}(V)$, then we have $F_D \cong F_V$.
Let $F_{\rho,\sF}^{0}:= F_{D,\sF}^0 \times_{F_V} F_{\rho}$.\index{$F_{\rho,\sF}^0$} Recall that $F_{\rho}$ is pro-representable and is formally smooth over $F_V$ of relative dimension $n^2-\dim_E H^0_{(\varphi, \Gamma)}(\End_{\cR_{E,L}}(D))$.

\begin{corollary}\label{apxthmDF}
	(1) Assume $\Hom_{(\varphi, \Gamma)}(\gr_i D,\gr_j D)=0$ for $i\neq j$, then $F_{\rho,\sF}^0$ is a subfunctor of $F_{\rho}$ and is pro-representable. 
	
	(2) Assume $H^2_{(\varphi,\Gamma)}(\Hom_{\cR_{E,L}}(D/\Fil_i D, \gr_i D))=0$ and $\gr_i D \otimes_{\cR_{E,L}} \cR_{E,L}(\delta_i^{-1})$ has distinct Hodge-Tate weights for $i=1, \dots, r$. Then $F_{\rho,\sF}^{0}$ is formally smooth of dimension
	\begin{multline}\label{dimF0}
		\dim_E H^2_{(\varphi, \Gamma)}(\End_{\sF}(D))+\dim_E H^0_{(\varphi,\Gamma)}(\End_{\sF}(D))-\dim_E H^0_{(\varphi, \Gamma)}(\End_{\cR_{E,L}}(D))\\
		+ n^2+ [L:\Q_p]\big(\frac{n(n-1)}{2}+r\big).
	\end{multline}
\end{corollary}
\begin{proof}
	By Proposition \ref{apxfm2} (1), $F_{\rho,\sF}^{0}$ is a subfunctor of $F_{\rho}$. By Lemma \ref{subf1}, $F_{\rho,\sF}:=F_{D,\sF} \times_{F_D} F_{\rho}$ is a subfunctor of (the pro-representable functor) $F_{\rho}$. Together with \cite[Prop.\ 3.4]{Che11}, we can deduce that $F_{\rho,\sF}$ is pro-representable (noting that we only need to show that
\[F_{\rho,\sF}(A'\times_A A'') \lra F_{\rho,\sF}(A') \times_{F_{\rho,\sF}(A)} F_{\rho,\sF}(A'')\]
is surjective whenever $A''\ra A$ is surjective). It then follows from Proposition \ref{repab0} that $F_{\rho,\sF}^{0}$ is pro-representable. By base change, $F_{\rho,\sF}^{0}$ is formally smooth over $F_{D,\sF}^0$ of relative dimension $n^2-\dim_E H^0_{(\varphi, \Gamma)}(\End_{\cR_{E,L}}(D))$. Together with Proposition \ref{apxfm2} (2), (2) follows. 
\end{proof}

\begin{remark}\label{remApxDF}
Assume that $(D, \sF)$ satisfies the properties in (\ref{generic1}) and that $\gr_i D \otimes_{\cR_{E,L}} \cR_{E,L}(\delta_i^{-1})$ has distinct Hodge-Tate weights for $i=1, \dots, r$. Then using a d\'evissage, one easily deduces that the assumptions in Corollary \ref{apxthmDF} (1) and (2) are satisfied, and that the terms in the first line of (\ref{dimF0}) are all zero. So in this case $F_{\rho,\sF}^{0}$ is (pro-representable) formally smooth of dimension $n^2+[L:\Q_p](\frac{n(n-1)}{2}+r)$.
\end{remark}

\subsection{Bernstein paraboline varieties}\label{s: DO}

By generalizing results in \cite[\S~2.2]{BHS1} on the trianguline variety, we construct and study a certain local Galois deformation space of a modulo $p$ Galois representation which consists of Galois representations admitting an $\Omega$-filtration. 

The following lemma follows easily from \cite[Thm.]{KPX}. We formulate it here since we will frequently use it.

\begin{lemma}\label{lemlfc}
Let $X$ be a reduced rigid analytic space over $E$ and $D$ a $(\varphi, \Gamma)$-module over $\cR_{X,L}$. Assume that, for $i=0, 1, 2$, there exists $d_i\in \Z_{\geq 0}$ such that for all $x\in X$, $\dim_{k(x)} H^i_{(\varphi, \Gamma)}(D_x)=d_i$ where $D_x:=x^* D$. Then $H^i_{(\varphi,\Gamma)}(D)$ is a locally free sheaf of rank $d_i$ over $X$ and for any morphism $f: Y \ra X$ of rigid spaces, we have $H^i_{(\varphi, \Gamma)}(f^*D)\cong H^i_{(\varphi,\Gamma)}(D) {\otimes}_{\co_X} \co_Y$.
\end{lemma}
\begin{proof}
By \cite[Thm.\ 4.4.5 (1)]{KPX} and \cite[Thm.]{KPX}, $H^i_{(\varphi,\Gamma)}(D)$ is a coherent sheaf over $X$, and there exists a complex $[C^0 \xrightarrow{d^0} C^1 \xrightarrow{d^1} C^2]$ of locally free sheaves of finite type over $X$ such that, for $f: Y\ra X$, $H^i_{(\varphi, \Gamma)}(f^*D)$ is isomorphic to the $i$-th cohomology of the complex 
	\begin{equation*}
		[C^0 {\otimes}_{\co_X} \co_Y \ra C^1 {\otimes}_{\co_X} \co_Y \ra C^2 {\otimes}_{\co_X} \co_Y].
	\end{equation*}
	In particular, we have $ H^2_{(\varphi, \Gamma)}(D){\otimes}_{\co_X} \co_Y \cong H^2_{(\varphi, \Gamma)}(f^*D) $. Applying this to points of $X$, we deduce that $H^2_{(\varphi, \Gamma)}(D) \otimes_{\co_X} k(x) \cong H^2_{(\varphi, \Gamma)}(D_x)$ has constant dimension $d_2$ for all $x \in X$. Since $X$ is reduced (and $H^2_{(\varphi, \Gamma)}(D)$ is coherent), this implies that $H^2_{(\varphi, \Gamma)}(D)$ is locally free of rank $d_2$. We deduce that $\Ker(d^1)$ is also locally free of finite type (as the kernel of a surjection between two locally free $\co_X$-modules of finite type is locally free of finite type), and thus $H^1_{(\varphi, \Gamma)}(D) {\otimes}_{\co_X} \co_Y \cong H^1_{(\varphi, \Gamma)}(f^*D) $. Repeating the above argument, we obtain that $H^1_{(\varphi, \Gamma)}(D)$ is locally free of rank $d_1$, that $\Ima(d^0)$ is locally free, and then again that $H^0_{(\varphi, \Gamma)}(D)$ is locally free of rank $d_0$ and $H^0_{(\varphi, \Gamma)}(D){\otimes}_{\co_X} \co_Y \cong H^0_{(\varphi, \Gamma)}(f^*D) $.
\end{proof}

We keep the setting of the beginning of \S~\ref{secDefOD} and fix $\textbf{h}=(\textbf{h}_i)_{i=1, \dots, n}=(h_{i, \tau})_{\substack{i=1, \dots, n\\ \tau \in \Sigma_L}} \in \Z^{\oplus n|\Sigma_L|}$ \textit{strictly} $P$-dominant. Let $(\ttr_i)\in (\Spec \cZ_{\Omega})^{\an}$, we say that a character $\delta$ of $Z_{L_P}(L)$ is \textit{generic for} $(\Omega, \textbf{h})$ if the following condition is satisfied:
\begin{itemize}
\item[] if there exist $i\neq j$ such that $\ttr_j=\ttr_i \otimes_E \eta$ for some smooth character $\eta$ of $L^{\times}$ (noting there are finitely many choices of $\eta$), then $\delta_i \delta_j^{-1} \eta z^{\textbf{h}_{s_i}-\textbf{h}_{s_j}} \neq z^{-\textbf{k}} $ and $\delta_i \delta_j^{-1} \eta z^{\textbf{h}_{s_i}-\textbf{h}_{s_j}}\neq \unr(q_L^{-1})z^{\textbf{k}} $ for all $\textbf{k}=(k_{\sigma})_{\sigma \in \Sigma_L}\in \Z_{\geq 0}^{|\Sigma_L|}$. 
\end{itemize} 
The set of such points is Zariski open and Zariski dense in $\widehat{Z_{L_P}(L)}$. For each $\ttr_i$, we have a natural finite morphism $\bG_m \ra \Spec \cZ_{\Omega_i}$, $\alpha \mapsto \ttr_i \otimes_E \unr(\alpha)$ (note that $\cZ_{\Omega_i} \cong E[x,x^{-1}]^{\mu_{\Omega_i}^{\unr}}$). We have and fix an isomorphism (depending on the choice of $\varpi_L$) $\bG_m^{\rig} \times \widehat{\co_L^{\times}} \xrightarrow{\sim} \widehat{L^{\times}}$, $(a, \chi) \mapsto \unr(a) \chi_{\varpi_L}$. We define 
\[\sZ:=(\Spec \cZ_{\Omega})^{\rig} \times \widehat{\cZ_{0,L}},\]
$\iota_{\ttr_i}$ as the composition $\iota_{\ttr_i}: \widehat{L^{\times}}\xrightarrow{\sim} \bG_m^{\rig} \times \widehat{\co_L^{\times}} \ra (\Spec \cZ_{\Omega_i})^{\rig} \times \widehat{\co_L^{\times}}$, and $\iota_{\ul{\ttr}}$ as the composition\index{$\sZ$}:
\begin{equation*}
\iota_{\ul{\ttr}}: \widehat{Z_{L_P}(L)} \xrightarrow{\sim} \prod_{i=1}^r \widehat{L^{\times}} \xrightarrow{\sim} \prod_{i=1}^r(\bG_m^{\rig} \times \widehat{\co_L^{\times}}) \xrightarrow{(\iota_{\ttr_i})} \prod_{i=1}^r \big((\Spec \cZ_{\Omega_i})^{\rig} \times \widehat{\co_L^{\times}}\big)\cong \sZ.
\end{equation*}
A point of $\sZ$ is called generic if its preimage in $\widehat{Z_{L_P}(L)}$ is generic for $(\Omega, \textbf{h})$. One can check that this notion is independent of the choice of $\{\ttr_i\}$. Denote by $\sZ^{\gen}\subset \sZ$ the set of points that are generic. One can also check that if $(\ul{x}, \chi)\in \sZ^{\gen}$ then $(\ul{x}, \chi_{\varpi_L})$ satisfies the condition in (\ref{conGene}). Any affinoid open in $\sZ$ can only have finitely many points that are not generic (since the same holds for $\widehat{Z_{L_P}(L)}$). Hence $\sZ^{\gen}$ is Zariski-open and Zariski-dense in $\sZ$. 

Let $\overline{\rho}: \Gal_L \ra \GL_n(k_E)$ be a continuous group morphism. Let $U_{\Omega, \textbf{h}}(\overline{\rho})$ be the subset of $(\Spf R_{\overline{\rho}})^{\rig} \times \sZ$ of the points $(\rho, \ul{x}, \chi)$ such that \index{$U_{\Omega, \textbf{h}}(\overline{\rho})$}
\begin{itemize}
	\item[(1)] $(\ul{x}, \chi)\in \sZ^{\gen}$;
	\item[(2)] $D_{\rig}(\rho)$ admits an $\Omega$-filtration $\sF=\{\Fil_i D_{\rig}(\rho)\}$ such that 
	\begin{equation}\label{galOF}
		\gr_i D_{\rig}(\rho) \otimes_{\cR_{k(x),L}} \cR_{k(x),L}(\chi_{i,\varpi_L}^{-1}) \hooklongrightarrow \Delta_{x_i} \otimes_{\cR_{k(x),L}} \cR_{k(x),L}(z^{\textbf{h}_{s_i}})
	\end{equation}
	and the image has Hodge-Tate weights $(\textbf{h}_{s_{i-1}+1}, \dots, \textbf{h}_{s_i})$.
\end{itemize}
We define $X_{\Omega, \textbf{h}}(\overline{\rho})$ to be the Zariski-closure of $U_{\Omega, \textbf{h}}(\overline{\rho})$ in $(\Spf R_{\overline{\rho}})^{\rig} \times \sZ$.\index{$X_{\Omega, \textbf{h}}(\overline{\rho})$} By definition $X_{\Omega, \textbf{h}}(\overline{\rho})$ is reduced and we have a natural morphism
\begin{equation*}
	\omega: X_{\Omega, \textbf{h}}(\overline{\rho}) \lra \sZ.
\end{equation*} 
We define an action of $\mu_{\Omega}=\{\psi=(\psi_i): Z_{L_P}(L) \ra E^{\times}\ |\ \psi_i \in \mu_{\Omega_i}\}$ on $\sZ$ such that $\psi=(\psi_i)\in \mu_{\Omega}$ sends $\big((\ttr_i), (\chi_i)\big)$ to $\big((\ttr_i \otimes_E \unr(\psi_i(\varpi_L))), (\chi_i \psi_i^0)\big)$. It induces an action of $\mu_{\Omega}$ on $(\Spf R_{\overline{\rho}})^{\rig} \times \sZ$ by acting trivially on $(\Spf R_{\overline{\rho}})^{\rig}$. By Lemma \ref{RmtOP} (2), $U_{\Omega, \textbf{h}}(\overline{\rho})$ is preserved by the action of $\mu_{\Omega}$. We then deduce that $X_{\Omega, \textbf{h}}(\overline{\rho})$ is also preserved by $\mu_{\Omega}$:

\begin{proposition}\label{twGal}
	A point $\big(\rho, (\ttr_i), (\chi_i)\big) \in (\Spf R_{\overline{\rho}})^{\rig} \times \sZ$ lies in $X_{\Omega, \textbf{h}}(\overline{\rho})$ if and only if the point $\big(\rho, (\ttr_i \otimes_E \unr(\psi_i(\varpi_L))), (\chi_i \psi_i^0)\big)$ lies in $X_{\Omega, \textbf{h}}(\overline{\rho})$ for all $\psi=(\psi_i)\in \mu_{\Omega}$.
\end{proposition}

Let $\psi=(\psi_i)$ be a smooth character of $Z_{L_P}(L)$, and $\Omega$ be the Bernstein component $\{\pi \otimes_E \psi\}_{\pi \in \Omega}$. Let $\textbf{h}'=(\textbf{h}'_i)_{i=1,\dots, n}=(h'_{i,\tau})_{\substack{i=1, \dots, n\\ \tau \in \Sigma_L}}\in \Z^{\oplus n|\Sigma_L|}$ be strictly $P$-dominant such that $\textbf{h}'-\textbf{h}=\fd \circ \dett_{L_P}$ for a weight $\fd=(\fd_i)_{i=1,\dots, r}$ of $\fz_{L_P,L}$. The condition (\ref{galOF}) is then equivalent to
\begin{multline*}
	\gr_i D_{\rig}(\rho) \otimes_{\cR_{k(x),L}} \cR_{k(x),L}\big(\chi_{i,\varpi_L}^{-1} \delta_{\fd_i,\varpi_L}^0 \psi_{i, \varpi_L}^0\big) \\
	\hooklongrightarrow \Delta_{x_i} \otimes_{\cR_{k(x),L}} \cR_{k(x),L}\big(\psi_{i, \varpi_L}^0\unr(\delta_{\fd_i}^{-1}(\varpi_L))\big) \otimes_{\cR_{k(x),L}} \cR_{k(x),L}(z^{\textbf{h}'_{s_i}}).
\end{multline*}
The isomorphism 
\begin{equation*}
	(\Spf R_{\overline{\rho}}) \times \sZ \xrightarrow{\sim} 	(\Spf R_{\overline{\rho}}) \times \sZ,\ \big(\rho, (\ttr_i), (\chi_i)\big) \mapsto \big(\rho, (\ttr_i \otimes_E (\psi_{i,\varpi_L}^0 (\delta_{\fd_i, \varpi_L}^{\unr})^{-1}), (\chi_i(\delta_{\fd_i}^0\psi_i^0)^{-1})\big)
\end{equation*}
sends bijectively $U_{\Omega, \textbf{h}}(\overline{\rho})$ to $U_{\Omega', \textbf{h}'}(\overline{\rho})$. 
We then deduce (compare with Proposition \ref{twBEi}):

\begin{proposition}\label{bpetw}
We have an isomorphism:
	\begin{equation*}
		X_{\Omega, \textbf{h}}(\overline{\rho}) \xlongrightarrow{\sim} X_{\Omega', \textbf{h}'}(\overline{\rho}), \ \big(\rho, (\ttr_i), (\chi_i)\big) \longmapsto \big(\rho, (\ttr_i\otimes (\psi_{i}^0)_{\varpi_L}(\delta_{\fd_i, \varpi_L}^{\unr})^{-1}), (\chi_i(\delta_{\fd_i}^0\psi_{i}^0)^{-1})\big). 
	\end{equation*}
\end{proposition}

\begin{remark}\label{remBPE}
	(1) Assume $P=B$, hence $L_P=Z_{L_P}=T$. Using the isomorphism (where $\boxtimes_{i=1}^r \pi_{x_i}$ is the smooth character of $T(L)=L_P(L)$ associated to $\ul{x}$):
	\begin{equation*}
	\iota_{\Omega, \textbf{h}}: \sZ \xlongrightarrow{\sim} \widehat{T(L)}, \ (\ul{x},\chi) \mapsto (\boxtimes_{i=1}^r \pi_{x_i}) \chi z^{\textbf{h}},
	\end{equation*}
	we view $X_{\Omega, \textbf{h}}(\overline{\rho})$ as a closed rigid subspace of $(\Spf R_{\overline{\rho}})^{\rig} \times \widehat{T(L)}$ via the following morphism, that we still denote by $\iota_{\Omega, \textbf{h}}$:
	\begin{equation*}
		\iota_{\Omega, \textbf{h}}:X_{\Omega, \textbf{h}}(\overline{\rho}) \lra (\Spf R_{\overline{\rho}})^{\rig} \times \sZ \xlongrightarrow{\id\times \iota_{\Omega, \textbf{h}}} (\Spf R_{\overline{\rho}})^{\rig} \times \widehat{T(L)}.
	\end{equation*}
	Such a closed rigid subspace is in fact independent of the choice of $(\Omega, \textbf{h})$ by Proposition \ref{bpetw}.
	By assumption, the injection in (\ref{galOF}) is actually an isomorphism. We then deduce that $U_{\Omega, \textbf{h}}(\overline{\rho})\subset (\Spf R_{\overline{\rho}})^{\rig} \times \widehat{T(L)}$ coincides with $U_{\tri}^{\square}(\overline{r})^{\reg}$ in \cite[\S~2.2]{BHS1}, hence $X_{\Omega, \textbf{h}}(\overline{\rho})$ coincides with the trianguline variety $X_{\tri}^{\square}(\overline{r})$ of \textit{loc.\ cit.}
	
	(2) By definition, for $(\rho, \ul{x}, \chi)\in U_{\Omega, \textbf{h}}(\overline{\rho})$, $\big(\ul{x}, \boxtimes_{i=1}^r (\chi_{i,\varpi_L}z^{\textbf{h}_{s_i}})\big)\in (\Spec \cZ_{\Omega})^{\rig} \times \widehat{Z_{L_P}(L)}$ is a (generic) parameter of the $\Omega$-filtration $\sF$ on $D_{\rig}(\rho)$. We will show in Corollary \ref{OFpw} below that, for any point $(\rho, \ul{x}, \chi)\in X_{\Omega, \textbf{h}}(\overline{\rho})$, $D_{\rig}(\rho)$ is naturally equipped with an $\Omega$-filtration $\sF$. However, $(\ul{x}, (\chi_{i,\varpi_L}z^{\textbf{h}_{s_i}}))$ is not forcedly in general a parameter of $\sF$. This phenomenon is closely related to the problem of (local) companion points (see Remark \ref{remNPara} and Example \ref{expcrpara}). 
\end{remark}

The following theorem, generalizing \cite[Thm.~2.6]{BHS1}, is the main result of this section.

\begin{theorem}\label{DFOL}
(1) The rigid analytic space $X_{\Omega, \textbf{h}}(\overline{\rho})$ is equidimensional of dimension
	\[n^2 + [L:\Q_p] \Big(\frac{n(n-1)}{2}+r\Big).\]
	
(2) The set $U_{\Omega, \textbf{h}}(\overline{\rho})$ is Zariski-open and Zariski-dense in $X_{\Omega, \textbf{h}}(\overline{\rho})$.
	
(3) The rigid space $U_{\Omega, \textbf{h}}(\overline{\rho})$ is smooth over $E$, and the morphism $\omega|_{U_{\Omega, \textbf{h}}(\overline{\rho})}: U_{\Omega, \textbf{h}}(\overline{\rho}) \ra \sZ$ is smooth.
\end{theorem}

By Theorem \ref{DFOL}, Corollary \ref{para} and Corollary \ref{rgloOF} (1) (applied to $X=X_{\Omega, \textbf{h}}(\overline{\rho})$), we get:

\begin{corollary}\label{OFpw}
	Let \ $x\!=\!(\rho, \ul{x}, \chi)\!\in \! X_{\Omega, \textbf{h}}(\overline{\rho})$, \ then \ $D_{\rig}(\rho)$ \ admits \ an \ $\Omega$-filtration $\sF\!=\!\{\Fil_i D_{\rig}(\rho)\}$ such that, for all $i=1, \dots, r$, 
	\begin{equation*}
		\gr_i D_{\rig}(\rho) \otimes_{\cR_{k(x),L}} \cR_{k(x),L}(\chi_{i,\varpi_L}^{-1})\Big[\frac{1}{t}\Big] \cong \Delta_{x_i}\Big[\frac{1}{t}\Big].
	\end{equation*}
\end{corollary}

In the rest of this section, we prove Theorem \ref{DFOL} by generalizing the proof of \cite[Thm.\ 2.6]{BHS1}. One difference is that, instead of having a smooth rigid space (the $\cS(\overline{r})$ of \emph{loc.\ cit.}) mapping onto the whole $U_{\Omega, \textbf{h}}(\overline{\rho})$ ($U_{\tri}(\overline{r})^{\reg}$ of \emph{loc.\ cit.}), we only have smooth rigid spaces mapping onto certain open subspaces of $U_{\Omega, \textbf{h}}(\overline{\rho})$ that cover $U_{\Omega, \textbf{h}}(\overline{\rho})$. 

We start with the construction of some auxiliary smooth rigid analytic spaces. For a reduced rigid space $X$, we denote by $\co_X^+$ the subsheaf of $\co_X$ of functions of norm less than $1$. For $i=1, \dots, r$, let $\alpha_i\in E^{\times}$, $\overline{\varrho}_i: \Gal_L \ra \GL_{n_i}(k_E)$ be a continuous representation and let $\xi_i$ be the (cuspidal) inertial type associated to $\Omega_i$. We consider the following functor:
\begin{equation}\label{func1}
	X \longmapsto \big\{\rho_X,\varrho_{i,X}, \chi_{i,X}, \Fil_{\bullet}, \nu_i\big\}/\sim
\end{equation}
where (the isomorphisms being defined in an obvious way)
\begin{itemize}
	\item[(1)] $X$ is a reduced rigid analytic space over $E$;
	\item[(2)] $\rho_X: \Gal_L \ra \GL_n(\co_X^+)$ (resp.\ $\varrho_{i,X}: \Gal_L \ra \GL_{n_i}(\co_X^+)$ for $i=1, \dots, r$) is a continuous morphism such that, for all $x\in X$, the reduction modulo the maximal ideal of $\co_{k(x)}$ of $\rho_x:=\rho_X \otimes_{\co_X^+} \co_{k(x)}$ (resp.\ of $\varrho_{i,x}:=\varrho_{i,X} \otimes_{\co_X^+} \co_{k(x)}$) is equal to $\overline{\rho}$ (resp.\ $\overline{\varrho}_i$);
	\item[(3)] $\varrho_{i,x}$ is de Rham of Hodge-Tate weights $\textbf{h}^i:=(\textbf{h}_{s_{i-1}+1}, \dots, \textbf{h}_{s_i})$ and of type $\xi_i$ for all $x\in X$;
	\item[(4)] $\chi_{i,X}: \co_L^{\times} \ra \co_X^{\times}$ is a continuous morphism such that, for all points $x\in X$, $(\{\ttr(\varrho_{i,x})\otimes_E \unr(\alpha_i)\}, \{\chi_{i,x}\})\in \sZ^{\gen}$ where $\chi_{i,x}=x^* \chi_{i,X}$ and $\ttr(\varrho_{i,x})$ is the Weil-Deligne representation associated to $\varrho_{i,x}$;
	\item[(5)] $\Fil_{\bullet} D_{\rig}(\rho_X)=\big(0=\Fil_0 D_{\rig}(\rho_X) \subsetneq \Fil_1 D_{\rig}(\rho_X) \subsetneq \cdots \subsetneq \Fil_r D_{\rig}(\rho_X)=D_{\rig}(\rho_X)\big)$ is an increasing filtration on $D_{\rig}(\rho_X)$ by $(\varphi, \Gamma)$-submodules over $\cR_{X,L}$ which are direct summands of $D_{\rig}(\rho_X)$ as $\cR_{X,L}$-modules;
	\item[(6)] $\nu_i: \gr_i D_{\rig}(\rho_X) \xrightarrow{\sim} D_{\rig}(\varrho_{i,X}) \otimes_{\cR_{X,L}} \cR_{X,L}((\chi_{i,X})_{\varpi_L} \unr(\alpha_i))$ is an isomorphism of $(\varphi, \Gamma)$-modules over $\cR_{X,L}$.
\end{itemize}

\begin{proposition}\label{auxis1}
	The functor in (\ref{func1}) is representable by a smooth reduced rigid analytic space over $E$ of dimension $(\sum_{i=1}^r n_i^2 + n^2) + [L:\Q_p] (\frac{n(n-1)}{2} +r)$, that we denote by $\cS_{\overline{\rho}}^0(\Omega,\textbf{h},\ul{\alpha},\{\overline{\varrho}_i\})$.
\end{proposition}
\begin{proof}
	For $i=1, \dots, r$, consider
	\begin{equation}\label{pstDF1}
		(\Spf R_{\overline{\varrho}_i}^{\pcr}(\xi_i,\textbf{h}^i))^{\rig} \ra (\Spec \cZ_{\Omega_i})^{\rig} \xrightarrow{\tw_{\alpha_i}} (\Spec \cZ_{\Omega_i})^{\rig}
	\end{equation}
where the first morphism is induced by the morphism in \cite[Thm.\ 4.1]{CEGGPS} (see also \cite[Prop.\ 4.3]{CEGGPS}), and the second morphism sends $\ttr_i$ to $\ttr_i \otimes_E \unr(\alpha_i)$. The morphism (\ref{pstDF1}) is given pointwise by $\varrho_i \mapsto \ttr(\varrho_i) \otimes_E \unr(\alpha_i)$. Taking their product (with the identity map on $ \widehat{\cZ_{0,L}}$), we define
\begin{equation}\label{etai}
\eta_{\ul{\alpha}, \{\overline{\varrho}_i\}}: \prod_{i=1}^r(\Spf R_{\overline{\varrho}_i}^{\pcr}(\xi_i,\textbf{h}^i))^{\rig} \times \widehat{\cZ_{0,L}} \longrightarrow (\Spec \cZ_{\Omega})^{\rig} \times \widehat{\cZ_{0,L}} \cong \sZ.
\end{equation}
Let $\cU:=\eta_{\ul{\alpha}, \{\overline{\varrho}_i\}}^{-1}(\sZ^{\gen})$, which is Zariski-open and Zariski-dense in $\prod_{i=1}^r(\Spf R_{\overline{\varrho}_i}^{\pcr}(\xi_i,\textbf{h}^i))^{\rig} \times \widehat{\cZ_{0,L}}$. Indeed, the Zariski-density follows from the fact that any affinoid open in the left hand side of (\ref{etai}) contains points with $\widehat{\cZ_{0,L}}$-entry $\chi=\boxtimes \chi_i$ satisfying $\wt(\chi_i (\chi_j)^{-1})_{\tau} \notin \Z$ for all $i\neq j$ and $\tau\in \Sigma_L$, and such points are sent to generic points via $\eta_{\ul{\alpha}, \{\overline{\varrho}_i\}}$. Let $\varrho_i^{\univ}$ be the universal Galois deformation over $(\Spf R_{\overline{\varrho}_i}^{\pcr}(\xi_i,\textbf{h}^i))^{\rig}$ and $\varrho_{i,\cU}^{\univ}$ be its pull-back over $\cU$ via the composition
 \[\cU\hookrightarrow \prod_{i=1}^r(\Spf R_{\overline{\varrho}_i}^{\pcr}(\xi_i,\textbf{h}^i))^{\rig} \times \widehat{\cZ_{0,L}}\twoheadrightarrow (\Spf R_{\overline{\varrho}_i}^{\pcr}(\xi_i,\textbf{h}^i))^{\rig}.\]
Likewise, let $\chi_{i,\cU}^{\univ}$ be the pull-back over $\cU$ of the universal character over $\widehat{\co_L^{\times}}$ via the composition
\[\cU \hookrightarrow \prod_{i=1}^r(\Spf R_{\overline{\varrho}_i}^{\pcr}(\xi_i,\textbf{h}^i))^{\rig} \times \widehat{\cZ_{0,L}}\twoheadrightarrow \widehat{\cZ_{0,L}} \cong (\widehat{\co_L^{\times}})^r \xrightarrow{\pr_i} \widehat{\co_L^{\times}}.\] 
Put
\begin{equation*}
\cD_{i,\cU}:= D_{\rig}(\varrho_{i,\cU}^{\univ}) \otimes_{\cR_{\cU,L}} \cR_{\cU,L}\big((\chi_{i,\cU}^{\univ})_{\varpi_L} \unr(\alpha_i)\big)
\end{equation*}
and $\cD_{i,z}:=z^* \cD_{i,\cU}$ for each $z\in \cU$. Since $\xi_i$ is a cuspidal inertial type, $\cD_{i,z}$ is irreducible for all $z$. Using the fact that $ \eta_{\ul{\alpha}, \{\overline{\varrho}_i\}}(z)\in \sZ^{\gen}$, one can calculate:
	\begin{equation*}
		\dim_{k(z)}H^s_{(\varphi,\Gamma)}(\cD_{i,z} \otimes_{\cR_{k(z),L}} \cD_{i,z}^{\vee})=
		\begin{cases}
			1 & s=0 \\
			n_i^2[L:\Q_p] +1 & s=1 \\
			0 & s=2
		\end{cases}
	\end{equation*}
	and by Lemma \ref{lem:generic}, for $i\neq j$:
	\begin{equation*}
		\dim_{k(z)}H^s_{(\varphi,\Gamma)}(\cD_{i,z} \otimes_{\cR_{k(z),L}} \cD_{j,z}^{\vee})=
		\begin{cases}
			0 & s=0 \\
			n_i n_j[L:\Q_p] & s=1 \\
			0 & s=2.
		\end{cases}
	\end{equation*}
	We deduce by Lemma \ref{lemlfc} that for any rigid analytic space $\cU'$ over $\cU$, all $H^s_{(\varphi,\Gamma)}(\cD_{i,\cU'} \otimes_{\cR_{\cU',L}} \cD_{j,\cU'}^{\vee})$ are locally free over $\cU'$, where $\cD_{i',\cU'}$ denotes the base change of $\cD_{i',\cU}$ over $\cU'$.	

Now we let $\cU_1:=\cU$, $\cC_{1,\cU_1}:=\cD_{1,\cU_1}$ and $\nu_1: \cC_{1,\cU_1} \xrightarrow{\sim} \cD_{1,\cU_1}$ be an isomorphism. Let $\cU_2 \ra \cU_1$ be the vector bundle of rank $n_1 n_2 [L:\Q_p]$ associated to the locally free $\co_{\cU_1}$-module $\Ext^1_{(\varphi, \Gamma)}(\cD_{2,\cU_1},\cC_{1,\cU_1})$ (see also the proof of \cite[Thm.\ 3.3]{Che13} and \cite[Thm.\ 2.4]{HeSc}). We have a universal extension of $(\varphi, \Gamma)$-modules over $\cR_{\cU_2, L}$:
\begin{equation*}
0 \ra \cC_{1,\cU_2} \ra \cC_{2,\cU_2} \xrightarrow{\nu_2} \cD_{2,\cU_2} \ra 0.
\end{equation*}
By similar arguments as in the previous paragraph and a d\'evissage, we prove that the $\co_{\cU_2}$-module $H_{(\varphi, \Gamma)}^1(\cC_{2,\cU_2} \otimes_{\cR_{\cU_2, L}} \cD_{3, \cU_2})$ is locally free of rank $[L:\Q_p](n_1+n_2)n_3$. We let $\cU_3 \ra \cU_2$ be the vector bundle associated to $H_{(\varphi, \Gamma)}^1(\cC_{2,\cU_2} \otimes_{\cR_{\cU_2, L}} \cD_{3, \cU_2}^{\vee})$. By induction, we finally obtain a sequence of rigid analytic spaces
	\begin{equation*}
		\cV:=\cU_r \ra \cU_{r-1} \ra \cdots \ra \cU_1 =\cU
	\end{equation*}
such that 
\begin{itemize}
\item[(1)] for $i\geq 2$, $\cU_{i}$ is a vector bundle of rank $[L:\Q_p] (\sum_{j=1}^{i-1}n_{j})n_{i}$ over $\cU_{i-1}$; 
\item[(2)] there is a (universal) $(\varphi, \Gamma)$-module $\cC_{i, \cU_{i}}$ over $\cR_{\cU_{i},L}$ equipped with an increasing filtration ($0=\Fil_0 \subsetneq \Fil_1 \subsetneq \cdots \subsetneq \Fil_i=\cC_{i,\cU_i}$) given by $(\varphi, \Gamma)$-submodules over $\cR_{\cU_{i},L}$ which are direct summands of $\cC_{i,\cU_i}$ as $\cR_{\cU_i,L}$-modules;
\item[(3)] there are isomorphisms of $(\varphi, \Gamma)$-modules over $\cR_{\cU_i,L}$: $\nu_j: \gr_j \cC_{i,\cU_i} \xrightarrow{\sim} \cD_{j,\cU_i}$ for $j\leq i$.
\end{itemize}
By construction and \cite[Thm.\ 3.3.8]{Kis08}, $\cV$ is smooth of dimension (recall $\fn_P$ is the Lie algebra over $E$ of the nilpotent radical $N_P$ of $P$)
	\[\dim \cU + [L:\Q_p] \dim_E \fn_P=\sum_{i=1}^r\Big(n_i^2+[L:\Q_p] \frac{n_i(n_i-1)}{2}\Big)+[L:\Q_p]r+[L:\Q_p] \dim_E \fn_P.\]
Now we apply the same argument as in the proof of \cite[Thm.~2.6]{BHS1} for the representability of $\cS^\square(\overline{r})$ of \textit{loc.\ cit}. Let $\cV^{\adm}$ be the maximal adic open of $\cV$ such that there exist a vector bundle $\cM$ over $\cV^{\adm}$ and a continuous morphism $\rho_{\cV^{\adm}}: \Gal_L \ra \Aut_{\co_{\cV^{\adm}}}(\cM)$ satisfying $D_{\rig}(\rho_{\cV^{\adm}}) \cong \cD_{\cV^{\adm}}$ (cf.\ \cite[Thm.\ 1.3]{Hel16}). Note that $\cV^{\adm}$ is also a rigid space by \cite[(1.1.11)]{Hub}. Let $\cV^{\adm, \square}$ be the $\GL_n$-torsor of the trivialization of the vector bundle $\cM$ . Let $\cV_0^{\adm,\square} \subset \cV^{\adm, \square}$ be the admissible open of points such that $\rho_{\cV_0^{\adm,\square}}:=\rho_{\cV^{\adm,\square}}|_{\cV_0^{\adm,\square}}$ has image in $\Gamma(\cV_0^{\adm,\square}, \co_{\cV_0^{\adm, \square}}^+)$. Finally let $\cS_{\overline{\rho}}^0(\Omega,\textbf{h},\ul{\alpha}, \{\overline{\varrho}_i\})$ be the admissible open of $\cV_0^{\adm, \square}$ such that the reduction modulo the maximal ideal of $\co_{k(x)}$ of $\rho_{\cV_0^{\adm,\square}} \otimes_{\co_{\cV_0^{\adm, \square}}^+} \co_{k(x)}$ is equal to $\overline{\rho}$ for $x\in \cS_{\overline{\rho}}^0(\Omega,\textbf{h},\ul{\alpha}, \{\overline{\varrho}_i\})$. We see that $\cS_{\overline{\rho}}^0(\Omega,\textbf{h},\ul{\alpha}, \{\overline{\varrho}_i\})$ is smooth of dimension
	\begin{multline*}
		\sum_{i=1}^r\Big(n_i^2+[L:\Q_p] \frac{n_i(n_i-1)}{2}\Big)+[L:\Q_p] r +[L:\Q_p] \dim_E \fn_P+n^2\\
		=\Big(\sum_{i=1}^r n_i^2 + n^2\Big) + [L:\Q_p] \frac{n(n-1)}{2} +[L:\Q_p]r.
	\end{multline*}
It is then formal to check that $\cS_{\overline{\rho}}^0(\Omega,\textbf{h},\ul{\alpha}, \{\overline{\varrho}_i\})$ represents the functor (\ref{func1}).
\end{proof}

We write $\cS^0_{\overline{\rho}}:=\cS_{\overline{\rho}}^0(\Omega,\textbf{h},\ul{\alpha}, \{\overline{\varrho}_i\})$ for simplicity. Consider the composition:
\begin{equation*}
\cS_{\overline{\rho}}^0 \lra \Big(\prod_{i=1}^r (\Spf R_{\overline{\rho}}^{\pcr}(\xi_i, \textbf{h}^i))^{\rig}\Big) \times \widehat{\cZ_{0,L}} \xlongrightarrow{\eta_{\ul{\alpha}, \{\overline{\varrho}_i\}}} (\Spec \cZ_{\Omega})^{\rig}\times \widehat{\cZ_{0,L}}\cong \sZ.
\end{equation*}
where the first map follows from the construction of $\cS^0_{\overline{\rho}}$ (note that this map is smooth). Let $\Delta_{\Omega_i}$ be the universal $p$-adic differential equation over $(\Spec \cZ_{\Omega_i})^{\rig}$ constructed in \S~\ref{sec_pDf} and $\chi_i^{\univ}$ be the universal character of $\co_L^{\times}$ over $\widehat{\co_L^{\times}}$. We let $\chi_{i,\cS_{\overline{\rho}}^0}^{\univ}$, $\Delta_{\Omega_i, \cS_{\overline{\rho}}^0}$ be the pull-back of $\chi_i^{\univ}$, $\Delta_{\Omega_i}$ over $\cS_{\overline{\rho}}^0$. Let $\rho_{\cS^0}^{\univ}$ be the universal $\Gal_L$-representation over $\cS_{\overline{\rho}}^0$. Similarly as in the proof of Proposition \ref{smooth11} below (using Lemma \ref{lemlfc} and a direct computation of the cohomology of $(\varphi, \Gamma)$-modules over Robba rings with coefficients in finite extensions of $E$), we can show that
\begin{equation}\label{HomOmega}
\Hom_{(\varphi, \Gamma)}\Big(\gr_i D_{\rig}(\rho_{\cS^0_{\overline{\rho}}}^{\univ})\otimes_{\cR_{\cS_{\overline{\rho}}^0,L}} \cR_{\cS_{\overline{\rho}}^0,L}\big((\chi_{i,\cS_{\overline{\rho}}^0}^{\univ})_{\varpi_L}^{-1}\big), \Delta_{\Omega_i, \cS_{\overline{\rho}}^0}\otimes_{\cR_{E,L}} \cR_{E,L}(z^{\textbf{h}_{s_i}})\Big)
\end{equation}
is locally free of rank $1$ over $\cS_{\overline{\rho}}^0$ for any $i=1, \dots, r$. We let $\cS_{\overline{\rho}}(\Omega,\textbf{h},\ul{\alpha}, \{\overline{\varrho}_i\})$ be the fibre product over $\cS^0_{\overline{\rho}}$ for all $i$ of the $\bG_m^{\rig}$-torsors trivializing the invertible modules in (\ref{HomOmega}), which is hence a $(\bG_m^{\rig})^r$-torsor over $\cS_{\overline{\rho}}^0$ and a reduced rigid analytic space over $E$ of dimension 
\begin{equation}\label{dimSrho}
\Big(r+\sum_{i=1}^r n_i^2 + n^2\Big) + [L:\Q_p] \Big(\frac{n(n-1)}{2} +r\Big).
\end{equation}
We have a natural commutative diagram
\begin{equation}\label{smCov1}
	\begin{CD}
		\cS_{\overline{\rho}}(\Omega,\textbf{h},\ul{\alpha}, \{\overline{\varrho}_i\}) @>>> (\Spf R_{\overline{\rho}})^{\rig} \times (\prod_{i=1}^r (\Spf R_{\overline{\rho}}^{\pcr}(\xi_i, \textbf{h}^i))^{\rig}) \times \widehat{\cZ_{0,L}} \\
		@V \kappa_{\ul{\alpha}, \{\overline{\varrho}_i\}} VV @VVV \\
		X_{\Omega, \textbf{h}}(\overline{\rho}) @>>> (\Spf R_{\overline{\rho}})^{\rig} \times \sZ.
	\end{CD}
\end{equation}
The existence of the upper horizontal morphism follows from the construction of $\cS_{\overline{\rho}}(\Omega,\textbf{h},\ul{\alpha},\{\overline{\varrho}_i\})$, and it is also clear that the composition of the upper horizontal with the right vertical morphism has image included in $X_{\Omega, \textbf{h}}(\overline{\rho})$, from which we obtain the left vertical morphism $ \kappa_{\ul{\alpha}, \{\overline{\varrho}_i\}}$.

Let $\rho^{\univ}$ be the universal framed Galois deformation of $\overline{\rho}$ over $(\Spf R_{\overline{\rho}})^{\rig}$. And we let $\rho_{X_{\Omega, \textbf{h}}(\overline{\rho})}^{\univ}$, $\Delta_{\Omega_i,X_{\Omega, \textbf{h}}(\overline{\rho})}$, $\chi_{i,X_{\Omega, \textbf{h}}(\overline{\rho})}^{\univ}$ be the pull-back of $\rho^{\univ}$, $\Delta_{\Omega_i}$, $\chi_i^{\univ}$ over $X_{\Omega, \textbf{h}}(\overline{\rho})$. Applying Corollary \ref{para} in the Appendix, we see that there exist a projective birational morphism
\[f: \widetilde{X}_{\Omega, \textbf{h}}(\overline{\rho}) \lra X_{\Omega, \textbf{h}}(\overline{\rho}),\]
a Zariski closed subset $Z \subset \widetilde{X}_{\Omega, \textbf{h}}(\overline{\rho})$ disjoint from $f^{-1}(U_{\Omega, \textbf{h}}(\overline{\rho}))$ and line bundles $\cL_i$ over $\sU:=\widetilde{X}_{\Omega, \textbf{h}}(\overline{\rho})\setminus Z$ such that $D_{\sU}:=f^* D_{\rig}(\rho_{X_{\Omega, \textbf{h}}(\overline{\rho})}^{\univ})|_{\sU}$ admits an increasing filtration $\Fil_{\bullet}D_{\sU}$ by $(\varphi, \Gamma)$-submodules over $\cR_{\sU,L}$ such that 
\begin{itemize}
	\item[(1)] $\Fil_i D_{\sU}$ are direct summands of $D_{\sU}$ as $\cR_{\sU,L}$-modules; 
	\item[(2)] one has embeddings $\gr_i D_{\sU} \otimes \cL_i \hooklongrightarrow \Delta_{\Omega_i, \sU} \otimes_{\cR_{\sU,L}} \cR_{\sU,L}((\chi_{i,\sU})_{\varpi_L}z^{\textbf{h}_{s_i}})$ where $\Delta_{\Omega_i, \sU}$, $\chi_{i,\sU}$ are the pull-backs of $\Delta_{\Omega_i,X_{\Omega, \textbf{h}}(\overline{\rho})}$, $\chi_{i,X_{\Omega, \textbf{h}}(\overline{\rho})}^{\univ}$ over $\sU$ respectively; 
	\item[(3)] for all $x\in \sU$, the above embedding restrict to injections $(\gr_i D_{\sU})_x \hookrightarrow \Delta_{\Omega_i,f(x)} \otimes_{\cR_{k(x),L}} \cR_{k(x),L}(\chi_{i,f(x),\varpi_L}z^{\textbf{h}_{s_i}})$ and $(\gr_i D_{\sU})_x \otimes_{\cR_{k(x),L}} \cR_{k(x),L}(\chi_{i,f(x),\varpi_L}^{-1})$ is de Rham of Hodge-Tate weights $\textbf{h}^i=(\textbf{h}_{s_{i-1}+1}, \dots, \textbf{h}_{s_i})$.
\end{itemize}
Note that, by the proof of Theorem \ref{KPXg}, we can and do assume that $f$ factors through a surjective birational morphism $\widetilde{X}_{\Omega, \textbf{h}}(\overline{\rho}) \twoheadrightarrow X_{\Omega, \textbf{h}}(\overline{\rho})'$ where $X_{\Omega, \textbf{h}}(\overline{\rho})'$ denotes the normalization of $X_{\Omega, \textbf{h}}(\overline{\rho})$.

Let $\sU_0$ be the preimage of $\sZ^{\gen}$ via the natural composition
\begin{equation*}
	\sU \lra \widetilde{X}_{\Omega, \textbf{h}}(\overline{\rho}) \lra X_{\Omega, \textbf{h}}(\overline{\rho}) \xlongrightarrow{\omega} \sZ.
\end{equation*}
Note that $\sU_0$ is Zariski-open in $\sU$ hence also in $\widetilde{X}_{\Omega, \textbf{h}}(\overline{\rho})$. It is also clear that $f^{-1}(U_{\Omega, \textbf{h}}(\overline{\rho}))\subset \sU_0$ (in fact $U_{\Omega, \textbf{h}}(\overline{\rho})$ is equal to the set of rigid analytic points of $f(\sU_0)$). Let $\sU_1$ be the fibre product over $\sU$ for all $i$ of the $\bG_m^{\rig}$-torsors trivializing the line bundle $\cL_i$. Let $x\in \sU$, and $\alpha_i\in E^{\times}$ (enlarging $E$ if necessary) for $i=1,\dots, r$ such that $\gr_i (D_{\sU})_x \otimes_{\cR_{E,L}} \cR_{E,L}(\chi_{i,f(x),\varpi_L}^{-1} \unr(\alpha_i^{-1}))$ is \'etale. Let $\sU_1(\ul{\alpha})^{\adm}$ be the maximal adic open subset of $\sU_1$ such that there exist a vector bundle $\varrho_{i,\sU_1(\ul{\alpha})^{\adm}}$ over $\sU_1(\ul{\alpha})^{\adm}$ and a continuous morphism $\Gal_L \ra \Aut_{\co_{\sU_1(\ul{\alpha})^{\adm}}}(\varrho_{i,\sU_1(\ul{\alpha})^{\adm}})$ satisfying (cf.\ \cite[Thm.\ 1.3]{Hel16}):
\begin{equation*}
D_{\rig}(\varrho_{i,\sU_1(\ul{\alpha})^{\adm}}) \cong \big(\gr_i D_{\sU_1} \otimes_{\cR_{\sU_1,L}} \cR_{\sU_1,L}(\chi_{i,\sU_1, \varpi_L}^{-1} \unr(\alpha_i^{-1}))\big)|_{\sU_1(\ul{\alpha})^{\adm}}.
\end{equation*}
Since $f$ is projective and $X_{\Omega, \textbf{h}}(\overline{\rho})$ is quasi-separated, $\widetilde{X}_{\Omega, \textbf{h}}(\overline{\rho})$ is quasi-separated. We then deduce that both $\sU_1$ and $\sU_1(\ul{\alpha})^{\adm}$ are also quasi-separated, and hence that $\sU_1(\ul{\alpha})^{\adm}$ is a rigid analytic space by \cite[(1.1.11)]{Hub}. It is also clear that any preimage of $x$ lies in $\sU_1(\ul{\alpha})^{\adm}$. Let $\sU_1(\ul{\alpha})^{\adm, \square}$ be the fibre product over $\sU_1(\ul{\alpha})^{\adm}$ for all $i$ of the $\GL_{n_i}$-torsors of the trivialization of $\varrho_{i,\sU_1(\ul{\alpha})^{\adm}}$ as $\co_{\sU_1(\ul{\alpha})^{\adm}}$-module. Let $\sU(\ul{\alpha})^{\adm, \square}\subset \sU_1(\ul{\alpha})^{\adm, \square}$ be the admissible open subset of points such that
\[\varrho_{i,\sU_1(\ul{\alpha})^{\adm, \square}}: \Gal_L \lra \GL_{n_i}\big(\Gamma(\sU(\ul{\alpha})^{\adm, \square}, \co_{\sU_1(\ul{\alpha})^{\adm, \square}})\big)\]
 has image in $\Gamma(\sU(\ul{\alpha})^{\adm, \square}, \co_{\sU_1(\ul{\alpha})^{\adm, \square}}^+)$. Finally, fix some continuous $\overline{\varrho}_i: \Gal_L \ra \GL_{n_i}(k_E)$ and let $\sU(\ul{\alpha}, \{\overline{\varrho}_i\})^{\adm, \square}\subset \sU(\ul{\alpha})^{\adm, \square}$ be the open locus such that, for $z\in \sU(\ul{\alpha}, \{\overline{\varrho}_i\})^{\adm, \square}$, $\varrho_{i,\sU(\ul{\alpha}, \{\overline{\varrho}_i\})^{\adm, \square}} \otimes \co_{k(z)}$ has reduction (modulo the maximal ideal of $\co_{k(z)}$) equal to $\overline{\varrho}_i$ for all $i$. From the universal property of $\cS_{\overline{\rho}}(\Omega,\textbf{h},\ul{\alpha}, \{\overline{\varrho}_i\})$, we obtain a natural morphism
\begin{equation*}
	\sU(\ul{\alpha}, \{\overline{\varrho}_i\})^{\adm, \square} \lra \cS_{\overline{\rho}}(\Omega,\textbf{h},\ul{\alpha}, \{\overline{\varrho}_i\}).
\end{equation*} 
Moreover, using the universal property of $(\Spf R_{\overline{\rho}})^{\rig} \times \widehat{\cZ_{0,L}}$ and the ``universal" property of $(\Spec \cZ_{\Omega})^{\rig}$ as in \cite[Prop.\ 4.3]{CEGGPS}, we see that the following diagram commutes
\begin{equation*}
	\begin{CD}
		\sU(\ul{\alpha}, \{\overline{\varrho}_i\})^{\adm, \square} @>>> \cS_{\overline{\rho}}(\Omega,\textbf{h},\ul{\alpha}, \{\overline{\varrho}_i\}) \\
		@VVV @VVV \\
		X_{\Omega, \textbf{h}}(\overline{\rho}) @>>> (\Spf R_{\overline{\rho}})^{\rig} \times \sZ.
	\end{CD}
\end{equation*}
By (\ref{smCov1}) and the fact that $X_{\Omega, \textbf{h}}(\overline{\rho})$ is a reduced closed subspace of $ (\Spf R_{\overline{\rho}})^{\rig} \times \sZ$, we deduce a commutative diagram:
\begin{equation}\label{sUalpharho}
\begin{gathered}
	\begindc{\commdiag}[200]
	\obj(0,2)[a]{$\sU(\ul{\alpha}, \{\overline{\varrho}_i\})^{\adm, \square} $}
	\obj(6,2)[b]{$ \cS_{\overline{\rho}}(\Omega,\textbf{h},\ul{\alpha}, \{\overline{\varrho}_i\})$}
	\obj(6,0)[c]{$X_{\Omega, \textbf{h}}(\overline{\rho}).$}
	\mor{a}{b}{}[0,\solidarrow]
	\mor{b}{c}{$\kappa_{\ul{\alpha}, \{\overline{\varrho}_i\}}$}[0,\solidarrow]
	\mor{a}{c}{}[0,\solidarrow]
	\enddc
\end{gathered}
\end{equation}
By Corollary \ref{para}, there exists a Zariski-open and Zariski-dense subset $\sV \subset X_{\Omega, \textbf{h}}(\overline{\rho})$ disjoint from $f(Z)$ such that the morphism $f$ induces an isomorphism $f^{-1}(\sV) \xrightarrow{\sim} \sV$. Consider
\[\sV(\ul{\alpha}, \{\overline{\varrho}_i\})^{\adm, \square}:= f^{-1}(\sV) \times_{\sU} \sU(\ul{\alpha}, \{\overline{\varrho}_i\})^{\adm, \square}.\]
Note that $\sV(\ul{\alpha}, \{\overline{\varrho}_i\})^{\adm, \square}$ can also be constructed from $f^{-1}(\sV)$ in the same way $\sU(\ul{\alpha}, \{\overline{\varrho}_i\})^{\adm, \square}$ was constructed from $\cU$. Using the same argument as in the first paragraph on page 1598 of \cite{BHS1}, one can prove an isomorphism
\[\sV(\ul{\alpha}, \{\overline{\varrho}_i\})^{\adm, \square} \xlongrightarrow{\sim} \kappa_{\ul{\alpha}, \{\overline{\varrho}_i\}} ^{-1}(\sV)\]
hence by (\ref{dimSrho}), $\sV(\ul{\alpha}, \{\overline{\varrho}_i\})^{\adm, \square}$ is smooth of dimension $(r+\sum_{i=1}^r n_i^2 + n^2) + [L:\Q_p] (\frac{n(n-1)}{2} +r)$. 
By \cite[Thm.\ 0.2]{KL}, for any $x\in \sV$, there exist an admissible open neighbourhood $\sV_x$ of $x$ in $\sV$ and $\alpha_i$, $\overline{\varrho}_i$ such that $\sV_x$ is contained in the image of $\sV(\ul{\alpha}, \{\overline{\varrho}_i\})^{\adm, \square} \ra \sV$ (for the corresponding $\ul{\alpha}$, $\{\overline{\varrho}_i\}$). As the morphism $\sV(\ul{\alpha}, \{\overline{\varrho}_i\})^{\adm, \square} \ra \sV$ is smooth of relative dimension $r+\sum_{i=1}^r n_i^2$, we deduce that $\sV$ is equidimensional of dimension $n^2+[L:\Q_p](\frac{n(n-1)}{2}+r)$. Since $\sV$ is Zariski dense in $X_{\Omega, \textbf{h}}(\overline{\rho})$, (1) of Theorem \ref{DFOL} follows.

\begin{proposition}\label{smoothi}
	The morphism $\kappa_{\ul{\alpha}, \{\varrho_i\}}: \cS_{\overline{\rho}}(\Omega, \textbf{h}, \ul{\alpha}, \{\overline{\varrho}_i\}) \ra X_{\Omega, \textbf{h}}(\overline{\rho})$ is smooth. 
\end{proposition}
\begin{proof}
Let $x\in \cS_{\overline{\rho}}(\Omega, \textbf{h}, \ul{\alpha}, \{\overline{\varrho}_i\})$, $y:=\kappa_{\ul{\alpha}, \{\overline{\varrho}_i\}}(x)$, $B:=\widehat{\co}_{\cS_{\overline{\rho}}(\Omega, \textbf{h}, \ul{\alpha},\{\overline{\varrho}_i\}),x}$, $A:=\widehat{\co}_{X_{\Omega, \textbf{h}}(\overline{\rho}),y}$. Let $N:=\sum_{i=1}^r n_i^2 +r$, it is enough to show that there exist $x_1, \dots, x_{N}\in B$ such that $B\cong A[[x_1, \dots, x_N]]$. Let $\rho_x: \Gal_L \ra \GL_n(k(x))$ be the image of $x$ in $(\Spf R_{\overline{\rho}})^{\rig}$. Recall we have a natural isomorphism $\widehat{\co}_{(\Spf R_{\overline{\rho}})^{\rig}, \rho_x}\cong R_{\rho_x}$ where $R_{\rho_x}$ is the universal framed deformation ring of $\rho_x$. Let $\sF$ be the associated filtration on $D_{\rig}(\rho_x)$ (as in (\ref{func1})), and denote by $R_{\rho_x, \sF}^{0}$ the local complete $k(x)$-algebra which (pro-)represents the functor $F_{\rho_x,\sF}^{0}$ (see Corollary \ref{apxthmDF} (1)). By Corollary \ref{apxthmDF} and Remark \ref{remApxDF} (note that, as $x$ is sent to $\sZ^{\gen}$, the hypothesis there are satisfied for $\sF$ by Lemma \ref{lem:generic}), $R_{\rho_x, \sF}^{0}$ is a quotient of $R_{\rho_x}$ and is formally smooth of dimension $n^2+[L:\Q_p](\frac{n(n-1)}{2}+r)$. We have a natural morphism
	\begin{equation}\label{equ: deftoB}
		R_{\rho_x,\sF}^{0} \lra B.
	\end{equation}
Indeed, for an ideal $I\subset \fm_B$ (the maximal ideal of $B$) with $\dim_{k(x)}B/I<\infty$, by the construction of $\cS_{\overline{\rho}}(\Omega,\textbf{h}, \ul{\alpha},\ul{\varrho})$ and by Lemma \ref{twDR}, we have a deformation of $(\rho_x,\sF)$ over $B/I$ lying in $F_{\rho_x,\sF}^{0}(B/I)$. By the universal property of $R_{\rho_x,\sF}^{0}$, this gives a natural morphism $R_{\rho_x,\sF}^{0} \ra B/I$. Taking the projective limit over all ideals $I$, we obtain (\ref{equ: deftoB}). Let $\varrho_{i,x}$ be the image of $x$ in $(\Spf R_{\overline{\varrho}_i}^{\pcr}(\xi_i,\textbf{h}^i))^{\rig}$. Using the fact that $\widehat{\co}_{(\Spf R_{\overline{\varrho}_i}^{\pcr}(\xi_i,\textbf{h}^i))^{\rig}, \varrho_{i,x}}$ (pro-)represents the functor of framed de Rham deformations of $\varrho_{i,x}$ and using Lemma \ref{twDR}, one can show that the tangent map $(\fm_B/\fm_B^2)^{\vee} \ra F_{\rho_x,\sF}^{0}(k(x)[\varepsilon]/\varepsilon^2)$ of (\ref{equ: deftoB}) is surjective (we leave the details to the reader). Together with the fact that both $B$ and $R_{\rho_x, \sF}^{0}$ are formally smooth, we deduce that (\ref{equ: deftoB}) is formally smooth of relative dimension $N$. There exist thus $x_1, \dots, x_N\in B$, such that 
	\begin{equation*}
		B\cong R_{\rho_x, \sF}^{0}[[x_1,\dots, x_N]].
	\end{equation*}
Since $R_{\rho_x, \sF}^{0}$ is a quotient of $R_{\rho_x}$, we deduce from the above isomorphism a surjective morphism $R_{\rho_x}[[x_1, \dots, x_N]]\twoheadrightarrow B$. The morphism $X_{\Omega, \textbf{h}}(\overline{\rho}) \ra (\Spf R_{\overline{\rho}})^{\rig}$ induces a morphism $R_{\rho_x} \ra A$. Using the commutative diagram (\ref{smCov1}), we see that the morphism $R_{\rho_x} \ra B$ factors through $A$. To sum up, we have obtained a surjective morphism
	\begin{equation*}
		A[[x_1, \dots, x_N]] \twoheadlongrightarrow B.
	\end{equation*}
Since $\dim A=\dim B-N$ and $B$ is formally smooth, the above morphism is an isomorphism if $A$ is integral. But this follows from exactly the same argument as in the first paragraph on page 1599 of \cite{BHS1} with $X_{\tri}^{\square}(\overline{r})$, $U^{\square}$, $\cS^{\square}(\overline{r})$ of \textit{loc.\ cit.} replaced by $X_{\Omega, \textbf{h}}(\overline{\rho})$, $\sU(\ul{\alpha},\{\overline{\varrho_i}\})^{\adm, \square}$ and $\cS_{\overline{\rho}}(\Omega, \textbf{h},\ul{\alpha}, \{\overline{\varrho_i}\})$ respectively.
\end{proof}

We use Proposition \ref{smoothi} to prove (2) of Theorem \ref{DFOL} (following the strategy in the proof of \cite[Thm.\ 2.6]{BHS1}). We also need to use adic spaces. By \cite[Prop.\ 1.7.8]{Hub}, $\Ima(\kappa_{\ul{\alpha}, \{\overline{\varrho}_i\}})$ is an adic open subset of (the adic space associated to) $X_{\Omega, \textbf{h}}(\overline{\rho})$ with rigid analytic points contained in $U_{\Omega, \textbf{h}}(\overline{\rho})$. Letting $\ul{\alpha}$, $\{\overline{\varrho}_i\}$ vary, the union of the $\Ima(\kappa_{\ul{\alpha}, \{\overline{\varrho}_i\}})$ is also an adic open subset $U$ of $X_{\Omega, \textbf{h}}(\overline{\rho})$. But it is easy to see that any point of the rigid space $U_{\Omega, \textbf{h}}(\overline{\rho})$ lies in $\Ima(\kappa_{\ul{\alpha}, \{\overline{\varrho}_i\}})$ for some $\ul{\alpha}$ and $\{\overline{\varrho}_i\}$, hence $U_{\Omega, \textbf{h}}(\overline{\rho})$ coincides with the rigid analytic points of the adic open subset $U$ of $X_{\Omega, \textbf{h}}(\overline{\rho})$. We show that $U$ is a Zariski-constructible subset (see \cite[Lemma 2.13]{BHS1}). Since $\sU_0$ is Zariski-open in $\widetilde{X}_{\Omega, \textbf{h}}(\overline{\rho})$ and $f$ is projective, we see by \cite[Lemma 2.14]{HeSc} that the set $f(\sU_0)$ is Zariski-constructible in (the adic space associated to) $X_{\Omega, \textbf{h}}(\overline{\rho})$. We claim $U=f(\sU_0)$. Indeed, both sets have the same rigid analytic points, i.e.\ those in $U_{\Omega, \textbf{h}}(\overline{\rho})$. Using \cite[Lemma 2.15]{HeSc}, the inclusion of adic spaces $\kappa_{\ul{\alpha}, \{\overline{\varrho}_i\}}^{-1}(f(\sU_0))\subseteq \cS_{\overline{\rho}}(\Omega, \textbf{h}, \ul{\alpha}, \{\overline{\varrho}_i\})$ is an isomorphism, which implies $\Ima(\kappa_{\ul{\alpha}, \{\overline{\varrho}_i\}})\subseteq f(\sU_0)$ (as subsets of the adic space associated to $X_{\Omega, \textbf{h}}(\overline{\rho})$) and hence $U\subset f(\sU_0)$. On the other hand, using (\ref{sUalpharho}) and the fact that the image of $\sU(\ul{\alpha}, \{\overline{\varrho}_i\})^{\adm, \square} $ in $\sU_0$ forms an adic open covering of $\sU_0$ when $\ul{\alpha}$, $\{\overline{\varrho}_i\}$ vary, we can deduce $f(\sU_0) \subset U$. Hence we see that the adic open subset $U$ of $X_{\Omega, \textbf{h}}(\overline{\rho})$ is Zariski-constructible. It then follows from \cite[Lemma 2.13]{HeSc} that $U$ is Zariski-open in the adic space associated to $X_{\Omega, \textbf{h}}(\overline{\rho})$, hence {\it a fortiori} $U_{\Omega, \textbf{h}}(\overline{\rho})$ is Zariski-open in the rigid space $X_{\Omega, \textbf{h}}(\overline{\rho})$. This concludes the proof of (2) of Theorem \ref{DFOL}. 

Finally, by \cite[Lemma 5.8]{BHS1} (applied to $\cS_{\overline{\rho}}(\Omega,\textbf{h},\ul{\alpha}, \{\overline{\varrho}_i\}) \ra U_{\Omega, \textbf{h}}(\overline{\rho}) \ra \sZ$), to show (3) of Theorem \ref{DFOL}, i.e.\ the smoothness of the morphism $\omega: U_{\Omega, \textbf{h}}(\overline{\rho}) \ra \sZ$, it is sufficient to show that the morphism $\cS_{\overline{\rho}}(\Omega,\textbf{h},\ul{\alpha}, \{\overline{\varrho}_i\}) \ra \sZ$ is smooth. From the construction of $\cS_{\overline{\rho}}(\Omega,\textbf{h},\ul{\alpha}, \{\overline{\varrho}_i\})$, this is a consequence of the following proposition.

\begin{proposition}\label{smooth11}
The morphism $\eta: (\Spf R_{\overline{\varrho}_i}^{\pcr}(\xi_i, \textbf{h}^i))^{\rig} \ra (\Spec \cZ_{\Omega_i})^{\rig}$ of \cite[Prop.\ 4.3]{CEGGPS} which sends a deformation $\varrho_i$ to $\ttr(\varrho_i)$ is smooth.
\end{proposition}
\begin{proof}
Since both source and target of $\eta$ are smooth, we only need to show that the tangent map $d \eta_x: T_{ (\Spf R_{\overline{\varrho}_i}^{\pcr}(\xi_i, \textbf{h}^i))^{\rig}, x} \ra T_{(\Spec \cZ_{\Omega_i})^{\rig}, \eta(x)}$ is surjective for any $x\in (\Spf R_{\overline{\varrho}_i}^{\pcr}(\xi_i, \textbf{h}^i))^{\rig}$, where $T_{Y,z}$ denotes the tangent space of $Y$ at $z$ for a point $z$ of a rigid space $Y$. For $x\in (\Spf R_{\overline{\varrho}_i}^{\pcr}(\xi_i, \textbf{h}^i))^{\rig}$, let $\varrho_{i,x}$ be the associated $\Gal_L$-representation and $\Delta_{i,\eta(x)}$ be the $p$-adic differential equation associated to $\eta(x)$. For a $(\varphi, \Gamma)$-module $D$, we put $D(\textbf{h}_{s_i}):=D \otimes_{\cR_{E,L}} \cR_{E,L}(z^{\textbf{h}_{s_i}})$. One easily computes for any point $x$
	\begin{equation*}
		\dim_{k(x)} H^i_{(\varphi, \Gamma)}\big(D_{\rig}(\varrho_{i,x})^{\vee} \otimes_{\cR_{k(x),L}} \Delta_{i,\eta(x)}(\textbf{h}_{s_i})\big)=
		\begin{cases}1 & i=0\\
			n_i^2[L:\Q_p] +1 & i=1 \\
			0 & i=2.
		\end{cases}
	\end{equation*}
By Lemma \ref{lemlfc}, we deduce in particular that $\Hom_{(\varphi, \Gamma)}(D_{\rig}(\varrho_i^{\univ}), \eta^* \Delta_{\Omega_i}(\textbf{h}_{s_i}))$ is an invertible sheaf over $(\Spf R_{\overline{\varrho}_i}^{\pcr}(\xi_i, \textbf{h}^i))^{\rig}$. Now fix $x\in (\Spf R_{\overline{\varrho}_i}^{\pcr}(\xi_i, \textbf{h}^i))^{\rig}$ and let
\[\psi \in \Spec k(x)[\varepsilon]/\varepsilon^2 \lra (\Spec \cZ_{\Omega_i})^{\rig}\]
be an element in $T_{(\Spec \cZ_{\Omega_i})^{\rig}, \eta(x)}$. Let $\Delta_{\psi}:= \psi^* \Delta_{\Omega_i}$. By the proof of Proposition \ref{apxfm0}, there exists a deformation $\widetilde{\varrho}_{i,x}$ of $\varrho_{i,x}$ over $k(x)[\varepsilon]/\varepsilon^2$ such that one has an embedding of $(\varphi, \Gamma)$-modules over $\cR_{\Spec k(x)[\varepsilon]/\varepsilon^2,L}$:
\[\jmath: D_{\rig}(\widetilde{\varrho}_{i,x}) \hookrightarrow \Delta_{\psi}(\textbf{h}_{s_i}).\]
As $\Delta_{\psi}$ is de Rham, so is $\widetilde{\varrho}_{i,x}$ (recall they have the same rank over $\cR_{k(x),L}$). Thus $\widetilde{\varrho}_{i,x}$ corresponds to an element
	\[\psi': \Spec k(x)[\varepsilon]/\varepsilon^2 \lra (\Spf R_{\overline{\varrho}_i}^{\pcr}(\xi_i, \textbf{h}^i))^{\rig}\]
in the tangent space at $x$. Let us prove that $d \eta_x$ maps $\psi'$ to $\psi$, or equivalently $\eta \circ \psi'=\psi$. Consider $\Delta_{\eta \circ \psi'}:=(\eta\circ \psi')^* \Delta_{\Omega_i}(\textbf{h}_{s_i})$. Using Lemma \ref{lemlfc}, a local generator of the invertible sheaf $\Hom_{(\varphi, \Gamma)}(D_{\rig}(\varrho_i^{\univ}), \eta^* \Delta_{\Omega_i}(\textbf{h}_{s_i}))$ induces by pull-back via $\psi'$ a morphism
	\[\iota: D_{\rig}(\widetilde{\varrho}_{i,x}) \lra (\eta\circ \psi')^* \Delta_{\Omega_i}(\textbf{h}_{s_i})\]
which is a generator of the $k(x)[\varepsilon]/\varepsilon^2$-module $\Hom_{(\varphi, \Gamma)}(D_{\rig}(\widetilde{\varrho}_{i,x}),\Delta_{\eta \circ \psi'})$. Since $D_{\rig}(\varrho_{i,x})$ is irreducible, it is not difficult to see that $\iota$ has to be injective. From the two injections $\jmath$, $\iota$, by comparing the Hodge-Tate weights and using \cite[Thm.~A]{Ber08a}, we can deduce $\Delta_{\eta \circ \psi'}\cong \Delta_{\psi}$. By the discussion in \S~\ref{sec_pDf}, we see that there exists a bijection $T_{(\Spec \cZ_{\Omega_i})^{\rig},\eta(x)} \xrightarrow{\sim} k(x)$, $f \mapsto a_f$ such that $f^* \Delta_{\Omega_i}\cong \Delta_{i,\eta(x)} \otimes_{\cR_{k(x),L}} \cR_{k(x)[\varepsilon]/\varepsilon^2,L}(\unr(1+a_f\varepsilon))$. Hence, for $*\in \{\psi, \eta\circ \psi' \}$, there is $a_*\in k(x)$ such that $\Delta_{*}\cong \Delta_{i,\eta(x)} \otimes_{\cR_{k(x),L}} \cR_{k(x)[\varepsilon]/\varepsilon^2,L}(\unr(1+a_*\varepsilon))$. By the proof of Lemma \ref{subf}, $\Delta_{\psi}\cong \Delta_{\eta\circ \psi'}$ implies $\unr(1+a_{\psi}\varepsilon)=\unr(1+a_{\eta \circ \psi'}\varepsilon)$, which further implies $a_{\psi}=a_{\eta \circ \psi'}$ and hence $\eta \circ \psi'=\psi$. The proposition follows.
\end{proof}

We end this paragraph by the following proposition on Sen weights which will be used in \S~\ref{secGrDV}.

\begin{proposition}\label{Senwt}
Let $x=(\rho, \ul{x}, \chi)\in X_{\Omega, \textbf{h}}(\overline{\rho})$. Then for $\tau \in \Sigma_L$, the set $\{\wt(\chi_i)_{\tau} + h_{j_i, \tau}\ |\ i=1,\dots, r, \ j_i=s_{i-1}+1, \dots, s_i\}$ is the set of the Sen $\tau$-weights of $\rho$.
\end{proposition}
\begin{proof}
Since $U_{\Omega, \textbf{h}}(\overline{\rho})$ is Zariski-open and Zariski-dense in $X_{\Omega, \textbf{h}}(\overline{\rho})$ and the Sen $\tau$-weights are analytic functions on $X_{\Omega, \textbf{h}}(\overline{\rho})$ (cf.\ \cite[Def. 6.2.11]{KPX}), we only need to prove the statement for points in $U_{\Omega, \textbf{h}}(\overline{\rho})$. Since any point of $U_{\Omega, \textbf{h}}(\overline{\rho})$ lies in the image of $\kappa(\ul{\alpha}, \{\overline{\varrho}_i\})$ for certain $\ul{\alpha}$, $\{\overline{\varrho}_i\}$, using the commutative diagram (\ref{smCov1}) we are reduced to prove the statement for points in $\cS_{\overline{\rho}}(\Omega, \textbf{h}, \ul{\alpha}, \{\overline{\varrho_i}\})$. For such a point $x=(\rho, \ul{\varrho}, \chi, \Fil_{\bullet}, \ul{\nu})$, by definition, the filtration $\Fil_{\bullet} D_{\rig}(\rho)$ satisfies $\gr_i D_{\rig}(\rho) \cong D_{\rig}(\varrho_{i,x}) \otimes_{\cR_{k(x),L}} \cR_{k(x),L}(\chi_i \unr(\alpha_i))$. We see that $\{\wt(\chi_i)_{\tau}+h_{j,\tau}\}_{j=s_{i-1}+1, \dots, s_i}$ is the set of Sen $\tau$-weight of $\gr_i D_{\rig}(\rho)$. The proposition follows. 
\end{proof}

\subsection{Potentially crystalline deformation spaces}\label{secPCD}

We study a variant of potentially crystalline deformation spaces, and we show it admits a cell decomposition with respect to Schubert cells of $\GL_n/P$. By studying the embeddings of the cells into our (various) Bernstein paraboline varieties, we prove the existence of local companion points on the Bernstein paraboline varieties for generic potentially crystalline representations.

We first recall some facts on inertial types. Let $d\in \Z_{\geq 1}$, $\xi: I_L \ra \GL_d(E)$ a cuspidal inertial type and $\ttr_d$ an absolutely irreducible Weil-Deligne representation over $E$ such that $\ttr_d|_{I_L}\cong \xi$. Assume $E$ contains all $d$-th roots of unity $\mu_d$. Let $L'$ be a finite extension of $L$ such that the action of $I_L$ on $\xi$ factors through the inertia subgroup $I(L'/L)\subset \Gal(L'/L)$. Using $\ttr_d \cong \ttr_d \otimes_E \unr(\alpha) \Rightarrow \wedge^d \ttr_d \otimes_E \unr(\alpha^d)=\wedge^d \ttr_d$, there exists $d_0|d$ such that $\{\alpha \in E^{\times}\ | \ \ttr_d \cong \ttr_d \otimes_E \unr(\alpha)\}=\mu_{d_0}$. Note that $d_0$ only depends on the inertial type $\xi$.

Let $\xi_0\subseteq \xi$ be an absolutely irreducible subrepresentation, and fix $F\in W_L$ a lifting of the arithmetic Frobenius. We denote by $F(\xi_0)$ the conjugate of $\xi_0$ by $F$ (so it is an absolutely irreducible representation of $I(L'/L)$). The following lemma follows from \cite[Lemma 4.4]{CEGGPS}.

\begin{lemma}\label{iner0}
	The integer $d_0$ is the minimal positive integer such that $F^{d_0}(\xi_0)\subseteq \xi_0$. Moreover, we have $\xi \cong \oplus_{i=0}^{d_0-1} F^i(\xi_0)$ as $I(L'/L)$-representations and the $F^i(\xi_0)$, $i=0, \dots, d_0-1$ are pairwise non-isomorphic. 
\end{lemma}

We fix a basis $\ul{e}=(\ul{e}_0, \dots, \ul{e}_{d_0-1})$ of $\xi$ such that $\ul{e}_i$ is a basis of $\xi_i:=F^i(\xi_0)$. Let $\ttr_1$, $\ttr_2$ be two Weil-Deligne representations of inertial type $\xi$ and fix $I(L'/L)$-equivariant isomorphisms for $i\in \{1,2\}$:
\begin{equation}\label{basisIner}
	E\ul{e} =\xi \xlongrightarrow{\sim} \ttr_i. 
\end{equation}
Consider the operators $\ttr_i(F)$ acting on (the underlying vector space of) $\xi$ via the above isomorphisms. Then $\ttr_2(F) \circ \ttr_1(F)^{-1}: \xi \ra \xi$ is $I(L'/L)$-equivariant, hence $\ttr_2(F) \circ \ttr_1(F)^{-1}$ preserves $\xi_i$ for $i=0, \dots, d_0-1$ and is equal to a scalar $\alpha_i\in E^\times$ when restricted to $\xi_i$ (note that $\alpha_i$ depends on the choices of the isomorphisms (\ref{basisIner})). Let $\alpha:=\prod_{i\in \Z/d_0} \alpha_i$, we have $\ttr_2(F^{d_0}) \circ \ttr_1(F^{d_0})^{-1}=\alpha$ on each $\xi_i$ hence on $\xi$. Moreover $\alpha$ is independent of the choice of the isomorphisms (\ref{basisIner}).

\begin{lemma}\label{inert2}
Let $\beta\in E$ be a $d_0$-th root of $\alpha$ (enlarging $E$ if necessary), then $\ttr_2\cong \ttr_1 \otimes_E \unr(\beta)$.
\end{lemma}
\begin{proof}
Since $\ttr_2$ and $\ttr_1$ have the same cuspidal inertial type, there exists $\beta'$ such that $\ttr_2\cong \ttr_1 \otimes_E \unr(\beta')$. Note that $\beta'$ is unique up to multiplication by an element in $\mu_{d_0}$. It follows that $\ttr_2(F^{d_0})=\ttr_1(F^{d_0}) (\beta')^{d_0}$, hence $(\beta')^{d_0}=\alpha$. 
\end{proof}

We now go back to the setting of \S~\ref{s: DO}. We let $\Omega_0$ be the Bernstein component of $\GL_n(L)$ associated to the cuspidal Bernstein component $\Omega$ of $L_P(L)$, i.e.\ $\Omega_0$ is the Bernstein component with cuspidal type $(L_P(L), \pi_{L_P}\cong \boxtimes_{i=1}^r \pi_i)$. Let $\cZ_{\Omega_0}$ be the centre of $\Omega_0$, and $\xi_i$ be the inertial type of $\Omega_i$ for $i=0,\dots, r$ (note that $\xi_0$ is different here from the $\xi_0$ of Lemma \ref{iner0}). We have thus $\xi_0\cong \oplus_{i=1}^r \xi_i$. Moreover, by \cite[Prop.\ 4.1]{Dat99} (see also \cite[\S~3.6]{CEGGPS}), there is a one-to-one correspondence between closed points of $\Spec \cZ_{\Omega_0}$ and semi-simple Weil-Deligne representations $\ttr$ of inertial type $\xi_0$ (so $N=0$ on $\ttr$ and $\ttr \cong \oplus_{i=1}^r \ttr_i$ for certain absolutely irreducible Weil-Deligne representations $\ttr_i$ of inertial type $\xi_i$). Let $\sW_{\Omega}:=\{w\in S_r\ |\ \Omega_i=\Omega_{w^{-1}(i)},\ \forall\ i=1,\dots, r\}$. By \cite[Prop.\ 2.1]{Dat99}, there is a natural isomorphism of $E$-algebras
\begin{equation*}
	\cZ_{\Omega}^{\sW_{\Omega}} \xlongrightarrow{\sim} \cZ_{\Omega_0}.
\end{equation*}
On closed points, the corresponding map $\Spec \cZ_{\Omega} \ra (\Spec \cZ_{\Omega})/\sW_{\Omega} \cong \Spec \cZ_{\Omega_0}$ sends the $r$-tuple $(\ttr_1,\dots,\ttr_r)$ to $\oplus_{i=1}^r \ttr_i$. 

Let $\{\widetilde{\xi}_1, \dots, \widetilde{\xi}_s\}$ be the set of isomorphic classes of $\{\xi_1, \dots, \xi_r\}$, and for $j\in \{1,\dots,s\}$ let $m_j:=|\{i \ |\ \xi_i \cong \widetilde{\xi}_j\}|$. Thus $\sum_{j=1}^s m_j=r$, $\sW_{\Omega}\cong \prod_{j=1}^s S_{m_j}$, and we have an isomorphism $\xi_0\cong \oplus_{j=1}^s \widetilde{\xi}_j^{\oplus m_j}$. By Lemma \ref{iner0}, we have a decomposition $\widetilde{\xi}_j \cong \oplus_{k\in \Z/d_j\Z} \widetilde{\xi}_{j,k}$ for some integer $d_j\geq 1$ dividing $\dim_E \widetilde{\xi}_j$, where $\widetilde{\xi}_{j,k}:=F^k(\widetilde{\xi}_j)$ for $j=1, \dots, s$ and $k\in \Z/d_j\Z$ are pairwise non-isomorphic absolutely irreducible representations of $I(L'/L)$ over $E$ (note that distinct inertial types do not have any common irreducible constituents). Let $f_j:=\dim_E \widetilde{\xi}_{j,0}$ ($=\dim_E \widetilde{\xi}_{j,k}$ for all $k\in \Z/d_j\Z$). Hence $\sum_{j=1}^s m_jd_jf_j=n$, $\sum_{j=1}^s m_j=r$, $d_jf_j=n_i$ if $\xi_i \cong \widetilde{\xi_j}$ and we have a decomposition
\begin{equation}\label{decomptype}
\xi_0\cong \bigoplus_{j=1}^s \Big(\bigoplus_{k\in \Z/d_j\Z} \big(\underbrace{\widetilde{\xi}_{j,k} \oplus \cdots \oplus \widetilde{\xi}_{j,k}}_{m_j}\big)\Big).
\end{equation}
In order to describe a Weil representation of inertial type $\xi_0$ in a more concrete way, we now fix a basis of $\xi_0$ with respect to the decomposition (\ref{decomptype}):
\[\ul{e}=(\ul{e}_{j,k})_{\substack{j=1, \dots, s \\ k \in \Z/d_j\Z}}=(\ul{e}_{j,k,i})_{\substack{j=1, \dots, s \\ k \in \Z/d_j\Z\\ i=1, \dots, m_j}}\]
where each $(\ul{e}_{j,k,i})$ for $i=1, \dots, m_j$ means a choice of a basis on the $f_j$-dimensional $E$-vector space $\widetilde{\xi}_{j,k}$. We choose these basis so that the following conditions are satisfied:
\begin{condition}\label{condIac}
	(1) For $1\leq i, i'\leq m_j$, the $E$-linear map $\widetilde{\xi}_{j,k}\rightarrow \widetilde{\xi}_{j,k}$ sending the basis $\ul{e}_{j,k,i}$ to the basis $\ul{e}_{j,k,i'}$ is $I(L'/L)$-equivariant.
	
	(2) The $E$-linear map $\iota_{j,k,i}:\widetilde{\xi}_{j,k}\rightarrow \widetilde{\xi}_{j,k+1}$ sending the basis $\ul{e}_{j,k,i}$ to the basis $\ul{e}_{j,k+1,i}$ satisfies
	\[\iota_{j,k,i}(g v)=(FgF^{-1}) \iota_{j,k,i}(v)\]
	for all $v\in \widetilde{\xi}_{j,k}$ and $g\in I(L'/L)$. 
\end{condition} 
Condition (1) is equivalent to the fact that, for any element $g\in I(L'/L)$, the matrices of $g$ in the basis $\ul{e}_{j,k,i}$ and $\ul{e}_{j,k,i'}$ are the same. Condition (2) is equivalent to the fact that, for any element $g\in I(L'/L)$, the matrix of $g$ in the basis $\ul{e}_{j,k,i}$ is equal to the matrix of $FgF^{-1}$ in the basis $\ul{e}_{j,k+1,i}$.

We fix a semi-simple Weil-Deligne representation $\ttr_0\cong \oplus_{j=1}^s \widetilde{\ttr}_j^{\oplus m_j}$ such that $\widetilde{\ttr}_j$ is of inertial type $\widetilde{\xi}_j$ for all $j$. So $\ttr_0$ is of inertial type $\xi_0$ and $N=0$ on $\ttr_0$. We fix an isomorphism $\iota_0$ of $I(L'/L)$-representations:
\begin{equation*}
	\iota_0: \langle \ul{e}\rangle=\xi_0 \xlongrightarrow{\sim} \ttr_0
\end{equation*}
where the notation $\langle \ul{e}\rangle$ means the $E$-vector spaces generated by the basis $\underline e$. The $F$-action on $\ttr_0$ then gives an endomorphism $\ttr_0(F)$ on $\xi_0$ which sends $\widetilde{\xi}_{j,k}^{\oplus m_j}$ to $\widetilde{\xi}_{j,k+1}^{\oplus m_j}$. Modifying the isomorphism $\iota_0$ if necessary and by condition (2) in our choice of the basis $\ul{e}$ of $\xi_0$, we can and do assume
\begin{equation}
	\label{choiBa}
	\ul{e}_{j,k+1,i}=\ttr_0(F) \ul{e}_{j,k,i} \text{ for $j=1, \dots, s$ and $k=0, \dots, d_j-2$}. 
\end{equation}

Let $\ttr$ be an arbitrary Weil-Deligne representation of inertial type $\xi_0$ (with $N$ possibly non-zero on $\ttr$), and fix again an $I(L'/L)$-equivariant isomorphism $\langle \ul{e}\rangle=\xi_0 \xlongrightarrow{\sim} \ttr$. The $F$-action on $\ttr$ gives another endomorphism $\ttr(F)$ on $\xi_0$ sending $\widetilde{\xi}_{j,k}^{\oplus m_j}$ to $\widetilde{\xi}_{j,k+1}^{\oplus m_j}$. The endomorphism $\ttr(F) \circ \ttr_0(F)^{-1}:$ of $\xi_0$ is $I(L'/L)$-equivariant, hence preserves each $\widetilde{\xi}_{j,k}^{\oplus m_j}$ for $j=1, \dots, s$ and $k\in \Z/d_j\Z$. Since $\End_{I(L'/L)}(\widetilde{\xi}_{j,k})\cong E$, the restriction of $\ttr(F) \circ \ttr_0(F)^{-1}$ to $\widetilde{\xi}_{j,k}^{\oplus m_j}$ is given in the basis $(\ul{e}_{j,k,i})_{i=1, \dots, m_j}$ by a matrix $A_{j,k}$ which lies in the image $\GL_{m_j}(E \mathrm{Id}_{f_j})$ of 
\begin{equation}\label{eImap}
	\GL_{m_j}(E) \hooklongrightarrow \GL_{m_jf_j}(E), \ (a_{uv})_{1 \leq u, v \leq m_j} \mapsto (a_{uv} \mathrm{Id}_{f_j})_{1 \leq u, v \leq m_j}
\end{equation}
where $\mathrm{Id}_{f_j}\in \GL_{f_j}(E)$ is the identity matrix.
Moreover, for $j=1, \dots, s$ and $k=0, \dots, d_j-2$, we see by (\ref{choiBa}) that $A_{j,k}$ is actually the matrix of the morphism $\ttr(F): \widetilde{\xi}_{j,k}^{\oplus m_j} \ra \widetilde{\xi}_{j,k+1}^{\oplus m_j}$ in the basis $\ul{e}_{j,k}$ and $\ul{e}_{j,k+1}$. Hence the matrix of the endomorphism $\ttr(F) \circ \ttr_0(F)^{-1}$ in the basis $\ul{e}$ is 
\begin{equation}
	\label{matrixW}
	A=\diag\Big(\{A_{j,k}\}_{\substack{j=1, \dots,s \\ k\in \Z/d_j\Z}}\Big)\in \GL_n(E).
\end{equation}
The converse also holds: given a matrix $A'=\diag(\{A'_{j,k}\})$ as in (\ref{matrixW}) with $A'_{j,k}\in \GL_{m_j}(E \mathrm{Id}_{f_j})$, one can associate a Weil-Deligne representation of inertial type $\xi_0$ with $N=0$ by letting $F$ act on the basis $\ul e$ by $A'\circ \ttr_0(F)$.

Now let $B_{j,k}$ be the preimage of $A_{j,k}$ via (\ref{eImap}) and put $B_j:= B_{j,0}B_{j,1} \cdots B_{j,d_j-1}\in \GL_{m_j}(E)$. Using \ again \ (\ref{choiBa}), \ we \ see \ that \ the \ image \ $A_j$ \ of \ $B_j$ \ via \ (\ref{eImap}) \ is \ the \ matrix \ of $\ttr(F^{d_j})\circ \ttr_0(F^{-d_j}) \big|_{\widetilde{\xi}_{j,d_j-1}^{\oplus m_j}}$ in the basis $\ul{e}_{j,d_j-1}$ of $\widetilde{\xi}_{j,d_j-1}^{\oplus m_j}$. Since both $\ttr(F^{d_j})$ and $\ttr_0(F^{-d_j})$ preserve $\widetilde{\xi}_{j,d_j-1}^{\oplus m_j}$ (actually they preserve $\widetilde{\xi}_{j,k}^{\oplus m_j}$ for all $k=0, \dots, d_j-1$), the conjugacy class of $A_j$ is independent of the choice of the basis of $\widetilde{\xi}_{j,d_j-1}^{\oplus m_j}$. We call $\ttr$ \textit{generic} if $\ttr^{\sss}\cong \oplus_{i=1}^r \ttr_i$ satisfies $\ttr_i\neq\ttr_{i'}$ and $\ttr_i \neq \ttr_{i'} \otimes_E \unr(q_L)$ for all $i\neq i'$. In particular, if $\ttr$ is generic, then $N=0$ on $\ttr$ and $\ttr\cong \ttr^{\sss}$. Note that all $\ttr$ of inertial type $\xi_0$ are generic if $\Omega_i \neq \Omega_{i'}$ for all $i \neq i'$.

\begin{lemma}\label{genenum}
With the above notation, $\ttr$ is generic if and only if, for any $j=1, \dots, s$, the eigenvalues $\alpha_{j,1}, \dots, \alpha_{j,m_j}$ of $B_j$ satisfy $\alpha_{j,i}\neq \alpha_{j,i'}$, and $\alpha_{j,i} \neq \alpha_{j,i'} q_L^{d_j}$ for $i\neq i'$. 
\end{lemma}
\begin{proof}
	Assume $\ttr$ is generic, in particular $\ttr\cong \oplus_{j=1}^s \oplus_{i=1}^{m_j} \ttr_{j,i}$. Let $\beta_{j,i}\in E^{\times}$ be such that $\ttr_{j,i}\cong \widetilde{\ttr}_j \otimes_E \unr(\beta_{j,i})$. We then deduce that there exists a basis $\ul{e}'_{j,d_j-1}$ of $\widetilde{\xi}_{j,d_j-1}^{\oplus m_j}$ such that the corresponding matrix of the operator $\ttr(F^{d_j}) \circ \ttr_0(F^{-d_j})\big|_{\widetilde{\xi}_{j,d_j-1}^{\oplus m_j}}$ is equal to
	\[\diag(\underbrace{\beta_{j,1}^{d_j}, \dots, \beta_{j,1}^{d_j}}_{f_j}, \dots, \underbrace{\beta_{j,m_j}^{d_j}\dots, \beta_{j,m_j}^{d_j}}_{f_j}).\]
	Thus $\{\alpha_{j,1}, \dots, \alpha_{j,m_j}\}=\{\beta_{j,1}^{d_j}, \dots, \beta_{j,m_j}^{d_j}\}$, and the ``only if" part then follows from definition of genericity. Now assume $\alpha_{j,1}, \dots, \alpha_{j,m_j}$ satisfy the conditions in the lemma, in particular, are distinct. By comparing dimensions, we easily see that, for each $\alpha_{j,i}$, the subspace $\ttr_{j,i,d_j-1}$ of $\widetilde{\xi}_{j,d_j-1}^{\oplus m_j}$ on which $\ttr(F^{d_j})\circ \ttr_0(F^{-d_j})$ acts via $\alpha_{j,i}$ is isomorphic to $\widetilde{\xi}_{j,d_j-1}$. Any subrepresentation $\widetilde{\xi}_{j,d_j-1}$ in $\widetilde{\xi}_{j,d_j-1}^{\oplus m_j}$ is preserved by $\ttr_0(F^{-d_j})$. We deduce then $\ttr_{j,i,d_j-1}$ is preserved by $\ttr(F^{d_j})$. Under the $\ttr(F)$-action, $\ttr_{j,i,d_j-1}$ then generates an irreducible Weil-Deligne subrepresentation $\ttr_{j,i}$ of $\ttr$ of inertial type $\widetilde{\xi}_j$. By Lemma \ref{inert2}, $\ttr_{j,i}\cong \widetilde{\ttr}_j \otimes_E \unr(\beta_{j,i})$ for any $d_j$-th root $\beta_{j,i}$ of $\alpha_{j,i}$. By the conditions on $\{\alpha_{j,i}\}$, we see that $\ttr\cong \oplus_{j=1}^s \oplus_{i=1}^{m_j} \ttr_{j,i}$ and that $\ttr$ is generic.
\end{proof}

Fix $\textbf{h}=(h_{i,\tau})_{\substack{i=1, \dots, n\\ \tau\in \Sigma_L}}\in \Z^{n[L:\Q_p]}$ with $h_{i,\tau}>h_{i+1,\tau}$. Consider $\fX_{\overline{\rho}}^{\pcr}(\xi_0,\textbf{h}):=(\Spf R_{\overline{\rho}}^{\pcr}(\xi_0, \textbf{h}))^{\rig}$.\index{$\fX_{\overline{\rho}}^{\pcr}(\xi_0,\textbf{h})$} By \cite[Thm.\ 4.1]{CEGGPS}, there is a natural morphism
\begin{equation*}
	\fX_{\overline{\rho}}^{\pcr}(\xi_0,\textbf{h}) \lra (\Spec \cZ_{\Omega_0})^{\rig} \cong (\Spec \cZ_{\Omega})^{\rig}/\sW_{\Omega}
\end{equation*}
which, pointwise, sends $\rho$ to the semi-simple Weil-Deligne representation associated to $\rho$. Let \index{$\widetilde{\fX}_{\overline{\rho}}^{\pcr}(\xi_0, \textbf{h})$}
\begin{equation*}
	\widetilde{\fX}_{\overline{\rho}}^{\pcr}(\xi_0, \textbf{h}):= \fX_{\overline{\rho}}^{ \pcr}(\xi_0,\textbf{h}) \times_{(\Spec \cZ_{\Omega_0})^{\rig} } (\Spec \cZ_{\Omega})^{\rig},
\end{equation*}
so a point of $\widetilde{\fX}_{\overline{\rho}}^{\pcr}(\xi_0, \textbf{h})$ is of the form $(\rho, (\ttr_i))$ with $\ttr(\rho)^{\sss} \cong \oplus_{i=1}^r \ttr_i$. In particular, the $r$-tuple $(\ttr_i)$ induces an $\Omega$-filtration on $D_{\rig}(\rho)$. Let $U_{\overline{\rho}}^{\pcr}(\xi_0, \textbf{h})$ (resp.\ $\widetilde{U}_{\overline{\rho}}^{\pcr}(\xi_0, \textbf{h})$) be the set of points $\rho\in \fX_{\overline{\rho}}^{\pcr}(\xi_0,\textbf{h})$ (resp.\ $(\rho, (\ttr_i))\in \widetilde{\fX}_{\overline{\rho}}^{\pcr}(\xi_0, \textbf{h})$) such that $\ttr(\rho)$ is generic. \index{$U_{\overline{\rho}}^{\pcr}(\xi_0, \textbf{h})$}\index{$\widetilde{U}_{\overline{\rho}}^{\pcr}(\xi_0, \textbf{h})$}

\begin{proposition}\label{geneDen}
The set $U_{\overline{\rho}}^{\pcr}(\xi_0, \textbf{h})$ is Zariski-open and Zariski-dense in $\fX_{\overline{\rho}}^{\pcr}(\xi_0, \textbf{h})$.
\end{proposition}
\begin{proof}
Let $U$ be an arbitrary (non-empty) connected admissible affinoid open of $\fX_{\overline{\rho}}^{\pcr}(\xi_0, \textbf{h})$. A major part of the proof is to show $U_{\overline{\rho}}^{\pcr}(\xi_0, \textbf{h}) \cap U$ is Zariski-open in $U$. We will show that a certain torsor $\widetilde{U}$ over $U$ admits a smooth morphism to $\big(\prod_{\substack{j=1, \dots, s \\ k \in \Z/d_j\Z}} \GL_{m_j} \big) \times \Res^{L}_{\Q_p}(\GL_n/B)$ (cf.\ (\ref{smoo1})), where, roughly speaking, the morphism to the first factor sends $\rho$ to the matrices $\{B_{j,k}\}$ associated to $\ttr(\rho)$ as above, and the morphism to the second factor sends $\rho$ to the Hodge filtration on $D_{\dR}(\rho)$. The torsor that we will use comes from choices of basis on the corresponding objects.
	
Recall that by \cite{Kis08} we have the following data:
\begin{itemize}
\item[(1)] a rank $n$ locally free $L_0' \otimes_{\Q_p} \co_U\cong \prod_{\tau\in \Sigma_{L_0'}} \co_U$-module $\DF_U\cong \prod_{\tau \in \Sigma_{L_0'}} \DF_{U,\tau}$ equipped with a semi-linear action of $(\Gal(L'/L), \varphi)$;
\item[(2)] a decreasing filtration $\Fil^{\bullet}=\prod_{\tau \in \Sigma_L} \Fil_{\tau}^{\bullet}$ on the rank $n$ locally free $L \otimes_{\Q_p} \co_U\cong \prod_{\tau \in \Sigma_L} \co_U$-module $\cD_U:=(\DF_U \otimes_{L_0'} L')^{\Gal(L'/L)}\cong \prod_{\tau \in \Sigma_L} \cD_{U,\tau}$ by $L \otimes_{\Q_p} \co_U$-submodules which are direct summands of $\cD_U$ as $\co_U$-module such that (letting $-h_{0,\tau}=-\infty$ and $-h_{n+1,\tau}=+\infty$)
		\begin{equation*}
			\rk_{\co_U} \Fil^i_{\tau} \cD_{U,\tau}=n-j+1, \text{ for $-h_{j-1,\tau}<i \leq -h_{j,\tau}$}.
		\end{equation*}
\end{itemize}
Moreover, for any point $x\in U$, the specialization of $(\DF_{U}, \cD_U)$ at $x$ is equal to $(D_{\pst}(\rho_x), D_{\dR}(\rho_x))$ where $\rho_x$ is the associated $\Gal_L$-representation. Shrinking $U$, we can and do assume that $\DF_U$ and $\cD_{U}$ are free over $\co_U$. We fix an embedding $\tau_0: L_0 \hookrightarrow E$ and put $\tau_i:=\tau_0 \circ \Frob^{-i}$ where $\Frob$ is the (absolute arithmetic) Frobenius on the Witt vectors. Then we have a decomposition
\[L_0' \otimes_{\Q_p} E \xlongrightarrow{\sim} \bigoplus_{i\in \Z/[L_0:\Q_p]\Z} \Bigg(\bigoplus_{\substack{\tau\in \Sigma_{L_0'}\\ \tau|_{L_0}=\tau_i}} E\Bigg).\]
and a corresponding decomposition $\DF_U \cong \oplus_{i\in \Z/[L_0:\Q_p]\Z} \DF_{U,i}$, where $\DF_{U,i}$ is free over $L_0' \otimes_{L_0, \tau_i} \co_U$ and preserved by $\Gal(L'/L)$. Shrinking $U$ and enlarging $E$ if necessary (and using $H^i(H,M)=0$ for all $i>0$, finite groups $H$ and $H$-modules $M$ over $\Q$), we can and do assume that there is a semi-linear $\Gal(L'/L)$-representation $V_0$, free of rank $n$ over $L_0' \otimes_{L_0, \tau_0} E$, such that we have a $\Gal(L'/L)$-equivariant isomorphism
	\begin{equation}\label{V0DF0}
		V_0 \otimes_E \co_U \xlongrightarrow{\sim} \DF_{U, 0}.
	\end{equation}
	For $i\in \Z/[L_0:\Q_p]\Z$, let $\Fro^i: L_0' \otimes_{L_0, \tau_0} E \ra L_0' \otimes_{L_0, \tau_0\circ \Frob^{-i}} E$ be the isomorphism of algebras sending $a \otimes b$ to $\Frob^i(a) \otimes b$. The morphism $\Fro^i$ is $\Gal(L'/L)$-equivariant, where $\Gal(L'/L)$ acts on both sides via its natural action on the factor $L'_0$. Put
	\begin{equation}\label{v0vi}
	V_i:=V_0 \otimes_{L_0' \otimes_{L_0, \tau_0} E, \Fro^i} (L_0' \otimes_{L_0, \tau_i} E),
	\end{equation}
which is a semi-linear $\Gal(L'/L)$-representation over $L_0' \otimes_{L_0, \tau_i} E$. The action of $\varphi^i$ on $\DF_U$ sends $\DF_{U, 0}$ to $\DF_{U,i}$. For $i=0, \dots, [L_0:\Q_p]-1$, the composition $V_0 \otimes_E \co_U \xrightarrow{\sim} \DF_{U,0} \xrightarrow{\varphi^i} \DF_{U,i}$ induces an $L_0'\otimes_{L_0, \tau_i} \co_U$-semilinear $\Gal(L'/L)$-equivariant isomorphism 
	\begin{equation}\label{isoii}
		V_i \otimes_E \co_U \xlongrightarrow{\sim} \DF_{U,i}.
	\end{equation}
	
	For each $\tau\in \Sigma_{L_0'}$ with $\tau|_{L_0}=\tau_0$, we have $V_{0,\tau}\cong \xi_0$ as $I(L'/L)$-representation over $E$. The basis $\ul{e}$ of $\xi_0$ (that we previously fixed) then gives a basis $\ul{e}_{\tau}=(e_{\tau,1}, \dots, e_{\tau,n})$ of $V_{0,\tau}$ satisfying Condition \ref{condIac} for the $I(L'/L)$-action. Then $(\ul{e}_{\tau})_{\tau|_{L_0}=\tau_0}$ form a basis of $V_0$ over $E$. Choosing the basis $(\ul{e}_{\tau})_{\tau|_{L_0}=\tau_0}$ is the same as choosing a basis of $V_0$, formed by bases of $V_{0,\tau}$ for each $\tau$ on which the $I(L'/L)$-action satisfies Condition \ref{condIac}. For $i=1, \dots, n$, let $\tilde{e}_i:=(e_{\tau,i})_{\tau|_{L_0}=\tau_0}\in V_0$. Then $\ul{\tilde{e}}=(\tilde{e}_1, \dots, \tilde{e}_n)$ is a basis of $V_0$ over $L_0' \otimes_{L_0, \tau_0} E$, which also gives a basis of $V_i$ over $L_0' \otimes_{L_0, \tau_i} E$ by $-\otimes 1$ (see (\ref{v0vi})). For $i\in \Z/[L_0:\Q_p]\Z$, there exists thus $P_i\in \GL_n(L' \otimes_{L_0,\tau_i} E)$ such that $(\fe_1, \dots, \fe_n):=(\tilde{e}_1, \dots, \tilde{e}_n)P_i$ is a basis of $(V_i \otimes_{L_0'} L')^{\Gal(L'/L)}$ over $L\otimes_{L_0, \tau_i} E$. For $\tau\in \Sigma_L$ and $i\in \Z/[L_0:\Q_p]\Z$ such that $\tau|_{L_0}=\tau_i$, we let $\ul{\fe}_{\tau}:=(\fe_{1,\tau}, \dots, \fe_{n,\tau})$ be the $\tau$-factor of $\ul{\fe}$, which is a basis of $\big((V_i \otimes_{L_0'} L')^{\Gal(L'/L)}\big)_{\tau}$ over $E$. 
	
	Let $\sG$ be the affine subgroup $\Aut_{\Gal(L'/L)}(V_0)$ of $\Aut(V_0)\cong \Res^{L_0'}_{L_0} \GL_n \times_{\Spec L_0,\tau_0} \Spec E$ where ``$\Aut$" means $L_0' \otimes_{L_0, \tau_0} E$-linear bijections. Note that $\sG$ is smooth by Cartier's theorem. Let $\widetilde{U}$ be the $\sG$-torsor over $U$ of $\Gal(L'/L)$-equivariant isomorphisms in (\ref{V0DF0}). Define
\[\DF_{\widetilde{U}}=\bigoplus_{i\in \Z/[L_0:\Q_p]\Z} \DF_{\widetilde{U},i}:=\bigoplus_{i\in \Z/[L_0:\Q_p]\Z} \DF_{U,i}\otimes_{\co_U} \co_{\widetilde{U}}\]
equipped with the action of $(\varphi, \Gal(L'/L))$ by extension of scalars. By definition, we have a $\Gal(L'/L)$-equivariant isomorphism of $\co_{\widetilde{U}}$-modules $V_0 \otimes_E \co_{\widetilde{U}} \xrightarrow{\sim} \DF_{\widetilde{U},0}$, which induces $\Gal(L'/L)$-equivariant isomorphisms $V_i\otimes_E \co_{\widetilde{U}} \xrightarrow{\sim} \DF_{\widetilde{U},i}$ for $i\in \Z/[L_0:\Q_p]\Z$ similarly as for (\ref{isoii}). These isomorphisms then induce isomorphisms
	\begin{equation}\label{DdrUtil}
		(V_i \otimes_{L_0'} L')^{\Gal(L'/L)}\otimes_E \co_{\widetilde{U}} \xlongrightarrow{\sim} (\DF_{\widetilde{U},i}\otimes_{L_0'} L')^{\Gal(L/L)}
	\end{equation}
(use that the left hand side of (\ref{DdrUtil}) is a direct summand of $(V_i \otimes_{L_0'} L') \otimes_E \co_{\widetilde{U}} \cong \DF_{\widetilde{U},i} \otimes_{L_0'} L$ to see that (\ref{DdrUtil}) is an isomorphism). Let $\cD_{\widetilde{U}}:=\cD_U \otimes_{\co_U} \co_{\widetilde{U}}\cong (\DF_{\widetilde{U}} \otimes_{L_0'} L')^{\Gal(L'/L)}$, and $\cD_{\widetilde{U}, \tau}$ be its $\tau$-factor for $\tau \in \Sigma_L$. Using the isomorphism (\ref{DdrUtil}), we obtain a basis $\ul{\fe}_{\tau} \otimes 1$ of $\cD_{\widetilde{U}, \tau}$ over $\co_{\widetilde{U}}$ for all $\tau\in \Sigma_L$. We denote by $\Fil_{\tau}^{\bullet}$ the filtration on $\cD_{\widetilde{U}, \tau}$ induced by the corresponding filtration on $\cD_{U,\tau}$ by extension of scalars. With respect to the basis $\ul{\fe}_{\tau} \otimes 1$, $\Fil_{\tau}^{\bullet}$ gives a flag $\sF_{\widetilde{U}, \tau}\in (\GL_n/B)(\co_{\widetilde{U}})$. Taking all embeddings $\tau\in \Sigma_\tau$, we thus obtain a morphism
	\begin{equation}\label{flagPst0}
		\widetilde{U} \lra \Res^{L}_{\Q_p} (\GL_n/B).
	\end{equation}
	
We let $F\in W_L$ act on $\DF_{{\widetilde{U}},0}$ via $\varphi^{-[L_0:\Q_p]} \circ \overline{F}$, where $\overline{F}$ denotes the image of $F$ in $\Gal(L'/L)$. We have a decomposition
\[\DF_{{\widetilde{U}},0}\cong \bigoplus_{\substack{\tau\in \Sigma_{L_0'} \\ \tau|_{L_0}=\tau_0}} \DF_{\widetilde{U}, \tau}\]
and each $\DF_{\widetilde{U}, \tau}$ is preserved by $F$ and $I(L'/L)$. For any $\tau\in \Sigma_{L_0'}$ such that $\tau|_{L_0}=\tau_0$, $\DF_{{\widetilde{U}},\tau}$ equipped with $F$ and the action of $I(L'/L)$ gives a Weil-Deligne representation of $W_L$ over $\co_{\widetilde{U}}$ (which factors through $W_L/I_{L'}$ and is independent of the choice of $\tau$). Using (\ref{V0DF0}), the action of $F$ on $\DF_{{\widetilde{U}},\tau}$ induces an operator denoted by $\ttr_{\widetilde{U}}(F)$ on $V_{0,\tau} \otimes_E \co_{\widetilde{U}}$. On the other hand, we have the operator $\ttr_0(F)$ on $V_{0,\tau}\cong \xi_0$ given as in the discussion above (\ref{choiBa}), and we still denote by $\ttr_0(F)$ its extension of scalars on $V_{0,\tau}\otimes_E \co_{\widetilde{U}}$. Similarly as in the discussion above (\ref{eImap}) and using
\[\Hom_{I(L'/L)}\big(\widetilde{\xi}_{j,k}\otimes_E \co_{\widetilde{U}}, \widetilde{\xi}_{j',k'}\otimes_E \co_{\widetilde{U}}\big)=\begin{cases}
		\co_{\widetilde{U}} & j=j', k=k'\\
		0& \text{otherwise,}
	\end{cases}\]
we see that the operator $\ttr_{\widetilde{U}}(F) \circ \ttr_0(F)^{-1}$ corresponds to a matrix $A_{\widetilde{U}}=\diag(\{A_{{\widetilde{U}},j,k}\}_{\substack{j=1, \dots, s\\ k\in \Z/d_j\Z}})$ where $A_{{\widetilde{U}},j,k}$ is the matrix of the $\co_{\widetilde{U}}$-linear endomorphism 
	\begin{equation*}\ttr_{\widetilde{U}}(F) \circ \ttr_0(F)^{-1}: \widetilde{\xi}_{j,k}^{\oplus m_j} \otimes_E \co_{\widetilde{U}} \xlongrightarrow{\sim} \widetilde{\xi}_{j,k}^{\oplus m_j} \otimes_E \co_{\widetilde{U}}
	\end{equation*}
	and is the image of a matrix $B_{{\widetilde{U}},j,k}\in \GL_{m_j}(\co_{\widetilde{U}})$ via the morphism $\GL_{m_j}(\co_{\widetilde{U}})\hookrightarrow \GL_{m_jf_j}(\co_U)$, $(a_{uv})\mapsto (a_{uv} \mathrm{I}_{f_j})$. The matrices $\{B_{{\widetilde{U}},j,k}\}$ give rise to a morphism
	\[\widetilde{U} \lra \prod_{\substack{j=1, \dots, s \\ k \in \Z/d_j\Z}} \GL_{m_j}.\]
Together with (\ref{flagPst0}), we finally obtain a morphism
	\begin{equation}\label{smoo1}
		f: \widetilde{U} \lra \Bigg(\prod_{\substack{j=1, \dots, s \\ k \in \Z/d_j\Z}} \GL_{m_j} \Bigg) \times \Res^{L}_{\Q_p} (\GL_n/B).
	\end{equation}
	
We prove that $f$ is smooth. Since both source and target are smooth over $E$ (using \cite[Thm.~3.3.8]{Kis08} for the source), we only need to show that the tangent map of $f$ at any point of $\widetilde{U}$ is surjective. Let $x\in \widetilde{U}$ and $f(x)=(\{B_{x,j,k}\}, \{\sF_{x,\tau}\})$. Let $u$ be an element in the tangent space of the right hand side of (\ref{smoo1}) at $f(x)$, and denote the corresponding element by:
\[(\{\widetilde{B}_{x,j,k}\}, \{\widetilde{\sF}_{x,\tau}\}) \in \Bigg(\Big(\prod_{\substack{j=1, \dots, s \\ k \in \Z/d_j\Z}} \GL_{m_j} \Big) \times \Res^{L}_{\Q_p} (\GL_n/B)\Bigg) (k(x)[\varepsilon]/\varepsilon^2).\]
Let \ $\widetilde{A}:=\diag(\{\widetilde{A}_{x,j,k}\}_{\substack{j=1, \dots, s\\ k\in \Z/d_j\Z}})\in \GL_n(k(x)[\varepsilon]/\varepsilon^2)$ \ where \ $\widetilde{A}_{x,j,k}$ \ is \ the \ image \ of \ $\widetilde{B}_{x,j,k}$ \ via $\GL_{m_j}(k(x)[\varepsilon]/\varepsilon^2)\hookrightarrow \GL_{m_jf_j}(k(x)[\varepsilon]/\varepsilon^2), (a_{uv}) \mapsto (a_{uv} \mathrm{I}_{f_j})$. We use $\widetilde{A}$ to construct a Deligne-Fontaine module $\widetilde{\DF}_x$ over $L_0' \otimes_{\Q_p} k(x)[\varepsilon]/\varepsilon^2$. Let $\widetilde{\DF}_x:=(\oplus_{i\in \Z/[L_0:\Q_p]\Z} V_i) \otimes_{E} k(x)[\varepsilon]/\varepsilon^2$, we have a decomposition $\widetilde{\DF}_x \cong \oplus_{\tau \in \Sigma_{L_0'}}\widetilde{\DF}_{x,\tau}$. To get a operator $\varphi$ on the whole $\widetilde{\DF}_x$, we only need to construct $\varphi^i: \widetilde{\DF}_{\tau} \ra \widetilde{\DF}_{\tau \circ \Frob^{-i}}$ for one $\tau\in \Sigma_{L_0'}$ and for all $i=1, \dots, [L_0':\Q_p]$. Recall we have fixed a basis $\ul{e}_{\tau}$ of each $\widetilde{\DF}_{\tau}$. For $i=1, \dots, [L_0':L_0]$, let $M_i\in \GL_n(E)$ such that $\overline{F}(e_{\tau,1}, \dots, e_{\tau,n})=(e_{\tau \circ \Frob_{L_0}^{-i}, 1}, \dots, e_{\tau \circ \Frob_{L_0}^{-i},n})M_i$ and put:
	\begin{equation*}
		\varphi^{[L_0:\Q_p]i}: \widetilde{\DF}_{\tau} \lra \widetilde{\DF}_{\tau \circ \Frob_{L_0}^{-i}}, \ (e_{\tau,1}, \dots, e_{\tau,n}) \longmapsto \big(e_{\tau \circ \Frob_{L_0}^{-i}, 1}, \dots, e_{\tau \circ \Frob_{L_0}^{-i},n}\big)M_i (\widetilde{A} M_0)^{-i}.
	\end{equation*}
For $i=1, \dots, [L_0':\Q_p]$, writing $i=i_0 [L_0:\Q_p] + j$ with $0\leq j < [L_0:\Q_p]$, we define
\[\varphi^i: \widetilde{\DF}_{\tau} \lra \widetilde{\DF}_{\tau \circ \Frob^{-i}},\ (e_{\tau,1}, \dots, e_{\tau,n})\longmapsto (e_{\tau \circ \Frob^{-i}, 1}, \dots, e_{\tau \circ \Frob^{-i},n})M_{i_0} (\widetilde{A} M_0)^{-i_0}.\]
Thus $\widetilde{\DF}_x$ equipped with the $(\Gal(L'/L),\varphi)$-action is a Deligne-Fontaine module over $L_0' \otimes_{\Q_p} k(x)[\varepsilon]/\varepsilon^2$ and $\widetilde{\DF}_x \cong x^* \DF_{\widetilde{U}}\equiv D_{\pst}(\rho_x) \pmod{\varepsilon}$. Put:
\[\widetilde{\cD}_x:=(\widetilde{\DF}_x \otimes_{L_0'} L')^{\Gal(L'/L)}\cong \bigoplus_{\tau\in \Sigma_L}\widetilde{\cD}_{x,\tau}.\]
For each $\tau$, $\ul{\fe}_{\tau}$ form a basis of $\widetilde{\cD}_{x,\tau}$ over $k(x)[\varepsilon]/\varepsilon^2$. Using the basis $\ul{\fe}_{\tau}$, we associate to the flag $\widetilde{\sF}_{x,\tau}$ a decreasing filtration $ \Fil^i_{\tau}$ on $\widetilde{\cD}_{x,\tau}$ satisfying
	\begin{equation*}
		\rk_{k(x)[\varepsilon]/\varepsilon^2} (\Fil^i_{\tau} \widetilde{\cD}_{x,\tau})=n-j+1 \text{ for $-h_{j-1,\tau}<i \leq -h_{j,\tau}$}.
	\end{equation*} 
We \ obtain \ a \ filtered \ $(\varphi, \Gal(L'/L))$-module \ $(\widetilde{\DF}_x, \widetilde{\cD}_x)$ \ which \ is \ a \ deformation \ of \ $(D_{\pst}(\rho_x), D_{\dR}(\rho_x))$ over $k(x)[\varepsilon]/\varepsilon^2$. Thus $(\widetilde{\DF}_x, \widetilde{\cD}_x)$ is weakly admissible, and by \cite{CF} is isomorphic to $(D_{\pst}(\widetilde{\rho}_x), D_{\dR}(\widetilde{\rho}_x))$ for a certain deformation $\widetilde{\rho}_x$ of $\rho_x$ over $k(x)[\varepsilon]/\varepsilon^2$. By choosing an appropriate basis of $\widetilde{\rho}_x$ over $k(x)[\varepsilon]/\varepsilon^2$, we see that
\[\big(\widetilde{\rho}_x, V_0 \otimes_E k(x)[\varepsilon]/\varepsilon^2\xlongrightarrow{\sim} \DF(\widetilde{\rho}_x)_0\cong \widetilde{\DF}_{x,0}\big)\]
gives an element in the tangent space of $\widetilde{U}$ at $x$ which is sent to $u$. 
	
The conditions in Lemma \ref{genenum} cut out a smooth Zariski-open and Zariski-dense subspace of $\prod_{\substack{j=1, \dots, s \\ k \in \Z/d_j\Z}} \GL_{m_j}$ which, by taking fibre product with $\Res^{L}_{\Q_p} (\GL_n/B)$, gives a smooth Zariski-open and Zariski-dense subspace of $\big(\prod_{\substack{j=1, \dots, s \\ k \in \Z/d_j\Z}} \GL_{m_j} \big) \times \Res^{L}_{\Q_p} (\GL_n/B)$. The proposition then follows by the same argument as in the last paragraph of the proof of \cite[Lemma 2.4]{BHS2}. 
\end{proof}

\begin{corollary}
The set $\widetilde{U}_{\overline{\rho}}^{\pcr}(\xi_0, \textbf{h})$ is Zariski-open and Zariski-dense in $\widetilde{\fX}_{\overline{\rho}}^{\pcr}(\xi_0, \textbf{h})$.
\end{corollary}
\begin{proof}
	Let \ $U$, $\tilde{U}$ \ be \ as \ in \ the \ proof \ of \ Proposition \ \ref{geneDen}. \ For \ each \ point \ $x=(B_{j,k})$ of $\prod_{\substack{j=1, \dots, s \\ k \in \Z/d_j\Z}} \GL_{m_j}$ we can associate a Weil-Deligne representation $\ttr_x$ of inertial type $\xi_0$ with $N=0$ as in the discussion below (\ref{matrixW}). By the same argument as in \cite[Prop.\ 4.3]{CEGGPS}, there exists a unique morphism $f: \prod_{\substack{j=1, \dots, s \\ k \in \Z/d_j\Z}} \GL_{m_j} \ra (\Spec \cZ_{\Omega_0})^{\rig}$ such that $\ttr_{f(x)}\cong \ttr_x^{\sss}$ and the composition $\widetilde{U} \ra U \ra \fX_{\overline{\rho}}^{\pcr}(\xi_0, \textbf{h}) \ra (\Spec \cZ_{\Omega_0})^{\rig}$ factors through $\prod_{\substack{j=1, \dots, s \\ k \in \Z/d_j\Z}} \GL_{m_j}$. Hence (\ref{smoo1}) induces a smooth morphism
	\begin{equation*}
		\widetilde{U} \times_{(\Spec \cZ_{\Omega_0})^{\rig}} (\Spec \cZ_{\Omega})^{\rig} \lra \Bigg(\Big(\prod_{\substack{j=1, \dots, s \\ k \in \Z/d_j\Z}} \GL_{m_j}\Big) \times_{(\Spec \cZ_{\Omega_0})^{\rig}} (\Spec \cZ_{\Omega})^{\rig}\Bigg) \times \Res^L_{\Q_p} (\GL_n/B). 
	\end{equation*}
	It follows from Lemma \ref{genenum} (see also the last paragraph in the proof of Proposition \ref{geneDen}) that the image of the generic points of $\widetilde{U} \times_{(\Spec \cZ_{\Omega_0})^{\rig}} (\Spec \cZ_{\Omega})^{\rig}$ is Zariski-dense and Zariski-open in $(\prod_{\substack{j=1, \dots, s \\ k \in \Z/d_j\Z}} \GL_{m_j}) \times_{(\Spec \cZ_{\Omega_0})^{\rig}} (\Spec \cZ_{\Omega})^{\rig}$. Since
\[\widetilde{U} \times_{(\Spec \cZ_{\Omega_0})^{\rig}} (\Spec \cZ_{\Omega})^{\rig} \lra U \times_{(\Spec \cZ_{\Omega_0})^{\rig}} (\Spec \cZ_{\Omega})^{\rig}\]
is smooth and surjective, we deduce that the generic points are Zariski-dense and Zariski-open in $ U \times_{(\Spec \cZ_{\Omega_0})^{\rig}} (\Spec \cZ_{\Omega})^{\rig}$.
\end{proof}

By the same argument as for \cite[Lemma 2.2]{BHS2}, we have

\begin{proposition}
	The rigid space $\widetilde{\fX}_{\overline{\rho}}^{ \pcr}(\xi_0, \textbf{h})$ is reduced.
\end{proposition}

Let $x=(\rho_x, (\ttr_{i,x}))\in \widetilde{U}_{\overline{\rho}}^{\pcr}(\xi_0, \textbf{h})$. The $\Omega$-filtration associated to $(\ttr_{i,x})$ induces an $\Omega$-filtration on $\DF(\rho_x)$, to which we can associate $w_x \in \sW^P_{\max,L}$ as in \S~\ref{introPcr} (where $w_x$ is denoted $w_{\sF}$). Recall that $w_x$ measures the relative position of the Hodge filtration and the $\Omega$-filtration on $D_{\dR}(\rho_x)$. For $w\in \sW^P_{\max,L}$, let $V_{\overline{\rho}}^{ \pcr}(\xi_0, \textbf{h})_w$ be the set of points $x=(\rho_x, (\ttr_{i,x}))$ in $\widetilde{U}_{\overline{\rho}}^{\pcr}(\xi_0, \textbf{h})$ satisfying $w_x=w$.\index{$V_{\overline{\rho}}^{ \pcr}(\xi_0, \textbf{h})_w$}

\begin{proposition}\label{VwSch}
	(1) The set $V_{\overline{\rho}}^{\pcr}(\xi_0, \textbf{h})_{w_0}$ is Zariski-open and Zariski-dense in $\widetilde{U}_{\overline{\rho}}^{\pcr}(\xi_0, \textbf{h})$ (hence in $\widetilde{\fX}_{\overline{\rho}}^{ \pcr}(\xi_0, \textbf{h})$).
	
	(2) For $w\in \sW^P_{\max,L}$, $V_{\overline{\rho}}^{ \pcr}(\xi_0, \textbf{h})_{w}$ is locally Zariski-closed in $\widetilde{U}_{\overline{\rho}}^{ \pcr}(\xi_0, \textbf{h})$. Moreover, if $w'\in\sW^P_{\max,L}$ and $w'\leq w$, then $V_{\overline{\rho}}^{ \pcr}(\xi_0, \textbf{h})_{w'}$ lies in the Zariski-closure $\overline{V_{\overline{\rho}}^{ \pcr}(\xi_0, \textbf{h})_{w}}$ of $V_{\overline{\rho}}^{ \pcr}(\xi_0, \textbf{h})_{w}$ in $\widetilde{U}_{\overline{\rho}}^{ \pcr}(\xi_0, \textbf{h})$.
\end{proposition}
\begin{proof}
	Let $U\subset U_{\overline{\rho}}^{ \pcr}(\xi_0, \textbf{h})$ be a non-empty affinoid open subset as in the proof of Proposition \ref{geneDen}. We freely use the notation of {\it loc.\ cit.} For $j=1, \dots, s$, let $B_{U,j}:=\prod_{i\in \Z/[L_0':L_0]} B_{U,j,i}$, and $P_j(X)\in \co_U[X]$ be the characteristic polynomial of $B_{U,j}$. Let $\sU$ be an \'etale covering of $U$ such that $P_j(X)=\prod_{i=1}^{m_j} (X-\alpha_{j,i})$ (with $\alpha_{j,i}\neq \alpha_{j,i'}$ for $i\neq i'$ since any point in $U$ is generic), and such that, for each $j,i$, there exists $\beta_{j,i}\in \co_{\sU}$ such that $\beta_{j,i}^{d_j}=\alpha_{j,i}$. Using the same argument as in the proof of Lemma \ref{genenum}, we have an isomorphism of Weil-Deligne representations over $\co_{\sU}$:
	\begin{equation*}
		\ttr_{\sU} \cong \oplus_{j=1}^s \oplus_{i=1}^{m_j} \widetilde{\ttr}_j \otimes_E \unr(\beta_{j,i}).
	\end{equation*} 
Let $\widetilde{\sU}:=\sU\times_{(\Spec \cZ_{\Omega_0})^{\rig}} (\Spec \cZ_{\Omega})^{\rig}$ and $r_{\widetilde{\sU}}:=r_{\sU} \otimes_{\co_{\sU}} \co_{\widetilde{\sU}}$. By the universal property of $\widetilde{\sU} \ra \sU$, we have an (ordered) $r$-tuple $(\ttr_{\widetilde{\sU}, i})_{i=1,\dots, r}$ of Weil-Deligne subrepresentations of $\ttr_{\widetilde{\sU}}$ over $\widetilde{\sU}$ such that $\oplus_{i=1}^r \ttr_{\widetilde{\sU}, i} \cong \oplus_{j=1}^s \oplus_{i=1}^{m_j} \widetilde{\ttr}_j \otimes_E \unr(\beta_{j,i})$ and $\ttr_{\widetilde{\sU},i}$ is of inertial type $\xi_i$. 
	
The Deligne-Fontaine module $\DF_{\widetilde{\sU}}:=\DF_{U} \otimes_{\co_U} \co_{\widetilde{\sU}}$ over $\widetilde{\sU}$ is isomorphic to the Deligne-Fontaine module associated to $\ttr_{\widetilde{\sU}}$. Let $\DF_{\widetilde{\sU},i}$ be the Deligne-Fontaine module associated to $\ttr_{\widetilde{\sU},i}$, we have thus $\DF_{\widetilde{\sU}}\cong \oplus_{i=1}^r \DF_{\widetilde{\sU},i}$. We let $\widetilde{\sU}^r$ be the $(\bG_m^{\rig})^r$-torsor over $\widetilde{\sU}$ of isomorphisms of Deligne-Fontaine modules:
	\begin{equation*}
		\oplus_{i=1}^r \DF_{\widetilde{\sU},i} \xlongrightarrow{\sim} \DF_{\widetilde{\sU}}.
	\end{equation*}
	By the universal property of $\widetilde{\sU}^r$, we have a universal isomorphism
	\begin{equation*}
		\oplus_{i=1}^r \DF_{\widetilde{\sU}^r,i} \xlongrightarrow{\sim} \DF_{\widetilde{\sU}^r},
	\end{equation*}
	hence universal embeddings $\oplus_{i=1}^j \DF_{\widetilde{\sU}^r,i} \hookrightarrow \DF_{\widetilde{\sU}^r}$ for $j=1,\dots, r$, where $*_{\widetilde{\sU}^r}$ denotes the base change of the corresponding object to $\co_{\widetilde{\sU}^r}$. 
	Let $W:=(L \otimes_{\Q_p} \co_{\widetilde{\sU}^r})^n$, and we equip $W$ with a filtration $\Fil_W^{\bullet}$ consisting of free $L\otimes_{\Q_p} \co_{\widetilde{\sU}^r}$-submodules which are direct summands of $W$ such that $\rk_{L \otimes_{\Q_p} \co_{\widetilde{\sU}^r}} \Fil^i_W W=s_i$ for $i=1,\dots, r$. Let $\widetilde{\sU}^r_P$ be the $\Res^L_{\Q_p} P$-torsor over $\widetilde{\sU}^r$ of isomorphisms
	\begin{equation*}
		(L \otimes_{\Q_p} \co_{\widetilde{\sU}^r})^n \xlongrightarrow{\sim} (\DF_{\widetilde{\sU}^r} \otimes_{L_0'} L')^{\Gal(L'/L)}
	\end{equation*}
which send $\Fil_W^i W\cong (L \otimes_{\Q_p} \co_{\widetilde{\sU}^r})^{s_i}$ onto $(\oplus_{i'=1}^i \DF_{\widetilde{\sU},i'} \otimes_{L_0'} L')^{\Gal(L'/L)}$. Via the universal isomorphism $(L \otimes_{\Q_p} \co_{\widetilde{\sU}^r_P})^n \xlongrightarrow{\sim} (\DF_{\widetilde{\sU}^r_P} \otimes_{L_0'} L')^{\Gal(L'/L)}$, the Hodge filtration on $(\DF_{\widetilde{\sU}^r_P} \otimes_{L_0'} L')^{\Gal(L'/L)}$ (which comes from the Hodge filtration on $(\DF_U \otimes_{L_0'} L')^{\Gal(L'/L)}$ by base change) induces a morphism
	\begin{equation}\label{flagHo}
		\widetilde{\sU}^r_P \lra \Res^L_{\Q_p} (\GL_n/B).
	\end{equation}
	By similar arguments as for (\ref{smoo1}), one can show that this morphism is smooth. 
	
	For $w=(w_{\tau})\in \sW^P_{\max,L}$, let $\widetilde{\sU}_{P,w}^r$ be the set of points which are sent to $V_{\overline{\rho}}^{ \pcr}(\xi_0, \textbf{h})_w$ via the (smooth surjective) morphism $\widetilde{\sU}_{P,w}^r \ra U$. We see that $x\in \widetilde{\sU}_{P,w}^r$ if and only if the image of $x$ under (\ref{flagHo}) is contained in the (generalized) Schubert cell $\prod_{\tau\in \Sigma_L}(P w_{\tau} B)/B$. Equivalently $\widetilde{\sU}_{P,w}^r$ is the inverse image of $\prod_{\tau\in \Sigma_L}(P w_{\tau} B)/B$ in $\widetilde{\sU}^r_P$. The proposition then follows from the corresponding facts on Schubert cells by the same argument as in the last paragraph of the proof of \cite[Lemma 2.4]{BHS2}.
\end{proof}

Let $w\in\sW^P_{\max,L}$, then $ww_{0,L}\in \sW^P_{\min,L}$ hence $ww_{0,L}(\textbf{h})$ is strictly $P$-dominant. Define
\begin{equation}\label{embpcydef}
	\iota: \widetilde{\fX}_{\overline{\rho}}^{ \pcr}(\xi_0, \textbf{h}) \lra \fX_{\overline{\rho}} \times (\Spec \cZ_{\Omega})^{\rig} \times \widehat{\cZ_{0,L}}, \ (\rho, (\ttr_i))\mapsto (\rho, (\ttr_i), 1).
\end{equation}
We consider $\iota^{-1}(X_{\Omega,ww_{0,L}(\textbf{h})}(\overline{\rho}))$, which is a Zariski-closed subspace of $\widetilde{\fX}_{\overline{\rho}}^{ \pcr}(\xi_0, \textbf{h})$. By the discussion in \S~\ref{introPcr} (in particular (\ref{inj000})) and using that, for any $x=(\rho, (\ttr_i))\in V_{\overline{\rho}}^{ \pcr}(\xi_0, \textbf{h})_{w}$, we have by definition $w=w_x$ (which is denoted $w_{\sF}$ in \S~\ref{introPcr}), one can check that
\[V_{\overline{\rho}}^{ \pcr}(\xi_0, \textbf{h})_{w}\subseteq \iota^{-1}(U_{\Omega,,ww_{0,L}(\textbf{h})}(\overline{\rho}))\subseteq \iota^{-1}(X_{\Omega,ww_{0,L}(\textbf{h})}(\overline{\rho})),\]
which implies
\begin{equation*}
	\overline{V_{\overline{\rho}}^{ \pcr}(\xi_0, \textbf{h})_{w}} \subset \iota^{-1}(X_{\Omega,ww_{0,L}(\textbf{h})}(\overline{\rho})).
\end{equation*}
By Proposition \ref{VwSch} (2), we deduce that, for $w'\in \sW^{P}_{\max,L}$ and $w'\leq w$, we have $\iota(V_{\overline{\rho}}^{ \pcr}(\xi_0, \textbf{h})_{w'}) \subseteq X_{\Omega,,ww_{0,L}(\textbf{h})}(\overline{\rho})$, and thus:

\begin{corollary}\label{compploc1}
Let $x=(\rho_x, (\ttr_{i,x}))\in \widetilde{U}_{\overline{\rho}}^{\pcr}(\xi_0, \textbf{h})$. For all $w\in\sW^{P}_{\max,L}$ such that $w\geq w_x$, we have 
	\begin{equation*}
		\iota(x)=\big(\rho_x, (\ttr_{1,x},\dots,\ttr_{r,x}), 1\big)\in X_{\Omega, ww_{0,L}(\textbf{h})}(\overline{\rho}).
	\end{equation*}
\end{corollary}

\begin{remark}\label{remNPara}
(1) The point $\iota(x)\in X_{\Omega, ww_{0,L}(\textbf{h})}(\overline{\rho})$ for $w\in \sW^{P}_{\max,L}$, $w\geq w_x$, is called a local companion point of $\iota(x)\in X_{\Omega, \textbf{h}}(\overline{\rho})$ (see also Corollary \ref{coLoccomp} below). 

(2) The case $P=B$ (and $\rho_x$ crystalline) was contained in \cite[Thm.\ 4.2.3]{BHS3}. Indeed, in this case, as discussed in Remark \ref{remBPE} (1), we have the isomorphism
\[\iota_{\Omega, ww_{0,L}(\textbf{h})}: X_{\Omega, ww_{0,L}(\textbf{h})}(\overline{\rho}) \xlongrightarrow{\sim} X_{\tri}^{\square}(\overline{\rho}) \ (\hookrightarrow (\Spf R_{\overline{\rho}})^{\rig} \times \widehat{T(L)})\]
where $X_{\tri}^{\square}(\overline{\rho})$ is the trianguline variety of \cite[\S~2.2]{BHS1}. And the isomorphism $\iota_{\Omega,ww_{0,L}(\textbf{h})}$ sends $\iota(x)$ in Corollary \ref{compploc1} exactly to the point $x_w$ in \cite[Thm.\ 4.2.3]{BHS3}. Note that the resulting points $\{x_w\}\subset X_{\tri}^{\square}(\overline{\rho})$ are distinct. 

(3) Note that, if $w\neq w_x$, then $\big((\ttr_{i,x}), \delta=\boxtimes_{i=1}^r z^{ww_0(\textbf{h})_i}\big)\in (\Spec \cZ_{\Omega})^{\rig} \times \widehat{Z_{L_P}(L)}$ is not a parameter of the $\Omega$-filtration induced by $(\ttr_{i,x})$ on $D_{\rig}(\rho_x)$ (compare with Remark \ref{remBPE} (2) and see Definition \ref{defOF} (2), (3) for a parameter of an $\Omega$-filtration).
\end{remark}

\section{The geometry of some schemes related to generalized Springer resolutions}\label{secGS}

In this section, we show some results of geometric representation theory concerning algebraic varieties which are a ``parabolic'' generalization of Grothendieck's and Springer's resolution of singularities. These results will be crucially used in \S~\ref{secMod} to describe the local rings of the Bernstein paraboline varieties at certain points.

\subsection{Preliminaries}\label{prel}

We let $G/E$ be a connected split reductive algebraic group. We recall/introduce certain schemes $X_P$ (where $P\subseteq G$ is a parabolic subgroup) which are related to parabolic versions of Grothendieck' simultaneous resolution of singularities for $G$.

We fix a Borel subgroup $B$ of $G$, and let $T\subset B$ be a maximal torus and $N\subset B$ the unipotent radical of $B$. Let $P\supseteq B$ be a parabolic subgroup of $G$, $L_P$ be the Levi subgroup of $P$ containing $T$ and $N_P$ the unipotent radical of $P$. Let $\fp$ be the Lie algebra of $P$ over $E$, $\fn_{P}$ its nilpotent radical, $\ur_{P}$ the full radical of $\fp$, $\fl_P$ the Lie algebra of $L_P$ over $E$, and $\fz_{L_P}$ the centre of $\fl_P$. We have $\ur_{P}\cong \fn_P \rtimes \fz_{L_P}$. Let $\ug$, $\ub$, $\fn$, $\ft$ be the Lie algebra over $E$ of $G$, $B$, $N$, $T$ respectively.

Let $\sW$ be the Weyl group of $G$ and $w_0\in \sW$ the element of maximal length. For $w\in \sW$, we also use $w$ to denote some lift\footnote{When we apply this notation for certain group operators, we always mean first applying group operations then taking a certain lift, e.g.\ we use $w_1w_2\in N_G(T)$ to denote a lift of $w_1w_2 \in \sW$ rather than the multiplication of a lift of $w_1$ and a lift of $w_2$.} of $w$ in $N_{G}(T)\subset G(E)$. As in \S~\ref{Nota2.1}, denote by $\sW^P_{\min}$ (resp.\ $\sW^P_{\max}$, resp.\ $\lWPmin$, resp.\ $\lWPmax$) the set of minimal (resp.\ maximal, resp.\ minimal, resp.\ maximal) length representatives of $\sW_{L_P}\backslash \sW$ (resp.\ $\sW_{L_P}\backslash \sW$, resp.\ $\sW/\sW_{L_P}$, resp.\ $\sW/\sW_{L_P}$). Note that $w\in \sW^P_{\min}$ if and only if $ww_0 \in \sW^P_{\max}$. Also $w\in \sW^P_{\min}$ (resp.\ $w\in \sW^P_{\max}$) if and only if $w^{-1}\in \lWPmin$ (resp.\ $w^{-1}\in \lWPmax$). For $w\in \sW$ or $\sW_{L_P}\backslash \sW$, denote by $w^{\min}\in \sW^P_{\min}$ (resp.\ $w^{\max}\in \sW^P_{\max}$) the corresponding representative of $\sW$. We use ``$\cdot$" to denote the \textit{dot} action of $\sW$ on the weights of $\ft$ (cf.\ \cite[Def. 1.8]{Hum08}). 

If $f: X \ra Y$ is a morphism of schemes and $Z$ a locally closed subscheme of $Y$, we denote $f^{-1}(Z):=X\times_Y Z$, which is a locally closed subscheme of $X$. If $X_1$, $X_2$ are locally closed subschemes of $X$, we denote $X_1\cap X_2:=X_1 \times_X X_2$.

Let $\tilde{\ug}_P$ be the closed $E$-subscheme of $G/P \times \ug$ defined by 
\begin{equation*}
	\{(gP, \psi) \in G/P \times \ug \ |\ \Ad(g^{-1}) \psi \in \ur_{P}\}
\end{equation*}
where as usual $\Ad(h)$ means the adjoint action of $h\in G$ (i.e.\ conjugation by $h$). We have an isomorphism of $E$-schemes (using standard notation, see e.g.\ \cite[\S~VI.8]{KiehlWeissauer}):
\begin{equation}\label{better}
	G \times^{P} \ur_{P} \xlongrightarrow{\sim}\tilde{\ug}_{P},\ (g, \psi) \mapsto (g P, \Ad(g)\psi).
\end{equation}
We see that $\tilde{\ug}_{P}$ is a vector bundle over $G/{P}$, hence is smooth and irreducible. We also have $\dim \tilde{\ug}_{P}=\dim G/{P} + \dim \ur_{P}=\dim \fn_{P}+ \dim \ur_{P}$. There are natural morphisms: 
\begin{equation*}
	\begin{array}{lll}
		\kappa_{P}: &\tilde{\ug}_{P} \lra \fz_{L_P} \ &(g P, \psi) \longmapsto \overline{\Ad(g^{-1}) \psi}\\
		q_{P}: &\tilde{\ug}_{P} \lra \ug \ &(g P, \psi) \longmapsto \psi \\
		\pi_P: &\tilde{\ug}_P \lra G/P \ &(gP, \psi) \longmapsto gP
	\end{array}
\end{equation*}
where $\overline{\Ad(g^{-1})}\psi$ is the image of $\Ad(g^{-1})\psi\in \ur_{P}$ via the canonical surjection $\ur_{P}\twoheadrightarrow \fz_{L_P}$. Alternatively we can see $q_P$ as the morphism $G \times^{P}\ur_{P}\ra \ug$, $(g, \psi) \longmapsto \Ad(g)\psi$. We denote $\tilde{\ug}:=\tilde{\ug}_{B}\simeq G \times^{B} \ub$ and define $q_B:\tilde{\ug}\ra \ug$ similarly to $q_P$. We put:
\[X_{P}:=\tilde{\ug} \times_{\ug} \tilde{\ug}_{P},\index{$X_P$}\]
which is also the closed subscheme of $G/B\times G/{P} \times \ug$ defined by 
\begin{equation*}
	X_P\simeq \{(g_1B, g_2 {P}, \psi) \in G/B \times G/{P} \times \ug\ |\ \Ad(g_1^{-1}) \psi \in \ub, \Ad(g_2^{-1}) \psi \in \ur_{P}\}.
\end{equation*}
It is easy to check that we have isomorphisms of $E$-schemes
\begin{equation}\label{Xpqpb}
	\begin{array}{llllll}
		X_P &\xlongrightarrow{\sim} &G \times^P q_B^{-1}(\ur_P) \ &(g_1 B, g_2P, \psi) &\longmapsto &\big(g_2, (g_2^{-1} g_1 B, \Ad(g_2^{-1}) \psi)\big)\\
		X_P& \xlongrightarrow{\sim} &G \times^B q_P^{-1}(\ub) \ &(g_1 B, g_2P, \psi) &\longmapsto &\big(g_1, (g_1^{-1} g_2P, \Ad(g_1^{-1}) \psi)\big).
	\end{array}
\end{equation}

\subsection{Analysis of the global geometry}\label{sec: grgg}

We prove useful statements on the geometry of the $E$-scheme $X_P$.

Let $\pi$ be the composition
\begin{equation}\label{pi}
\pi:X_{P} \hookrightarrow G/B \times G/{P} \times \ug \twoheadrightarrow G/B \times G/{P}.
\end{equation}
We equip $G/B \times G/{P}$ with an action of $G$ by diagonal left multiplication.
For $w\in \sW$, write
\[ U_{w}:=G (w, 1) (B \times {P})=G(1, w^{-1}) (B \times P) \subset G/B \times G/{P}.\]
Note that $U_w$ only depends on the the right coset $\sW_{L_P} w$ (or equivalently, the left coset $w^{-1} \sW_{L_P}$). In fact, we have an isomorphism
\[G \times^B (G/P) \xlongrightarrow{\sim} G/B \times G/P,\ (g_1, g_2P) \longmapsto (g_1B, g_1g_2P)\]
which induces an isomorphism
\[G\times^B (Bw^{-1}P/P) \xlongrightarrow{\sim} U_w.\]
Likewise, we have an isomorphism
\[G \times^P (G/B) \xlongrightarrow{\sim} G/B \times G/P,\ (g_1, g_2B) \longmapsto (g_1g_2 B, g_1P)\]
which induces an isomorphism
\[G \times^P (PwB/B) \xlongrightarrow{\sim} U_w.\]
By the standard Bruhat decomposition of $G/P$ and $G/B$, we deduce
\[G/B \times G/{P}=\sqcup_{w\in \sW^P_{\min}} U_{w}=\sqcup_{w\in \sW^P_{\max}} U_w.\]
We also deduce that $U_{w}$ is a locally closed subscheme of $G/B \times G/{P}$ which is smooth of dimension 
\begin{multline}\label{dimUw}
	\dim G-\dim B+\dim (Bw^{-1}P/P)=\dim G- \dim B + \lg(w^{\min})\\
	=\dim G-\dim P+\lg(w^{\max})=\dim G-\dim P +\dim (PwB/B)
\end{multline}
where, for the first equality, we use $\lg((w^{\min})^{-1})=\lg(w^{\min})$ and $(w^{\min})^{-1}\in \lWPmin$.

We let $V_{w}:=\pi^{-1}(U_{w})$.\index{$V_w$}

\begin{proposition}\label{geoVw}
	The surjection $V_{w} \twoheadrightarrow U_{w}$ induced by $\pi$ is a (geometric) vector bundle of relative dimension $\dim \ur_{P}-\lg(w^{\min})$.
\end{proposition}
\begin{proof}
	Let $y=(gB, gw^{-1}{P})=(gB, g(w^{\min})^{-1}{P})\in U_{w}\subset G/B \times G/{P}$. One can check that 
	\begin{equation}\label{fiber0}
		\pi^{-1}(y)=y \times \Ad(g)\Big(\Ad((w^{\min})^{-1})\fz_{L_P} \oplus \big(\fn \cap \Ad((w^{\min})^{-1}) \fn_{P}\big)\Big).
	\end{equation}
	Since $(w^{\min})^{-1}\in \lWPmin$, one has
	\begin{multline*}\dim \big(\fn \cap \Ad((w^{\min})^{-1}) \fn_{P}\big) =\dim \big(\fn \cap \Ad((w^{\min})^{-1})\fn\big)-\dim \fn_{L_{P}}\\ =\dim \fn-\lg((w^{\min})^{-1})-\dim \fn_{L_{P}}=\dim \fn_{P} -\lg(w^{\min}).
	\end{multline*} 
	The proposition then follows by the same argument as for \cite[Prop.\ 2.2.1]{BHS3}. 
\end{proof}

From (\ref{dimUw}) and Proposition \ref{geoVw} we obtain that $V_w$ is equidimensional of dimension $\dim G-\dim B+\dim \ur_{P}$. Let $X_{w}$ be the closed subscheme of $X_{P}$ defined as the {\it reduced} Zariski-closure of $V_{w}$ in $X_{P}$.\index{$X_w$} By the same argument as in the first part of the proof of \cite[Prop.\ 2.2.5]{BHS3}, we have:

\begin{corollary}\label{irrcmpXP}
The scheme $X_{P}$ is equidimensional of dimension 
$\dim G-\dim B+\dim \ur_{P}$.
The irreducible components of $X_P$ are $\{X_w\}_{w\in \sW^P_{\min}}=\{X_w\}_{w\in \sW^P_{\max}}$ and $V_w$ is open in $X_w$.
\end{corollary}

\begin{remark}\label{remXw}
(1) We could equip the underlying closed subset $X_w$ with another scheme structure, namely the scheme theoretic image of the open subscheme $X_P \setminus \cup_{\sW_{L_P}w'\neq \sW_{L_P}w} X_{w'}$ of $X_P$. Let us denote it by $X_w'$. When $P=B$, by \cite[Thm.\ 2.2.6]{BHS3}, $X_P$ is reduced, so $X_w=X_w'$ as closed subschemes of $X_P$. However, in general, it is not clear to the authors if $X_P$ is reduced, or if $X_P$ is Cohen-Macaulay (for instance one can easily check that the last paragraph of the proof of \cite[Prop.\ 2.2.5]{BHS3} does not extend to $P\ne B$).
	
(2) For $w\in \sW$ let us define $q_B^{-1}(\ur_{P})_w^0$ (resp.\ $q_P^{-1}(\ub)_w^0$) as the preimage of $PwB/B$ (resp.\ of $BwP/P$) via the composition $q_B^{-1}(\ur_P) \hookrightarrow G/B \times \ug \twoheadrightarrow G/B$ (resp.\ $q_P^{-1}(\ub) \hookrightarrow G/P \times \ug \twoheadrightarrow G/P$). By similar arguments as in Proposition \ref{geoVw}, the map $q_B^{-1}(\ur_P)_w^0 \ra PwB/B$ (resp.\ $q_P^{-1}(\ub)_w^0 \ra BwP/P$) is a vector bundle of relative dimension $\dim \ur_P-\lg(w^{\min})$. We define $q_B^{-1}(\ur_{P})_w$ (resp.\ $q_P^{-1}(\ub)_w$) to be the reduced Zariski-closure of $q_B^{-1}(\ur_P)_w^0$ (resp.\ $q_P^{-1}(\ub)_w^0$) in $q_B^{-1}(\ur_P)$ (resp.\ in $q_P^{-1}(\ub)$). The scheme $q_B^{-1}(\ur_P)$ (resp.\ $q_P^{-1}(\ub)$) is equidimensional of dimension $\dim P-\dim B+\dim \ur_P$ (resp.\ $\dim \ur_P$) with irreducible components given by $\{q_B^{-1}(\ur_P)_w\}_{w\in \sW^P_{\min}}$ (resp.\ $\{q_P^{-1}(\ub)_{w}\}_{w\in \lWPmin}$). Moreover $q_B^{-1}(\ur_{P})_w^0$ (resp.\ $q_P^{-1}(\ub)_w^0$) is open in $q_B^{-1}(\ur_{P})_w$ (resp.\ $q_P^{-1}(\ub)_w$). From (\ref{Xpqpb}) we also have isomorphisms of reduced $E$-schemes for $w\in \sW^P_{\min}$
	\begin{equation*}
		X_w\cong G\times^P q_B^{-1}(\ur_P)_w\cong G \times^B q_P^{-1}(\ub)_{w^{-1}}.
	\end{equation*}
\end{remark}

\begin{lemma}\label{w'w}
	For $w,w'\in \sW$ we have that $X_{w}\cap V_{w'}\neq \emptyset$ implies $w'^{\min}\leq w^{\min}$ ($\Leftrightarrow w'^{\max} \leq w^{\max} \Leftrightarrow w'^{\min} w_0 \geq w^{\min} w_0$).
\end{lemma}
\begin{proof}
	The lemma follows by the same argument as in the proof of \cite[Lemma 2.2.4]{BHS3}, noting that in $G \times^P (G/B)$ we have $\overline{U_{w}} \cap U_{w'}\neq \emptyset \Rightarrow \overline{PwB/B} \supset Pw'B/B \Rightarrow w^{\max} \geq w'^{\max} \Rightarrow w^{\min} \geq w'^{\min}$ (where $\overline{(-)}$ means Zariski-closure). 
\end{proof}

Denote by $\kappa_B: X_{P} \ra \ft$ (resp.\ $\kappa_P: X_P \ra \ft$) the morphism 
\begin{equation}\label{kappaBP}
(g_1 B, g_2 {P}, \psi) \longmapsto \overline{\Ad(g_1^{-1}) \psi}\ \ \big(\text{resp.\ \ }(g_1B, g_2 P, \psi)\longmapsto \overline{\Ad(g_2^{-1}) \psi}\big)
\end{equation}
where $\overline{\Ad(g_1^{-1})}\psi$ is the image of $\Ad(g_1^{-1})\psi\in \ub$ via $\ub\twoheadrightarrow \ft$ (and see \S~\ref{prel} for $\overline{\Ad(g_2^{-1})}\psi$). Note that $\kappa_P$ factors through $\fz_{L_P} \hookrightarrow \ft$. For $*=B,P$ denote by $\kappa_{*,w}$ the restriction of $\kappa_{*}$ to $X_{w}$. 

\begin{lemma}\label{LMwt}
	For $w\in \sW$ we have $\kappa_{P,w}=\Ad(w) \circ \kappa_{B,w}$, where $\Ad(w): \ft \ra \ft$ denotes the morphism induced by the adjoint action of $\sW$ on $\ft$. In particular, the following diagram commutes
	\begin{equation*}
		\begin{CD}
			X_{w} @> \kappa_{B,w} >> \ft \\
			@V \kappa_{P,w} VV @VVV \\
			\ft @>>> \ft/\sW
		\end{CD}
	\end{equation*}
	where $\ft/\sW:=\Spec R_{\ft}^{\sW}$ (with $\ft:=\Spec R_{\ft}$) and the two morphisms $\ft \ra \ft/\sW$ are both the canonical surjection.
\end{lemma}
\begin{proof}
	This is the argument for \cite[Lemma~2.3.4]{BHS3}, let us recall it. As $\ft/\sW$ is affine, it is separated, hence the diagonal embedding $\ft/\sW\rightarrow \ft/\sW\times_E\ft/\sW$ is a closed immersion, and so is $\ft\times_{\ft/\sW}\ft\rightarrow \ft\times_E\ft$ by base change along $\ft\times_E\ft\rightarrow \ft/\sW\times_E\ft/\sW$. Since the diagram clearly commutes with $V_{w}$ instead of $X_w$ and $V_w$ is Zariski-dense in $X_w$, the lemma follows.
\end{proof}

In particular, $\Ad(w) \circ \kappa_{B,w}$ only depends on the coset $\sW_{L_P} w$ as the same holds for $\kappa_{P,w}$. 

Consider the affine $E$-scheme $\sT_{P}:=\ft \times_{\ft/\sW} \fz_{L_P}$. We have a morphism of $E$-schemes
\[(\kappa_B, \kappa_P):X_P\longrightarrow \sT_{P}.\index{$\sT_P$}\]

\begin{lemma}\label{lemWT0}
	The irreducible components of $\sT_{P}$ are $\{\sT_{w}\}_{w\in \sW^P_{\min}}=\{\sT_w\}_{w\in \sW^P_{\max}}$ where $\sT_w$ is the reduced $E$-scheme:\index{$\sT_w$}
	\begin{equation*}
		\sT_{w}:=\{(\Ad(w^{-1}) z, z), \ z\in \fz_{L_P}\}
	\end{equation*}
	(so $\sT_w$ only depends on the coset $\sW_{L_P}w$). Moreover $X_{w}$ is the unique irreducible component of $X_{P}$ such that $(\kappa_B, \kappa_P)(X_{w})=\sT_{w}$.
\end{lemma}
\begin{proof}
	Since $\ft\rightarrow \ft/\sW$ is finite, we deduce $\dim \sT_{P}=\dim \fz_{L_P}=\dim \sT_{w}$ for any $w\in \sW$. It is also clear that $\sT_{w} \cong \fz_{L_P}$ is irreducible. The first part of the lemma easily follows. By Lemma \ref{LMwt}, $(\kappa_B, \kappa_P)(X_{w})\subset \sT_{w}$. From (\ref{fiber0}), we see that the restriction $(\kappa_B, \kappa_P)|_{V_{w}}: V_{w}\ra \sT_{w}$ is surjective. The second part of the lemma follows. 
\end{proof}

For a scheme $Y$ and a point $y\in Y$, recall that we denote $\co_{Y,y}$ the local ring of $Y$ at $y$, $\widehat{\co}_{Y,y}$ the completion of $\co_{Y,y}$ along its maximal ideal, and $\widehat Y_y:=\Spf \widehat{\co}_{Y,y}$ the associated formal scheme (whose underlying topological space is one point). 

\begin{lemma}\label{Theta0}
	Let $x$ be a closed point of $X_{P}$, $w, w'\in \sW$. Assume $x\in X_w$. Then the composition $\widehat{X}_{w,x} \hookrightarrow \widehat{X}_{P,x} \ra \widehat{\sT}_{P,(\kappa_B,\kappa_P)(x)}$ factors through $\widehat{\sT}_{w', (\kappa_B, \kappa_P)(x)}\hookrightarrow \widehat{\sT}_{{P},(\kappa_B,\kappa_P)(x)}$ if and only if $\sW_{L_P}w'=\sW_{L_P}w$.
\end{lemma}
\begin{proof}
	Let $y:=(\kappa_B, \kappa_P)(x)$, a closed point of ${\sT}_{P}$. We have a commutative diagram of local rings 
	\begin{equation}\label{diagLC}
		\begin{CD}\co_{\sT_{P},y} @>>> \co_{X_{P},x} @>>> \co_{X_w, x} \\ 
			@VVV @VVV @VVV \\
			\widehat{\co}_{\sT_{P},y} @>>> \widehat{\co}_{X_{P},x} @>>> \widehat{\co}_{X_w, x}
		\end{CD}
	\end{equation}
	where the vertical maps are injective by Krull's intersection theorem.
	By assumption, the bottom composition factors through $\widehat{\co}_{\sT_{w'}, y}$. Using the commutative diagram
	\begin{equation*}
		\begin{CD}
			\co_{\sT_{P},y} @>>> \co_{\sT_{w'}, y} \\ 
			@VVV @VVV \\
			\widehat{\co}_{\sT_{P},y} @>>> \widehat{\co}_{\sT_{w'},y}
		\end{CD}
	\end{equation*} 
	and the injectivity of all the vertical maps in (\ref{diagLC}), we deduce that the upper composition in (\ref{diagLC}) factors through $\co_{\sT{w'},y}$. In particular the map $X_P\rightarrow \sT_{P}$ sends the generic point of $X_w$ to the generic point of $\sT_{w'}$. By Lemma \ref{lemWT0}, we must have $w'=w$.
\end{proof}

\begin{remark}
Lemma \ref{Theta0} is the analogue of \cite[Lemma~2.5.2]{BHS3} where the normality of $X_w$ there was used in the proof. However, this normality is in fact useless there, arguing as in the above proof.
\end{remark}

Recall that $\psi \in \ur_P$ is called {\it regular} if the subgroup $C_G(\psi):=\{g\in G, \ \Ad(g)\psi=\psi\}$ satisfies $\dim C_G(\psi)\leq \dim C_G(\psi')$ for all $\psi'\in \ur_P$. We write $\ur_P^{\reg}$ for the subset of regular elements in $\ur_P$, which is preserved under the $P$-action, and $\fz_{L_P}^{\reg}:=\fz_{L_P}\cap \ur_P^{\reg}$. When $\psi\in \fz_{L_P}$ we have $\psi\in \fz_{L_P}^{\reg}$ exactly when $C_G(\psi)=L_P$. We say that $\psi \in \ur_P$ is {\it regular semi-simple} if $\psi\in P\fz_{L_P}^{\reg}$ (for the adjoint action of $P$ on $\ur_P$) and we write $\ur_P^{\reg-\sss}\subset \ur_P^{\reg}$ for the subset of regular semi-simple elements in $\ur_P$. It is well-known that both $\ur_P^{\reg-\sss}$ and $\ur_P^{\reg}$ are Zariski-open (Zariski-dense) in $\ur_P$ and that for each $\psi\in \ur_P^{\reg}$ we have $\dim C_G(x)=\dim L_P$. Moreover the product map gives an isomorphism of (smooth irreducible) $E$-schemes:
\begin{equation}\label{P/LP}
	P/L_P\times \fz_{L_P}^{\reg} \xlongrightarrow{\sim} \ur_P^{\reg-\sss}.
\end{equation}
Let $\tilde{\ug}_{P}^{\reg-\sss}$ be the Zariski-open (Zariski-dense) subset of $\tilde{\ug}_P$ corresponding to $G \times^P \ur_P^{\reg-\sss}$ via the isomorphism (\ref{better}), and $X_P^{\reg-\sss}:=\tilde{\ug}_P^{\reg-\sss} \times_{\ug} \tilde{\ug}$, which is Zariski-open in $X_P\cong \tilde{\ug}_P\times_{\ug} \tilde{\ug}$. As $V_w\cap X_P^{\reg-\sss}\neq \emptyset$ for any $w\in \sW$ (use (\ref{fiber0}) for instance), $X_P^{\reg-\sss}$ is Zariski-dense in $X_P$.

\begin{proposition}\label{genesmoo}
	The scheme $X_P^{\reg-\sss}$ is smooth over $E$. Moreover, the composition $X_P^{\reg-\sss}\! \!\hookrightarrow X_P \xrightarrow{\kappa_P} \fz_{L_P}$ is smooth. 
\end{proposition}
\begin{proof}
	Since $\fz_{L_P}$ is smooth over $E$, it is enough to prove the second statement. Let $q_P^{-1}(\ub)^{\reg-\sss}:=\tilde{\ug}_{P}^{\reg-\sss}\cap q_P^{-1}(\ub)=\tilde{\ug}_{P}^{\reg-\sss}\times_{\tilde{\ug}_{P}} q_P^{-1}(\ub)$, which is Zariski-open in $q_P^{-1}(\ub)$ and Zariski-closed in $\tilde{\ug}_{P}^{\reg-\sss}$. An argument similar to (\ref{Xpqpb}) gives an isomorphism of $E$-schemes
	\begin{equation}\label{isoregss}
		X_P^{\reg-\sss}\xlongrightarrow{\sim} G \times^B q_P^{-1}(\ub)^{\reg-\sss}.
	\end{equation}
	Seeing $q_P^{-1}(\ub)^{\reg-\sss}$ inside $G \times^P \ur_P^{\reg-\sss}$ via (\ref{better}), it is enough to prove that the composition
	\begin{equation}\label{kappareg}
		\kappa_P^{\reg-\sss}:q_P^{-1}(\ub)^{\reg-\sss}\hookrightarrow G \times^P \ur_P^{\reg-\sss} \longrightarrow \fz_{L_P},\ \ 
	\end{equation}
	is smooth, where the second map is $(g,\psi)\mapsto \overline\psi$ (note that $\kappa_P^{\reg-\sss}$ is $B$-equivariant with the trivial action of $B$ on the target). Indeed, under (\ref{isoregss}), the composition in the statement is the composition
	\[ G \times^B q_P^{-1}(\ub)^{\reg-\sss}\buildrel{\id\times \kappa_P^{\reg-\sss}}\over\longrightarrow G\times^B \fz_{L_P}\simeq G/B\times \fz_{L_P}\twoheadrightarrow \fz_{L_P}\]
	(where the last surjection is the canonical projection), which is smooth as both maps are.
	
	By (\ref{P/LP}), we have $G \times^P \ur_P^{\reg-\sss}\simeq G/L_P \times \fz_{L_P}^{\reg}$. Under this isomorphism, the last map in (\ref{kappareg}) is $(gL_P,\psi)\mapsto \psi$ (where $\psi\in \fz_{L_P}^{\reg}$), and $q_P^{-1}(\ub)^{\reg-\sss}$ is the closed subscheme of $G/L_P \times \fz_{L_P}^{\reg}$ defined by $\{(gL_P, \psi)\ |\ \Ad(g) \psi \in \ub\}$. Denote by $Y$ the inverse image of $q_P^{-1}(\ub)^{\reg-\sss}$ under the smooth surjective map $G \times \fz_{L_P}^{\reg}\twoheadrightarrow G/L_P \times \fz_{L_P}^{\reg}$, then by base change $Y\rightarrow q_P^{-1}(\ub)^{\reg-\sss}$ is also smooth surjective, and using \cite[\href{https://stacks.math.columbia.edu/tag/02K5}{Lemma 02K5}]{stacks-project} it is enough to prove that the composition $Y\hookrightarrow G \times \fz_{L_P}^{\reg}\twoheadrightarrow \fz_{L_P}^{\reg}$ is smooth. It is enough to prove this Zariski-locally on $Y$, and since $(Bw_0Nw)_{w\in \sW}$ is a Zariski-open covering of $G$, one can replace $G\times \fz_{L_P}^{\reg}$ by $Bw_0Nw\times \fz_{L_P}^{\reg}$ (for an arbitrary $w\in \sW$) and $Y$ by $Y \cap (Bw_0Nw \times \fz_{L_P}^{\reg})$. Since the multiplication induces an isomorphism of schemes $B \times w_0Nw \xlongrightarrow{\sim} Bw_0Nw$ and since $Y$ is $B$-invariant, we have $Y\simeq B \times Z$ for the closed subscheme $Z:=\{(w_0nw, \psi)\ |\ \Ad(w_0nw) \psi \in \ub\}$ of $w_0Nw \times \fz_{L_P}^{\reg}$. As the projection $B\times Z\twoheadrightarrow Z$ is smooth, it is enough to prove that $Z\rightarrow \fz_{L_P}^{\reg}, (w_0nw, \psi)\mapsto \psi$ is smooth.
	
	Since $\Ad(nw) \psi\in \ub$ and $\Ad(w_0)\ub \cap \ub =\ft$, it follows that
	\[Z=\{(w_0nw, \psi)\ |\ \Ad(w_0nw) \psi \in \ft\}\simeq \{(n, \psi)\in N\times \fz_{L_P}^{\reg}\ |\ \Ad(nw) \psi \in \ft\}.\]
	Since $\Ad(w) \psi \in \ft$ and the adjoint action of unipotent elements on $\ft$ doesn't change the diagonal entries, we have
	\begin{multline*}
		Z\simeq \{(n, \psi)\in N\times \fz_{L_P}^{\reg}\ |\ \Ad(nw) \psi =\Ad(w)\psi\}\simeq \{(n, \psi)\in N\times \fz_{L_P}^{\reg}\ |\ \Ad(w^{-1}nw) \psi =\psi\}\\
		\simeq w(N\cap L_P)w^{-1}\times \fz_{L_P}^{\reg}
	\end{multline*}
	where the last isomorphism follows from $C_G(\psi)=L_P$. As $w(N\cap L_P)w^{-1}$ is an (affine) smooth scheme, so is the projection $w(N\cap L_P)w^{-1}\times \fz_{L_P}^{\reg}\twoheadrightarrow \fz_{L_P}^{\reg}$, which finishes the proof.
\end{proof}

\subsection{Analysis of the local geometry}

In this paragraph we study the local geometry of $X_w$ at certain points.

\begin{theorem}\label{unibranch}
	Let $w\in \sW$ and $x=(g_1B, g_2P, 0)\in X_w$, then the scheme $X_w$ is unibranch at $x$.
\end{theorem}
\begin{proof}
	By Remark \ref{remXw} (2), it is enough to prove the same statement with the irreducible component $Y_w:=q_P^{-1}(\ub)_{w}$ of the closed $E$-subscheme $q_P^{-1}(\ub)$ of $G \times^P \ur_P$. We want to prove that the normalization of the (reduced) local ring of $Y_w$ at a point $(g,0)$ ($g\in G$) is still a local ring. The argument below is strongly inspired by the proof of \cite[Lemma~3.4.8]{LLHLM} and we give full details.
	
	We see the $E$-scheme ${\mathbb A}^1$ as an algebraic multiplicative monoid. The scheme $G \times \ur_P$ is endowed with a left action of ${\mathbb A}^1$ by $a(g,\psi):=(g,a\psi)$ (where $g\in G, \psi\in \ur_P$). As the adjoint action of $P$ on $\ur_P$ is linear, this action of ${\mathbb A}^1$ descends to a left action on $G \times^P \ur_P$. It also preserves $(Bw,\ur_P\cap \Ad(w^{-1})\ub)\subset G \times \ur_P$, hence its image $q_P^{-1}(\ub)_w^0$ in $G \times^P \ur_P$ (see Remark \ref{remXw} (2)), hence its Zariski-closure $Y_w$. Let $Y_w^{{\mathbb G}_{\rm m}}\subseteq Y_w$ be the closed subscheme (with its reduced structure) of fixed points by ${\mathbb G}_{\rm m}$, where ${\mathbb G}_{\rm m}$ is seen as a Zariski-open subgroup of the monoid ${\mathbb A}^1$. Since $a\psi=\psi\ \forall a\in {\mathbb G}_{\rm m}\Leftrightarrow \psi=0$, we have $(G \times^P \ur_P)^{{\mathbb G}_{\rm m}}\simeq G\times^P0\simeq G/P\times \{0\}$, and we easily deduce that $Y_w^{{\mathbb G}_{\rm m}}\simeq C_w\times \{0\}\subseteq G/P\times \{0\}$ where $C_w$ is the Zariski closure of $BwP/P$ in $G/P$. In particular we see that the action of ${\mathbb A}^1$ on $Y_w^{{\mathbb G}_{\rm m}}$ is also trivial.
	
	Consider the normalization map $f:\widetilde Y_w\twoheadrightarrow Y_w$, which is a finite surjective birational morphism between two noetherian irreducible $E$-schemes (see for instance \cite[\href{https://stacks.math.columbia.edu/tag/035E}{\S~035E}]{stacks-project} and \cite[\href{https://stacks.math.columbia.edu/tag/035E}{\S~0BXQ}]{stacks-project}). Note that $f$ is an isomorphism above $BwP/P\times \{0\}\subseteq Y_w^{{\mathbb G}_{\rm m}}$ as $BwP\times ^P \{0\}$ is contained in the image of $(Bw,\ur_P\cap \Ad(w^{-1})\ub)$ in $Y_w$, which is the {\it smooth open} set $q_P^{-1}(\ub)_w^0$ of $Y_w$ (see Remark \ref{remXw} (2)). By \cite[Rem.\ 3.4.3]{LLHLM}, it is enough to prove that the (geometric) fiber $f^{-1}(g,0)$ is a connected scheme for $g\in G$. This is {\it a fortiori} true if $f^{-1}(g,0)$ is geometrically connected, hence we can extend scalars from $E$ to an algebraic closure of $E$, which we do from now on in this proof (still using the notation $E$).
	
	The composition ${\mathbb A}^1\times \widetilde Y_w\longrightarrow {\mathbb A}^1\times Y_w\longrightarrow Y_w$ is surjective (as both maps are surjective, where the last map is the action of ${\mathbb A}^1$ on $Y_w$). By \cite[\href{https://stacks.math.columbia.edu/tag/035J}{Lemma~035J}]{stacks-project}, since ${\mathbb A}^1\times \widetilde Y_w$ is normal this composition factors as ${\mathbb A}^1\times \widetilde Y_w\longrightarrow \widetilde Y_w\longrightarrow Y_w$, which induces a canonical action of ${\mathbb A}^1$ on $\widetilde Y_w$ such that the map $f$ is ${\mathbb A}^1$-equivariant. By exactly the same argument, we also have an action of the Borel $B$ on $\widetilde Y_w$ such that $f$ is $B$-equivariant. Moreover these two actions of ${\mathbb A}^1$ and $B$ commute on $\widetilde Y_w$ as they do on $Y_w$. The action of ${\mathbb A}^1$ on $\widetilde Y_w$ is again trivial on the closed subscheme $\widetilde Y_w^{{\mathbb G}_{\rm m}}\subseteq \widetilde Y_w$ (with its reduced structure), but we need another argument than for $Y_w^{{\mathbb G}_{\rm m}}$. Consider the morphism of $E$-schemes:
	\begin{equation*}
		m:{\mathbb A}^1\times \widetilde Y_w \longrightarrow \widetilde Y_w\times \widetilde Y_w,\ (a,\widetilde y)\longmapsto (\widetilde y,a\widetilde y)
	\end{equation*}
	and see $\widetilde Y_w^{{\mathbb G}_{\rm m}}$ as a closed subscheme of $\widetilde Y_w\times \widetilde Y_w$ via the diagonal embedding. Then $m^{-1}(\widetilde Y_w^{{\mathbb G}_{\rm m}})$ is a closed subscheme of ${\mathbb A}^1\times \widetilde Y_w$ which obviously contains the locally closed subscheme ${\mathbb G}_{\rm m}\times \widetilde Y_w^{{\mathbb G}_{\rm m}}$. Hence it also contains its Zariski-closure, which is ${\mathbb A}^1\times \widetilde Y_w^{{\mathbb G}_{\rm m}}$. In particular the action of ${\mathbb A}^1$ on $\widetilde Y_w$ is trivial on $\widetilde Y_w^{{\mathbb G}_{\rm m}}$.
	
	We prove that we have an isomorphism $\widetilde Y_w^{{\mathbb G}_{\rm m}}\xlongrightarrow{\sim} f^{-1}(Y_w^{{\mathbb G}_{\rm m}})^{\rm red}$ of (reduced) closed subschemes of $\widetilde Y_w$. Since $f$ is ${\mathbb G}_{\rm m}$-equivariant we have a closed embedding $\widetilde Y_w^{{\mathbb G}_{\rm m}}\hookrightarrow f^{-1}(Y_w^{{\mathbb G}_{\rm m}})^{\rm red}$, hence it is enough to prove that ${\mathbb G}_{\rm m}$ acts trivially on $f^{-1}(Y_w^{{\mathbb G}_{\rm m}})^{\rm red}$ inside $\widetilde Y_w$. Since we are over an algebraically closed field $E$, the action of ${\mathbb G}_{\rm m}$ on $f^{-1}(Y_w^{{\mathbb G}_{\rm m}})^{\rm red}$ is trivial if and only if the action of $E^\times$ is trivial on the set of $E$-points $f^{-1}(Y_w^{{\mathbb G}_{\rm m}})(E)$ (see for instance Remark $4$ in \cite[p.~38]{fog}). Since any such $E$-point is in $f^{-1}(g,0)(E)$ for some $(g,0)\in Y_w^{{\mathbb G}_{\rm m}}$, it is enough to prove that $E^\times$ acts trivially on $f^{-1}(g,0)(E)$. Since $f$ is a finite ${\mathbb G}_{\rm m}$-equivariant morphism, $f^{-1}(g,0)(E)$ is a finite set stable under the action of $E^\times$. Hence there is an integer $n\gg 0$ such that $x^n$ acts trivially on $f^{-1}(g,0)(E)$ for any $x\in E^\times$ (for instance $n=|f^{-1}(g,0)(E)|!$). But as $E$ is assumed algebraically closed, any element in $E^\times$ is of the form $x^n$, hence $E^\times$ acts trivially on $f^{-1}(g,0)(E)$. 
	
	We now consider the morphism $0:\widetilde Y_w\longrightarrow \widetilde Y_w^{{\mathbb G}_{\rm m}},\ x\longmapsto 0.x$ where $0\in {\mathbb A}^1$. Since ${\mathbb A}^1$ acts trivially on $\widetilde Y_w^{{\mathbb G}_{\rm m}}\subset \widetilde Y_w$, the morphism $0$ is surjective. As $\widetilde Y_w$ is irreducible and the image of an irreducible set is irreducible, it follows that the closed subset $\widetilde Y_w^{{\mathbb G}_{\rm m}}$ of $\widetilde Y_w$ is irreducible. Since $\widetilde Y_w^{{\mathbb G}_{\rm m}}\xrightarrow{\sim} f^{-1}(Y_w^{{\mathbb G}_{\rm m}})^{\rm red}$, we deduce that $f$ induces a finite birational surjective morphism of noetherian irreducible $E$-schemes $f^{-1}(Y_w^{{\mathbb G}_{\rm m}})^{\rm red}\rightarrow Y_w^{{\mathbb G}_{\rm m}}$ (it is birational since it is an isomorphism above the open subset $BwP/P\times \{0\}$ of $Y_w^{{\mathbb G}_{\rm m}}\simeq C_w\times \{0\}$). But $C_w$, and hence $Y_w^{{\mathbb G}_{\rm m}}$, are normal schemes by \cite{MS}. One then uses Zariski's connectedness theorem (see e.g.\ \cite[\S~III.9.V]{Mu}) applied to the morphism $f^{-1}(Y_w^{{\mathbb G}_{\rm m}})^{\rm red}\rightarrow Y_w^{{\mathbb G}_{\rm m}}$, which implies that all the fibers $f^{-1}(g,0)$ are connected schemes and finishes the proof.
\end{proof}

\begin{remark}
The proof of Theorem \ref{unibranch} only works for the points $(g,0)$ of $Y_w$. In particular, when $P\ne B$, we do not know the unibranch property of $Y_w$ at points $(g,\psi)\in Y_w$ where $\psi\in \ur_P\setminus\{0\}$ is not regular semi-simple (the regular semi-simple case being a consequence of Proposition \ref{genesmoo}). Recall that, when $P=B$, this is known since $Y_w$ is normal at every point (\cite[Thm.~2.3.6]{BHS3}).
\end{remark}

If $x\in X_w$ is a closed point, recall that the tangent space $T_x X_w$ of $X_w$ at $x$ can be identified with the $k(x)[\varepsilon]/(\varepsilon^2)$-points of $X_w$ mapping to $x$ via $k(x)[\varepsilon]/(\varepsilon^2)\twoheadrightarrow k(x)$, $\varepsilon \mapsto 0$. Recall also that $\dim_{k(x)} T_x X_w\geq \dim X_w=\dim(G/B) + \dim \ur_{P}$.

\begin{proposition}\label{propsmoX}
	Assume $x=(g_1 B, g_2P, 0) \in X_w \cap V_{w'}$ (which implies $w'^{\max}\leq w^{\max}$ by Lemma \ref{w'w}), then
	\begin{equation*}
		\dim_{k(x)} T_x X_w \leq \dim T_{\pi(x)} \overline{U_w}+\dim \fz_{L_P}^{w(w')^{-1}}+\lg((w')^{\max} w_0)
	\end{equation*}
	where $\fz_{L_P}^{w(w')^{-1}}:=\{z \in \fz_{L_P} \ |\ \Ad(w(w')^{-1})z=z\}$. \index{$\fz_{L_P}^{w(w')^{-1}}$}
\end{proposition}	
\begin{proof}
	The proposition follows by similar arguments as in the proof of \cite[Prop.\ 2.5.3]{BHS3}, to which we refer the reader for more details (e.g.\ on the notation). The closed embedding $X_{P} \hookrightarrow G /B \times G/P \times \ug$ induces a closed embedding $X_w \hookrightarrow \overline{U_w} \times \ug$. Let $\widehat {x}=(\widehat {g}_1 B, \widehat {g}_2 P, \psi\varepsilon)\in T_x X_w$, where we see $(\widehat {g}_1 B, \widehat {g}_2 P)\in T_{\pi(x)} \overline{U_w}$. As $\widehat {x}\in X_{P}(k(x)[\varepsilon]/(\varepsilon^2))$, we have $\Ad(g_1^{-1})\psi \in \ub$ and $\Ad(g_2^{-1}) \psi \in \ur_{P}$. As $x\in V_{w'}$, there exists $g\in G$ such that $(g_1 B, g_2 P)=(g w' B, gP)$. Replacing $g_1$ by $g w'$, and $g_2$ by $g$, we have thus 
	\begin{equation}\label{Vw'Xw0}
		\Ad(g^{-1}) \psi \in \ur_{P} \cap \Ad(w') \ub.
	\end{equation}
By Lemma \ref{LMwt}, we have $\kappa_P(\widehat{x})=\Ad(w) \kappa_{B}(\widehat{x})$ and hence 
	\begin{equation}\label{Vw'Xw1}
	\overline{\Ad(g^{-1})\psi}=\Ad(w) \overline{\Ad((w')^{-1} g^{-1}) \psi}\in \fz_{L_P}.
	\end{equation}
Writing $\Ad(g^{-1}) \psi=\lambda+\eta$ with $\lambda\in \fz_{L_P}$ and $\eta \in \fn_{P}$, we deduce from (\ref{Vw'Xw0}) and (\ref{Vw'Xw1}): 
	\begin{equation}\label{condpsi}
	\eta\in \fn_{P} \cap \Ad(w') \fn\textrm{\ \ and\ \ }\lambda\in \fz_{L_P}^{w(w')^{-1}}.
	\end{equation}
We have $\dim (\fn_{P} \cap \Ad(w') \fn)=\dim (\fn \cap \Ad(w'^{\max}) \fn)=\lg(w'^{\max} w_0)$ (where the first equality follows from $\fn_{L_P}\cap \Ad(w'^{\max}) \fn=0$). Together with (\ref{condpsi}), the proposition follows.
\end{proof}

\begin{corollary}\label{corosmo}
With the notation of Proposition \ref{propsmoX}, assume that $\overline{U_w}$ is smooth at the point $\pi(x)$ and that
	\begin{equation}\label{reffzPL}
		\dim \fz_{L_P}^{w(w')^{-1}}+\lg(w^{\max})-\lg(w'^{\max})=\dim \fz_{L_P}.
	\end{equation}
	Then $X_w$ is smooth at $x$.
\end{corollary}
\begin{proof}
	Recall $ \dim \overline{U_w}=\dim G-\dim B+\lg(w^{\min})$. Under the assumptions in the statement, we have by Proposition \ref{propsmoX} (and using $\lg(w^{\max})-\lg(w^{\min})=\dim \fn_{L_P}$):
	\begin{multline*}
		\dim_{k(x)} T_x X_w \\
		\leq \dim G-\dim B+\lg(w^{\min})+\dim \fz_{L_P}-\big(\lg(w^{\max})-\lg((w')^{\max})\big)+\lg(w_0)-\lg((w')^{\max}) \\
		\!\!\!\!\!\!\!\!\!\!=\dim G -\dim B -\dim \fn_{L_P}+\dim \fz_{L_P}+\lg(w_0) =\dim G-\dim B+\dim \fn_{P}+\dim \fz_{L_P}\\
		\ \ \ \ =\dim X_w.\ \ \ \ \ \ \ \ \ \ \ \ \ \ \ \ \ \ \ \ \ \ \ \ \ \ \ \ \ \ \ \ \ \ \ \ \ \ \ \ \ \ \ \ \ \ \ \ \ \ \ \ \ \ \ \ \ \ \ \ \ \ \ \ \ \ \ \ \ \ \ \ \ \ \ \ \ \ \ \ \ 
	         \ \ \ \ \ \ \ \ \ \ \ \ \ \ \ \ \ \ \ \ \ \ \ \ 
	\end{multline*}
	The corollary follows. 
\end{proof}

\begin{remark}\label{remsmo1}
(1) By \cite[Thm.~6.0.4]{BiLa} and \cite[Cor.~6.2.11]{BiLa}, if $\lg(w'^{\max})\geq \lg(w^{\max})-2$, then $\pi(x)$ is a smooth point of $\overline{U_w}$. Indeed, by \textit{loc.\ cit.}, under this assumption, $\overline{PwB/B}=\overline{Bw^{\max}B/B}$ is smooth at any point of $Bw'^{\max} B/B$. Then using the $P$-action, we deduce that $\overline{PwB/B}$ is smooth at any point of $Pw'B/B$, hence $\overline{U_w}$ is smooth at any point of $U_w'$. If moreover $\lg(w'^{\max})\geq \lg(w^{\max})-1$, it is clear that (\ref{reffzPL}) holds, so $X_w$ is smooth at $x$. When $\lg(w'^{\max})=\lg(w^{\max})-2$ however, (\ref{reffzPL}) does not necessarily hold when $P\neq B$: for example, when $P$ is a maximal parabolic subgroup, then $\dim \fz_{L_P}^{w(w')^{-1}}=1=\dim \fz_{L_P}-1<\lg(w^{\max})-\lg(w'^{\max})$. See Remark \ref{bruhInv} for related discussions.
	
(2) The proof of Proposition \ref{propsmoX} actually shows that its statement and the one of Corollary \ref{corosmo} hold with $X_w$ replaced by $X_w'$ in Remark \ref{remXw} (1). In particular, the closed immersion $X_w\hookrightarrow X_w'$ is an isomorphism on local rings at points satisfying the assumptions in Corollary \ref{corosmo}.
\end{remark}

\subsection{Characteristic cycles} \label{secCCyc}

We study the fibres $\kappa_{P,w}^{-1}(\{0\})\hookrightarrow X_w$, and show that they are closely related to the characteristic cycles of certain $G$-equivariant regular holonomic $D$-modules over $G/B \times G/P$.
\index{$Z_P$}

Let $Z_P:=\kappa_B^{-1}(\{0\})^{\red}=\kappa_P^{-1}(\{0\})^{\red}\hookrightarrow X_P$\index{$Z_P$} (see (\ref{kappaBP}), one may view $Z_P$ as a generalized Steinberg variety). Let $\cN\subset \ug$ denote the nilpotent cone and put $\widetilde{\cN}_P:=\{(gP, \psi) \in G/P \times \cN\ | \ \Ad(g^{-1}) \psi \in \fn_P\}$. We have $\widetilde{\cN}_P\cong G \times^P \fn_P$, $(gP, \psi) \mapsto (g, \Ad(g^{-1})\psi)$. We write $\widetilde{\cN}:=\widetilde{\cN}_B$ (defined as $\widetilde{\cN}_P$ with $B$ instead of $P$). The morphism $q_B$ (resp.\ $q_P$) in \S~\ref{prel} restricts to the so-called Springer resolution (resp.\ generalized Springer resolution): $q_B: G \times^B \fn \ra \cN$, $(g, \psi) \mapsto \Ad(g) \psi$ (resp.\ $q_P: G \times^P \fn_P \ra \cN$, $(g,\psi)\mapsto \Ad(g) \psi$). We have
\begin{equation*}
	Z_P \cong G \times^P q_B^{-1}(\fn_P)^{\red} \cong G \times^B q_P^{-1}(\fn)^{\red}.
\end{equation*}
The morphism $\pi$ (see (\ref{pi})) restricts to a natural morphism $\pi_Z: Z_P \ra G/B \times G/P$. Similarly to \S~\ref{sec: grgg}, for $w\in \sW$ we put $V_w':=\pi_Z^{-1}(U_w)=V_w \cap Z_P$, and let $Z_w$ be the Zariski-closure of $V_w'$ in $Z_P$ with the reduced scheme structure. We put $q_B^{-1}(\fn_P)_w^0:=q_B^{-1}(\ur_P)_w^0 \cap q_B^{-1}(\fn_P)$ and $q_P^{-1}(\fn)_w^0:=q_P^{-1}(\ub)_w^0 \cap q_P^{-1}(\fn)$ (see Remark \ref{remXw} (2) for $q_B^{-1}(\ur_P)_w^0$ and $q_P^{-1}(\ub)_w^0$). We define $q_B^{-1}(\fn_P)_w$ (resp.\ $q_P^{-1}(\fn)_w$) as the Zariski-closure of $q_B^{-1}(\fn_P)_w^0$ (resp.\ of $q_P^{-1}(\fn)_w^0$) in $q_B^{-1}(\fn_P)$ (resp.\ in $q_P^{-1}(\fn)$). By similar arguments as in \S~\ref{sec: grgg}, we have

\begin{proposition}
	(1) The scheme $Z_P$ (resp.\ $q_P^{-1}(\fn)$, resp.\ $q_B^{-1}(\fn_P)$) is equidimensional of dimension $\dim \fn+\dim \fn_P$ (resp.\ $\dim \fn_P$, resp.\ $\dim \fn$) with set of irreducible components given by $\{Z_w\}_{w\in \sW/\sW_{L_P}}$ \big(resp.\ $\{q_P^{-1}(\fn)_w\}_{w\in \sW/\sW_{L_P}}$, resp.\ $\{q_B^{-1}(\fn_P)_w\}_{w\in \sW_{L_P}\backslash \sW}$\big).
	
	(2) We have $Z_w\cong G\times^P q_B^{-1}(\fn_P)^{\red}_w \cong G \times^B q_P^{-1}(\fn)^{\red}_{w^{-1}}$.
\end{proposition}

\begin{remark}\label{stenBP}
One can show that $q_B^{-1}(\fn)$ is also equidimensional of dimension $\dim \fn$, hence the subscheme $q_B^{-1}(\fn_P)^{\red}$ of $q_B^{-1}(\fn)$ is isomorphic to a union of irreducible components of $q_B^{-1}(\fn)$. The irreducible components of $q_B^{-1}(\fn)$ are $\{q_B^{-1}(\fn)_w\}_{w\in \sW}$, where $q_B^{-1}(\fn)_w$ is the Zariski closure of the preimage $q_B^{-1}(\fn)_w^0$ of $BwB/B$ via $q_B^{-1}(\fn) \ra G/B$. Using similar argument as in the proof of \cite[Prop.\ 2.2.1]{BHS3} (see in particular \cite[(2.6)]{BHS3}), one can show that $q_B^{-1}(\fn)_w^0\subseteq q_B^{-1}(\fn_P)$ if and only if $w\in \sW_{\max}^P$. If so we have $q_B^{-1}(\fn)_w^0=q_B^{-1}(\fn_{P})_w^0$. Hence we deduce $q_B^{-1}(\fn_P)_w=q_B^{-1}(\fn)_{w^{\max}}$ for $w\in \sW$.
\end{remark}

By exactly the same argument as in the proof of Theorem \ref{unibranch} replacing everywhere $\ur_P$ by $\fn_P$, we have the following theorem which is interesting in its own right (and which is new even for $P=B$, see \cite[Rem.\ 2.4.2]{BHS3}):

\begin{theorem}\label{unibranch2}
	Let $w\in \sW$ and $x=(g_1 B, g_2 P, 0)\in Z_{w}$, then the scheme $Z_w$ is unibranch at $x$.
\end{theorem}

For a finite type $E$-scheme $Z$, denote by $Z^0(Z)$ the free abelian group generated by the irreducible components $\{Z_i\}$ of $Z$.\index{$Z^0(Z)$} Given a scheme $Y$ whose underlying topological space is a union of irreducible components of $Z$, put\index{$[Y]$}
\begin{equation}\label{cha00}
[Y]:=\sum_{i} m(Z_i,Y)[Z_i]
\end{equation}
where $m(Z_i,Y)$ is the length of the $\co_{Y,\eta_i}$-module $\co_{Y,\eta_i}$ at the generic point $\eta_i$. 

For $w\in \sW$, let $\overline{X}_w:=\kappa_{P,w}^{-1}(\{0\})=X_w \times_{\kappa_{P,w}, \fz_{L_P}} \{0\}$, where ``$\times_{\kappa_{P,w}, \fz_{L_P}}$" means taking the fibre product over $\fz_{L_P}$ via the morphism $\kappa_{P,w}:X_w\ra \fz_{L_P}$. Note that we do {\it not} take the reduced associated scheme. By the same argument as in page 320 of \cite{BHS3}, $\overline{X}_w$ is equidimensional of dimension $\dim Z$. So each irreducible component of $\overline{X}_w$ is $Z_{w'}$ for some $w'\in \sW$. We have the following easy fact.

\begin{lemma}\label{mul0}
	We have $m(Z_w, \overline{X}_w)=1$, and $m(Z_{w'}, \overline{X}_w)\geq 1$ implying $w'^{\max} \leq w^{\max}$.
\end{lemma}
\begin{proof}
	By the proof of Proposition \ref{geoVw}, we have $V_w \times_{\kappa_{P,w}, \fz_{L_P}}\{0\}\cong V_w \cap Z_P=V_w'$. The first part follows. The second part follows easily from Corollary \ref{w'w}.
\end{proof}
We construct some other cycles which are closely related to (the characteristic cycles of) parabolic Verma modules (see Proposition \ref{equcycl} below). Let $\fz_{L_P}^{\sd}\subset \fz_{L_P}$ be the set of strictly dominant integral weight $\lambda=(\lambda_1, \dots, \lambda_r)\in \fz_{L_P}$, i.e.\ $\lambda_i\in \Z$ and $\lambda_i>\lambda_{i'}$ for $i>i'$. In particular $\fz_{L_P}^{\sd}\subset \fz_{L_P}^{\reg}$ (see the end of \S~\ref{sec: grgg} for $\fz_{L_P}^{\reg}$). This assumption on $\lambda$ will be used in the proof of Proposition \ref{cycbc} in Appendix \ref{appenB}. Let $\fz_{\lambda}:=E\lambda\hookrightarrow \fz_{L_P}$ (a one-dimensional vector subspace of $\fz_{L_P}$) and $\ur_{P,\lambda}:=\fz_{\lambda} +\fn_P \hookrightarrow \ur_P$. Define:
\begin{eqnarray*}
\tilde{\ug}_{P,\lambda}&:=&G\times^P \ur_{P,\lambda} \hookrightarrow G\times^P \ur_P,\\
X_{P,\lambda}&:=&\tilde{\ug}_{P,\lambda} \times_{\ug} \tilde{\ug}\hookrightarrow X_P.
\end{eqnarray*}
Let $\pi_{\lambda}$ be the restriction of $\pi$ to $X_{P,\lambda}$. For $w\in \sW$ put $V_{w,\lambda}:=\pi^{-1}_{\lambda}(U_w)$, which is a vector bundle of relative dimension $\dim \ur_{P,\lambda}-\lg(w^{\min})$ over $U_w$. We let $X_{w,\lambda}$ be the reduced Zariski closure of $V_{w,\lambda}$ in $X_{P,\lambda}$. We have that $X_{P,\lambda}$ is equidimensional of dimension $\dim G- \dim P +\dim \ur_{P,\lambda}$ with irreducible components given by $\{X_{w,\lambda}\}$. Let $\overline{X}_{w,\lambda}:=X_{w,\lambda} \times_{\kappa_{P,w}, \fz_{\lambda}} \{0\}$.

\begin{lemma}\label{XwXwl}
(1) We have $m(Z_w, \overline{X}_{w,\lambda})=1$, and $m(Z_{w'},\overline{X}_{w,\lambda})\geq 1$ implies $w'^{\max} \leq w^{\max}$.
	
(2) We have $X_{w,\lambda}\cong (X_w \times_{\kappa_{P,w}, \fz_{L_P}} \fz_{\lambda})^{\red}$. Consequently, for $w'^{\max} \leq w^{\max}$, $m(Z_{w'},\overline{X}_w)\geq m(Z_{w'},\overline{X}_{w,\lambda})$. Moreover $m(Z_{w'},\overline{X}_w)>0$ if and only if $m(Z_{w'},\overline{X}_{w,\lambda})>0$. 
\end{lemma}
\begin{proof}
	(1) follows by the same argument as for Lemma \ref{mul0}. We have a natural closed immersion $X_{w,\lambda}\hookrightarrow (X_w \times_{\kappa_{P,w}, \fz_{L_P}} \fz_{\lambda})^{\red}$. By the same argument as on Page 320 of \cite{BHS3}, each irreducible component of $X_w \times_{\kappa_{P,w}, \fz_{L_P}} \fz_{\lambda}$ has dimension $\dim X_w-(\dim \fz_{L_P}-\dim \fz_{\lambda})=\dim X_{w,\lambda}$ and is thus some $X_{w',\lambda}$ for $w'\in \sW$. However, using Lemma \ref{lemWT0}, it is easy to see that, if $\sW_{L_P}w'\neq \sW_{L_P}w$, then $V_{w',\lambda}$ can not be contained in $(X_w \times_{\kappa_{P,w}, \fz_{L_P}} \fz_{\lambda})^{\red}$. The first part of (2) follows. 
	Together with the isomorphism $\overline{X}_w \cong (X_w\times_{\kappa_{P,w}, \fz_{L_P}} \fz_{\lambda}) \times_{\kappa_{P,w}, \fz_{\lambda}} \{0\}$, the other parts also follows.
\end{proof}

\begin{remark}\label{remXwXlam}
If $P$ is maximal, then $\fz_{\lambda}+\fz=\fz_{L_P}$. We have in this case $X_{w}\cong \fz \times X_{w,\lambda}$, and hence $\overline{X}_w\cong \overline{X}_{w,\lambda}$. If $P=B$, then $X_w$ is Cohen-Macaulay by \cite[Prop.\ 2.3.3]{BHS3}. As $X_w \times_{\kappa_{P,w}, \fz_{L_P}} \fz_{\lambda}$ is cut out by $(\dim X_w-\dim X_w \times_{\kappa_{P,w}, \fz_{L_P}} \fz_{\lambda})$-equations in $X_w$, we can deduce that $X_w \times_{\kappa_{P,w}, \fz_{L_P}} \fz_{\lambda}$ is also Cohen-Macaulay. Using Proposition \ref{genesmoo} and similar arguments in the proof of \cite[Thm.\ 2.2.6]{BHS3}, we can show	that $X_w \times_{\kappa_{P,w}, \fz_{L_P}} \fz_{\lambda}$ is reduced, hence equal to $X_{w,\lambda}$. By the proof of Lemma \ref{XwXwl}, we have in this case $\overline{X}_{w,\lambda}\cong \overline{X}_w$. 
\end{remark}

Let $\co^{\fp}$ be the parabolic BGG category $\co$ associated to $\fp$ (cf.\ \cite[\S~9.3]{Hum08}) and $\co^{\fp}(0)$ the full subcategory of $\co^{\fp}$ consisting of objects with trivial infinitesimal character. Let $\Mod_{\rh}(D_{G/B}, P)$ be the category of regular holonomic $P$-equivariant $D$-modules over $G/B$, which is the same as the category of coherent $P$-equivariant $D$-modules over $G/B$ by \cite[Thm.\ 11.6.1]{HTT}. By the Beilinson-Bernstein localization, we have an equivalence of categories with inverse given by taking global sections (see for example \cite[Thm.\ 11.5.3]{HTT}):
\begin{equation*}
	\Loc_{\BB}:\co^{\fp}(0) \xlongrightarrow{\sim}\Mod_{\rh}(D_{G/B},P) .
\end{equation*}	
Let $\Mod_{\rh}(D_{G/P}, B)$ (resp.\ $\Mod_{\rh}(D_{G/B\times G/P}, G)$) be the category of regular holonomic $B$-equiva\-riant (resp.\ $G$-equivariant for the diagonal $G$-action) $D$-modules over $G/P$ (resp.\ over $G/B \times G/P$). Put
\begin{eqnarray*}
	i_B: G/B \hooklongrightarrow G/B \times G/P, \ gB \mapsto (gB, P),\\
	i_P: G/P \hooklongrightarrow G/B \times G/P, \ gP \mapsto (B, gP).
\end{eqnarray*}
As in \cite[Prop.\ 13.1.1]{HTT} (see also \cite[Lemma 1.4]{Tan}), we have $R^j i_B^* \fM=0$ (resp.\ $R^j i_P^* \fM=0$) for $\fM\in \Mod_{\rh}(D_{G/B},P)$ (resp.\ for $\fM\in \Mod_{\rh}(D_{G/P},B)$) and $j>0$. Moreover the functor $i_B^*$ (resp.\ $i_P^*$) induces an equivalence of categories
$\Mod_{\rh}(D_{G/B \times G/P}, G) \xrightarrow{\sim} \Mod_{\rh}(D_{G/B}, P)$ (resp.\ $ \Mod_{\rh}(D_{G/B \times G/P}, G) \xrightarrow{\sim}\Mod_{\rh}(D_{G/P}, B)$). 

For a smooth algebraic variety $Y$, let $T^* Y$ be the cotangent bundle. For a regular holonomic $D$-module $\fM$ over $Y$, denote by $\Ch(\fM) \subset T^* Y$ the associated characteristic variety (cf.\ \cite[\S~2.2]{HTT}).\index{$\Ch(\fM)$}
We have $T^* G/B \cong G\times^B (\ug/\ub)^{\vee} \cong G\times^B \fn$, and $T^* G/P\cong G \times^P (\ug/\fp)^{\vee}\cong G\times^P \fn_P$ (see for example \cite[Lemma 1.4.9]{ChGi} where we identify $\ug$ with $\ug^{\vee}$ using the Killing form, cf.\ \cite[\S~5]{HumIntro}). By the same argument as in the proof of \cite[Prop.\ 2.4.4]{BHS3}, we have:

\begin{proposition}\label{cycbc}
For $\fM\in \Mod_{\rh}(D_{G/B \times G/P}, G)$
\[\Ch(\fM)\cong G \times^B \Ch(i_P^* \fM) \cong G\times^P \Ch(i_B^* \fM) \subseteq T^*(G/B \times G/P)\]
is equidimensional of dimension $\dim Z_P=\dim G/B+\dim G/P$. In particular, the underlying topological space of $\Ch(\fM)$ is a union of irreducible components of $Z_P$. 
\end{proposition}

For $\fM\in \Mod_{\rh}(D_{G/B \times G/P},G)$, $[\Ch(\fM)]\in Z^0(Z_P)$ is thus well-defined. 
For $w\in \sW$, let
\[M_P(w^{\max} w_0 \cdot 0):=\text{U}(\ug) \otimes_{\text{U}(\fp)} (w^{\max}w_0 \cdot 0) \in \co^{\fp}(0)\]
be the parabolic Verma module, and $L(w^{\max} w_0\cdot 0)$ be the (unique) simple quotient of $M_P(w^{\max} w_0 \cdot 0)$ in $\co^{\fp}(0)$. Denote by $\fM_P(w^{\max}w_0\cdot 0)$ (resp.\ $\fL(w^{\max}w_0 \cdot 0)$) the $D$-module over $G/B$ associated to $M_P(w^{\max} w_0 \cdot 0)$ (resp.\ $L(w^{\max} w_0\cdot 0)$) via $i_B^* \circ \Loc_{\BB}$. The following proposition will be proved in Appendix \ref{appenB}.

\begin{proposition}\label{equcycl}
	Let $\lambda\in \fz_{L_P}^{\sd}$, then we have $[\overline{X}_{w,\lambda}]=[\Ch(\fM_P(w^{\max}w_0 \cdot 0))]$ for all $w\in \sW$.
\end{proposition}

\begin{conjecture}\label{conjCC}
For $w\in \sW$, we have $[\overline{X}_{w}]=[\Ch(\fM_P(w^{\max}w_0 \cdot 0))]$.
\end{conjecture}

\begin{remark}
Conjecture \ref{conjCC} holds for $P=B$ by \cite[(6.2.3)]{Gin86} or \cite[Prop.\ 2.14.2]{BeRi} (see Remark \ref{remXwXlam}). It also holds in the case $P$ is maximal by Proposition \ref{equcycl} and Remark \ref{remXwXlam}.
\end{remark}

\begin{theorem}\label{thmcycl}
(1) The classes $\{Z_w\}$, $\{[\Ch(\fM_P(w^{\max} w_0 \cdot 0))]\}$, $\{[\Ch(\fL(w^{\max} w_0 \cdot 0)]\}$, $\{[\overline{X}_w]\}$ for $w\in \sW_{L_P}\backslash \sW$ are a basis of the finite free $\Z$-module $Z^0(Z_P)$.
	
(2) For $\lambda\in \fz_{L_P}^{\sd}$ and $w\in \sW$, we have \index{$b_{w,w'}$}
	\begin{equation}\label{chaf01}
		[\overline{X}_{w,\lambda}]=\sum_{w'\in \sW_{L_P}\backslash \sW} b_{w, w'} [\Ch(\fL(w'^{\max} w_0 \cdot 0))]
	\end{equation}
	where $b_{w,w'}\in \Z_{\geq 0}$ is the multiplicity of $L(w'^{\max}w_0 \cdot 0)$ in $M_P(w^{\max} w_0 \cdot 0)$ (hence $b_{w,w'}$ only depends on the cosets $\sW_{L_P} w$ and $\sW_{L_P}w'$, $b_{w, w}=1$ and $b_{w,w'}=0$ unless $w'^{\max} \leq w^{\max}$).
	
(3) Let $w,w'\in \sW$, there are integers $a_{w,w'}\in \Z_{\geq 0}$ only depending on the cosets $\sW_{L_P}w$, $\sW_{L_P}w'$ such that \index{$a_{w,w'}$}
	\begin{equation*}
		[\Ch(\fL(w^{\max}w_0 \cdot 0))]=\sum_{w'\in \sW_{L_P}\backslash \sW} a_{w,w'} [Z_{w'}]\in Z^0(Z_P)
	\end{equation*}
	where $a_{w, w}=1$ and $a_{w,w'}=0$ unless $w'^{\max} \leq w^{\max}$. Moreover, if $w'^{\max}<w^{\max}$ and $Bw'^{\max}B/B$ is contained in the smooth locus of the Zariski-closure of $Bw^{\max} B/B$, then $a_{w,w'}=0$.
\end{theorem}
\begin{proof}
	(1) follows from the same argument as in the proof of \cite[Thm.\ 2.4.7]{BHS3} (with \cite[Prop.\ 2.4.6]{BHS3} replaced by Proposition \ref{equcycl} and using Lemma \ref{XwXwl} (2)). (2) is a direct consequence of Proposition \ref{equcycl}. \ By \ Proposition \ \ref{cycbc}, \ (3) \ will \ follow \ from \ a \ parallel \ statement for $[\Ch(i_B^* \fL(w^{\max}w_0 \cdot 0))]$. Let $a_{w,w'}\in \Z_{\geq 0}$ such that (see Remark \ref{stenBP} for the second equality):
	\begin{equation*}
		[\Ch(i_B^* \fL(w^{\max}w_0 \cdot 0))]=\sum_{w'\in \sW_{L_P}\backslash \sW} a_{w,w'} [q_B^{-1}(\fn_P)_{w'}]=\sum_{w'\in \sW_{L_P}\backslash \sW} a_{w,w'}[q_B^{-1}(\fn)_{w'^{\max}}].
	\end{equation*}
	(3) follows then from \cite[Thm.\ 2.4.7 (iii)]{BHS3}.
\end{proof}

\begin{remark}\label{remKLrel}
By \cite[Conj.\ 3.27]{CaCo} and \cite[Thm.\ 3.28]{CaCo} (see also \cite{Deo87}, \cite{Doug}), the coefficients $b_{w,w'}$ in Theorem \ref{thmcycl} may be described using certain relative Kazhdan-Lusztig polynomials. However, we couldn't find a precise statement in the literature and we don't need such description in the paper. We remark that, by \cite[Thm.\ 3.28]{CaCo}, $b_{w,w'}$ is also equal to the multiplicity of $L(w'^{\max}w_0 \cdot \lambda)$ in $M_P(w^{\max}w_0 \cdot \lambda)$ for any integral dominant weight $\lambda$. Finally, using \cite[Thm.\ 9.4 (b)]{Hum08} and \cite[Ex.\ 8.3 (a)]{Hum08}, we easily deduce $b_{w,w'}=1$ when $w'^{\max}< w^{\max}$ and $\lg(w'^{\max})=\lg(w^{\max})-1$.
\end{remark}

Let $\overline{X}_P:=X_P \times_{\kappa_{P}, \fz_{L_P}} \{0\}$. If $M$ is a coherent $\co_{\overline{X}_P}$-module, we define its class $[M]\in Z^0(Z_P)$ as in (\ref{cha00}) with $m(Z_w,Y)$ replaced by the length $m(Z_w,M)$ of the $\co_{\overline{X}_P, \eta_{Z_w}}$-module $M_{\eta_{Z_w}}$. Let $x$ be a closed point in $\overline{X}_P$ (or equivalently in $Z_P$), the complete local rings $\widehat{\co}_{Z_P,x}$, $\widehat{\co}_{Z_w,x}$ are equidimenisonal, and the set of irreducible components of $\Spec \widehat{\co}_{Z_P,x}$ is the (disjoint) union for all $w\in \sW_{L_P}\backslash \sW$ of the sets of irreducible components of $\Spec \widehat{\co}_{Z_w,x}$. Note that in general we don't know whether $\Spec \widehat{\co}_{Z_w,x}$ is irreducible (see the discussion above \cite[Lemma 2.5.5]{BHS3}). However if the $\fn_P$-coordinate of $x$ is zero, then $\Spec \widehat{\co}_{Z_w,x}$ is irreducible by Theorem \ref{unibranch2}. Put $\widehat{M}_x:=M \otimes_{\co_{\overline{X}}} \widehat{\co}_{\overline{X}_P,x}$, and define $[\widehat{M}_x]\in Z^0(\Spec \widehat{\co}_{Z_P,x})$ similarly to $[M]$ above.\index{$[M]$} By the same argument as in the proof of \cite[Lemma 2.5.5]{BHS3}, we have
\begin{equation}\label{chaf02}
	[\widehat{M}_x]=\sum_{w\in \sW_{L_P}\backslash \sW} m(Z_w,M)[\Spec \widehat{\co}_{Z_w,x}] \in Z^0(\Spec \widehat{\co}_{Z_P,x}),
\end{equation}
where $[\Spec \widehat{\co}_{Z_w,x}]\in Z^0(\Spec \widehat{\co}_{Z_P,x})$ is defined similarly as in (\ref{cha00}). For $w\in \sW_{L_P} \backslash \sW$, put (see Theorem \ref{thmcycl} (3) for $a_{w,w'}\in \Z_{\geq 0}$):
\begin{equation*}
	[\widehat{\fL}(w^{\max} w_0 \cdot 0)_x]:=\sum_{w'\in \sW_{L_P}\backslash \sW} a_{w,w'} [\Spec \widehat{\co}_{Z_{w'},x}] \in Z^0(\Spec \widehat{\co}_{Z_P,x}).
\end{equation*}

\begin{lemma}\label{lemcyccomp}
Let $w\in \sW$. 
	
(1) Let $\lambda\in \fz_{L_P}^{\sd}$, we have
	\begin{equation}\label{chaf00}
		[\widehat{\co}_{\overline{X}_{w,\lambda},x}]=\sum_{w'\in \sW_{L_P}\backslash \sW} b_{w,w'}[\widehat{\fL}(w'^{\max}w_0\cdot 0)_x]\in Z^0(\Spec \widehat{\co}_{Z_P,x}).
	\end{equation}
	
	(2) For $w'\in \sW$ there are integers $c_{w,w'}, c_{w,w'}' \in \Z_{\geq 0}$, depending only on the cosets $\sW_{L_P}w$ and $\sW_{L_P}w'$, satisfying the following conditions:
	\begin{itemize}
	\item$c_{w,w}=c_{w,w}'=1$;
	\item $c_{w,w'}'\geq c_{w,w'}$;
	\item $c_{w,w'}>0\Longleftrightarrow c_{w,w'}'>0$;
	\item $c_{w,w'}=c_{w,w'}'=0$ except when $w'^{\max}\leq w^{\max}$;
	\item we have the equalities in $Z^0(\Spec \widehat{\co}_{Z_P,x})$:
	\[ [\widehat{\co}_{\overline{X}_{w,\lambda},x}]=\sum_{w'\in \sW_{L_P}\backslash \sW} c_{w,w'}[\Spec \widehat{\co}_{Z_{w'},x}]\text{\ \ and\ \ } [\widehat{\co}_{\overline{X}_{w},x}]=\sum_{w'\in \sW_{L_P}\backslash \sW} c_{w,w'}'[\Spec \widehat{\co}_{Z_{w'},x}].\] 
	\end{itemize}
	
(3) If $x\in \overline{X}_w$ is a smooth point of $X_w$, then $\widehat{\co}_{\overline{X}_{w},x}\cong \widehat{\co}_{\overline{X}_{w,\lambda},x}$, in particular (\ref{chaf00}) holds with $\widehat{\co}_{\overline{X}_{w,\lambda},x}$ replaced by $\widehat{\co}_{\overline{X}_{w},x}$.
\end{lemma}
\begin{proof}
	(1) follows from (\ref{chaf01}) using (\ref{chaf02}). (2) follows from Lemma \ref{XwXwl} by putting $c_{w,w'}:=m(Z_{w'}, \overline{X}_{w,\lambda})$ and $c_{w,w'}':=m(Z_{w'}, \overline{X}_w)$. For (3), it suffices to show $\co_{\overline{X}_{w,\lambda},x}\cong \co_{\overline{X}_w, x}$ if $X_w$ is smooth at $x$. Using the description of $\overline{X}_{w,\lambda}$ and $\overline{X}_{w}$ in the proof of Lemma \ref{XwXwl}, we only need to show that $X_{w,\lambda}':=X_w \times_{\kappa_{P,w}, \fz_{L_P}} \fz_{\lambda}$ is reduced at the point $x$. It is easy to see that $X_{w,\lambda}$ is cut out by ($\dim \fz_{L_P}-\dim \fz_{\lambda}$)-equations in $X_w$. As $X_w$ is smooth at $x$ and $\dim X_{w,\lambda}=\dim X_w-(\dim \fz_{L_P}-\dim \fz_{\lambda})$, there exists an open neighbourhood $U \subset X_{w,\lambda}'$ of $x$ such that $U$ is a local complete intersection, hence Cohen-Macaulay. It is thus sufficient to show that $U$ is generically reduced (cf.\ \cite[Prop.\ 5.8.5]{EGAiv1}). By Proposition \ref{genesmoo}, $X_P^{\reg-\sss}$ is smooth. It is also Zariski-open and Zariski-dense in $X_P$. We deduce $X_P^{\reg-{\sss}}=\sqcup_{w'\in \sW/\sW_{L_P}} (X_P^{\reg-\sss} \cap X_{w'})$ where $\{X_P^{\reg-\sss} \cap X_{w'}\}$ are the irreducible components of $X_P^{\reg-\sss}$. In particular, we have that $X_P^{\reg-\sss} \cap X_w$ is Zariski-open and Zariski-dense in $X_w$. Moreover $\kappa_P: X_P^{\reg-\sss} \cap X_w \ra \fz_{L_P}$ is smooth by Proposition \ref{genesmoo}. We deduce that $X_{w,\lambda}^{\reg-\sss}:=(X_P^{\reg-\sss} \cap X_w) \times_{\kappa_{P}, \fz_{L_P}} \fz_{\lambda}$ is smooth and is Zariski-open in $X_{w,\lambda}'$. One easily checks $X_P^{\reg-\sss} \cap V_{w,\lambda} \neq \emptyset$ (e.g.\ using (\ref{fiber0})), so $X_P^{\reg-\sss}\cap V_{w,\lambda}$ is non-empty Zariski-open, hence Zariski-dense, in $V_{w,\lambda}$, which implies that $X_{w,\lambda}^{\reg-\sss}$ is Zariski-dense in $X_{w,\lambda}$. As $X_{w,\lambda}^{\reg-\sss}$ is smooth, we see that $U$ is generically smooth hence generically reduced. The lemma follows.
\end{proof}

\section{Local models for the Bernstein paraboline varieties}\label{secMod}

Generalizing \cite[\S~3]{BHS3}, we show that the geometry of the Bernstein paraboline varieties of \S~\ref{secDefva} is closely related to the schemes considered in \S~\ref{secGS}. We use the notation of \S~\ref{secDefva} and \S~\ref{secGS} applied to $G=\GL_n$. When applied to $G=(\Res^L_{\Q_p} \GL_n)\times_{\Spec \Q_p} \Spec E\simeq \prod_{\Sigma_L}\GL_n$, we modify the notation of \S~\ref{secGS} by adding a subscript ``$L$" to each scheme considered in \S~\ref{secGS} (to stress the field ``$L$" and to distinguish from the case $G=\GL_n$), for instance $\ug_L$ is $\ug \otimes_{\Q_p} L$, $X_{P,L}\cong \prod_{\tau\in \Sigma_L} X_P$ and we have $\tilde{\ug}_{P,L}$, etc.

\subsection{Almost de Rham $B_{\dR}$-representations}\label{sec61}

We define and study certain groupoids of deformations of an almost de Rham $B_{\dR}$-represen\-tation of $\Gal_L$. 

Let $B_{\pdR}^+$ be the algebra $B_{\dR}^+[\log(t)]$ defined in \cite[\S~4.3]{Fo04} and $B_{\pdR}=B_{\pdR}^+ \otimes_{B_{\dR}^+} B_{\dR}$. Recall $B_{\pdR}^+$ (resp.\ $B_{\pdR}$) is equipped with a natural action of $\Gal_L$ extending the usual $\Gal_L$-action on $B_{\dR}^+$ (resp.\ $B_{\dR}$) such that $g(\log(t))=\log(t)+\log(\chi_{\cyc}(g))$. Moreover there is a unique $B_{\dR}$-derivation $\nu_{B_{\pdR}}$ of $B_{\pdR}$ such that $\nu_{\pdR}(\log(t))=-1$. It is clear that $\nu_{B_{\pdR}}$ preserves $B_{\pdR}^+$ and commutes with $\Gal_L$. 

We fix an almost de Rham representation $W$ of $\Gal_L$ over $B_{\dR} \otimes_{\Q_p} E$, i.e.\ $W$ is a free $B_{\dR} \otimes_{\Q_p} E$-module equipped with a semi-linear $\Gal_L$-action such that $\dim_L (B_{\pdR} \otimes_{B_{\dR}} W)^{\Gal_L}=\dim_{B_{\dR}} W$ (cf.\ \cite[Thm.~4.1 (2)]{Fo04}). Let $P\subseteq \GL_n$ be a parabolic subgroup as in \S~\ref{Nota2.1}. Let $\cF_{\bullet}=(\cF_i)_{0\leq i \leq r}$ be a $P$-filtration on $W$, i.e.\ $0=\cF_0\subseteq \cF_1\subseteq \cdots \subseteq \cF_r$ are $B_{\dR} \otimes_{\Q_p} E$-subrepresentations of $W$ such that $\cF_i/\cF_{i-1}$ is free of rank $n_i$ over $B_{\dR} \otimes_{\Q_p} E$ for $1\leq i \leq r$. We assume that $\cF_i/\cF_{i-1}$ is de Rham for $1 \leq i \leq r$, so $\cF_i/\cF_{i-1}\cong (B_{\dR} \otimes_{\Q_p} E)^{\oplus n_i}$ as $\Gal_L$-representation.

For $A$ in $\Art(E)$, we call $B_{\dR} \otimes_{\Q_p} A$-representation of $\Gal_L$ a free $B_{\dR} \otimes_{\Q_p} A$-module of finite rank endowed with a semi-linear action of $\Gal_L$ (so $\Gal_L$ acts trivially on $A$). We define $X_{W, \cF_{\bullet}}$ as the following groupoid over $\Art(E)$:\index{$X_{W,\cF_{\bullet}}$}
\begin{enumerate}
	\item[(1)] The objects of $X_{W, \cF_{\bullet}}$ are the quadruples $(A, W_A,\cF_{A, \bullet}, \iota_A)$ where 
	\begin{itemize}
		\item $A\in \Art(E)$ and $W_A$ is a $B_{\dR} \otimes_{\Q_p} A$-representation of $\Gal_L$;
		\item $\cF_{A, \bullet}=(\cF_{A,i})_{0\leq i \leq r}$ is a $P$-filtration on $W_A$ by $B_{\dR} \otimes_{\Q_p} A$-subrepresentations of $\Gal_L$ such that $\cF_{A,0}=0$ and $\cF_{A,i}/\cF_{A,i-1}$, $1\leq i \leq r$, is free of rank $n_i$ over $B_{\dR} \otimes_{\Q_p} A$ and isomorphic to $(\cF_i/\cF_{i-1}) \otimes_{B_{\dR} \otimes_{\Q_p} E} \varepsilon_{A,i}$ for some rank one $B_{\dR} \otimes_{\Q_p} A$-representation $\varepsilon_{A,i}$;
		\item $\iota_A: W_A \otimes_A E \xrightarrow{\sim} W$ is an isomorphism of $B_{\dR} \otimes_{\Q_p} E$-representations which induces isomorphisms $\iota_A: \cF_{A,i} \otimes_{A} E \xrightarrow{\sim} \cF_i$ for $i\in \{0,\dots,r\}$.
	\end{itemize}
	\item[(2)] A morphism $(A, W_A, \cF_{A, \bullet}, \iota_A) \ra (A', W_{A'}, \cF_{A',\bullet}, \iota_{A'})$ is a morphism $A \ra A'$ in $\Art(E)$ and an isomorphism $W_{A} \otimes_A A' \xrightarrow{\sim} W_{A'}$ of $B_{\dR} \otimes_{\Q_p} A'$-representations which is compatible with $\iota_A$ and $\iota_{A'}$ and induces isomorphisms $\cF_{A, i} \otimes_A A' \xrightarrow{\sim} \cF_{A',i}$ for $i\in \{0,\dots,r\}$.
\end{enumerate}
Note that the rank one $\varepsilon_{A,i}$ in (1) and the $\cF_{A,i}$ are automatically almost de Rham since extensions of almost de Rham representations are always almost de Rham by \cite[\S~3.7]{Fo04}.

We fix an isomorphism of $L\otimes_{\Q_p} E$-modules:
\[\alpha: (L \otimes_{\Q_p} E)^n \xlongrightarrow{\sim} D_{\pdR}(W):=(B_{\pdR} \otimes_{B_{\dR}} W)^{\Gal_L},\]
and we let $X_W$, $X_{W}^{\square}$ be the groupoids over $\Art(E)$ defined as in \cite[\S~3.1]{BHS3} where ``$(-)^{\square}$" is with respect to $\alpha$. We have a natural morphism of groupoids $X_{W, \cF_{\bullet}} \ra X_W$ sending $(A,W_A, \cF_{A, \bullet}, \iota_A)$ to $(A,W_A, \iota_A)$. We put $X_{W, \cF_{\bullet}}^{\square}:=X_{W, \cF_{\bullet}}\times_{X_W} X_W^{\square}$.\index{$X_{W,\cF_{\bullet}}^{\square}$} The objects of $X_{W,\cF_{\bullet}}^{\square}$ are the $5$-tuples $(A,W_A, \cF_{A, \bullet}, \iota_A, \alpha_A)$ where $(A,W_A, \cF_{A, \bullet}, \iota_A)$ is an object in $X_{W, \cF_{\bullet}}$ and $\alpha_A$ is an isomorphism $\alpha_A: (L \otimes_{\Q_p} A)^n \xrightarrow{\sim} D_{\pdR}(W_A)$ such that $\alpha_A \equiv \alpha \pmod{\fm_A}$. A morphism 
$(A, W_A, \cF_{A, \bullet}, \iota_A, \alpha_A) \ra (A', W_{A'}, \cF_{A',\bullet}, \iota_{A'}, \alpha_{A'})$ is a morphism $(A, W_A, \cF_{A, \bullet}, \iota_A) \ra (A', W_{A'}, \cF_{A',\bullet}, \iota_{A'})$ in $X_{W, \cF_{\bullet}}$ such that the following diagram commutes
\begin{equation*}
	\begin{CD}
		(L \otimes_{\Q_p} A)^n \otimes_A A' @> \alpha_A \otimes 1 >> D_{\pdR}(W_A) \otimes_{A} A' \\
		@| @V \wr VV \\
		(L \otimes_{\Q_p} A')^n @> \alpha_{A'} >> D_{\pdR}(W_{A'}).
	\end{CD}
\end{equation*}
Let $(A, W_A, \cF_{A, \bullet}, \iota_A, \alpha_A)$ be an object in $X_{W, \cF_{\bullet}}^{\square}$. Recall that the $B_{\dR}$-derivation $\nu_{B_{\pdR}}$ on $B_{\pdR}$ induces an $L \otimes_{\Q_p} A$-linear nilpotent operator $\nu_{W_A}$ on $D_{\pdR}(W_A)$. We denote the matrix $\alpha_A^{-1} \circ \nu_{W_A} \circ \alpha_A$ by $N_{W_A}\in M_n(L\otimes_{\Q_p} A)=\ug_L(A)$. We let $\cD_{A, \bullet}:=(\cD_{A, i})_i$ with $\cD_{A,i}:=D_{\pdR}(\cF_{A,i})$. 

\begin{lemma}\label{lemrp}
	With the above notation $(\alpha_{A}^{-1}(\cD_{A, \bullet}), N_{W_A})\in \widetilde{\ug}_{P,L}(A)$.
\end{lemma} 
\begin{proof}
	The $P$-filtration $\cD_{A, \bullet}$ is stable by $\nu_{W_A}$. It is then sufficient to show that the induced action of $\nu_{W_A}$ on $\cD_{A,i}/\cD_{A,i-1}$ is a scalar (in $L \otimes_{\Q_p} A$). Since the $\cF_{A,i}$ are almost de Rham and the $\cF_i/\cF_{i-1}$ are de Rham, we have an isomorphism of $E$-vector spaces $\cD_{A,i}/\cD_{A,i-1} \cong D_{\dR}(\cF_i/\cF_{i-1}) \otimes_{L\otimes_{\Q_p} E} D_{\pdR}(\varepsilon_{A,i})$ which is compatible with $\nu_{W_A}$. Hence $\nu_{W_A}$ is given on $\cD_{A,i}/\cD_{A,i-1}$ by $\nu_{\cF_i/\cF_{i-1}}\otimes \id + \id \otimes \nu_{\varepsilon_{A,i}}=\id \otimes \nu_{\varepsilon_{A,i}}$, which is the multiplication by the scalar $\nu_{\varepsilon_{A,i}}\in L \otimes_{\Q_p} A$.
\end{proof}

We denote by $\widehat{\tilde{\ug}}_{P,L}$ the completion of $\tilde{\ug}_{P,L}$ at the point $(\alpha^{-1} (\cD_{\bullet}), N_W)\in \tilde{\ug}_{P,L}(E)$, that we also see in the obvious way as a functor from $\Art(E)$ to sets. If $X$ is a groupoid over $\Art(E)$, we denote by $|X|$ the functor on $\Art(E)$ such that $|X|(A)$ is the set of isomorphism classes of the category $X(A)$ (see \cite[Appendix]{Kis09} and \cite[\S~3.1]{BHS3} for more details). The following proposition easily follows from \cite[Lemma 3.1.4]{BHS3}.

\begin{proposition}\label{XWF}
	The groupoid $X_{W, \cF_{\bullet}}^{\square}$ over $\Art(E)$ is pro-representable. The functor 
	\begin{equation*}
		(A, W_A, \cF_{A, \bullet}, \iota_A, \alpha_A) \longmapsto (\alpha_{A}^{-1}(\cD_{A, \bullet}), N_{W_A})
	\end{equation*}
	induces an isomorphism of functors between $|X_{W, \cF_{\bullet}}^{\square}|$ and $\widehat{\tilde{\ug}}_{P,L}$. In particular $|X_{W, \cF_{\bullet}}^{\square}|$ is pro-represented by a formally smooth noetherian complete local ring of residue field $E$ and dimension $[L:\Q_p](\dim \fn_P+\dim \ur_P)$.
\end{proposition}

\begin{remark}\label{equivgrfu}
	As discussed in \cite[\S~3.1]{BHS3}, the morphism $X_W^{\square} \ra |X_W^{\square}|$ is actually an equivalence. We then easily deduce $X_{W,\cF_{\bullet}}^{\square} \xrightarrow{\sim} |X_{W,\cF_{\bullet}}^{\square}|$.
\end{remark}

Since $\nu_{\cF_i/\cF_{i-1}}=0$, the morphism $\widetilde{\ug}_{P,L} \xrightarrow{\kappa_P} \fz_{L_P,L}$ (cf.\ \S~\ref{prel}) induces a morphism $\widehat{\tilde{\ug}}_{P,L} \xrightarrow{\kappa_P} \widehat{\fz}_{L_P,L}$ where $\widehat{\fz}_{L_P}$ denotes the completion of $\fz_{L_P}$ at $0$. Consider the following composition of groupoids over $\Art(E)$
\begin{equation}\label{kawF}
	\kappa_{W, \cF_\bullet}: X_{W, \cF_{\bullet}}^{\square} \lra | X_{W, \cF_{\bullet}}^{\square}| \xlongrightarrow{\sim} \widehat{\tilde{\ug}}_{P,L} \xlongrightarrow{\kappa_P} \widehat{\fz}_{L_P,L}.
\end{equation}
One checks that (\ref{kawF}) actually factors through
\[\kappa_{W,\cF_{\bullet}}: X_{W,\cF_{\bullet}} \lra \widehat{\fz}_{L_P,L}.\]
The morphism (\ref{kawF}) has the following functorial interpretation. Let $x_A:=(W_A, \cF_{A, \bullet}, \iota_A)\in X_{W, \cF_{\bullet}}$, then the endomorphism $\nu_{W_A}$ on $D_{\pdR}(W_A)$ induces an endomorphism $\nu_{A,i}$ on every $D_{\pdR}(\cF_{A,i})/D_{\pdR}(\cF_{A,i-1})\cong D_{\pdR}(\cF_{A,i}/\cF_{A,i-1})$. As in the proof of Lemma \ref{lemrp}, $\nu_{A,i}$ is a scalar in $L \otimes_{\Q_p} A$ which is $0$ modulo $\fm_A$. It follows that
\begin{equation}\label{pdRwt}
	\kappa_{W, \cF_\bullet}(x_A)=(\nu_{A,1}, \dots, \nu_{A,r})\in \widehat{\fz}_{L_P,L}(A).
\end{equation}

\subsection{$(\varphi, \Gamma)$-modules of type $\Omega$ over $\cR_{E,L}[\frac{1}{t}]$}\label{sec62}

We study certain groupoids of deformations of a $(\varphi, \Gamma)$-module over $\cR_{E,L}[\frac{1}{t}]$ equipped with an $\Omega$-filtration. 

Let $\Omega=\prod_{i=1}^r \Omega_i$ be a cuspidal component of $L_P(L)$. Let $A\in \Art(E)$ and $\cM$ be a $(\varphi, \Gamma)$-module of rank $n$ over $\cR_{A,L}[1/t]$. For $i=1,\dots, r$, let $x_i$ be a closed point of $\Spec \cZ_{\Omega_i}$ and let $\Delta_{x_i}$ be the associated $p$-adic differential equation. We call $\cM$ \textit{of type $\Omega$} if there exists a filtration $\cM_{\bullet}=(\cM_i)_{0\leq i \leq r}$ by $(\varphi, \Gamma)$-submodules of $\cM$ over $\cR_{A,L}[1/t]$ such that $\cM_0=0$ and $\cM_i/\cM_{i-1} \cong \Delta_{x_i} \otimes_{\cR_{E,L}} \cR_{A,L}(\delta_i)[\frac{1}{t}]$ for some continuous character $\delta_i: L^{\times} \ra A^{\times}$. Such a filtration $\cM_{\bullet}$ is called an \textit{$\Omega$-filtration} of $\cM$, and $(\ul{x}, \delta)=((x_i), \boxtimes_{i=1}^r\delta_i)\in (\Spec \cZ_{\Omega})^{\rig}\times \widehat{Z_{L_P}(L)}$ is called a parameter of $\cM_{\bullet}$ (compare with Definition \ref{defOF}). 

\begin{lemma}\label{paradif}
Let $\cM$ be a $(\varphi, \Gamma)$-module of type $\Omega$ over $\cR_{E,L}[1/t]$ and $\cM_{\bullet}$ an $\Omega$-filtration of $\cM$ of parameter $(\ul{x}, \delta)$. Then all parameters of $\cM_{\bullet}$ are of the form $(\ul{x}', \delta')$ such that, for $i=1, \dots, r$, $\Delta_{x_i'}\cong \Delta_{x_i} \otimes_{\cR_{E,L}} \cR_{E,L}(\psi_i)$ and $\delta_i'=\delta_i\psi_i^{-1} \eta_iz^{\textbf{k}}$ for some unramified character $\psi_i$ of $L^{\times}$, $\eta_i\in \mu_{\Omega_i}$ and $\textbf{k}\in \Z^{|\Sigma_L|}$.
\end{lemma}
\begin{proof}
	We have $\Delta_{x_i} \otimes_{\cR_{E,L}} \cR_{E,L}(\delta_i)[1/t]\cong \Delta_{x_i'}\otimes_{\cR_{E,L}} \cR_{E,L}(\delta_i')[1/t]$ if and only if for sufficiently large $N$:
	\begin{equation*}
		\Hom_{(\varphi, \Gamma)}\big(\Delta_{x_i}, t^{-N} \Delta_{x_i'} \otimes_{\cR_{E,L}} \cR_{E,L}(\delta_i'\delta_i^{-1})\big)\neq 0
	\end{equation*}
	The lemma then follows by the same argument as in the proof of Lemma \ref{RmtOP}. 
\end{proof}

We now fix a $(\varphi, \Gamma)$-module $\cM$ over $\cR_{E,L}[\frac{1}{t}]$ of type $\Omega$ and an $\Omega$-filtration $\cM_{\bullet}$. We define the groupoid $X_{\cM, \cM_{\bullet}}$ over $\Art(E)$ as follows:
\begin{itemize}
	\item[(1)] The objects of $X_{\cM, \cM_{\bullet}}$ are the quadruples $(A, \cM_A, \cM_{A,\bullet}, j_A)$ where $A\in \Art(E)$, $\cM_A$ is a $(\varphi, \Gamma)$-module over $\cR_{A,L}[\frac{1}{t}]$ of type $\Omega$, $\cM_{A, \bullet}$ is an $\Omega$-filtration of $\cM_A$, and $j_A$ is an isomorphism $\cM_A \otimes_A E \xrightarrow{\sim} \cM$ which induces isomorphisms $\cM_{A,i}\otimes_A E \xrightarrow{\sim} \cM_i$.
	\item[(2)] A morphism $(A, \cM_A, \cM_{A,\bullet}, j_A) \ra (A', \cM_{A'}, \cM_{A', \bullet}, j_{A'})$ is a morphism $A \ra A'$ in $\Art(E)$ and an isomorphism $\cM_A \otimes_A A' \xrightarrow{\sim} \cM_{A'}$ which is compatible with the morphisms $j_A$, $j_{A'}$ and induces isomorphisms $\cM_{A,i} \otimes_A A' \xrightarrow{\sim} \cM_{A',i}$ for $i=1, \dots, r$. 
\end{itemize}

\begin{lemma}\label{parauni}
	Let $(\ul{x}, \delta)$ be a parameter of $\cM_{\bullet}$ and $(A,\cM_A, \cM_{A,\bullet}, j_A)\in X_{\cM, \cM_{\bullet}}$. There exists a unique character $\delta_A=\boxtimes_{i=1}^r \delta_{A,i}: Z_{L_P}(L)\ra A^{\times}$ such that $\delta_{A,i}\equiv \delta_i \pmod{\fm_A}$ and $(\ul{x}, \delta_A)$ is a parameter of $\cM_{A, \bullet}$.
\end{lemma}
\begin{proof}
	Let $\delta_A=\boxtimes_{i=1}^r \delta_{A,i}: Z_{L_P}(L) \ra A^{\times}$ be a continuous character such that $\cM_{A,i}/\cM_{A,i-1}\cong \Delta_{x_i}\otimes_{\cR_{E,L}} \cR_{A,L}(\delta_{A,i})[\frac{1}{t}]$. Denote by $\overline{\delta}_{A,i}: \delta_{A,i} \ra A^{\times} \ra E^{\times}$ the reduction of $\delta_{A,i}$ modulo $\fm_A$. We have $\Delta_{x_i}\otimes_{\cR_{E,L}} \cR_{E,L}(\overline{\delta}_{A,i})[\frac{1}{t}] \cong \Delta_{x_i}\otimes_{\cR_{E,L}} \cR_{E,L}(\delta_{i})[\frac{1}{t}]$. It follows that 
	\begin{equation*}
		\Hom_{(\varphi, \Gamma)}\big(\Delta_{x_i}, t^{-m}\Delta_{x_i}	\otimes_{\cR_{E,L}} \cR_{E,L}(\delta_i^{-1} \overline{\delta}_{A,i}) \big)\!\cong \!H^0_{(\varphi, \Gamma)}\big(t^{-m}\Delta_{x_i} \otimes_{\cR_{E,L}} \Delta_{x_i}^{\vee} \otimes_{\cR_{E,L}} \cR_{E,L}(\delta_i^{-1} \overline{\delta}_{A,i})\big)\!\cong \!E
	\end{equation*}
	for $m\gg 0$. By the same argument as in the proof of Lemma \ref{RmtOP}, we see that there exists an algebraic character $\chi_i$ of $L^{\times}$ and $\psi_i \in \eta_{\Omega_i}$ such that $\overline{\delta}_{A,i}=\delta_i \psi_i \chi_i$. Since $\Delta_{x_i}[\frac{1}{t}] \cong \Delta_{x_i}\otimes_{\cR_{E,L}} \cR_{E,L}(\psi_i^{-1} \chi_i^{-1}) [\frac{1}{t}]$, replacing $\delta_{A,i}$ by $\delta_{A,i} \psi_i^{-1} \chi_i^{-1}$, the existence in the lemma follows.
	
	Assume $\delta_A$, $\delta_A'$ are two characters satisfying the properties in the lemma. We have an injection
	\begin{equation*}
		A \hookrightarrow H^0_{(\varphi, \Gamma)}\Big(\Delta_{x_i} \otimes_{\cR_{E,L}} \Delta_{x_i}^{\vee} \otimes_{\cR_{E,L}} \cR_{A,L}(\delta_{A,i}^{-1} \delta_{A,i}')\Big[\frac{1}{t}\Big]\Big),
	\end{equation*}
	where by definition $H^0_{(\varphi, \Gamma)}(\cN):=\Hom_{(\varphi, \Gamma)}(\cR_A[1/t], \cN)$ for a $(\varphi, \Gamma)$-module $\cN$ over $\cR_{A,L}[1/t]$ (see \cite[\S~3.3]{BHS3} for the cohomology of $(\varphi, \Gamma)$-modules over $\cR_{E,L}[1/t]$). We write $\Delta_{x_i} \otimes_{\cR_{E,L}} \Delta_{x_i}^{\vee} \cong \cR_{E,L} \oplus (\Delta_{x_i} \otimes_{\cR_{E,L}} \Delta_{x_i}^{\vee})^0$, and we have $H^0_{(\varphi, \Gamma)}((\Delta_{x_i} \otimes_{\cR_{E,L}} \Delta_{x_i}^{\vee})^0[\frac{1}{t}])=0$ (using \cite[(3.11)]{BHS3}). By an easy d\'evissage on $A$, we deduce that 
	\begin{equation*}
		H^0_{(\varphi, \Gamma)}\Big(\Delta_{x_i} \otimes_{\cR_{E,L}} \Delta_{x_i}^{\vee} \otimes_{\cR_{E,L}} \cR_{A,L}(\delta_{A,i}^{-1} \delta_{A,i}')\Big[\frac{1}{t}\Big]\Big)\cong H^0_{(\varphi, \Gamma)}\Big(\cR_{A,L}(\delta_{A,i}^{-1} \delta_{A,i}')\Big[\frac{1}{t}\Big]\Big).
	\end{equation*}
	By \cite[Lemma 3.3.4]{BHS3} and $\delta_{A,i} \equiv \delta'_{A,i}\pmod{\fm_A}$, we must have $\delta_{A,i}=\delta_{A,i}'$.
\end{proof}

Let $\delta$ be a continuous character of $Z_{L_P}(L)$, that we also view as a point of $\widehat{Z_{L_P}(L)}$. We denote by $\widehat{Z_{L_P}(L)}_{\delta}$ the completion of $\widehat{Z_{L_P}(L)}$ at $\delta$.\index{$\widehat{Z_{L_P}(L)}_{\delta}$} It is easy to see that the functor
\[A\in \Art(E) \longmapsto \{\delta_A= \boxtimes_{i=1}^r\delta_{A,i}: Z_{L_P}(L) \ra A^{\times}, \ \delta_{A,i} \equiv \delta_i \pmod{\fm_A}\}\]
is pro-represented by $\widehat{Z_{L_P}(L)}_{\delta}$. By Lemma \ref{parauni}, we have a morphism of groupoids over $\Art(E)$:
\[\omega_{\delta}: X_{\cM, \cM_{\bullet}} \lra \widehat{Z_{L_P}(L)}_{\delta},\ \ (A, \cM_A, \cM_{A,\bullet}, j_A)\longmapsto \delta_A\index{$X_{\cM,\cM_{\bullet}}$}.\]

Recall we have a functor $W_{\dR}$ from the category of $(\varphi, \Gamma)$-module over $\cR_{E,L}[\frac{1}{t}]$ to the category of $B_{\dR} \otimes_{\Q_p} E$-representations of $\Gal_L$ (cf.\ \cite[Lemma 3.3.5 (ii)]{BHS3}). Moreover, by \textit{loc.\ cit.}, for $A\in \Art(E)$, $W_{\dR}$ sends a $(\varphi, \Gamma)$-module of rank $n$ over $\cR_{A,L}[\frac{1}{t}]$ to a $B_{\dR} \otimes_{\Q_p} A$-representation of $\Gal_L$ of rank $n$. Let $W:=W_{\dR}(\cM)$ and $\cF_i:=W_{\dR}(\cM_i)$. Assume that for one parameter (equivalently all parameters) $(\ul{x}, \delta)$ of $\cM_{\bullet}$, we have that $\delta_i$ is locally algebraic for all $i$. Then it is easy to see that $W_{\dR}(\cM_i/\cM_{i-1})\cong W_{\dR}(\cM_i)/W_{\dR}(\cM_{i-1})$ is de Rham \big(hence $\cong (B_{\dR} \otimes_{\Q_p} E)^{\oplus n_i}$\big). For $(A, \cM_A, \cM_{A,\bullet}, j_A)\in X_{\cM, \cM_{\bullet}}$, let $W_A:=W_{\dR}(\cM_A)$ and $\cF_{\bullet,i}:=W_{\dR}(\cM_{A,i})$. If $(\ul{x}, \delta_A)$ is a parameter of $\cM_{A, \bullet}$, then we have 
\begin{equation}\label{wdri}
	\cF_{A,i}/\cF_{A,i-1}\cong W_{\dR}(\cM_{A,i}/\cM_{A,i-1}) \cong W_{\dR}\Big(\Delta_{x_i}\Big[\frac{1}{t}\Big]\Big) \otimes_{B_{\dR} \otimes_{\Q_p} E} W_{\dR}\Big(\cR_{A,L}(\delta_{A,i})\Big[\frac{1}{t}\Big]\Big),
\end{equation}
where, for the last isomorphism, we use \cite[Prop.\ 2.2.6 (2)]{Ber08II} and the fact that $W_{\dR}(D[\frac{1}{t}])=W_{\dR}^+(D)[\frac{1}{t}]$ for a $(\varphi, \Gamma)$-module $D$ over $\cR_{E,L}$.
Let $\iota_A$ be the composition 
\begin{equation*}
	\iota_A: W_{\dR}(\cM_A) \otimes_A E \xlongrightarrow{\sim} W_{\dR}(\cM_A \otimes_A E) \xlongrightarrow{\sim} W_{\dR}(\cM)
\end{equation*}
where the last isomorphism is induced by $j_A$. By (\ref{wdri}), we see that $(W_A, \cF_{A, \bullet}, \iota_A)\in X_{W, \cF_{\bullet}}$, so $W_{\dR}$ defines a morphism $X_{\cM, \cM^{\bullet}} \ra X_{W, \cF_{\bullet}}$. Let $X_{\cM}$ be the groupoid over $\Art(E)$ defined as $X_{\cM, \cM_{\bullet}}$ but forgetting everywhere the $\Omega$-filtrations. It is easy to see that $W_{\dR}$ defines a morphism of groupoids $X_{\cM} \ra X_W$, and that the following diagram of groupoids commutes
\begin{equation*}
	\begin{CD}
		X_{\cM, \cM_{\bullet}} @> W_{\dR} >> X_{W, \cF_{\bullet}} \\
		@VVV @VVV \\
		X_{\cM} @> W_{\dR} >> X_W.
	\end{CD}
\end{equation*}
We fix an isomorphism $\alpha: (L\otimes_{\Q_p} E)^n \xrightarrow{\sim} D_{\pdR}(W)$, so we have the groupoids $X_W^{\square}$, $X_{W, \cF_\bullet}^{\square}$ over $\Art(E)$ (cf.\ \S~\ref{sec61}). We put\index{$X_{\cM,\cM_{\bullet}}$}
\begin{equation}\label{moreX}
\begin{gathered}
\begin{array}{ccl}
	&& X_{\cM}^{\square}:=X_{\cM} \times_{X_W} X_W^{\square}, \\
	&& X_{\cM, \cM_{\bullet}}^{\square}:=X_{\cM, \cM_{\bullet}} \times_{X_{W, \cF_\bullet}} X_{W, \cF_\bullet}^{\square} \cong X_{\cM, \cM_{\bullet}} \times_{X_W} X_{W}^{\square}
\end{array}
\end{gathered}
\end{equation}
and note that $X_{\cM}^{\square}\ra X_{\cM}$, $X_{\cM, \cM_{\bullet}}^{\square}\ra X_{\cM, \cM_{\bullet}}$ are formally smooth of relative dimension $[L:\Q_p]n^2$ by base change. We fix a parameter $(\ul{x}, \delta)$ of $\cM_{\bullet}$. For $A\in \Art(E)$, the natural map 
\begin{equation*}
	\delta_A=\boxtimes_{i=1}^r\delta_{A,i} \longmapsto (\wt(\delta_{A,i})-\wt(\delta_i)) \in (L \otimes_{\Q_p} A)^r \cong \widehat{\fz}_{L_P,L}(A)
\end{equation*}
induces a morphism of formal schemes $\wt-\wt(\delta): \widehat{Z_{L_P}(L)}_{\delta} \ra \widehat{\fz}_{L_P,L}$.

\begin{proposition}\label{wtmap1}
	The following diagram of groupoids over $\Art(E)$ is commutative:
	\begin{equation*}
		\begin{CD}
			X_{\cM, \cM_{\bullet}} @> W_{\dR} >> X_{W, \cF_{\bullet}} \\
			@V \omega_{\delta} VV @V \kappa_{W,\cF_{\bullet}} VV \\
			\widehat{Z_{L_P}(L)}_{\delta} @> \wt-\wt(\delta) >> \widehat{\fz}_{L_P,L}.
		\end{CD}
	\end{equation*}
\end{proposition}
\begin{proof}
	The proposition follows from (\ref{pdRwt}) and \cite[Lemma 3.3.6 (ii)]{BHS3} (which trivially generalizes to the case $\cM=\cR_{A,L}(\delta)[\frac{1}{t}]^{\oplus m}$).
\end{proof}

We call a parameter $(\ul{x}, \delta)$ of $\cM_{\bullet}$ \textit{generic} if the following condition is satisfied:
\begin{equation}\label{geneMf}
\begin{gathered}
\begin{array}{lll}
	&&\text{for $i \neq j$, if $\Delta_{x_i} \cong \Delta_{x_j} \otimes_{\cR_{E,L}} \cR_{E,L}(\psi)$ for some smooth character $\psi$ of $L^{\times}$, then}\\
	&&\text{$\delta_i \delta_{j}^{-1} \psi \neq z^{\textbf{k}}$ and $\delta_i \delta_{j}^{-1} \psi \neq\unr(q_L^{-1}) z^{\textbf{k}}$ for any $\textbf{k} \in \Z^{|\Sigma_L|}$.}
\end{array}
\end{gathered}
\end{equation}

By Lemma \ref{paradif}, if $\cM_{\bullet}$ admits a generic parameter, then any parameter of $\cM_{\bullet}$ is generic, and in this case we call $\cM_{\bullet}$ (or even $\cM$ if $\cM_{\bullet}$ is understated) \textit{generic}.

\begin{remark}\label{geneM1t}
(1) Let $D$ be a $(\varphi, \Gamma)$-module of rank $n$ over $\cR_{E,L}$ such that $D[\frac{1}{t}]\cong \cM$. The $\Omega$-filtration $\cM_{\bullet}$ on $\cM$ induces then an $\Omega$-filtration $\sF$ on $D$. It is straightforward to check that if $\cM_{\bullet}$ is generic then $\sF$ is generic in the sense of (\ref{conGene}). 

(2) Let $\rho$ as in \S~\ref{introPcr} and use the notation of \textit{loc.\ cit.} The $\Omega$-filtration $\sF$ on $\cM(\rho)\cong D_{\rig}(\rho)[\frac{1}{t}]$ is generic if and only if $\rho$ is generic. 
\end{remark}

For $1\leq i, j \leq r$, $i\neq j$, denote by $\cN_{i,j}^0:=\Delta_{x_i} \otimes_{\cR_{E,L}} \Delta_{x_j}^{\vee} \otimes_{\cR_{E,L}} \cR_{E,L}(\delta_{i} \delta_{j}^{-1})$ and $\cN_{i,j}:=\cN_{i,j}^0[\frac{1}{t}]$.

\begin{lemma}\label{phiGamCohotin}
	Assume $(\ul{x}, \delta)$ is a generic parameter of $\cM_{\bullet}$ and let $1\leq i, j \leq r$, $i\neq j$.
	
	\noindent (1) We have $H^0_{(\varphi, \Gamma)}(\cN_{i,j})=H^2_{(\varphi, \Gamma)}(\cN_{i,j})=0$ and 
	\begin{equation*}
		\dim_E H^1_{(\varphi, \Gamma)}(\cN_{i,j})=[L:\Q_p] n_i n_j.
	\end{equation*}
	
	\noindent (2) Suppose that $\delta$ is locally algebraic, then the natural morphism 
	\begin{equation}\label{cohog}
		H^1_{(\varphi, \Gamma)}(\cN_{i,j}) \lra H^1\big(\Gal_L, W_{\dR}(\cN_{i,j})\big)
	\end{equation}
	is an isomorphism.
	
	\noindent (3) Suppose that $\delta$ is locally algebraic and let $A \in \Art(E)$, $\delta_{A,i}, \delta_{A,j}: L^{\times} \ra A^{\times}$ be continuous characters such that $\delta_{A,i} \equiv \delta_i$, $\delta_{A,j} \equiv \delta_j \pmod{\fm_A}$, and $\cN_{i,j,A}:=\Delta_{x_i} \otimes_{\cR_{E,L}} \Delta_{x_j}^{\vee} \otimes_{\cR_{E,L}} \cR_{E,L}(\delta_{A,i}\delta_{A,j}^{-1})[\frac{1}{t}]$. Then the natural morphism
	\begin{equation*}
		H^1_{(\varphi, \Gamma)}(\cN_{i,j,A}) \lra H^1\big(\Gal_L, W_{\dR}(\cN_{i,j,A})\big)
	\end{equation*}
	is surjective.
\end{lemma}
\begin{proof}
	(1) We claim that for $s=\{0,1,2\}$, 
	\begin{equation}\label{cohotinv}
	H^s_{(\varphi, \Gamma)}(\cN_{i,j})\cong H^s_{(\varphi, \Gamma)}(t^{-k} \cN_{i,j}^0) \text{ for $k$ sufficiently large.}
	\end{equation}
		Indeed, identifying the cohomology of $(\varphi, \Gamma)$-modules and the Galois cohomology of $B$-pairs (see for example \cite[\S~3]{Na}), we deduce from the morphism $t^{-k} \cN_{i,j}^0 \hookrightarrow t^{-k-1} \cN_{i,j}^0$ a long exact sequence 
	\begin{multline}\label{longNij}
	0 \ra H^0_{(\varphi, \Gamma)}(t^{-k} \cN_{i,j}^0) \ra H^0_{(\varphi, \Gamma)}(t^{-k-1} \cN_{i,j}^0) \ra H^0_{(\varphi, \Gamma)}(t^{-k-1} \cN_{i,j}^0/t^{-k} \cN_{i,j}^0)\\
	\ra H^1_{(\varphi, \Gamma)}(t^{-k} \cN_{i,j}^0) \ra H^1_{(\varphi, \Gamma)}(t^{-k-1} \cN_{i,j}^0) \ra H^1_{(\varphi, \Gamma)}(t^{-k-1} \cN_{i,j}^0/t^{-k} \cN_{i,j}^0)\\
	\ra H^2_{(\varphi, \Gamma)}(t^{-k} \cN_{i,j}^0) \ra H^2_{(\varphi, \Gamma)}(t^{-k-1} \cN_{i,j}^0) \ra H^2_{(\varphi, \Gamma)}(t^{-k-1} \cN_{i,j}^0/t^{-k} \cN_{i,j}^0).
	\end{multline}
	By \cite[Thm.\ 4.7]{Liu07}, we have $H^2_{(\varphi, \Gamma)}(t^{-k-1} \cN_{i,j}^0/t^{-k} \cN_{i,j}^0)=0$ and
	\[\dim_E H^0_{(\varphi, \Gamma)}(t^{-k-1} \cN_{i,j}^0/t^{-k} \cN_{i,j}^0)=\dim_E H^1_{(\varphi, \Gamma)}(t^{-k-1} \cN_{i,j}^0/t^{-k} \cN_{i,j}^0).\]
	By \cite[Lemma 5.1.1]{BD2}, we have
	\[H^0_{(\varphi, \Gamma)}(t^{-k-1} \cN_{i,j}^0/t^{-k} \cN_{i,j}^0) \cong H^0\big(\Gal_L, t^{-k-1} W_{\dR}^+(\cN_{i,j}^0)/t^{-k} W_{\dR}^+(\cN_{i,j}^0)\big).\]
	As $\Delta_{x_i} \otimes_{\cR_{E,L}} \Delta_{x_j}^{\vee}$ is de Rham of constant Hodge-Tate weight $0$, it follows that $W_{\dR}^+(\cN_{i,j}^0)\cong W_{\dR}^+\big(\cR_{E,L}(\delta_{i} \delta_{j}^{-1})\big)^{\oplus n_in_j}$, hence
\begin{multline*}
H^0\big(\Gal_L, t^{-k-1} W_{\dR}^+(\cN_{i,j}^0)/t^{-k} W_{\dR}^+(\cN_{i,j}^0)\big) \\
\cong H^0\big(\Gal_L, t^{-k-1} W_{\dR}^+\big(\cR_{E,L}(\delta_{i} \delta_{j}^{-1})\big)/t^{-k} W_{\dR}^+\big(\cR_{E,L}(\delta_{i} \delta_{j}^{-1})\big)\big)^{\oplus n_i n_j}.
\end{multline*}
	 By \ \cite[Lemma 2.16]{Na}, \ the \ latter \ is \ zero \ when \ $k$ \ is \ sufficiently \ large. \ We \ conclude \ that $H^s_{(\varphi, \Gamma)}(t^{-k-1} \cN_{i,j}^0/t^{-k} \cN_{i,j}^0)=0$ for $s=\{0,1,2\}$ and $k$ sufficiently large. By (\ref{longNij}) and \cite[(3.11)]{BHS3}, (\ref{cohotinv}) follows. Then (1) follows easily from the proof of Lemma \ref{lem:generic} (see also Remark \ref{geneM1t}) and \cite[Thm.\ 1.2 (1)]{Liu07}.
	
	(2) Since both sides of (\ref{cohog}) have dimension $[L:\Q_p]n_in_j$ over $E$, it suffices to show that the map is injective. By (\ref{cohotinv}), it is enough to show that $H^1_g(t^{-k}\cN_{i,j}^0):=\Ker[H^1(t^{-k} \cN_{i,j}^0) \ra H^1(\Gal_L,W_{\dR}(\cN_{i,j}))]$ is zero when $k$ is sufficiently large. Put 
	\begin{multline*}
		H^1_e\big(t^{k}(\cN_{i,j}^0)^{\vee} \otimes_{\cR_{E,L}}\cR_{E,L}(\chi_{\cyc})\big)\\
		:=\Ker\Big[H^1_{(\varphi, \Gamma)}\big(t^{k}(\cN_{i,j}^0)^{\vee} \otimes_{\cR_{E,L}}\cR_{E,L}(\chi_{\cyc})\big) \xlongrightarrow{f} H^1\big(\Gal_L, W_e(\cN_{i,j} \otimes_{\cR_{E,L}} \cR_{E,L}(\chi_{\cyc}))\big)\Big]
	\end{multline*}
	where $W_e(D)$ denotes the $B_e$-module associated to a $(\varphi, \Gamma)$-module $D$, see \cite[Prop.\ 2.2.6 (1)]{Ber08II}. By \cite[Prop.\ 2.11]{Na}, we only need to show $f$ is zero for $k$ sufficiently large. Let $s$ be an integer small enough so that $t^s (\cN_{i,j}^0)^{\vee} \otimes_{\cR_{E,L}}\cR_{E,L}(\chi_{\cyc})$ has only negative Hodge-Tate weights. For $k\geq s$, the morphism $f$ factors through (see the discussion below \cite[(3.22)]{BHS3})
	\begin{equation}\label{mapks}
		H^1_{(\varphi, \Gamma)}(t^{k}(\cN_{i,j}^0)^{\vee} \otimes_{\cR_{E,L}}\cR_{E,L}(\chi_{\cyc}))\lra H^1_{(\varphi, \Gamma)}(t^{s}(\cN_{i,j}^0)^{\vee} \otimes_{\cR_{E,L}}\cR_{E,L}(\chi_{\cyc})).
	\end{equation}
	By (an easy generalization of) Lemma \ref{lem:generic}, and \cite[Thm.\ 1.2 (1)]{Liu07}, $\dim_E H^1_{(\varphi, \Gamma)}(t^{k}(\cN_{i,j}^0)^{\vee} \otimes_{\cR_{E,L}}\cR_{E,L}(\chi_{\cyc}))=[L:\Q_p]n_in_j$. Using the cohomology of $B$-pairs (cf.\ \cite[\S~3]{Na}), we see that the map in (\ref{mapks}) lies in the following long exact sequence, where $W^+:=W_{\dR}^+((\cN_{i,j}^0)^{\vee} \otimes_{\cR_{E,L}}\cR_{E,L}(\chi_{\cyc}))$:
	\begin{multline*}
		0 \lra H^0_{(\varphi, \Gamma)}\big(t^{k}(\cN_{i,j}^0)^{\vee} \otimes_{\cR_{E,L}}\cR_{E,L}(\chi_{\cyc})\big) \lra H^0_{(\varphi, \Gamma)}\big(t^{s}(\cN_{i,j}^0)^{\vee} \otimes_{\cR_{E,L}}\cR_{E,L}(\chi_{\cyc})\big) \\
		\lra H^0(\Gal_L, t^s W^+/t^k W^+) \lra H^1_{(\varphi, \Gamma)}\big(t^{k}(\cN_{i,j}^0)^{\vee} \otimes_{\cR_{E,L}}\cR_{E,L}(\chi_{\cyc})\big) \\
		\lra H^1_{(\varphi, \Gamma)}\big(t^{s}(\cN_{i,j}^0)^{\vee} \otimes_{\cR_{E,L}}\cR_{E,L}(\chi_{\cyc})\big).
	\end{multline*}
	By the same argument as in the proof of Lemma \ref{lem:generic}, it follows that the terms $H^0_{(\varphi,\Gamma)}(\bullet)$ are both zero. When $k\gg s$, it is easy to see that $\dim_E H^0(\Gal_L, t^s W^+/t^k W^+)=[L:\Q_p] n_i n_j$, hence (\ref{mapks}) and the map $f$ are both zero. This concludes the proof of (2).
	
By d\'evissage, (3) follows from (2) together with the fact the functor $W \mapsto H^1(\Gal_L, W)$, on almost de Rham representations of $\Gal_L$ over $B_{\dR} \otimes_{\Q_p} E$, is right exact (see for example the discussion below \cite[(3.19)]{BHS3}). 
\end{proof}

\begin{theorem}\label{fsm}
	Assume $(\ul{x}, \delta)$ is a generic parameter of $\cM_{\bullet}$ such that $\delta$ is locally algebraic. The induced morphism of groupoids over $\Art(E)$
	\[X_{\cM, \cM_{\bullet}} \lra \widehat{Z_{L_P}(L)}_{\delta} \times_{\widehat{\fz}_{L_P,L}} X_{W, \cF_{\bullet}}\]
	is formally smooth. 
\end{theorem}
\begin{proof}
	The theorem follows by the same argument as in the proof of \cite[Thm.\ 3.4.4]{BHS3} with Lemma 3.4.2 and Lemma 3.4.3 of \textit{loc.\ cit.}\ replaced by Lemma \ref{phiGamCohotin}.
\end{proof}

\begin{corollary}\label{fsm2}
With the assumption of Theorem \ref{fsm}, the morphisms $X_{\cM, \cM_{\bullet}} \ra X_{W, \cF_{\bullet}}$ and $X_{\cM, \cM_{\bullet}}^{\square} \ra X_{W, \cF_{\bullet}}^{\square}$ are formally smooth.
\end{corollary}
\begin{proof}
	By \cite[Lemma 3.5.5]{BHS3}, the morphism $\wt-\wt(\delta): \widehat{Z_{L_P}(L)}_{\delta} \ra \widehat{\fz}_{L_P,L}$ is formally smooth of relative dimension $r$. Together with Theorem \ref{fsm}, the first part of the corollary follows. The second part follows from the first part by base change.
\end{proof}

The following proposition is analogous to \cite[Prop.\ 3.4.6]{BHS3} (for a closed immersion of groupoids over $\Art(E)$, see the discussion before \cite[Prop.\ 3.4.6]{BHS3}).

\begin{proposition}\label{relrepMM}
	Assume that $\cM_{\bullet}$ is generic, then the morphism $X_{\cM, \cM_{\bullet}} \ra X_{\cM}$ of groupoids over $\Art(E)$ is relatively representable and is a closed immersion.
\end{proposition}
\begin{proof}
	Since $\cM_{\bullet}$ is generic, using Lemma \ref{phiGamCohotin} (1) and an argument analogous to \cite[Prop.\ 2.3.6]{BCh}, an $\Omega$-filtration $\cM_{A,\bullet}$ on a deformation $\cM_A$ of $\cM$ is unique if it exists. We deduce that $|X_{\cM, \cM_{\bullet}}|$ is a subfunctor of $|X_{\cM}|$ and that we have an equivalence of groupoids over $\Art(E)$:
	\[X_{\cM,\cM_{\bullet}} \cong X_{\cM} \times_{|X_{\cM}|} |X_{\cM,\cM_{\bullet}}|.\]
Hence we only need to show that $|X_{\cM, \cM_{\bullet}}|\hookrightarrow |X_{\cM}|$ is relatively representable. By \cite[\S~23]{Ma}, it is enough to check the following three conditions:
	\begin{itemize}
		\item[(1)] If $A \ra A'$ is a morphism in $\Art(E)$ and $(\cM_A, \cM_{A, \bullet}, j_A)\in |X_{\cM, \cM_{\bullet}}|(A)$, then $(\cM_A \otimes_A A', \cM_{A, \bullet} \otimes_A A', j_A \otimes_A A')\in |X_{\cM, \cM_{\bullet}}|(A')$.
		\item[(2)] If $A \ra A'$ is an injective morphism in $\Art(E)$, $(\cM_A, j_A)\in |X_{\cM}|(A)$, and if $(\cM_A \otimes_A A', j_A \otimes_A A') \in |X_{\cM, \cM_{\bullet}}|(A')\hookrightarrow |X_{\cM}|(A')$, then $(\cM_A, j_A)\in |X_{\cM, \cM_{\bullet}}|(A)$.
		\item[(3)] If $A, A'\in \Art(E)$, $(\cM_A, j_A)\in |X_{\cM, \cM_{\bullet}}|(A)$ and $(\cM_{A'}, j_{A'})\in |X_{\cM, \cM_{\bullet}}|(A')$, then $(\cM_A \times_{\cM} \cM_{A'}, j_{A\times_E A'}:=j_A \circ \pr_1=j_{A'} \circ \pr_2)\in |X_{\cM, \cM_{\bullet}}|(A \times_E A')$.
	\end{itemize}
	(1) is clear. For (3), we just note that $\cM_{A\times_E A'}:=\cM_A \times_{\cM} \cM_{A'}$ admits the $\Omega$-filtration $\cM_{A\times_E A',i}:=\cM_{A,i} \times_{\cM_i} \cM_{A',i}$. We prove (2). Let $\cM_{A,1}:=\cM_{A',1} \cap \cM_A$: this is a $(\varphi, \Gamma)$-module over $\cR_{E,L}[\frac{1}{t}]$ (as it is a submodule of $\cM_A$) which is equipped with an action of $A$. By an easy variation of the proof of Proposition \ref{repab0}, one can show that there exists a continuous character $\delta_{A,1}: L^{\times} \ra A^{\times}$ such that $\delta_{A,1}\equiv \delta_{1} \pmod{\fm_A}$ and $\cM_{A,1}\cong \Delta_{x_1}\otimes_{\cR_{E,L}} \cR_{A,L}(\delta_{A,1})[1/t]$. As $\cM_A/\cM_{A,1}\hookrightarrow \cM_{A'}/\cM_{A',1}$ and $\cM_{A'}/\cM_{A',1}$ is free over $\cR_{E,L}[1/t]$, so is $\cM_{A}/\cM_{A,1}$. By \cite[Lemma 2.2.3 (i)]{BCh}, $\cM_A/\cM_{A,1}$ is free over $A$. It follows by \cite[Lemma 2.2.3 (ii)]{BCh} that $\cM_A/\cM_{A,1}$ ($\hookrightarrow \cM_{A'}/\cM_{A',1}$) is free over $\cR_{A,L}[\frac{1}{t}]$. Using $\cM_A/\cM_{A,1}\hookrightarrow \cM_{A'}/\cM_{A',1}$, by an induction argument, we can construct the desired filtration $\cM_{A, i}$ on $\cM_A$, and hence $(\cM_A, j_A)\in |X_{\cM, \cM_{\bullet}}|(A)\hookrightarrow |X_{\cM}|(A)$. This concludes the proof. 
\end{proof}
By \cite[Lemma 3.5.3 (i)]{BHS3} and the same argument as in the proof of \textit{loc.\ cit.}, we have (where the third relative representability follows from the second one by base change):
\begin{lemma}\label{relrep}
Assume that $\cM_{\bullet}$ is generic, the morphisms $X_{\cM} \ra X_W$, $X_{\cM, \cM_{\bullet}} \ra X_{W, \cF_{\bullet}}$ and $X_{\cM, \cM_{\bullet}}^{\square} \ra X_{W, \cF_{\bullet}}^{\square}$ are relatively representable. 
\end{lemma}

\begin{proposition}\label{XMM}
Assume that $\cM_{\bullet}$ is a generic, then the groupoid $X_{\cM, \cM_{\bullet}}^{\square}$ over $\Art(E)$ is pro-representable. The functor $|X_{\cM, \cM_{\bullet}}^{\square}|$ is pro-represented by a formally smooth noetherian complete local ring of residue field $E$ and dimension $[L:\Q_p](n^2+\dim_E \ur_P)$.
\end{proposition}
\begin{proof}
	The first statement follows from Lemma \ref{relrep} and Proposition \ref{XWF}. By Proposition \ref{XWF} and Corollary \ref{fsm2}, $|X_{\cM, \cM_{\bullet}}^{\square}|$ is pro-represented by a formally smooth noetherian complete local ring of residue field $E$. We only need to calculate $|X_{\cM, \cM_{\bullet}}^{\square}|(E[\varepsilon]/\varepsilon^2)$.
	
	Assume that $(\ul{x}, \delta)$ is a parameter of $\cM_{\bullet}$. For each $i=1, \dots, r$, we fix an isomorphism $\beta_i: \Delta_{x_i} \otimes_{\cR_{E,L}} \cR_{E,L}(\delta_i)[\frac{1}{t}] \xrightarrow{\sim} \cM_i/\cM_{i-1}$. As in the proof of \cite[Prop.\ 3.5.7]{BHS3}, we introduce a new groupoid $X_{\cM, \cM_{\bullet}}^{\ver}$ over $\Art(E)$: 
	\begin{itemize}
		\item[(1)] The objects of $X_{\cM, \cM_{\bullet}}^{\ver}$ are the $5$-tuples $(A, \cM_A, \cM_{A, \bullet}, j_A, \ul{\beta}_A)$, where $(A, \cM_A, \cM_{A, \bullet}, j_A)$ is an object of $X_{\cM, \cM_{\bullet}}$ and $\ul{\beta}_A=(\beta_{A,i})$ is a collection of isomorphisms $\beta_{A,i}: \Delta_{x_i} \otimes_{\cR_{E,L}} \cR_{A,L}(\delta_{A,i})[\frac{1}{t}]\xrightarrow{\sim} \cM_{A,i}/\cM_{A,i-1}$ (where $\delta_{A,i}$ is as in Lemma \ref{parauni}) which are compatible with $j_A$ and $\ul{\beta}=(\beta_i)$. 
		\item[(2)] A morphism in $X_{\cM,\cM_{\bullet}}^{\ver}$ is a morphism in $X_{\cM, \cM_{\bullet}}$ which is compatible with the liftings of $\ul{\beta}$.
	\end{itemize}
	We have $|X_{\cM, \cM_{\bullet}}^{\ver}|\cong X_{\cM, \cM_{\bullet}}^{\ver}$. For each $i\in \{1,\dots,r\}$, we also use $\cM_{\bullet}$ to denote the induced filtration on $\cM_i$, and we define $X_{\cM_i, \cM_{\bullet}}^{\ver}\cong |X_{\cM_i, \cM_{\bullet}}^{\ver}|$ similarly to $X_{\cM, \cM_{\bullet}}^{\ver}\cong |X_{\cM, \cM_{\bullet}}^{\ver}|$. We first use an induction argument on $i$ (inspired by the proof of \cite[Thm.\ 3.3]{Che13}) to show that the functor $|X_{\cM_i, \cM_{\bullet}}^{\ver}|$ is pro-represented by a formally smooth noetherian complete local ring of residue field $E$. It is easy to see that $|X_{\cM_1, \cM_{\bullet}}^{\ver}|$ is pro-represented by $\widehat{\co}_{\widehat{L^{\times}}, \delta_1}\cong E[[x_1, \dots, x_{[L:\Q_p]+1}]]$. Assume that $|X_{\cM_{i-1}, \cM_{\bullet}}^{\ver}|$ is pro-represented by a formally smooth noetherian complete local ring $R_{i-1}$ of residue field $E$ and dimension \[i-1+[L:\Q_p]\bigg(i-1+\sum_{1\leq j <j'\leq i-1} n_j n_j'\bigg).\]
	Let $S_i$ denote the completion of $R_{i-1} \otimes_E \widehat{\co}_{\widehat{L^{\times}}, \delta_i}$ with respect to the maximal ideal generated by the maximal ideal of $R_{i-1}$ and the one of $\widehat{\co}_{\widehat{L^{\times}},\delta_i}$. So $S_i$ is a noetherian complete local ring which is formally smooth over $E$ of dimension $i+[L:\Q_p](i-1+\sum_{1\leq j <j'\leq i-1} n_j n_j')$. For any morphism $S_i \ra A$ with $A\in \Art(E)$, let $\cM_{A,i-1}$ be the $(\varphi, \Gamma)$-module over $\cR_{A,L}[\frac{1}{t}]$ given by the pull-back along $R_{i-1} \ra S_i \ra A$ of the universal $(\varphi, \Gamma)$-module over $\cR_{R_{i-1},L}[\frac{1}{t}]$, and let $\delta_{A,i}$ be the character $L^{\times} \ra \widehat{\co}_{\widehat{L^{\times}}, \delta_i}^{\times} \ra S_i^{\times} \ra A^{\times}$. Let 
	\begin{equation*}
		N_i:=\varprojlim_{\substack{A\in \Art(E) \\ S_i \twoheadrightarrow A}} \Ext^1_{(\varphi, \Gamma)}\Big(\Delta_{x_i}\otimes_{\cR_{E,L}}\cR_{A,L}(\delta_{A,i})\Big[\frac{1}{t}\Big], \cM_{A,i-1}\Big).
	\end{equation*}
	Since $\cM_{\bullet}$ is generic, we deduce by Lemma \ref{phiGamCohotin} (1) (and a d\'evissage) that $N_i$ is a free $S_i$-module of rank $[L:\Q_p] (n_i \sum_{j=1}^{i-1} n_j)$. By definition, $[\cM_i]\in \Ext^1_{(\varphi, \Gamma)}(\Delta_{x_i} \otimes_{\cR_{E,L}} \cR_{E,L}(\delta_i)[1/t], \cM_{i-1})$, which corresponds then to a maximal ideal $\fm_i$ with residue field $E$ of the polynomial $S_i$-algebra $\Symm_{S_i} N_i^{\vee}$. We let $R_i$ be the completion of $\Symm_{S_i} N_i^{\vee}$ at $\fm_i$. Thus
	\[R_i\cong E\big[\!\big[x_1, \dots, x_{i+[L:\Q_p](i+\sum_{1\leq j<j' \leq i} n_jn_j')}\big]\!\big]\]
	and one can directly check that $|X_{\cM_i, \cM_{\bullet}}^{\ver}|$ is pro-represented by $R_i$. In particular $|X_{\cM, \cM_{\bullet}}^{\ver}|$ is pro-represented by a formally smooth noetherian complete local ring of dimension $r+[L:\Q_p] \dim \ur_{P}$.
	
	Now we define $X_{\cM, \cM_{\bullet}}^{\square, \ver}$ as $X_{\cM, \cM_{\bullet}}^{\square} \times_{X_{\cM, \cM_{\bullet}}} X_{\cM, \cM_{\bullet}}^{\ver}$. Since $X_{\cM, \cM_{\bullet}}^{\ver}\cong |X_{\cM, \cM_{\bullet}}^{\ver}|$ is pro-represen\-table, it is easy to see that $X_{\cM, \cM_{\bullet}}^{\square, \ver}$ is pro-representable (by adding formal variables corresponding to the framing $\alpha$). The morphism $X_{\cM, \cM_{\bullet}}^{\square, \ver} \ra X_{\cM, \cM_{\bullet}}^{\ver}$ is formally smooth of relative dimension $n^2[L:\Q_p]$. As $X_{\cM, \cM_{\bullet}}^{\square, \ver}$ can be constructed from $X_{\cM, \cM_{\bullet}}^{\square}$ by adding frames (with respect to $\ul{\beta}$) and $\End_{(\varphi, \Gamma)}(\Delta_{x_i} \otimes_{\cR_{E,L}} \cR_{A,L}(\delta_{A,i})[1/t])\cong A$ (which follows from the proof of Lemma \ref{parauni}), the morphism $X_{\cM, \cM_{\bullet}}^{\square, \ver} \ra X_{\cM, \cM_{\bullet}}^{\square}$ is formally smooth of relative dimension $r$. We then compute:
	\begin{equation*}
		\dim_E |X_{\cM, \cM_{\bullet}}^{\square}|(E[\varepsilon]/\varepsilon^2)=n^2[L:\Q_p]+(r+[L:\Q_p] \ur_P)-r=[L:\Q_p](n^2+\dim \ur_P) 
	\end{equation*}
which concludes the proof.
\end{proof}

\subsection{$(\varphi,\Gamma)$-modules of type $\Omega$ over $\cR_{E,L}$}

We study certain groupoids of deformations of a $(\varphi, \Gamma)$-module over $\cR_{E,L}$ equipped with an $\Omega$-filtration. We keep the notation of \S~\ref{sec62}.

Let $D$ be a $(\varphi, \Gamma)$-module of rank $n$ over $\cR_{E,L}$ and $\cM:=D[\frac{1}{t}]$. Let $W^+:=W_{\dR}^+(D)$ be the associated $B_{\dR}^+ \otimes_{\Q_p} E$-representation of $\Gal_L$ and $W:=W_{\dR}(\cM)\cong W^+[\frac{1}{t}]$. Assume that $W$ is almost de Rham. We define the groupoids $X_D$, $X_{W^+}$ over $\Art(E)$ of deformations of (respectively) $D$, $W^+$ as in \cite[\S~3.5]{BHS3}. Recall that we have natural morphisms $X_D \ra X_{W^+}$ (induced by the functor $W_{\dR}^+(-)$) and $X_D \ra X_{\cM}$, $X_{W^+} \ra X_W$ (inverting $t$), and that the following diagram commutes:
\begin{equation*}
	\begin{CD}
		X_D @>>> X_{\cM} \\
		@VVV @VVV \\
		X_{W^+} @>>> X_W.
	\end{CD}
\end{equation*}
By \cite[Prop.\ 3.5.1]{BHS3}, the induced morphism $X_D \ra X_{\cM} \times_{X_{W}} X_{W^+}$ is an equivalence.
Fix an isomorphism $\alpha: (L \otimes_{\Q_p} E)^n \xrightarrow{\sim} D_{\pdR}(W)$ so we have the groupoid $X_W^{\square}$ over $\Art(E)$ (\S~\ref{sec61} or \cite[\S~3.1]{BHS3}). We put
\begin{equation*}
	X_{W^+}^{\square}:=X_{W^+} \times_{X_{W}} X_W^{\square}, \ X_D^{\square}:=X_D \times_{X_W} X_W^{\square}. 
\end{equation*}
We assume that $D$ has distinct Sen weights $(h_{1,\tau}>h_{1,\tau}> \cdots> h_{n,\tau})_{\tau \in \Sigma_L}$. Then $W^+$ is regular in the sense of \cite[Def.\ 3.2.4]{BHS3}. Let $(A, W_A^+, \iota_A,\alpha_A) \in X_{W^+}^{\square}$ and $W_A:=W_A^+[1/t]$, then the $B_{\dR}^+ \otimes_{\Q_p} A$-lattice $W_A^+$ of $W_A$ induces a complete flag $\Fil_{W_A^+, \bullet}=\Fil_{W_A^+, \bullet}(D_{\pdR}(W_A)):=(\Fil_{W_A^+, i}(D_{\pdR}(W_A)))_{i=1, \dots, n}$ of $D_{\pdR}(W_A)$ by the formula
\begin{equation*}
	\Fil_{W_A^+, i}(D_{\pdR}(W_A)):=\bigoplus_{\tau\in \Sigma_L} \Fil_{W_A^+}^{-h_{n+1-i,\tau}}(D_{\pdR, \tau}(W_A)):=\bigoplus_{\tau \in \Sigma_L} (t^{-h_{n+1-i,\tau}} W_A^+)_{\tau}^{\Gal_L},
\end{equation*}
where
\begin{eqnarray*}
D_{\pdR, \tau}(W_A)&:=&D_{\pdR}(W_A)\otimes_{L \otimes_{\Q_p} E} (L \otimes_{L,\tau} E)\\
(t^{-h_{n+1-i,\tau}} W_A^+)_{\tau}&:=&t^{-h_{n+1-i,\tau}}\big(W_A^+ \otimes_{L \otimes_{\Q_p} E} (L \otimes_{L,\tau} E)\big).
\end{eqnarray*}
Since the flag $\Fil_{W_A^+, \bullet}$ is stable under the endomorphism $\nu_{W_A}$ of $D_{\pdR}(W_A)$ (see \S~\ref{sec61}), it follows that we have:
\[\big(\alpha_A^{-1} (\Fil_{W_A^+,\bullet}), N_{W_A}=\alpha_A^{-1} \circ \nu_{W_A} \circ \alpha_A\big)\in \tilde{\ug}_L(A).\]
Denote by $\widehat{\tilde{\ug}}_L$ the completion of $\tilde{\ug}_L$ at $(\alpha^{-1} (\Fil_{W^+,\bullet}), N_{W})$. By \cite[Thm.\ 3.2.5]{BHS3}, $X_{W^+}^{\square}$ is pro-representable and we have an isomorphism of functors
\begin{equation*}
	|X_{W^+}^{\square}| \xlongrightarrow{\sim} \widehat{\tilde{\ug}}_L, \ (W_A^+, \iota_A, \alpha_A)\longmapsto \big(\alpha_A^{-1} (\Fil_{W_A^+,\bullet}), N_{W_A}\big).
\end{equation*}
Consider the composition
\[\kappa_{W^+}: X_{W^+}^{\square} \lra | X_{W^+}^{\square}| \xlongrightarrow{\sim} \widehat{\tilde{\ug}}_L \xlongrightarrow{\kappa_B} \widehat{\ft}_L\]
where $\widehat{\ft}_L$ denotes the completion of $\ft_L$ at $0$ (and $\kappa_B$ is defined in (\ref{kappaBP})). The morphism $\kappa_{W^+}$ factors through a map still denoted by $\kappa_{W^+}: X_{W^+} \lra \widehat{\ft}_L$.

We call $D$ \textit{of type $\Omega$} if $D$ admits an $\Omega$-filtration $D_{\bullet}$ (see Definition \ref{defOF} (1)). The $D_i$ for $i=1,\dots,r$ are saturated $(\varphi, \Gamma)$-submodules of $D$ and $\cM_{\bullet}:=(\cM_i)_{0\leq i \leq r}:=(D_i[\frac{1}{t}])_{0 \leq i \leq r}$ is an $\Omega$ filtration of $\cM$ as in \S~\ref{sec62}. We assume that $\cM_{\bullet}$ is generic (which implies that $D_{\bullet}$ is generic, see Remark \ref{geneM1t} (1)). Assume $D$ is of type $\Omega$ and let $D_{\bullet}$ be an $\Omega$-filtration of $D$. We put $\cF_{\bullet}:=W_{\dR}(\cM_{\bullet})=W_{\dR}^+(D_{\bullet})[\frac{1}{t}]$ and define the following groupoids over $\cR_{E,L}$\index{$X_{W^+, \cF_{\bullet}}$} \index{$X_{D, \cM_{\bullet}}$} \index{$X_{W^+, \cF_{\bullet}}^{\square}$} \index{$X_{D, \cM_{\bullet}}^{\square}$}
\begin{eqnarray*}
	X_{W^+, \cF_{\bullet}}:=X_{W^+} \times_{X_W} X_{W, \cF_{\bullet}}, && X_{W^+, \cF_{\bullet}}^{\square}:=X_{W^+, \cF_{\bullet}} \times_{X_W} X_W^{\square}=X_{W^+} \times_{X_W} X_{W, \cF_{\bullet}}^{\square},\\
	X_{D, \cM_{\bullet}}:=X_D \times_{X_{\cM}} X_{\cM, \cM_{\bullet}}, &&X_{D, \cM_{\bullet}}^{\square}:=X_{D, \cM_{\bullet}} \times_{X_D} X_D^{\square}=X_{D, \cM_{\bullet}} \times_{X_W} X_W^{\square},
\end{eqnarray*}
where we have used $X_D^{\square}=X_D \times_{X_W} X_W^{\square}$.

\begin{proposition}\label{relrepD}
	The morphisms of groupoids $X_{D, \cM_{\bullet}} \ra X_{W^+, \cF_{\bullet}}$ and $X_{D, \cM_{\bullet}}^{\square} \ra X_{W^+, \cF}^{\square}$ are formally smooth and relatively representable. 
\end{proposition}
\begin{proof}
Since $X_D \cong X_{\cM} \times_{X_{W}} X_{W^+}$, we have $X_{D, \cM_{\bullet}} \cong (X_{W^+} \times_{X_W} X_{\cM}) \times_{X_{\cM}} X_{\cM, \cM_{\bullet}}\cong X_{W^+} \times_{X_W} X_{\cM, \cM_{\bullet}}$. The first part then follows by base change from Lemma \ref{relrep} and Corollary \ref{fsm2}. The second part follows from the first again by base change. 
\end{proof}

We define:
\[\cD_{\bullet}:=(\cD_i)_{1 \leq i \leq r}:=(D_{\pdR}(\cF_i))_{1\leq i \leq r}=\big(D_{\pdR}(W_{\dR}(\cM_i))\big)_{1\leq i \leq r}.\]
Using $X_{P,L}\cong \tilde{\ug}_{P,L} \times_{\ug_L} \tilde{\ug}_L$, we have:
\begin{equation}\label{assPtGS}
	y:=\big(\alpha^{-1}(\cD_{\bullet}), \alpha^{-1}(\Fil_{W^+, \bullet}), N_W\big)\in X_{P,L}(E).
\end{equation}

\begin{proposition}\label{XDM}
	(1) The \ groupoid \ $X_{W^+, \cF_{\bullet}}^{\square}$ \ over \ $\Art(E)$ \ is \ pro-representable. \ The \ functor $|X_{W^+, \cF_{\bullet}}^{\square}|$ is pro-represented by the formal scheme $\widehat{X}_{P,L,y}$.
	
	(2) The groupoid $X_{D, \cM_{\bullet}}^{\square}$ over $\Art(E)$ is pro-representable. The functor $|X_{D, \cM_{\bullet}}^{\square}|$ is pro-represen\-ted by a formal scheme which is formally smooth of relative dimension $[L:\Q_p] \dim \fp$ over $\widehat{X}_{P,L, y}$. 
\end{proposition}
\begin{proof}
	(1) follows by the same argument as in the proof of \cite[Cor.\ 3.5.8]{BHS3} (with \cite[Cor.\ 3.1.9]{BHS3} replaced by Proposition \ref{XWF}). Using $X_{D, \cM_{\bullet}} \cong X_{W^+} \times_{X_W} X_{\cM, \cM_{\bullet}}$ and $X_{W^+, \cF_{\bullet}}=X_{W^+} \times_{X_W} X_{W, \cF_{\bullet}}$, we deduce $X_{D, \cM_{\bullet}} \cong X_{W^+, \cF_{\bullet}} \times_{X_{W, \cF_{\bullet}}} X_{\cM, \cM_{\bullet}}$ and thus:
	\begin{equation}\label{missing}
	X_{D, \cM_{\bullet}}^{\square} \cong X_{W^+, \cF_{\bullet}}^{\square} \times_{X_{W, \cF_{\bullet}}} X_{\cM, \cM_{\bullet}}\cong X_{W^+, \cF_{\bullet}}^{\square} \times_{X_{W, \cF_{\bullet}}^{\square}} X_{\cM, \cM_{\bullet}}^{\square}.
	\end{equation}
The first part of (2) follows from (1) and Proposition \ref{relrepD}. From (\ref{missing}) and the fact that, for each groupoid $Y$ in the fibre product on the right hand side of (\ref{missing}), we have $Y\cong |Y|$ (see Remark \ref{equivgrfu} for $X_{W,\cF_{\bullet}}^{\square}$, the others being similar),
we deduce $|X_{D, \cM_{\bullet}}^{\square}| \cong |X_{W^+, \cF_{\bullet}}^{\square}| \times_{|X_{W, \cF_{\bullet}}^{\square}|} |X_{\cM, \cM_{\bullet}}^{\square}|$. By Corollary \ref{fsm2}, Proposition \ref{XWF} and Proposition \ref{XMM}, the functor $|X_{\cM, \cM_{\bullet}}^{\square}|$ is formally smooth over $|X_{W, \cF_{\bullet}}^{\square}|$ of relative dimension
\[[L:\Q_p](n^2+\dim \ur_P)-[L:\Q_p](\dim \fn_P + \dim \ur_P)=[L:\Q_p] \dim \fp.\]
The second part of (2) then follows by base change from the above fiber product.
\end{proof}

For $w=(w_{\tau})_{\tau\in \Sigma_L}\in \sW_{L}\cong \sW^{|\Sigma_L|}$, let $X_w:=\prod_{\tau\in \Sigma_L} X_{w_{\tau}}\hookrightarrow X_{P,L}$, where $X_{w_{\tau}} \hookrightarrow X_P$ is defined as in \S~\ref{sec: grgg}. This is an irreducible component of $X_{P,L}$ which only depends on the coset $\sW_{L_P,L}w$. We put (using Proposition \ref{XDM} (1) and with $\widehat{X}_{w,y}$ empty if $y\notin X_{w}$)\index{$X_{W^+, \cF_{\bullet}}^{\square, w}$}:
\begin{equation}\label{XWFw}
	X_{W^+,\cF_{\bullet}}^{\square,w}:=X_{W^+, \cF_{\bullet}}^{\square} \times_{|X_{W^+, \cF_{\bullet}}^{\square}|} \widehat{X}_{w,y}.
\end{equation}
 
\begin{corollary}\label{prorepr+}
The groupoid $X_{W^+, \cF_{\bullet}}^{\square, w}$ over $\Art(E)$ is pro-representable. The functor $|X_{W^+, \cF_{\bullet}}^{\square, w}|$ is pro-represented by the formal scheme $\widehat{X}_{w,y}$.
\end{corollary}
\begin{proof}
This follows from (\ref{XWFw}) and the equivalence of groupoids $X_{W^+, \cF_{\bullet}}^{\square} \xrightarrow{\sim} |X_{W^+, \cF_{\bullet}}^{\square}|$.
\end{proof}

We define the groupoid $X_{W^+, \cF_{\bullet}}^{w}$ over $\Art(E)$ as the subgroupoid of $X_{W^+, \cF_{\bullet}}$ which is the image of $X_{W^+, \cF_{\bullet}}^{\square, w}$ by the forgetful morphism $X_{W^+, \cF_{\bullet}}^{\square} \ra X_{W^+, \cF_{\bullet}}$.\index{$X_{W^+, \cF_{\bullet}}^{w}$} Thus, for $A\in \Art(E)$, the objects of $X_{W^+, \cF_{\bullet}}^w(A)$ are the quadruples $(W_A^+, \cF_{A, \bullet}, \iota_A, j_A)$ in $X_{W^+, \cF_{\bullet}}(A)$ such that
\[\big(\alpha_A^{-1}(\Fil_{W_A^+, \bullet}), \alpha_A^{-1}(\cD_{A, \bullet}), \alpha_A^{-1}\circ \nu_{W_A}\circ \alpha_A\big)\in X_w(A)\]
for one (equivalently any) isomorphism $\alpha_A: (L \otimes_{\Q_p} A) \xrightarrow{\sim} D_{\pdR}(W_A)$. As in \cite[(3.26)]{BHS3} there is an equivalence of groupoids $X_{W^+, \cF_{\bullet}}^{\square, w} \xrightarrow{\sim} X_{W^+, \cF_{\bullet}}^w \times_{X_{W^+, \cF_{\bullet}}} X_{W^+, \cF_{\bullet}}^{\square}$. We define then\index{$X_{D, \cM_{\bullet}}^{\square, w}$} \index{$X_{D, \cM_{\bullet}}^{w}$}
\begin{equation}\label{XDMbullet}
	X_{D, \cM_{\bullet}}^{\square, w}:=X_{D, \cM_{\bullet}}^{\square} \times_{X_{W^+, \cF_{\bullet}}^{\square}} X_{W^+, \cF_{\bullet}}^{\square, w}\ {\rm and}\ X_{D, \cM_{\bullet}}^{w}:=X_{D, \cM_{\bullet}} \times_{X_{W^+, \cF_{\bullet}}} X_{W^+, \cF_{\bullet}}^{w}.
\end{equation}

\begin{proposition}\label{wclos}
	The morphisms of groupoids $ X_{W^+, \cF_{\bullet}}^w \ra X_{W^+, \cF_{\bullet}}$, $X_{W^+, \cF_{\bullet}}^{\square, w} \ra X_{W^+, \cF_{\bullet}}^{\square}$, $X_{D, \cM_{\bullet}}^{w} \ra X_{D, \cM_{\bullet}}$ and $X_{D, \cM_{\bullet}}^{\square, w} \ra X_{D, \cM_{\bullet}}^{\square}$ are relatively representable and are closed immersions.
\end{proposition}
\begin{proof}
	The proposition follows by the same argument as in the proof of \cite[Prop.\ 3.5.10]{BHS3} with \cite[(3.25)]{BHS3} replaced by (\ref{XWFw}).
\end{proof}

Define
\begin{equation}\label{cSy}
\begin{gathered}
\begin{array}{ccl}
\cS(y)&:=&\{w\in \sW_L\ |\ y\in X_{w}(E)\}=\{w \in \sW_L\ |\ \widehat{X}_{w,y}\neq \emptyset\}\\
&=&\{w\in \sW_L\ |\ X_{W^+, \cF_{\bullet}}^w \neq 0\}=\{w\in \sW_L\ |\ X_{D, \cM_{\bullet}}^w \neq 0\}.
\end{array}
\end{gathered}
\end{equation}
Each groupoid $Y$ in the isomorphism $X_{D, \cM_{\bullet}}^{\square, w} \cong X_{D, \cM_{\bullet}}^{\square} \times_{X_{W^+, \cF_{\bullet}}^{\square}} X_{W^+, \cF_{\bullet}}^{\square, w}$ is equivalent to the associated functor $|Y|$ (as all the automorphisms of an object in the groupoid are trivial). Hence $|X_{D, \cM_{\bullet}}^{\square, w}| \cong |X_{D, \cM_{\bullet}}^{\square}| \times_{|X_{W^+, \cF_{\bullet}}^{\square}|} |X_{W^+, \cF_{\bullet}}^{\square, w}|$. From Proposition \ref{XDM} and Corollary \ref{prorepr+}, we deduce:

\begin{corollary}\label{XDMw}
	If $w\in \cS(y)$, the functor $|X_{D, \cM_{\bullet}}^{\square, w}|$ is pro-representable by a noetherian complete local $E$-algebra which is formally smooth of relative dimension $[L:\Q_p]\dim \fp$ over $\widehat{X}_{w,y}$.
\end{corollary}

The map $(\kappa_B, \kappa_P): X_{P,L} \ra \ft_L \times_{\ft_L/\sW_L} \fz_{L_P,L}={\sT}_{P,L}$ induces a morphism $\widehat{X}_{P, L,y} \ra \widehat{\sT}_{P,L,(0,0)}$. Denote by $\Theta$ the composition
\begin{equation*}
	X_{D, \cM_{\bullet}}^{\square} \lra X_{W^+, \cF_{\bullet}}^{\square} \xlongrightarrow{\sim} |X_{W^+, \cF_{\bullet}}^{\square}| \xlongrightarrow{\sim} \widehat{X}_{P,L,y} \lra \widehat{\sT}_{P,L, (0,0)},
\end{equation*}
which factors through a morphism still denoted by $\Theta: X_{D, \cM_{\bullet}} \ra \widehat{\sT}_{P,L,(0,0)}$. By Lemma \ref{Theta0}, we have

\begin{corollary}\label{ThetaDM}
	Let $w\in \cS(y)$ and $w'\in \sW_L$, then the morphisms $X_{D, \cM_{\bullet}}^{\square,w}\hookrightarrow X_{D, \cM_{\bullet}}^{\square} \ra \widehat{\sT}_{P,L,(0,0)}$ and $X_{D, \cM_{\bullet}}^{w}\hookrightarrow X_{D, \cM_{\bullet}} \ra \widehat{\sT}_{P,L,(0,0)}$ of groupoids over $\Art(E)$ induced by $\Theta$ factor through the embedding $\widehat{\sT}_{w',(0,0)}\hookrightarrow \widehat{\sT}_{P,L,(0,0)}$ if and only if $\sW_{L_P,L}w'^=\sW_{L_P,L}w$. 
\end{corollary}

\subsection{Galois representations and Bernstein paraboline varieties}\label{secGrDV}

We show that the completed local rings of the Bernstein paraboline varieties at generic points of distinct integral weights can be described (up to formally smooth morphisms) by completed local rings on the variety $X_{P,L}$.

Let $\rho:\Gal_L \ra \GL_n(E)$ be a continuous group morphism and let $V$ be the associated representation of $\Gal_L$. Let $X_{\rho}$ be the groupoid over $\Art(E)$ of deformations of the group morphism $\rho$, and $X_V$ be the groupoid over $\Art(E)$ of deformations of the representation $V$ (so $X_{\rho}$ can be viewed as the groupoid of framed deformations of $V$). The natural morphism $X_{\rho} \ra X_V$ is relatively representable and formally smooth of relative dimension $n^2$. Let $D:=D_{\rig}(V)$, we have then an equivalence $X_V \xrightarrow{\sim} X_D$. The morphism $X_{\rho} \ra |X_{\rho}|$ is an equivalence. In fact, this holds for any groupoid with ``$\rho$" in subscript in this section.

Assume that $D$ is almost de Rham with distinct Sen weights and that $D$ admits a generic $\Omega$-filtration. Let $\cM_{\bullet}$ be a generic $\Omega$-filtration on $\cM=D[\frac{1}{t}]$ (recall $\Omega$ is fixed in \S~\ref{sec62}) and put $X_{V, \cM_{\bullet}}:=X_V \times_{X_D} X_{D, \cM_{\bullet}}$, $X_{\rho, \cM_\bullet}:=X_{\rho} \times_{X_V} X_{V,\cM_{\bullet}} \cong X_{\rho} \times_{X_{\cM}} X_{\cM, \cM_{\bullet}}$. Note that $X_{\rho, \cM_{\bullet}} \ra X_{\rho}$ is a closed immersion by Proposition \ref{relrepMM} and base change. For $w\in \sW_L$, we put
\[X_{V, \cM_{\bullet}}^w:=X_V \times_{X_D} X_{D, \cM_{\bullet}}^w\textrm{ and }X_{\rho, \cM_{\bullet}}^w:=X_{\rho} \times_{X_V} X_{V, \cM_{\bullet}}^w\index{$X_{\rho, \cM_{\bullet}}^{w}$}\]
(where we use $W^+:=W_{\dR}^+(D)\cong B_{\dR}^+ \otimes_{\Q_p} V$ in the definition of $X_{D, \cM_{\bullet}}^w$, see (\ref{XDMbullet})). Let $y\in X_{P,L}$ be as in (\ref{assPtGS}) and $\cS(y)$ as in (\ref{cSy}).

\begin{theorem}\label{thmrM}
	(1) The functor $|X_{\rho, \cM_{\bullet}}|$ (resp.\ $|X_{\rho, \cM_{\bullet}}^w|$ for $w\in \cS(y)$) is pro-representable by an equidimensional noetherian complete local ring $R_{\rho, \cM_{\bullet}}$ (resp.\ $R_{\rho,\cM_{\bullet}}^w$) of residue field $E$ and dimension $n^2+[L:\Q_p](\frac{n(n-1)}{2}+r)$. If $\rho$ is moreover de Rham\footnote{Or equivalently potentially crystalline as $D$ admits a generic $\Omega$-filtration.}, then $R_{\rho,\cM_{\bullet}}^w\cong R_{\rho,\cM_{\bullet}}/\fp_w$ for a minimal prime ideal $\fp_w$ of $R_{\rho,\cM_{\bullet}}$. Finally, in this case, the map $w \mapsto \fp_w$ is a bijection between $\cS(y)$ and the set of minimal prime ideals of $R_{\rho,\cM_{\bullet}}$.\index{$R_{\rho,\cM_{\bullet}}$} \index{$R_{\rho,\cM_{\bullet}}^w$}
	
	(2) The morphism $|X_{\rho, \cM_{\bullet}}^{w'}| \ra |X_{V, \cM_{\bullet}}^{w'}| \hookrightarrow |X_{V, \cM_{\bullet}}| \cong |X_{D, \cM_{\bullet}}| \xrightarrow{\Theta} \widehat{\sT}_{P,L,(0,0)}$ of groupoids over $\Art(E)$ factors through $\widehat{\sT}_{w,(0,0)} \hookrightarrow \widehat{\sT}_{P,L,(0,0)}$ if and only if $\sW_{L_P,L}w'=\sW_{L_P,L}w$.
\end{theorem}
\begin{proof}
	(1) We have $X_{\rho, \cM_{\bullet}}\cong X_{\rho} \times_{X_{\cM}} X_{\cM, \cM_{\bullet}}$. By Proposition \ref{relrepMM} and the fact that $X_{\rho}$ is pro-representable, we deduce that $X_{\rho, \cM_{\bullet}}$ is pro-representable. We have $X_{\rho, \cM_{\bullet}}\xrightarrow{\sim} |X_{\rho, \cM_{\bullet}}|$ and we let $R_{\rho, \cM_{\bullet}}$ be the noetherian complete local ring which pro-represents $ |X_{\rho, \cM_{\bullet}}|$. Define
	\[X_{\rho,\cM_{\bullet}}^{\square}:=X_{\rho, \cM_{\bullet}} \times_{X_{\cM, \cM_{\bullet}}} X_{\cM, \cM_{\bullet}}^{\square} \cong X_{\rho} \times_{X_{\cM}} X_{\cM, \cM_{\bullet}}^{\square}\]
	which is formally smooth of relative dimension $[L:\Q_p]n^2$ over $X_{\rho, \cM_{\bullet}}$ by base change (see (\ref{moreX})). Since $X_{\rho} \ra X_V \cong X_D$ is relatively representable and formally smooth of relative dimension $n^2$, so is the morphism $X_{\rho, \cM_{\bullet}}^{\square} \ra X_D \times_{X_{\cM}} X_{\cM, \cM_{\bullet}}^{\square} \cong X_{D, \cM_{\bullet}}^{\square}$. Together with Proposition \ref{XDM} and Corollary \ref{irrcmpXP}, we deduce that $R_{\rho,\cM_{\bullet}}$ is equidimensional and 
	\begin{eqnarray}\label{dimmissing}
		\nonumber\dim R_{\rho, \cM_{\bullet}}&=&n^2+([L:\Q_p]\dim \fp) + \Big([L:\Q_p]\big(\frac{n(n-1)}{2}+\dim \ur_P\big)\Big)-n^2[L:\Q_p]\\
		&=&n^2+[L:\Q_p]\Big(\frac{n(n-1)}{2}+r\Big).
	\end{eqnarray}
	Let $X_{\rho, \cM_{\bullet}} ^{\square,w}:=X_{\rho, \cM_{\bullet}}^w \times_{X_{\cM, \cM_{\bullet}}} X_{\cM, \cM_{\bullet}}^{\square} \cong X_{D, \cM_{\bullet}}^{\square,w} \times_{X_D} X_{\rho}$. As $X_{\rho}$ is relatively representable over $X_D$, we deduce by Corollary \ref{XDMw} (and the fact $X_{D,\cM_{\bullet}}^{\square, w} \xrightarrow{\sim} |X_{D, \cM_{\bullet}}^{\square, w}|$) that $X_{\rho, \cM_{\bullet}}^{\square, w}$ is pro-representable. It is also easy to see $X_{\rho, \cM_{\bullet}} ^{\square,w} \xrightarrow{\sim} |X_{\rho, \cM_{\bullet}} ^{\square,w}|$. As $X_{\rho}$ is formally smooth of relative dimension $n^2$ over $X_D$, using Corollary \ref{XDMw} again we have formally smooth morphisms (the first of relative dimension $[L:\Q_p]n^2$, the second of relative dimension $n^2+[L:\Q_p]\dim \fp$)
	\begin{equation}\label{fsmors}
		|X_{\rho, \cM_{\bullet}}^w| \longleftarrow |X_{\rho, \cM_{\bullet}} ^{\square,w}| \lra \widehat{X}_{w,y}.
	\end{equation}
As $X_{\rho,\cM_{\bullet}}^w \cong X_{D,\cM_{\bullet}}^w \times_{X_{D, \cM_{\bullet}}} X_{\rho, \cM_{\bullet}}$, by (\ref{dimmissing}), Proposition \ref{wclos} and (\ref{fsmors}) (and the fact $X_{\rho,\cM_{\bullet}}^w \xrightarrow{\sim} |X_{\rho,\cM_{\bullet}}^w|$), it follows that $|X_{\rho, \cM_{\bullet}}^w|$ is pro-representable by a (reduced) local complete noetherian ring of residue field $E$ and dimension $n^2+[L:\Q_p](\frac{n(n-1)}{2}+r)$. When $\rho$ is de Rham, the parameter $N_W$ in $y$ is zero, hence $\widehat{\co}_{X_w,y}$ is a domain by Theorem \ref{unibranch}. The second part of (1) follows. Using Corollary \ref{irrcmpXP}, the last part of (1) also follows. Part (2) follows easily from Corollary \ref{ThetaDM}.
\end{proof}

\begin{corollary}
	For $w\in \cS(y)$, we have 
	\begin{equation*}
		\dim_E |X_{\rho, \cM_{\bullet}}^w|(E[\varepsilon]/\varepsilon^2)=[L:\Q_p] \dim \fp+n^2-n^2[L:\Q_p] +\dim_E \widehat{X}_{w,y}(E[\varepsilon]/\varepsilon^2).
	\end{equation*}
\end{corollary}
\begin{proof}
	The groupoid $X_{\rho,\cM_{\bullet}}^{\square, w}$ is formally smooth of relative dimension $n^2$ over $X_{V, \cM_{\bullet}}^{\square,w}\cong X_{D, \cM_{\bullet}}^{\square,w}$, and is formally smooth of relative dimension $[L:\Q_p] n^2$ over $X_{\rho,\cM_{\bullet}}^{w}$. We have $Y\xrightarrow{\sim} |Y|$ for $Y\in\{X_{\rho,\cM_{\bullet}}^{\square, w}, X_{\rho,\cM_{\bullet}}^{w}, X_{D,\cM_{\bullet}}^{\square, w}\}$, and thus
	\begin{eqnarray*}
		\dim_E |X_{\rho, \cM_{\bullet}}^w|(E[\varepsilon]/\varepsilon^2)&=&\dim_E|X_{D,\cM_{\bullet}}^{\square, w}|(E[\varepsilon]/\varepsilon^2)+n^2-[L:\Q_p]n^2 \\
		&=&[L:\Q_p] \dim \fp+n^2-n^2[L:\Q_p] +\dim_E \widehat{X}_{w,y}(E[\varepsilon]/\varepsilon^2)
	\end{eqnarray*}
	where the second equality follows from Corollary \ref{XDMw}.
\end{proof}

Let $w_y\in \sW^P_{\max,L}$ such that $\pi(y)=(\alpha^{-1}(\cD_{\bullet}), \alpha^{-1}(\Fil_{W^+, \bullet}))$ lies in the $G_L$-orbit of $(w_y,1)$ in $G_L/B_L \times G_L/P_L$. By Lemma \ref{w'w} and the equalities in (\ref{cSy}), we have

\begin{proposition}\label{exCon}
	Let $w\in \sW_L$, if $X_{\rho, \cM_{\bullet}}^w\neq \emptyset$, then $w^{\max} \geq w_{y}$. 
\end{proposition}

Now fix a group morphism $\overline{\rho}: \Gal_L \ra \GL_n(k_E)$ and a strictly $P$-dominant weight $\textbf{h}\in \Z^{n|\Sigma_L|}$ of $G_L$ as in \S~\ref{s: DO}. Let $x=(\rho, \ul{x}, \chi)$ be a point in $X_{\Omega, \textbf{h}}(\overline{\rho}) \hookrightarrow \fX_{\overline{\rho}} \times (\Spec \cZ_{\Omega})^{\rig} \times \widehat{\cZ_{0,L}}$. Assume that $\rho$ is almost de Rham (hence $\chi$ is locally algebraic by Proposition \ref{Senwt}) and has distinct Sen weights.

\begin{corollary}\label{uniqMfil}
We have that $\cM:=D_{\rig}(\rho)[\frac{1}{t}]$ has an $\Omega$-filtration $\cM_{\bullet}$ of parameter $(\ul{x}, \chi_{\varpi_L})$. Moreover, if the parameter $(\ul{x}, \chi_{\varpi_L})$ is generic, then $\cM_{\bullet}$ is the unique $\Omega$-filtration of parameter $(\ul{x}, \chi_{\varpi_L})$ on $\cM$. 
\end{corollary}
\begin{proof}
	The existence follows from Corollary \ref{OFpw}. By Lemma \ref{phiGamCohotin} (1) and the same argument as in the proof of Corollary \ref{Filuniq}, the uniqueness follows.
\end{proof}

Recall that $X_{\rho}\cong |X_{\rho}|$ is equivalent to $\widehat{(\fX_{\overline{\rho}})}_{\rho}$ (cf.\ \cite[\S~2.3]{Kis09}). We have a natural morphism of formal schemes
\begin{equation*}
	\widehat{X_{\Omega, \textbf{h}}(\overline{\rho})}_x \lra \widehat{(\fX_{\overline{\rho}})}_{\rho} \cong X_{\rho}.
\end{equation*}

\begin{proposition}\label{closedemR}
	(1) The canonical morphism $\widehat{X_{\Omega, \textbf{h}}(\overline{\rho})}_x \lra X_{\rho}$ factors through a morphism
	\[\widehat{X_{\Omega, \textbf{h}}(\overline{\rho})}_x \lra X_{\rho, \cM_{\bullet}}.\]
	(2) The morphisms $\widehat{X_{\Omega, \textbf{h}}(\overline{\rho})}_x \ra X_{\rho}$ and $\widehat{X_{\Omega, \textbf{h}}(\overline{\rho})}_x \ra X_{\rho, \cM_{\bullet}}$ are closed immersions of groupoids over $\Art(E)$. 
\end{proposition}
\begin{proof}
	(1) follows from the same argument as in the proof of \cite[Prop.\ 3.7.2]{BHS3} with \cite[Cor.\ 6.3.10]{KPX} replaced by Corollary \ref{para} and Corollary \ref{rgloOF} (2). (2) follows from the same argument as in the proof of \cite[Prop.\ 3.7.3]{BHS3}.
\end{proof}

Consider the composition
\begin{equation*}
	\Theta_x: \widehat{X_{\Omega, \textbf{h}}(\overline{\rho})}_x \hookrightarrow X_{\rho,\cM_{\bullet}} \ra X_{V, \cM_{\bullet}} \cong X_{D, \cM_{\bullet}} \xrightarrow{\Theta} \widehat{\sT}_{P,L,(0,0)}.
\end{equation*}
Let $\textbf{h}'=(h'_{1,\tau}>h'_{2,\tau}>\cdots>h'_{n,\tau})_{\substack{i=1, \dots, n\\ \tau \in \Sigma_L}}$ be the Sen weights of $\rho$. Then by Proposition \ref{Senwt}, there exists $w_x=(w_{x,\tau})_{\tau\in \Sigma_L}\in \sW^{P}_{\min,L}$ such that, for $j=1, \dots, n$, $h'_{w_{x,\tau}^{-1}(j),\tau}=\wt(\chi_i)_{\tau}+h_{j,\tau}$ where $i$ is the integer such that $s_{i-1}<j\leq s_i$.

\begin{proposition}\label{propTHeta}
	The morphism $\Theta_x$ factors through $\widehat{\sT}_{w_xw_{0,L},(0,0)}$.
\end{proposition}
\begin{proof}
	The proposition follows by the same argument as for \cite[Lemma 3.7.4]{BHS3} using Proposition \ref{Senwt}.
\end{proof}

\begin{corollary}\label{locModdv}
Assume moreover that $\rho$ is de Rham. The closed immersion $\widehat{X_{\Omega, \textbf{h}}(\overline{\rho})}_x \hookrightarrow X_{\rho,\cM_{\bullet}}$ factors through an isomorphism 
	\begin{equation}\label{locMdOm}
		\widehat{X_{\Omega, \textbf{h}}(\overline{\rho})}_x \xlongrightarrow{\sim} X_{\rho, \cM_{\bullet}}^{w_xw_{0,L}}.
	\end{equation}
In particular, $X_{\Omega, \textbf{h}}(\overline{\rho})$ is unibranch, hence irreducible, at $x$.
\end{corollary}
\begin{proof}
	By Theorem \ref{DFOL} (1) and Theorem \ref{thmrM} (1), $\dim \widehat{X_{\Omega, \textbf{h}}(\overline{\rho})}_x=\dim X_{\rho,\cM_{\bullet}}$. We also know that $X_{\Omega, \textbf{h}}(\overline{\rho})$ is reduced and equidimensional by Proposition \ref{DFOL} (1). Hence $\widehat{X_{\Omega, \textbf{h}}(\overline{\rho})}_x$ is isomorphic to a union of irreducible components $\Spec R_{\rho,\cM_{\bullet}}^w$ of $\Spec R_{\rho,\cM_{\bullet}}$ (see Theorem \ref{thmrM} (1)). However, it follows from Proposition \ref{propTHeta} and Theorem \ref{thmrM} (2) that $\Spec R_{\rho,\cM_{\bullet}}^{w}$ can not be contained in $\widehat{X_{\Omega, \textbf{h}}(\overline{\rho})}_x$ if $\sW_{L_P,L}w\neq \sW_{L_P,L}(w_xw_{0,L})$. We then obtain the isomorphism in (\ref{locMdOm}). 
\end{proof}

Let $y\in X_{P,L}(k(x))$ be the point in (\ref{assPtGS}) associated to $(\rho, \cM_{\bullet})$. By Corollary \ref{locModdv} and Proposition \ref{exCon}, we have

\begin{corollary}\label{rescomp}
Assume moreover that $\rho$ is de Rham, then $w_x w_{0,L}\geq w_y$.
\end{corollary}

By corollary \ref{corosmo}, we deduce:

\begin{corollary}\label{smDefVar}
Assume that $\rho$ is de Rham, that $\pi(y)\in G_L/P_L \times G_L/B_L$ lies in the smooth locus of the closure of $G_L(1,w_xw_{0,L}) P_L \times B_L$, and that
	\begin{equation*}
		\dim \fz_{L_P,L}^{w_xw_{0,L} w_y^{-1}}+\lg(w_xw_{0,L})-\lg(w_y)=\dim \fz_{L_P,L}.
	\end{equation*}
	Then $X_{\Omega, \textbf{h}}(\overline{\rho})$ is smooth at the point $x=(\rho, \ul{x}, \chi)$.
\end{corollary}

By the discussion in Remark \ref{remsmo1} (1), we obtain the following special case that will be frequently used:

	\begin{corollary}\label{smDefVar2}
		Assume that $\rho$ is de Rham, that $\lg(w_y)\geq \lg(w_xw_{0,L})-2$ and that 
		\begin{equation*}
			\dim \fz_{L_P,L}-\dim \fz_{L_P,L}^{w_xw_{0,L} w_y^{-1}}=2 \text{ \ if \ $\lg(w_y)=\lg(w_xw_{0,L})-2$}.
		\end{equation*}
	 Then $X_{\Omega, \textbf{h}}(\overline{\rho})$ is smooth at the point $x=(\rho, \ul{x}, \chi)$.
		\end{corollary}
		
\begin{remark}\label{remnonsm}
When $P\neq B$, it could happen that $\dim \fz_{L_P,L}-\dim \fz_{L_P,L}^{w_xw_{0,L} w_y^{-1}}=1$ while $\lg(w_y)=\lg(w_xw_{0,L})-2$ (see Remark \ref{bruhInv} below, and compare with \cite[Rem.\ 4.1.6]{BHS3}). In this case, we don't know if $X_{\Omega, \textbf{h}}(\overline{\rho})$ is smooth at the point $x$. 
\end{remark}

As a quick application, we obtain the following full description of local companion points which completes Corollary \ref{compploc1}. Changing notation, let $z:=(\rho, (\ttr_1,\dots,\ttr_r))\in \widetilde{U}_{\overline{\rho}}^{\pcr}(\xi_0, \textbf{h})$ be as in \S~\ref{secPCD} (where it was denoted $x$ (see above Proposition \ref{VwSch})): $\rho$ is now a generic potentially crystalline representation with distinct Hodge-Tate weights $\textbf{h}$, $\xi_0=\oplus_{i=1}^r \xi_i$ with $\xi_i$ the inertial type of $\Omega_i$, and $(\ttr_1,\dots,\ttr_r)\in (\Spec \cZ_{\Omega})^{\rig}$ is an (ordered) $r$-tuple of absolutely irreducible Weil-Deligne representations such that $\ttr(\rho)\cong \oplus_{i=1}^r \ttr_i$. Recall we have attached to the point $z$ an element $w_z\in \sW_{\max,L}^P$ (see before Proposition \ref{VwSch} where $w_z$ was denoted $w_x$). In fact, if we let $y$ be the point in (\ref{assPtGS}) associated to $z$ (for $D=D_{\rig}(\rho)$ and $\cM_{\bullet}$ the $\Omega$-filtration on $\cM=D[\frac{1}{t}]$ associated to the filtration $\{\oplus_{i=1}^j \ttr_i\}_{j=1}^r$ on $\ttr(\rho)$ as in \S~\ref{introPcr}), then by definition $w_z=w_y$.

\begin{corollary}\label{coLoccomp}
	Let $w\in \sW_{\max,L}^P$, with the above notation the point
	\begin{equation*}
		\big(\rho, (\ttr_1,\dots,\ttr_r), 1\big)\in \fX_{\overline{\rho}} \times (\Spec \cZ_{\Omega})^{\rig} \times \widehat{\cZ_{0,L}}
	\end{equation*}
	lies in $X_{\Omega,ww_{0,L}(\textbf{h})}(\overline{\rho})$ if and only if $w \geq w_z$. 
\end{corollary}
\begin{proof}
	The ``if" part is Corollary \ref{compploc1}. The ``only if" part follows from Corollary \ref{rescomp}.
\end{proof}

\begin{remark}
	The case for $P=B$ (and $\rho$ crystalline) was proved in \cite[Thm.\ 4.2.3]{BHS3}, see Remark \ref{remNPara} for related discussions.
\end{remark}

\subsection{Galois cycles}\label{secGalCyc}

We construct certain cycles on the deformation space $X_{\rho}\cong \widehat{(\fX_{\overline{\rho}})}_{\rho}$ of a characteristic zero representation $\rho$ of $\Gal_L$.

Let $\overline{\rho}$ be as in \S~\ref{secGrDV} and $\rho\in \fX_{\overline{\rho}}(E)$. Recall that the local complete noetherian $E$-algebra $\widehat{\co}_{\fX_{\overline{\rho},\rho}}$ is equidimensional of dimension $n^2+[L:\Q_p] n^2$ and pro-represents the functor $|X_{\rho}|$ of (framed) deformations of $\rho$ over $\Art(E)$. Denote by $Z(\Spec \widehat{\co}_{\fX_{\overline{\rho}, \rho}})$ (resp.\ $Z^d(\Spec \widehat{\co}_{\fX_{\overline{\rho},\rho}})$ for $d\in \Z_{\geq 0}$) the free abelian group generated by the irreducible closed subschemes (resp.\ the irreducible closed subschemes of codimension $d$) in $\Spec \widehat{\co}_{\fX_{\overline{\rho},\rho}}$. If $A$ is a noetherian complete local ring which is a quotient of $\widehat{\co}_{\fX_{\overline{\rho},\rho}}$, we set\index{$Z^i(-)$}
\begin{equation*}
	[\Spec A]:=\sum_{\fp} m(\fp, A) [\Spec A/\fp] \in Z(\Spec \widehat{\co}_{\fX_{\overline{\rho},\rho}})
\end{equation*}
where the sum is over the minimal prime ideals $\fp$ of $A$, $m(\fp, A)\in \Z_{\geq 0}$ is the finite length of $A_{\fp}$ as a module over itself and $[\Spec A/\fp]$ is the irreducible component $\Spec A/\fp$ seen in $Z(\Spec \widehat{\co}_{\fX_{\overline{\rho}, \rho}})$.

Assume that $\rho$ has integral distinct $\tau$-Sen weights for each $\tau\in \Sigma_L$, and that $D_{\rig}(\rho)[1/t]$ admits an $\Omega$-filtration $\cM_{\bullet}$ of generic parameter $(\ul{x}, \delta)\in \Spec \cZ_{\Omega} \times \widehat{Z_{L_P}(L)}$ (so $\delta$ is locally algebraic). We refer to \S~\ref{secCCyc} (and the very beginning of \S~\ref{secMod}) for the schemes $\overline{X}_{P,L}$ and $\overline{X}_w$. Let $y \in \overline{X}_{P,L}(E) \subseteq X_{P,L}(E)$ be the point associated to $(\rho, \cM_{\bullet})$ in (\ref{assPtGS}) (depending on a choice of framing $\alpha$). Let $w\in \sW_L$ such that $y\in \overline{X}_w(E) \subset X_w(E)$ (which implies $w^{\min}w_{0,L} \geq w_y$ by Proposition \ref{exCon}). Similarly as in \cite[\S~4.3]{BHS3}, by Proposition \ref{wtmap1} and Theorem \ref{thmrM} (and its proof), we have a commutative diagram of groupoids over $\Art(E)$:
\begin{equation*}
	\text{\xymatrixcolsep{4pc}\xymatrix{
			X_{\rho,\cM_{\bullet}}^w \ar[d] & X_{\rho, \cM_{\bullet}}^{\square, w} \ar[l] \ar[r] \ar [d] & \widehat{X}_{w,y} \ar[d] \\
			X_{\rho, \cM_{\bullet}} \ar[d] \ar[rd]^{\omega_{\delta}} & X_{\rho, \cM_{\bullet}}^{\square} \ar[l] \ar[r] & \widehat{X}_{P,L, y} \ar[d]^{\kappa_P}\\
			X_{\rho} & \widehat{Z_{L_P}(L)}_{\delta} \ar[r]^{\wt-\wt(\delta)} & \widehat{\fz}_{L_P,L}}
	}
\end{equation*}
where we still denote by $\omega_{\delta}$ the composition $X_{\rho, \cM_{\bullet}} \ra X_{\cM, \cM_{\bullet}} \xrightarrow{\omega_{\delta}} \widehat{Z_{L_P}(L)}_{\delta}$. Taking everywhere (except for $X_\rho$) the fibres over $0\in \widehat{\fz}_{L_P,L}$, we obtain the commutative diagram of affine schemes over $E$:
\begin{equation*}
	\text{\xymatrix{
			\Spec \overline{R}_{\rho,\cM_{\bullet}}^w \ar[d] & \Spec \overline{R}_{\rho, \cM_{\bullet}}^{\square, w} \ar[l] \ar[d] \ar[r] & \Spec \widehat{\co}_{\overline{X}_w, y} \ar[d] \\
			\Spec \overline{R}_{\rho,\cM_{\bullet}} \ar[d] & \Spec \overline{R}_{\rho,\cM_{\bullet}}^{\square} \ar[l] \ar[r] & \Spec \widehat{\co}_{\overline{X}_{P,L},y} \\
			\Spec \widehat{\co}_{\fX_{\overline{\rho}},\rho}.
	}}
\end{equation*}
For $w=(w_{\tau})\in \sW_L$, we denote $Z_{w}:=\prod_{\tau\in \Sigma_L} Z_{w_{\tau}}\hookrightarrow Z_{P,L}$ (see \S~\ref{secCCyc} for $Z_{w_\tau}$). The irreducible components of $\Spec \widehat{\co}_{\overline{X}_{P,L}, y}$ are the union of the irreducible components of $\Spec \widehat{\co}_{Z_{w},y}$ for $w$ such that $y\in Z_{w}(E)$ (the last condition does not depend on the choice of the framing $\alpha$). By pull-back and smooth descent, we obtain a bijection between the irreducible components of $\Spec \widehat{\co}_{\overline{X}_{P,L}, y}$ and the irreducible components of $\Spec \overline{R}_{\rho,\cM_{\bullet}}$. In particular, $\Spec \overline{R}_{\rho,\cM_{\bullet}}$ is equidimensional of dimension $n^2+[L:\Q_p]\frac{n(n-1)}{2}$ (equivalently of codimension $[L:\Q_p]\frac{n(n+1)}{2}$ in $\Spec \widehat{\co}_{\fX_{\overline{\rho}, \rho}}$). For $w\in \sW_L$, denote by 
\begin{equation}\label{fZw}
	\fZ_{w}\in Z^{[L:\Q_p]\frac{n(n+1)}{2}}(\Spec \widehat{\co}_{\fX_{\overline{\rho}, \rho}})
\end{equation}
the cycle corresponding via the embedding $\Spec \overline{R}_{\rho, \cM_{\bullet}}\hookrightarrow \Spec \widehat{\co}_{\fX_{\overline{\rho}, \rho}}$ to the cycle $[\Spec \widehat{\co}_{Z_{w}, y}]$ under this bijection. By Theorem \ref{unibranch2}, we have:

\begin{lemma}\label{irrcyc}
	Assume that $\rho$ is de Rham. Then the cycle $\fZ_w$ is irreducible for $w\in \sW_L$.
\end{lemma}
	
\begin{remark}\label{remPocrycycl}
It is easy to see that $Z_{w_{0,L}}=G_L/B_L \times G_L/P_L \times \{0\}$. Indeed, we have a natural embedding $G_L/B_L \times G_L/P_L \times \{0\}\hookrightarrow Z_{w_{0,L}}$, which has to be an isomorphism since both are irreducible schemes of the same dimension. Thus if $\fZ_{w_0}\neq 0$, the nilpotent operator $N_W$ associated to $\rho$ is zero, hence $\rho$ is de Rham. As $\rho$ admits a generic $\Omega$-filtration $\cM_{\bullet}$, $\rho$ is potentially crystalline. Let $\textbf{h}$ be the Hodge-Tate weights of $\rho$ and $\xi_0$ the inertial type of $\rho$ (thus $\xi_0$ has the form $\oplus_{i=1}^r \xi_i$). By the same argument as in \cite[Rem.\ 4.3.1]{BHS3} (using \cite{Kis08}), it follows that $\fZ_{w_0}= [\widehat{\co}_{\fX_{\overline{\rho}}^{\pcr}(\xi_0, \textbf{h}),\rho}]$.
\end{remark}

For $w=(w_{\tau})$ and $w'=(w'_{\tau})\in \sW_L$, put $a_{w,w'}:=\prod_{\tau\in \Sigma_L} a_{w_{\tau}, w'_{\tau}}$ and $b_{w,w'}:=\prod_{\tau\in \Sigma_L} b_{w_{\tau}, w'_{\tau}}$ where $a_{w_{\tau}, w'_{\tau}}$, $b_{w_{\tau}, w'_{\tau}}$ are given as in Theorem \ref{thmcycl} (applied to $G=\GL_n$). Put
\begin{equation}\label{cycGaldef}
	\fC_{w}:=\sum_{w'\in \sW_{L_P,L}\backslash \sW_L} a_{w,w'} \fZ_{w'} \in Z^{[L:\Q_p]\frac{n(n+1)}{2}}(\Spec \widehat{\co}_{\fX_{\overline{\rho},\rho}}).
\end{equation}
Note that $\fZ_w$ and $\fC_w$ are independent of the choice of the representative $w$ in its associated class in $\sW_{L_P,L}\backslash \sW_L$.
\begin{lemma}\label{lemnonzdR}
	Assume that $\rho$ is de Rham, and let $w\in \sW_L$. The followings are equivalent: (1) $\fC_{w} \neq 0$,
	 (2) $\fZ_{w} \neq 0$,
		(3) $w^{\max} \geq w_y$. 
\end{lemma}
\begin{proof}
	(2) $\Rightarrow$ (1) is clear. 
	
	(1) $\Rightarrow$ (3): 
	If $\fC_{w}\neq 0$, then $\fZ_{w'}\neq 0$ ($\Rightarrow y\in X_{w'}$) for some $w'$ with $w'^{\max}\leq w^{\max}$. By Proposition \ref{exCon}, $w'^{\max} \geq w_y$ hence $w^{\max} \geq w_y$. 
	
	(3) $\Rightarrow$ (2): As $\rho$ is de Rham (i.e.\ the entry $N_W$ in $y$ is zero), if $w^{\max} \geq w_y$, then $y$ is contained in the Zariski-closure of $\big(G_L(w^{\max},1) B_L \times P_L\big) \times \{0\}$ in $G_L/B_L \times G_L/P_L \times \{0\} \hookrightarrow Z_{P,L}$, thus $y\in Z_{w}\Rightarrow \fZ_{w}\neq 0$. 
\end{proof}

Now let $x=(\rho, \ul{x},\chi)\in X_{\Omega,\textbf{h}}(\overline{\rho})(E)$ such that $\rho$ is de Rham and $(\ul{x}, \chi)$ is generic. Let $\cM_{\bullet}$ be the unique $\Omega$-filtration on $D_{\rig}(\rho)[1/t]$ of parameter $(\ul{x}, \chi_{\varpi_L})$ (cf.\ Corollary \ref{uniqMfil}) and $y$ be the point (\ref{assPtGS}) of $X_{P,L}$ associated to $(\rho, \cM_{\bullet})$. Recall that we have defined elements $w_y\in \sW^P_{\max,L}$ and $w_x\in \sW^P_{\min,L}$ (see above Proposition \ref{exCon} for $w_y$ and above Proposition \ref{propTHeta} for $w_x$). Let $X_{\Omega,\textbf{h}}(\overline{\rho})_{\wt(\chi)}$ denote the fibre of $X_{\Omega,\textbf{h}}(\overline{\rho})$ at $\wt(\chi)$ (via the morphism $X_{\Omega,\textbf{h}}(\overline{\rho}) \ra \widehat{\cZ_{0,L}} \xrightarrow{\wt} \fz_{L_P,L}$). The following conjecture is a consequence of Conjecture \ref{conjCC} and Lemma \ref{lemnonzdR}:

\begin{conjecture}\label{conjcyc1}
We have 
	\begin{equation*}
		[\Spec \widehat{\co}_{X_{\Omega,\textbf{h}}(\overline{\rho})_{\wt(\chi)},x}]=\sum_{\substack{w\in \sW_{L_P,L}\backslash \sW_L\\ w_y \leq w^{\max} \leq w_xw_{0,L}}} b_{w_xw_{0,L},w} \fC_{w} \in Z^{[L:\Q_p] \frac{n(n+1)}{2}}(\Spec \widehat{\co}_{\fX_{\overline{\rho}, \rho}}).
	\end{equation*} 
\end{conjecture} 

\begin{remark}
It follows from Lemma \ref{lemcyccomp} that Conjecture \ref{conjcyc1} holds if $x$ is moreover a smooth point of $X_{\Omega,\textbf{h}}(\overline{\rho})_{\wt(\chi)}$.
\end{remark}

\section{Applications}

Under the Taylor-Wiles hypothesis, we show several (global) results on $p$-adic automorphic representations including a classicality result, and the existence of all expected companion constituents for certain parabolic subgroup $P$.

\subsection{Automorphy cycles}

We use locally analytic representation theory to construct certain cycles on patched Bernstein eigenvarieties.

\subsubsection{Representation theoretic preliminaries}\label{sec711}

We give some preliminaries on locally analytic representations which we will use in \S~\ref{cyclebern}. We use the notation of \S~\ref{abCon}. 

Denote by $\co^{\fp}_{\alg} \subset \co^{\fp}$ the full subcategory of objects with integral weights (see \cite{OS}). Let $V$ be an admissible locally analytic representation of $G_p$. Let $M\in \co^{\fp}_{\alg}$, the $\fp$-action on $M$ canonically extends to a $P(\Q_p)$-action (\cite[Lemma 3.2]{OS}). We equip $\Hom_{\text{U}(\ug)}(M, V)$ with the left action of $P(\Q_p)$ given by $(pf)(v):=p f(p^{-1}v)$ for $p\in P(\Q_p)$, $v\in V$. Note that this action does preserve $\Hom_{\text{U}(\ug)}(M, V)$ because we have for $p\in P(\Q_p)$, $X\in \text{U}(\ug)$ and $v\in V$:
\begin{multline*}
(pf)(Xv)=p f(p^{-1} X v)=pf\big(\Ad(p^{-1})(X) p^{-1} v\big)=p \Ad(p^{-1})(X) f(p^{-1} v)=X p f(p^{-1}v)\\
=X(pf)(v).
\end{multline*}
We also see that the derived $\fp$-action is trivial, so the $P(\Q_p)$-action on $\Hom_{\text{U}(\ug)}\big(M, V\big)$ is smooth. 
We equip $\Hom_{\text{U}(\ug)}\big(M, \Pi\big)^{N_P^0}$ with a natural Hecke action of $ L_P(\Q_p)^+$ defined as in (\ref{Upheck}).

For any $M\in \co^{\ub}$ (in particular $M\in \co^{\fp}_{\alg}$), we endow $\Hom_{\text{U}(\ug)}(M,V)$ with a canonical topology of space of compact type as follows. Choose finitely many weights $\lambda^i$ of the Lie algebra of $T_{\sI}$ such that $\oplus_i M_{B_{\sI}}(\lambda^i) \twoheadrightarrow M$ where $M_{B_{\sI}}(\lambda^i):=\text{U}(\ug)\otimes_{\text{U}(\ub_{\sI})}\lambda^i$. Then $\Hom_{\text{U}(\ug)}(M,V)$ is naturally a closed subspace of $\oplus_i V[\ub_{\sI}=\lambda^i]$ that we endow with the induced topology. Using that a continuous bijection of vector spaces of compact type is a topological isomorphism, one easily checks that this doesn't depend on the choice of the $\lambda^i$ (for two choices $\lambda^i$ and $\mu^j$, consider $(\oplus_i M_{B_{\sI}}(\lambda^i)) \oplus (\oplus_j M_{B_{\sI}}(\mu^j)) \twoheadrightarrow M$) and that for a morphism $M \ra N$ in $\co^{\ub}$, the induced morphism $\Hom_{\text{U}(\ug)}(N,V) \ra \Hom_{\text{U}(\ug)}(M,V)$ is continuous. We denote by $\Hom_{\text{U}(\ug)}(M,V)^{N_P^0}_{\fss}$ the finite slope part of $\Hom_{\text{U}(\ug)}(M,V)^{N_P^0}$ defined as in \cite[\S~3.2]{Em11}. 
The following lemma follows by the same argument as in the proof of \cite[Lemma 5.2.1]{BHS3} (see \cite[Prop.\ 3.2.4 (ii)]{Em11} for the second isomorphism). 

\begin{lemma}\label{lemadj01}
Assume that $V$ is very strongly admissible (\cite[Def.\ 0.12]{Em2}), let $\pi$ be a finite length smooth representation of $L_P(\Q_p)$ over $E$ and $M\in \co_{\alg}^{\fp}$. There are natural bijections
	\begin{eqnarray*}
		\Hom_{G_p}\Big(\cF_{P^-(\Q_p)}^{G_p}\big(\Hom_E(M,E)^{\fn_{P^-}^{\infty}}, \pi(\delta_P^{-1})\big), V\Big)
		&\xlongrightarrow{\sim} & \Hom_{ L_P(\Q_p)^+}\big(\pi, \Hom_{\text{U}(\ug)}(M, V)^{N_P^0}\big)\\
		&\cong& \Hom_{L_P(\Q_p)}\big(\pi, \Hom_{\text{U}(\ug)}(M, V)^{N_P^0}_{\fss}\big)
	\end{eqnarray*}
	where $\Hom_E(M,E)^{\fn_{P^-}^{\infty}}\!\subset \Hom_E(M,E)$ is the object in $\co^{\fp^-}_{\alg}$ defined in \cite[\S~3]{Br13II} and $\delta_P$ is the modulus character of $P(\Q_p)$.
\end{lemma} 

Fix $\fd$ an integral weight of $\fz_{L_P}$ and $\sigma$ a cuspidal type of $L_P(\Q_p)$ as in \S~\ref{sec3.1.1}. Consider 
\begin{eqnarray*}
	V(M, \fd, \sigma)&:=&\Big(\big(\Hom_{\text{U}(\ug)}(M, V)^{N_P^0}_{\fss} \otimes_E(\delta_{\fd}\circ \dett_{L_P})\big) \widehat{\otimes}_E \cC^{\Q_p-\la}(Z_{L_P}^0,E) \otimes_E \sigma^{\vee}\Big)^{L_P^0}\\
	&\cong &\Hom_{L_P^0}\Big(\sigma, \big(\Hom_{\text{U}(\ug)}(M, V)^{N_P^0}_{\fss} \otimes_E(\delta_{\fd}\circ \dett_{L_P})\big) \widehat{\otimes}_E \cC^{\Q_p-\la}(Z_{L_P}^0,E)\Big)
\end{eqnarray*}
where $L_P(\Q_p)$ acts on $\big(\Hom_{\text{U}(\ug)}(M, V)^{N_P^0}_{\fss} \otimes_E(\delta_{\fd}\circ \dett_{L_P} )\big)\widehat{\otimes}_E \cC^{\Q_p-\la}(Z_{L_P}^0,E)$ via the diagonal action with $L_P(\Q_p)$ acting on $\cC^{\Q_p-\la}(Z_{L_P}^0,E)$ via (\ref{dett1}). As $V(M, \fd, \sigma)$ is a closed subspace of $\big(\Hom_{\text{U}(\ug)}(M, V)^{N_P^0}_{\fss} \otimes_E(\delta_{\fd}\circ \dett_{L_P} )\big)\widehat{\otimes}_E \cC^{\Q_p-\la}(Z_{L_P}^0,E) \otimes_E \sigma^{\vee}$ (where $\sigma^{\vee}$ is equipped with the finest locally convex topology), it is also a space of compact type. We equip $V(M, \fd, \sigma)$ with an action of $\cZ_0 \times \Delta_0 \times \cZ_{\Omega}$ similarly as for $B_{\sigma, \lambda}(V)$ in (\ref{BslV}): $\cZ_0\cong Z_{L_P}^0$ acts via the regular action on $\cC^{\Q_p-\la}(Z_{L_P}^0,E)$, $\Delta_0\cong Z_{L_P}(\Q_p)$ acts via the diagonal action on $\Hom_{\text{U}(\ug)}(M, V)^{N_P^0}_{\fss} \otimes_E(\delta_{\fd}\circ \dett_{L_P}) \widehat{\otimes}_E \cC^{\Q_p-\la}(Z_{L_P}^0,E) $, and the $\cZ_{\Omega}$-action comes from the isomorphism (recalling that $\cZ_{\Omega}\cong \End_{L_P(\Q_p)}(\cind_{L_P^0}^{L_P(\Q_p)} \sigma)$):
\begin{equation*}
	V(M, \fd, \sigma)\cong \Hom_{L_P(\Q_p)}\!\Big(\!\cind_{L_P^0}^{L_P(\Q_p)} \sigma, \big(\Hom_{\text{U}(\ug)}(M, V)^{N_P^0}_{\fss} \otimes_E(\delta_{\fd}\circ \dett_{L_P})\big) \widehat{\otimes}_E \cC^{\Q_p-\la}(Z_{L_P}^0,E)\Big).
\end{equation*}
Similarly as in the discussion above Lemma \ref{BCmod}, the $\Delta_0$-action is equal to the action induced from $\cZ_{\Omega}$ via $Z_{L_P}(\Q_p)\ra \cZ_{\Omega}$. 
By considering the action of $\fl_P$ on $V(M, \fd, \sigma)$, we have also
\begin{eqnarray*}
	V(M, \fd, \sigma)&\cong \!&\Big(\big(\Hom_{\text{U}(\ug)}(M, V)^{N_P^0}_{\fss} \otimes_E(\delta_{\fd}\circ \dett_{L_P})\big) \widehat{\otimes}_E \big(\cC^{\infty}(Z_{L_P}^0,E) \otimes_E \delta_{\fd}^0\big) \otimes_E \sigma^{\vee}\Big)^{L_P^0}\nonumber\\
	&\cong \!&\Big(\Hom_{\text{U}(\ug)}(M, V)^{N_P^0}_{\fss} \otimes_E(\delta_{\fd}\circ \dett_{L_P}) \otimes_E \big(\cC^{\infty}(Z_{L_P}^0,E) \otimes_E \delta_{\fd}^0\big) \otimes_E \sigma^{\vee}\Big)^{L_P^0}
\end{eqnarray*}
where $\cC^{\infty}(Z_{L_P}^0,E)$ denotes the space of smooth $E$-valued functions on $Z_{L_P}^0$, where $\cC^{\infty}(Z_{L_P}^0,E) \otimes_E \delta_{\fd}^0$ embeds into $\cC^{\Q_p-\la}(Z_{L_P}^0,E)$ by $f \otimes 1 \mapsto [z \mapsto f(z) \delta_{\fd}^0(z)]$, and where the second isomorphism follows from \cite[Prop.\ 1.2]{Koh2011} and the fact that $\cC^{\infty}(Z_{L_P}^0,E)$ is topologically isomorphic to a direct limit of finite dimensional $E$-vector spaces (each equipped with the finest locally convex topology).
 
\begin{lemma}\label{VBiso}
Let $\lambda$ be an integral $P$-dominant weight of $T_{\sI}$, we have an isomorphism of vector spaces of compact type (cf.\ Notation \ref{charanot} and the discussion after Lemma \ref{diftyp})
	 \begin{equation*}
		\iota: V(M_P(\lambda),\fd, \sigma) \xlongrightarrow{\sim} B_{\Omega, \lambda}(V)[\fz_0=0]
	\end{equation*}
such that for $(\alpha, \beta)\in \cZ_{\Omega} \times \cZ_0$ and $v\in V(M_P(\lambda),\fd, \sigma)$:
\[\iota\big((\alpha, \beta)\cdot v\big)=\delta_{\fd}^0(\beta)\Big(\big(\iota_{(\delta_{\fd, \ul{\varpi}}^{\unr})^{-1}}(\alpha), \beta\big)\cdot \iota(v)\Big).\]
\end{lemma}
\begin{proof}
We have isomorphisms
	\begin{equation*}
		\Hom_{\text{U}(\ug)}(M_P(\lambda),V)\cong \Hom_{\fl_P}(L(\lambda)_P, V^{\fn_P})\cong (V^{\fn_P} \otimes_E L(\lambda)_P^{\vee})^{\fl_P}
	\end{equation*}
	which are moreover topological isomorphisms if $\Hom_{\fl_P}(L(\lambda)_P, V^{\fn_P})\cong (V^{\fn_P} \otimes_E L(\lambda)_P^{\vee})^{\fl_P}$ is equipped with the induced topology as closed subspace of $V \otimes_E L(\lambda)_P^{\vee}$. We deduce then an $ L_P(\Q_p)^+$-equivariant isomorphism (where $ L_P(\Q_p)^+$ acts on the right hand side by $(zf)(v)=zf(z^{-1}v)$): 
	\begin{equation}\label{missingagain}
		\Hom_{\text{U}(\ug)}(M_P(\lambda),V)^{N_P^0}\cong \Hom_{\fl_P}(L(\lambda)_P, V^{N_P^0}).
	\end{equation}
	This isomorphism induces $L_P(\Q_p)$-equivariant isomorphisms by the universal property of the finite slope part functor \cite[Prop.\ 3.2.4 (2)]{Em11} (for the first) and Lemma \ref{lemSLnalg} (for the second):
	\begin{equation}\label{JPMV}
		\Hom_{\text{U}(\ug)}(M_P(\lambda),V)^{N_P^0}_{\fss}\cong \Hom_{\fl_P}(L(\lambda)_P, V^{N_P^0}_{\fss})\cong J_P(V)_{\lambda}[\fz_{L_P}=0].
	\end{equation}
From (\ref{JPMV}) we deduce topological isomorphisms
	\begin{multline*}
	V(M_P(\lambda), \fd, \sigma)\\
	 \!\!\!\!\!\cong \Big(\big(J_P(V)_{\lambda}[\fz_{L_P}=0] \otimes_E((\delta_{\fd, \ul{\varpi}}^{\unr}\delta_{\fd,\ul{\varpi}}^0)\circ \dett_{L_P})\big) \widehat{\otimes}_E \cC^{\Q_p-\la}(Z_{L_P}^0,E) \otimes_E \sigma^{\vee}\Big)^{L_P^0}\\ 
		\!\!\!\!\!\!\!\!\!\!\!\!\!\!\!\!\!\!\!\!\cong \Big(\big(J_P(V)_{\lambda}\otimes_E(\delta_{\fd, \ul{\varpi}}^{\unr}\circ \dett_{L_P})\big) \widehat{\otimes}_E \cC^{\Q_p-\la}(Z_{L_P}^0,E)[\fz_{L_P}=0]\otimes_E \sigma^{\vee}\Big)^{L_P^0} \\
		\ \ \ \ \ \ \ \ \ \ \cong \Hom_{L_P(\Q_p)}\!\Big(\!\cind_{L_P^0}^{L_P(\Q_p)} \!\sigma, \big(J_P(V)_{\lambda}\otimes_E(\delta_{\fd, \ul{\varpi}}^{\unr}\circ \dett_{L_P})\big) \widehat{\otimes}_E \cC^{\Q_p-\la}(Z_{L_P}^0,E)[\fz_{L_P}=0]\Big)\\
		\ \ \ \ \ \ \ \ \ \ \ \ \cong \Hom_{L_P(\Q_p)}\!\Big(\!((\delta_{\fd, \ul{\varpi}}^{\unr})^{-1}\circ \dett_{L_P}) \otimes_E\cind_{L_P^0}^{L_P(\Q_p)}\! \sigma, J_P(V)_{\lambda}\widehat{\otimes}_E \cC^{\Q_p-\la}(Z_{L_P}^0,E)[\fz_{L_P}=0]\Big)\\ 
		\cong \Hom_{L_P(\Q_p)}\Big(\cind_{L_P^0}^{L_P(\Q_p)} \sigma, J_P(V)_{\lambda}\widehat{\otimes}_E \cC^{\Q_p-\la}(Z_{L_P}^0,E)[\fz_{L_P}=0]\Big) \cong B_{\Omega, \lambda}(V)[\fz_0=0]
	\end{multline*}
where the second isomorphism uses $\cC^{\Q_p-\la}(Z_{L_P}^0,E) \otimes_E \delta_{\fd,\ul{\varpi}}^0\cong \cC^{\Q_p-\la}(Z_{L_P}^0,E)$ and the fifth uses (\ref{twistun}). The last part of the statement on the $\cZ_{\Omega}\times \cZ_0$-action follows by the same argument as in the proof of Lemma \ref{twBE}.
\end{proof}

\begin{lemma}\label{decompVM}
For $M\in \co^{\fp}_{\alg}$, we have 
	\begin{equation*}
		V(M, \fd, \sigma) 
		\cong \bigoplus_{\delta, \chi}V(M, \fd, \sigma)[\fm_{\chi}][\fm_{\delta}^{\infty}]=\bigoplus_{\fm \in \Spm \cZ_{\Omega},\chi}V(M, \fd, \sigma)[\fm_{\chi}][\fm^{\infty}]
	\end{equation*}	
where $\delta$ (resp.\ $\chi$) runs through the smooth characters of $\Delta_0\cong Z_{L_P}(\Q_p)$ (resp.\ through the locally algebraic characters of $\cZ_0 \cong Z_{L_P}^0$ of weight $\fd$), and $\fm_{\delta}\subset E[\Delta_0]$ (resp.\ $\fm_{\chi}\subset E[\cZ_0]$) is the maximal ideal associated to the character $\delta$ (resp.\ $\chi$). Moreover, each term in the direct sums is finite dimensional over $E$. 
\end{lemma}
\begin{proof}
As \ $M\in \co_{\alg}^{\fp}$, \ there \ exist \ finitely \ many \ $P$-dominant \ integral \ weights \ $\lambda_i$ \ such \ that $\oplus_{i} M_P(\lambda_i) \twoheadrightarrow M$. Using successively the following facts:
\begin{itemize}
\item[(1)]$\Hom_{\text{U}(\ug)}(-,V)$ is left exact;
\item[(2)]taking $N_P^0$-invariant vectors is exact on smooth representations of $P(\Q_p)$;
\item[(3)]taking $(-)_{\fss}$ preserves injectivity (cf.\ \cite[Prop.\ 3.2.6 (ii)]{Em11});
\item[(4)]$(-)^{L_P^0}$ is left exact,
\end{itemize}
we deduce an injection:
\begin{equation*}
V(M, \fd, \sigma) \hooklongrightarrow \oplus_i V(M_P(\lambda_i), \fd, \sigma).
\end{equation*}
It is sufficient to prove the statement for each $V(M_P(\lambda_i), \fd, \sigma)$. Indeed, if this holds, then any vector $v\in V(M, \fd, \sigma)$ generates, under the action of $\cZ_{\Omega}\times \cZ_0$, a finite dimensional $E$-vector space, and from this we easily deduce the decompositions in the lemma for $V(M, \fd, \sigma)$. Moreover for each $\fm\in \Spec \cZ_{\Omega}$ and each $\chi$, the vector space $V(M, \fd, \sigma)[\fm_{\chi}][\fm^{\infty}]$ is finite dimensional as it is a subspace of $\oplus_i V(M_P(\lambda_i), \fd, \sigma)[\fm_{\chi}][\fm^{\infty}]$. However, the statement in the lemma for $V(M_P(\lambda_i), \fd, \sigma)$ follows from Lemma \ref{VBiso} and Lemma \ref{compfini}. 
\end{proof}

\begin{lemma}\label{exactnessBE0}
Let $\fm\subset \cZ_{\Omega}$ be a maximal ideal and $\chi$ a locally algebraic character of $\cZ_0$ of weight $\fd$. Let $V$ be a continuous representation of $G_p$ on a Banach space over $E$ such that $V^{\vee}$ is a finitely generated projective $\co_E[[K_p]][1/p]$-module where $K_p:=G(\Z_p)$. Denote by $V^{\an}\subseteq V$ the subspace of locally analytic vectors (\cite[Def.\ 3.5.3]{Em04}), which is an admissible locally analytic representation of $G_p$ on a $E$-vector space of compact type (\cite[Prop.\ 6.2.4]{Em04}). Then the functor $V^{\an}(-,\fd, \sigma)[\fm_{\chi}][\fm^{\infty}]$ is exact on $\co_{\alg}^{\fp}$.
\end{lemma}
\begin{proof}
Let $\pi_{\fm}$ be the smooth $L_P(\Q_p)$-representation associated to $\fm$ and $\omega$ its central character. Let $\fm_{\omega}\subset E[\Delta_0]$ be the maximal ideal associated to $\omega$. Then $V^{\an}(M, \fd, \sigma)[\fm_{\chi}][\fm^{\infty}]=V^{\an}(M, \fd, \sigma)[\fm^{\infty}][\fm_{\chi}]$ is a direct summand of $V^{\an}(M, \fd, \sigma)[\fm_{\omega}^{\infty}][\fm_{\chi}]$, and it is enough to prove the statement with $[\fm^{\infty}]$ replaced by $[\fm_{\omega}^{\infty}]$. Let
\[\delta:= \big(\big((\chi(\delta_{\fd}^0)^{-1})_{\ul{\varpi}} (\delta_{\fd,\ul{\varpi}}^{\unr})^{-1}\big) \circ \dett_{L_P}\big)\omega=\big(\big(\chi \delta_{\fd}^{-1}\big) \circ \dett_{L_P}\big) \omega\]
which is a smooth character of $Z_{L_P}(\Q_p)$. Let $\fm_{\delta} \subset E[Z_{L_P}(\Q_p)]$ (resp.\ $\fm_{\delta,+}\subset E[Z_{L_P}(\Q_p)^+]$) be the maximal ideal associated to $\delta$. For $M\in \co_{\alg}^{\fp}$, we have (by unwinding the $\Delta_0$-action, and noting that $L_P(\Q_p)$ acts on $\chi$ via $\chi_{\varpi}^{-1} \circ \dett_{L_P}$, cf.\ (\ref{dett1}))
	\begin{eqnarray}\label{ismVM1}
		\nonumber V^{\an}(M, \fd, \sigma)[\fm_{\chi}][\fm_{\omega}^{\infty}]& \cong \!\!&\big(\Hom_{\text{U}(\ug)}(M, V^{\an})^{N_P^0}_{\fss} [\fm_{\delta}^{\infty}]\otimes_E(\delta_{\fd}\circ \dett_{L_P}) \otimes_E \chi \otimes_E \sigma^{\vee}\big)^{L_P^0}\\
		&\cong \!\!&\big(\Hom_{\text{U}(\ug)}(M, V^{\an})^{N_P^0} [\fm_{\delta,+}^{\infty}]\otimes_E(\delta_{\fd}\circ \dett_{L_P}) \otimes_E \chi \otimes_E \sigma^{\vee}\big)^{L_P^0}
	\end{eqnarray}
	where the second isomorphism follows from \cite[Prop.\ 3.2.11]{Em11} and \cite[Lemma 3.2.8]{Em11}. In particular, the right hand side of (\ref{ismVM1}) is finite dimensional by Lemma \ref{decompVM}. By \cite[Lemma 5.2.5]{BHS3}, the functor $\Hom_{\text{U}(\ug)}(-,V^{\an})$ is exact, hence so is the functor $\Hom_{\text{U}(\ug)}(-,V^{\an})^{N_P^0}$ as $\pi\mapsto \pi^{N_P^0}$ is exact on smooth representations $\pi$ of $P(\Q_p)$. We deduce that, for an exact sequence $0 \ra M_1 \ra M_2 \ra M_3 \ra 0$ in $\co^{\fp}_{\alg}$, the following sequence is exact:
	\begin{multline}\label{exactlP0}
		0 \ra \big(\Hom_{\text{U}(\ug)}(M_3, V^{\an})^{N_P^0} \otimes_E(\delta_{\fd}\circ \dett_{L_P}) \otimes_E \chi\otimes_E \sigma^{\vee}\big)^{L_P^0} \\
		\ra \big(\Hom_{\text{U}(\ug)}(M_2, V^{\an})^{N_P^0} \otimes_E(\delta_{\fd}\circ \dett_{L_P})\otimes_E \chi\otimes_E \sigma^{\vee}\big)^{L_P^0} \\
		\ra \big(\Hom_{\text{U}(\ug)}(M_1, V^{\an})^{N_P^0} \otimes_E(\delta_{\fd}\circ \dett_{L_P}) \otimes_E \chi \otimes_E \sigma^{\vee}\big)^{L_P^0} \ra 0.
	\end{multline}
By (\ref{ismVM1}) it is enough to show that the sequence (\ref{exactlP0}) stays exact after applying $[\fm_{\delta,+}^{\infty}]$. 
	
We now explain how to prove this by a generalization of the arguments in the last paragraph of the proof of \cite[Prop.\ 4.1]{BH2}. We let $H$ be as in \textit{loc.\ cit.}\ and replace $N_P^0$ by $H \cap N_P(\Q_p)$ (which won't cause any problem). Using \cite[Prop.\ 3.3.2]{Em11}, we choose finitely many $z_i\in Z_{L_P}(\Q_p)^+$ such that $z_i N_{P}^0 z_i^{-1} \subset (N_{P}^0)^p$ for $N_{P}^0=N_{P}(\Q_p) \cap H$, and such that $Z_{L_P}^0$ and the $z_i$ generate $Z_{L_P}(\Q_p)$ as group. 
	For $M=M_P(\lambda)\in \co_{\alg}^{\fp}$, we have by (\ref{missingagain}) a topological isomorphism
	\begin{multline}\label{homGM}
	\big(\Hom_{\text{U}(\ug)}(M, V^{\an})^{N_P^0} \otimes_E(\delta_{\fd}\circ \dett_{L_P})\otimes_E \chi\otimes_E \sigma^{\vee}\big)^{L_P^0} \\
	\cong \big((V^{\an})^{N_P^0}\otimes_E L(\lambda)_P^{\vee}\otimes_E(\delta_{\fd}\circ \dett_{L_P})\otimes_E \chi\otimes_E \sigma^{\vee}\big)^{L_P^0}.
	\end{multline} 
	By an easy generalization of \cite[Lemma 5.3]{BHS2} (see also the proof of Proposition \ref{Pr}), we can write the right hand side (hence the left hand side) of (\ref{homGM}) as an increasing union over $j$ of BH-subspaces $V_j$ such that each $z_i$ preserves $V_j$ and acts on (the underlying Banach space) $V_j$ via a compact operator. The same holds for a general $M\in \co_{\alg}^{\fp}$ as $(\Hom_{\text{U}(\ug)}(M, V^{\an})^{N_P^0} \otimes_E(\delta_{\fd}\circ \dett_{L_P})\otimes_E \chi\otimes_E \sigma^{\vee})^{L_P^0}$ is a closed subspace of a finite direct sum $\bigoplus_i (\Hom_{\text{U}(\ug)}(M_P(\lambda_i), V^{\an})^{N_P^0} \otimes_E(\delta_{\fd}\circ \dett_{L_P})\otimes_E \chi\otimes_E \sigma^{\vee})^{L_P^0}$ for some integral $P$-dominant weights $\lambda_i$. We can now apply the argument in the last paragraph of the proof of \cite[Prop.\ 4.1]{BH2} with the BH-subspaces $\Pi_h$ of \textit{loc.\ cit.} replaced by the BH-subspaces $V_j$ (for general $M$) discussed above to conclude. 
\end{proof}

\subsubsection{Cycles on patched Bernstein eigenvarieties}\label{cyclebern}

We construct certain cycles in the completed local rings (at some specific points) of patched Bernstein eigenvarieties. 

We assume that we are in the setting of \S~\ref{secPBern} and we use without comment the notation and constructions of {\it loc.\ cit}. We fix $\lambda$ an integral $P$-dominant weight, $\Omega$ a cuspidal Bernstein component of $L_P(\Q_p)$ and $\overline\rho$ a $U^p$-modular continuous representation of $\Gal_F$ over $k_E$ (where $U^p$ is a ``prime-to-$p$ level''). We have the associated patched Bernstein eigenvariety $\cE_{\Omega, \lambda}^{\infty}(\overline{\rho})$, see (\ref{patchedbernstein}).

Let $\fd$ be an integral weight of $\fz_{L_P}$ and put $\mu:=\lambda+\fd\circ \dett_{L_P}$. Let $\cE_{\Omega, \lambda}^{\infty}(\overline{\rho})_{\fd}$ be the fibre of $\cE_{\Omega, \lambda}^{\infty}(\overline{\rho})$ at $\fd$ via $\cE_{\Omega, \lambda}^{\infty}(\overline{\rho}) \ra \widehat{\cZ_0} \ra \fz_0^{\vee}\cong \fz_{L_P}^{\vee}$. Note that using \cite[Lemma 6.2.5]{Che} and \cite[Lemma. 6.2.10]{Che}, we can deduce from Corollary \ref{PBEdim} (1), (2) and Proposition \ref{PBEFre} that $\cE_{\Omega, \lambda}^{\infty}(\overline{\rho})_{\fd}$ is equidimensional of dimension 
\[g+|S|n^2 +\sum_{v\in S_p}\Big([F_{\widetilde{v}}:\Q_p] \frac{n(n-1)}{2}\Big).\]
Let $\cM_{\mu, \lambda}:=\cM_{\Omega, \lambda}^{\infty} \otimes_{\co_{\cE_{\Omega, \lambda}^{\infty}(\overline{\rho})}} \co_{\cE_{\Omega, \lambda}^{\infty}(\overline{\rho})_{\fd}}$. Using \ (\ref{glosecBE}) \ the \ vector \ space \ of \ compact \ type $\Gamma\big( \cE_{\Omega, \lambda}^{\infty}(\overline{\rho})_{\fd}, \cM_{\mu, \lambda}\big)^{\vee}$ is topologically isomorphic to the following vector spaces of compact type (which are of compact type by an obvious generalization to the patched case of the discussion above Lemma \ref{lemadj01}):
\begin{multline*}
\big(J_P(\Pi^{R_{\infty}-\an}_{\infty})_{\lambda}\widehat{\otimes}_E \cC^{\Q_p-\la}(Z_{L_P}^0,E)\otimes_E \sigma^{\vee}\big)^{L_P^0}[\fz_0=\fd] \\
\!\!\!\!\!\!\!\!\!\!\!\!\!\!\!\!\!\!\!\!\!\!\!\!\!\!\!\!\!\!\!\!\!\!\!\!\!\!\!\!\!\!\!\!\!\!\!\!\!\!\cong \big(J_P(\Pi^{R_{\infty}-\an}_{\infty})_{\lambda}\widehat{\otimes}_E \cC^{\Q_p-\la}(Z_{L_P}^0,E)[\fz_{L_P}=\fd]\otimes_E \sigma^{\vee}\big)^{L_P^0} \\
\!\!\!\!\!\!\cong \!\big(J_P(\Pi^{R_{\infty}-\an}_{\infty})_{\lambda}[\fz_{L_P}=\fd\circ \dett_{L_P}]\widehat{\otimes}_E \cC^{\Q_p-\la}(Z_{L_P}^0,E)[\fz_{L_P}=\fd]\otimes_E \sigma^{\vee}\big)^{L_P^0} \\
\ \ \ \ \ \ \ \ \ \ \ \ \ \ \ \ \cong \!\Big(\!\Hom_{\text{U}(\ug)}(M_P(\mu), \Pi^{R_{\infty}-\an}_{\infty})^{N_P^0}_{\fss}\!\otimes_E (\delta_{\fd} \circ \dett_{L_P}) \widehat{\otimes}_E \cC^{\Q_p-\la}(Z_{L_P}^0,E)[\fz_{L_P}=\fd]\otimes_E \sigma^{\vee}\Big)^{\!L_P^0} \\
\cong\!\Big(\Hom_{\text{U}(\ug)}(M_P(\mu), \Pi^{R_{\infty}-\an}_{\infty})^{N_P^0}_{\fss} \otimes_E (\delta_{\fd} \circ \dett_{L_P}) \widehat{\otimes}_E \cC^{\Q_p-\la}(Z_{L_P}^0,E)\otimes_E \sigma^{\vee}\Big)^{L_P^0}
\end{multline*}
where the second and fourth isomorphism follow by considering the action of $\fl_P$ (see also (\ref{disfib0})), the third isomorphism follows from (\ref{JPMV}) and the natural isomorphism
\begin{equation}\label{JacVM2}
J_P(\Pi^{R_{\infty}-\an}_{\infty})_{\mu}[\fz_{L_P}=0]\cong J_P(\Pi^{R_{\infty}-\an}_{\infty})_{\lambda}[\fz_{L_P}=\fd \circ \dett_{L_P}] \otimes_E (\delta_{\fd}^{-1} \circ \dett_{L_P}).
\end{equation}
The natural surjection $M_P(\mu) \twoheadrightarrow L(\mu)$ induces an injection
\begin{multline}\label{LmuMmu}
	\Big(\big(\Hom_{\text{U}(\ug)}(L(\mu), \Pi^{R_{\infty}-\an}_{\infty})^{N_P^0}_{\fss}\otimes_E (\delta_{\fd} \circ \dett_{L_P})\big)\widehat{\otimes}_E \cC^{\Q_p-\la}(Z_{L_P}^0,E)\otimes_E \sigma^{\vee}\Big)^{L_P^0} \\ \hooklongrightarrow
	\Big(\big(\Hom_{\text{U}(\ug)}(M_P(\mu), \Pi^{R_{\infty}-\an}_{\infty})^{N_P^0}_{\fss}\otimes_E (\delta_{\fd} \circ \dett_{L_P})\big)\widehat{\otimes}_E \cC^{\Q_p-\la}(Z_{L_P}^0,E)\otimes_E \sigma^{\vee}\Big)^{L_P^0}.
\end{multline}
As in the discussion above Lemma \ref{lemadj01}, we equip the right hand side of (\ref{LmuMmu}) with the topology induced by the one on $\Pi^{R_{\infty}-\an}_{\infty}\otimes_E (\delta_{\fd} \circ \dett_{L_P})\widehat{\otimes}_E \cC^{\Q_p-\la}(Z_{L_P}^0,E)\otimes_E \sigma^{\vee}$, and the left hand side of (\ref{LmuMmu}) with the topology induced by the one on the right hand side, which identifies it with a closed subspace (in particular, all spaces are of compact type). It is also clear that the morphism in (\ref{LmuMmu}) is equivariant under the action of $R_{\infty}\times \cZ_0 \times \cZ_{\Omega}$. In particular the left hand side of (\ref{LmuMmu}) is preserved by $\co_{\cE_{\Omega, \lambda}^{\infty}(\overline{\rho})_{\fd}}$. Let $\cL_{\mu, \lambda}$ be the quotient of the $\co_{\cE_{\Omega, \lambda}^{\infty}(\overline{\rho})_{\fd}}$-module $\cM_{\mu, \lambda}$ such that 
\begin{multline}\label{Lmulam}
	\Gamma\big( \cE_{\Omega, \lambda}^{\infty}(\overline{\rho})_{\fd}, \cL_{\mu, \lambda}\big)^{\vee}\\
	\cong \Big(\big(\Hom_{\text{U}(\ug)}(L(\mu), \Pi^{R_{\infty}-\an}_{\infty})^{N_P^0}_{\fss} \otimes_E (\delta_{\fd} \circ \dett_{L_P})\big) \widehat{\otimes}_E \cC^{\Q_p-\la}(Z_{L_P}^0,E)\otimes_E \sigma^{\vee}\Big)^{L_P^0}.
\end{multline} 
The $\co_{\cE_{\Omega, \lambda}^{\infty}(\overline{\rho})_{\fd}}$-module $\cL_{\mu, \lambda}$ is finitely generated, and its schematic support defines a (possibly empty) Zariski-closed rigid subspace $\fZ_{\mu, \lambda}^0$ of $\cE_{\Omega, \lambda}^{\infty}(\overline{\rho})_{\fd}$. We denote by $\fZ_{\mu,\lambda}\subseteq \fZ_{\mu,\lambda}^{0,\red}$ the union of its irreducible components of dimension $\dim \cE_{\Omega, \lambda}^{\infty}(\overline{\rho})_{\fd}$, which is still Zariski-closed in $\cE_{\Omega, \lambda}^{\infty}(\overline{\rho})_{\fd}$.

Next we move to a similar discussion for the completion of the patched Bernstein eigenvarieties at some specific points. Let $x:=(x^p, x_p, \ul{x}, \chi_x)\in (\Spf R_{\infty}^p)^{\rig} \times (\Spf R_{\overline{\rho}_p}^{\square})^{\rig} \times (\Spec \cZ_{\Omega})^{\rig} \times \widehat{\cZ_0}$. We write $y
:=(x^p,x_p) \in (\Spf R_{\infty})^{\rig}\cong (\Spf R_{\infty}^p)^{\rig} \times (\Spf R_{\overline{\rho}_p}^{\square})^{\rig}$. Let $\fm_y$, $\fm_{\ul{x}}$, $\fm_{\chi_x}$ be the associated maximal ideals of $R_{\infty}[1/p]$, $\cZ_{\Omega}$, $E[\cZ_0]$ respectively. Let $\fd_x:=\wt(\chi_x)$ and $\lambda^x:=\lambda+\fd_x \circ \dett_{L_P}$. By definition, we have an isomorphism of finitely generated $\widehat{\co}_{\cE_{\Omega, \lambda}^{\infty}(\overline{\rho}),x}$-modules (for such sheaves of modules, we identify them with their global sections with no ambiguity):
\begin{equation*}
	\cM^{\infty}_{\Omega, \lambda} \otimes_{\co_{\cE_{\Omega, \lambda}^{\infty}(\overline{\rho})}} \widehat{\co}_{\cE_{\Omega, \lambda}^{\infty}(\overline{\rho}),x} \cong \Big(\big(J_P(\Pi_{\infty}^{R_{\infty}-\an})_{\lambda}[\fm_{y}^{\infty}] \widehat{\otimes}_E \cC^{\Q_p-\la}(Z_{L_P}^0,E)[\fm_{\chi_x}^{\infty}] \otimes_E \sigma^{\vee}\big)^{L_P^0}[\fm_{\ul{x}}^{\infty}]\Big)^{\vee}
\end{equation*}
and an isomorphism of finitely generated $\widehat{\co}_{\cE_{\Omega, \lambda}^{\infty}(\overline{\rho})_{\fd_x},x}$-modules:
\begin{multline}\label{extralabel}
	\cM^{\infty}_{\Omega, \lambda} \otimes_{\co_{\cE_{\Omega, \lambda}^{\infty}(\overline{\rho})}} \widehat{\co}_{\cE_{\Omega, \lambda}^{\infty}(\overline{\rho})_{\fd_x},x} \\
	\cong 
	\Big(\big(J_P(\Pi_{\infty}^{R_{\infty}-\an})_{\lambda}[\fm_{y}^{\infty}] \widehat{\otimes}_E \cC^{\Q_p-\la}(Z_{L_P}^0,E)[\fz_{L_P}=\fd_x][\fm_{\chi_x}^{\infty}] \otimes_E \sigma^{\vee}\big)^{L_P^0}[\fm_{\ul{x}}^{\infty}]\Big)^{\vee}\\
	\cong \Big(\big(J_P(\Pi_{\infty}^{R_{\infty}-\an})_{\lambda}[\fm_{y}^{\infty}] \otimes_E \chi_x \otimes_E \sigma^{\vee}\big)^{L_P^0}[\fm_{\ul{x}}^{\infty}]\Big)^{\vee}
\end{multline}
which are both non-zero if and only if $x\in \cE_{\Omega, \lambda}^{\infty}(\overline{\rho})$. Assume in the sequel $\fd_x\in \Z^{[F^+:\Q]}$. By modifying the Bernstein centre $\Omega$ and using Proposition \ref{twBEi} (which obviously generalizes to the patched case), we can and do assume that the character $\chi_x$ is algebraic. We have
\begin{multline}\label{adj00}
 \Big(J_P(\Pi_{\infty}^{R_{\infty}-\an})_{\lambda}[\fm_{y}^{\infty}] \otimes_E \chi_x \otimes_E \sigma^{\vee}\Big)^{L_P^0}[\fm_{\ul{x}}^{\infty}] \\
\!\!\!\!\!\!\!\!\!\!\!\!\!\!\!\!\!\!\!\!\!\!\!\!\!\!\!\!\!\!\!\!\!\!\cong \Big(J_P(\Pi_{\infty}^{R_{\infty}-\an})_{\lambda}[\fm_{y}^{\infty}][\fz_{L_P}=\fd_x \circ \dett_{L_P}]\otimes_E \chi_x \otimes_E \sigma^{\vee}\Big)^{L_P^0}[\fm_{\ul{x}}^{\infty}] \\
\cong \!\Big(\!\Hom_{\text{U}(\ug)}\big(M_P(\lambda^x), \Pi^{R_{\infty}-\an}_{\infty}\big)_{\fss}^{N_P^0}[\fm_y^{\infty}]\otimes_E (\delta_{\fd_x} \circ \dett_{L_P})\otimes_E \chi_x\otimes_E \sigma^{\vee}\Big)^{\!L_P^0}\! [\fm_{\ul{x}}^{\infty}] 
\end{multline}
where the second isomorphism follows from (\ref{JPMV}) and (\ref{JacVM2}) (be careful with the action of $\Delta_0$). 

\begin{lemma}\label{exacBE2}
	With the above notation, the functor 
	\begin{equation*}
		M\longmapsto \Big(\Hom_{\text{U}(\ug)}\big(M, \Pi^{R_{\infty}-\an}_{\infty}\big)^{N_P^0}_{\fss}[\fm_y^{\infty}]\otimes_E (\delta_{\fd_x} \circ \dett_{L_P})\otimes_E \chi_x\otimes_E \sigma^{\vee}\Big)^{L_P^0} [\fm_{\ul{x}}^{\infty}]
	\end{equation*} 
	on the category $\co^{\fp}_{\alg}$ is exact.
\end{lemma}
\begin{proof}
	For each $t\in \Z_{\geq 1}$, let $\cI_t$ be the kernel of the composition $S_{\infty}[1/p]\ra R_{\infty}[1/p]\twoheadrightarrow R_{\infty}[1/p]/\fm_y^t$. Then $\Pi_{\infty}[\cI_t]^{\vee}$ is a finite projective $\co_E[[K_p]][1/p]$-module (see \S~\ref{secPBern}) and we have $\Pi_{\infty}[\cI_t]^{\an}=\Pi^{R_{\infty}-\an}_{\infty}[\cI_t]$ (which easily follows from \cite[(3.2)]{BHS1}). By Lemma \ref{exactnessBE0} and Lemma \ref{decompVM} applied to $V:=\Pi_{\infty}[\cI_t]$, the functor $M\mapsto (\Hom_{\text{U}(\ug)}(M, V)^{N_P^0}_{\fss}\otimes_E (\delta_{\fd_x} \circ \dett_{L_P})\otimes_E \chi_x\otimes_E \sigma^{\vee})^{L_P^0} [\fm_{\ul{x}}^{\infty}]$ is exact. We then argue as at the end of the proof of \cite[Thm.\ 5.5]{BHS2}.
\end{proof}

Define 
\begin{equation}\label{Cmxy}
	\cM_{\lambda^x, \lambda,y}:=	\cM^{\infty}_{\Omega, \lambda} \otimes_{\co_{\cE_{\Omega, \lambda}^{\infty}(\overline{\rho})}} \widehat{\co}_{\cE_{\Omega, \lambda}^{\infty}(\overline{\rho})_{\fd_x},x} \cong \cM_{\lambda^x, \lambda} \otimes_{\co_{\cE_{\Omega, \lambda}^{\infty}(\overline{\rho})_{\fd_x}}} \widehat{\co}_{\cE_{\Omega, \lambda}^{\infty}(\overline{\rho})_{\fd_x},x}
\end{equation}
which by (\ref{extralabel}) is isomorphic to the strong dual of the space in (\ref{adj00}). Let $L(\mu)$ be an irreducible constituent of $M_P(\lambda^x)$ and define $\cN_{\mu,\lambda,y}$ as the dual of the following space of compact type:
\begin{equation}\label{duNml}
	\Big(\Hom_{\text{U}(\ug)}\big(L(\mu), \Pi^{R_{\infty}-\an}_{\infty}\big)^{N_P^0}_{\fss}[\fm_y^{\infty}]\otimes_E (\delta_{\fd_x} \circ \dett_{L_P})
	\otimes_E \chi_x \otimes_E \sigma^{\vee}\Big)^{L_P^0} [\fm_{\ul{x}}^{\infty}].
\end{equation}
By Lemma \ref{exacBE2}, $\cN_{\mu, \lambda, y}$ is a subquotient of the $\widehat{\co}_{\cE_{\Omega, \lambda}^{\infty}(\overline{\rho})_{\fd_x},x}$-module $\cM_{\lambda^x, \lambda, y}$. One can show moreover that $\cN_{\mu,\lambda, y}$ is preserved by the action of $R_{\infty} \times \cZ_0 \times \cZ_{\Omega}$, hence is also a finitely generated $\widehat{\co}_{\cE_{\Omega, \lambda}^{\infty}(\overline{\rho})_{\fd_x},x}$-module. 

\begin{lemma}\label{copconcyc}
	We have $\cN_{\mu,\lambda,y}\neq 0$ if and only if 
	\begin{equation*}
		\Hom_{G_p}\Big(\cF_{P^-(\Q_p)}^{G_p}\big(L^-(-\mu), \pi_{\ul{x}} \otimes_E \big((\delta_{\fd_x}^{-1} \chi_{x, \ul{\varpi}})\circ \dett_{L_P}\big) \otimes_E \delta_P^{-1}\big), \Pi_{\infty}^{R_{\infty}-\an}[\fm_y]\Big) \neq 0.
	\end{equation*}
\end{lemma}
\begin{proof}
	By definition $\cN_{\mu,\lambda, y}\neq 0$ if and only if 
	\begin{equation*}
		\Big(\Hom_{\text{U}(\ug)}\big(L(\mu), \Pi^{R_{\infty}-\an}_{\infty}\big)^{N_P^0}_{\fss}[\fm_y]\otimes_E (\delta_{\fd_x} \circ \dett_{L_P})\otimes_E \chi_x\otimes_E \sigma^{\vee}\Big)^{L_P^0} [\fm_{\ul{x}}]\neq 0.
	\end{equation*}
	Using Lemma \ref{BCmod} (applied to $M=\cZ_{\Omega}/\fm_{\ul{x}}$) and Lemma \ref{lemadj01}, the above vector space is isomorphic to (be careful with the $L_P(\Q_p)$-action on the factor $\chi_x$):
	\begin{multline*}
		\Hom_{L_P(\Q_p)}\Big(\pi_{\ul{x}}, \Hom_{\text{U}(\ug)}\big(L(\mu), \Pi^{R_{\infty}-\an}_{\infty}[\fm_y]\big)^{N_P^0}_{\fss}\otimes_E (\delta_{\fd_x} \circ \dett_{L_P}) \otimes_E (\chi_{x,\ul{\varpi}})^{-1} \circ \dett_{L_P}\Big) \\
		\ \ \ \ \ \ \ \ \ \ \ \ \ \ \cong \Hom_{L_P(\Q_p)}\Big(\pi_{\ul{x}} \otimes_E \big((\delta_{\fd_x}(\chi_{x,\ul{\varpi}})^{-1}) \circ \dett_{L_P}\big), \Hom_{\text{U}(\ug)}\big(L(\mu), \Pi^{R_{\infty}-\an}_{\infty}[\fm_y]\big)^{N_P^0}_{\fss}\Big)\\
		\cong 	\Hom_{G_p}\Big(\cF_{P^-(\Q_p)}^{G_p}\big(L^-(-\mu), \pi_{\ul{x}} \otimes_E \big((\delta_{\fd_x}^{-1} \chi_{x, \ul{\varpi}})\circ \dett_{L_P}\big) \otimes_E \delta_P^{-1}\big), \Pi_{\infty}^{R_{\infty}-\an}[\fm_y]\Big).
	\end{multline*}
	The lemma follows.
\end{proof}

It follows from Theorem \ref{R=T0} and Proposition \ref{closedemR} (2) that the natural morphism $\widehat{\co}_{\fX_{\infty},y} \ra \widehat{\co}_{\cE_{\Omega, \lambda}^{\infty}(\overline{\rho})_{\fd_x},x}$ is surjective. Consequently $\cN_{\mu,\lambda, y}$ is a finitely generated $\widehat{\co}_{\fX_{\infty},y}$-module. Let $\lambda'$ be another integral $P$-dominant weight and $\fd$ an integral weight of $\fz_{L_P}$ such that $L(\mu)$ is an irreducible constituent of $M_P(\lambda'+\fd \circ \dett_{L_P})$. Let $\ul{x'}\in \Spec \cZ_{\Omega}$ such that $\pi_{\ul{x'}}\cong \pi_{\ul{x}} \otimes_E (\delta_{\fd-\fd_x, \ul{\varpi}}^{\unr}\circ \dett_{L_P})$ and $\chi':=\delta_{\fd}^0$ (recall $\chi_x=\delta_{\fd_x}^0$). Let $x':=(y,\ul{x}', \chi')\in (\Spf R_{\infty})^{\rig} \times (\Spec \cZ_{\Omega})^{\rig} \times \widehat{\cZ_0}$, hence $\lambda^{x'}=\lambda'+(\wt(\chi') \circ \dett_{L_P})=\lambda'+\fd \circ \dett_{L_P}$. Using similar arguments as in the proof of Lemma \ref{twBE} (especially in the last paragraph), we have a topological isomorphism which commutes with $R_{\infty}$:
\begin{multline*}
	\Big(\Hom_{\text{U}(\ug)}\big(L(\mu), \Pi^{R_{\infty}-\an}_{\infty}\big)^{N_P^0}_{\fss}[\fm_y^{\infty}]\otimes_E (\delta_{\fd_x} \circ \dett_{L_P})\otimes_E \chi_x \otimes_E \sigma^{\vee}\Big)^{L_P^0} [\fm_{\ul{x}}^{\infty}] \\
	\xlongrightarrow{\sim}\Big(\Hom_{\text{U}(\ug)}\big(L(\mu), \Pi^{R_{\infty}-\an}_{\infty}\big)^{N_P^0}_{\fss}[\fm_y^{\infty}]\otimes_E (\delta_{\fd} \circ \dett_{L_P})\otimes_E \chi'\otimes_E \sigma^{\vee}\Big)^{L_P^0} [\fm_{\ul{x}'}^{\infty}].
\end{multline*}
Note however that the $\cZ_{\Omega}\times \cZ_0$-actions on both sides differ by a twist.\footnote{As we will not use this result, we leave the curious readers work out the precise twist.} We then obtain (using Lemma \ref{exacBE2} for the second part):

\begin{lemma}\label{lemNmu'}
We have an isomorphism of $\widehat{\co}_{\fX_{\infty},y}$-modules $\cN_{\mu,\lambda, y} \cong \cN_{\mu,\lambda',y}$. Consequently, if $\cN_{\mu,\lambda, y}\neq 0$, then $\cM_{\lambda^{x'}, \lambda',y}\neq 0$, hence $x'=(y,\ul{x}', \chi')\in \cE_{\Omega, \lambda'}^{\infty}(\overline{\rho})$.
\end{lemma}
Keep the notation of Lemma \ref{lemNmu'} and suppose moreover $\lambda^{x'}=\lambda^x$, then by similar arguments (with $L(\mu)$ replaced by $M_P(\lambda^x)$), we obtain an isomorphism of $\widehat{\co}_{\fX_{\infty},y}$-modules $\cM_{\lambda^{x}, \lambda, y}\cong \cM_{\lambda^x, \lambda',y}$. 
We will use the notation $\cN_{\mu,y}:=\cN_{\mu,\lambda, y}$, $\cM_{\lambda^x, y}:=\cM_{\lambda^x, \lambda,y}$ in the sequel when we are only concerned with the $R_{\infty}$-action. As $\cE_{\Omega, \lambda}^{\infty}(\overline{\rho})_{\fd_x}$ is equidimensional of dimension $g+|S|n^2 +\sum_{v\in S_p}([F_{\widetilde{v}}:\Q_p] \frac{n(n-1)}{2})$, so is $\widehat{\co}_{\cE_{\Omega, \lambda}^{\infty}(\overline{\rho})_{\fd_x},x}$ (see for example the discussion on page 309 of \cite{BHS3}).
For a finitely generated $\widehat{\co}_{\cE_{\Omega, \lambda}^{\infty}(\overline{\rho})_{\fd_x},x}$-module $\cN$ we set (where $Z^i(-)$ is defined in the same way as in \S~\ref{secGalCyc}):
\[[\cN]:=\sum_{\fZ} m(\fZ, \cN)[ \fZ]\in Z^0(\Spec \widehat{\co}_{\cE_{\Omega, \lambda}^{\infty}(\overline{\rho})_{\fd_x},x})\subset Z^{[F^+:\Q]\frac{n(n+1)}{2}}(\Spec \widehat{\co}_{\fX_{\infty}, y})\]
where $ \fZ=\Spec \co_{\eta_{\fZ}}$ runs through the irreducible components of $\Spec \widehat{\co}_{\cE_{\Omega, \lambda}^{\infty}(\overline{\rho})_{\fd_x},x}$ ($\eta_{\fZ}$ is the associated generic point and $\co_{\eta_{\fZ}}$ the localization at $\eta_{\fZ}$), and $m(\fZ,\cN)$ is the length of the $\co_{\eta_{\fZ}}$-module $\cN_{\eta_{\fZ}}$. We have by Lemma \ref{exacBE2}: 
\begin{equation}\label{equcyc11}[\cM_{\lambda^x,y}]=\sum_{\mu} b_{\lambda^x,\mu} [\cN_{\mu,y}]
\end{equation}where $b_{\lambda^x,\mu} $ denotes the multiplicity of $L(\mu)$ as an irreducible constituent in $M_P(\lambda^x)$. The following lemma is straightforward (comparing (\ref{Lmulam}) with (\ref{duNml})):

\begin{lemma}\label{loccycglocy}
	We have $[\cN_{\mu,y}]\neq 0$ (resp.\ $\cN_{\mu,y}\neq 0$) if and only if $x\in\fZ_{\mu,\lambda}$ (resp.\ $x\in \fZ_{\mu,\lambda}^0$).
\end{lemma}

Assume that $\rho_{x,\widetilde{v}}$ is generic potentially crystalline with distinct Hodge-Tate weights $\textbf{h}_{\widetilde{v}}$ for all $v|p$, and that $x^p$ lies in the smooth locus of $\fX_{\infty}^p:=(\Spf R_{\infty}^p)^{\rig}$. By Corollary \ref{locModdv}, $\fX_{\infty}^p \times X_{\Omega, \textbf{h}}(\overline{\rho}_p)$ is irreducible at the point $x$ (see Theorem \ref{R=T0} for the notation). We deduce:

\begin{corollary}\label{corR=T0}
	The embedding in Theorem \ref{R=T0} induces a local isomorphism at the point $x$, and $\cE_{\Omega, \lambda}^{\infty}(\overline{\rho})$ is irreducible at $x$.
\end{corollary}

Let $\textbf{h}=(\textbf{h}_{\widetilde{v}})$ and $\mu=(\mu_{\widetilde{v}})=(\mu_{\widetilde{v},i,\tau})_{\substack{i=1, \dots, n\\ v\in S_p, \tau \in \Sigma_{\widetilde{v}}}}$ with $\mu_{\widetilde{v},i,\tau}=h_{\widetilde{v},i,\tau}+i-1$. By Proposition \ref{galPBE}, there exists $w_x=(w_{x,\widetilde{v}})\in \sW^P_{\min}\cong \prod_{v\in S_p} \sW^{P_{\widetilde{v}}}_{\min, F_{\widetilde{v}}}$ such that $\lambda^x=w_x \cdot \mu$. For each $v \in S_p$ and $w_{\widetilde{v}} \in \sW_{F_{\widetilde{v}}}$, let $\fC_{\widetilde{v}, w_{\widetilde{v}}}\in Z^{[F_{\widetilde{v}}:\Q_p]\frac{n(n+1)}{2}}(\Spec \widehat{\co}_{\fX_{\overline{\rho}_{\widetilde{v}}, \rho_{\widetilde{v}}}})$ be the cycle defined in (\ref{cycGaldef}) (applied to $L=F_{\widetilde{v}}$, $\rho=\rho_{\widetilde{v}}$ and so on). For each $w=(w_{\widetilde{v}})\in \sW_F$, we put
\begin{equation}\label{Cwf}
	\fC_{w,F}:=[\Spec \widehat{\co}_{\fX_{\infty}^p, x^p}] \times \prod_{v\in S_p} \fC_{\widetilde{v}, w_{\widetilde{v}}} \ \in \ Z^{[F^+:\Q]\frac{n(n+1)}{2}}(\Spec \widehat{\co}_{\fX_{\infty}, y}).
\end{equation}
Similarly, let $\fZ_{w,F}:=[\widehat{\co}_{\fX_{\infty}^p, x^p}] \times \prod_{v\in S_p} \fZ_{\widetilde{v}, w_{\widetilde{v}}} \ \in \ Z^{[F^+:\Q]\frac{n(n+1)}{2}}(\Spec \widehat{\co}_{\fX_{\infty}, y})$, where $\fZ_{\widetilde{v}, w_{\widetilde{v}}}$ is the cycle defined in (\ref{fZw}) (applied to $L=F_{\widetilde{v}}$, $\rho=\rho_{\widetilde{v}}$).
Note that both $\fC_{w,F}$ and $\fZ_{w,F}$ are independent of the choice of the representative $w$ in its associated class in $\sW_{L_P}\backslash \sW_F$. By Lemma \ref{irrcyc}, we have:

\begin{lemma}\label{irrcyc2}
	For $w\in \sW_F$, the cycle $\fZ_{w,F}$ is irreducible.
\end{lemma}
	
For $w,w'\in \sW_F$, let \[a_{w,w'}:=\prod_{v\in S_p} a_{w_{\widetilde{v}}, w_{\widetilde{v}}'}, \ c_{w,w'}:=\prod_{v\in S_p} c_{w_{\widetilde{v}},w_{\widetilde{v}}'}, \ c_{w,w'}':=\prod_{v\in S_p} c_{w_{\widetilde{v}},w_{\widetilde{v}}'}'\]
where $a_{w_{\widetilde{v}}, w_{\widetilde{v}}'}\in \Z_{\geq 0}$ are those $a_{w,w'}$ defined above (\ref{cycGaldef}) (applied to $L=F_{\widetilde{v}}$ and so on), $c_{w_{\widetilde{v}},w_{\widetilde{v}}'} ,c_{w_{\widetilde{v}},w_{\widetilde{v}}'}' \in \Z_{\geq 0}$ are the integers $c_{w,w'}$, $c_{w,w'}'$ respectively in Lemma \ref{lemcyccomp} (2) (adapted to the setting $G=\Res^{F_{\widetilde{v}}}_{\Q_p} \GL_n$ and so on). By (\ref{cycGaldef}) and Theorem \ref{thmcycl} (3), we have 
\begin{equation}\label{cycGaldef2}
\fC_{w,F}=\sum_{w'\in \sW_{L_P}\backslash \sW_F} \!\!\! a_{w,w'} \fZ_{w',F} =\sum_{\substack{w'\in \sW_{L_P}\backslash \sW_F \\ w'^{\max}\leq w^{\max}}} \!\!\! a_{w,w'} \fZ_{w',F} \in Z^{[F^+:\Q_p]\frac{n(n+1)}{2}}(\Spec \widehat{\co}_{\fX_{\infty}, y}).
\end{equation}
We let $w_{y,\widetilde{v}}\in \sW_{\max,F_{\widetilde{v}}}^{P_{\widetilde{v}}}$ be the element associated to $\rho_{\widetilde{v}}$ (and $\Omega$) as in \S~\ref{introPcr} (denoted by $w_{\sF}$ there) and
\begin{equation}\label{wyprod}
w_y:=(w_{y,\widetilde{v}}) \in \sW_{\max}^{P}=\prod_{v\in S_p} \sW_{\max,F_{\widetilde{v}}}^{P_{\widetilde{v}}}
\end{equation}
which is exactly the product for $v\in S_p$ of the $w_y$ in \S~\ref{secGalCyc} for $L=F_{\widetilde{v}}$. By Lemma \ref{lemnonzdR} we have:

\begin{corollary}\label{nonzero}
	For $w=(w_{\widetilde{v}})\in \sW_F$, we have $\fC_{w,F}\neq 0 \Longleftrightarrow \fZ_{w,F}\neq 0 \Longleftrightarrow w^{\max}\geq w_y$.
\end{corollary}

Corollary \ref{nonzero} allows to further refine (\ref{cycGaldef2}) as:

\begin{equation}\label{cycGaldef3}
\fC_{w,F}=\sum_{\substack{w'\in \sW_{L_P}\backslash \sW_F \\ w_y \leq w'^{\max}\leq w^{\max}}} \!\!\!\!\! a_{w,w'} \fZ_{w',F} \in Z^{[F^+:\Q_p]\frac{n(n+1)}{2}}(\Spec \widehat{\co}_{\fX_{\infty}, y}).
\end{equation}

\begin{corollary}\label{cycGaA}
(1) We have	
\begin{equation*}
	[\Spec \widehat{\co}_{\cE_{\Omega, \lambda}^{\infty}(\overline{\rho})_{\fd_{x}},x}]=\sum_{\substack{w\in \sW_{L_P}\backslash \sW_{F}\\ w_y \leq w^{\max} \leq w_xw_{0,F}}} \!\!\!\!\!\! c_{w_xw_{0,F},w}' \fZ_{w,F} \in Z^{[F^+:\Q]\frac{n(n+1)}{2}}(\Spec \widehat{\co}_{\fX_{\infty}, y}).
\end{equation*}

(2) Assume $x$ is a smooth point of $\cE_{\Omega, \lambda}^{\infty}(\overline{\rho})$, then we have:
\begin{equation*}
		[\Spec \widehat{\co}_{\cE_{\Omega, \lambda}^{\infty}(\overline{\rho})_{\fd_{x}},x}]=\sum_{\substack{w\in \sW_{L_P}\backslash \sW_{F}\\ w_y \leq w^{\max} \leq w_xw_{0,F}}}\!\!\!\!\!\! b_{w_xw_{0,F},w} \fC_{w,F} \in Z^{[F^+:\Q]\frac{n(n+1)}{2}}(\Spec \widehat{\co}_{\fX_{\infty}, y})
\end{equation*}
	where $b_{w,w'}:=\prod_{v\in S_p} b_{w_{\widetilde{v}}, w'_{\widetilde{v}}}$ for $w=(w_{\widetilde{v}})$,$w'=(w'_{\widetilde{v}})\in \sW_F$, which is equal to the multiplicity of $L(w'^{\max} w_{0,F} \cdot 0)$ in $M_P(w^{\max}w_{0,F} \cdot 0)$.
\end{corollary}
\begin{proof}
	(1) (resp.\ (2)) is a consequence of Corollary \ref{corR=T0}, Lemma \ref{lemcyccomp} (2) (resp.\ Lemma \ref{lemcyccomp} (1), (3)) and Corollary \ref{nonzero}.
\end{proof}

\subsection{Companion constituents}

In this section, we prove our main (global) results on $p$-adic automorphic representations. 

For $v\in S_p$, we let $\rho_{\widetilde{v}}$ be an $n$-dimensional generic potentially crystalline representation of $\Gal_{F_{\widetilde{v}}}$ over $E$ with distinct Hodge-Tate weights $\textbf{h}_{\widetilde{v}}=(h_{\widetilde{v},i,\tau})_{\substack{i=1, \dots, n\\ \tau \in \Sigma_{\widetilde{v}}}}$. Let $\lambda:=(\lambda_{\widetilde{v}}):=(\lambda_{\widetilde{v},i,\tau})_{\substack{i=1, \dots, n\\ v\in S_p, \tau \in \Sigma_{\widetilde{v}}}}$ with $\lambda_{\widetilde{v},i,\tau}=h_{\widetilde{v},i,\tau}+i-1$ (so $\lambda$ is integral $P$-dominant). Let $x:=(y,\ul{x},1)=(x^p, x_p, \ul{x},1)\in (\Spf R_{\infty})^{\rig} \times (\Spec \cZ_{\Omega})^{\rig} \times \widehat{\cZ_0}$ (so $y:=(x^p, x_p)\in \Spf R_{\infty})^{\rig}$). Let $w=(w_{\widetilde{v}}) \in \sW^P_{\min}$, and assume $x$ lies in $\cE_{\Omega, w\cdot \lambda}^{\infty}(\overline{\rho})$. Assume moreover that $x^p$ lies in the smooth locus of $\fX_{\infty}^p$. We let $w_y\in \sW^P_{\max}$ be as in (\ref{wyprod}).

\begin{lemma}\label{lemcycnovan0}
Let $w'\in \sW^P_{\min}$, if $[\cN_{w'\cdot \lambda, y}]\ne 0$ then $w'w_{0,F} \geq w_y$.
\end{lemma}
\begin{proof}
	If $[\cN_{w'\cdot \lambda, y}]\ne 0$, we deduce by Lemma \ref{lemNmu'} (applied to $\lambda'=w' \cdot \textbf{h}$, $\fd=\fd_x=0$) that $(y, \ul{x},1)\in \cE_{\Omega, w'\cdot \lambda}^{\infty}(\overline{\rho})$. Then using Theorem \ref{R=T0} and Corollary \ref{coLoccomp}, we get $w'w_{0,F} \geq w_y$. 
\end{proof}
The following lemma is a consequence of (\ref{equcyc11}), Remark \ref{remKLrel} and Lemma \ref{lemcycnovan0} (compare with Corollary \ref{cycGaA} (2)):
\begin{lemma}
	We have an equality of cycles in $Z^{[F^+:\Q]\frac{n(n+1)}{2}}(\widehat{\co}_{\fX_{\infty}, y})$:
	\begin{equation}\label{equcyc12}
	[\cM_{w\cdot \lambda},y]=\sum_{\substack{w'\in \sW_{\min}^P\\ w_y \leq w'w_{0,F} \leq ww_{0,F}}} \!\!\!\!\!\! b_{ww_{0,F},w'w_{0,F}} [\cN_{w' \cdot \lambda,y}].
	\end{equation}
\end{lemma}
\begin{proposition}\label{nonVancyc}
	We have $[\cN_{w\cdot \lambda, y}] \neq 0$. 
\end{proposition}
\begin{proof}
	By Theorem \ref{R=T0} and Corollary \ref{coLoccomp}, we have $ww_{0,F} \geq w_y$.
	
	If $ww_{0,F}=w_y$, then by Corollary \ref{corR=T0} and Corollary \ref{smDefVar2}, we know that $\cE_{\Omega, w\cdot \lambda}^{\infty}(\overline{\rho})$ is smooth at the point $x$. By Corollary \ref{PBEdim} (3), it follows that $\cM_{\Omega, w\cdot \lambda}^{\infty}$ is (non-zero) locally free at the point $x$, hence $[\cM_{w\cdot \lambda,y}]\neq 0$. We easily deduce from (\ref{equcyc12}) $[\cM_{w\cdot \lambda,y}]=[\cN_{w\cdot \lambda, y}]$
hence $[\cN_{w\cdot \lambda, y}] \neq 0$. 
	
	Now assume $ww_{0,F}>w_y$. We let $\fX_{\overline{\rho}_p}^{\pcr}(\xi_0, \textbf{h}):=\prod_{v\in S_p} \fX_{\overline{\rho}_{\widetilde{v}}}^{\pcr}(\xi_{0,\widetilde{v}}, \textbf{h}_{\widetilde{v}})$ where $\xi_0=(\xi_{0,\widetilde{v}})$ is the inertial type of $\rho_p=(\rho_{\widetilde{v}})$. For $w'=(w'_{\widetilde{v}})\in \sW^{P}_{\max}$, we let
\[V_{\overline{\rho}_p}^{\pcr}(\xi_0, \textbf{h})_{w'}:=\prod_{v\in S_p}V_{\overline{\rho}_{\widetilde{v}}}^{\pcr}(\xi_{0,\widetilde{v}}, \textbf{h}_{\widetilde{v}})_{w'_{\widetilde{v}}}\subset \fX_{\overline{\rho}_p}^{\pcr}(\xi_0, \textbf{h})\]
where each term in the product is defined as in Proposition \ref{VwSch}. As $ww_{0,F}> w_y^{\max}$, by Proposition \ref{VwSch}, we see that $\rho_p=(\rho_{\widetilde{v}})$ lies in the Zariski-closure $\overline{V_{\overline{\rho}_p}^{\pcr}(\xi_0, \textbf{h})_{ww_{0,F}}}$ of $V_{\overline{\rho}_p}^{\pcr}(\xi_0, \textbf{h})_{ww_{0,F}}$ in $\fX_{\overline{\rho}_p}^{\pcr}(\xi_0, \textbf{h})$. Let $\cC_{ww_{0,F}}$ be the connected component of $\overline{V_{\overline{\rho}_p}^{\pcr}(\xi_0, \textbf{h})_{ww_{0,F}}}$ at the point $\rho_p$, and let $\cU^p$ be the smooth locus of the irreducible component of $\fX_{\infty}^p$ at the point $x^p$. The embedding in (\ref{embpcydef}) induces a closed embedding
	\begin{equation*}
		\iota:	\overline{V_{\overline{\rho}_p}^{\pcr}(\xi_0, \textbf{h})_{ww_{0,F}}} \times \cU^p \hooklongrightarrow X_{\Omega, w\cdot \lambda}(\overline{\rho}_p) \times \cU^p.
	\end{equation*}
	We have $\iota(\rho_p, x^p)=x\in \cE_{\Omega, w\cdot \lambda}^{\infty}(\overline{\rho})$. Moreover, by Corollary \ref{locModdv}, it follows that $X_{\Omega, w\cdot \lambda}(\overline{\rho}_p)\times \cU^p$ is irreducible at any point in the image of $\iota$. We then deduce that $\iota$ induces a closed embedding
	\begin{equation}\label{iotaww0F}
		\iota: \cC_{ww_{0,F}} \times \cU^p \hooklongrightarrow \cE_{\Omega, w\cdot \lambda}^{\infty}(\overline{\rho}).
	\end{equation}
	Let $\cV_{ww_{0,F}} :=\cC_{ww_{0,F}} \cap V_{\overline{\rho}_p}^{\pcr}(\xi_0, \textbf{h})_{ww_{0,F} }$, which is Zariski dense in $\cC_{ww_{0,F}}$, $(x_p',x'^p)$ be any point in $\cV_{ww_{0,F}}\times \cU^p$, and put $y':=(x_p', x'^p)\in \fX_{\infty}$. Since $ww_{0,F}=w_{y'}$ (where, similarly as for $w_y$, $w_{y'}$ is associated to $x_p'$ and $\Omega$, noting that the framing in $x_p'$ does not cause problems), by the same argument as for the case $ww_{0,F}=w_y$ above (but now applied with $\iota(y')$ instead of $x$), we have $[\cN_{w\cdot \lambda, y'}]\neq 0$. By Lemma \ref{loccycglocy}, this implies $\iota(y') \in \fZ_{w\cdot \lambda, w\cdot \lambda}$. As $\iota(y)$ lies in the Zariski closure in $\cE_{\Omega, w\cdot \lambda}^{\infty}(\overline{\rho})$ of the set of all points $\iota(y')$ as above, and $\fZ_{w\cdot \lambda, w\cdot \lambda}$ is by definition Zariski-closed in $\cE_{\Omega, w\cdot \lambda}^{\infty}(\overline{\rho})$, we deduce $\iota(y)\in\fZ_{w\cdot \lambda, w\cdot \lambda}$, hence (by Lemma \ref{loccycglocy} again) $[\cN_{w\cdot \lambda, y}]\neq 0$.
\end{proof}

\begin{corollary}\label{ptclas11}
We have $[\cN_{w'\cdot \lambda, y}]\neq 0$ for all $w'\in \sW^P_{\min}$ such that $w'w_{0,F}\geq ww_{0,F}$ (equivalently $w'\leq w$). 
\end{corollary}
\begin{proof}
	Assume first $\lg(w')=\lg(w)+1$. As $[\cN_{w\cdot \lambda, y}]\neq 0$, by (\ref{equcyc12}), $[\cM_{w'\cdot \lambda, y}]\neq 0$ (note that $b_{ww_{0,F},w'w_{0,F}}=1$, cf.\ Remark \ref{remKLrel}) hence $(x^p, x_p, \ul{x}, 1)\in \cE_{\Omega, w'\cdot \lambda}^{\infty}(\overline{\rho})$ (cf.\ (\ref{Cmxy})). Applying Proposition \ref{nonVancyc} with $w$ replaced by $w'$, we obtain $[\cN_{w'\cdot \lambda, y}]\neq 0$. We can then start again this argument with a $w''$ such that $\lg(w'')=\lg(w')+1$. Using \cite[Thm.\ 2.5.5]{BjBr}, the corollary follows from an obvious induction.
\end{proof}

\begin{remark}
In particular, if $x\in \cE_{\Omega, w\cdot \lambda}^{\infty}(\overline{\rho})$, then $x\in \cE_{\Omega, w'\cdot \lambda}^{\infty}(\overline{\rho})$ for all $w'\in \sW^P_{\min}$, $w'\leq w$. In the trianguline case (i.e.\ $P=B$), it was proved in \cite[Thm.\ 5.5]{BHS2} that this holds without assuming $\rho_{\widetilde{v}}$ to be generic potentially crystalline with distinct Hodge-Tate weights (see also Remark \ref{remNPara} (2)). The proof was quite different and based on the fact that, for any $w'\leq w$ (in $\sW_F$), $L(w\cdot \lambda)$ is an irreducible constituent of the Verma module $M(w'\cdot \lambda)$, which does not hold in general if $w',w\in \sW_{\min}^P$ and $M(w' \cdot \lambda)$ is replaced by $M_P(w' \cdot \lambda)$. When $P\neq B$, the authors do not know if this statement holds without the potentially crystalline assumption. 
\end{remark}

\begin{corollary}
	Assume that we are in the same situation as Conjecture \ref{conjCPP} (we use the notation of \textit{loc.\ cit}.) Let $x^p$ be the point of $(\Spf R_{\infty}^p)^{\rig}$ associated to $\fm$, and assume moreover that $x^p$ is a smooth point of $(\Spf R_{\infty}^p)^{\rig}$. Then Conjecture \ref{conjCPP} (1) is equivalent to Conjecture \ref{conjCPP} (2).
\end{corollary}
\begin{proof}
	By Lemma \ref{lemCPCC}, we are left to show that, if $x=(\fm, \boxtimes_{v\in S_p} \pi_{L_{P_{\widetilde{v}}}}, 1)\in \cE_{\Omega, w\cdot \lambda}^{\infty}(\overline{\rho})$, then $\widehat{\otimes}_{v\in S_p} C(w_{\widetilde{v}}, \sF_{\widetilde{v}})\hookrightarrow \Pi_{\infty}^{R_{\infty}-\an}[\fm]$. However, this follows directly from Proposition \ref{nonVancyc} and Lemma \ref{copconcyc} (applied to $\mu=w\cdot \lambda$, $\fm_y=\fm$ etc.).
\end{proof}

\begin{corollary}[Classicality]\label{classFM}
Let $\overline{\rho}$ be an $U^p$-modular continuous representation of $\Gal_F$ over $k_E$. Let $\cE_{\Omega, \lambda}(U^p, \overline{\rho})$ be as in (the end of) \S~\ref{galois} and $x=(\eta,\pi_{L_P},\chi)\in \cE_{\Omega, \lambda}(U^p, \overline{\rho})$. Assume
	\begin{enumerate}
		\item[(1)] Hypothesis \ref{TayWil};
		\item[(2)] $\rho_{x,\widetilde{v}}$ is generic potentially crystalline with distinct Hodge-Tate weights for all $v\in S_p$;
		\item[(3)] $\rho_{x,\widetilde{v}}$ is a smooth point of $(\Spf R_{\overline{\rho}_{\widetilde{v}}})^{\rig}$ for $v\in S\setminus S_p$.
	\end{enumerate} 
	Then $\widehat{S}(U^p,E)^{\lalg}[\bT^S=\eta]\neq 0$, i.e.\ $\rho_x$ is associated to a classical automorphic representation of $G(\bA_{F^+})$. 
\end{corollary}
\begin{proof}
	By Proposition \ref{twBEi}, we can and do assume $\chi=1$ (by modifying $\lambda$ and $\Omega$). Then by Proposition \ref{galBE1}, there exist a dominant weight $\mu$ and $w\in \sW^P_{\min}$ such that $\lambda=w\cdot \mu$. By Hypothesis \ref{TayWil} and Proposition \ref{bePbe}, we can associate to $x$ a point, still denoted by $x$, of the form $(y=(x^p,x_p), \pi_{L_P},1)$ in $\cE_{\Omega, \lambda}^{\infty}(\overline{\rho})$ satisfying $\widehat{S}(U^p,E)[\bT^S=\eta]\cong \Pi_{\infty}[\fm_y]$. Moreover, by the assumption (3), $x^p$ is a smooth point of $(\Spf R_{\infty}^p)^{\rig}$. By Corollary \ref{ptclas11}, we see $\cN_{\mu,y} \neq 0$ hence by Lemma \ref{copconcyc} (noting that, as $\mu$ is dominant, the Orlik-Strauch representation $\cF_{P^-(\Q_p)}^{G_p}(-)$ there is locally algebraic) $\widehat{S}(U^p,E)[\bT^S=\eta]^{\lalg}\cong \Pi_{\infty}[\fm_y]^{\lalg}\neq 0$. The corollary follows.
\end{proof}

\begin{remark}
(1) Using similar arguments as in the proof of \cite[Thm.\ 3.9]{BHS2}, one may remove the assumption (3) in Corollary \ref{classFM} but assuming the weight $\lambda+\wt(\chi)\circ \dett_{L_P}$ is dominant.

(2) If $S\setminus S_p=\{v_1\}$ where $v_1$ is a finite place of $F^+$ as in \cite[\S~2.3]{CEGGPS} and $U^p$ is chosen as in \textit{loc.\ cit.}, then $(\Spf R_{\infty}^p)^{\rig}$ is smooth (so the assumption (3) is automatically satisfied). 
\end{remark}

\begin{theorem}\label{class00}
	Assume that we are in the setting of \S~\ref{s:pAF}, and let $\rho: \Gal_F \ra \GL_n(E)$ be a continuous representation unramified outside $S$ such that the modulo $p$ reduction $\overline{\rho}$ is irreducible and $U^p$-modular. Let $\fm_{\rho}$ be the associated maximal ideal of $\bT^S[1/p]$ (via (\ref{ES})). Assume
	\begin{enumerate}[label=(\arabic*)]
		\item Hypothesis \ref{TayWil};
		\item $\rho_{\widetilde{v}}$ is generic potentially crystalline with distinct Hodge-Tate weights for all $v\in S_p$;
		\item $\rho_{\widetilde{v}}$ is a smooth point of $(\Spf R_{\overline{\rho}_{\widetilde{v}}})^{\rig}$ for $v\in S\setminus S_p$;
		\item there exists a parabolic subgroup $P\supset \prod_{v\in S_p}\Res^{F_{\widetilde{v}}}_{\Q_p} B$ such that $J_P(\widehat{S}(U^p,E)^{\an}[\fm_{\rho}])$ has non-zero locally algebraic vectors for $L_P^D(\Q_p)$.
	\end{enumerate}
	Then $\widehat{S}(U^p,E)[\fm_{\rho}]^{\lalg}\neq 0$, i.e.\ $\rho$ is associated to a classical automorphic representation of $G(\bA_{F^+})$. 
\end{theorem}
\begin{proof}
	Let $P$ be a minimal parabolic subgroup such that the assumption (4) holds. Thus there exists an integral $P$-dominant weight $\lambda$ such that $J_P(\widehat{S}(U^p,E)^{\an}[\fm_{\rho}])_{\lambda}\neq 0$ (see (\ref{jplambda})). As $J_P(\widehat{S}(U^p,E)^{\an}[\fm_{\rho}])_{\lambda}$ is an essentially admissible locally analytic representation of $L_P(\Q_p)$ and is smooth for the $L_P^D(\Q_p)$-action, there exists a compact open subgroup $H\subset L_P^D(\Z_p)$ such that $J_P(\widehat{S}(U^p,E)^{\an}[\fm_{\rho}])_{\lambda}^{H}$ is a non-zero essentially admissible locally analytic representation of $Z_{L_P}(\Q_p)$. There exists thus a continuous character $\delta$ of $Z_{L_P}(\Q_p)$ such that
	\[J_P(\widehat{S}(U^p,E)^{\an}[\fm_{\rho}])_{\lambda}^{H} [Z_{L_P}(\Q_p)=\delta]\neq 0.\]
	We let $\chi$ be a continuous character of $Z_{L_P}^0$ such that $\wt(\chi)\circ \dett_{L_P}=\wt(\delta)$ and put $\delta^{\infty}:=\delta (\chi_{\ul{\varpi}}^{-1} \circ \dett_{L_P})$, which is a smooth character of $Z_{L_P}(\Q_p)$. Consider the (non-zero) $L_P(\Q_p)$-representation:
	\begin{equation*}
		(J_P(\widehat{S}(U^p,E)^{\an}[\fm_{\rho}])_{\lambda}\otimes_E \chi_{\ul{\varpi}}^{-1} \circ \dett_{L_P})[Z_{L_P}(\Q_p)=\delta^{\infty}]
	\end{equation*}
which is smooth and admissible by \cite[Cor.\ 6.4.14]{Em04}. Let $\pi_{L_P}$ be an irreducible subrepresentation of $(J_P(\widehat{S}(U^p,E)^{\an}[\fm_{\rho}])_{\lambda}\otimes_E \chi_{\ul{\varpi}}^{-1} \circ \dett_{L_P})[Z_{L_P}(\Q_p)=\delta^{\infty}]$, so we have
\begin{equation}\label{injLP1}
	\pi_{L_P} \otimes_E (\chi_{\ul{\varpi}} \circ \dett_{L_P}) \otimes_E L(\lambda)_P\hooklongrightarrow J_P(\widehat{S}(U^p,E)^{\an}[\fm_{\rho}]).
\end{equation} 
Let $\Omega$ be the Bernstein component of $\pi_{L_P}$. As $P$ is minimal, using (\ref{injLP1}), it is not difficult to see that $\Omega$ is cuspidal. By (\ref{fiberE}) and (\ref{injLP1}), it follows that $(\fm_{\rho}, \pi_{L_P}, \chi)\in \cE_{\Omega, \lambda}(U^p, \overline{\rho})$. The theorem then follows from Corollary \ref{classFM}.
\end{proof}

\begin{remark}
The case where $P=\prod_{v\in S_p}\Res^{F_{\widetilde{v}}}_{\Q_p} B$ was essentially proved (without the assumption (3)) in \cite[Thm.\ 5.1.3]{BHS3} (see also \cite[Rem.\ 5.1.5]{BHS3}).
\end{remark}

Next we prove results towards the socle conjecture in \cite{Br13I}. Let $\lambda=(\lambda_{\widetilde{v},i,\tau})$ be an integral dominant weight of $G_p$, $\textbf{h}=(h_{\widetilde{v},i,\tau})$ with $h_{\widetilde{v},i,\tau}=\lambda_{\widetilde{v},i,\tau}-i+1$ (hence $\textbf{h}$ is strictly dominant). For $w\in \sW^P_{\min}$, let $\sS_{w}$ be the set of points $x=(y=(x^p,x_p), \ul{x}, 1)\in \cE_{\Omega, \lambda}^{\infty}(\overline{\rho})_0$ (recall that $\cE_{\Omega, \lambda}^{\infty}(\overline{\rho})_0$ is the fibre of $\cE_{\Omega, \lambda}^{\infty}(\overline{\rho})$ at the weight $0$ of $\fz_{L_P}$) such that:
\begin{itemize}
	\item[(1)] for $v\in S\setminus S_p$, $\rho_{x,\widetilde{v}}$ is a smooth point in $(\Spf R_{\overline{\rho}_{\widetilde{v}}})^{\rig}$;
	\item[(2)] for $v\in S_p$, $\rho_{x,\widetilde{v}}$ is generic potentially crystalline of Hodge-Tate weights $\textbf{h}_{\widetilde{v}}$ and of inertial type $\xi_{0,\widetilde{v}}$ (which is determined by $\Omega$);
	\item[(3)] $w_y=ww_{0,F}$.
\end{itemize}
Put $\sS:=\cup_{w\in \sW^P_{\min}} \sS_w$. Note that $\sS_1$ consists of non-critical points in $\sS$. 
 
\begin{proposition}\label{proppcrycl}
Let $x\in \sS$, then we have $[\cN_{\lambda, y}] \in \Z \fC_{w_{0,F},F} \in Z^{\frac{n(n+1)}{2}[F^+:\Q]}(\widehat{\co}_{\fX_{\infty}, y})$ (cf.\ (\ref{Cwf})).
\end{proposition}
\begin{proof}
Let $\sigma$ be a cuspidal $L_P(\Z_p)$-type for $\Omega$. For any $s \in \Z_{\geq 1}$, put $\pi_{\ul{x},s}:=(\cind_{L_P^0}^{L_P(\Q_p)} \sigma)^{\infty} \otimes_{\cZ_{\Omega}} \cZ_{\Omega}/\fm_{\ul{x}}^s$. By Lemma \ref{BCmod} and Lemma \ref{lemadj01}, we have (noting that $\cN_{\lambda,y}^{\vee}$ is isomorphic to the direct limit of the following $E$-vector spaces over $s$, cf.\ (\ref{duNml}))
	\begin{eqnarray}
		&&\!\!\!\!\!\!\!\!\!\!\!\!\!\!\!\!\!\!\!\!\!\!\!\!\!\!\!\!\!\!\!\!\!\!\!\!\!\!\!\!\!\!\!\!\!\!\!\!\!\!\!\!\! \big(\Hom_{\text{U}(\ug)}(L(\lambda), \Pi_{\infty}^{R_{\infty}-\an}[\fm_y^s])_{\fss}^{N_P^0} \otimes_E \sigma^{\vee}\big)^{L_P^0}[\fm_{\ul{x}}^s] \nonumber 
		\\
		&\xrightarrow{\sim}& \Hom_{L_P(\Q_p)}\Big(\pi_{\ul{x},s}, \Hom_{\text{U}(\ug)}(L(\lambda), \Pi_{\infty}^{R_{\infty}-\an}[\fm_y^s])_{\fss}^{N_P^0} \Big) \nonumber 
		\\
		&\xrightarrow{\sim} &
		\Hom_{G_p}\Big(\cF_{P^-(\Q_p)}^{G_p}\big(L^-(-\lambda), \pi_{\ul{x}, s} \otimes_E\delta_P^{-1}\big), \Pi_{\infty}^{R_{\infty}-\an}[\fm_y^{s}]\Big) \nonumber \\
	&	\cong &
		\Hom_{G_p}\big(L(\lambda) \otimes_E (\Ind_{P^-(\Q_p)}^{G_p}\pi_{\ul{x}, s} \otimes_E \otimes_E\delta_P^{-1})^{\infty}, \Pi_{\infty}^{R_{\infty}-\an}[\fm_y^{s}]\big)\nonumber \\
	&	\hookrightarrow& 
	\Hom_{G_p}\Big(L(\lambda) \otimes_E (\cind_{K_p}^{G_p} \widetilde{\sigma}), \Pi_{\infty}^{R_{\infty}-\an}[\fm_y^{s}]\Big) \nonumber \\
	&\cong&
		\Hom_{K_p}(L(\lambda)\otimes_E \widetilde{\sigma}, \Pi_{\infty}^{R_{\infty}-\an})[\fm_y^{s}] \label{Nlambday1}
	\end{eqnarray}
	where $\widetilde{\sigma}$ is a $K_p$-type of $G_p$ for $(\Ind_{P^-(\Q_p)}^{G_p} \pi_{\ul{x}} \otimes_E\delta_P^{-1})^{\infty}$ (note that $(\Ind_{P^-(\Q_p)}^{G_p}\pi_{\ul{x}, s} \otimes_E\delta_P^{-1})^{\infty}$ is isomorphic to a successive extension of $(\Ind_{P^-(\Q_p)}^{G_p} \pi_{\ul{x}} \otimes_E\delta_P^{-1})^{\infty}$). By \cite[Lemma 4.17]{CEGGPS}, the $R_\infty$-action on the last term in (\ref{Nlambday1}) factors through $R_{\infty}^p \widehat{\otimes}_{\co_E} R_{\overline{\rho}_p}^{\pcr}(\xi_0, \textbf{h})$ (where we recall that $R_{\overline{\rho}_p}^{\pcr}(\xi_0, \textbf{h})=\widehat{\otimes}_{v\in S_p} R_{\overline{\rho}_{\widetilde{v}}}^{\pcr}(\xi_{0,\widetilde{v}}, \textbf{h}_{\widetilde{v}})$). We deduce that the $R_\infty$-action on $\cN_{\lambda,y}$ also factors through $R_{\infty}^p \widehat{\otimes}_{\co_E} R_{\overline{\rho}_p}^{\pcr}(\xi_0, \textbf{h})$. By Remark \ref{remPocrycycl}, we have $\fC_{w_{0,F}, F}=[\widehat{\co}_{\fX^p_{\infty}, x^p} \widehat{\otimes} \widehat{\co}_{\fX_{\overline{\rho}_p}^{\pcr}(\xi_0, \textbf{h}), \rho_p}]$ (here the completion in $\widehat{\otimes}$ is with respect to the maximal ideal associated to $(x^p, \rho_p)$). The proposition follows.
\end{proof}

Let $x\in \sS$ and recall that, by Corollary \ref{corR=T0}, $\cE_{\Omega, \lambda}^{\infty}(\overline{\rho})$ is irreducible at the point $x$. Let $m_y\in \Z_{\geq 1}$ such that $\cM_{\Omega, \lambda}^{\infty}$ is locally free of rank $m_y$ in the smooth locus of a sufficiently small neighbourhood of $x$ in $\cE_{\Omega, \lambda}^{\infty}(\overline{\rho})$. We view $x$ as a point of $(\Spf R_{\infty})^{\rig} \times (\Spec \cZ_{\Omega})^{\rig} \times \widehat{\cZ_0}$. For $w\in \sW^P_{\min}$, we want to understand when $x\in \cE_{\Omega, w\cdot \lambda}^{\infty}(\overline{\rho})$.

\begin{lemma}\label{lgleq1}
	If $\lg(w_y) \geq \lg(w_{0,F})-1$, then for any $w\in \sW^P_{\min}$ such that $ww_{0,F} \geq w_y$, we have $[\cN_{w \cdot \lambda, y}]\neq 0$, hence $x\in \cE_{\Omega, w\cdot \lambda}^{\infty}(\overline{\rho})$. Moreover, $\cM_{\Omega, w\cdot \lambda}^{\infty}$ is locally free of rank $m_y$ at $x$.
\end{lemma}
\begin{proof}
	Since $\lg(w_y) \geq \lg(w_{0,F})-1$, by Corollary \ref{corR=T0} and Corollary \ref{smDefVar2}, $\cE_{\Omega, \lambda}^{\infty}(\overline{\rho})$ is smooth at the point $x$. Hence by Corollary \ref{PBEdim} (3), $\cM_{\Omega, \lambda}^{\infty}$ is locally free of rank $m_y$ at the point $x$. By Corollary \ref{cycGaA} (2) and (\ref{equcyc12}) (noting that, as $\lg(w_y) \geq \lg(w_{0,F})-1$, for $w\in \sW_{\min,F}^P$ we have $b_{w_{0,F}, ww_{0,F}}=\begin{cases}
	1 & ww_{0,F}\geq w_y \\ 0 & \text{otherwise}
	\end{cases}$, see Remark \ref{remKLrel} and Lemma \ref{mult00} (1) below), we have: 
	\begin{equation}\label{equcycl000}
	 \sum_{w_y\leq ww_{0,F} \leq w_{0,F}}\!\!\!\!	[\cN_{w \cdot \lambda, y}]=[\cM_{\lambda, y}]=m_y \sum_{w_y\leq ww_{0,F} \leq w_{0,F}}\!\!\!\! \fC_{ww_{0,F}, y}\in Z^{[F^+:\Q]\frac{n(n+1)}{2}}(\widehat{\co}_{\fX_{\infty}, y}).
	\end{equation}
	By Lemma \ref{nonVancyc}, $[\cN_{\lambda,y}]\neq 0$ and the case where $w_y=w_{0,F}$ (which implies $w=1$) follows. 
	Assume now $\lg(w_y) = \lg(w_{0,F})-1$ and $ww_{0,F}=w_y$. 
	By (\ref{equcycl000}) and Proposition \ref{proppcrycl}, we have 
	\begin{equation}\label{equcycl001}
	[\cN_{w \cdot \lambda, y}]=m_y \fC_{w_{0,F}, y}+m_y \fC_{ww_{0,F},y}-[\cN_{\lambda, y}]\in m_y \fC_{ww_{0,F},y}+\Z \fC_{w_{0,F}, y}.
	\end{equation}
	By (\ref{cycGaldef3}) and Lemma \ref{nonzero}, it is easy to see $\fC_{ww_{0,F},y}\notin \Z\fC_{w_{0,F}, y}$. We deduce $[\cN_{w \cdot \lambda, y}]\neq 0$, hence $x\in \cE_{\Omega, w\cdot \lambda}^{\infty}(\overline{\rho})$. Again by Corollary \ref{corR=T0} and Corollary \ref{smDefVar2}, $\cE_{\Omega, w\cdot \lambda}^{\infty}(\overline{\rho})$ is smooth at $x$ hence $\cM_{\Omega, w\cdot \lambda}^{\infty}$ is locally free at $x$, say of rank $m_y'$. Hence $[\cN_{w \cdot \lambda, y}]=[\cM_{w \cdot \lambda, y}]=m_y' \fC_{ww_{0,F}, y}\in \Z_{\geq 0} \fC_{ww_{0,F}, y}$. By (\ref{equcycl001}) and the fact $\fC_{ww_{0,F},y}\notin \Z\fC_{w_{0,F}, y}$ (again), we deduce $[\cN_{w \cdot \lambda, y}]=m_y \fC_{ww_{0,F}, y}$ (i.e.\ $m_y'=m_y$) and $[\cN_{\lambda, y}]=m_y \fC_{w, y}$. This concludes the proof.
\end{proof}

\begin{lemma}\label{indstLem1}
Let $w\in \sW^P_{\min}$.
	
(1) Assume that, for any $x'=(y', \ul{x'}, 1)\in \sS_{w}$, we have $[\cN_{w \cdot \lambda, y'}]\neq 0$. Then for any $x=(y,\ul{x},1)\in \sS$ with $w_y\leq ww_{0,F}$, we have $[\cN_{w \cdot \lambda, y}]\neq 0$.
	
(2) Keep the assumption in (1), and assume moreover that, for all $x'\in \sS_{w}$, $\cM_{\Omega, w\cdot \lambda}^{\infty}$ is locally free of rank $m_{y'}$ in the smooth locus of a sufficiently small neighbourhood of $x'$. Then, for any $x=(y,\ul{x},1)\in \sS$ with $w_y\leq ww_{0,F}$, $\cM_{\Omega, w\cdot \lambda}^{\infty}$ is locally free of rank $m_y$ in the smooth locus of a sufficiently small neighbourhood of $x$.
\end{lemma}
\begin{proof}
The lemma follows by the same arguments as in Step 9 of the proof of \cite[Thm.\ 5.3.3]{BHS3}. We include a proof for the convenience of the reader.
	
	(1) By Lemma \ref{loccycglocy}, we only need to show $x\in \fZ_{w\cdot \lambda, w\cdot \lambda}$. We use the notation in the proof of Proposition \ref{nonVancyc}. Denote by $\cC \subset \fX_{\overline{\rho}_p}^{\pcr}(\xi_0, \textbf{h})$ the irreducible component containing $\rho_{x,p}$. The embedding in (\ref{embpcydef}) induces an embedding
	\begin{equation*}
		\iota: \cC \times \cU^p \hooklongrightarrow X_{\Omega, \textbf{h}}(\overline{\rho}_p) \times \cU^p
	\end{equation*}
	which sends $y$ to $x\in \cE_{\Omega, \lambda}^{\infty}(\overline{\rho})\hookrightarrow X_{\Omega, \textbf{h}}(\overline{\rho}_p) \times \cU^p$. 
	By Corollary \ref{corR=T0}, $X_{\Omega, \lambda}(\overline{\rho}_p) \times \cU^p$ is irreducible at any point in the image of $\iota$. We then deduce $\iota(\cC \times \cU^p)\subset \cE_{\Omega, \lambda}^{\infty}(\overline{\rho})$ (this is similar to (\ref{iotaww0F})). It is clear that $\cV_{ww_{0,F}}\subset \cC_{ww_{0,F}}\subset \cC$ (recalling that $\cC$ is irreducible and smooth, cf.\ \cite[Thm.\ 3.3.8]{Kis08}), and that any point in $\iota(\cV_{ww_{0,F}} \times \cU^p)\subset \cE_{\Omega, \lambda}^{\infty}(\overline{\rho})$ lies in $\sS_{w}$. Recall the morphism in 
	(\ref{embpcydef}) induces 
	\begin{equation*}
		\iota: \cC_{ww_{0,F}} \times \cU^p \hooklongrightarrow X_{\Omega, w(\textbf{h})}(\overline{\rho}_p) \times \cU^p.
	\end{equation*}
	For each $y'\in \cV_{ww_{0,F}}\times \cU^p$ with $x'=\iota(y')\in \sS_w$, we have by assumption $[\cN_{w \cdot \lambda}]\neq 0$, hence by Lemma \ref{loccycglocy} $x'\in \fZ_{w \cdot \lambda, w\cdot \lambda}\subset \cE_{\Omega, w\cdot \lambda}^{\infty}(\overline{\rho}) \hookrightarrow X_{\Omega, w(\textbf{h})}(\overline{\rho}_p) \times \cU^p$. 
	As $x$ lies in the Zariski-closure of such $x'$, it follows $x=\iota(y)\in \fZ_{w\cdot \lambda, w\cdot \lambda}$. (1) is then a consequence of Lemma \ref{loccycglocy}.
	
	(2) By (1), we have $x\in \fZ_{w\cdot \lambda, w\cdot \lambda}\hookrightarrow \cE_{\Omega, w\cdot \lambda}^{\infty}(\overline{\rho})$. Let $\cD_{w}$ (resp.\ $\cD_{1}$) be the irreducible component of $\cE_{\Omega, w\cdot \lambda}^{\infty}(\overline{\rho})$ (resp.\ $\cE_{\Omega, \lambda}^{\infty}(\overline{\rho})$) containing $x$. As in the proof of (1), we have $\iota(\cV_{ww_{0,F}})\subset \cD_1$ and $\iota(\cV_{ww_{0,F}})\subset \cD_w$. Let $U_w$ be a Zariski-open neighbourhood of $x$ in $\cE_{\Omega, w\cdot \lambda}^{\infty}(\overline{\rho})$. Since $\cE_{\Omega, w\cdot \lambda}^{\infty}(\overline{\rho})$ is irreducible at $x$ (Corollary \ref{corR=T0}), shrinking $U_w$, we can and do assume that $U_w$ and the smooth locus $U_w^{\sm}$ of $U_w$ are both irreducible (note that $U_w^{\sm}$ is Zariski-open Zariski-dense in $U_w$, cf.\ \cite[Prop.\ 2.3]{BHS1}). There exists $m$ such that $\cM_{\Omega, w\cdot \lambda}^{\infty}$ is locally free on $U_w^{\sm}$ of constant rank $m$. Similarly, we let $U_1$ be an irreducible Zariski-open neighbourhood of $x$ in $\cE_{\Omega, \lambda}^{\infty}(\overline{\rho})$ such that $U_1^{\sm}$ is also irreducible. Then $\cM_{\Omega, \lambda}^{\infty}$ is locally free on $U_1^{\sm}$ of rank $m_y$. We have $(x^p, x_p)\in \iota^{-1}(U_1) \cap \iota^{-1}(U_w) \cap (\cC_{ww_{0,F}}\times \cU^p)$. Since $\cV_{ww_{0,F}} \times \cU^p$ is Zariski-dense in $\cC_{ww_{0,F}} \times \cU^p$, we deduce $\iota^{-1}(U_1) \cap \iota^{-1}(U_w) \cap (\cV_{ww_{0,F}} \times \cU^p)\neq \emptyset$, and we let $y'=(x'^p, x'_p)$ be a point in the intersection. By Corollary \ref{corR=T0} and Corollary \ref{smDefVar2}, the point $x':=\iota(y')\in U_w$ is in fact a smooth point of $U_w$. Hence by assumption $m=m_{y'}=m_{y}$. This finishes the proof.
\end{proof}

\begin{lemma}\label{induccp}
	Let $x=(y,\ul{x},1)\in \sS$ and assume $\lg(w_y)\leq \lg(w_{0,F})-2$. Assume that, for all $w \in \sW^P_{\min}$, with $ww_{0,F} > w_y$, we have 
	\begin{itemize}
		\item $[\cN_{w \cdot \lambda, y}]\neq 0$ (so $x\in \cE_{\Omega, w\cdot \lambda}^\infty(\overline{\rho})$);
		\item $\cM_{\Omega, w\cdot \lambda}^{\infty}$ is locally free of rank $m_y$ in the smooth locus of a neighbourhood of $x\in \cE_{\Omega, w\cdot \lambda}^\infty(\overline{\rho})$.
	\end{itemize}
	We assume that one of the following two conditions holds:
	\begin{enumerate}[label=(\arabic*)]
		\item there exists $w\in \sW_{\max}^P$ such that
		\begin{itemize}
		\item $w>w_y$ and $\lg(w)=\lg(w_y)+2$;
		\item $\dim \fz_{L_P}^{ww_y^{-1}}=\dim \fz_{L_P}-2$;
		\item the Bruhat interval $[w_y, w]=\{w'| w_y<w'<w\}$ is full in $\sW_{\max}^P$, i.e.\ there exist distinct $w_1$, $w_2\in \sW_{\max}^P$ such that $w_y< w_1, w_2<w$, $\lg(w_1)=\lg(w_2)=\lg(w_y)+1$ (e.g.\ see \cite[Lemma 5.2.7]{BHS3});
		\end{itemize}
		\item there exists $w\in \sW_{\max}^P$ such that
		\begin{itemize}
		\item $w>w_y$ and $\lg(w)=\lg(w_y)+2$;
                 \item the Bruhat interval $[w_y,w]$ is not full in $\sW_{\max}^P$.
                 \end{itemize}
	\end{enumerate}
	Then $[\cN_{w_yw_{0,F} \cdot \lambda, y}]\neq 0$ (which implies $x\in \cE_{\Omega, w_yw_{0,F}\cdot \lambda}^{\infty}(\overline{\rho})$), and $\cM_{\Omega, w_yw_{0,F}\cdot \lambda}^{\infty}$ is locally free of rank $m_y$ at the point $x$. 
\end{lemma}

Before proving Lemma \ref{induccp}, we first give an easy lemma on the multiplicities $b_{w,w'}$ in Theorem \ref{thmcycl} (2).

\begin{lemma}\label{mult00}
Let $w,w'\in \sW_{\max}^P$ such that $w>w'$.
	
(1) If $\lg(w')=\lg(w)-1$, then $b_{w,w'}=1$.
	
(2) If $\lg(w')=\lg(w)-2$, then
	$b_{w,w'}=\begin{cases}1 & \text{$[w',w]$ is full in $\sW^P_{\max}$}\\
		0 & \text{otherwise}.
	\end{cases}$
\end{lemma}
\begin{proof}
	We only prove (2), (1) following from similar (and easier) arguments. We only need to prove a similar statement for a single $\GL_n(E)$ (rather than $G_p$). We let hence $P\subset \GL_n$ be a parabolic subgroup, $w, w'\in \sW_{\max}^P\subset \sW$, and $b_{w,w'}$ be the multiplicity of $L(w'w_0 \cdot 0)$ in $M_P(ww_0\cdot 0)$. Recall first that $L(w'w_0 \cdot 0)$ has multiplicity one in the Verma module $M(ww_0 \cdot 0)$ (see for instance \cite[Ex.\ 8.3(a)]{Hum08}). By \cite[Thm.\ 9.4(b)]{Hum08}, we have an exact sequence
	\begin{equation}\label{vermReso}
		\oplus_{\alpha\in S_P} M(s_{\alpha} ww_0 \cdot 0) \ra M(ww_0 \cdot 0) \ra M_P(ww_0 \cdot 0) \ra 0
	\end{equation} 
	where $S_P$ denotes the set of simple roots of $L_P$ and $s_{\alpha}$ denotes the corresponding simple reflection. 
	Note that for each $\alpha\in S_P$ we have $s_{\alpha}w<w$ and $\lg(s_{\alpha}w)=\lg(w)-1$ (e.g.\ see the proof of Proposition \ref{paraVerGeo}). 
	
	We claim that there exists a unique $\alpha \in S_P$ such that $w'w_0>s_{\alpha} ww_0$ if and only if $[w',w]$ is not full in $\sW^P_{\max}$. Indeed, we know there exist two, and only two, elements $w_i\in \sW$ such that $w'<w_i<w$ (see \cite[Lemma 5.2.7]{BHS3}) and by \cite[Thm.\ 2.5.5]{BjBr}, at least one of $w_i$, say $w_1$, lies in $\sW^P_{\max}$ . If $w'w_0>s_{\alpha} ww_0$ for some $\alpha\in S_P$, we then deduce $w'<s_{\alpha}w<w$, hence $s_{\alpha}w\in [w',w]$. However, $s_{\alpha}w\notin \sW^P_{\max}$ so $[w',w]$ is not full in $\sW^P_{\max}$. If $[w',w]$ is not full in $\sW^P_{\max}$, there then exists a unique element $w''\notin \sW^P_{\max}$ such that $w'<w''<w$. This implies \[w'=w'^{\max}<w''<w''^{\max} \leq w^{\max}=w.\]
	 As $\lg(w)=\lg(w')+2$, this implies $w''^{\max}=w^{\max}=w$. Using $\lg(w)=\lg(w'')+1$, it follows that there exists $s_{\alpha}\in S_P$ such that $w''=s_{\alpha} w$. 
	
	So if $[w',w]$ is full, $L(w'w_0 \cdot 0)$ does not appear in $M(s_{\alpha}ww_0 \cdot 0)$ for all $s_{\alpha}\in S_P$. By (\ref{vermReso}) and the discussion above it, we obtain $b_{w,w'}=1$. If $[w',w]$ is not full, let $\alpha\in S_P$ such that $w'<s_{\alpha} w< w$. Then $L(w'w_0\cdot 0)$ has multiplicity one in $M(s_{\alpha} ww_0 \cdot 0)$ (and zero in $M(s_{\alpha'} ww_0 \cdot 0)$ for $\alpha'\in S_P$, $\alpha'\neq \alpha$). As $M(s_{\alpha} ww_0 \cdot 0) \hookrightarrow M(ww_0 \cdot 0)$, we deduce by counting the multiplicities of $L(w'w_0 \cdot 0)$ in the first two terms of (\ref{vermReso}) that $b_{w,w'}=0$. 
\end{proof}
\begin{proof}[Proof of Lemma \ref{induccp}]
	Assume condition (1) holds. The proof is the same as Step 10 of the proof of \cite[Thm.\ 5.3.3]{BHS3}. As $w, w_1, w_2>w_y$, by assumption, $x \in \cE_{\Omega, w_iw_{0,F} \cdot \lambda}^{\infty}(\overline{\rho})$ for $i\in \{1,2, \emptyset\}$. Moreover, by Corollary \ref{corR=T0} and Corollary \ref{smDefVar2}, for $i\in \{1,2, \emptyset\}$, $\cE_{\Omega, w_iw_{0,F} \cdot \lambda}^{\infty}(\overline{\rho})$ is smooth at $x$. By Corollary \ref{PBEdim} (3), $\cM_{\Omega, w_iw_{0,F} \cdot \lambda}^{\infty}$ is locally free of rank $m_y$ at $x$. Similarly as for (\ref{equcycl000}), using Corollary \ref{cycGaA} (2) (applied to the completion of $\cE_{\Omega, w_iw_{0,F} \cdot \lambda}^{\infty}(\overline{\rho})_0$ at $x$ and to the completion of $\cE_{\Omega, ww_{0,F} \cdot \lambda}^{\infty}(\overline{\rho})_0$ at $x$), (\ref{equcyc12}) and Lemma \ref{mult00}, we obtain equations of cycles in $Z^{[F^+:\Q]\frac{n(n+1)}{2}}(\widehat{\co}_{\fX_{\infty}, y})$ for $i=1,2$:
	\begin{equation}\label{equcyc001}
		[\cN_{w_y w_{0,F} \cdot \lambda, y}]+[\cN_{w_iw_{0,F} \cdot \lambda, y}]=m_y\fC_{w_y, y}+m_y \fC_{w_i,y},
	\end{equation}
	\begin{equation}\label{equcyc002}	[\cN_{w_yw_{0,F} \cdot \lambda, y}]+[\cN_{w_1w_{0,F} \cdot \lambda, y}]+[\cN_{w_2w_{0,F} \cdot \lambda, y}]+[\cN_{ww_{0,F} \cdot \lambda, y}]= m_y (\fC_{w_y,y}+\fC_{w_1,y}+\fC_{w_2,y}+\fC_{w,y}).
	\end{equation}
	If $[\cN_{w_y w_{0,F} \cdot \lambda, y}] \neq 0$, then $x$ is a smooth point of $\cE_{\Omega, w_yw_{0,F}\cdot \lambda}^{\infty}(\overline{\rho})$ (by Corollary \ref{corR=T0} and Corollary \ref{smDefVar2}) and $\cM_{\Omega, w_yw_{0,F} \cdot \lambda}^{\infty}$ is locally free at $x$ (by Corollary \ref{PBEdim} (3)). By Corollary \ref{cycGaA} (2) and (\ref{equcyc12}), it follows $[\cN_{w_yw_{0,F} \cdot \lambda, y}]=[\cM_{w_yw_{0,F} \cdot \lambda, y}]\in \Z \fC_{w_y,y}$. Note that $[\cN_{w_yw_{0,F}}]\in \Z \fC_{w_y,y}$ obviously holds if $[\cN_{w_y w_{0,F} \cdot \lambda, y}] = 0$. In summary, there exists $m_y'\in \Z_{\geq 0}$ such that $[\cN_{w_y w_{0,F} \cdot \lambda, y}]=m_y' \fC_{w_y,y}$. We then deduce from (\ref{equcyc001}): $[\cN_{w_iw_{0,F} \cdot \lambda, y}]=m_y \fC_{w_i,y}+(m_y-m_y') \fC_{w_y, y}$. By Theorem \ref{thmcycl} (3) and Remark \ref{remsmo1} (1), for $w', w''\in \sW_{\max}^P$, $a_{w',w''}=0$ if $w'\neq w''$ and $\lg(w'')\geq \lg(w')-2$. Combined with (\ref{cycGaldef3}) and using Lemma \ref{irrcyc2}, it follows that, for $w'\in \{w_y, w_1, w_2, w\}$, we have $\fC_{w',y}=\fZ_{w',y}\neq 0$, and thus $\fC_{w',y}$ is irreducible. The equations (\ref{equcyc001}) and (\ref{equcyc002}) then imply:
	\begin{eqnarray*}	
		[\cN_{w_iw_{0,F} \cdot \lambda, y}]&=&m_y \fZ_{w_i, y}+(m_y'-m_y) \fZ_{w_y, y},\\
	\	[\cN_{ww_{0,F} \cdot \lambda, y}]&=&m_y \fZ_{w,y}+(m_y-m_y') \fZ_{w_y,y}.
\end{eqnarray*}
	As $\fZ_{w_y, y}$ can only have non-negative coefficients in $[\cN_{ww_{0,F} \cdot \lambda, y}]$ and $[\cN_{w_iw_{0,F} \cdot \lambda, y}]$, we must have $m_y'=m_y$. The proposition in this case follows.
	
	Assume condition (2) holds. Let $w_1$ be the unique element in $\sW^P_{\max}$ such that $w_y<w_1<w$. Similarly as above, we have that $\fC_{w',y}=\fZ_{w',y}\neq 0$ (and $\fC_{w',y}$ is irreducible) for $w'\in \{w_y, w_1, w\}$, that $\cE_{\Omega, w_1w_{0,F}\cdot \lambda}^{\infty}(\overline{\rho})$ is smooth at the point $x$, and that $\cM^{\infty}_{\Omega, w_1w_{0,F}\cdot \lambda}$ is locally free of rank $m_y$ at $x$. We have as in (\ref{equcyc001}):
	\begin{equation*}
		[\cN_{w_y w_{0,F} \cdot \lambda, y}]+[\cN_{w_1w_{0,F} \cdot \lambda, y}]=m_y(\fC_{w_y, y}+\fC_{w_1,y})=m_y( \fZ_{w_y,y}+\fZ_{w_1,y}).
	\end{equation*}
	By the same argument as in the last paragraph, there exists $m_y'\in \Z_{\geq 0}$ such that $[\cN_{w_y w_{0,F} \cdot \lambda, y}]=m_y'\fZ_{w_y,y}$ and $[\cN_{w_1w_{0,F} \cdot \lambda, y}]=m_y\fZ_{w_1,y}+(m_y-m_y')\fZ_{w_y,y}$ (so $m_y\geq m_y'$). By Corollary \ref{cycGaA} (1) (applied with ``$\lambda=ww_{0,F} \cdot \lambda$'' and $\fd_x=0$), we have 
	\begin{equation}\label{cyca}
	[\Spec \widehat{\co}_{\cE_{\Omega, ww_{0,F}\cdot\lambda}^{\infty}(\overline{\rho})_0,x}]=\sum_{\substack{w'\in \sW_{L_P}\backslash \sW_F \\ w_y \leq w'^{\max} \leq w}} \!\!\!\! c_{w, w'}' \fZ_{w',F} \in Z^{[F^+:\Q_p]\frac{n(n+1)}{2}} (\Spec \widehat{\co}_{\fX_{\infty},y}).
	\end{equation}
By the discussion on $a_{w',w''}$ in the last paragraph, we can deduce from Corollary \ref{lemcyccomp} (1), (2) that $c_{w,w'}=b_{w,w'}$ for $w'$ as in (the sum of) (\ref{cyca}). Together with Lemma \ref{mult00} (2), we see that $c_{w,w_y}=0$, hence that $c_{w,w_y}'=0$ by Lemma \ref{lemcyccomp} (2). We can thus refine (\ref{cyca}) as
	\begin{equation}\label{cycb}
[\Spec \widehat{\co}_{\cE_{\Omega, ww_{0,F}\cdot\lambda}^{\infty}(\overline{\rho})_0,x}]=\fZ_{w, F}+c_{w,w_1}'\fZ_{w_1,F} \in Z^{[F^+:\Q_p]\frac{n(n+1)}{2}} (\Spec \widehat{\co}_{\fX_{\infty},y}).
\end{equation}
	 As the $\widehat{\co}_{\fX_{\infty},y}$-action on $\cM_{ww_{0,F}\cdot \lambda,y}$ factors through $\widehat{\co}_{\cE_{\Omega, ww_{0,F}\cdot\lambda}^{\infty}(\overline{\rho})_0,x}$, using (\ref{cycb}) it follows that
	\begin{equation*}
	[\cM_{ww_{0,F}\cdot \lambda,y}]=a_0 \fZ_{w,F}+a_1 \fZ_{w_1,F}
	\end{equation*}
	for some $a_0$, $a_1\in \Z_{\geq 0}$. Using (\ref{equcyc12}) and Lemma \ref{lemcyccomp} (2), we get 
 \[ [\cN_{w_1w_{0,F} \cdot \lambda,y}]+[\cN_{ww_{0,F} \cdot \lambda, y}]=[\cM_{ww_{0,F} \cdot \lambda, y}]=a_0 \fZ_{w,F}+a_1 \fZ_{w_1,F},\]
 hence
\begin{equation}\label{cycc}
[\cN_{ww_{0,F} \cdot \lambda, y}]+m_y \fZ_{w_1,y}+(m_y-m_y') \fZ_{w_y,y}=a_0 \fZ_{w,F}+a_1 \fZ_{w_1,F}.
\end{equation}
 As $\fZ_{w,F}$, $\fZ_{w_1,F}$ are distinct from $\fZ_{w_y,y}$, and as $\fZ_{w_y,y}$ can only have non-negative coefficients in $[\cN_{ww_{0,F} \cdot \lambda, y}]$, (\ref{cycc}) implies (using $m_y\geq m_y'$) $m_y-m_y'=0$, hence $[\cN_{w_y w_{0,F} \cdot \lambda, y}]=m_y\fZ_{w_y,y}$. This concludes the proof. 
\end{proof}

\begin{remark}\label{reobst}
If none of the conditions (1) or (2) in Lemma \ref{induccp} hold, then for any $w\in \sW_{\max}^P$ satisfying $w>w_y$ and $\lg(w)=\lg(w_y)+2$, we have that $[w_y,w]$ is full in $\sW_{\max}^P$ and $\dim \fz_{L_P}^{ww_y^{-1}}>\dim \fz_{L_P}-2$. We use the notation in the proof of Lemma \ref{induccp} assuming condition (1). The main difference now is that we don't know if $\cE_{\Omega, w w_{0,F} \cdot \lambda}^{\infty}(\overline{\rho})$ is smooth at $x$ (see Remark \ref{remnonsm}). Consequently (by Corollary \ref{cycGaA} (1) and Lemma \ref{mult00}, and using similar arguments as in the proof of Lemma \ref{induccp} in the case of condition (2)), the equation (\ref{equcyc002}) has to be replaced by an equation of the form (the other equations being unchanged):
\begin{multline}\label{cycrpl}
[\cN_{w_yw_{0,F} \cdot \lambda, y}]+[\cN_{w_1w_{0,F} \cdot \lambda, y}]+[\cN_{w_2w_{0,F} \cdot \lambda, y}]+[\cN_{ww_{0,F} \cdot \lambda, y}]\\
=a_0\fZ_{w,y}+a_1\fZ_{w_1,y}+a_2\fZ_{w_2,y}+ a_3\fZ_{w_y,y}
\end{multline}
with $a_i\in \Z_{\geq 0}$. Unfortunately, we do not have more control on these coefficients $a_i$. The equations (\ref{equcyc001}) and (\ref{cycrpl}) seem not enough to deduce $[\cN_{w_yw_{0,F} \cdot \lambda, y}]=m_y \fZ_{w_y,y}$ or even $[\cN_{w_yw_{0,F} \cdot \lambda, y}]\neq 0$. Note that, though we also don't have much control on the coefficients $a_0$, $a_1$ in the proof of Lemma \ref{induccp} when condition (2) holds, the argument here can work as there are fewer terms.
\end{remark}

\begin{proposition}\label{corocomp0}
	Let $x=(y,\ul{x},1)\in \sS$ and $w\in \sW_{\min}^P$ such that $ww_{0,F}\geq w_y$. Assume that, for all $w'\in \sW_{\min}^P$ such that $w'w_{0,F}\geq ww_{0,F}$, one of the following properties holds
	\begin{itemize}
		 \item[(1)] $\lg(w'w_{0,F})\geq \lg(w_{0,F})-1$;
		 \item[(2)] one of the conditions (1), (2) in Lemma \ref{induccp} holds with $w_y$ replaced by $w'w_{0,F}$.
\end{itemize}
Then $[\cN_{w\cdot \lambda, y}]\neq 0$ (which implies $x\in \cE_{\Omega, w\cdot \lambda}^{\infty}(\overline{\rho})$), and $\cM_{\Omega, w\cdot \lambda}^{\infty}$ is locally free of rank $m_y$ in the smooth locus of a sufficiently small neighbourhood of the point $x$. 
\end{proposition}
\begin{proof}
By Lemma \ref{indstLem1}, we are reduced to the case $x\in \sS_w$ (i.e.\ $w_y=ww_{0,F}$). The case where $\lg(w w_{0,F})\geq \lg(w_{0,F})-1$ already follows from Lemma \ref{lgleq1}. Assume $\lg(w w_{0,F})\leq \lg(w_{0,F})-2$, by Lemma \ref{induccp} and the assumption in the proposition, we only need to prove
\begin{eqnarray}
&&\text{$[\cN_{w'\cdot \lambda, y}]\neq 0$ and $\cM_{\Omega, w'\cdot \lambda}^{\infty}$ is locally free of rank $m_y$ in the}\nonumber \\ &&\text{smooth locus of a sufficiently small neighbourhood of the point $x$}\label{exw'}
\end{eqnarray}
for all $w'\in \sW_{\min}^P$ with $w'w_{0,F}>ww_{0,F}$. To prove (\ref{exw'}), again by Lemma \ref{indstLem1}, we only need to show (\ref{exw'}) for $x\in \sS_{w'}$. By an obvious induction, we are finally reduced to the case $\lg(ww_{0,F})\geq \lg(w_{0,F})-1$, which follows from Lemma \ref{lgleq1}.
\end{proof}
	
\begin{corollary}\label{thmComCon}
Let $x=(y,\ul{x},1)\in \sS$. Assume that, for any $v\in S_p$, any two factors $\GL_{n_{\widetilde{v},i}}$ in $L_{P_{\widetilde{v}}}$ with $n_{\widetilde{v},i}>1$ (if they exist) are not adjacent in the product $L_{P_{\widetilde{v}}}\cong \prod_{i=1}^{r_{\widetilde{v}}} \GL_{n_{\widetilde{v},i}}$. Then $[\cN_{w \cdot \lambda, y}]\neq 0$ for $w \in \sW^P_{\min}$ if and only if $ww_{0,F} \geq w_y$.
\end{corollary}
\begin{proof}
The ``only if" part follows from Lemma \ref{lemcycnovan0} (which actually holds without the assumption on the adjacent $\GL_{n_{\widetilde{v},i}}$ in the statement). By Proposition \ref{bruhInt} in the appendix, under the assumption in the statement, the condition in Corollary \ref{corocomp0} holds for all $w\in \sW_{\min}^P$. The ``if" part then follows by Corollary \ref{corocomp0}.
\end{proof}

\begin{remark}
Without the assumption on the adjacent $\GL_{n_{\widetilde{v},i}}$ in Corollary \ref{thmComCon} (for example, when $P_{\widetilde{v}}$ is maximal with $L_{P_{\widetilde{v}}}\cong \GL_{n_{\widetilde{v},1}} \times \GL_{n_{\widetilde{v},2}}$ and $n_{\widetilde{v},1}, n_{\widetilde{v},2}>1$), there may exist $w_y$ with $\lg(w_y)\leq \lg(w_{0,F})-2$ such that none of the conditions (1) or (2) in Lemma \ref{induccp} hold (see Remark \ref{bruhInv}). As discussed in Remark \ref{reobst}, in this case we don't know how to deduce $[\cN_{w_yw_{0,F} \cdot \lambda, y}]\neq 0$ from $[\cN_{w\cdot \lambda,y}]\neq 0$ for $w\in \sW_{\min}^P$ such that $ww_{0,F}>w_y$. Using the method of \cite{Wu21}, it may be possible to obtain some cases where the $P_{\widetilde{v}}$ don't satisfy the assumption in Corollary \ref{thmComCon}, but it seems that the maximal parabolic case mentioned above is still resisting.
\end{remark}

\begin{theorem}\label{casparticulier}
Assume we are in the setting of Conjecture \ref{conjSoc}. Assume moreover:
	\begin{enumerate}[label=(\arabic*)]
	\item Hypothesis \ref{TayWil};
	\item $\widehat{S}(U^p,E)^{\lalg}[\fm_{\rho}]\ne 0$;
	\item for all $v\in S_p$, the condition in Corollary \ref{thmComCon} holds for the parabolic subgroup $P_{\widetilde{v}}$ associated to $\sF_{\widetilde{v}}$.
	\end{enumerate}
	Then Conjecture \ref{conjSoc} holds.
\end{theorem}
\begin{proof}
We use the notation of \S~\ref{seccompCP}. Note that, by local-global compatibility in the classical local Langlands correspondence (\cite{Car12}), the assumption (2) in the statement is equivalent to the existence of an embedding $\widehat{\otimes}_{v\in S_p} C(1, \sF_{\widetilde{v}})\hookrightarrow \widehat{S}(U^p,E)^{\lalg}[\fm_{\rho}]$. We let $\Omega:=\Omega_{\sF}$ and $\ul{x}\in (\Spec \cZ_{\Omega})^{\rig}$ be the point associated to $\pi_{L_P}$ ($\pi_{L_P}$ defined as in \S~\ref{seccompCP}). By the assumption (2) and the second sentence, we have $z=(\eta_{\rho}, \ul{x},1)\in \cE_{\Omega, \lambda}(U^p,\overline{\rho})$. By Proposition \ref{bePbe}, $z$ corresponds to a point $x=(y=(x^p,x_p), \ul{x}, 1) \in \cE_{\Omega, \lambda}^{\infty}(\overline{\rho})$. We have $\Pi_{\infty}^{R_{\infty}-\an}[\fm_y]\cong \widehat{S}(U^p,E)^{\an}[\fm_{\rho}]$. By similar arguments as above \cite[Lemma 4.6]{BHS1}, $x^p$ is a smooth point of $\fX_{\infty}^p$. By Corollary \ref{thmComCon} and Lemma \ref{copconcyc}, it follows that there is an embedding
	\begin{equation*}
		\widehat{\otimes}_{v\in S_p} C(w_{\widetilde{v}}, \sF_{\widetilde{v}}) \hooklongrightarrow \Pi_{\infty}^{R_{\infty}-\an}[\fm_y]\cong \widehat{S}(U^p,E)^{\an}[\fm_{\rho}]
	\end{equation*}
	if and only if $w_{\widetilde{v}}\leq w_{\sF_{\widetilde{v}}}w_{0, F_{\widetilde{v}}}$ for all $v\in S_p$ (noting that $(w_{\sF_{\widetilde{v}}})=w_y$).
\end{proof}

\begin{remark}
(1) By (the proof of) Theorem \ref{class00}, the assumption (2) in Theorem \ref{casparticulier} can also be replaced by the assumptions (2), (3), (4) in Theorem \ref{class00}, the assumption (4) being for $P=\prod_{v\in S_p} P_{\widetilde{v}}$ associated to $\{\sF_{\widetilde{v}}\}$.
	
(2) By the same argument but using Corollary \ref{corocomp0} instead of Corollary \ref{thmComCon}, we can prove 
	\begin{equation*}
		\widehat{\otimes}_{v\in S_p} C(w_{\widetilde{v}}, \sF_{\widetilde{v}}) \hooklongrightarrow \Pi_{\infty}^{R_{\infty}-\an}[\fm_y]\cong \widehat{S}(U^p,E)^{\an}[\fm_{\rho}]
	\end{equation*}
if $w_{\widetilde{v}}\leq w_{\sF_{\widetilde{v}}}w_{0, F_{\widetilde{v}}}$ for all $v\in S_p$ and $w=(w_{\widetilde{v}})$ satisfies the conditions in Corollary \ref{corocomp0}. 
\end{remark}

In particular, we have the following special case of Theorem \ref{casparticulier}:

\begin{corollary}
Assume that we are in the setting of Conjecture \ref{conjSoc}, and that:
	\begin{enumerate}[label=(\arabic*)]
	\item Hypothesis \ref{TayWil} holds;
	\item $\widehat{S}(U^p,E)^{\lalg}[\fm_{\rho}] \neq 0$;
	\item for all $v\in S_p$, each $\ttr(\rho_{\widetilde{v}})$ has at most one irreducible constituent of dimension $>1$.
	\end{enumerate}
For $v\in S_p$, let $\sF_{\widetilde{v}}$ be a minimal parabolic filtration of $\ttr(\rho_{\widetilde{v}})$. Then Conjecture \ref{conjSoc} holds for $\sF=(\sF_{\widetilde{v}})$.
\end{corollary}

\begin{remark}
The case where all irreducible constituents of $\ttr(\rho_{\widetilde{v}})$ are $1$-dimensional (i.e.\ the (generic) crystabelline case) was proved in \cite[Thm.\ 5.3.3]{BHS3}.	
\end{remark}

\appendix

\renewcommand*{\thesection}{\Alph{section}}

\section{Appendix}

\subsection{$\Omega$-filtration in families}\label{globaltriangulation}

Generalizing \cite[\S~6.3]{KPX} we show that $\Omega$-filtrations on $(\varphi, \Gamma)$-modules (cf.\ \S~\ref{secDefOD}) can interpolate in $p$-adic families. The results in this section may be viewed as a parabolic analogue of the (so-called) global triangulations. 

The following theorem generalizes \cite[Thm.\ 6.3.9]{KPX}.

\begin{theorem}\label{KPXg}
Let $n\geq r\in \Z_{\geq 1}$. Let $X$ be a reduced rigid analytic space over $E$, $M$ a $(\varphi,\Gamma)$-module over $\cR_{X,L}$ of rank $n$, and $\Delta$ a $(\varphi,\Gamma)$-module over $\cR_{X,L}$ of rank $r$ such that, for any point $x$ of $X$, $\Delta_x$ is irreducible and de Rham of constant Hodge-Tate weight $0$. Let $\textbf{h}=(\textbf{h}_{\tau})_{\tau\in \Sigma_L}=(h_{i,\tau})_{\substack{i=1,\dots, r\\ \tau\in \Sigma_L}}$ be an integral positive strictly dominant weight (i.e.\ $h_{1,\tau}> h_{2,\tau}>\cdots > h_{r,\tau}\geq 0$ for all $\tau \in \Sigma_L$). Assume that there exists a Zariski-dense subset $X_{\alg}$ of points of $X$ such that for all $x\in X_{\alg}$:
\begin{enumerate}
\item[(1)] $H^0_{(\varphi,\Gamma)}(\Delta_x \otimes_{\cR_{k(x),L}} M_x^{\vee})$ is one-dimensional over $k(x)$;
\item[(2)] $\Ima \eta_x$ \ is \ de \ Rham \ of \ Hodge-Tate \ weights \ $\textbf{h}$ \ for \ $0\neq \!\eta_x \in \Hom_{(\varphi, \Gamma)}(M_x, \Delta_x)\cong H^0_{(\varphi,\Gamma)}(\Delta_x \otimes_{\cR_{k(x),L}} M_x^{\vee})$.
\end{enumerate}
Then there exist
\begin{itemize}
\item a proper birational surjective morphism $f: X' \ra X$;
\item a unique homomorphism $\eta: f^* M \ra f^* \Delta$ (up to multiplication by $\co_{X'}^{\times}$)
\end{itemize}
such that the following properties are satisfied:
\begin{enumerate}[label=(\arabic*)]
\item the set $Z$ of closed points $x\in X'$ failing to have properties (a) and (b) below is Zariski-closed and disjoint from $f^{-1}(X_{\alg})$:
\begin{enumerate}
\item $\Hom_{(\varphi,\Gamma)}(M_x, \Delta_x)$ is one-dimensional and generated by the pull-back $\eta_x: M_x \ra \Delta_x$ of $\eta$;
\item $\Ima(\eta_x)$ is de Rham of Hodge-Tate weights $\textbf{h}$;
\end{enumerate}
\item the kernel of $\eta$ is a $(\varphi, \Gamma)$-module over $\cR_{X',L}$ of rank $n-r$;
\item the cokernel of $\eta$ is locally $t$-torsion;
\item $\Ima \eta |_{X' \setminus Z}$ is a $(\varphi,\Gamma)$-module of rank $r$, and for all $x\in X' \setminus Z$, the induced morphism $(\Ima \eta)_x \ra \Delta_{f(x)}$ is injective.
\end{enumerate}
\end{theorem}
\begin{proof}
The theorem follows by a variation of the proof of \cite[Thm.\ 6.3.9]{KPX}. We include a proof for the convenience of the reader. One main difference from \textit{loc.\ cit.} is that in our case, for any $x\in X_{\alg}$, $\Ima(\eta_x)$ is not saturated in $\Delta_x$ except when $r=1$. As in the first paragraph of \emph{loc.\ cit.}, we are reduced to the case where $X$ is reduced, normal and connected (thus any coherent sheaf over $X$ has constant generic rank).
	
Step 1. By \cite[Thm.\ 4.4.5 (1)]{KPX} and \cite[Thm.]{KPX}, the complex $\cC_{(\varphi,\Gamma)}^{\bullet}(\Delta \otimes_{\cR_{X',L}} M^{\vee})$ is perfect and concentrated in degree $[0,2]$, i.e.\ is quasi-isomorphic to $[P_0 \ra P_1 \ra P_2]$ where the $P_i$ are locally free $\co_X$-modules of finite ranks. By \cite[Cor.\ 6.3.6 (2)]{KPX}, we can obtain a proper birational morphism $f_0: X_0' \rightarrow X$ with $X_0'$ reduced and normal (and irreducible) such that
\begin{itemize}
\item $H^0_{(\varphi,\Gamma)}(f_0^* (\Delta \otimes_{\cR_{X,L}} M^{\vee}))$ is flat;
\item $H^i_{(\varphi,\Gamma)}(f_0^* (\Delta \otimes_{\cR_{X,L}} M^{\vee}))$ has Tor-dimension less than or equal to $1$ for $i=1,2$.
\end{itemize}
Note that $f_0^{-1}(X_{\alg})$ is Zariski-dense in $X'_0$. Using the condition (1) and the above flatness, we see that $\cL:=H^0_{(\varphi,\Gamma)}(f_0^* (\Delta \otimes_{\cR_{X,L}} M^{\vee}))$ is locally free of rank $1$ over $\co_{X'_0}$. We have thus a tautological morphism of $(\varphi,\Gamma)$-modules over $\cR_{X'_0,L}$:
\begin{equation}\label{bira1}
\eta: f_0^* M \otimes_{\co_{X'_0}} \cL \lra f_0^* \Delta.
\end{equation}
It is clear that, for any $x\in X'_0$, the induced morphism $\eta_x: f_0^*M_x \ra f_0^*\Delta_x$ is non-zero. We have actually an exact sequence as in \cite[(6.3.9.1)]{KPX}:
\begin{multline*}
0 \ra H^0_{(\varphi,\Gamma)}(f_0^* (\Delta \otimes_{\cR_{X',L}} M^{\vee})) \otimes_{\co_{X'}} k(x) \ra H^0_{(\varphi,\Gamma)}\big(f_0^* (\Delta \otimes_{\cR_{X',L}} M^{\vee})\otimes_{\co_{X'}} k(x)\big)\\
\ra \Tor_1^{X'_0}\big(H^1_{(\varphi,\Gamma)}(f_0^* (\Delta \otimes_{\cR_{X',L}} M^{\vee})),k(x)\big) \ra 0
\end{multline*}
We deduce that the set $Z_0$ of points $x\in X'_0$ such that $H^0_{(\varphi,\Gamma)}(f_0^* (\Delta \otimes_{\cR_{X',L}} M^{\vee})) \otimes_{\co_{X'}} k(x)$ is not isomorphic to $H^0_{(\varphi,\Gamma)}(f_0^* (\Delta \otimes_{\cR_{X',L}} M^{\vee})\otimes_{\co_{X'}} k(x))$ is Zariski-closed in $X'_0$ (with its complement given by the flat locus of the coherent sheaf $H^1_{(\varphi,\Gamma)}(f_0^* (\Delta \otimes_{\cR_{X',L}} M^{\vee}))$).
	
Step 2. Let $Q$ be the cokernel of $\eta$. For any point $x\in X'_0$, since $\Delta_x$ is irreducible and $\eta_x$ is non-zero, we deduce that $Q_x=\Coker \eta_x$ is $t$-torsion. Let $\Spm A$ be an affinoid open of $X'$. By \cite[Lemma 2.2.9]{KPX}, the base change of the morphism $\eta$ via $\Spm A \ra \cR_{X',L}$ admits a model $\epsilon_A^r: M_A^r \ra \Delta_A^r$ and the cokernel $Q_A^r$ is a model of $Q_A$ (the restriction of $Q$ to $\cR_{A,L}$) (where we use the notation of \cite[\S~2]{KPX} for $M_A^r$ etc.). We have that $Q_A^{(r/p, r]}$ is a finitely presented module over $\cR_{A,L}^{(r/p,r]}$. Since $Q_x=\Coker \eta_x$ is $t$-torsion for all $x\in X'$, we see $Q_A^{(r/p,r]}$ is supported on the zero locus of $t$ in $\Spm \cR_{A,L}^{(r/p,r]}$, which is a rigid analytic space finite over $A$. We deduce that $Q_A^{(r/p,r]}$ is annihilated by $t^N$ for some $N\in \Z_{>0}$. Since $Q_A^{(r/p^m, r/p^{m-1}]} \cong (\varphi^{m-1})^* Q_A^{(r/p,r]}$, we deduce that $Q_A^r\cong \prod_{m \in \Z_{\geq 1}} Q_A^{(r/p^m, r/p^{m-1}]}$ is also annihilated by $t^N$ (note that $Q_A^r$ is coadmissible by \cite[Lemma 2.1.4 (5)]{KPX}). Applying \cite[Cor.\ 6.3.6]{KPX} to the finitely generated $A$-module $Q_A^{(r/p,r]}$, we obtain $g: \Spm A' \ra \Spm A$ such that $g^* Q_A^{(r/p,r]}$ has Tor-dimension less than $1$. Using the $\varphi$-action as above, we deduce that $g^*Q_A^{(r/p^m, r/p^{m-1}]}$ also has Tor-dimension less than $1$, and hence that $Q_{A'}^r:=g^* Q_A^r\cong \prod_{m \in \Z_{\geq 1}} g^*Q_A^{(r/p^m, r/p^{m-1}]}$ has Tor-dimension less than $1$ as $A'$-module (using the fact that a direct product of flat $A'$-modules is flat, as $A'$ is noetherian). The morphisms $g: A' \ra A$ (with $A$ varying) glue to a birational projective morphism $g: X' \ra X'_0$ such that $Q_{X'}:=g^* Q$ has Tor-dimension less than $1$, and we finally obtain $f: X' \xrightarrow{g} X'_0 \ra X$. By the above discussion, $Q_{X'}$ is locally $t$-torsion. 
	
Step 3. We have an exact sequence (by pulling-back (\ref{bira1}) via $g$, and where we use $\cL$ to denote the pull-back of $\cL$ in (\ref{bira1})):
\begin{equation*}
f^* M \otimes_{\co_{X'}} \cL \xrightarrow{\eta} f^* \Delta \ra Q_{X'} \ra 0.
\end{equation*}
Specializing to a point $x\in X'$ and using the fact that $Q_{X'}$ has Tor-dimension less than $1$, we get 
\begin{equation}\label{spcia1}
0 \ra \Tor_1^{X'}(Q_{X'}, k(x)) \ra \Ima(\eta)_x \ra (f^*\Delta)_x \ra Q_{X',x} \ra 0,
\end{equation}
and $\Tor_i^{X'}(\Ima(\eta), k(x))=0$ for $i>0$. Then, specializing $0 \ra \Ker(\eta) \ra (f^* M \otimes_{\co_{X'}} \cL) \ra \Ima (\eta) \ra 0$ to $x$, we get
\begin{equation}\label{spcial2}
0 \ra \Ker (\eta)_x \ra (f^*M)_x \ra \Ima (\eta)_x \ra 0.
\end{equation}
Since $Q_{X'}$ is locally $t$-torsion, both $ \Tor_1^{X'}(Q_{X'}, k(x))$ and $Q_{X',x}$ are annihilated by a certain power of $t$. We deduce that $\Ima(\eta_x)$ is a $(\varphi, \Gamma)$-module of rank $r$ over $\cR_{k(x),L}$, and $\Ker(\eta_x)$ is a $(\varphi, \Gamma)$-module of rank $n-r$ over $\cR_{k(x),L}$ (recall that $\eta_x$ denotes the induced morphism $(f^*M)_x \ra (f^* \Delta)_x$). We also deduce from (\ref{spcia1}) and (\ref{spcial2}):
\begin{equation*}
0 \ra \Ker(\eta)_x \ra \Ker(\eta_x) \ra \Tor_1^{X'}(Q_{X'}, k(x))\ra 0.
\end{equation*}
Since $ \Tor_1^{X'}(Q_{X'}, k(x))$ is $t$-torsion, it follows that $\Ker(\eta)_x$ is a $(\varphi,\Gamma)$-module of rank $n-r$ for all $x\in X'$. By the same argument as in the last paragraph of the proof of \cite[Thm.\ 6.3.9]{KPX}, we deduce that $\Ker(\eta)$ is a $(\varphi,\Gamma)$-module of rank $n-r$ (in the sense of \cite[Def.\ 2.2.12]{KPX}). 
	
Step 4. We prove that $U:=\{x\in X'\ |\ \Tor_1^{X'}(Q_{X'},k(x)) = 0\}$ is Zariski-open and contains $f^{-1}(X_{\alg})$.
Let $\Spm A$ be an affinoid in $X'$ and $Q_A^r$ be a model of $Q_A:=Q_{X'}|_{\Spm A}$. For $x\in \Spm A$, by the same argument as in Step 2, we see that the following statements are equivalent:
\begin{itemize}
\item $Q_A^{(r/p,r]}$ is flat over $A$ at $x$;
\item $Q_A^{r}$ is flat over $A$ at $x$.
\end{itemize}
Since $Q_{X'}$ has Tor-dimension less than $1$ (by Step 2), $U \cap \Spm A$ is exactly the flat locus of $Q_A^{(r/p,r]}$. We deduce that $U$ is Zariski-open. Now assume $f^{-1}(X_{\alg}) \cap \Spm A$ is Zariski-dense in $\Spm A$ (noting that such $\Spm A$ can cover $X'$). For each $x\in f^{-1}(X_{\alg}) \cap \Spm A$, by the assumption (2) in the theorem, $\Ima(\eta_x)\subset \Delta_{f(x)}$ is de Rham of Hodge-Tate weights $\textbf{h}$. We deduce that $Q_A^{(r/p,r]} \otimes_A k(x)$ has constant dimension (determined by $\textbf{h}$). Indeed, using \cite[Thm.~A]{Ber08a}, we can deduce that $Q_A^r \otimes_{A} k(x)=\Delta_{f(x)}^r/\Ima(\eta_x)^r$ admits a filtration by $\cR_{k(x),L}^r$-submodules of graded pieces given by $\{\cR_{k(x),L}^r/\prod_{\tau\in \Sigma_L} t_{\tau}^{h_{i,\tau}-h_{i+1,\tau}}\big\}_{i=1,\dots, r}$ where $t_{\tau}\in \cR_{k(x),L}$ is the element defined in \cite[Not.\ 6.2.7]{KPX}, and $h_{r+1,\tau}:=0$ for all $\tau$. Since $\cR_{k(x),L}^r$ is flat over $\cR_{k(x),L}^{(r/p,r]}$, we see that $Q^{(r/p,r]} \otimes_A k(x)$ admits a filtration with graded pieces given by $\{\cR_{k(x),L}^{(r/p,r]}/\prod_{\tau\in \Sigma_L} t_{\tau}^{h_{i,\tau}-h_{i+1,\tau}}\big\}_{i=1,\dots, r}$. Hence $\dim_{k(x)} Q+A^{(r/p,r]} \otimes_A k(x)$ is constant for $x\in f^{-1}(X_{\alg}) \cap \Spm A$. 
Together with the fact $f^{-1}(X_{\alg}) \cap \Spm A$ is Zariski-dense in $\Spm A$, we deduce that $Q_A^{(r/p,r]}$ is locally free at points in $f^{-1}(X_{\alg}) \cap \Spm A$. In particular, $f^{-1}(X_{\alg}) \cap \Spm A \subseteq U \cap \Spm A$.
	
Step 5. Consider the restriction $D$ of $\Ima \eta$ on $U$ (see Step 4). Specializing the exact sequence
\begin{equation*}
0 \ra D \ra (f^* \Delta)_U \ra Q_U\ra 0
\end{equation*}
at each point $x\in U$, we get 
\begin{equation*}
0 \ra D_x \ra (f^*\Delta)_x \ra Q_x\ra 0.
\end{equation*}
In particular, $D_x$ is a $(\varphi, \Gamma)$-module of rank $r$. By the same argument as in the last paragraph of the proof of \cite[Thm.\ 6.3.9]{KPX}, we deduce that $D$ is a $(\varphi, \Gamma)$-module (\cite[Def.\ 2.2.12]{KPX}). For $x\in U$, since $(f^*\Delta)_x$ is de Rham, $D_x$ is de Rham as well. Since $f^{-1}(X_{\alg})$ is Zariski-dense in $U$, by interpolating the Sen weights (see for example \cite[Def.\ 6.2.11]{KPX}), $D_x$ has Sen weights $\textbf{h}$. In summary, $D_x$ is de Rham of Hodge-Tate weights $\textbf{h}$ for all $x\in U$. The theorem follows by taking $Z$ to be the union of the complement of $U$ and $g^{-1}(Z_0)$ (see in particular Step 3 for the properties (2) and (3), Step 1 for (1)(a), and Step 4 \& 5 for (1)(b) and (4)).
\end{proof}

Note that in Theorem \ref{KPXg} $f^{-1}(X_{\alg})$ is Zariski-dense in $X'\setminus Z$.

\begin{corollary}\label{para}
Let $X$ be a reduced rigid analytic space over $E$. Let $n_i\in \Z_{\geq 1}$ for $1\leq i \leq r$ and $n:=\sum_{i=1}^{r} n_i$. Assume we have the following data:
\begin{enumerate}[label=(\arabic*)]
\item a $(\varphi, \Gamma)$-module $M$ over $\cR_{X,L}$ of rank $n$;
\item for each $i$, a $(\varphi, \Gamma)$-module $\Delta_i$ over $\cR_{X,L}$ of rank $n_i$ such that for all $x\in X$, $\Delta_{i,x}$ is irreducible and de Rham of constant Hodge-Tate weight $0$;
\item for each $i$, a continuous character $\delta_i$ of $L^{\times}$ over $X$;
\item for each $i$, an integral positive strictly dominant weight $\textbf{h}_i=(\textbf{h}_i)_{\tau \in \Sigma_L}=(h_{i,j,\tau})_{\substack{j=1, \dots, n_i \\ \tau \in \Sigma_L}}$ (i.e.\ $h_{i,1,\tau} > \cdots > h_{i,n_i,\tau}\geq 0$ for all $\tau \in \Sigma_L$);
\item a Zariski-dense subset $X_{\alg}$ of closed points of $X$ such that for all $x\in X_{\alg}$, $M_x$ admits a filtration $0 =\Fil_0 M_x \subsetneq \Fil_1 M_x \subsetneq \cdots \subsetneq \Fil_{r} M_x=M_x$ satisfying
\begin{itemize}
\item $\dim_{k(x)} \Hom_{(\varphi,\Gamma)}(\Fil_i M_x, \Delta_{i,x} \otimes_{\cR_{k(x),L}} \cR_{k(x),L}(\delta_{i,x}))=1$;
\item for any non-zero morphism of $(\varphi, \Gamma)$-modules $\eta_{i,x}: \Fil_i M_x \ra \Delta_x \otimes_{\cR_{k(x),L}} \cR_{k(x),L}(\delta_{i,x})$, $\Ima \eta_x \otimes_{\cR_{k(x),L}} \cR_{k(x),L}(\delta_{i,x}^{-1})$ is de Rham of Hodge-Tate weights $\textbf{h}_i$.
\end{itemize}
\end{enumerate}
Then there exist 
\begin{itemize}
\item a proper birational surjective morphism $f: X' \ra X$ of reduced rigid analytic spaces; 
\item a filtration $0=\Fil_0 (f^* M) \subsetneq \cdots \subsetneq \Fil_r (f^* M)=f^* M$ on $f^* M$ by $(\varphi, \Gamma)$-submodules over $\cR_{X',L}$;
\item line bundles $\cL_i$ over $X'$ for $1\leq i \leq r$;
\item unique morphisms $\eta_i:\Fil_i(f^* M) \otimes_{\co_{X'}} \cL_i \rightarrow f^*(\Delta_i \otimes_{\cR_{X,L}} \cR_{X,L}(\delta_i))$ for $1\leq i \leq r$;
\item a Zariski-closed subset $Z$ of $X$ disjoint from $X_{\alg}$
\end{itemize}
such that the following properties are satisfied
\begin{enumerate}[label=(\arabic*)]
\item for each $i$, the cokernel of $\eta_i$ is locally $t$-torsion and $\Ker \eta_i =\Fil_{i-1}(f^* M)\otimes_{\co_{X'}} \cL_i$;
\item the $(\Fil^i f^*M)|_{X'\setminus Z}$ are direct summands of $f^* M|_{X'\setminus Z}$ as $\cR_{X'\setminus Z,L}$-modules;
\item for all $x\in X'\setminus Z$, the $k(x)$-vector space
\[ \Hom_{(\varphi,\Gamma)}\big((\Fil_i f^* M)_x, \Delta_{i,f(x)} \otimes_{\cR_{k(x),L}} \cR_{k(x),L}(\delta_{i,f(x)})\big)\]
is one dimensional generated by the pull-back of $\eta_i$, and the $(\varphi, \Gamma)$-module $\Ima(\eta_{i,x}) \otimes_{\cR_{k(x),L}} \cR_{k(x),L}(\delta_{i,k(x)}^{-1})$ is de Rham of Hodge-Tate weights $\textbf{h}_i$ for $1\leq i\leq r$.
\end{enumerate}
\end{corollary}
\begin{proof}
The corollary follows from Theorem \ref{KPXg} by induction. We first apply Theorem \ref{KPXg} to $(M \otimes_{\cR_{X,L}} \cR_{X,L}(\delta_r^{-1}), \Delta_r, X)$ (= the data $(M,\Delta,X)$ in the theorem) to obtain the data:
\[\big(f_r: X_r' \ra X,\ Z_r\subset X_r',\ \eta_r: f_r^* (M \otimes_{\cR_{X,L}} \cR_{X,L}(\delta_r^{-1}))\ra f_r^* \Delta_r\big)\] corresponding respectively to $f$, $Z$, $\eta$ in Theorem \ref{KPXg}. Next we apply Theorem \ref{KPXg} to $(\Ker \eta_r \otimes_{\cR_{X_r',L}} \cR_{X_r',L}(\delta_r \delta_{r-1}^{-1}), f_r^* \Delta_{r-1}, X_r')$ to obtain $(f_{r-1}: X_{r-1}' \ra X_r'$, $Z_{r-1}\subset X_{r-1}'$, $\eta_{r-1})$. By iterating this argument, we finally get $f: X':=X_1' \xrightarrow{f_1} X_2' \ra \cdots \ra X_r' \xrightarrow{f_r} X$ and we put $Z$ to be the union of the pull-backs of the $Z_i$'s. The corollary follows.
\end{proof}

\begin{corollary}\label{rgloOF}
Keep the setting of Corollary \ref{para} and let $x\in X'$.

(1) We have that $(f^* M)_x \cong M_{f(x)}$ admits a filtration $0=\Fil_0 M_{f(x)} \subsetneq \Fil_1 M_{f(x)} \subsetneq \cdots \subsetneq \Fil_r M_{f(x)}=M_{f(x)}$ by saturated $(\varphi, \Gamma)$-submodules of $M_{f(x)}$ such that $\gr_i M_{f(x)}[\frac{1}{t}]\cong \Delta_{i,x} \otimes_{k(x)} \cR_{k(x),L}(\delta_{i,x})[\frac{1}{t}]$.

(2) Let $A \in \Art(E)$ and $\Spec A \ra X'$ be a morphism sending the only point of $\Spec A$ to $x$. Let $M_A$ denote the pullback of $f^* M$ along $\Spec A \ra X'$. Then $M_A:=M_A[\frac{1}{t}]$ admits a filtration $0=\Fil_0 M_A \subsetneq \Fil_1 M_A \subsetneq \cdots \subsetneq \Fil_r M_A=M_A$ by $(\varphi, \Gamma)$-modules over $\cR_{A,L}[\frac{1}{t}]$ such that $\gr_i M_A\cong \Delta_{i,A} \otimes_{\cR_{A,L}} \cR_{A,L}(\delta_{i,A})[\frac{1}{t}]$.
\end{corollary}
\begin{proof}
Specializing the objects in Corollary \ref{para} at the point $x$, we have $(\varphi, \Gamma)$-modules over $\cR_{k(x),L}$: 
 \[\{(\Fil_i f^* M)_x=(\Ker \eta_{i+1})_x\}_{i=1,\dots, r-1}, \ \{\Ker \eta_{i,x}\}_{i=1, \dots, r}, \ \{\Delta_{i,f(x)}\}_{i=1, \dots, r}\]
 where (using the fact that $\Coker \eta_{i,x}$ is $t$-torsion):
 \[\rk_{\cR_{k(x),L}}(\Fil_i f^* M)_x=\rk_{\cR_{k(x),L}}(\Ker \eta_{i+1})_x=\sum_{j=1}^i n_j\textrm{ and }\rk_{\cR_{k(x),L}} \Ker(\eta_{i,x})=\sum_{j=1}^{i-1} n_j.\]
We also have morphisms of $(\varphi, \Gamma)$-modules over $\cR_{k(x),L}$:
	\begin{equation*}
	\begin{cases}
	\Ker(\eta_i)_x \lra \Ker(\eta_{i,x})& i=1, \dots, r\\
	\Ker(\eta_{i,x}) \hookrightarrow (\Fil_i f^*M)_x \cong \Ker (\eta_{i+1})_x & i=1, \dots, r-1.
	\end{cases}
	\end{equation*}
	We have that $\Ker(\eta_{i,x})$ is saturated in $(\Fil_if^*M)_x$ for all $i$. Since $\Coker \eta_i$ is locally $t$-torsion, by similar arguments as in Step 3 of the proof of Theorem \ref{KPXg}, we see that the cokernel of the morphism $\Ker(\eta_i)_x \ra \Ker(\eta_{i,x})$ is $t$-torsion, and hence this morphism is injective (as both source and target have the same rank over $\cR_{k(x),L}$). We take thus $\Fil_{r} M_{f(x)}:=M_{f(x)}$ and let $\Fil_i M_{f(x)}$ be the saturation of the image of the composition (note that some of the injections may not be saturated, for example the first one):
	\begin{multline*}
	(\Fil_{i} f^*M)_x\cong \Ker(\eta_{i+1})_x \hookrightarrow \Ker(\eta_{i+1,x}) \hookrightarrow (\Fil_{i+1}f^*M)_x \\ \hookrightarrow \Ker(\eta_{i+2,x}) \cdots \hookrightarrow \Ker(\eta_{r,x}) \hookrightarrow M_{f(x)}.
	\end{multline*} 
	Since $\Ker(\eta_{i+1})_x/\Ker(\eta_{i,x})=(\Fil_i f^* M)_x/\Ker(\eta_{i,x}) \hookrightarrow \Delta_{i,x} \otimes_{\cR_{k(x),L}} \cR_{k(x),L}(\delta_{i,x})$ with $t$-torsion cokernel (both source and target having the same rank), we deduce
	\[\gr_i M_{f(x)}\Big[\frac{1}{t}\Big]\cong \Ker(\eta_{i+1})_x\Big[\frac{1}{t}\Big]/\Ker(\eta_{i,x})\Big[\frac{1}{t}\Big]\cong \Delta_{i,x} \otimes_{\cR_{k(x),L}} \cR_{k(x),L}(\delta_{i,x})\Big[\frac{1}{t}\Big].\]
	Part (1) follows. Part (2) follows by similar arguments.
\end{proof}

\subsection{Characteristic cycles of parabolic Verma modules}\label{appenB}

We prove Proposition \ref{equcycl} by a generalization of Ginzburg's method (\cite[\S~6.3]{Gin86}).

We let $B\subset P \subset G$ be as in \S~\ref{secGS}, and we use without comment the notation there. For a smooth algebraic variety $X$ over $E$, we denote by $D_X$ the sheaf of differential operators on $X$, see for instance \cite[\S~1.1]{HTT}. Recall that $\co_X$ is equipped with a canonical left $D_X$-module structure. We will use below the notation of \cite{HTT}.

 For $w\in \sW$, we define $C_w:=BwB/B$ and $C_w^P:=PwB/B$, which are locally closed (smooth) subschemes of $G/B$, and $C_{P,w}:=BwP/P$, which is a locally closed (smooth) subscheme of $G/P$. So $C_w^P$ (resp.\ $C_{P,w}$) only depends on the image of $\sW$ in $\sW_{L_P}\backslash \sW$ (resp.\ in $\sW/\sW_{L_P}$). Denote by $\overline{C_w}$, (resp.\ $\overline{C_w^P}$, resp.\ $\overline{C_{P,w}}$) the Zariski closure of $C_w$ (resp.\ $C_w^P$, resp.\ $C_{P,w}$) in $G/B$ (resp.\ $G/B$, resp.\ $G/P$). We have
 \[G/B=\sqcup_{w\in \sW} C_w=\sqcup_{w\in \sW_{L_P}\backslash \sW} C_w^P\textrm{\ \ and\ \ }G/P=\sqcup_{w\in \sW/\sW_{L_P}} C_{P,w}\]
 and we remark that all these schemes are over $\Spec E$.

For $w\in \sW$, consider $j_w: C_w^P\hookrightarrow G/B$. By an easy variation of \cite[Lemma 5.1]{Shan12} and \cite[Rem.\ 5.2]{Shan12} to the case $G/B$, the open immersion $C_w^P \hookrightarrow \overline{C_w^P}$ is affine, hence so is $j_w$. Consider the direct image $\int_{j_w} \co_w^P=R (j_w)_* (D_{G/B \leftarrow C_w^P}\otimes_{D_{C_w^P}}^L \co_w^P)\in D^b_{\qc}(D_{G/B})$ (cf.\ \cite[\S~1.5]{HTT}). As $j_w$ is a locally closed immersion (i.e.\ the composition of an open immersion and a closed immersion), by \cite[Ex.\ 1.3.2]{HTT} and \cite[Ex.\ 1.5.12]{HTT}, we deduce that $D_{G/B \leftarrow C_w^P}$ is locally free over $D_{C_w^P}$. Together with the fact that $j_w$ is affine, we have $\int_{j_w} \co_w^P \cong (j_w)_* (D_{G/B \leftarrow C_w^P} \otimes_{D_{C_w^P}} \co_{C_w^P})=: \fN_w^P$, which is a $P$-equivariant coherent $D_{G/B}$-module, hence by \cite[Thm.\ 11.6.1 (i)]{HTT} a $P$-equivariant holonomic regular $D_{G/B}$-module. The following proposition is may-be well-known to experts, but we couldn't find a reference.

\begin{proposition}\label{paraVerGeo}
	We have $\Gamma(G/B, \fN_w^P) \cong M_P(w^{\max}w_0 \cdot 0)^{\vee}$ where $(-)^{\vee}$ denotes the dual in the BGG category $\co^{\ub}$ (cf.\ \cite[\S~3.2]{Hum08}).
\end{proposition} 
\begin{proof}
	We have $C_w^P=\sqcup_{u\in \sW_{L_P}} C_{uw^{\max}}$. For $u\in \sW_{L_P}$, we have $\lg(uw^{\max})=\lg(w^{\max})-\lg(u)$. Indeed, let $\Phi^+$ (resp.\ $\Phi^-$) be the set of positive (resp.\ negative) roots with respect $B$, then for $w'\in \sW$, $\lg(w')=\big|\{\alpha \in \Phi^+\ |\ w'(\alpha)\in \Phi^-\}\big|$. As $w^{\max}$ has maximal length in $\sW_{L_P} w$, $w^{\max}(\Phi^+)\cap \Phi_{L_P}^+=\emptyset$ hence $\lg(w^{\max})=\big|w^{\max}(\Phi^+) \cap (\Phi^-\setminus \Phi_{L_P}^-)\big| + \big|\Phi_{L_P}^-\big|$. For $u\in \sW_{L_P}$, $u$ preserves the sets $\Phi^+\setminus \Phi_{L_P}^+$ and $\Phi^- \setminus \Phi_{L_P}^-$, hence 
	\begin{eqnarray*}
		\lg(uw^{\max})&=&\big|uw^{\max}(\Phi^+)\cap \Phi^-\big|\ =\ \big|uw^{\max}(\Phi^+)\cap (\Phi^- \setminus \Phi_{L_P}^-)\big|+\big|u(\Phi_{L_P}^-) \cap \Phi_{L_P}^-\big| 
		\\
		&=&\big|w^{\max}(\Phi^+) \cap (\Phi^-\setminus \Phi_{L_P}^-)\big|+ \big|\Phi_{L_P}^-\big| - \big|u(\Phi_{L_P}^+) \cap \Phi_{L_P}^-\big| \ =\ \lg(w^{\max})-\lg(u).
	\end{eqnarray*}
	Let $d:=|\sW_{L_P}|$. For $k=0, \dots, d$, denote by $C_{w}^k:=\sqcup_{\substack{u\in \sW_{L_P}\\ \lg(u) \leq k}} C_{uw^{\max}}$ and $Z_w^k:=\sqcup_{\substack{u\in \sW_{L_P}\\ \lg(u)=k}} C_{uw^{\max}}$. By the above discussion, it is clear that $Z_w^k$ is closed in $C_w^k$ and $C_w^{k-1}=C_w^k \setminus Z_w^k$ is open in $C_w^k$, hence open in $C_w^P$. Note that $Z_w^k$ and $C_{w}^k$ are all smooth. Denote by $i_k: Z_w^k \hookrightarrow C_w^k$, $j_k:= C_w^{k-1} \hookrightarrow C_w^k$, and without ambiguity, we put $i: Z_w^k \hookrightarrow C_w^d=C_w^P$, $j: C_w^k \hookrightarrow C_w^P$ for all $k$. Applying \cite[Prop.\ 1.7.1]{HTT}, for $k=1, \dots, d$, we have distinguished triangles in $D_{\qc}^b(D_{C_w^k})$ (where we use the notation of \textit{loc.\ cit.}, for instance see page 33 of \cite{HTT} for $i_k^{\dagger}$, $j_k^{\dagger}$):
	\begin{equation}\label{dtricwk}
		\int_{i_k} i_k^{\dagger} \co_{C_w^k} \ra \co_{C_w^k} \ra \int_{j_k} j_k^{\dagger} \co_{C_w^k}\xrightarrow{+1}.
	\end{equation}
	Applying $\int_j=Rj_*$ to (\ref{dtricwk}) (for the open immersion $j: C_w^k \hookrightarrow C_w^P$, cf.\ \cite[Ex.\ 1.5.22]{HTT}), we obtain distinguished triangles in $D_{\qc}^b(D_{C_w^P})$:
	\begin{equation*}
		\int_j	\int_{i_k} i_k^{\dagger} \co_{C_w^k} \ra \int_j \co_{C_w^k} \ra \int_j \int_{j_k} j_k^{\dagger} \co_{C_w^k}\xrightarrow{+1}.
	\end{equation*}
	We have by definition $i_k^{\dagger} \co_{C_w^k}\cong \co_{Z_w^k}[k]$. Indeed, by the discussion below \cite[Def.\ 1.3.1]{HTT}, we can deduce that the derived inverse image of $\co_{C_w^k}$ via $i_k$ as $D_{C_w^k}$-module is the same as the derived inverse image of $\co_{C_w^k}$ via $i_k$ as $\co_{C_w^k}$-module (which is $\co_{Z_w^k}$ as $\co_{C_w^k}$ is obviously flat over $\co_{C_w^k}$). Using \cite[Prop.\ 1.5.21]{HTT} and \cite[Prop.\ 1.5.24]{HTT}, we deduce
\[\int_j\int_{i_k} i_k^{\dagger} \co_{C_w^k}\cong \int_i \co_{Z_w^k}[k]\cong i_* (D_{C_w^P \leftarrow Z_w^k} \otimes_{D_{Z_w^k}} \co_{Z_w^k})[k].\]
We also have $\int_j \int_{j_k} j_k^{\dagger} \co_{C_w^k} \cong \int_j j_k^{\dagger} \co_{C_w^k} \cong \int_j \co_{C_w^{k-1}}\cong Rj_* \co_{C_w^{k-1}}$ (cf.\ \cite[Ex.\ 1.5.22]{HTT} for the last equality and note that there is an abuse of notation here: $j$ in the first term is the embedding $C_w^k \hookrightarrow C_w^P$, while $j$ in the other terms is the embedding $C_w^{k-1}\hookrightarrow C_w^P$). In summary, we obtain distinguished triangles in $D_{\qc}^b(D_{C_w^P})$:
	\begin{equation}\label{dtricwk0}
		i_* (D_{C_w^P \leftarrow Z_w^k} \otimes_{D_{Z_w^k}} \co_{Z_w^k})[k] \ra Rj_* \co_{C_w^k} \ra Rj_* \co_{C_w^{k-1}} \xrightarrow{+1}.
	\end{equation}
	By taking the long exact cohomology sequence of (\ref{dtricwk0}) in the case $k=d$ (recall $C_w^d=C_w^P$), it follows: 
	\begin{equation*}
		R^l j_*\co_{C_w^{d-1}}\cong \begin{cases}
			\co_{C_w^d} & l=0\\
			\int_i \co_{Z_w^{d}} & l=d-1 \\
			0 & \text{otherwise}.
		\end{cases}
	\end{equation*}
	Using induction and (\ref{dtricwk0}) with $k$ decreasing, we can show for $k\geq 2$:
	\begin{equation*}
		R^l j_* \co_{C_w^{k-1}}\cong \begin{cases}
			j_*\co_{C_w^{k}} & l=0\\
			0 & l=1, \dots,k-2 \ (\text{if $k>2$}) \\
			R^{l} j_* \co_{C_w^k} & l\geq k+1,
		\end{cases}
	\end{equation*}
	and we have an exact sequence for $k\geq 2$
	\begin{equation}\label{seqloc0}
		0 \ra R^{k-1} j_* \co_{C_w^{k-1}} \ra \int_i \co_{Z_w^{k}} \ra R^{k} j_* \co_{C_w^k} \ra R^k j_* \co_{C_w^{k-1}} \ra 0.
	\end{equation}
	Since $j: C_w^0 \ra C_w^P$ is affine, $R^l j_* \co_{C_w^0}=0$ for all $l>0$. By (\ref{dtricwk0}) for $k=1$, we deduce
	\begin{equation}\label{seqloc1}
		0 \ra j_* \co_{C_w^1} (\cong \co_{C_w^P}) \ra j_* \co_{C_w^0} \ra \int_i \co_{Z_w^1} \ra R^1 j_* \co_{C_w^1} \ra 0
	\end{equation}
	and $R^l j_* \co_{C_w^1}=0$ for $l\geq 2$. This last fact together with $R^l j_* \co_{C_w^k} \cong R^l j_* \co_{C_w^{k-1}}$ for $l\geq k+1$ and $k \geq 2$ imply $R^l j_* \co_{C_w^k}=0$ for $l \geq k+1$ and $k\in \{0,\dots,d\}$. Hence (\ref{seqloc0}) becomes (for $k\geq 2$)
	\begin{equation*}
		0 \ra R^{k-1} j_* \co_{C_w^{k-1}} \ra \int_i \co_{Z_w^{k}} \ra R^{k} j_* \co_{C_w^k} \ra 0.
	\end{equation*}
	These exact sequences together with (\ref{seqloc1}) form a long exact sequence (noting that $C_w^0=Z_w^0$):
	\begin{equation*}
		0 \ra \co_{C_w^P} \ra \int_i \co_{Z_w^0} \ra \int_i \co_{Z_w^1} \ra \cdots \ra \int_i \co_{Z_w^d} \ra 0.
	\end{equation*}
	Applying $\int_{j_w}$ (which is exact, since $j_w$ is an affine immersion, see the discussion above Proposition \ref{paraVerGeo}), we finally obtain a long exact sequence (where we also use $j_w$ to denote the affine embeddings $Z_w^i \hookrightarrow G/B$ for all $i$)
	\begin{equation*}
		0 \ra \int_{j_w} \co_{C_w^P} \ra \int_{j_w} \co_{Z_w^0} \ra \int_{j_w} \co_{Z_w^1} \ra \cdots \ra \int_{j_w} \co_{Z_w^d} \ra 0.
	\end{equation*}
	Taking global sections (which is exact, cf.\ \cite[Thm.\ 11.2.3]{HTT}), and using \cite[Prop.\ 12.3.2 (ii)]{HTT} (where $\mathbb{D}$ of \textit{loc.\ cit.} is also referred to as the Verdier dual for coherent left $D$-modules) and \cite[Thm.\ 2.4 (ii)]{Gin86}, we obtain a long exact sequence of $\ug$-modules
	\begin{multline}\label{resoBGG}
		0 \ra \Gamma(G/B, \fN_w^P) \ra M(w^{\max}w_0\cdot 0)^{\vee} \ra \oplus_{\substack{u\in \sW_{L_P}\\ \lg(u)=1}} M(uw^{\max} w_0 \cdot 0)^{\vee} \ra \cdots \\
		\cdots\ra M(w^{\min} w_0 \cdot 0)^{\vee} \ra 0.
	\end{multline}
	By \cite[Prop.\ 9.6]{Hum08}, $\Gamma(G/B,\fN_w^P)$ has the same formal character as $M_P(w^{\max}w_0 \cdot 0)^{\vee}$. By \cite[Thm.\ 11.5.3]{HTT}, $\Gamma(G/B,\fN_w^P)\in \co^{\fp}(0)$. Taking duals, we deduce from (\ref{resoBGG}) a surjective morphism $M(w^{\max}w_0 \cdot 0) \twoheadrightarrow \Gamma(G/B,\fN_w^P)^{\vee}$ which, by \cite[Thm.\ 9.4 (c)]{Hum08}, has to factor through $M_P(w^{\max}w_0 \cdot 0) \twoheadrightarrow \Gamma(G/B,\fN_w^P)^{\vee}$. As both have the same formal character, It follows $M_P(w^{\max}w_0 \cdot 0) \cong \Gamma(G/B,\fN_w^P)^{\vee}$, which concludes the proof. (Note that we cannot apply directly \cite[Thm.\ 9.4 (b)]{Hum08} to the dual of (\ref{resoBGG}) to deduce this result as it is {\it a priori} not clear if the map $\oplus_{\substack{u\in \sW_{L_P}\\ \lg(u)=1}} M(uw^{\max} w_0 \cdot 0)\ra M(w^{\max}w_0\cdot 0)$ is the same as in {\it loc.\ cit.}).
\end{proof}

Recall we have equivalences of categories (see the discussion above Proposition \ref{cycbc}):
\begin{equation*}
	\Mod_{\rh}(D_{G/P}, B) \xleftarrow[\sim]{i_P^*}\Mod_{\rh}(D_{G/B \times G/P}, G) \xrightarrow[\sim]{i_B^*} \Mod_{\rh}(D_{G/B}, P).
\end{equation*}
We remark that the equivalence of categories $\iota:=i_P^* \circ (i_B^*)^{-1}$ in general does not induce isomorphisms of $\text{U}(\ug)$-modules when taking global sections. For example, for $P=B$, we have $\Gamma(G/B, \iota(\fL(w \cdot 0)))\cong L(w^{-1}\cdot 0)$.

For $w\in \sW$ and $U_w=G (w,1) B \times P$ (seen in $G/B \times G/P$), we have $i_B^{-1}(U_w)=C_w^P$ (resp.\ $i_P^{-1}(U_w)=C_{P,w^{-1}}$) and $U_w \cong G \times^P C_w^P$ (resp.\ $U_w \cong G\times^B C_{P,w^{-1}}$). We use $j_w$ to denote the embeddings $C_w^P \hookrightarrow G/B$, $U_w \hookrightarrow G/B \times G/P$ and $C_{P,w} \hookrightarrow G/P$. Let $\widetilde{\fN}_w:=\int_{j_w} \co_{U_w} \in \Mod_{\rh}(D_{G/B\times G/P}, G)$, and $\fN_{P,w}:=\int_{j_w} \co_{C_{P,w}}\cong (j_w)_*(D_{G/P \leftarrow C_{P,w}} \otimes_{D_{C_{P,w}}} \co_{C_{P,w}})$ (where the isomorphism follows from the fact that $j_w: C_{P,w} \ra G/P$ is an affine locally closed immersion, see the discussion below (\ref{BWlam}), see also \cite[Prop.\ 1.4.5]{Brion}). Similarly as in \cite[(13.1.7)]{HTT} (with one $B$ replaced by $P$, which does not cause any problem), we have $i_B^*\widetilde{\fN}_w \cong \fN^P_w$ and $i_P^* \widetilde{\fN}_w\cong\fN_{P, w^{-1}}$.

Recall from \S~\ref{secCCyc} that $\fM_P(w^{\max}w_0\cdot 0)\in \Mod_{\rh}(D_{G/B\times G/P}, G)$ satisfies
\[i_B^* \fM_P(w^{\max}w_0\cdot 0) \cong \Loc_{\BB} \big(M_P(w^{\max}w_0\cdot 0)\big)\in \Mod_{\rh}(D_{G/B},P). \] 
By Proposition \ref{paraVerGeo}, $i_B^* \fM_P(w^{\max}w_0 \cdot 0)$ is isomorphic to the Verdier dual of $\fN_w^P \cong i_B^*\widetilde{\fN}_w$ (for example see \cite[Thm.\ 2.4]{Gin86} and see \cite[\S~2.6]{HTT} for the Verdier dual of coherent left $D$-modules). From this, together with \cite[Thm.\ 2.7.1]{HTT} and the fact that $i_B^*$ induces an equivalence of categories, we deduce that $\fM_P(w^{\max} w_0 \cdot 0)$ is isomorphic to the Verdier dual of $\widetilde{\fN}_w$.

We can now prove Proposition \ref{equcycl} by generalizing the proof of \cite[Thm.\ 6.2]{Gin86}.

\begin{proof}[Proof of Proposition \ref{equcycl}]
	As $\fM_P(w^{\max} w_0 \cdot 0)$ is isomorphic to the Verdier dual of $\widetilde{\fN}_w$, both have the same characteristic cycle (cf.\ \cite[Prop.\ 2.6.12]{HTT}). It is sufficient to show $[\overline{X}_{w,\lambda}]=[\Ch(\widetilde{\fN}_w)]$. 
	Consider $q_{P,\lambda}: G \times^P \ur_{P,\lambda} \ra \ug$. We have as in (\ref{Xpqpb}) $X_{P,\lambda}\cong G \times^B q_{P,\lambda}^{-1}(\ub)$. Similarly as in Remark \ref{remXw} (2), we see that $q_{P,\lambda}^{-1}(\ub)$ is equidimensional of dimension $\dim \ur_{P,\lambda}$ with irreducible components given by $\{q_{P,\lambda}^{-1}(\ub)_w\}_{w\in \sW/\sW_{L_P}}$, where $q_{P,\lambda}^{-1}(\ub)_w$ denotes the Zariski closure of the preimage $q_{P,\lambda}^{-1}(\ub)_w^0$ of $C_{P,w}$ in $q_{P,\lambda}^{-1}(\ub)$ (with the reduced subscheme structure) via the composition $q_{P,\lambda}^{-1}(\ub) \ra G/P \times \ug \ra G/P$. We also have $X_{w,\lambda}\cong G \times^B q_{P,\lambda}^{-1}(\ub)_w$.
	Let $\kappa_{P,\lambda}$ denote the morphism $G \times^P \ur_{P,\lambda} \ra \fz_{\lambda}$, $(g,\psi) \mapsto \overline{\psi}$ and $\overline{q_{P,\lambda}^{-1}(\ub)_w}:=q_{P,\lambda}^{-1}(\ub)_w \times_{\kappa_{P,\lambda}, \fz_{\lambda}} \{0\}$ (with the canonical scheme structure). Then we have an isomorphism of schemes
	\begin{equation*}
		\overline{X}_{w,\lambda}\cong G \times^B \overline{q_{P,\lambda}^{-1}(\ub)_{w^{-1}}}.
	\end{equation*} 
	By Proposition \ref{cycbc}, it is sufficient to show (as cycles in $T^* G/P$)
	\begin{equation*}
		\big[\overline{q_{P,\lambda}^{-1}(\ub)_{w^{-1}}}\big] \cong [\Ch(i_P^* \widetilde{\fN}_w)]=[\Ch(\fN_{P,w^{-1}})]=\Big[\Ch\Big(\int_{j_{w^{-1}}} \co_{C_{P,w^{-1}}}\Big)\Big].
	\end{equation*}
	By abuse of notation, we still denote $\lambda\circ \dett_{L_P}$ by $\lambda$, which is now a dominant weight of $\ft$ (with respect to $B$). 
	Let $\chi_{-\lambda}$ be the character of $P$ (which factors through $L_P$) of weight $-\lambda$ over $E$. Put $\cL_{-\lambda}:=G \times^P \chi_{-\lambda}$, which is a line bundle over $G/P$. By the Borel-Weil-Bott theorem, $H^0(G/P, \cL_{-\lambda})\cong L^-(-\lambda)$ (:= the finite dimensional algebraic representation of $G$ of lowest weight $-\lambda$, i.e.\ $L^-(-\lambda)^{N^-}=E (-\lambda)$ for the unipotent radical $N^-$ of the Borel subgroup $B^-$ opposite to $B$). Indeed, we have by definition
	\begin{equation}
		\label{BWlam}H^0(G/P, \cL_{-\lambda})\cong \{f: G \ra E \text{ algebraic functions}\ | \ f(gp)=\chi_{-\lambda}^{-1}(p) f(g), \ \forall\ p \in P, g\in G\}
	\end{equation}
	with the $G$-action given by $(gf)(g')=f(g^{-1} g')$. Let $L_P^{\lambda}\subset L_P$ be the kernel of $\chi_{-\lambda}: L_P \ra E^{\times}$. Consider $\kappa: Y:=G/(L_P^{\lambda} N_P) \twoheadrightarrow G/P$. Each element in $H^0(G/P, \cL_{-\lambda})$ can be viewed as an algebraic function on $Y$. Let $\fe$ be the highest weight vector of $L(\lambda)\cong L^-(-\lambda)^{\vee}$ (which is the finite dimensional algebraic representation of $G$ of highest weight $\lambda$ with respect to $B$), i.e.\ $N \fe=\fe$ and $\beta \fe=\chi_\lambda(\beta)\fe=\chi_{-\lambda}^{-1}(\beta) \fe$ for all $\beta\in P$ where $N\subset B$ is the unipotent radical. Let $\fe^*$ be the lowest weight vector of $L^-(-\lambda)$. Put $f_1: G \ra E$, $g \mapsto \langle \fe^*, g\fe\rangle$ (with $\langle-,-\rangle$ the natural pairing between $L(\lambda)$ and $L^-(-\lambda)$), which corresponds to $\fe^*$ via (\ref{BWlam}). Let 
	\[f_{w^{-1}}:=w^{-1}f_1=[g\mapsto \langle \fe^*, wg \fe\rangle]\] which corresponds to $w^{-1} \fe^*$ via (\ref{BWlam}) and which we can and do view as an algebraic function on $Y$. Applying \cite[Prop.\ 1.4.5]{Brion} (that easily generalizes to our connected split reductive group $G$) and using the natural surjection $G/B \twoheadrightarrow G/P$\footnote{This surjection induces a surjection $G \times^B \chi_{-\lambda} \twoheadrightarrow G \times^P \chi_{-\lambda}$, and we have $H^0(G/B, G \times^B \chi_{-\lambda})\cong L^-(-\lambda)$. We can then deduce the desired results for Schubert cells in $G/P$ from those for Schubert cells in $G/B$ given in \cite[Prop.\ 1.4.5]{Brion}.}, we can deduce that $\kappa^{-1}(\partial \overline{C_{P,w^{-1}}})$ (where $\partial \overline{C_{P,w^{-1}}}=\overline{C_{P,w^{-1}}}\setminus C_{P,w^{-1}}$) is exactly the zero locus of $f_{w^{-1}}|_{\kappa^{-1}(\overline{C_{P,w^{-1}}})}$.
	
	 Let $\cU\subset G/P$ be (Zariski-open) complement of the zero locus of the (global) section $f_{w^{-1}}$ of the line bundle $\cL_{-\lambda}$ over $G/P$. So $\kappa^{-1}(\cU)=Y\setminus f_{w^{-1}}^{-1}(0)$ and $C_{P,w^{-1}}=\overline{C_{P,w^{-1}}} \cap \cU$, in particular $C_{P,w^{-1}}$ is Zariski-closed in $\cU$. Denote by $\iota: C_{P,w^{-1}} \hookrightarrow \cU$ the closed embedding, then we have (see \cite[Ex.\ 2.3.8]{HTT} for the first isomorphism, the last isomorphism is induced by the Killing form)
	\begin{equation*}
		\Ch\Big(\int_{\iota} \co_{C_{P,w^{-1}}}\Big)\cong T_{C_{P,w^{-1}}}^* \cU\cong T^*_{C_{P,w^{-1}}} G/P \hookrightarrow T^* G/P \cong G\times^P (\ug/\fp)^{\vee}\cong G\times^P \fn_P.
	\end{equation*}
	where $T_{C_{P,w^{-1}}}^* G/P$ denotes the conormal bundle of $C_{P,w^{-1}}$ in $G/P$. Recall we have $q_P: T^*G/P \cong G \times^P \fn_P \ra \ug$, $(g, \psi)\mapsto \Ad(g) \psi$, and $q_P^{-1}(\fn)^{\red}=\cup_{u\in \sW/\sW_{L_P}} T_{C_{P,u}}^* G/P \hookrightarrow T^* G/P$ (for example by the same argument as in the proof of \cite[Prop.\ 3.3.4]{ChGi}). 
	
	We now apply Ginz\-burg's method in \cite[\S~6.3]{Gin86} to $\Lambda:=\Ch(\int_{\iota} \co_{C_{P,w^{-1}}})$ and $f:=f_{w^{-1}}$ in order to calculate $\Ch(\int_{j_{w^{-1}}} \co_{C_{P,w^{-1}}})=\Ch(\fN_{P,w^{-1}})$ (don't confuse $j_{w^{-1}}$ with the above $\iota$, and note that $j_{w^{-1}}$ is not closed in general). We have $T^* G/(L_P^{\lambda} N_P)\cong G \times^{L_P^{\lambda}N_P} \ur_{P,\lambda}$ (identifying $\ug$ to $\ug^{\vee}$ via the Killing form). For $x\in Bw^{-1} P$, let $df_x\in E\lambda+\fn_P$ such that $(x,df_x)\in df \subset G \times^{L_P^{\lambda}N_P} (E \lambda+\fn_P)\cong T^* G/(L_P^{\lambda} N_P)$. As the map $q: G \times^{L_P^{\lambda}N_P} (E \lambda+\fn_P) \ra \ug$, $(g, \psi)\mapsto \Ad(g) \psi$ coincides with the moment map $T^* G/(L_P^{\lambda} N_P) \ra \ug^\vee$ ($\cong \ug$) (cf.\ \cite[\S~1.4]{ChGi}), we have for $X\in \ug$ the equality $(q(x,df_x))(X)=\langle \fe^*, w (-X) x \fe\rangle$. By multiplying $x$ on the right by an element of $P$, we can and do assume that $x$ has the form $uw^{-1} \in C_{P,w^{-1}}$ with $u\in B$ satisfying $w uw^{-1} \in B^-$. Let $K(\cdot, \cdot)$ denote the Killing form on $\ug$, we can calculate:
	\begin{equation}\label{formome}
		q(uw^{-1}, df_{uw^{-1}})(X)=-\chi_{w^{-1}(\lambda)}(u)K\big(\Ad(uw^{-1})(\lambda),X\big).
	\end{equation}
	Hence $q(uw^{-1}, df_{uw^{-1}}) \in E^{\times} \Ad(uw^{-1})(\lambda)$. By unwinding the definition of $q(-,-)$, this implies $df_{uw^{-1}}\in E^{\times} \lambda$. 
	
It is clear that $G/(L_P^{\lambda} N_P)$ is a principal $\bG_m$-bundle over $G/P$, and there is a natural induced $\bG_m$-action on $T^* G/(L_P^{\lambda} N_P)$ such that $(T^* G/(L_P^{\lambda}N_P))/\bG_m\cong G\times^P \ur_{P,\lambda}$. For $a\in E^{\times}$, we claim that we have by (\ref{formome}) an isomorphism of schemes: 
	\begin{equation}\label{forcyc22}
	(\kappa^* \Lambda+adf)/\bG_m \cong q_{P,\lambda}^{-1}(-aw^{-1}(\lambda)+\fn) \cap \kappa_{P,\lambda}^{-1}(-a\lambda).
	\end{equation} 
	Indeed, by checking the formula at each closed point of $C_{P,w^{-1}}$, we can obtain the equality after taking the reduced subscheme structure on both sides. However, as in the proof of \cite[Thm.\ 6.3]{Gin86}, the left hand side is locally isomorphic to a translation of $\Lambda\cong T^*_{C_{P,w^{-1}}} G/P\subset T^*G/P \cong G \times^P \fn_P \subset G\times^P \ur_{P,\lambda}$ in $G\times^P \ur_{P,\lambda}$, hence is reduced. On the other hand, we have a closed immersion 
	\begin{equation}\label{w-1lambda}
q_{P,\lambda}^{-1}(-aw^{-1}(\lambda)+\fn) \times_{G \times^P \ur_{P,\lambda}} \kappa_{P,\lambda}^{-1}(-a\lambda)\hooklongrightarrow q_P^{-1}(\ub) \times_{G \times^P \ur_P} \kappa_P^{-1}(-a\lambda),
	\end{equation}
	which is an equality on closed points. By Proposition \ref{genesmoo} and its proof, $q_P^{-1}(\ub) \times_{G \times^P \ur_P} \kappa_P^{-1}(-a\lambda)$ is \ smooth \ and \ Zariski \ closed \ in \ $q_P^{-1}(\ub)$. \ It \ is \ easy \ to \ see \ that \ the \ closed \ subschemes $\kappa_P^{-1}(-a\lambda) \times_{G \times^P \ur_P} q_P^{-1}(\ub)_{w'}$ (for $w'\in \sW/\sW_{L_P}$) of $q_P^{-1}(\ub) \times_{G \times^P \ur_P} \kappa_P^{-1}(-a\lambda)$ are disjoint, hence each is open and smooth (and reduced). We deduce
\begin{equation*}
	q_P^{-1}(\ub) \times_{G \times^P \ur_P} \kappa_P^{-1}(-a\lambda) = \sqcup_{w'\in \sW/\sW_{L_P}} \kappa_P^{-1}(-a\lambda) \times_{G \times^P \ur_P} q_P^{-1}(\ub)_{w'}.
	\end{equation*}
Then one can check that (\ref{w-1lambda}) factors through a closed immersion 
	\begin{equation*}
	q_{P,\lambda}^{-1}(-aw^{-1}(\lambda)+\fn) \times_{G \times^P \ur_{P,\lambda}} \kappa_{P,\lambda}^{-1}(-a\lambda)\hooklongrightarrow q_P^{-1}(\ub)_{w^{-1}} \times_{G \times^P \ur_P} \kappa_P^{-1}(-a\lambda),
	\end{equation*}
	which is bijective on closed points hence is an isomorphism since the right hand side is reduced. In particular, the right hand side of (\ref{forcyc22}) is also reduced so (\ref{forcyc22}) holds. 
	
	Let $\Lambda^{\sharp}:=q_{P,\lambda}^{-1}(E^{\times} w^{-1}(\lambda) +\fn) \cap \kappa_{P,\lambda}^{-1}(E^{\times} \lambda)$. By similar arguments as above, one can show isomorphisms of reduced schemes
	\[\Lambda^{\sharp}\xrightarrow{\sim} q_{P,\lambda}^{-1}(\ub)_{w^{-1}} \times_{G \times^P \ur_{P,\lambda}} \kappa_{P,\lambda}^{-1}(E^{\times} \lambda) \xrightarrow{\sim}q_{P}^{-1}(\ub)_{w^{-1}} \times_{G \times^P \ur_P} \kappa_P^{-1}(E^{\times} \lambda).\]
	We also have $\Lambda^{\sharp} \cong q_{P,\lambda}^{-1}(\ub)_{w^{-1}} \times_{G \times^P \ur_P} \kappa_{P}^{-1}(\fz^{\reg-\sss})$ (e.g.\ using the fact that both have the same $\overline{E}$-points and are reduced), hence $\Lambda^{\sharp}$ is Zariski-open (and Zariski-dense) in $q_{P,\lambda}^{-1}(\ub)_{w^{-1}}$.
	Thus the scheme theoretic image of $\Lambda^{\sharp}$ in $(T^* G/(L_P^{\lambda} N_P))/\bG_m \cong G \times^P \ur_{P,\lambda}$ is just $q_{P,\lambda}^{-1}(\ub)_{w^{-1}}$.
	By Ginzburg's formula (\cite[Thm.\ 6.3]{Gin86}, see also \cite[Thm.\ 3.2]{Gin86}, and note that the theorem is ``algebraic" so it can be applied with $\bC$ replaced by $E$) applied to the case where $j$ is the open immersion $\cU \hookrightarrow G/P$ and $\mathscr{M}$ is $\int_{\iota} \co_{C_{P,w^{-1}}}$ (so $\gr \mathscr{M}$ is just $\Ch(\int_{\iota} \co_{C_{P,w^{-1}}})$ and $\lim_{s \ra 0} (\gr \mathscr{M})^s$ in \textit{loc.\ cit.}\ is the fibre of the scheme theoretic image of $\Lambda^{\sharp}$ at $0$ via $\kappa_{P}$), we obtain:
	\begin{equation*}
	\Big[\Ch\Big(\int_{j_{w^{-1}}} \co_{C_{P,w^{-1}}}\Big)\Big] = \Big[q_{P,\lambda}^{-1}(\ub)_{w^{-1}} \times_{ \kappa_{P,\lambda}, \fz_{\lambda}} \{0\}\Big].
	\end{equation*}
	The proposition follows. 
\end{proof}

\subsection{Bruhat intervals of length $2$}

We show some properties of Bruhat intervals of length $2$ in certain parabolic quotient of $S_n$, that are used in Corollary \ref{thmComCon}.

We write $w\in S_n$ in the form $w=\begin{pmatrix}
	1& 2&\cdots & n\\w^{-1}(1) & w^{-1}(2) & \cdots & w^{-1}(n)
\end{pmatrix}$. For $i\neq j\in \{1,\dots, n\}$, we denote by $t_{ij}$ the transposition exchanging $i$ and $j$. If $w=\begin{pmatrix}
1& 2&\cdots & n\\ a_1 & a_2 & \cdots &a_n
\end{pmatrix}$ and $i<j$, then $t_{ij} w=\begin{pmatrix}
1 & \cdots & i & \cdots &j &\cdots &n \\
a_1& \cdots & a_j& \cdots & a_i & \cdots & a_n
\end{pmatrix}$. Let $r\geq 1$, and $n_i\in \Z_{\geq 1}$ for $1\leq i \leq r$ such that $\sum_{i=1}^r n_i=n$. Put $s_i:=\sum_{j=1}^i n_i$ and $J_i:=\{s_{i-1}+1, \dots, s_{i}\}$. We have thus $\{1, \dots, n\}=\sqcup_{i=1}^{r} J_i$. Let $\sW_J$ be the subgroup of $S_n$ generated by the $t_{jk}$ for $j,k\in J_i$, $i=1, \dots, r$ (so $\sW_J\cong \prod_{i=1}^r S_{n_i}$). Denote by $\sW_{\max}^J\subseteq S_n$ the set of maximal length representatives of $\sW_J \backslash S_n$.

\begin{lemma}\label{maxcrit}
	We have $w=\begin{pmatrix}
		1& 2&\cdots & n\\ a_1 & a_2 & \cdots &a_n
	\end{pmatrix}\in \sW_{\max}^J $ if and only if for any $j=1, \dots, r$ the sequence $(a_j)_{j\in J_i}$ is decreasing.
\end{lemma}

Let $[w_1, w_2]$ be a Bruhat interval of length $2$, i.e.\ $w_1<w_2$, $\lg(w_2)=\lg(w_1)+2$, and $[w_1, w_2]=\{w'\in S_n \ |\ w_1<w'<w_2\}$. Recall $|[w_1, w_2]|=2$ (cf.\ \cite[Lemma 5.2.7]{BHS3}). Assume $w_1, w_2\in \sW_{\max}^J$, and denote by $[w_1, w_2]_J:=\{w'\in W^J_{\max} \ | w_1< w' <w_2\}$. Recall $[w_1,w_2]_J$ is called \textit{full} if $[w_1,w_2]_J=[w_1,w_2]$. As $\lg(w_2)=\lg(w_1)+2$, there exist reflections $t_{ab}$, $t_{cd}$ such that $w_2=t_1 t_2 w_1$. We call a full interval $[w_1, w_2]_J$ (of length $2$) \textit{nice} if the integers $a$, $b$, $c$, $d$ cannot be contained in two $J_i$.

\begin{remark}
Identifying $S_n$ with the Weyl group $\sW$ of $\GL_n$, $J$ corresponds to a parabolic subgroup $P$ of $\GL_n$ containing $B$ such that $\sW_{L_P}=\sW_J$. It is easy to see that, if $[w_1, w_2]_J$ is full and nice, then $\dim \fz_{L_P}^{w_2w_1^{-1}}=\dim \fz_{L_P}-2$ (see Proposition \ref{propsmoX} for the notation). 
\end{remark}

\begin{proposition}\label{bruhInt}
	Let $w=\begin{pmatrix}
		1& 2&\cdots & n\\ a_1 & a_2 & \cdots &a_n
	\end{pmatrix}\in \sW_{\max}^J$ with $\lg(w)\leq \lg(w_{0})-2$. Assume that the partition $\{J_i\}$ satisfies: if $|J_i|>1$ then $|J_{i-1}|=1$ (if $i\geq 2$) and $|J_{i+1}|=1$ (if $i\leq r-1$). Then there exists $u\in \sW_{\max}^J$ such that $u>w$, $\lg(u)=\lg(w)+2$ and one of the following two properties is satisfied:
	\begin{enumerate}
		\item[(1)] $[w,u]_J$ is full and nice;
		\item[(2)] $[w,u]_J$ is not full.
	\end{enumerate}
\end{proposition}

We will frequently use the following easy lemma (see for example \cite[Lemma 2.1.4]{BjBr}):

\begin{lemma}\label{findcover} Let $w=\begin{pmatrix}
		1& 2&\cdots & n\\ a_1 & a_2 & \cdots &a_n
	\end{pmatrix}\in S_n$ and $i<j$. Then $t_{ij} w$ is a cover of $w$ (i.e.\ $t_{ij}w>w$ and $\lg(t_{ij}w)=\lg(w)+1$) if and only if $a_i<a_j$ and there does not exist $i$, $j$, $k$ with $i<k<j$ such that $a_i<a_k<a_j$. 
\end{lemma}

\begin{proof}[Proof of Proposition \ref{bruhInt}]
	Let $I:=\{j\ |\ a_j<a_{j+1}\}$. Since $w\in \sW_{\max}^J$, for any $j=1, \dots, r$, the sequence $(a_j)_{j\in J_i}$ is decreasing. We have thus $I\subset \{s_j\ |\ j=1, \dots, r-1\}$. Since $w\neq w_{0}$, $I\neq \emptyset$. We prove the proposition by a (somewhat tedious) case-by-case discussion. 
	
	{\bf Case (1)}: Assume there exist $s_{k_1}, s_{k_2} \in I$ such that $k_2>k_1+1$. In this case, we don't need the assumption on $\{J_i\}$. For $k_i$, consider the set $I_i:=\{(j_1, j_2)\ | \ j_1 \in J_{k_i}, j_2\in J_{k_i+1}, a_{j_1}<a_{j_2}\}$. We have $(s_i, s_i+1)\in I_i$, so $I_i\neq \emptyset$. Let $(j_{i,1}, j_{i,2})\in I_i$ such that $j_{i,1}-j_{i,2}$ is maximal. We have thus $a_{j_{i,1}}<a_{j_{i,2}}$ and $\begin{cases}a_{j_{i,1}-1}>a_{j_{i,2}} \text{ if } j_{i,1}-1\in J_{k_i} \\ a_{j_{i,1}}>a_{j_{i,2}+1} \text{ if }j_{,2}+1\in J_{k_i+1}.\end{cases}$\!\!\!\!By Lemma \ref{findcover}, $t_{j_{i,1}j_{i,2}}w$ is a cover of $w$ (in fact, for any $(j_1,j_2)\in I_i$, $t_{j_1j_2}w$ is a cover of $w$). By the choice of $(j_{i,1},j_{i,2})$ and Lemma \ref{maxcrit}, we see that $t_{j_{i,1}j_{i,2}}w\in \sW^J_{\max}$. As $k_2>k_1+1$, $t_{j_{1,1}j_{1,2}}$ and $t_{j_{2,1} j_{2,2}}$ commute. Put $u:=t_{j_{1,1}j_{1,2}}t_{j_{2,1} j_{2,2}} w$, then it is easy to see that $[w,u]_J=[w,u]=w< t_{j_{1,1}j_{1,2}}w, t_{j_{2,1} j_{2,2}} w< u$ is full and nice.
	
	{\bf Case (2)}: There exists $k$ such that $I=\{s_k\}$. We have either $n_k=1$ or $n_{k+1}=1$. Note that we cannot have $n_k=n_{k+1}=1$ since if so $\lg(w)=\lg(w_0)-1$. 
	
	(2.1) Assume $n_k=1$, let $j\in J_{k+1}$ be maximal such that $a_{s_k}<a_{j}$. We have $j\geq s_{k}+2$ since otherwise $\lg(w)=\lg(w_0)-1$. Put $u:=t_{s_k (j-1)} t_{s_k j} w$, then $\lg(u)=\lg(w)+2$ and $[w,u]=w< t_{s_k j} w, t_{s_k(j-1)} w< u$. However $t_{s_k(j-1)}\notin \sW_{\max}^J$. So $[w,u]_J$ is not full. 
	
	(2.2) The case $n_{k+1}=1$ is parallel to (2.1). 
	
	{\bf Case (3)}: There exists $k$ such that $I=\{s_k,s_{k+1}\}$. By our assumption on $\{J_i\}$, we can further divide this case into two cases: $n_{k+1}=1$ or $n_{k}=n_{k+2}=1$. 
	
	(3.1) $n_{k+1}=1$:	let $j\in J_k$ be minimal such that $a_j<a_{s_k+1}$ (note that $s_k+1=s_{k+1}$ in this case). 
	
	(3.1.1) If $j<s_k$, put $u:=t_{j+1 s_{k+1}}t_{j s_{k+1}}w\in \sW^J_{\max}$. In $S_n$, the interval $[w,u]$ is given by $w<t_{js_{k+1}}w, t_{(j+1)s_{k+1}} w <u$. However, $t_{(j+1)s_{k+1}} w\notin \sW^J_{\max}$. So $[w,u]_J$ is not full. 
	
	(3.1.2) If $j=s_k$, let $j_1\in J_{k+2}$ (resp.\ $j_2\in J_{k+2}$) be maximal such that $a_{j_1}>a_{s_{k+1}}$ (resp.\ $a_{j_2}>a_{s_k}$) (recalling that $a_{s_{k+1}+1}>a_{s_{k+1}}>a_{s_k}$). Note that $j_2\geq j_1$. 
	
	(3.1.2.1) If $a_{s_{k+1}}>a_{j_2}$, then put $u:=t_{s_{k+1}j_2} t_{s_k s_{k+1}}w\in \sW^J_{\max}$. Then $[w,u]_J=[w,u]=w<t_{s_ks_{k+1}} w, t_{s_k j_2} w <u$. Indeed, as $a_{s_{k+1}}>a_{j_2}$, $t_{s_kj_2}w$ is a cover of $w$. It is also easy to see that $[w,u]_J$ is nice.
	
	(3.1.2.2) If $a_{s_{k+1}}<a_{j_2}$ (so $j_1=j_2$), put (again) $u:=t_{s_{k+1}j_2} t_{s_k s_{k+1}}w\in \sW^J_{\max}$. In this case we have $[w,u]_J=[w,u]=w<t_{s_ks_{k+1}} w, t_{s_{k+1} j_2} w < u$, and $[w,u]_J$ is nice (so the only difference with (3.1.2.1) is that $t_{s_k j_2}w$ is replaced by $t_{s_{k+1} j_2}w$).
	
	(3.2) $n_k=n_{k+2}=1$. Let $j_1\in J_{k+1}$ (resp.\ $j_2\in J_{k+1}$) be maximal (resp.\ minimal) such that $a_{j_1}>a_{s_k}$ (resp.\ $a_{j_2}<a_{s_{k+2}}$).
	
	(3.2.1) If $j_1\geq s_k+1$ (resp.\ $j_2 \leq s_{k+1}-1$), which implies $n_{k+1}>1$, then one can use the same argument as in (2.1) (resp.\ (2.2)) to find $u\in \sW^J_{\max}$ such that $[w,u]_J$ is not full.
	
	(3.2.2) \ If $j_1=s_k+1$ \ and \ $j_2=s_{k+1}$, \ and \ if \ $n_{k+1}>1$ \ (so \ $j_2>j_1$), \ then \ put $u\!:=\!t_{s_k (s_k+1)} t_{s_{k+1}s_{k+2}} w\!= t_{s_{k+1}s_{k+2}}t_{s_k (s_k+1)} w$. It is easy to see $[w,u]_J=w< t_{s_k (s_k+1)}w, t_{s_{k+1}s_{k+2}} w<u$ is full and nice.
	
	(3.2.3) If $n_{k+1}=1$, this is a special case of (3.1.2.2).
	
	This concludes the proof.	
\end{proof}

\begin{remark}\label{bruhInv}
Without the assumption on $\{J_i\}$ in Proposition \ref{bruhInt}, it could happen that for any $u\in \sW_{\max}^J$, $u>w$ and $\lg(u)=\lg(w)+2$, $[w,u]_J$ is full but not nice. For example, let $n=4$, $r=2$, $J_1=\{1,2\}$, $J_2=\{3,4\}$, and $w=t_{12}t_{34}t_{23}=\begin{pmatrix}
	1& 2 &3 & 4\\
	3 & 1 & 4 & 2
\end{pmatrix}\in \sW_{\max}^J$. Then $u=t_{34}t_{23}t_{12}t_{23}t_{34}=\begin{pmatrix}
1 & 2 & 3 & 4\\
4 & 2 & 3 &1
\end{pmatrix}$ is the only element in $\sW_{\max}^J$ such that $u>w$ and $\lg(u)=\lg(w)+2$. One can check $[w,u]_J=[w,u]=\Big\{\begin{pmatrix}
1 & 2 & 3 & 4\\
4&1&3&2
\end{pmatrix}, \begin{pmatrix}
1 & 2 & 3 & 4\\
3& 2 & 4 & 1
\end{pmatrix}\Big\}$ hence is full. However, in this case, any full interval (of length 2) is not nice. 
\end{remark}

\subsection{Errata to \cite{BHS3}}

In the equality following the definition of the scheme $Z$ in \cite[(2.11)]{BHS3}, $\widetilde \cN \times_{\cN}\widetilde \cN$ should be replaced by $(\widetilde \cN \times_{\cN}\widetilde \cN)^{\red}$ (in fact, it is possible that $\widetilde \cN \times_{\cN}\widetilde \cN=(\widetilde \cN \times_{\cN}\widetilde \cN)^{\red}$, but we don't need it).\\

\noindent
In \cite[\S~4.3]{BHS3}, all $[K:\Q_p] \frac{n(n+3)}{2}$ should be replaced by $n+[K:\Q_p] \frac{n(n+1)}{2}$.\\

\noindent
In \cite[\S~5]{BHS3}, the BGG category $\mathcal O$ should be replaced by its full subcategory ${\mathcal O}_{\alg}$ of objects with integral weights (see \cite{OS}).

\Addresses
	\printindex
	\end{document}